\documentclass[11pt,a4paper,reqno]{amsart} 

\usepackage{amsmath}
\usepackage{amssymb}
\usepackage{euscript}
\usepackage{multirow}
\usepackage{graphicx}
\usepackage{comment}
\usepackage{xcolor}
\usepackage[colorlinks=true,linktocpage=true,pagebackref=false, citecolor=black,linkcolor=black]{hyperref}
\usepackage{pifont}
\usepackage{tikz,pgfplots}
\usepackage{mathabx}
\usepackage[all]{xy}
\usepackage{enumitem}

\pgfplotsset{compat=1.10}
\usepgfplotslibrary{fillbetween}
\usetikzlibrary{cd,arrows,decorations.pathmorphing,backgrounds,automata,positioning,fit,matrix}

\pgfplotsset{soldot/.style={color=black,only marks,mark=*}} \pgfplotsset{holdot/.style={color=black,fill=white,only marks,mark=*}}

\newtheorem{thm}{Theorem}[subsection]

\newtheorem{lem}[thm]{Lemma}
\newtheorem{prop}[thm]{Proposition}

\newtheorem{cor}[thm]{Corollary}
\newtheorem{crit}[thm]{Criterion}

\newtheorem{alg}[thm]{Algorithm}

\newtheorem{notation}[thm]{Notation}

\newtheorem{questions}[thm]{Questions}

\theoremstyle{definition}
\newtheorem{defn}[thm]{Definition}
\newtheorem{defns}[thm]{Definitions}

\newtheorem*{ack}{Acknowledgements}

\theoremstyle{remark}
\newtheorem{remark}[thm]{Remark}
\newtheorem{remarks}[thm]{Remarks}
\newtheorem{example}[thm]{Example}
\newtheorem{examples}[thm]{Examples}

\numberwithin{equation}{section}
\numberwithin{figure}{section}

 \newcommand{\N}{{\mathbb N}}
\newcommand{\Z}{{\mathbb Z}} \newcommand{\R}{{\mathbb R}}
\newcommand{\Q}{{\mathbb Q}} \newcommand{\C}{{\mathbb C}}

 \newcommand{\PP}{{\mathbb P}}

\newcommand{\reg}{{\mathcal R}} 
 
 \newcommand{\I}{{\mathcal I}}
\newcommand{\G}{\mathbb{G}}

\newcommand{\gtp}{{\mathfrak p}} \newcommand{\gtq}{{\mathfrak q}}
\newcommand{\gtm}{{\mathfrak m}} \newcommand{\gtn}{{\mathfrak n}}
\newcommand{\gta}{{\mathfrak a}} \newcommand{\gtb}{{\mathfrak b}}
 \newcommand{\gtQ}{{\mathfrak Q}}
\newcommand{\gtd}{{\mathfrak d}} 
 \newcommand{\gtA}{{\mathfrak A}}
 \newcommand{\gtM}{{\mathfrak M}} \newcommand{\gtN}{{\mathfrak N}}  

\newcommand{\Bb}{{\EuScript B}}

\newcommand{\Zz}{{\EuScript Z}}
\newcommand{\cinfty}{{\EuScript{C}^\infty}}
\newcommand{\czero}{{\EuScript{C}^0}}

\newcommand{\Reg}{\operatorname{Reg}}
\newcommand{\Sing}{\operatorname{Sing}}

\newcommand{\qf}{\operatorname{qf}}

\newcommand{\cl}{\operatorname{Cl}}

\newcommand{\hgt}{\operatorname{ht}}

\newcommand{\tr}{\operatorname{tr}}

\newcommand{\id}{\operatorname{id}}

\newcommand{\x}{{\tt x}}
\newcommand{\y}{{\tt y}} 
\newcommand{\z}{{\tt z}}
\renewcommand{\t}{{\tt t}}

\newcommand{\ii}{{\tt i}}

\newcommand{\ol}{\overline}

\newcommand{\qr}{{\ol{\Q}^r}}

\newcommand{\qbar}{{\ol{\Q}}}

\newcommand{\ove}{{\ol{E}}}

\newcommand{\kr}{{\ol{K}^r}}
\newcommand{\krn}{{(\kr)^n}}
\newcommand{\kbar}{{\ol{K}}}
\newcommand{\kbarn}{{\kbar^n}}

\newcommand{\mr}{\mathrm}
\newcommand{\zcl}{\mr{Zcl}}

\newcommand{\II}{{\mc{I}}}
\newcommand{\ZZ}{{\mc{Z}}}
\newcommand{\mc}{\mathcal}
\newcommand{\sing}{\mr{Sing}}
\newcommand{\mk}{\mathfrak}
\newcommand{\sfh}{\mathsf{h}}

\makeatletter
\def\@tocline#1#2#3#4#5#6#7{\relax
 \ifnum #1>\c@tocdepth % then omit
 \else
 \par \addpenalty\@secpenalty\addvspace{#2}%
 \begingroup \hyphenpenalty\@M
 \@ifempty{#4}{%
 \@tempdima\csname r@tocindent\number#1\endcsname\relax
 }{%
 \@tempdima#4\relax
 }%
 \parindent\z@ \leftskip#3\relax \advance\leftskip\@tempdima\relax
 \rightskip\@pnumwidth plus4em \parfillskip-\@pnumwidth
 #5\leavevmode\hskip-\@tempdima
 \ifcase #1
 \or\or \hskip 1em \or \hskip 2em \else \hskip 3em \fi%
 #6\nobreak\relax
 \dotfill\hbox to\@pnumwidth{\@tocpagenum{#7}}\par
 \nobreak
 \endgroup
 \fi}
\makeatother
\begin{document}

$\;$
\vspace{-2em}
\title[Subfield-algebraic geometry]{Subfield-algebraic geometry}

\author{Jos\'e F. Fernando}
\address{Departamento de \'Algebra, Geometr\'\i a y Topolog\'\i a, Facultad de Ciencias Matem\'aticas, Universidad Complutense de Madrid, Plaza de Ciencias 3, 28040 MADRID (SPAIN)}
\email{josefer@mat.ucm.es}

\author{Riccardo Ghiloni}
\address{Dipartimento di Matematica, Universit\`a degli studi di Trento, Via Sommarive 14, 38123 Povo-Trento (ITALY)}
\email{riccardo.ghiloni@unitn.it}


\maketitle 

\renewcommand{\contentsname}{Table of Contents}
{
\footnotesize
\setcounter{tocdepth}{2}
\tableofcontents
}

%%%
\section{Introduction}

\subsection{The idea of subfield-algebraic geometry and overview of the contents}
This paper is devoted to the study of the geometric properties of algebraic sets in an affine or projective space over a fixed field $L$ that derive from polynomials with coefficients in a fixed subfield $K$ of~$L$. Here we develop the foundational elements and results of the new theory resulting from this study, which we call `subfield-algebraic geometry'.

Our main goal is to study the real case, more precisely the case in which $L$~is a real closed field and $K$ is an ordered subfield of $L$. The most interesting phenomena appear when $K$ is not a real closed field. A main example to keep in mind is the one in which $L$ is the field $\R$ of real numbers or the field $\qr$ of real algebraic numbers and $K$ is the field $\Q$ of rational numbers.

Recall that an ordering of a field $L$ is a total order relation $\leq$ on $L$ that is compatible with the field operations of $L$ in the following sense: $\forall x,y,z\in L$, $(x\leq y\Longrightarrow x+z\leq y+z)$ and $(0\leq x,0\leq y\Longrightarrow0\leq xy)$. The field $L$ is ordered if it endowed with such an ordering. The field $L$ is real if it admits at least one ordering. A field $L$ is a real closed field, r.c.f.\ for short, if $L$ is real but all the proper algebraic extensions of $L$ are not. A real closed field admits only one ordering. If $L$ is an ordered field, e.g., a real closed field endowed with its unique ordering, then an ordered subfield of $L$ is a subfield of $L$ endowed with the ordering induced by the one of $L$. 

A real field $L$ has always characteristic zero, and an ordered field $L$ is a r.c.f.\ if and only if $L[\ii]$ is the algebraic closure of $L$, where $\ii:=\sqrt{-1}$. It is therefore not surprising that, to achieve our goal of studying the real case, we make extensive use of the complex case in which $L$ is an algebraically closed field of characteristic zero, an a.c.f.\ for short. For this reason, we also carefully analyze and present the complex case.

\vspace{3mm}
The starting question of this paper is:

\begin{itemize}[leftmargin=*]
\item[] \textit{Question A.} What is the foundational concept of algebraic geometry?
\end{itemize}

\noindent Complex and real algebraic geometers agree on the answer: \textsf{algebraic~set}.

\begin{itemize}[leftmargin=*]
\item[] \textit{Answer A in the complex case.} Quoting page 1 of the book `Algebraic Geometry~I.\ Complex Projective Varieties' \cite{mumford} by Mumford: ``The beginning of the whole subject is the following definition: A closed algebraic subset $X$ of $\C^n$ is the set of zeros of a finite set of polynomials''. The symbol $\C$ denotes the field of complex numbers as usual.
\end{itemize}

\begin{itemize}[leftmargin=*]
\item[] \textit{Answer A in the real case.} Quoting page 1 of the book `Real Algebraic Geometry' \cite{bcr} by Bochnak, Coste and Roy: ``In simplest terms, algebraic geometry is the study of the set of solutions of a system of polynomial equations. The main goal of real algebraic geometry is the study of real algebraic sets i.e.\ subsets of $\R^n$ defined by polynomial equations.''
\end{itemize}

\noindent Let us try to give a more precise answer. Let $n\in\N^*:=\N\setminus\{0\}$ and let $L$ be either an a.c.f.\ or a r.c.f.. Given sets $F\subset L[\x]:=L[\x_1,\ldots,\x_n]$ and $X\subset L^n$, define
\begin{align*}
\II_L(X)&:=\{f\in L[\x]: f(x)=0,\ \forall x\in X\},\\
\ZZ_L(F)&:=\{x\in L^n: f(x)=0,\ \forall f\in F\}.
\end{align*}
If $F:=\{f_1,\ldots,f_s\}\subset L[\x]$ is a finite set of polynomials, we write $\ZZ_L(f_1,\ldots,f_s):=\ZZ_L(F)$. The set $X\subset L^n$ is algebraic if $X=\ZZ_L(F)$ for some $F\subset L[\x]$: this is the standard concept of algebraic set.

\begin{itemize}[leftmargin=*]
\item[] \textit{Answer $A'$.} In more precise terms: algebraic geometry over $L$ is the study of those geometric properties of the algebraic sets $X\subset L^n$ that derive from the zero ideals $\II_L(X)$ of $L[\x]$.
\end{itemize}

\noindent Let $K$ be a subfield of $L$. Consider $K[\x]$ as a subset of $L[\x]$, and $K^n$ as a subset of $L^n$. Given a set $X\subset L^n$, define
$$
\II_K(X):=\{f\in K[\x]: f(x)=0,\ \forall x\in X\}.
$$
Observe that~$X$ is a subset of the whole $L^n$, but $\II_K(X)$ consists only of polynomials with coefficients in $K$ that vanish on~$X$, so it is an ideal of $K[\x]$. We say that the set $X\subset L^n$ is \emph{$K$-algebraic} if $X=\ZZ_L(F)$ for some $F\subset K[\x]$. Evidently, one can calculate $\II_K(X)$ from $\II_L(X)$ by simply writing $\II_K(X)=\II_L(X)\cap K[\x]$. Moreover, $X\subset L^n$ is a $K$-algebraic set if and only if $X=\ZZ_L(\II_K(X))$. Here we present some examples of $K$-algebraic sets in the cases $L|K=\C|\Q$ or $\R|\Q$:
\begin{itemize}
\item The singletons $\{\sqrt{2}\}\subset\C$, $\{\sqrt{2}\}\subset\R$ are not $\Q$-algebraic and the singleton $\{\sqrt[3]{2}\}\subset\C$ neither. Instead, $\{\sqrt[3]{2}\}=\ZZ_\R(\x_1^3-2)\subset\R$ is a $\Q$-algebraic set. 
\item The complex line $\ZZ_\C(\x_1-\sqrt[3]{2}\x_2)\subset\C^2$ is not $\Q$-algebraic, whereas the real line $\ZZ_\R(\x_1-\sqrt[3]{2}\x_2)=\ZZ_\R(\x_1^3-2\x_2^3)\subset\R^2$ is a $\Q$-algebraic set.
\item The planes $\ZZ_\C(\x_1+\sqrt{2}\x_2+\sqrt[4]{2}\x_3)\subset\C^3$ and $\ZZ_\R(\x_1+\sqrt{2}\x_2+\sqrt[4]{2}\x_3)\subset\R^3$ are not $\Q$-algebraic. We will prove this fact later using Galois theory.
\end{itemize}

\noindent Consider again the fixed extension of fields $L|K$, where $L$ is either an a.c.f.\ or a r.c.f..

\begin{itemize}[leftmargin=*]
\item[] \textit{Question B.} What do we mean by subfield-algebraic geometry, or more precisely by $K$-alge\-braic geometry over $L$, in short $L|K$-algebraic geometry?
\end{itemize}

\begin{itemize}[leftmargin=*]
\item[] \textit{Answer B.} $L|K$-algebraic geometry means the study of the geometric properties of the $K$-algebraic sets $X\subset L^n$ that derive from the zero ideals $\II_K(X)$ of $K[\x]$.
\end{itemize}

\noindent As an $L$-algebraic subset of $L^n$ is a usual algebraic subset of $L^n$, the $L|L$-algebraic geometry is the standard algebraic geometry over $L$. If $X\subset L^n$ is $K$-algebraic, then $X$ is also algebraic in $L^n$, so a natural question arises.

\begin{itemize}[leftmargin=*]
\item[] \textit{Question C.} Is there a direct link between the $L|K$-algebraic geometry of a $K$-algebraic set $X\subset L^n$ and the standard $L|L$-algebraic geometry of $X\subset L^n$? More concretely, is it true that $\II_L(X)$ is generated by $\II_K(X)$ in $L[\x]$?
\end{itemize}
The answers in the complex case and in the real case are different in general.
\begin{itemize}[leftmargin=*]
\item[] \textit{Positive answer C in the complex case.}\ A standard result from commutative algebra asserts that, if $\gta$ is a radical ideal of $K[\x]$, then $\gta L[\x]=\gta\otimes_KL$ is a radical ideal of $L[\x]$. Therefore, if $L$ is an a.c.f.\ and $X\subset L^n$ is a $K$-algebraic set, then Hilbert's Nullstellensatz \cite[pag.320-321]{hil} implies that $\II_L(X)=\II_K(X)L[\x]$, because $\II_L(X)=\II_L(\ZZ_L(\II_K(X)))=\II_L(\ZZ_L(\II_K(X)L[\x]))=
\sqrt{\II_K(X)L[\x]}=\II_K(X)L[\x]$. Consequently, if $L$ is an a.c.f., a subset $X$ of $L^n$ is the common zero set of a family of polynomials in $K[\x]$ if and only if $X\subset L^n$ is defined over $K$ in the sense that $\II_L(X)=\II_K(X)L[\x]$. 
\end{itemize}

\begin{itemize}[leftmargin=*]
\item[] \textit{A negative answer C in the real case.}\ If $L$ is a r.c.f.\ and $K=\Q$, there exist $K$-algebraic sets $X\subset L^n$ such that $\II_L(X)\neq\II_K(X)L[\x]$. For instance, if $X:=\{\sqrt[3]{2}\}\subset L$, then
$
\II_L(X)=(\x_1-\sqrt[3]{2})L[\x_1]\supsetneqq(\x_1^3-2)L[\x_1]=\II_K(X)L[\x_1].
$
Consequently, there exist subsets $X$ of $L^n$ that are the common zero sets of families of polynomials in $K[\x]$, but are not defined over $K$ in the sense that $\II_L(X)\neq\II_K(X)L[\x]$. This will be a crucial point along this work.
\end{itemize}
In the complex case, as we will see, the behavior of $K$-algebraic sets is tame and its description is clear, although not trivial. Consequently, we mainly focus on the real case.

Assume that $L$ is a r.c.f.\ and $K$ is an ordered subfield of $L$. If $X\subset L^n$ is $K$-algebraic, the intersection $X\cap K^n$ is algebraic in $K^n$, and a further natural question arises.
\begin{itemize}[leftmargin=*]
\item[] \textit{Question D.} Is there a direct link between the $L|K$-algebraic geometry of a $K$-algebraic set $X\subset L^n$ and the standard $K|K$-algebraic geometry of $X\cap K^n\subset K^n$? More concretely, is it possible to compute $\II_K(X)\subset K[\x]$ from $\II_K(X\cap K^n)\subset K[\x]$?
\end{itemize}
This is a very deep question that leads to Diophantine geometry and has a negative answer in general.
\begin{itemize}[leftmargin=*]
\item[] \textit{A negative answer D.} Suppose that $K=\Q$. For each $h\geq 3$, consider the Fermat curve $F_h:=\ZZ_L(\x_1^{2h}+\x_2^{2h}-2^h)\subset L^2$. Fermat's Last Theorem of Wiles \cite{wi} implies that $F_h\cap\Q^2=\varnothing$, so $F_h\cap\Q^2$ provides no information about $F_h$. More generally, Faltings' solution of Mordell's conjecture \cite{fa83,fwgss} states that, if $H$ is a curve of $L^n$ of genus $\geq2$ that is a $\Q$-algebraic set, then $\ZZ_L(H)$ is a finite set for each finitely generated extension $L|\Q$.
\end{itemize}
Actually, there are remarkable situations where the answer is positive. 
\begin{itemize}[leftmargin=*]
\item[] \textit{A positive answer D.} If $L$ and $K$ are both r.c.f., then $\II_K(X)=\II_K(X\cap K^n)$ (and $\II_L(X)=\II_K(X)L[\x]$) by the `exten\-sion of coefficients' procedure due to Tarski-Seidenberg's principle. (The same is true if $L$ and $K$ are both a.c.f.).
\end{itemize}
Also in this case, that is, if $L$ and $K$ are both r.c.f., the behavior of $K$-algebraic subsets of~$L^n$ is tame and its description is clear, although not trivial. The latter considerations suggest focusing on the real case in which $L$ is a r.c.f.\ whereas $K$ is not: for instance, $L|K=\R|\Q$ or $\qr|\Q$. We will see that in this situation the geometry of $K$-algebraic sets is rich in many new phenomena.

The main purpose of this paper is to introduce and study the foundational concepts of $K$-Zariski topology of $L^n$, $K$-dimension $\dim_K(X)$ of a $K$-algebraic set $X\subset L^n$, and $E|K$-nonsingular and $E|K$-singular points of the $K$-algebraic set $X\subset L^n$ for every extension of fields $L|E|K$, and to compare them with the corresponding concepts of Zariski topology, dimension and nonsingular and singular points in the standard context of the $L|L$-algebraic and $K|K$-algebraic geometries. We include numerous explicit examples and calculations. Finally, we present four examples of how the developed theory can be applied.

\vspace{3mm}

The paper is organized into six sections and three appendices. Below we provide an overview of the contents.

\vspace{3mm}
\textsc{Section \ref{s1}} provides basic definitions, constructions and results concerning $K$-algebraic sets. A subset of $L^n$ is $K$-algebraic if it is the common zero set in $L^n$ of a family of polynomials in $K[\x]$. The $K$-algebraic subsets of $L^n$ are the closed sets of a topology of $L^n$ that we call $K$-Zariski topology. Using this topology, we introduce the concepts of $K$-Zariski closure in $L^n$, and $K$-irreducible, $K$-Zariski locally closed and $K$-constructible subsets of $L^n$. The $K$-Zariski topology is coarser than the usual ($L$-)Zariski topology of $L^n$, so it is Noetherian and we can speak about the $K$-irreducible components of any $K$-algebraic subset of $L^n$. Given a $K$-algebraic set $X\subset L^n$, we define the $K$-dimension $\dim_K(X)$ of $X\subset L^n$ as the Krull dimension of the ring $K[\x]/\II_K(X)$. We study the main properties of these concepts and their interactions.

Noether's normalization theorem applied to $\II_K(X)$ allows us to prove that $\dim_L(X)\leq\dim_K(X)$ and the equality holds when $\II_L(X)=\II_K(X)L[\x]$, for instance when $L$ is an a.c.f.. Assuming that $X\subset L^n$ is the usual ($L$-)Zariski closure of a usual ($K$-)algebraic set $Y\subset K^n$, we investigate the relationships between vanishing ideals, irreducible components and dimensions of $X\subset L^n$ and $Y\subset K^n$. These relationships are particularly strong when $L$ and $K$ are both a.c.f.\ or r.c.f., because `extension of coefficients' procedure works by Tarski-Seidenberg's principle (see Appendix \ref{appendix}). In fact, in this case, we have $\II_L(X)=\II_K(X)L[\x]$ and $\II_K(X)=\II_K(X\cap K^n)$.

An important step in the development of the theory is to understand the behavior of the $K$-Zariski closure of a complex algebraic set. Suppose for a while that $L$ is an a.c.f., $\kbar$ is the algebraic closure of $K$ in $L$, $X$ is an arbitrary ($L$-)algebraic subset of $L^n$, and $T$ is the $K$-Zariski closure of $X$ in $L^n$, that is, $T=\ZZ_L(\II_K(X))$. If $X\subset L^n$ is not $\kbar$-algebraic, we have no control over $T$: if $X\subset L^n$ is in addition irreducible, the `transcendence of $X\subset L^n$ over $K$' is revealed by the strict inequality $\dim_L(X)<\dim_L(T)$. 

If $X\subset L^n$ is $\kbar$-algebraic, then:
\begin{itemize}
\item We define the Galois completion of $X\subset L^n$ as the smallest subset of $L^n$ that contains $X$ and is invariant under the action of the Galois group $G:=G(L:K)$, that is, the set $\bigcup_{\psi\in G}\psi_n(X)$, where $\psi_n:L^n\to L^n$ is the bijection $\psi_n(z_1,\ldots,z_n):=(\psi(z_1),\ldots,\psi(z_n))$ associated to each $\psi\in G$. It holds $T=\bigcup_{\psi\in G}\psi_n(X)$.
\item We provide an algorithm to compute $T$ and $\II_K(T)$, starting from any finite system of polynomial equations of $X\subset L^n$ over $\kbar$. We involve any finite Galois sub\-extension of $\kbar|K$ containing all the coefficients of the previous equations.
\item We deduce that $\dim_L(X)=\dim_L(T)=\dim_K(T)=\dim_K(X)$. 
\item We analyze in more detail the case of hypersurfaces.
\end{itemize}

Some relevant consequences for complex $K$-algebraic sets are the following (recall that $L$ is assumed to be an a.c.f.\ at this time): 
\begin{itemize}
\item A $\kbar$-algebraic subset of $L^n$ is $K$-algebraic if and only if it is invariant under the action of the Galois group $G=G(L:K)$.
\item All the $L$-irreducible components of a $K$-irreducible $K$-algebraic subset of $L^n$ have the same dimension and are provided by the algorithm mentioned above.
\item If $X\subset L^n$ is a $K$-irreducible $K$-algebraic set, $Y_1,\ldots,Y_s$ are its $L$-irreducible components, $E$ is any intermediate field (that is, $K\subset E\subset L$) and $W_1,\ldots,W_t$ are the $E$-irreducible components of $X\subset L^n$, then the following complex clustering phenomenon occurs: $t\leq s$ and there exists a partition of $\{1,\ldots,s\}=J_1\sqcup\cdots\sqcup J_t$ such that $\{Y_j:j\in J_i\}$ is the family of $L$-irreducible components of $W_i\subset L^n$ for each $i\in\{1,\ldots,t\}$.
\item As we said, the invariance of dimension works when $L$ is an a.c.f.: if $X\subset L^n$ is $K$-algebraic, then $\dim_L(X)=\dim_K(X)$. This invariance of dimension remains still valid in the case when $L$ is a r.c.f., even if the equality $\II_L(X)=\II_K(X)L[\x]$ does not hold. Consequently, when $L$ is either an a.c.f.\ or a r.c.f., we can speak about the dimension of any $K$-algebraic set $X\subset L^n$ without referring to $K$. Such invariance of dimension is a tool of crucial importance for developing the theory presented in this paper.

The invariance of dimension does not always work. Using Faltings' theorem, we present explicit examples of extensions of fields $L|K$ and $K$-algebraic sets $X\subset L^n$ such that $\dim_L(X)<\dim_K(X)$.
\item We revisit the concept of complexification using the algorithm mentioned above. 
\end{itemize} 

This section concludes with the projective case. We introduce projective $K$-algebraic sets, and present $L|K$-versions of both the main theorem of elimination theory and Chevalley's projection theorem. The proofs of these results highlight some of the difficulties that can be encountered when trying to generalize a classical $L|L$-result to the $L|K$-context.

\vspace{3mm}
\textsc{Section \ref{s2}} focuses on real $K$-algebraic sets. Suppose now that $L=R$ is a real closed field and $K$ is a subfield of $R$. Denote $\kr$ the algebraic closure of $K$ in~$R$, $\kbar:=\kr[\ii]$ the algebraic closure of $K$ and $C:=R[\ii]$ the algebraic closure of $R$. We introduce and study the concepts of complex and real Galois completions of a $\kr$-algebraic set $Y\subset R^n$. Let $Z\subset C^n$ be the complexification of $Y\subset R^n$, that is, the Zariski closure of $Y$ in $C^n$. We define the complex Galois completion $T$ of $Y$ as the Galois completion of $Z$, and the real Galois completion $T^r$ of $Y$ as the intersection of $T$ with $R^n$, that is, $T=\bigcup_{\psi\in G}\psi_n(Z)$ and $T^r:=\bigcup_{\psi\in G}(\psi_n(Z)\cap R^n)$, where $G$ is the Galois group $G(C:K)$. We prove that $T$ and $T^r$ are the $K$-Zariski closures of $Y$ in $C^n$ and in $R^n$, respectively. The complexification of $T^r\subset R^n$ is contained in $T$, however this inclusion may be strict. We describe an algorithm to compute $T$, $T^r$ and $\II_K(T)=\II_K(T^r)$, beginning from any finite system of generators of $\II_\kr(Y)$ in $\kr[\x]$ and any finite Galois subextension of $\kbar|K$ containing all the coefficients of such generators.

These results have several consequences for real $K$-algebraic sets. We summarize next the most relevant ones:
\begin{itemize}
\item The $\kr$-algebraic subset $Y$ of $R^n$ is $K$-alge\-braic if and only if it is $G$-stable, that is, if $\psi_n(Z)\cap R^n\subset Y$ for all $\psi\in G=G(C:K)$.
\item Let $X\subset R^n$ be a $K$-irreducible $K$-algebraic set of dimension $d$. Choose an ($R$-)irreducible component $Y$ of $X$ of its same dimension $d$, and denote again $Z\subset C^n$ the complexification of $Y\subset R^n$. Then $Y\subset R^n$ is a $\kr$-irreducible component of $X$, whose complex Galois completion $T=\bigcup_{\psi\in G}\psi_n(Z)$ is equal to the $K$-Zariski closure of $X$ in $C^n$ and whose real Galois completion $T^r=\bigcup_{\psi\in G}(\psi_n(Z)\cap R^n)$ coincides with $X$. In addition, $T\subset C^n$ is $K$-irreducible, $\{\psi_n(Z)\}_{\psi\in G}$ is the family of all $C$-irreducible components of $T\subset C^n$ and each $\psi_n(Z)\subset C^n$ has ($C$-)dimension equal to $d$. Moreover, $\dim(\psi_n(Z)\cap R^n)\leq d$ for all $\psi\in G$ and this inequality can be strict for some $\psi$. The family of intersections $\psi_n(Z)\cap R^n$ having dimension $d$ coincides with that of all $R$-irreducible components of $X\subset R^n$ of dimension $d$, whereas it may happen that $X\subset R^n$ has $R$-irreducible components of dimension strictly smaller than $d$. We define the $K$-bad set of $X$ as the union of all the intersections $\psi_n(Z)\cap R^n$ of dimension strictly smaller than $d$. See Figure \ref{im:poly}.
\item We extend the previous results and the definition of $K$-bad set to arbitrary real $K$-algebraic sets, possibly $K$-reducible.
\item We study the clustering phenomenon in the real case.
\item We deal with the special case of real hypersurfaces separately. We introduce and study the concepts of $K$-geometric polynomial and $K$-geometric hypersurface. A polynomial $f\in K[\x]$ is $K$-geometric in $R^n$ if $\II_K(\ZZ_R(f))=(f)K[\x]$. The zero set of a $K$-geometric polynomial in $R^n$ is called $K$-geometric hypersurface of $R^n$. An enlightening example is the following: the polynomial $\x_1^3-2\in\Q[\x_1]$ is $\Q$-geometric but not $R$-geometric in~$R$, whereas its zero set $\{\sqrt[3]{2}\}$ is both a $\Q$-geometric and $R$-geometric hypersurface of $R$, because $\{\sqrt[3]{2}\}$ is also the zero set of the $R$-geometric polynomial $\x_1-\sqrt[3]{2}\in R[\x_1]$ in $R$.
 
Given a $\kr$-geometric polynomial $g\in\kr[\x]$ in $R^n$, we simplify the algorithm mentioned above to compute the complex and real Galois completions of $\ZZ_R(g)\subset R^n$, and analyze the different phenomena that appear.
\end{itemize}

Let us give the idea of how the algorithm works on an explicit example. 
\begin{example}
Let $K:=\Q$, let $R$ be a real closed field, let $g:=\x_1+\sqrt{2}\x_2+\sqrt[4]{2}\x_3\in\qr[\x]:=\qr[\x_1,\x_2,\x_3]$, let $Y:=\ZZ_R(g)\subset R^3$ and let $Z\subset C^3$ be the complexification of $Y$. Observe that $\II_\qr(Y)=(g)\qr[\x]$, so $g$ is $\qr$-geometric in $R^3$ and $Z=\ZZ_C(g)$. Consider the finite Galois subextension $E:=\Q[\sqrt[4]{2},\ii]|\Q$ of $\qbar|\Q$ and the Galois group $G':=G(E:\Q)$. The group $G'$~acts on the coefficients of $g$ providing (twice) the following four distinct polynomials:
\begin{align*}
g_0&:=g, \quad g_1:=\x_1-\sqrt{2}\x_2+\ii\sqrt[4]{2}\x_3,\\
g_2&:=\x_1+\sqrt{2}\x_2-\sqrt[4]{2}\x_3, \quad g_3:=\x_1-\sqrt{2}\x_2-\ii\sqrt[4]{2}\x_3.
\end{align*}
Define:
\begin{align*}
&Z_a:=\ZZ_C(g_a)\subset C^3\;\text{ and }\;X_a:=Z_a\cap R^3\;\text{ for each }a\in\{0,1,2,3\},\\
&\textstyle T:=\bigcup_{a=0}^3Z_a, \quad X:=\bigcup_{a=0}^3X_a, \quad g^\bullet:=g_0g_1g_2g_3\in\Q[\x].
\end{align*}
Observe that $X_0=Y$ and $Z_0=Z$. We have:
\begin{itemize}
\item $T$ is the complex Galois completion of $Y$, that is, the $\Q$-Zariski closure of $Y$ in $C^3$.
\item $T\subset C^3$ is $\Q$-irreducible but $C$-reducible and its $C$-irreducible components are the four distinct planes $Z_0,Z_1,Z_2,Z_3$ of $C^3$. 
\item $X_0$ and $X_2$ are two distinct planes of $R^3$. $X_1$ and $X_3$ are both equal to the line of $R^3$ with equations $\x_1-\sqrt{2}\x_2=0,\x_3=0$. Observe that $g_1g_3\in R[\x]$ and $\ZZ_R(g_1g_3)=X_1=X_3$.
\item $X$ is the real Galois completion of $Y$, that is, the $\Q$-Zariski closure of $Y$ in $R^3$.
\item $X\subset R^3$ is $\Q$-irreducible but $R$-reducible. The planes $X_0,X_2$ and the line $X_1=X_3$ of $R^3$ are its $R$-irreducible components. The $\Q$-bad set of $X$ is equal to the line $X_1=X_3$. See Figure \ref{im:poly}.
\item $T=\ZZ_C(g^\bullet)$ and $X=\ZZ_R(g^\bullet)$. The polynomial $g^\bullet$ is reducible in $R[\x]$ and $g^\bullet=g_0g_2(g_1g_3)$ is its factorization in $R[\x]$. In addition, $g^\bullet$ is irreducible in $\Q[\x]$, $\Q$-geometric in $R^3$ but not $R$-geometric in $R^3$. The algebraic set $X$ is a $\Q$-geometric but not $R$-geometric hypersurface of $R^3$.
\item The complexification of the line $X_1=X_3$ of $R^3$ is strictly contained in the plane $Z_a$ of $C^3$ for $a\in\{1,3\}$. Thus, the complexification of $X$ is strictly contained in~$T$. $\sqbullet$
\end{itemize}
\end{example}

The previous example is part of Examples \ref{exa:gc}, where slightly different notations are used and further phenomena are shown.

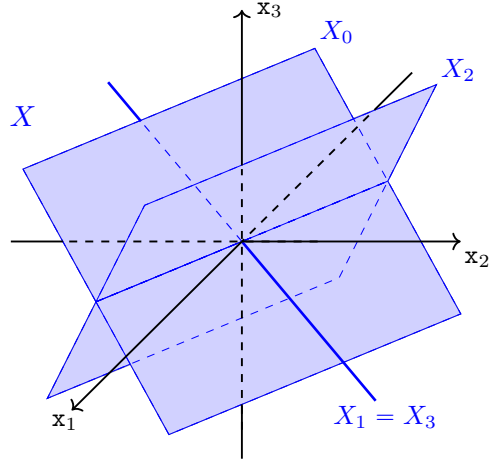
\begin{figure}[!ht]
\begin{center}
\begin{tikzpicture}[scale=0.32]

\draw[fill=blue!60,opacity=0.30,draw] (-6,-2.5) -- (-8,-6.5) -- (-4.58,-5.1) -- (-3,-8) -- (9,-3) -- (6,2.5) -- (8,6.5) -- (4.58,5.1) -- (3,8) -- (-9,3) -- (-6,-2.5);

\draw[blue,line width=0.01mm,dashed] (-6,-2.5) -- (6,2.5) -- (8,6.5) -- (-4,1.5) -- (-6,-2.5);
\draw[blue,line width=0.01mm,dashed] (-6,-2.5) -- (6,2.5) -- (4,-1.5) -- (-8,-6.5) -- (-6,-2.5);
\draw[blue,line width=0.01mm,dashed] (-6,-2.5) -- (-3,-8) -- (9,-3) -- (6,2.5) -- (-6,-2.5);
\draw[blue,line width=0.01mm,dashed] (-6,-2.5) -- (-9,3) -- (3,8) -- (6,2.5) -- (-6,-2.5);

\draw[blue,line width=0.05mm] (6,2.5) -- (9,-3) -- (-3,-8) -- (-9,3) -- (3,8) -- (4.58,5.1);
\draw[blue,line width=0.05mm] (6,2.5) -- (8,6.5) -- (-4,1.5) -- (-6,-2.5) -- (0,0);
\draw[blue,line width=0.05mm] (-6,-2.5) -- (-8,-6.5) -- (-4.6,-5.085);
\draw[blue,line width=0.05mm] (-6,-2.5) -- (6,2.5);

\draw[blue,line width=0.35mm] (-4.15,4.98) -- (-5.5,6.6);
\draw[blue,line width=0.15mm,dashed] (0,0) -- (-4.2,5.04);
\draw[blue,line width=0.35mm] (5.5,-6.6) -- (0,0);

\draw[line width=0.25mm] (-9.5,0) -- (-7.35,0);
\draw[line width=0.25mm,dashed] (-7.1,0) -- (3.375,0);
\draw[->,line width=0.25mm] (0,0) -- (9,0);
\draw[line width=0.25mm] (0,-9) -- (0,-6.75);
\draw[line width=0.25mm,dashed] (0,-7.8) -- (0,3);
\draw[->,line width=0.25mm] (0,3.15) -- (0,9.6);
\draw[line width=0.25mm] (7,7) -- (5.45,5.45);
\draw[line width=0.25mm,dashed] (5.22,5.22) -- (0,0);
\draw[->,line width=0.25mm] (0,0) -- (-7,-7);

\draw (-7,-7.5) node{\small$\!\!\!\x_1$};
\draw (9.45,-0.75) node{\small$\;\,\x_2$};
\draw (0.75,9.6) node{\small$\;\;\,\x_3$};
\draw (3.9,8.7) node{\small\color{blue}$X_0$};
\draw (5.9,-7.28) node{\small\color{blue}$X_1=X_3$};
\draw (8.9,7.1) node{\small\color{blue}$X_2$};
\draw (-9,5.2) node{\color{blue}$X$};
\end{tikzpicture}
\end{center}
\caption{The real Galois completion $X$ of the $\qr$-algebraic plane $Y=X_0$ of $R^3$ together with its three $R$-irreducible components: the planes $X_0,X_2$ and the line $X_1=X_3$ of $R^3$ ($X\subset R^3$ is $\Q$-irreducible, coincides with the $\Q$-Zariski closure of $Y$ in $R^3$ and has the line $X_1=X_3$ as its $\Q$-bad set).}\label{im:poly}
\end{figure}

In Section \ref{s2}, we also introduce and study the notions of $K$-reliable family of polynomials and of real algebraic set defined over $K$. A finite subset $\{f_1,\ldots,f_r\}$ of $K[\x]$ is said to be $K$-reliable if $\II_F(\ZZ_F(f_1,\ldots,f_r))=(f_1,\ldots,f_r)F[\x]$ for all real closed field $F$ that contains $K$ as an ordered subfield, that is, $F\supset K$ and the ordering of $F$ extends that of $K$. A polynomial $f\in K[\x]$ is $K$-reliable if so is $\{f\}\subset K[\x]$ or, equivalently, if $f$ is $F$-geometric in $F^n$ for all real closed fields $F$ containing $K$ as an ordered subfield. Fix a real closed field $R$ such that $K$ is an ordered subfield of $R$. An algebraic set $X\subset R^n$ is defined over $K$ if $\II_R(X)=\II_K(X)R[\x]$. Observe that, if $X\subset R^n$ is defined over $K$, then it is also $K$-algebraic, because $X=\ZZ_R(\II_K(X))$.

We highlight the following results:
\begin{itemize}
\item A polynomial in $K[\x]$ is $K$-reliable if and only if it is $F$-geometric in $F^n$ for at least one real closed field $F$ that contains $K$ as an ordered subfield, for instance $F=R$.
\item If a polynomial in $K[\x]$ is $K$-reliable, then it is also $K$-geometric in $R^n$. If $K$ is a real closed field, then the `extension of coefficients' procedure implies that the inverse implication is also true, so a polynomial in $K[\x]$ is $K$-reliable if and only if it is $K$-geometric in $R^n$.
 
If $K:=\Q$, the polynomial $\x_1^3-2\in\Q[\x_1]$ is $\Q$-geometric in $R$, but not $\Q$-reliable. The polynomial $g^\bullet\in\Q[\x_1,\x_2,\x_3]$ mentioned above is also $\Q$-geometric in~$R^3$, but not $\Q$-reliable. Thus, the singleton $\{\sqrt[3]{2}\}\subset R$ and the set $X=\ZZ_R(g^\bullet)\subset R^3$ of Figure \ref{im:poly} are examples of $\Q$-geometric hypersurfaces, that are algebraic sets not defined over $\Q$.

If $h\geq 3$, the Fermat curve $F_h:=\ZZ_R(\x_1^{2h}+\x_2^{2h}-2^{h})$ is an example of an irreducible geometric hypersurface of $R^2$ defined over $\Q$ without rational points.
\item Given a $K$-algebraic set $X\subset R^n$, the following equivalences hold: $X$ is defined over $K$ $\Leftrightarrow$ $\II_R(X)$ is generated in $R[\x]$ by a finite set of polynomials in $K[\x]$ $\Leftrightarrow$ the $K$-Zariski closure of $X$ in $C^n$ coincides with the complexification of $X\subset R^n$ $\Leftrightarrow$ $X=\ZZ_R(f_1,\ldots,f_r)$ for some $K$-reliable family $\{f_1,\ldots,f_r\}\subset K[\x]$ $\Leftrightarrow$ all the $K$-irreducible components of $X$ are defined over $K$. Moreover, if $X\subset R^n$ is $K$-irreducible and defined over $K$, then its $K$-bad set is empty. The singleton $\{\sqrt[3]{2}\}\subset R$ is an example of $\Q$-irreducible $\Q$-algebraic set with empty $\Q$-bad set, not defined over $\Q$.
\end{itemize} 

We conclude the section by studying the underlying real structure of a complex algebraic set. Given an algebraic set $X\subset C^n$, the underline real structure $X^R\subset R^{2n}=C^n$ of~$X$ is $X^R:=\{(x,y)\in R^n\times R^n=R^{2n}:x+\ii y\in X\}$. Observe that $X^R\subset R^{2n}$ is $K$-algebraic, provided $X\subset C^n$ is $K[\ii]$-algebraic. We prove some basic properties of the underline real structures and provide explicit examples of complex $\Q$-algebraic sets whose underline real structures are ($\Q$-algebraic but) not defined over $\Q$.

The Real Nullstellensatz plays a key role in this second section.

The previous results show some similarities between $L|K$-algebraic geometry (when $L$ is either an a.c.f.\ or a r.c.f.) and complex and real analytic geometry. For instance, we analogize the good behavior of complex $K$-algebraic sets with the coherence of complex analytic sets. In the real analytic case, coherence is not guaranteed, and there are two kinds of (global) real analytic sets: those that are the zero set of a coherent sheaf, that we analogize with the real $K$-algebraic sets, and those that are coherent, which we analogize with the real algebraic sets defined over~$K$.

\vspace{3mm}

In \textsc{Section \ref{s3}}, we define and study the concepts of local ring, Zariski tangent space and differential of a regular map in the context of $L|K$-algebraic geometry. The concepts and results presented here will be used in the subsequent two sections.

Consider again the extension of fields $L|K$, where $L$ is either an a.c.f.\ or a r.c.f.. Let $E$ be any subfield of $L$ containing $K$, and let $\ove^\sqbullet$ be the algebraic closure of $E$ in $L$. Pick any point $a:=(a_1,\ldots,a_n)\in L^n$. Define the ideal $\gtn_a:=\II_E(\{a\})=\{f\in E[\x]:f(a)=0\}$ of $E[\x]$, and the subring $E[a]:=E[a_1,\ldots,a_n]$ of $L$. The first question is to understand when $\gtn_a$ is maximal or, equivalently, when the ring $E[a]\cong E[\x]/\gtn_a$ is a field. Standard commutative algebra provides the solution: the ideal $\gtn_a$ of $E[\x]$ is maximal if and only if $a\in(\ove^\sqbullet)^n$.

Consider a $K$-algebraic set $X\subset L^n$ and suppose that $a\in X\cap(\ove^\sqbullet)^n$. We define the $E|K$-local ring $\reg^{E|K}_{X,a}$ of $X$ at $a$ as the local ring $\reg^{E|K}_{X,a}:=E[\x]_{\gtn_a}/(\II_K(X)E[\x]_{\gtn_a})$. Of course, if $L=E=K$, then $\ove^\sqbullet=L$, $\reg^{E|K}_{X,a}$ coincides with the usual local ring of an algebraic set $X\subset L^n$ at $a$, and $a$ is an arbitrary point of $X=X\cap L^n$. The $E|K$-local rings of $X\subset L^n$ are invariant under the action of the Galois group $G(L:E)$ when $L$ is an a.c.f.\ and under the action of the Galois group $G(L[\ii]:E)$ when $L$ is a r.c.f.. More precisely, if $a\in X\cap(\ove^\sqbullet)^n$ and $b$ is any point of the $E$-Zariski closure of $\{a\}$ in $L^n$, then $b\in X\cap(\ove^\sqbullet)^n$ and $\gtn_b=\gtn_a$, so $\reg^{E|K}_{X,b}=\reg^{E|K}_{X,a}$.

Let $a:=(a_1,\ldots,a_n)\in(\ove^\sqbullet)^n$. For each $i\in\{1,\ldots,n\}$, the subring $E[a_1,\ldots,a_{i-1}]$ of $L$ coincides with the subfield $E(a_1,\ldots,a_{i-1})$ of $L$ genera\-ted by $\{a_1,\ldots,a_{i-1}\}$ over $E$, where $E(a_1,\ldots,a_{i-1})$ denotes $E$ when $i=1$. Thus, we can find a polynomial $f_i\in E[\x_1,\ldots,\x_i]$ that is monic with respect to $\x_i$ and such that $f_i(a_1,\ldots,a_{i-1},\x_i)$ is the minimum polynomial of $a_i$ over $E[a_1,\ldots,a_{i-1}]=E(a_1,\ldots,a_{i-1})$. It holds:
\begin{itemize}
\item $E[\x]_{\gtn_a}$ is a regular local ring of dimension $n$, and the polynomials $f_1,\ldots,f_n$ constitute one of its regular systems of parameters.
\item The residue field of $E[\x]_{\gtn_a}$ is isomorphic to $E[a]$ via the isomorphism $\frac{f}{g}+\gtn_aE[\x]_{\gtn_a}\mapsto f(a)(g(a))^{-1}$.
\item The completion of $E[\x]_{\gtn_a}$ is isomorphic to $E[a][[\x]]$ via an isomorphism $\varphi$ sending each $f_i$ to $\x_i$.
\item If $\gta$ is an ideal of $K[\x]$ contained in $\gtn_a=\II_E(\{a\})$, the completion of $E[\x]_{\gtn_a}/(\gta E[\x]_{\gtn_a})$ is isomorphic to $E[a][[\x]]/(\varphi(\gta)E[a][[\x]])$.
\end{itemize} 

Pick a point $a\in X\cap(\ove^\sqbullet)^n$. We define the $E|K$-Zariski tangent space $T^{E|K}_a(X)$ of $X$ at $a$ as $T^{E|K}_a(X):=\{v\in E[a]^n:\langle\nabla g(a),v\rangle=0,\ \forall g\in\II_K(X)\}$, where $\langle\nabla g(a),v\rangle:=\sum_{i=1}^n\frac{\partial g}{\partial \x_i}(a)v_j$ if $v=(v_1,\ldots,v_n)$. Let $\gtM_a:=\gtn_a E[\x]_{\gtn_a}/(\II_K(X)E[\x]_{\gtn_a})$ be the maximal ideal of $\reg^{E|K}_{X,a}$. Using the regular system of parameters $\{f_1,\ldots,f_n\}$ of $E[\x]_{\gtn_a}$, one shows that $\gtM_a/\gtM_a^2$ and the dual of $T^{E|K}_a(X)$ are isomorphic as $E[a]$-vector spaces. 

Let $Y\subset L^m$ be another $K$-algebraic set and let $f:X\to Y$ be a $K$-regular map at $a$, that is, $f=\big(\frac{p_1}{q},\ldots,\frac{p_m}{q}\big)$ on a $K$-Zariski open neighborhood of $a$ in $X$ for some $p_1,\ldots,p_m,q\in K[\x]$. The next goal is to define the notion of $E|K$-differential of $f$ at $a$. Let $v\in T^{E|K}_a(X)$. We consistently define $J_a(f)v$ as the column vector in $E[a]^n$ whose $i^{\mr{th}}$-entry is $q(a)^{-1}\langle\nabla p_i(a),v\rangle$, and prove that $\langle\nabla h(f(a)),J_a(f)v\rangle=0$ in $E[a]$ for all $h\in\II_K(Y)$. Observe that $E[f(a)]$ is a subfield of $E[a]$. It may happen that $E[f(a)]$ is strictly contained in $E[a]$, so $J_a(f)v$ may not belong to $T^{E|K}_{f(a)}(Y)$. This phenomenon appears for instance in the following example: If $E=K:=\Q$, $X:=\ZZ_L(\x_1-\x_2^2)\subset L^2$, $Y:=L$, $f:X\to Y$ is the projection $(x_1,x_2)\mapsto x_1$ and $a:=(\sqrt{2},\sqrt[4]{2})\in X\cap(\qbar^\sqbullet)^2$, then the column vector $v:=(2\sqrt[4]{2},1)^T$ (here the exponent $^T$ indicates the transpose operation) belongs to $T^{\Q|\Q}_a(X)\subset(\Q[\sqrt[4]{2}])^2$, but $J_a\big(\frac{\x_1}{1}\big)v=2\sqrt[4]{2}\not\in\Q[\sqrt{2}]=T^{\Q|\Q}_{f(a)}(Y)$.

Under the additional condition that $E[f(a)]=E[a]$, we define the $E|K$-differential $d^{E|K}_af:T^{E|K}_a(X)\to T^{E|K}_{f(a)}Y$ of $f$ at $a$ as the $E[a]$-linear map $v\mapsto J_a(f)v$, and prove that the dual map of $d^{E|K}_af$ corresponds to the pullback map $\gtM_{f(a)}/\gtM_{f(a)}^2\to\gtM_a/\gtM_a^2$ induced by $f$, as in the standard case when $L=E=K$. 

\vspace{3mm}

\textsc{Section \ref{s4}} is devoted to $E|K$-nonsingular and $E|K$-singular points of a $K$-algebraic set. As in the previous section, $L|K$ is an extension of fields such that $L$ is either an a.c.f.\ or a r.c.f., $E$ is any subfield of $L$ containing $K$, and $\ove^\sqbullet$ is the algebraic closure of $E$ in $L$.

Let $X\subset L^n$ be a $K$-algebraic set and let $a\in X\cap(\ove^\sqbullet)^n$. We say that $a$ is an $E|K$-nonsingular point of $X$ of dimension $e$ if the local ring $\reg^{E|K}_{X,a}$ is regular and of dimension $e$. Denote $\Reg^{E|K}(X,e)$ the set of all $E|K$-nonsingular points of $X$ of dimension $e$. An $E|K$-nonsingular point of $X\subset L^n$ of dimension $\dim(X)$ is simply called $E|K$-nonsingular point of $X\subset L^n$, and the set of all such points is also denoted $\Reg^{E|K}(X)$. We define the set $\Sing^{E|K}(X):=(X\cap(\ove^\sqbullet)^n)\setminus\Reg^{E|K}(X)$ and call $E|K$-singular points of $X\subset L^n$ the points of $\Sing^{E|K}(X)$. If $L=E=K$, then any point $a$ of the algebraic set $X\subset L^n$ is either a $L|L$-nonsingular or $L|L$-singular point of $X\subset L^n$ if and only if $a$ is either a nonsingular or singular point of $X\subset L^n$ in the usual sense, respectively. Thus, $\Reg^{L|L}(X)$ coincides with the set $\Reg(X)$ of usual nonsingular points of $X\subset L^n$, whereas $\Sing^{L|L}(X)$ coincides with the set $\Sing(X)$ of usual singular points of $X\subset L^n$.

Let $a\in(\ove^\sqbullet)^n$ and let $\gta$ be an ideal of $K[\x]$ such that $\gta\subset\gtn_a:=\{f\in E[\x]:f(a)=0\}$. We define the rank ${\rm rk}_a(\gta)$ of $\gta$ at $a$ as the dimension of the $E[a]$-vector subspace of $E[a]^n$ generated by the set of gradients $\{\nabla g(a)\}_{g\in\gta}$. Define $r:={\rm rk}_a(\gta)$ and $A:=E[\x]_{\gtn_a}/(\gta E[\x]_{\gtn_a})$.

Combining some results of the previous sections and some classical theorems of commutative algebra concerning regular local rings (suitably adapted to the present $E|K$-context), we prove the following basic facts:
\begin{itemize}
\item The rank ${\rm rk}_a(\gta)$ is equal to the dimension of the $E[a]$-vector space $(\gta E[a]_{\gtn_a}+\gtN_a^2)/\gtN_a^2$, where $\gtN_a=\gtn_aE[\x]_{\gtn_a}$ is the maximal ideal of $E[\x]_{\gtn_a}$.
\item The local ring $A$ is regular $\Leftrightarrow$ there exist $r$ elements $g_1,\ldots,g_r$ of $\gta$ such that $\gta E[\x]_{\gtn_a}=(g_1,\ldots,g_r)E[\x]_{\gtn_a}$ $\Leftrightarrow$ $\dim(A)=n-r$. Moreover, if $A$ is not regular, the minimal cardinality of a system of generators of $\gta E[\x]_{\gtn_a}$ in $E[\x]_{\gtn_a}$ is $>r$ and $\dim(A)<n-r$.
\item The following $E|K$-Jacobian criterion holds: the local ring $A$ is regular of dimension $e$ $\Leftrightarrow$ there exist polynomials $f_1,\ldots,f_{n-e}\in\gta$ and a $E$-Zariski open (or a Euclidean open) neighborhood $U$ of $a$ in $(\ove^\sqbullet)^n$ such that the gradients $\nabla f_1(a),\ldots,\nabla f_{n-e}(a)$ are linearly independent in $E[a]^n$ and $\ZZ_{\ove^\sqbullet}(\gta)\cap U=\ZZ_{\ove^\sqbullet}(f_1,\ldots,f_{n-e})\cap U$.
\item Suppose that $\gta$ is a prime ideal of $K[\x]$. Let $Y$ be the zero set $\ZZ_{\ove^\sqbullet}(\gta)$ of $\gta$ in $(\ove^\sqbullet)^n$ and let $S$ be the set of all points $b\in Y$ such that the local ring $E[\x]_{\gtn_b}/(\gta E[\x]_{\gtn_b})$ is not regular. Then $\II_K(Y)=\gta$ if and only if $S\subsetneqq Y$. Moreover, if $\II_K(Y)=\gta$, then $\dim(\reg^{E|K}_{Y,b})=\dim(Y)$ for each $b\in Y$, and $S=\{b\in Y:{\rm rk}_b(\gta)<n-\dim(Y)\}$ is a $K$-algebraic subset of $Y$ of dimension $<\dim(Y)$.
\end{itemize}

These results allow us to study the structure of $E|K$-nonsingular and $E|K$-singular loci of any $K$-algebraic sets $X\subset L^n$, and to compare them with the $L|K$-nonsingular and $L|K$-singular loci of $X\subset L^n$:
\begin{itemize}
\item $\mr{Sing}^{E|K}(X)\subset(\ove^\sqbullet)^n$ is a $K$-algebraic subset of $X\cap(\ove^\sqbullet)^n\subset(\ove^\sqbullet)^n$ of dimension $<\dim(X)$, so $\Reg^{E|K}(X)$ is a non-empty $K$-Zariski open subset of $X\cap(\ove^\sqbullet)^n\subset(\ove^\sqbullet)^n$.
\item If $X\subset L^n$ is $K$-irreducible, $\Reg^{E|K}(X,e)\neq\varnothing$ implies $e=\dim(X)$. 
\item If $X_1,\ldots,X_s$ are the $K$-irreducible components of $X$, $a\in X\cap(\ove^\sqbullet)^n$ and $e\in\N$, then the following equivalences hold: $a\in\Reg^{E|K}(X,e)$ $\Leftrightarrow$ there exists a unique index $i\in\{1,\ldots,s\}$ such that $a\in \Reg^{E|K}(X_i)\setminus\bigcup_{j\in\{1,\ldots,s\}\setminus\{i\}}X_j$ and $\dim(X_i)=e$ $\Leftrightarrow$ there exist $f_1,\ldots,f_{n-e}\in\II_K(X)$ and a $E$-Zariski open (or a Euclidean open) neighborhood $U$ of $a$ in $L^n$ such that the gradients $\nabla f_1(a),\ldots,\nabla f_{n-e}(a)$ are linearly independent in $L^n$ and $X\cap U=\ZZ_L(f_1,\ldots,f_{n-e})\cap U$.
\item $\mr{Sing}^{L|K}(X)\subset L^n$ is the extension of coefficients of $\mr{Sing}^{E|K}(X)\subset(\ove^\sqbullet)^n$ {(from $\ove^\sqbullet$)} to $L$, that is, $(\Sing^{E|K}(X))_L=\Sing^{L|K}(X)$. In particular, $\Sing^{E|K}(X)=\Sing^{L|K}(X)\cap(\ove^\sqbullet)^n$ and it holds: $\mr{Sing}^{L|K}(X)=\varnothing$ if and only if $\mr{Sing}^{E|K}(X)=\varnothing$.
Consequently, $(\Reg^{E|K}(X,e))_L=\Reg^{L|K}(X,e)$ and $\Reg^{E|K}(X,e)=\Reg^{L|K}(X,e)\cap(\ove^\sqbullet)^n$ for each $e\in\N$. In particular, $\Reg^{L|K}(X)=X$ if and only if $\Reg^{E|K}(X)=X\cap(\ove^\sqbullet)^n$.
\item We prove also a result concerning the $K$-algebraicity of differences of $K$-algebraic sets: If $X$ and $Z$ are $K$-algebraic subsets of $L^n$ such that $Z\subset X$, $\dim(X)=\dim(Z)$ and $Z=\Reg^{L|K}(Z)\subset\Reg^{L|K}(X)$, then $X\setminus Z$ is also a $K$-algebraic subset of $L^n$. In addition, we have $\Reg^{L|K}(X\setminus Z)=\Reg^{L|K}(X)\setminus Z$, provided $Z\neq\Reg^{L|K}(X)$.
\end{itemize}

We deal with the case of hypersurfaces separately.

We combine the previous results with some from the previous sections to describe the relationship between the usual ($L|L$-)nonsingular and ($L|L$-)singular loci and the $L|K$-nonsingular and $L|K$-singular loci of any $K$-algebraic set $X\subset L^n$. Here the concept of $K$-bad set $B_K(X)$ of $X\subset L^n$ plays a central role, when $L$ is a r.c.f..

\begin{itemize}
\item If $L$ is an a.c.f.\ or $L|K$ is an extension of r.c.f.\ or $L$ is a r.c.f.\ and $X\subset L^n$ is defined over $K$, one has the expected equalities: $\Reg^{L|L}(X,e)=\Reg^{L|K}(X,e)$ for each $e\in\N$, so $\Reg(X):=\Reg^{L|L}(X)$ and $\Sing(X):=\Sing^{L|L}(X)$ coincide with $\Reg^{L|K}(X)$ and $\Sing^{L|K}(X)$, respectively.
\item Suppose that $L$ is a r.c.f., whereas its ordered subfield $K$ is not. Let $\kr\subset L$ be the real closure of $K$. Then $\Sing^{L|K}(X)$ is a $K$-algebraic subset of $X$ of dimension $<\dim(X)$, $\Sing(X)$ and $B_K(X)$ are $\kr$-algebraic subsets of $X$ of dimension $<\dim(X)$, and $\Sing^{L|K}(X)=\Sing(X)\cup B_K(X)$. Consequently, $\Reg^{L|K}(X)$ is a non-empty $K$-Zariski open subset of $X$, $\Reg(X)$ is a non-empty $\kr$-Zariski open subset of $X$, and $\Reg^{L|K}(X)=\Reg(X)\setminus B_K(X)$.
 
In addition, we present some explicit examples in the case $K=\Q$ that describe the possible behaviors of the sets $\Sing(X)$, $B_K(X)$ and $\Sing^{L|K}(X)$. In particular, we prove the sharpness of the previous result in the following sense: there exist $\Q$-algebraic sets $X\subset L^n$ such that $\Sing(X)\subset L^n$ and $B_\Q(X)\subset L^n$ are $\qr$-algebraic but not $\Q$-algebraic, whereas their union $\Sing^{L|\Q}(X)\subset L^n$ is $\Q$-algebraic.
\end{itemize}

\vspace{3mm}

\textsc{Section \ref{s5}} deals with four applications of the theory developed in the previous sections.

Let $R$ be a real closed field and let $K$ be an ordered subfield of $R$ (for example $K:=\Q$).

Subsection \ref{s51} is devoted to some diophantine properties of real $K$-algebraic sets. Let $X\subset R^n$ be a $K$-algebraic set and let $X(K):=X\cap K^n$. We prove that, if $X(K)$ is $K$-Zariski dense in~$X$, then $X\subset R^n$ is defined over $K$. Equivalently, if $X\subset R^n$ is not defined over $K$, then $X(K)$ is not $K$-Zariski dense in $X$. Furthermore, if $X\subset R^n$ is $K$-irreducible but $R$-reducible, then $\dim_K(X(K))<\dim(X)$.

In Subsection \ref{subsec:Whitney}, we adapt to our context the classical concepts of semialgebraic set, Nash manifold and stratification. We introduce the $K$-semialgebraic sets $S\subset R^n$, $K$-algebraic partial manifolds $M\subset R^n$, and $K$-algebraic stratifications of $K$-semialgebraic sets $S\subset R^n$. The latter concept concerns finite partitions $\{M_i\}_{i\in I}$ of~$S$, where the $M_i\subset R^n$ are semialgebraically connected $K$-algebraic partial manifolds satisfying the frontier condition. We prove that each $K$-semialgebraic set admits a Whitney regular $K$-algebraic stratification, answering affirmatively to a problem posed by Wies{\l}aw Paw{\l}ucki.

Subsection \ref{subsec:proj2} is devoted to a $R|K$-generic projection theorem. Let $X\subset R^n$ be a $K$-algebraic set of dimension $d$. Define the $R|K$-embedded dimension $e^{R|K}(X)$ of $X$ as the maximum $R$-vector dimension of the $R|K$-Zariski tangent space $T^{R|K}_a(X)$, when $a$ varies in $X$. Suppose that $r:=\max\{e^{L|K}(X)+d-1,2d+1\}<n$. Write $x=(x_1,\ldots,x_n)=(x',x'')\in R^n$, where $x':=(x_1,\ldots,x_r)$ and $x'':=(x_{r+1},\ldots,x_n)$, and consider $x$, $x'$ and $x''$ as column vectors. Denote $\mc{M}_{r,n-r}(K)$ the set of all $r\times(n-r)$-matrices with coefficients in $K$, and identify $\mc{M}_{r,n-r}(K)$ with $K^{r(n-r)}$. For each $A\in\mc{M}_{r,n-r}(K)$, let $\pi_A:R^n\to R^r$ be the projection in the direction of the $R$-vector subspace $\{(x',x'')\in L^n:x'=Ax''\}$ of $R^n$, that is, the $R$-linear map $\pi_A(x):=x'-Ax''$. We prove that there exists a non-empty $K$-Zariski open subset $\Omega$ of $\mc{M}_{r,n-r}(K)=K^{r(n-r)}$ such that, for each $A\in\Omega$, the set $\pi_A(X)\subset R^r$ is $K$-algebraic and the restriction of $\pi_A$ from $X\subset R^n$ to $\pi_A(X)\subset R^r$ is a $K$-biregular isomorphism. This result remains valid if $R$ is replaced with any algebraically closed field. 

In Subsection \ref{nash-tognoli-Q}, we present the main results of \cite{GS}, namely a `Nash-Tognoli theorem over the rationals' and a version of it for real algebraic sets with isolated singularities. These results were obtained in \cite{GS} making use of the $L|K$-algebraic geometry developed in the previous sections (in the case $L|K=\R|\Q$) and further developments obtained~in~\cite{GS}. The celebrated Nash-Tognoli theorem asserts that: \textit{Every compact smooth manifold $M$ of dimension $d$ is smoothly diffeomorphic to a nonsingular algebraic subset $M'$ of $\R^{2d+1}$}. Theorem 1.7 of \cite{GS} states that: \textit{$M'\subset\R^{2d+1}$ can be chosen to be $\Q$-nonsingular $\Q$-algebraic}. This result guarantees for the first time that, up to smooth diffeomorphisms, every compact smooth manifold $M$ can be described by means of finitely many exact data, such as a finite system of generators of the ideal $\II_\Q(M')$ of $\Q[\x_1,\ldots,\x_{2d+1}]$. According to \cite[Def.1.9]{GS}, a $\Q$-algebraic set $X\subset\R^n$ is called $\Q$-determined if $\Reg^\Q(X)=\Reg(X)$. Theorem 1.10 of \cite{GS} asserts that: \textit{Every real algebraic set with isolated singularities is semialgebraically homeomorphic to a $\Q$-determined $\Q$-algebraic set with isolated singularities}. Consequently, we also have: \textit{Every real algebraic set germ with an isolated singularity is semialgebraically homeomorphic to a real $\Q$-determined $\Q$-algebraic set germ with an isolated singularity} (see \cite[Thm.1.14]{GS}).

\vspace{3mm}

The paper ends with three appendices. In Appendix \ref{appendix}, we explain how to derive the `extension of coefficients' procedure for algebraically closed fields from the corresponding one for real closed fields. We give the definition of Euclidean topology on an algebraically closed field and show that this concept is not unique in general, providing an explicit example suggested by Elias~Baro. In Appendix \ref{appendix:b}, we present Laksov's Nullstellensatz that generalizes Hilbert's and Real Nullstellens\"atze. Appendix \ref{appendix:c} contains the proofs of some results stated in the paper. These proofs follow well-known arguments. However, they are not straightforwardly obtained, because they are supplemented by concepts and results introduced and proven for the first time in this paper.

%%%

\subsection{Related previous works on this topic}
$(\mr{i})$ Some of the concepts and results of this paper were already known. Definition \ref{def:overK} of real algebraic set defined over~$K$ is a special case of Definition~3 on page 30 of \cite{to2} due to Tognoli. Definition \ref{def:331} of $K$-regular map is a reformulation of Definition~4 on page 30 of \cite{to2}. In \cite[pp.28-29]{to2}, Tognoli proved Lemma \ref{k0}, Corollary \ref{cor-cap} and part of Proposition \ref{prop:zar} (see Remark \ref{to22}). In \cite{k}, M.\ Kato proved Corollary \ref{k}. Theorem \ref{thm:NTQ} is a $\Q$-version of the celebrated Nash-Tognoli theorem proven in \cite[Thm.1.7]{GS}. Proofs of a strong version of Theorem \ref{thm:NTQ} were proposed in \cite[Thm.2, p.56]{to2} and \cite[Thm.0.1]{bt}. However, these proofs are not complete and therefore such a strong version of Theorem \ref{thm:NTQ} cannot be considered valid, see \cite[Sect.1.3]{GS}.

Two important tools used in this work are the $L|K$-version of Hilbert's Nullstellensatz provided in Corollary \ref{kreliablec} and the classical Real Nullstellensatz \cite{d,r}. Corollary \ref{kreliablec} is certainly well-known for complex algebraic geometers, although we have not found it stated with such generality in the literature. $L|K$-versions of Real Nullstellensatz and Positivstellensatz were proven by Stengle in \cite{st}. If $L|K$ is an algebraic extension, the mentioned $L|K$-versions of Hilbert's and Real Nullstellens\"atze were generalized by Laksov in \cite{la}. Appendix \ref{appendix:b} is devoted to such generalization that we call Laksov's Nullstellensatz. 

$(\mr{ii})$ Let $R$ be a real closed field and let $K$ be an ordered subfield of $R$. A subset of $R^n$ is called \emph{$K$-semialgebraic} if it can be described as a finite Boolean combination of polynomial equalities and inequalities with coefficients in $K$ (see \cite[Ch.6.\S7]{abr}). Let $S\subset R^n$ be a usual semialgebraic set. In \cite{jrs}, the authors study the problem of algorithmically computing the semialgebraically connected components $S_i$ of $S$, and prove that if $S\subset R^n$ is $K$-semialgebraic then each $S_i$ is also $K$-semialgebraic. We will use this result in Subsection \ref{subsec:Whitney}.

$(\mr{iii})$ In \cite{Te90}, Teissier proposed an example of a complex analytic surface singularity in $\C^3$ that is not Whitney-equisingular to any complex surface singularity in $\C^3$ defined over $\Q$. A~complete proof of a corrected version of this example was recently given by Parusi\'{n}ski and Paunescu in \cite{PP}. This `negative $\Q$-approximation' example is related to the `affirmative $\Q$-approximation' results of Section \ref{nash-tognoli-Q}, especially Theorems \ref{thm:main}, \ref{thm:main-2} and \ref{thm:main-germs}. However, these results are not contradictory, since in the latter theorems the `affirmative $\Q$-approximations' occur through semialgebraic homeomorphisms and not through Whitney-equisingular deformations (of the complexification).

$(\mr{iv})$ Hilbert's $17^{\mr{th}}$ Problem asks whether every polynomial in $\R[\x]$ non-negative on $\R^n$ is a sum of squares of rational functions with coefficients in $\R$. This problem was solved in the affirmative by E.\ Artin. It is also known that the denominators cannot be omitted: there exist non-negative polynomials in $\R[\x]$ that are not sums of squares of polynomials in $\R[\x]$. A related problem raised by Sturmfels asks whether every polynomial in $\Q[\x]$ that is a sum of squares of polynomials in $\R[\x]$ is also a sum of squares of polynomials in $\Q[\x]$. This problem is related to the existence of exact positivity certificates for real polynomials in semidefinite programming, see \cite{bv,klyz,sz,wsv}. The answer to Sturmfels' problem is negative. By a smart use of Galois theory, Scheiderer provided explicit examples of polynomials in $\Q[\x]$ that are sums of squares in $\R[\x]$, but not in $\Q[\x]$. For more information on this result of Scheiderer and its extensions and applications, we refer to \cite{be,fe,hi,laplagne,msv,nds,ns,q}.

\subsection{To ease the reading} Along the rest of the paper, the abbreviations `a.c.f.' for algebraically closed fields of characteristic zero and `r.c.f.' for real closed field will no longer be used.

All the algebraically closed fields considered below will be of chara\-cteristic zero. For short, we will write `algebraically closed field' meaning `algebraically closed field of characteristic zero'.

Sometimes we will use the terms `complex' and `real' to mean any algebraically closed field or any real closed field, respectively. 

For all extensions of fields $L|K$ in which $L$ will be real closed, we will assume that $K$ is endowed with the ordering induced by that of $L$. In particular, if $K$ will also be a real closed field, the ordering of $L$ will extend that of $K$.

We will say that $L|E|K$ is an extension of fields if $L|E$ and $E|K$ are extensions of fields, that is, $K$ is a subfield of $E$ which is a subfield of $L$. Analogously, we will say that $L|F|E|K$ is an extension of fields if $L|F$, $F|E$ and $E|K$ are extensions of fields.

The letter $n$ will indicate a fixed positive natural number and, for each field $L$, the symbol $L[\x]$ will be the abbreviation for $L[\x_1,\ldots,\x_n]$, the ring of polynomials in the indeterminates $\x_1,\ldots,\x_n$ with coefficients in $L$.

In some parts of the paper, we will consider families $\{Z_\sigma\}_{\sigma\in G}$ of subsets of $L^n$ in which repetitions may occur. When necessary, we will implicitly assume that repetitions have been eliminated. This assumption will be clear from the context.

At the beginning of each section and some subsections, we will specify the type of fields that will be used in that section or subsection.

\vspace{1em}

One may think that many proofs in this paper are immediate generalizations of proofs of standard $L|L$-algebraic geometry. Actually, this happens only few times. Such proofs are duly brief, lacking in detail, and begins with the phrase `A standard argument works'. We decided to add them for the sake of completeness. Most of the results and proofs in this paper are original and specific to $L|K$-algebraic geometry. 

\begin{ack}
This work was supported by the ``National Group for Algebraic
and Geometric Structures, and their Applications'' (GNSAGA - INDAM).

The first author is supported by Spanish STRANO PID2021-122752NB-I00 and Grupos UCM 910444. This paper has been developed during several one month research stays of the first author in the Dipartimento di Matematica of the Universit\`a di Trento. The first author would like to thank the department for the invitations and the very pleasant working conditions.

The second author is partially supported by Spanish STRANO PID2021-122752NB-I00.
\end{ack}

%%%
\section{$K$-algebraic sets}\label{s1}

\subsection{Preliminary definitions and properties}

{\textit{Fix $n\in\N^*:=\N\setminus\{0\}$, fix any field $L$ and fix any subfield $K$ of $L$.}}

\subsubsection{$K$-algebraic sets and $K$-Zariski topology} In order to shorten notations, we denote $L[\x]:=L[\x_1,\ldots,\x_n]$. Consider $K[\x]$ as a subset of $L[\x]$, and $K^n$ as a subset of $L^n$. Given sets $F\subset L[\x]$ and $S\subset L^n$, define
\begin{align*}
\ZZ_L(F)&:=\{x\in L^n: f(x)=0,\ \forall f\in F\},\\
\II_K(S)&:=\{f\in K[\x]: f(x)=0,\ \forall x\in S\}.
\end{align*}
Observe that $\ZZ_L(F)$ is an algebraic subset of $L^n$ and $\II_K(S)=\II_L(S)\cap K[\x]$ is an ideal of $K[\x]$. If $F=\{f_1,\ldots,f_s\}\subset L[\x]$, we set $\ZZ_L(f_1,\ldots,f_s):=\ZZ_L(F)$.

Let us introduce the notion of $K$-algebraic subset of $L^n$.

\begin{defn}\label{def:K-alg}
Let $X$ be a subset of $L^n$. We say that $X$ is a \emph{$K$-algebraic subset of $L^n$}, or $X\subset L^n$ is a \emph{$K$-algebraic set}, if $X=\ZZ_L(F)$ for some $F\subset K[\x]$. $\sqbullet$
\end{defn}

\begin{remark}
Observe that $X\subset L^n$ is a $K$-algebraic set if and only if $X=\ZZ_L(\II_K(X))$. $\sqbullet$
\end{remark}

If $K=L$, an $L$-algebraic subset of $L^n$ is a usual algebraic subset of $L^n$. The family of all $K$-algebraic subsets of $L^n$ provides a topology on $L^n$ coarser than the usual Zariski topology on $L^n$. In particular, it is a Noetherian topology.

\begin{defn}\label{def:K-zar}
The \emph{$K$-Zariski topology of $L^n$} is the Noetherian topology of $L^n$ whose closed sets are the $K$-algebraic subsets of~$L^n$. We call the open sets of this topology \emph{$K$-Zariski open} subsets of $L^n$ and its closed sets \emph{$K$-Zariski closed} subsets of $L^n$.

Given a subset $S$ of $L^n$, we define the \emph{$K$-Zariski topology of $S$} as the relative topology of $S$ induced by the $K$-Zariski topology of $L^n$. As above, we call the open sets of this topology \emph{$K$-Zariski open} subsets of $S$ and its closed sets \emph{$K$-Zariski closed} subsets of $S$. We also say that a $K$-Zariski closed subset of $S$ is a \emph{$K$-algebraic} subset of $S$. $\sqbullet$
\end{defn}

By the Noetherianity of the $K$-Zariski topology of $L^n$, each $K$-algebraic subset $X$ of $L^n$ is the zero set of finitely many polynomials $f_1,\ldots,f_r\in K[\x]$. If in addition the field $L$ is real, then $X\subset L^n$ is the zero set of the single polynomial $f_1^2+\ldots+f_r^2\in K[\x]$. 

If $E$ is a subfield of $L$ containing $K$, the $K$-Zariski topology of $E^n$ coincides with the relative topology of $E^n$ induced by the $K$-Zariski topology of $L^n$.

\begin{defn}\label{def:K-constructible}
We say that a subset $S$ of $L^n$ is \emph{$K$-Zariski locally closed} if it is locally closed in $L^n$ with respect to the $K$-Zariski topology, that is, $S=S_1\setminus S_2$ for some $K$-Zariski closed subsets $S_1$ and $S_2$ of $L^n$. We also say that $S$ is \emph{$K$-constructible} if it is a Boolean combination of $K$-Zariski closed subsets of $L^n$ or, equivalently, if $S$ is the disjoint union of finitely many $K$-Zariski locally closed subsets of $L^n$. $\sqbullet$ 
\end{defn}

We will use the concept of $K$-constructible set in Subsections \ref{subsec:proj} and \ref{subsec:proj2}.

\subsubsection{$K$-irreducibility and $K$-irreducible components}\label{kirredk}
Let us introduce the concepts of $K$-irredu\-cibility and $K$-irreducible components of a $K$-algebraic set and present some of their properties.

\begin{defn}
Given a $K$-algebraic subset $X$ of $L^n$, we say that $X$ is \emph{$K$-irreducible} if it is irreducible with respect to the $K$-Zariski topology of $X\subset L^n$ or, equivalently, if it is non-empty and there do not exist $K$-algebraic subsets $X_1$ and $X_2$ of $L^n$ such that $X_1\subsetneqq X$, $X_2\subsetneqq X$ and $X=X_1\cup X_2$. Otherwise, we say that $X$ is \emph{$K$-reducible}. $\sqbullet$ 
\end{defn}

\begin{lem}\label{lem:prime}
Let $X\subset L^n$ be a $K$-algebraic set. Then $X$ is $K$-irreducible if and only if $\II_K(X)$ is a prime ideal of $K[\x]$.
\end{lem}
\begin{proof}
A standard argument works. If $X$ is $K$-reducible, $X\neq\varnothing$ and $X=X_1\cup X_2$ for some $K$-algebraic sets $X_1,X_2\subset L^n$ with $X_i\subsetneqq X$ for each $i\in\{1,2\}$, then there exists $f_i\in\II_K(X_i)\setminus\II_K(X)$, so $\II_K(X)$ is not prime, because $f_1f_2\in\II_K(X)$. Conversely, if $\II_K(X)$ is not prime and $\II_K(X)\neq K[\x]$, there exist $g_1,g_2\in K[\x]\setminus\II_K(X)$ such that $g_1g_2\in\II_K(X)$. Denote $X_i'\subset L^n$ the $K$-algebraic set $X\cap\ZZ_L(g_i)$ for each $i\in\{1,2\}$. It follows that $X$ is $K$-reducible, because $X_i'\subsetneqq X$ and $X=X_1'\cup X_2'$. 
\end{proof}

A special case of the previous result reads as follows: if $f\in K[\x]$ and $\II_K(\ZZ_L(f))=(f)K[\x]$, then $\ZZ_L(f)\subset L^n$ is $K$-irreducible if and only if $f$ is irreducible as a polynomial of $K[\x]$.

\begin{lem}\label{inclusion}
Let $X,Y_1,\ldots,Y_r\subset L^n$ be $K$-algebraic sets such that $X$ is $K$-irreducible and $X\subset\bigcup_{i=1}^rY_i$. Then $X\subset Y_j$ for some $j\in\{1,\ldots,r\}$.
\end{lem} 
\begin{proof}
A standard argument works. Suppose $X\not\subset Y_j$ for all $j\in\{1,\ldots,r\}$. In particular, $r\geq2$. Let $s\in\{1,\ldots,r-1\}$ be the maximum $j$ such that $X\not\subset\bigcup_{i=1}^j Y_i$. Define the $K$-algebraic sets $X_1,X_2\subset L^n$ as $X_1:=X\cap(\bigcup_{i=1}^s Y_i)$ and $X_2:=X\cap X_{s+1}$. We have $X_1\subsetneqq X$, $X_2\subsetneqq X$ and $X=X_1\cup X_2$. This is impossible, because $X$ is $K$-irreducible. Consequently, $X\subset Y_j$ for some $j\in\{1,\ldots,r\}$, as required.
\end{proof}

As the $K$-topology of $L^n$ is Noetherian, the following general result holds, see \cite[Prop.1.5, p.5]{ha} for a proof.

\begin{lem}\label{lem:irred}
For each $K$-algebraic subset $X$ of $L^n$, there exists a finite family $\{X_1,\ldots,X_r\}$ of $K$-irreducible $K$-algebraic subsets of $L^n$, uniquely determined by $X$, such that $X=\bigcup_{i=1}^rX_i$ and $X_i\not\subset\bigcup_{j\in\{1,\ldots,r\}\setminus\{i\}}X_j$ for all $i\in\{1,\ldots,r\}$. We call the $X_i$ the \emph{$K$-irreducible components} of~$X$. 
\end{lem}

By Lemma \ref{inclusion}, condition `\emph{$X_i\not\subset\bigcup_{j\in\{1,\ldots,r\}\setminus\{i\}}X_j$ for all $i\in\{1,\ldots,r\}$}' is equi\-valent to `\emph{$X_i\not\subset X_j$ for all $i,j\in\{1,\ldots,r\}$ with $i\neq j$}'.

If $K=L$, a $K$-irreducible $K$-algebraic subset of $L^n$ is a usual irreducible algebraic subset of $L^n$ and the $K$-irreducible components of an algebraic subset of $L^n$ are its usual irreducible components.

\subsubsection{$K$-Zariski closure and $K$-dimension}
As one can expect, our next step is to define the $K$-Zariski closure of a subset of $L^n$ and to analyze its main properties.

\begin{defn}
Let $S$ be a subset of $L^n$. We denote $\zcl_{L^n}^K(S)$ the \emph{$K$-Zariski closure of $S$ (in~$L^n$)}, that is, the closure of $S$ with respect to the $K$-Zariski topology of $L^n$. If $K=L$, we write $\zcl_{L^n}(S)$ to denote $\zcl_{L^n}^L(S)$. $\sqbullet$
\end{defn}

Observe that $\zcl_{L^n}^K(S)=\ZZ_L(\II_K(S))$ and $\zcl_{L^n}(S)\subset\zcl_{L^n}^K(S)$. Thus, $\zcl_{L^n}^K(\zcl_{L^n}(S))=\zcl_{L^n}^K(S)$ and $\II_K(S)=\II_K(\zcl_{L^n}(S))=\II_K(\zcl_{L^n}^K(S))$. If $K=L$, then $\zcl_{L^n}^K(S)=\zcl_{L^n}(S)$ is the usual Zariski closure of $S$ in $L^n$. 

\begin{remark}
The inclusion $\zcl_{L^n}(S)\subset\zcl_{L^n}^K(S)$ may be strict. For instance, if $L$ is a real closed field and $K=\Q$, then $\zcl_L(\{\sqrt{2}\})=\{\sqrt{2}\}\subsetneqq\zcl_L^\Q(\{\sqrt{2}\})=\ZZ_L(\x_1^2-2)=\{-\sqrt{2},\sqrt{2}\}$. $\sqbullet$
\end{remark}

Let us introduce the notion of (algebraic) $K$-dimension of a set $S\subset L^n$.

\begin{defn}\label{def:K-dim}
Given a subset $S$ of $L^n$, we define the \emph{$K$-dimension $\dim_K(S)$ of $S$ (in $L^n$)} as the Krull dimension of the ring $K[\x]/\II_K(S)$. If $S=\varnothing$, then we set $\dim_K(S):=-1$. Otherwise, $\dim(S)$ is a natural number. $\sqbullet$
\end{defn}

By \cite[Cor.13.4]{e} and \cite[(14.G) Cor.1, pag.91]{m}, it holds 
\begin{equation}\label{eisenbud}
\dim_K(S)=\dim(K[\x]/\II_K(S))=n-\hgt(\II_K(S)).
\end{equation}

If $K=L$, then $\dim_L(S)$ coincides with the usual dimension of $\zcl_{L^n}(S)$ in $L^n$.

\begin{remark}
Let $S_1,\ldots,S_r$ be subsets of $L^n$ and let $S:=\bigcup_{i=1}^rS_i\subset L^n$. As $\II_K(S)=\bigcap_{i=1}^r\II_K(S_i)$, we know that $\hgt(\II_K(S))=\min\{\hgt(\II_K(S_1)),\ldots,\hgt(\II_K(S_r))\}$. By \eqref{eisenbud}, it follows that $\dim_K(S)=\max\{\dim_K(S_1),\ldots,\dim_K(S_r)\}$. $\sqbullet$
\end{remark}

\begin{lem}\label{dimirred}
Let $X,Y\subset L^n$ be $K$-algebraic sets such that $Y\subsetneqq X$ and $X$ is $K$-irreducible. Then $\dim_K(Y)<\dim_K(X)$.
\end{lem}
\begin{proof}
A standard argument works. Suppose $\dim_K(Y)=\dim_K(X)$. Let $\gtp$ be a prime ideal of $K[\x]$ associated to $\II_K(Y)$ such that $\hgt(\II_K(X))=\hgt(\II_K(Y))=\hgt(\gtp)$. As $\II_K(X)\subset \II_K(Y)\subset\gtp$ and the ideal $\II_K(X)$ of $K[\x]$ is prime, we deduce $\II_K(X)=\II_K(Y)=\gtp$, so $X=\ZZ_L(\II_K(X))=\ZZ_L(\II_K(Y))=Y$. Consequently, if $Y\subsetneqq X$, then $\dim_K(Y)<\dim_K(X)$, as required.
\end{proof}

\begin{lem}\label{lem:irreducibility}
Let $K$ be a subfield of $E$, let $Y\subset L^n$ be an $E$-algebraic set and let $X:=\zcl_{L^n}^K(Y)$ be its $K$-Zariski closure in $L^n$. If $Y$ is $E$-irreducible, then $X\subset L^n$ is $K$-irreducible.
\end{lem}
\begin{proof}
As the ideal $\II_E(Y)$ of $E[\x]$ is prime and $\II_K(X)=\II_K(Y)=\II_E(Y)\cap K[\x]$, the ideal $\II_K(X)$ of $K[\x]$ is prime as well. By Lemma \ref{lem:prime}, $X$ is $K$-irreducible.
\end{proof}

\begin{remark}
The converse of Lemma \ref{lem:irreducibility} is not true. For instance, if $L=E$ is any real closed field $R$ and $K=\Q$, then the set $Y=X=\{-\sqrt{2},\sqrt{2}\}$ is a $\Q$-irreducible $\Q$-algebraic subset of~$R$, which is reducible as an algebraic subset of~$R$. $\sqbullet$
\end{remark}

\begin{lem}\label{lem:K-irred}
Let $K$ be a subfield of $E$, let $Y\subset L^n$ be an $E$-algebraic set, let $Y_1,\ldots,Y_r$ be the $E$-irreducible components of~$Y$, let $X:=\zcl_{L^n}^K(Y)$ and let $X_i:=\zcl_{L^n}^K(Y_i)$ for each $i\in\{1,\ldots,r\}$. We have:
\begin{itemize}
\item[$(\mr{i})$] There exists a subset $\{i_1,\ldots,i_s\}$ of $\{1,\ldots,r\}$ such that $X_{i_1},\ldots,X_{i_s}$ are the $K$-irreducible components of $X$.
\item[$(\mr{ii})$] If $\dim_K(X_i)=\dim_K(X_j)$ for all $i,j\in\{1,\ldots,r\}$, then there exists a surjective map $\eta:\{1,\ldots,r\}\to\{i_1,\ldots,i_s\}$ such that $X_i=X_{\eta(i)}$ for each $i\in\{1,\ldots,r\}$.
\end{itemize}
\end{lem}
\begin{proof}
As $X=\zcl_{L^n}^K(\bigcup_{i=1}^rY_i)=\bigcup_{i=1}^r\zcl_{L^n}^K(Y_i)=\bigcup_{i=1}^rX_i$, there exists a subset $\{i_1,\ldots,i_s\}$ of $\{1,\ldots,r\}$ of minimal cardinality such that $\bigcup_{\ell=1}^sX_{i_\ell}=X$. By Lemmas \ref{lem:irred} and \ref{lem:irreducibility}, $X_{i_1},\ldots,X_{i_s}$ are the $K$-irreducible components of $X$.

Suppose that $\dim_K(X_i)=\dim_K(X_j)$ for all $i,j\in\{1,\ldots,r\}$. As each $X_i$ is contained in $X=\bigcup_{j=1}^sX_{i_j}$, by Lemma \ref{inclusion}, there exists $j\in\{1,\ldots,s\}$ such that $X_i\subset X_{i_j}$. As $\dim_K(X_i)=\dim_K(X_{i_j})$, Lemma \ref{dimirred} assures that $X_i=X_{i_j}$. If $X_i=X_{i_h}$ for some $h\in\{1,\ldots,s\}\setminus\{j\}$, then $X_{i_h}=X_i=X_{i_j}$, which is impossible. We define $\eta$ in such a way that it maps each $i$ to the unique $i_j$ such that $X_i=X_{i_j}$. As $\eta(i_j)=i_j$ for each $j\in\{1,\ldots,s\}$, $\eta$ is surjective, as required.
\end{proof}

\begin{remark}
In the statement of Lemma \ref{lem:K-irred}$(\mr{ii})$, we cannot omit the condition `$\dim_K(X_i)=\dim_K(X_j)$ for all $i,j\in\{1,\ldots,r\}$'. Indeed, it may happen that there exists $i\in\{1,\ldots,r\}$ such that $X_i\subsetneqq X_j$ for all $j\in\{1,\ldots,r\}\setminus\{i\}$. For instance, let $L=E$ be any real closed field $R$, let $K:=\Q$ and let $Y=Y_1\cup Y_2\subset R^2$ be the algebraic set with two irreducible components $Y_1:=\{(-\sqrt{2},0)\}$ and $Y_2:=\ZZ_R(\x_1-\sqrt{2})\subset R^2$. Then $X_1:=\zcl_{R^2}^\Q(Y_1)=\{(-\sqrt{2},0),(\sqrt{2},0)\}\subsetneqq X_2:=\zcl_{R^2}^\Q(Y_2)=\ZZ_R(\x_1^2-2)=X:=\zcl_{R^2}^\Q(Y)$. In particular, $X\subset R^2$ is $\Q$-irreducible. Observe that $\dim_\Q(X_1)=0$, whereas $\dim_\Q(X_2)=1$. $\sqbullet$
\end{remark}

\subsection{Restriction and extension of fields}

\emph{Recall that $L|K$ is an arbitrary extension of fields. Along this subsection, $\Bb=\{u_j\}_{j\in J}$ denotes a basis of $L$ as a $K$-vector space.}

\subsubsection{Restriction of coefficients} The next result coincides with equation $(1)$ on page 28 and Lemma 1(1) on page 29 of \cite{to2}, see also \cite{k}. For completeness, we provide a proof.

\begin{lem}[{\cite{to2}}]\label{k0}
Each polynomial $f\in L[\x]$ can be written uniquely as $f=\sum_{j\in J}u_jf_j$, where $f_j\in K[\x]$ and only finitely many of the polynomials $f_j$ are nonzero.
\end{lem}
\begin{proof}
{\sc Existence.} Write $f:=\sum_\nu b_\nu\x^\nu$, where $b_\nu\in L$ and only finitely many $b_\nu$ are nonzero. Write $b_\nu=\sum_{j\in J}a_{j\nu}u_j$, where only finitely many $a_{j\nu}\in K$ are nonzero. We deduce $f=\sum_\nu\big(\sum_{j\in J}a_{j\nu}u_j\big)\x^\nu=\sum_{j\in J}u_j\big(\sum_\nu a_{j\nu}\x^\nu\big)=\sum_{j\in J}u_jf_j$, where $f_j:=\sum_\nu a_{j\nu}\x^\nu\in K[\x]$.

\noindent{\sc Uniqueness.} Let $\{f_j\}_{j\in J}$ be polynomials in $K[\x]$ such that only finitely many of the $f_j$ are nonzero and $\sum_{j\in J}u_jf_j=0$. Write $f_j:=\sum_\nu a_{j\nu}\x^\nu\in K[\x]$. Thus, $0=\sum_{j\in J}u_jf_j=\sum_{j\in J}u_j\big(\sum_\nu a_{j\nu}\x^\nu\big)=\sum_\nu\big(\sum_{j\in J}a_{j\nu}u_j\big)\x^\nu$, where only finitely many of the $a_{j\nu}$ are nonzero. Consequently, $\sum_{j\in J}a_{j\nu}u_j=0$ for each multi-index $\nu$. We conclude $a_{j\nu}=0$ for each pair $(j,\nu)$, so $f_j=0$ for all $j\in J$.
\end{proof}

Let us see some useful consequences. The next result is contained in \cite{k}. For completeness, we provide a proof of this result.

\begin{cor}[{\cite{k}}]\label{k}
If $\gta$ is an ideal of $K[\x]$ and $g_1,\ldots,g_r\in K[\x]$ are generators of the ideal $\gta L[\x]$ in $L[\x]$, then $g_1,\ldots,g_r$ are also generators of the ideal $(\gta L[\x])\cap K[\x]$ in $K[\x]$. In particular, we have:
$$
(\gta L[\x])\cap K[\x]=\gta.
$$ 
\end{cor}
\begin{proof}
We may assume that $u_{j_0}=1$ for some $j_0\in J$. Let $f\in(\gta L[\x])\cap K[\x]$ and write $f=\sum_{i=1}^rf_ig_i$, where $f_i\in L[\x]$. By Lemma \ref{k0}, we can write $f_i=\sum_{j\in J}u_jf_{ij}$, where $f_{ij}\in K[\x]$ and only finitely many of them are nonzero. Then
$$\textstyle
u_{j_0}f=f=\sum_{i=1}^rf_ig_i=\sum_{i=1}^r\big(\sum_{j\in J}u_jf_{ij}\big)g_i=\sum_{j\in J}u_j\big(\sum_{i=1}^rf_{ij}g_i\big).
$$
As $f\in K[\x]$ (and only finitely many of the polynomials $\sum_{i=1}^rf_{ij}g_i\in K[\x]$ are nonzero), Lemma~\ref{k0} implies that $f=\sum_{i=1}^rf_{ij_0}g_i$. This proves that $g_1,\ldots,g_r$ are generators of $(\gta L[\x])\cap K[\x]$ in $K[\x]$.

If we assume $g_1,\ldots,g_r$ are generators of $\gta$ in $K[\x]$ (they are also generators of $\gta L[\x]$ in $L[\x]$), the preceding part of the proof implies that $g_1,\ldots,g_r$ are also generators of $(\gta L[\x])\cap K[\x]$ in $K[\x]$, so the ideals $\gta$ and $(\gta L[\x])\cap K[\x]$ of $K[\x]$ coincide.
\end{proof}

\begin{cor}\label{intkx}
Suppose $n\geq2$, set $\tilde{\x}:=(\x_1,\ldots,\x_{n-1})$ and consider $L[\tilde{\x}]$ as a subring of $L[\x]$ and $K[\tilde{\x}]$ as a subring of $K[\x]$ in the usual way. Then, if $\gta$ is an ideal of $K[\x]$, we have:
$$
(\gta L[\x])\cap L[\tilde{\x}]=(\gta\cap K[\tilde{\x}])L[\tilde{\x}].
$$
\end{cor}
\begin{proof}
The inclusion $(\gta\cap K[\tilde{\x}])L[\tilde{\x}]\subset(\gta L[\x])\cap L[\tilde{\x}]$ is clear, because $\gta\cap K[\tilde{\x}]\subset\gta L[\x]$. We prove now the converse inclusion. Pick $h\in(\gta L[\x])\cap L[\tilde{\x}]$. Then there exist finitely many polynomials $f_1,\ldots,f_r\in L[\x]$ and $g_1,\ldots,g_r\in\gta$ such that $h=\sum_{i=1}^rf_ig_i$. By Lemma \ref{k0}, we can write uniquely $f_i=\sum_{j\in J}u_jf_{ij}$, where $f_{ij}\in K[\x]$ and only finitely many of the polynomials $f_{ij}$ are nonzero. We have:
$$\textstyle
h=\sum_{i=1}^rf_ig_i=\sum_{i=1}^r(\sum_{j\in J}u_jf_{ij})g_i=\sum_{j\in J}u_j(\sum_{i=1}^rf_{ij}g_i),
$$
where the polynomials $\sum_{i=1}^rf_{ij}g_i$ belong to $\gta\subset K[\x]$ and only finitely many of them are non\-zero. As $h\in L[\tilde{\x}]$, using again Lemma \ref{k0}, we can also write uniquely $h=\sum_{j\in J}u_jh_j$, where $h_j\in K[\tilde{\x}]$ and only finitely many of the polynomials $h_j$ are non-zero. As $K[\tilde{\x}]\subset K[\x]$, by the uniqueness assertion in Lemma \ref{k0}, we deduce $\sum_{i=1}^rf_{ij}g_i=h_j\in \gta\cap K[\tilde{\x}]$ for each $j\in J$, so $h=\sum_{j\in J}u_jh_j\in(\gta\cap K[\tilde{\x}])L[\tilde{\x}]$, as required.
\end{proof}

\begin{cor}\label{lem:a}
Let $\gta_1,\ldots,\gta_s$ be ideals of $K[\x]$ and let $\gtn$ be a prime ideal of $L[\x]$. Then we have: $\bigcap_{\ell=1}^s(\gta_\ell L[\x])=(\bigcap_{\ell=1}^s\gta_\ell)L[\x]$ and $\bigcap_{\ell=1}^s(\gta_\ell L[\x]_\gtn)=(\bigcap_{\ell=1}^s\gta_\ell)L[\x]_\gtn$. 
\end{cor}
\begin{proof}
The inclusions `$\supset$' are evident. Let us prove the converse inclusions. Pick a polynomial $f_0\in\bigcap_{\ell=1}^s(\gta_\ell L[\x])$. By Lemma \ref{k0}, we have $f_0=\sum_{j\in J}u_jf_{0j}$, where $f_{0j}\in K[\x]$ and only finitely many of the $f_{0j}$ are nonzero. Fix $\ell\in\{1,\ldots,s\}$. Let $\{g_1,\ldots,g_r\}$ be a system of generators of $\gta_\ell$ in $K[\x]$ and let $f_1,\ldots,f_r\in L[\x]$ be such that $f_0=\sum_{i=1}^rf_ig_i$. By Lemma \ref{k0}, for each $i\in\{1,\ldots,r\}$, we have $f_i=\sum_{j\in J}u_jf_{ij}$, where $f_{ij}\in K[\x]$ and only finitely many of the $f_{ij}$ are nonzero. As $\sum_{j\in J}u_jf_{0j}=f_0=\sum_{i=1}^r(\sum_{j\in J}u_jf_{ij})g_i=\sum_{j\in J}u_j(\sum_{i=1}^rf_{ij}g_i)$, the uniqueness assertion in Lemma \ref{k0} assures that $f_{0j}=\sum_{i=1}^rf_{ij}g_i\in\gta_\ell$ for each $j\in J$. 

This is true for each $\ell$, so $f_{0j}\in\bigcap_{\ell=1}^s\gta_\ell\subset(\bigcap_{\ell=1}^s\gta_\ell)L[\x]$ for each $j\in J$. It follows that $f_0=\sum_{j\in J}u_jf_{0j}\in(\bigcap_{\ell=1}^s\gta_\ell)L[\x]$ too.

Next, consider $\frac{f}{g}\in\bigcap_{\ell=1}^s(\gta_\ell L[\x]_\gtn)$. It suffices to show that $f\in(\bigcap_{\ell=1}^s\gta_\ell)L[\x]_\gtn$. As $f\in\bigcap_{\ell=1}^s(\gta_\ell L[\x]_\gtn)$, for each $\ell\in\{1,\ldots,s\}$, there exists $h_\ell\in L[\x]\setminus\gtn$ such that $h_\ell f\in\gta_\ell L[\x]$. As $L[\x]\setminus\gtn$ is a multiplicative subset of $L[\x]$, the polynomial $h:=\prod_{\ell=1}^sh_\ell$ belongs to $L[\x]\setminus\gtn$. In addition, $hf\in\bigcap_{\ell=1}^s(\gta_\ell L[\x])$. By the preceding part of the proof, we have $hf\in(\bigcap_{\ell=1}^s\gta_\ell)L[\x]$, so $f\in(\bigcap_{\ell=1}^s\gta_\ell)L[\x]_\gtn$, as required.
\end{proof}

The next result was originally proved by Tognoli in Lemma 1(3) on page 29 of \cite{to2}.

\begin{cor}[{\cite{to2}}]\label{cor-cap}
If $X$ is an algebraic subset of $L^n$, then $X\cap K^n$ is an algebraic subset of $K^n$.
\end{cor}
\begin{proof}
Let $f_1,\ldots,f_r\in L[\x]$ be polynomials such that $X=\ZZ_L(f_1,\ldots,f_r)$. By Lemma \ref{k0}, each $f_i\in L[\x]$ can be uniquely written as $f_i=\sum_{j\in J}u_jf_{ij}$, where $f_{ij}\in K[\x]$ and only finitely many of the $f_{ij}$ are nonzero. Observe that a point $x\in K^n$ belongs to $X$ if and only if $0=f_i(x)=\sum_{j\in J}u_jf_{ij}(x)$ for all $i\in\{1,\ldots,r\}$. As each $f_{ij}(x)$ belongs to $K$, one deduces that $x\in\ZZ_K(\{f_{ij}\}_{i\in\{1,\ldots,r\},j\in J})$. Consequently, $X\cap K^n=\ZZ_K(\{f_{ij}\}_{i\in\{1,\ldots,r\},j\in J})$, as required.
\end{proof}

\begin{cor}\label{LK-zar}
If $S$ is a subset of $K^n$, then $\zcl_{K^n}(S)=\zcl_{L^n}(S)\cap K^n=\zcl_{L^n}^K(S)\cap K^n$, $\zcl_{L^n}(\zcl_{K^n}(S))=\zcl_{L^n}(S)$ and $\zcl_{L^n}^K(\zcl_{K^n}(S))=\zcl_{L^n}^K(S)$.
\end{cor}
\begin{proof}
Let $X:=\zcl_{L^n}(S)\subset L^n$ and let $X':=\zcl_{L^n}^K(S)\subset L^n$. By Corollary \ref{cor-cap}, $X\cap K^n\subset K^n$ is algebraic. As $S\subset X\cap K^n$ and $X\subset X'$, it follows that $\zcl_{K^n}(S)\subset X\cap K^n\subset X\subset X'$, $X\subset\zcl_{L^n}(\zcl_{K^n}(S))\subset\zcl_{L^n}(X)=X$ and $X'\subset\zcl_{L^n}^K(\zcl_{K^n}(S))\subset\zcl_{L^n}^K(X')=X'$, so $X=\zcl_{L^n}(\zcl_{K^n}(S))$ and $X'=\zcl_{L^n}^K(\zcl_{K^n}(S))$. In addition,
$$
\zcl_{K^n}(S)\subset X\cap K^n\subset X'\cap K^n=\ZZ_L(\II_K(S))\cap K^n=\ZZ_K(\II_K(S))=\zcl_{K^n}(S),
$$
so $\zcl_{K^n}(S)=X\cap K^n=X'\cap K^n$, as required.
\end{proof}

We compare next the dimensions $\dim_K(X)$ and $\dim_L(X)$ of a $K$-algebraic set $X\subset L^n$.

\begin{lem}\label{intersection}
Let $\gta$ be an ideal of $K[\x]$ and let $r\in\{0,\ldots,n\}$. If $\gta\cap K[\x_{r+1},\ldots,\x_n]=(0)$, then $\gta L[\x]\cap L[\x_{r+1},\ldots,\x_n]=(0)$, where $K[\x_{r+1},\ldots,\x_n]=L[\x_{r+1},\ldots,\x_n]:=(0)$ if $r=n$.
\end{lem}
\begin{proof}
As the case $r=n$ is evident, we can assume $r<n$. Choose a system of generators $\{f_1,\ldots,f_s\}$ of $\gta$ in $K[\x]$. Let $g_0\in\gta L[\x]\cap L[\x_{r+1},\ldots,\x_n]$ and let $g_1,\ldots,g_s\in L[\x]$ be such that $g_0=\sum_{i=1}^sg_if_i$. By Lemma \ref{k0}, we write uniquely $g_i=\sum_{j\in J}u_jg_{ij}$, where $g_{ij}\in K[\x]$ for each $i\in\{0,\ldots,s\}$ and only finitely many of the polynomials $g_{ij}$ are nonzero. By the uniqueness property stated in such lemma, we deduce that $g_{0j}\in K[\x_{r+1},\ldots,\x_n]$ for each $j\in J$, because $g_0\in L[\x_{r+1},\ldots,\x_n]$. Thus,
$$\textstyle
\sum_{j\in J}u_jg_{0j}=g_0=\sum_{i=1}^sg_if_i=\sum_{i=1}^s\big(\sum_{j\in J}u_jg_{ij}\big)f_i=\sum_{j\in J}u_j\big(\sum_{i=1}^sg_{ij}f_i\big).
$$
Using again the uniqueness property, $g_{0j}=\sum_{i=1}^sg_{ij}f_i\in\gta\cap K[\x_{r+1},\ldots,\x_n]=(0)$ for each $j\in J$. We deduce $g_{0j}=0$ for each $j\in J$, so $g_0=0$ and $\gta L[\x]\cap L[\x_{r+1},\ldots,\x_n]=(0)$, as required.
\end{proof}

\begin{cor}\label{dimkl}
If $X\subset L^n$ is a $K$-algebraic set, then $\dim_K(X)=\dim(L[\x]/(\II_K(X)L[\x]))$. If in addition $\II_L(X)=\II_K(X)L[\x]$, then $\dim_K(X)=\dim_L(X)$ and $\hgt(\II_L(X))=\hgt(\II_K(X))$.
\end{cor}
\begin{proof}
If $\II_K(X)=(0)$, then $\dim_K(X)=n=\dim(L[\x]/(\II_K(X)L[\x]))$. Suppose $\II_K(X)\neq(0)$. By Noether's normalization theorem \cite[Thm.3.4.1]{gp} applied to the ideal $\II_K(X)$ of $K[\x]$, up to a linear change of coordinates in $L^n$ induced by an invertible $(n\times n)$-matrix $M$ with coefficients in $K$, there exist an integer $r\in\{1,\ldots,n\}$ and monic polynomials $f_i\in\II_K(X)\cap K[\x_{i+1},\ldots,\x_n][\x_i]\subset \II_K(X)L[\x] \cap L[\x_{i+1},\ldots,\x_n][\x_i]$ for each $i\in\{1,\ldots,r\}$, and $\II_K(X)\cap K[\x_{r+1},\ldots,\x_n]=(0)$. Observe that, to guarantee that the invertible matrix $M$ has coefficients in $K$, we have to use that $K$ is infinite (recall that $K$ has characteristic zero), so {$K$ is $L$-Zariski dense in $L$}.

Let $j_K:K[\x_{r+1},\ldots,\x_n]\to K[\x]/\II_K(X)$ and $j_L:L[\x_{r+1},\ldots,\x_n]\to L[\x]/(\II_K(X)L[\x])$ be the standard homomorphisms. Lemma \ref{intersection} implies that $\II_K(X)L[\x]\cap L[\x_{r+1},\ldots,\x_n]=(0)$, so the homomorphisms $j_K$ and $j_L$ are both injective. The existence of the monic polyno\-mials $f_i$ also implies that the homomorphisms $j_K$ and $j_L$ are finite. Consequently, by \cite[axiom D3, p.219]{e}, we have $\dim(K[\x]/\II_K(X))=\dim(K[\x_{r+1},\ldots,\x_n])=n-r$ and $\dim(L[\x]/(\II_K(X)L[\x]))=\dim(L[\x_{r+1},\ldots,\x_n])=n-r$. Thus, $\dim_K(X)=n-r=\dim(L[\x]/(\II_K(X)L[\x]))$.

If in addition $\II_L(X)=\II_K(X)L[\x]$, then $\dim_K(X)=\dim(L[\x]/\II_L(X))=\dim_L(X)$ and $\hgt(\II_L(X))=n-\dim_L(X)=n-\dim_K(X)=\hgt(\II_K(X))$ by \eqref{eisenbud}.
\end{proof}

In the next result, we show that $\dim_L(X)\leq\dim_K(X)$ for each extension of fields $L|K$ and for each $K$-algebraic set $X\subset L^n$.

\begin{cor}\label{cor:dimLleqdimK}
If $X\subset L^n$ is a $K$-algebraic set, then $\dim_L(X)\leq\dim_K(X)$.
\end{cor}
\begin{proof}
As $\II_K(X)L[\x]\subset\II_L(X)$, we have a surjective homomorphism $L[\x]/(\II_K(X)L[\x])\to L[\x]/\II_L(X)$. Consequently, $\dim(L[\x]/(\II_K(X)L[\x]))\geq\dim(L[\x]/\II_L(X))=\dim_L(X)$. Corollary \ref{dimkl} implies that $\dim_K(X)\geq\dim_L(X)$.
\end{proof}

\begin{remark}\label{rem:dimLleqdimK}
The previous result extends to all subsets of $L^n$: \emph{for each set $S\subset L^n$, it holds $\dim_L(S)\leq\dim_K(S)$.} Indeed, if we set $X:=\zcl_{L^n}^K(S)\subset L^n$, then $\II_K(X)=\II_K(S)$ and $\II_L(X)\subset\II_L(S)$, so $\dim_K(X)=\dim_K(S)$ and $\dim_L(S)\leq\dim_L(X)$. By Corollary \ref{cor:dimLleqdimK}, we deduce: $\dim_L(S)\leq\dim_L(X)\leq\dim_K(X)=\dim_K(S)$, as required. $\sqbullet$
\end{remark}

The relationships between $\dim_L(X)$ and $\dim_K(X)$ of a $K$-algebraic set $X\subset L^n$ will be explored further in Subsection \ref{sdimension}.

Let us study the Zariski closures in~$L^n$ of algebraic subsets of $K^n$.

\begin{prop}\label{prop:zar}
Let $Y\subset K^n$ be an algebraic set and let $X:=\zcl_{L^n}(Y)\subset L^n$. We have:
\begin{itemize}
\item[$(\mr{i})$] $\II_L(X)=\II_K(Y)L[\x]$, $\II_K(X)=\II_K(Y)$ and $X=\ZZ_L(\II_K(Y))$. In particular, $X\subset L^n$ is $K$-algebraic and $X=\zcl_{L^n}^K(Y)$.
\item[$(\mr{ii})$] $X\cap K^n=Y$.
\item[$(\mr{iii})$] If $Y_1,\ldots,Y_s$ are the irreducible components of $Y\subset K^n$ and $X_i:=\zcl_{L^n}(Y_i)$ for each $i\in\{1,\ldots,s\}$, then $X_1,\ldots,X_s$ are both the ($L$-)irreducible components of $X\subset L^n$ and the $K$-irreducible components of $X\subset L^n$. In particular, the family of irreducible components of $X\subset L^n$ coincides with the family of $K$-irreducible components of $X$. In addition, $X\subset L^n$ is irreducible if and only if it is $K$-irreducible or, equivalently, if $Y\subset K^n$ is irreducible.
\item[$(\mr{iii}')$] If $X_1,\ldots,X_s$ are the irreducible components of $X\subset L^n$, then $X_1\cap K^n,\ldots,X_s\cap K^n$ are the irreducible components of $Y\subset K^n$.
\item[$(\mr{iv})$] $\dim_L(X)=\dim_K(X)=\dim_K(Y)$.
\end{itemize}
\end{prop}
\begin{proof}
$(\mr{i})$ If $g\in\II_K(Y)$, then $Y\subset\ZZ_K(g)\subset\ZZ_L(g)$, so $X=\zcl_{L^n}(Y)\subset\ZZ_L(g)$. It follows that $\II_K(Y)\subset\II_L(X)$, so $\II_K(Y)L[\x]\subset\II_L(X)$.

To prove the converse, pick $f\in \II_L(X)$. By Lemma \ref{k0}, we can write $f=\sum_{j\in J}u_jf_j$, where $f_j\in K[\x]$ and only finitely many of the $f_j$ are nonzero. If $x\in Y$, then $f_j(x)\in K$ for each $j\in J$, and $0=f(x)=\sum_{j\in J}u_jf_j(x)$, so $f_j(x)=0$ for each $j\in J$. Thus, each polynomial $f_j$ belongs to $\II_K(Y)$, hence $f=\sum_{j\in J}u_jf_j\in \II_K(Y)L[\x]$. This proves the inclusion $\II_L(X)\subset \II_K(Y)L[\x]$, so the equality $\II_L(X)=\II_K(Y)L[\x]$. By Corollary \ref{k}, we deduce $\II_K(X)=\II_K(Y)$. As $X\subset L^n$ is Zariski closed, $X=\ZZ_L(\II_L(X))=\ZZ_L(\II_K(Y))=\zcl_{L^n}^K(Y)$.

$(\mr{ii})$ It is an immediate consequence of Corollary \ref{LK-zar}.

$(\mr{iii})$ First, suppose $Y\subset K^n$ is irreducible. We have to prove: $X\subset L^n$ is irreducible as well.

Assume there exist algebraic sets $Z_1,Z_2\subset L^n$ such that $Z_1\subsetneqq X$, $Z_2\subsetneqq X$ and $X=Z_1\cup Z_2$. By $(\mr{ii})$, $Y=X\cap K^n=(Z_1\cap K^n)\cup(Z_2\cap K^n)$. By Corollary \ref{cor-cap}, $Z_1\cap K^n$ and $Z_2\cap K^n$ are algebraic subsets of $K^n$. As $Y\subset K^n$ is irreducible, we may assume $Y=Z_1\cap K^n$. Thus, $Z_2\subset X=\zcl_{L^n}(Z_1\cap K^n)\subset Z_1$, which is a contradiction. Consequently, $X \subset L^n$ is irreducible.

Let us complete the proof of $(\mr{iii})$. Let $Y_1,\ldots,Y_s$ and $X_1,\ldots,X_s$ be as in the statement of $(\mr{iii})$. We have proven that $X_i\cap K^n=Y_i$, and $X_i\subset L^n$ is irreducible for each $i$. As $Y=\bigcup_{i=1}^sY_i$, we have
$$\textstyle
X=\zcl_{L^n}\left(\bigcup_{i=1}^sY_i\right)=\bigcup_{i=1}^s\zcl_{L^n}\left(Y_i\right)=\bigcup_{i=1}^sX_i.
$$
As $Y_i\not\subset\bigcup_{j\neq i}Y_j$ for each $i$, it holds $X_i\not\subset\bigcup_{j\neq i}X_j$ for each $i$ as well. Otherwise, $X_i\subset\bigcup_{j\neq i}X_j$ for some $i$, so $Y_i=X_i\cap K^n\subset(\bigcup_{j\neq i}X_j)\cap K^n=\bigcup_{j\neq i}(X_j\cap K^n)=\bigcup_{j\neq i}Y_i$, which is a contradiction. This proves that $X_1,\ldots,X_s$ are the irreducible components of $X\subset L^n$. By $(\mr{i})$, each set $X_i\subset L^n$ is $K$-algebraic. As $X_i\subset L^n$ is irreducible, it is also $K$-irreducible. It follows that $X_1,\ldots,X_s$ are also the $K$-irreducible components of $X\subset L^n$. 

$(\mr{iii}')$ This item is an immediate consequence of $(\mr{ii})$, $(\mr{iii})$ and the uniqueness of irreducible components. 

$(\mr{iv})$ By $(\mr{i})$, we have $\II_L(X)=\II_K(X)L[\x]$ and $\II_K(X)=\II_K(Y)$. The first equality and Corollary \ref{dimkl} imply that $\dim_L(X)=\dim_K(X)$, the second equality that $\dim_K(X)=\dim_K(Y)$, as required.
\end{proof}

\begin{remark}\label{to22}
The equality $\II_L(X)=\II_K(X)L[\x]$ of item $(\mr{i})$ was originally proved by Tognoli in Lemma 1(2) on page 29 of \cite{to2}. Observe that in \cite{to2} the Zariski closure $X=\zcl_{L^n}(Y)$ of $Y$ in $L^n$ is called the completion of $Y$ in $L$, see Definition 1 on page 28 of \cite{to2}. In Lemma 3 on page 29 of the same paper \cite{to2}, Tognoli also proves the first part of $(\mr{iii})$, i.e., the statement `\textit{If $Y_1,\ldots,Y_s$ are the irreducible components of $Y\subset K^n$ and $X_i:=\zcl_{L^n}(Y_i)$ for each $i\in\{1,\ldots,s\}$, then $X_1,\ldots,X_s$ are the ($L$-)irreducible components of $X\subset L^n$}'. $\sqbullet$
\end{remark}

\begin{cor}\label{<=dim}
Let $X\subset L^n$ be an algebraic set and consider the algebraic subset $Y:=X\cap K^n$ of $K^n$. We have:
\begin{itemize}
\item[$(\mr{i})$] $\dim_K(Y)\leq\dim_L(X)$.
\item[$(\mr{ii})$] If $X\subset L^n$ is irreducible and $\dim_K(Y)=\dim_L(X)$, then $Y\subset K^n$ is irreducible and $X=\zcl_{L^n}(Y)$.
 \end{itemize}
\end{cor}
\begin{proof}
Let $X':=\zcl_{L^n}(Y)\subset X$. By Proposition \ref{prop:zar}$(\mr{iv})$, $\dim_K(Y)=\dim_L(X')\leq\dim_L(X)$. This proves $(\mr{i})$. If $\dim_K(Y)=\dim_L(X)$ and $X\subset L^n$ is irreducible, then $\dim_L(X')=\dim_L(X)$, so $X'=X$. Proposition \ref{prop:zar}$(\mr{iii})$ implies that $Y=X'\cap K^n\subset K^n$ is irreducible, as required. 
\end{proof}

\begin{remark}
In statement $(\mr{i})$ of the previous result, the inequality can be strict. Assume $K$ is real closed and $L=K[\ii]$. The $K$-algebraic set $X:=\ZZ_L(\x_1^2+\ldots+\x_n^2+1)\subset L^n$ is a hypersurface of $L^n$, whereas $Y=X\cap R^n$ is the empty set.  $\sqbullet$
\end{remark}

\subsubsection{Extension of coefficients}
We refer the reader to \cite[\S5]{bcr} for the details concerning extension of coefficients for real closed fields. Similar results hold for extensions of algebraically closed fields \cite[Def.3.1.13 \& Prop.3.1.14 \& Thm.3.2.2]{dm}. In Subsection A.1 of Appendix A, we sketch the construction for deriving the `extension of coefficients' procedure for algebraically closed fields from that for real closed fields.

We reformulate Proposition \ref{prop:zar} in a context where the extension of coefficients can be used. First, we give a definition.

\begin{defn}\label{def:L|K-acac-rcrc}
Let $L$ be a field and let $K$ be a subfield of $L$. We say that $L|K$ is an \emph{extension of algebraically closed fields} if $L$ and $K$ are both algebraically closed. Similarly, $L|K$ is an \emph{extension of real closed fields} if $L$ and $K$ are both real closed. $\sqbullet$
\end{defn}

If $L|K$ is an extension of real closed fields, then $K$ is an ordered subfield of $L$, that is, the ordering of $L$ extends the one of $K$. Indeed, if $(R,\leq)$ is a real closed field and $x,y\in R$, then $x\leq y$ if and only if there exists $z\in R$ such that $y-x=z^2$, see \cite[Thm.1.2.2]{bcr}.

\begin{prop}\label{extension-zar}
Suppose that $L|K$ is either an extension of algebraically closed fields or an extension of real closed fields. Let $Y\subset K^n$ be an algebraic set and let $Y_L\subset L^n$ be the extension of coefficients of $Y$ to~$L$. We have:
\begin{itemize}
\item[$(\mr{i})$] $Y_L=\zcl_{L^n}(Y)=\zcl_{L^n}^K(Y)$.
\item[$(\mr{ii})$] $\II_L(Y_L)=\II_K(Y_L)L[\x]$, $\II_K(Y_L)=\II_K(Y)$ and $Y_L\cap K^n=Y$.
\item[$(\mr{iii})$] If $Y_1,\ldots,Y_s$ are the irreducible components of $Y\subset K^n$ and $X_i:=(Y_i)_L$ for each $i\in\{1,\ldots,s\}$, then $X_1,\ldots,X_s$ are both the ($L$-)irreducible components of $X\subset L^n$ and the $K$-irreducible components of $X\subset L^n$. In particular, the family of irreducible components of $Y_L\subset L^n$ coincides with the family of $K$-irreducible components of $Y_L$. In addition, $Y_L\subset L^n$ is irreducible if and only if it is $K$-irreducible or, equivalently, if $Y\subset K^n$ is irreducible.
\item[$(\mr{iii}')$] If $X_1,\ldots,X_s$ are the irreducible components of $Y_L\subset L^n$, then $X_1\cap K^n,\ldots,X_s\cap K^n$ are the irreducible components of $Y\subset K^n$.
\item[$(\mr{iv})$] $\dim_L(Y_L)=\dim_K(Y_L)=\dim_K(Y)$. 
\end{itemize}
\end{prop}
\begin{proof}
By Proposition \ref{prop:zar}, it is enough to show that $Y_L=\zcl_{L^n}(Y)$. Let $X:=\zcl_{L^n}(Y)$ and let $g_1,\ldots,g_r$ be generators of $\II_K(Y)$ in $K[\x]$. By Proposition \ref{prop:zar}$(\mr{i})$ and \cite[Prop.5.1.1]{bcr}, we have $\II_L(X)=(g_1,\ldots,g_r)L[\x]$ and $X=\ZZ_L(g_1,\ldots,g_r)=(\ZZ_K(g_1,\ldots,g_r))_L=Y_L$, as required.
\end{proof}

\begin{cor}\label{inter}
Suppose that $L|K$ is either an extension of algebraically closed fields or an extension of real closed fields. Let $X\subset L^n$ be a $K$-algebraic set. We have:
\begin{itemize}
\item[$(\mr{i})$] $X=\zcl_{L^n}(X\cap K^n)=\zcl_{L^n}^K(X\cap K^n)$.
\item[$(\mr{ii})$] $\II_L(X)=\II_K(X)L[\x]$ and $\II_K(X)=\II_K(X\cap K^n)$.
\item[$(\mr{iii})$] If $X_1,\ldots,X_s$ are the irreducible components of $X\subset L^n$, then $X_1,\ldots,X_s$ are also the $K$-irreducible components of $X$ and $X_1\cap K^n,\ldots,X_s\cap K^n$ are the irreducible components of $X\cap K^n\subset K^n$. 
In particular, $X\subset L^n$ is irreducible if and only if it is $K$-irreducible or, equivalently, if $X\cap K^n\subset K^n$ is irreducible.
\item[$(\mr{iv})$] $\dim_L(X)=\dim_K(X)=\dim_K(X\cap K^n)$.
\end{itemize}
\end{cor}
\begin{proof}
As $X\subset L^n$ is $K$-algebraic, we have $X=(X\cap K^n)_L$ (we use here \cite[Prop.5.1.1]{bcr}), so we can apply Proposition \ref{extension-zar} to $Y:=X\cap K^n$.
\end{proof}

\begin{remark}\label{rem:fermat1}
In the preceding result, we cannot omit the condition `\textit{$L|K$ is either an extension of algebraically closed fields or an extension of real closed fields}'. For instance, if $K=\Q$ and $L$ is any real closed field $R$, Fermat's Last Theorem of Wiles \cite{wi} implies that the $\Q$-algebraic curve $F_k\subset R^2$ of equation $\x_1^{2k}+\x_2^{2k}=2^k$ have no rational points for $k\geq 3$, whereas it has infinitely many real algebraic points. $\sqbullet$
\end{remark}

\subsubsection{Transcendence and dimension} The next result describes a transcendence phenomenon related to the $L$-dimension of the $K$-Zariski closures.

\begin{lem}\label{big}
Suppose that $L$ is either an algebraically closed field or a real closed field. Denote $\kbar^\sqbullet$ the algebraic closure of $K$ in $L$. Let $X\subset L^n$ be an irreducible algebraic set and let $T:=\zcl_{L^n}^K(X)$. If $X$ is not a $\kbar^\sqbullet$-algebraic subset of $L^n$, then $\dim_L(X)<\dim_L(T)$.
\end{lem}
\begin{proof}
Let $Z:=\zcl_{L^n}^{\kbar^\sqbullet}(X)$. By Lemma \ref{lem:irreducibility}, the $\kbar^\sqbullet$-algebraic set $Z\subset L^n$ is $\kbar^\sqbullet$-irreducible. By Corollary \ref{inter}$(\mr{iii})$, the algebraic set $Z\subset L^n$ is irreducible. As $K\subset\kbar^\sqbullet$, we have $Z\subset T$. As $X\subsetneqq Z$ by hypothesis, we deduce $\dim_L(X)<\dim_L(Z)\leq\dim_L(T)$, as required.
\end{proof}

A simple example is the following. Take $K=\Q$ and $L=\R$, denote $\qr$ the field of real algebraic numbers and pick $p\in\R\setminus\qr$. The singleton $X:=\{p\}\subset\R$ is an irreducible algebraic set, which is not $\qr$-algebraic. In this case, $T:=\zcl_\R^\Q(X)=\R$, so $\dim_\R(X)=0<1=\dim_\R(T)$.

\subsubsection{Some well-known results of commutative algebra} We conclude this part with some well-known results from commutative algebra. Let us begin with a result regarding radical ideals.

\begin{lem}[{\cite[Ch.V,\! \S15,\! Prop.5]{b}}]\label{radical}
Let $L|K$ be an extension of fields. If $\gta$ is a radical ideal of $K[\x]$, then $\gta L[\x]=\gta\otimes_KL$ is also a radical ideal of $L[\x]$.
\end{lem}

We deduce a version of Hilbert's Nullstellensatz that will be useful later.

\begin{cor}[{$C|K$-Nullstellensatz}]\label{kreliablec}
Let $C|K$ be an extension of fields such that $C$ is algebraically closed. If $\gta$ is a radical ideal of $K[\x]$, then $\II_C(\ZZ_C(\gta))=\gta C[\x]$ and $\II_K(\ZZ_C(\gta))=\gta$. Equivalently, if $\gtb$ is an arbitrary ideal of $K[\x]$ and $\sqrt{\gtb}$ is its radical in $K[\x]$, then $\II_C(\ZZ_C(\gtb))=\sqrt{\gtb} C[\x]$ and $\II_K(\ZZ_C(\gtb))=\sqrt{\gtb}$.
\end{cor}
\begin{proof}
By Hilbert's Nullstellensatz, $\II_C(\ZZ_C(\gta))=\II_C(\ZZ_C(\gta C[\x]))=\sqrt{\gta C[\x]}$. By Lemma \ref{radical}, $\gta C[\x]=\gta\otimes_KC$ is a radical ideal of $C[\x]$, so $\II_C(\ZZ_C(\gta))=\gta C[\x]$. By Corollary \ref{k}, we have $\II_K(\ZZ_C(\gta))=\gta$. In particular, if we set $\gta:=\sqrt{\gtb}$, we deduce $\II_C(\ZZ_C(\gtb))=\II_C(\ZZ_C(\sqrt{\gtb}))=\sqrt{\gtb}C[\x]$ and $\II_K(\ZZ_C(\gtb))=\II_K(\ZZ_C(\sqrt{\gtb}))=\sqrt{\gtb}$, as required.
\end{proof}

The next result will also be useful later.

\begin{lem}\label{extension}
Let $C|K$ be an extension of fields such that $C$ is algebraically closed. Then each automorphism of $K$ extends to an automorphism of $C$.
\end{lem}
\begin{proof}
For the sake of completeness, we provide a proof. Let $\sigma:K\to K$ be an automorphism and let $S:=\{s_h\}_{h\in H}$ be a transcendence basis of $C$ over $K$. Consider the subring $K[S]$ of $C$. For each $r\in\N$ and for each polynomial $P(\y)=\sum_\nu a_\nu\y^\nu$ in $K[\y_1,\ldots,\y_r]$, define $P^\sigma(\y):=\sum_\nu \sigma(a_\nu)\y^\nu\in K[\y_1,\ldots,\y_r]$, where $K[\y_1,\ldots,\y_r]:=K$ if $r=0$. Consider the map $\phi_\sigma:K[S]\to K[S]$ given by: $\phi_\sigma(P(s_{h_1},\ldots,s_{h_r})):=P^\sigma(s_{h_1},\ldots,s_{h_r})$ for each finite subset $\{s_{h_1},\ldots,s_{h_r}\}$ of $S$ and each $P\in K[\y_1,\ldots,\y_r]$. As the elements of $S$ are algebraically independent over $K$, we deduce that $\phi_\sigma$ is a well-defined automorphism that extends $\sigma$ and its inverse is $\phi_{\sigma^{-1}}$. The automorphism $\phi_\sigma$ of $K[S]$ extends to the automorphism $\phi'_\sigma$ of the field of fractions $K(S)$ of $K[S]$ defined by $\phi'_\sigma(ab^{-1}):=\phi_\sigma(a)(\phi_\sigma(b))^{-1}$ for all $a,b\in K[S]$ with $b\neq0$. As $C$ is the algebraic closure of $K(S)$, the uniqueness of algebraic closure guarantees the existence an automorphism $\Phi_\sigma$ of $C$ that extends $\phi'_\sigma$ (see \cite[Thms.32\&33, pp.106-108]{zs1} for a proof). Thus, $\Phi_\sigma|_K=\sigma$, as required.
\end{proof}

\begin{remark}
If the Galois group $G(C:K)$ of $C|K$ is not trivial, the previous extension $\Phi_\sigma$ is not unique. Pick $\Psi\in G(C:K)\setminus\{\id_C\}$ and observe that $(\Phi_\sigma\circ\Psi)|_K=\sigma$, whereas $\Phi_\sigma\circ\Psi\neq\Phi_\sigma$. $\sqbullet$
\end{remark}

\subsection{$K$-Zariski closure, Galois completion and its algorithmic computation}\label{sgc0}

\emph{Along this subsection, $C|K$ is an extension of fields such that $C$ is algebraically closed, $G$ is the Galois group $G(C:K)$ and $\kbar$ is the algebraic closure of $K$ in $C$.} Observe that $\kbar$ coincides with the algebraic closure of $K$, and the full Galois group $G(\kbar:K)$ of $K$ is isomorphic to $G(C:K)/G(C:\kbar)$.

\subsubsection{Galois completion of a complex algebraic set} 

Our next purpose is to study the $K$-Zariski closure of an algebraic set $X\subset C^n$. Some restrictions concerning $X$ are needed in order to have satisfactory results. Obviously, we have to avoid the transcendence phenomenon described in Lemma \ref{big}. Thus, we focus on the $K$-Zariski closure of $\kbar$-algebraic subsets of $C^n$. We characterize such Zariski closures by means of the Galois group $G$ and we provide an algorithm to compute them involving certain finite Galois subextensions of $\kbar|K$.

For each automorphism (of fields) $\psi:C\to C$, define the isomorphism (of $\Q$-vector spaces) $\psi_n:C^n\to C^n$ and the ring automorphism $\widehat{\psi}:C[\x]\to C[\x]$ by
\begin{align}
&\psi_n(z_1,\ldots,z_n):=(\psi(z_1),\ldots,\psi(z_n)),\label{psi}\\
&\textstyle\widehat{\psi}\big(\sum_\nu a_\nu\x^\nu\big):=\sum_\nu\psi(a_\nu)\x^\nu.\label{hatpsi}
\end{align}

\begin{defn}\label{gc0def}
Let $S$ be a subset of $C^n$. We call the \emph{Galois completion of $S\subset C^n$ (with respect to the extension of fields $C|K$)} the subset $\bigcup_{\psi\in G}\psi_n(S)$ of $C^n$. $\sqbullet$
\end{defn}

\begin{defn}\label{def:sigma}
Let $E|K$ be any extension of fields and let $G(E:K)$ be the Galois group of $E|K$. Given $\sigma\in G(E:K)$ and $g=\sum_\nu a_\nu\x^\nu\in E[\x]$, we define $g^\sigma:=\sum_\nu\sigma(a_\nu)\x^\nu\in E[\x]$. $\sqbullet$
\end{defn}

\begin{alg}\label{gc}
The algorithm works as follows:
\begin{itemize}
\item[(0)] Start with a $\kbar$-algebraic set $X\subset C^n$.
\item[(1)] Choose polynomials $g_1,\ldots,g_r\in\kbar[\x]$ such that $X=\ZZ_C(g_1,\ldots,g_r)$.
\item[(2)] Choose any finite Galois subextension $E|K$ of $\kbar|K$ that contains all the coefficients of the polynomials $g_1,\ldots,g_r$ and set $G':=G(E:K)$.
\item[(3)] For each $\sigma\in G'$, define $Z^\sigma:=\ZZ_C(g_1^\sigma,\ldots,g_r^\sigma)\subset C^n$.
\item[(4)] Consider the $\kbar$-algebraic set $T:=\bigcup_{\sigma\in G'}Z^\sigma$.
\end{itemize}
\end{alg}

The following result shows as the reader can expect that the output $T$, the Galois completion $\bigcup_{\psi\in G}\psi_n(X)$ of $X\subset C^n$ and the $K$-Zariski closure of $X$ in $C^n$ coincide. In fact, it provides a procedure to compute $\II_K(T)$ from a finite collection of polynomials in $\kbar[\x]$ whose zero set is~$X$.

\begin{thm}[Galois completion and $K$-Zariski closure]\label{thm:gc0}
Let $X\subset C^n$ be a $\kbar$-algebraic set and let $T\subset C^n$ be a $\kbar$-algebraic set obtained applying Algorithm {\em \ref{gc}}. For each $\sigma\in G'$, choose an automorphism $\Phi_\sigma:C\to C$ such that $\Phi_\sigma|_E=\sigma$ (see Lemma \emph{\ref{extension}}).

We have:
\begin{itemize}
\item[$(\mr{i})$] $\psi_n(X)=\ZZ_C(\widehat{\psi}(g_1),\ldots,\widehat{\psi}(g_r))$ and $\II_C(\psi_n(X))=\widehat{\psi}(\II_C(X))$ for each $\psi\in G$. In particular, $Z^\sigma=\Phi_{\sigma,n}(X)$ and $\II_C(Z^\sigma)=\widehat{\Phi}_\sigma(\II_C(X))$ for each $\sigma\in G'$, where $\Phi_{\sigma,n}:=(\Phi_\sigma)_n$.
\item[$(\mr{ii})$] $\{\widehat{\psi}(g_i):\,\psi\in G\}=\{\widehat{\Phi}_\sigma(g_i):\, \sigma\in G'\}$ for each $i\in\{1,\ldots,r\}$, $\{\psi_n(X):\, \psi\in G\}=\{Z^\sigma:\,\sigma\in G'\}$ and $T=\bigcup_{\sigma\in G'}Z^\sigma=\bigcup_{\psi\in G}\psi_n(X)$. In particular, $T\subset C^n$ is the Galois completion of $X\subset C^n$.
\item[$(\mr{ii}')$] Given any $\sigma\in G'$, we have $\dim_C(Z^\sigma)=\dim_C(X)$. Moreover, $Z^\sigma\subset C^n$ is irreducible if and only if $X\subset C^n$ is $\kbar$-irreducible (or, equivalently, if $X\subset C^n$ is irreducible).
\item[$(\mr{iii})$] Let ${\mathfrak H}\subset(g_1,\ldots,g_r)\kbar[\x]\cap E[\x]$ be the set of all products of the form $\prod_{\sigma\in G'}h_\sigma$, where $h_\sigma\in\{g_1^\sigma,\ldots,g_r^\sigma\}$ for each $\sigma\in G'$. Then $T=\ZZ_C({\mathfrak H})$ and $\widehat{\psi}({\mathfrak H})={\mathfrak H}$ for each $\psi\in G$.
\item[$(\mr{iv})$] For each $h\in{\mathfrak H}$, define
$$
P_h(\t):=\prod_{\tau\in G'}(\t-h^\tau)=\t^d+\sum_{j=1}^d(-1)^jq_{hj}\t^{d-j}\in E[\x][\t],
$$
where $d$ is the order of $G'$ and set ${\mathfrak G}:=\{q_{hj}\}_{h\in{\mathfrak H},\,j\in\{1,\ldots,d\}}$. Then ${\mathfrak G}$ is a subset of $K[\x]$ and $T=\ZZ_C({\mathfrak G})$.
\item[$(\mr{v})$] $T=\zcl^K_{C^n}(X)=\ZZ_C((g_1,\ldots,g_r)\kbar[\x]\cap K[\x])$.
\item[$(\mr{vi})$] $\II_K(T)=\II_K(X)=\sqrt{{\mathfrak G}K[\x]}$. Moreover, $\II_C(T)=\II_K(X)C[\x]$ and 
$$
\dim_C(X)=\dim_C(T)=\dim_K(T)=\dim_K(X).
$$
\end{itemize}
\end{thm}
\begin{proof}
$(\mr{i})$ Observe that $\psi(g(x))=\widehat{\psi}(g)(\psi_n(x))$ for each $x\in C^n$, $g\in C[\x]$ and $\psi\in\ G$. As $\psi$ and $\widehat{\psi}$ are automorphisms, we deduce that $\II_C(\psi_n(X))=\widehat{\psi}(\II_C(X))$. Moreover, $\psi_n(x)\in\ZZ_C(\widehat{\psi}(g_1),\ldots,\widehat{\psi}(g_r))$ if and only if $x\in \ZZ_C(g_1,\ldots,g_r)=X$, so $\psi_n(X)=\ZZ_C(\widehat{\psi}(g_1),\ldots,\widehat{\psi}(g_r))$.

The second part of statement $(\mr{i})$ follows readily.

$(\mr{ii})$ As $E|K$ is a Galois extension, each automorphism $\psi\in G$ restricts to an element $\sigma:=\psi|_E$ of $G'$. Conversely, by Lemma \ref{extension}, each automorphism $\sigma\in G'$ extends to an automorphism $\psi_\sigma\in G$. We deduce that the set $\{\widehat{\psi}(g_i):\,\psi\in G\}$ coincides with the set $\{\widehat{\psi}_\sigma(g_i):\,\sigma\in G'\}=\{g_i^\sigma:\,\sigma\in G'\}$ for each $i\in\{1,\ldots,r\}$. By $(\mr{i})$, the sets $\{\psi_n(X):\, \psi\in G\}$ and $\{Z^\sigma:\,\sigma\in G'\}$ coincide and are finite, so $T=\bigcup_{\sigma\in G'}Z^\sigma=\bigcup_{\psi\in G}\psi_n(X)$ is a $\kbar$-algebraic subset of $C^n$.

$(\mr{ii}')$ As $\II_C(Z^\sigma)=\widehat{\Phi}_\sigma(\II_C(X))$ by $(\mr{i})$, and $\widehat{\Phi}_\sigma$ is an automorphism of $C[\x]$, the quotient rings $C[\x]/\II_C(Z^\sigma)$ and $C[\x]/\II_C(X)$ are isomorphic and have the same dimension. In addition, the ideal $\II_C(Z^\sigma)$ of $C[\x]$ is prime if and only if so is $\II_C(X)$. Thus, $Z^\sigma\subset C^n$ is irreducible if and only if so is $X\subset C^n$. By Corollary \ref{inter}$(\mr{iii})$, the latter condition is equivalent to the $\kbar$-irreducibility of $X\subset C^n$.

$(\mr{iii})$ As $Z^\sigma=\ZZ_C(g_1^\sigma,\ldots,g_r^\sigma)=\ZZ_C((g_1^\sigma,\ldots,g_r^\sigma)\kbar[\x])$, we have
$$\textstyle
T=\bigcup_{\sigma\in G'}Z^\sigma=\bigcup_{\sigma\in G'}\ZZ_C((g_1^\sigma,\ldots,g_r^\sigma)\kbar[\x])=\ZZ_C\big(\prod_{\sigma\in G'}(g_1^\sigma,\ldots,g_r^\sigma)\kbar[\x]\big).
$$
The product ideal $\prod_{\sigma\in G'}(g_1^\sigma,\ldots,g_r^\sigma)\kbar[\x]$ of $\kbar[\x]$ is generated by $\mathfrak{H}\subset E[\x]$, so $T=\ZZ_C({\mathfrak H})$. Pick $h\in\mathfrak{H}$. By definition of $\mathfrak{H}$, there exists a function $q:G'\to\{1,\ldots,r\}$ such that $h=\prod_{\sigma\in G'}g^\sigma_{q(\sigma)}$. Denote $e$ the identity isomorphism on $E$. As $g_i^e=g_i$ for each $i\in\{1,\ldots,r\}$, we deduce $h\in(g_1,\ldots,g_r)\kbar[\x]$. Fix $\psi\in G$ and define the translation map $\Psi:G'\to G'$ by $\Psi(\sigma):=\psi|_E\circ\sigma$. Observe that $\Psi$ is bijective and its inverse is the translation map $\Psi^{-1}:G'\to G'$, $\sigma\mapsto\psi^{-1}|_E\circ\sigma$. Define the function $Q:G'\to\{1,\ldots,r\}$ by $Q:=q\circ\Psi^{-1}$. We have: $\widehat{\psi}(h)=\prod_{\sigma\in G'}g^{\Psi(\sigma)}_{q(\sigma)}=\prod_{\sigma\in G'}g^{\Psi(\sigma)}_{Q(\Psi(\sigma))}=\prod_{\tau\in G'}g^{\tau}_{Q(\tau)}\in\mathfrak{H}$. We conclude that $\widehat{\psi}(\mathfrak{H})=\mathfrak{H}$ for each $\psi\in G$.

$(\mr{iv})$ Fix $\sigma\in G'$. Let $\sigma^*$ be the unique automorphism of $E[\x,\t]$ that extends $\sigma$ and satisfies $\sigma^*(\t)=\t$ and $\sigma^*(\x_i)=\x_i$ for each $i\in\{1,\ldots,n\}$. Let us prove: {\em $\sigma^*(P_h)=P_h$ for each $h\in{\mathfrak H}$}.

As the translation map $G'\to G',\ \tau\mapsto\sigma\circ\tau$ is bijective, we have
$$\textstyle
\sigma^*(P_h)=\prod_{\tau\in G'}(\t-\sigma^*(h^\tau))=\prod_{\tau\in G'}(\t-h^{\sigma\circ\tau})=P_h.
$$
Consequently, $P_h(\t)=\prod_{\tau\in G'}(\t-h^\tau)\in K[\x][\t]$ for each $h\in{\mathfrak H}$. 

As $\widehat{\psi}(\mathfrak{H})=\mathfrak{H}$ for each $\psi\in G$, the coefficients of $P_h$ corresponding to $\t^{d-j}$ for $j\geq1$ belongs to the ideal ${\mathfrak H}\kbar[\x]$ of $\kbar[\x]$, so the (finite) set ${\mathfrak G}$ constituted by all such coefficients is contained in ${\mathfrak H}\kbar[\x]\cap K[\x]\subset(g_1,\ldots,g_r)\kbar[\x]\cap K[\x]$. As $P_h(h)=0$, we have $h^d\in{\mathfrak G}\kbar[\x]$ for each $h\in{\mathfrak H}$. It follows that ${\mathfrak G}\kbar[\x]\subset{\mathfrak H}\kbar[\x]\subset\sqrt{{\mathfrak G}\kbar[\x]}$. Thus, $T=\ZZ_C({\mathfrak H})$ coincides with the $K$-algebraic set $\ZZ_C({\mathfrak G})\subset C^n$.

$(\mr{v})$ As $T\subset C^n$ is a $K$-algebraic set containing $X$, we deduce $\zcl_{C^n}^K(X)\subset T$.

Let $f_1,\ldots,f_\ell\in K[\x]$ be such that $\zcl_{C^n}^K(X)=\ZZ_C(f_1,\ldots,f_\ell)$. If $\psi\in G$ and $j\in\{1,\ldots,\ell\}$, then $0=\psi(f_j(x))=f_j(\psi_n(x))$ for all $x\in X$. Thus, $T=\bigcup_{\psi\in G}\psi_n(X)\subset\zcl_{C^n}^K(X)$, so $T=\zcl_{C^n}^K(X)$. 

As $g_1,\ldots,g_r\in\II_\kbar(X)$, we have ${\mathfrak G}\subset(g_1,\ldots,g_r)\kbar[\x]\cap K[\x]\subset\II_{\kbar}(X)\cap K[\x]=\II_K(X)$, so
\begin{align*}
T=\zcl_{C^n}^K(X)=\ZZ_C(\II_K(X))\subset\ZZ_C((g_1,\ldots,g_r)\kbar[\x]\cap K[\x])\subset\ZZ_C({\mathfrak G})=T,
\end{align*}
so $T=\ZZ_C((g_1,\ldots,g_r)\kbar[\x]\cap K[\x])$.

$(\mr{vi})$ As $T=\zcl_{C^n}^K(X)$, we have $\II_K(T)=\II_K(X)$ so $\dim_K(T)=\dim_K(X)$. Define $\gta:=\sqrt{{\mathfrak G}K[\x]}$. As $T=\ZZ_C({\mathfrak G})=\ZZ_C(\gta)$, we deduce by Corollary \ref{kreliablec} that $\II_C(T)=\gta C[\x]$ and $\II_K(T)=\gta$, so $
\II_C(T)=\II_K(T)C[\x]$. By Corollary \ref{dimkl}, we have $\dim_C(T)=\dim_K(T)$. By $(\mr{ii})$\&$(\mr{ii}')$, $\dim_C(T)=\max_{\sigma\in G'}\{\dim_C(Z^\sigma)\}=\dim_C(X)$, as required.
\end{proof}

\begin{defn}
A set $S\subset C^n$ is \em $G$-invariant \em if $\psi_n(S)=S$ for each $\psi\in G$. $\sqbullet$
\end{defn}

Theorem \ref{thm:gc0} has the following immediate consequence:

\begin{cor}[Complex $K$-algebraicity and $G$-invariance]\label{cor:invariance}
A $\kbar$-algebraic subset of $C^n$ is $K$-algebraic if and only if it is $G$-invariant.
\end{cor}

We highlight the following two dimensional consequences of Theorem \ref{thm:gc0}.

\begin{cor}\label{kdim}
If $X\subset C^n$ is a $\kbar$-algebraic set (for instance, a $K$-algebraic set), then $\dim_C(X)=\dim_K(X)$.
\end{cor}

\begin{cor}\label{dimalg}
Let $L|K$ be an algebraic extension and let $X\subset L^n$ be an algebraic set. Then $\dim_L(X)=\dim_K(X)$.
\end{cor}
\begin{proof}
First observe that $\ol{L}=\kbar$. Let $Z$ be the Zariski closure of $X$ in $\ol{L}^n=\kbar^n$. By Proposition \ref{prop:zar}$(\mr{iv})$, we have $\dim_L(X)=\dim_\kbar(Z)$. As $Z\subset\kbar^n$ is a $\kbar$-algebraic set, Corollary \ref{kdim} assures that $\dim_\kbar(Z)=\dim_K(Z)$. As $X\subset Z\subset\zcl_{\kbar^n}^K(X)\subset\kbar^n$, we have $\II_K(Z)=\II_K(X)$, so $\dim_K(Z)=\dim_K(X)$. We conclude $\dim_L(X)=\dim_\kbar(Z)=\dim_K(Z)=\dim_K(X)$, as required.
\end{proof}

\subsubsection{Simultaneous Galois completion}\label{sgc}
Observe that Algorithm \ref{gc} can be applied simultaneously to finitely many $\kbar$-algebraic sets $X_1,\ldots,X_s\subset C^n$. To that end, we take a finite Galois extension $E|K$ that contains the coefficients of polynomials $\{g_{j1},\ldots,g_{jr_j}\}_{j\in\{1,\ldots,s\}}\subset\kbar[\x]$ such that $X_j=\ZZ_C(g_{j1},\ldots,g_{jr_j})$ for each $j\in\{1,\ldots,s\}$. Thus, the previous Algorithm \ref{gc} provides us $T_1:=\zcl_{C^n}^K(X_1),\ldots,T_s:=\zcl_{C^n}^K(X_s)$ simultaneously.

\subsubsection{Galois presentation of a complex $K$-algebraic set}\label{gp0}
Recall that $C|K$ is an extension of fields such that $C$ is algebraically closed. Let $X\subset C^n$ be a $K$-algebraic set. We want to find a `minimal' algebraic set $Y\subset C^n$ whose $K$-Zariski closure is $X$. 
 
\begin{lem}\label{lem:gp0}
Let $X\subset C^n$ be a $K$-irreducible $K$-algebraic set and let $Y\subset C^n$ be a $C$-irredu\-cible component of $X$. We have:
\begin{itemize}
\item[$(\mr{i})$] $Y\subset C^n$ is a $\kbar$-irreducible component of $X$.
\item[$(\mr{ii})$] Apply Algorithm \emph{\ref{gc}} to $Y\subset C^n$. Choose polynomials $g_1,\ldots,g_r\in\kbar[\x]$ such that $Y=\ZZ_C(g_1,\ldots,g_r)$ and a finite Galois subextension $E|K$ of $\kbar|K$ that contains all the coefficients of $g_1,\ldots,g_r$. Consider the finite Galois group $G':=G(E:K)$, the family $\{Z^\sigma:=\ZZ_C(g_1^\sigma,\ldots,g_r^\sigma)\}_{\sigma\in G'}$ of algebraic subsets of $C^n$ and the Galois completion $T:=\bigcup_{\sigma\in G'}Z^\sigma$ of $Y$. Then $T=\zcl_{C^n}^K(Y)=X$.
\item[$(\mr{iii})$] The family $\{Z^\sigma\}_{\sigma\in G'}$ coincides with the one of all $C$-irreducible components of $X$. Moreover, $\dim_C(Z^\sigma)=\dim_C(X)$ for each $\sigma\in G'$.
\end{itemize}
\end{lem}
\begin{proof}
$(\mr{i})$ As $X\subset C^n$ is $K$-algebraic, it is also $\kbar$-algebraic so $Y\subset C^n$ is a $\kbar$-irreducible component of $X$ by Corollary \ref{inter}$(\mr{iii})$. In particular, as $Y\subset C^n$ is $\kbar$-algebraic, we can apply Algorithm \ref{gc} to $Y$.

$(\mr{ii})\ \&\ (\mr{iii})$ Let $Y_1$ be a $C$-irreducible component of $X$ of dimension $d:=\dim_C(X)$. By $(\mr{i})$, $Y_1\subset C^n$ is a $\kbar$-irreducible $\kbar$-algebraic set. In particular, we can apply Algorithm \ref{gc} to $Y_1$ and obtain: a finite Galois group $G'_1$, a family $\{Z^\sigma_1\}_{\sigma\in G'_1}$ of algebraic subsets of $C^n$ and the Galois completion $T_1$ of $Y_1$. By Theorem \ref{thm:gc0}$(\mr{v})(\mr{vi})$, $T_1=\zcl_{C^n}^K(Y_1)\subset X$ and $d=\dim_C(Y_1)=\dim_K(T_1)$. Corollary \ref{kdim} implies that $d=\dim_K(X)$, so $\dim_K(T_1)=\dim_K(X)$ and Lemma \ref{dimirred} implies that $T_1=X$. By Theorem \ref{thm:gc0}$(\mr{ii}')$, $\{Z^\sigma_1\}_{\sigma\in G'_1}$ is the family of all $C$-irreducible components of $X$ and all such components have the same dimension $d$. In particular, $Y\subset C^n$ has dimension $d$, so we can apply the preceding argument for $Y_1:=Y$.
\end{proof}

\begin{cor}\label{cho}
Let $X\subset C^n$ be a $K$-reducible $K$-algebraic set and let $X_1,\ldots,X_s$ be the $K$-irreducible components of $X$. 
Choose any index $i\in\{1,\ldots,s\}$ and a $C$-irreducible component $Y$ of $X_i\subset C^n$. Then $Y\not\subset\bigcup_{j\in\{1,\ldots,s\}\setminus\{i\}}X_j$.
\end{cor}
\begin{proof}
If $Y\subset\bigcup_{j\in\{1,\ldots,s\}\setminus\{i\}}X_j$, then Lemma \ref{lem:gp0}$(\mr{ii})$ would imply that $X_i=\zcl_{C^n}^K(Y)\subset\bigcup_{j\in\{1,\ldots,s\}\setminus\{i\}}X_j$, which is a contradiction.
\end{proof}

We introduce now the concept of Galois presentation of a $K$-algebraic subset of $C^n$.

\begin{defn}\label{def:gp0}
Let $X\subset C^n$ be a $K$-algebraic set and let $(X_1,\ldots,X_s)$ be the $K$-irreducible components of $X$ listed in some order. Choose a $C$-irreducible component $Y_i$ of $X_i$ for each $i\in\{1,\ldots,s\}$. Apply Algorithm \ref{gc} simultaneously to $Y_1,\ldots,Y_s$ and obtain: a finite Galois group $G'$ and finite families $\{Z_1^\sigma\}_{\sigma\in G'},\ldots,\{Z_s^\sigma\}_{\sigma\in G'}$ of algebraic subsets of $C^n$ (such that $X_i=\bigcup_{\sigma\in G'}Z_i^\sigma$ for each $i\in\{1,\ldots,s\}$ by Lemma \ref{lem:gp0}$(\mr{ii})$). We call the tuple
$$
(Y_1,\ldots,Y_s;G';\{Z_1^\sigma\}_{\sigma\in G'},\ldots,\{Z_s^\sigma\}_{\sigma\in G'})
$$
a \emph{Galois presentation of $X\subset C^n$} and $(Y_1,\ldots,Y_s)$ the \emph{start} of such a presentation. For simplicity, we say that $X=\bigcup_{i=1}^s\bigcup_{\sigma\in G'}Z_i^\sigma$ is a Galois presentation of $X\subset C^n$ with start $(Y_1,\ldots,Y_s):=(Z_1^e,\ldots,Z_s^e)$, where $e$ is the identity of $G'$. $\sqbullet$
\end{defn}

\subsubsection{Complex clustering phenomenon} We present two more consequences of Lemma~\ref{lem:gp0}.

\begin{lem}\label{lem:238}
Let $X\subset C^n$ be a $K$-irreducible $K$-algebraic set, let $C|E|K$ be an extension of fields, let $\gtp_1,\ldots,\gtp_t\subset E[\x]$ be the minimal prime ideals of $\II_E(X)$ and let $W_i:=\ZZ_C(\gtp_i)\subset C^n$ for each $i\in\{1,\ldots,t\}$. We have:
\begin{itemize}
\item[$(\mr{i})$] $\II_E(X)=\II_K(X)E[\x]$.
\item[$(\mr{ii})$] $W_1,\ldots,W_t$ are the $E$-irreducible components of $X\subset C^n$.
\item[$(\mr{iii})$] Let $Y_1,\ldots,Y_s$ be the $C$-irreducible components of $X\subset C^n$. Then $t\leq s$ and there exists an injective map $\eta:\{1,\ldots,t\}\to\{1,\ldots,s\}$ such that $Y_{\eta(i)}$ is a $C$-irreducible component of $W_i\subset C^n$ for each $i\in\{1,\ldots,t\}$.
\item[$(\mr{iv})$] $\dim_E(W_i)=\dim_C(W_i)=\dim_C(X)=\dim_K(X)$ for each $i\in\{1,\ldots,t\}$.
\item[$(\mr{v})$] $\II_E(W_i)=\gtp_i$ and $\hgt(\gtp_i)=\hgt(\II_K(X))$ for each $i\in\{1,\ldots,t\}$.
\end{itemize}
\end{lem}
\begin{proof}
By Corollary \ref{kreliablec}, we have $\II_C(X)=\II_K(X)C[\x]$. Corollary \ref{k} implies
$$
\II_E(X)=\II_C(X)\cap E[\x]=\II_K(X)C[\x]\cap E[\x]=(\II_K(X)E[\x])C[\x]\cap E[\x]=\II_K(X)E[\x].
$$
As $\gtp_i$ is a prime ideal of $E[\x]$, $\gtp_i$ is also a radical ideal of $E[\x]$, so $\II_C(W_i)=\gtp_iC[\x]$ and $\II_E(W_i)=\II_C(W_i)\cap E[\x]=\gtp_iC[\x]\cap E[\x]=\gtp_i$. Moreover, the $E$-algebraic set $W_i\subset C^n$ is $E$-irreducible. As $ \II_E(X)\subset E[\x]$ is radical and $X\subset C^n$ is $E$-algebraic, it follows that $\II_E(X)=\bigcap_{i=1}^t\gtp_i$ and $X=\ZZ_C(\II_E(X))=\bigcup_{i=1}^tW_i$. Observe that $W_i\not\subset W_j$ for each $i,j\in\{1,\ldots,t\}$ with $i\neq j$. Otherwise, $W_i\subset W_j$ and $\gtp_i=\II_E(W_i)\supset\II_E(W_j)=\gtp_j$, which is a contradiction. This proves that $W_1,\ldots,W_t$ are the $E$-irreducible components of $X\subset C^n$.

For each $i\in\{1,\ldots,t\}$, let $W_{i,1},\ldots,W_{i,T_i}$ be the $C$-irreducible components of $W_i\subset C^n$. As $\bigcup_{i=1}^t\bigcup_{\ell=1}^{T_i}W_{i,\ell}=X$, by the uniqueness of the decomposition into $C$-irreducible components of $X=\bigcup_{j=1}^sY_j$, we have that $\{Y_1,\ldots,Y_s\}\subset\bigcup_{i=1}^t\{W_{i,1},\ldots,W_{i,T_i}\}$.

Let us prove: \emph{for each $i\in\{1,\ldots,t\}$, there exist $j_i\in\{1,\ldots,s\}$ and $\ell_i\in\{1,\ldots,T_i\}$ such that $W_{i,\ell_i}=Y_{j_i}$}. Suppose this is not true. Then there exists $I\in\{1,\ldots,t\}$ such that $\{Y_1,\ldots,Y_s\}\cap\{W_{I,1},\ldots,W_{I,t_I}\}=\varnothing$, so $\{Y_1,\ldots,Y_s\}\subset\bigcup_{i\in\{1,\ldots,t\}\setminus\{I\}}\{W_{i,1},\ldots,W_{i,T_i}\}$ and
$$
\textstyle
X=\bigcup_{j=1}^sY_j\subset\bigcup_{i\in\{1,\ldots,t\}\setminus\{I\}}\bigcup_{\ell=1}^{T_i}W_{i,\ell}=\bigcup_{i\in\{1,\ldots,t\}\setminus\{I\}}W_i\subsetneqq X,
$$
which is a contradiction.

Define the function $\eta:\{1,\ldots,t\}\to\{1,\ldots,s\}$ as $\eta(i):=j_i$. Observe that $\eta$ is injective: indeed, if $i$ and $i'$ are indexes in $\{1,\ldots,t\}$ with $j_i=j_{i'}$, then $W_{i,\ell_i}=Y_{j_i}=Y_{j_{i'}}=W_{i',\ell_{i'}}$ for some $\ell_i\in\{1,\ldots,T_i\}$ and $\ell_{i'}\in\{1,\ldots,T_{i'}\}$, and Lemma \ref{lem:gp0}$(\mr{ii})$ implies that $W_i=\zcl_{C^n}^E(W_{i,\ell_i})=\zcl_{C^n}^E(W_{i',\ell_{i'}})=W_{i'}$, so $i=i'$.

By Lemma \ref{lem:gp0}$(\mr{iii})$, we have $\dim_C(W_{i,\ell_i})=\dim_C(Y_{j_i})=\dim_C(X)$ for each $i\in\{1,\ldots,t\}$. As $W_{i,\ell_i}\subset W_i\subset X$, it follows $\dim_C(W_i)=\dim_C(X)$. By Corollary \ref{kdim}, $\dim_C(W_i)=\dim_E(W_i)$ and $\dim_C(X)=\dim_K(X)$. Thus, $\dim_E(W_i)=\dim_C(W_i)=\dim_C(X)=\dim_K(X)$. Consequently, by \eqref{eisenbud}, $\hgt(\gtp_i)=\hgt(\II_E(W_i))=n-\dim_E(W_i)=n-\dim_K(X)=\hgt(\II_K(X))$, as required.
\end{proof}

\begin{thm}[{$C|E|K$-clustering phenomenon}] \label{thm:238}
Let $X\subset C^n$ be a $K$-irreducible $K$-algebraic set, let $Y_1,\ldots,Y_s$ be the $C$-irreducible components of $X\subset C^n$, let $C|E|K$ be an extension of fields and let $W_1,\ldots,W_t$ be the $E$-irreducible components of $X\subset C^n$. Then $t\leq s$ and there exists a partition of $\{1,\ldots,s\}=J_1\sqcup\cdots\sqcup J_t$ such that $\{Y_j:j\in J_i\}$ is the family of $C$-irreducible components of $W_i\subset C^n$ for each $i\in\{1,\ldots,t\}$.
\end{thm}
\begin{proof}
For each $i\in\{1,\ldots,t\}$, let $W_{i,1},\ldots,W_{i,T_i}$ be the $C$-irreducible components of $W_i\subset C^n$. By Lemma \ref{lem:238}$(\mr{iii})$, $t\leq s$ and there exists an injective map $\{1,\ldots,t\}\to\{1,\ldots,s\},\,i\mapsto j_i$ such that $Y_{j_i}=W_{i,\ell_i}$ for some $\ell_i\in\{1,\ldots,T_i\}$. Rearranging the indices if necessary, we can assume that $j_i=i$, that is, $Y_i$ is a $C$-irreducible component of the $E$-algebraic set $W_i\subset C^n$ for each $i\in\{1,\ldots,t\}$.

Fix $i\in\{1,\ldots,t\}$. By Lemma \ref{lem:gp0}$(\mr{i})$, we apply Algorithm \ref{gc} to $Y_i\subset C^n$ with respect to both points of view: $Y_i\subset C^n$ is both a $\kbar$-irreducible component of the $K$-irreducible $K$-algebraic set $X\subset C^n$ and a $\ove$-irreducible component of the $E$-irreducible $E$-algebraic set $W_i\subset C^n$. Choose polynomials $g_1,\ldots,g_r\in\kbar[\x]\subset\ove[\x]$ such that $Y_i=\ZZ_C(g_1,\ldots,g_r)$, a finite Galois subextension $E_1|K$ of $\kbar|K$ that contains all the coefficients of $g_1,\ldots,g_r$ and a finite Galois subextension $E_2|E$ of $\ove|E$ that contains all the coefficients of $g_1,\ldots,g_r$. Consider the finite Galois groups $G_1:=G(E_1:K)$ and $G_2:=G(E_2:E)$, the families $\{Z_1^\rho:=\ZZ_C(g_1^\rho,\ldots,g_r^\rho)\}_{\rho\in G_1}$ and $\{Z_2^\sigma:=\ZZ_C(g_1^\sigma,\ldots,g_r^\sigma)\}_{\sigma\in G_2}$ of algebraic subsets of $C^n$ and the Galois completions $T_1:=\bigcup_{\rho\in G_1}Z_1^\rho$ and $T_2:=\bigcup_{\sigma\in G_2}Z_2^\sigma$ of $Y_i\subset C^n$. By Lemma \ref{lem:gp0}$(\mr{ii})(\mr{iii})$, we have: $T_1=X$, $\{Z_1^\rho:\rho\in G_1\}=\{Y_1,\ldots,Y_s\}$, $T_2=W_i$ and $\{Z_2^\sigma:\sigma\in G_2\}=\{W_{i1},\ldots,W_{i,T_i}\}$.

Let us prove: $\{Z_2^\sigma:\sigma\in G_2\}\subset\{Z_1^\rho:\rho\in G_1\}$. Choose $\sigma\in G_2$. By Lemma \ref{extension}, there exists an automorphism $\Phi_\sigma$ of $C$ such that $\Phi_\sigma|_{E_2}=\sigma$. As $\Phi_\sigma$ fixes $E$ (and hence $K$) and $E_1|K$ is a Galois extension, the automorphism $\Phi_\sigma$ restricts to an element $\tau:=\Phi_\sigma|_{E_1}$ of $G_1$. Observe that $E_2$ and $E_1$ are both subfields of $C$, so their intersection $E_2\cap E_1$ is a subfield of~$C$. As all the coefficients of $g_1\ldots,g_r$ belong to $E_2\cap E_1$, we have that $g_h^\sigma=\widehat{\Phi}_\sigma(g_h)=g_h^\tau$ for each $h\in\{1,\ldots,r\}$, so $Z_2^\sigma=\ZZ_C(g_1^\sigma,\ldots,g_r^\sigma)=\ZZ_C(g_1^\tau,\ldots,g_r^\tau)=Z_1^\tau$. This proves the inclusion $\{Z_2^\sigma:\sigma\in G_2\}\subset\{Z_1^\rho:\rho\in G_1\}$.

Consequently, for each $i\in\{1,\ldots,t\}$, the family $\{W_{i,1},\ldots,W_{i,T_i}\}$ of $C$-irreducible components of $W_i\subset C^n$ coincides with $\{Y_j:j\in J_i\}$ for some subset $J_i$ of $\{1,\ldots,s\}$.

The sets $J_i$ are pairwise disjoint. Otherwise, there exist $i,k\in\{1,\ldots,t\}$ and $\ell\in\{1,\ldots,s\}$ such that $i\neq k$ and $\ell\in J_i\cap J_k$, so $Y_\ell$ is a $C$-irreducible component of both $W_i\subset C^n$ and $W_k\subset C^n$. This is a contradiction by Corollary \ref{cho}.

The sets $J_i$ cover $\{1,\ldots,s\}$. Otherwise, there exists $\ell\in\{1,\ldots,s\}\setminus\bigcup_{i=1}^tJ_i$ and $X=\bigcup_{i=1}^tW_i=\bigcup_{i=1}^t\bigcup_{j\in J_i}Y_j\subset\bigcup_{h\in\{1,\ldots,s\}\setminus\{\ell\}}Y_h\subsetneqq X$, which is a contradiction.
\end{proof}

\begin{example}
Consider the extension of fields $C|E|K$ where $E|K=\Q(\sqrt{2})|\Q$. Consider the polynomial $p(\x_1):=(\x_1^2-2)^2-2\in\Q[\x_1]$ and the $\Q$-algebraic set $X:=\ZZ_C(p)\subset C$. The polynomial $p$ is irreducible in $\Q[\x_1]$ and $X\subset C$ is $\Q$-irreducible. Moreover, the following four singletons $Y_1$, $Y_2$, $Y_3$, $Y_4$ are the $C$-irreducible components of $X\subset C$:
$$\textstyle
Y_1:=\big\{-\sqrt{2+\sqrt{2}}\big\},\;\; Y_2:=\big\{-\sqrt{2-\sqrt{2}}\big\},\;\; Y_3:=\big\{\sqrt{2-\sqrt{2}}\big\}, \;\; Y_4:=\big\{\sqrt{2+\sqrt{2}}\big\}.
$$
Set $p_\pm(\x_1):=\x_1^2-2\pm\sqrt{2}\in\Q(\sqrt{2})[\x_1]$. As the polynomials $p_\pm$ in $\Q(\sqrt{2})[\x_1]$ are irreducible, $X\subset C$ has two $\Q(\sqrt{2})$-irreducible components, namely $W_1$ and $W_2$ where
$$
W_1:=\ZZ_C(p_+)=Y_2\cup Y_3, \quad W_2:=\ZZ_C(p_-)=Y_1\cup Y_4. \;\text{ $\sqbullet$}
$$
\end{example}

\subsubsection{Hypersurfaces}
As before, $C|K$ is an extension of fields such that $C$ is algebraically closed and $\kbar\subset C$ is the algebraic closure of $K$.

Let $H$ be a field, let $f\in H[\x]$ be a non-constant polynomial and let $f=uf_1^{m_1}\cdots f_s^{m_s}$ be the factorization of $f$ in $H[\x]$, that is, $u\in H^*:=H\setminus\{0\}$, each $f_i$ is an irreducible polynomial in $H[\x]$, the polynomials $f_i$ and $f_j$ are non-associated for all $i,j\in\{1,\ldots,s\}$ with $i\neq j$, and each $m_i$ is a positive natural number (if $H$ is algebraically closed, we can assume $u=1$). In case each exponent $m_i$ is equal to $1$, we say that $f$ is \emph{square-free}.

\begin{remark}\label{squaresquarefree}
A well-known result that we will use below is the following: \textit{Let $H'|H$ be any extension of fields of characteristic zero and let $f$ be a polynomial in $H[\x]$. Then $f$ is square-free as a polynomial in $K[\x]$ if and only if $f$ is square-free as a polynomial in $H'[\x]$}. The `if' implication is evident. To prove the `only if' implication, we use the properties of the discriminant. Indeed, as $H$ is infinite (because it has characteristic zero), after a suitable change of coordinates in $H^n$, we can assume that $f$ is monic with respect to $\x_n$, see \cite[Lem.2.1.6]{jp}. As $f$ is square-free in $H[\x]$, its discriminant with respect to $\x_n$ is non-zero, so it is also square-free in $H'[\x]$. $\sqbullet$
\end{remark}

\begin{remark}\label{rem:ref0}
Let $E|K$ be a finite Galois subextension of $\ol{K}|K$ and let $G':=G(E:K)$. Given $g\in E[\x]$ and $\sigma\in G'$, define $g^\sigma\in E[\x]$ as in Definition \ref{def:sigma}. Then $g^*:=\prod_{\sigma\in G'}g^\sigma\in K[\x]$, because it is invariant under the action of $G'$ (that is, $(g^*)^\tau=g^*$ for each $\tau\in G'$). Even if $g$ is square-free, $g^*$ needs not to be square-free, see Examples \ref{exa:gc}. $\sqbullet$
\end{remark}

A particular but interesting case of Algorithm \ref{gc} concerns hypersurfaces. In this situation, we can simplify the presentation of Algorithm \ref{gc} and obtain the following:

\begin{alg}[Galois completion of a polynomial]\label{gcpol}
The algorithm works as follows:
\begin{itemize}
\item[(0)] Start with a non-constant polynomial $g\in\kbar[\x]$.
\item[(1)] Choose a finite Galois subextension $E|K$ of $\kbar|K$ that contains all the coefficients of $g$ and denote $G'$ the corresponding Galois group $G(E:K)$.
\item[(2)] Define $g^*:=\prod_{\sigma\in G'}g^\sigma\in K[\x]$ (see Remark \emph{\ref{rem:ref0}}). Factorize $g^*=uf_1^{m_1}\cdots f_s^{m_s}$ in $K[\x]$, where $u\in K^*=K\setminus\{0\}$, the $f_i$ are pairwise non-associated irreducible polynomials in $K[\x]$ and the $m_i$ are positive integers.
\item[(3)] Define the polynomial $g^\bullet:=f_1\cdots f_s\in K[\x]$, which satisfies $\sqrt{(g^*)K[\x]}=(g^\bullet)K[\x]$.
\end{itemize}
\end{alg}

Applying Theorem \ref{thm:gc0} to $X=\ZZ_C(g)\subset C^n$ and keeping the notations of its statement, we obtain: $\mathfrak{H}=\{g^*\}$, $\mathfrak{G}=\{\binom{d}{j}(g^*)^j\}_{j\in\{1,\ldots,d\}}$, $\II_K(X)=\sqrt{\mathfrak{G}K[\x]}=\sqrt{(g^*)K[\x]}=(g^\bullet)K[\x]$, $\ZZ_C(g^\bullet)\subset C^n$ is the Galois completion of $X$ and $\II_C(\zcl_{C^n}^K(X))=(g^\bullet)C[\x]$. This provides:
\begin{align}
&\II_K(\ZZ_C(g))=(g^\bullet)K[\x]=\sqrt{(g^*)K[\x]},\label{22}\\
&\ZZ_C(g^\bullet)=\zcl_{C^n}^K(\ZZ_C(g)),\label{23}\\
&\II_C(\ZZ_C(g^\bullet))=(g^\bullet)C[\x].\label{22b}
\end{align}

\begin{defn}\label{gbullet}
We call $g^\bullet\in K[\x]$ the {\em Galois completion} of the polynomial $g\in\kbar[\x]$. $\sqbullet$
\end{defn}

\begin{remark}\label{welld-sfree}
By \eqref{22}, $g^\bullet$ is uniquely determined by $g$ (up to multiplication by elements of~$K^*$). By \eqref{22b}, $g^\bullet$ is square-free as a polynomial in $C[\x]$ (see also Remark \ref{squaresquarefree}). Thus, the Galois completion of each polynomial in $\kbar[\x]$ is well-defined (up to multiplication by elements of $K^*$) and square-free as a polynomial in $C[\x]$. $\sqbullet$
\end{remark}

\begin{remark}\label{gcpolr}
Pick a non-constant polynomial $g\in\kbar[\x]$ and apply Algorithm \ref{gcpol} to $g$ obtaining the polynomials $g^*,g^\bullet\in K[\x]$.

$(\mr{i})$ If $g$ is irreducible in $\kbar[\x]$, then $g^\bullet$ is irreducible in $K[\x]$ and $g^*=u(g^\bullet)^m$ for some $u\in K^*$ and $m\in\N^*$. Indeed, by Hilbert's Nullstellensatz, by Proposition \ref{extension-zar}$(\mr{iii})$ and by Lemma \ref{lem:irreducibility}, we deduce that $\ZZ_\kbar(g)\subset\kbar^n$ is $\kbar$-irreducible, $\ZZ_C(g)=(\ZZ_\kbar(g))_C\subset C^n$ is $C$-irreducible and $\zcl_{C^n}^K(\ZZ_C(g))\subset C^n$ is $K$-irreducible. Lemma \ref{lem:prime} and equalities \eqref{22} assure that $\II_K(\zcl_{C^n}^K(\ZZ_C(g)))=\II_K(\ZZ_C(g))=(g^\bullet)K[\x]=\sqrt{(g^*)K[\x]}$ is a prime ideal of $K[\x]$, so $g^\bullet$ is irreducible in $K[\x]$ and $g^*=u(g^\bullet)^m$ for some $u\in K^*$ and $m\in\N^*$, as required.

$(\mr{ii})$ Suppose that $g\in\kbar[\x]$ is reducible and denote $g=g_1^{\alpha_1}\cdots g_r^{\alpha_r}$ its factorization in $\kbar[\x]$. Suppose further that, in step {\it (1)} of Algorithm \ref{gcpol} we are applying to $g$, a finite Galois extension $E|K$ has been chosen in such a way that $E$ contains all the coefficients of $g_1,\ldots,g_r$ (and in particular all the coefficients of $g$). We keep the notations of Algorithm \ref{gcpol}. Let $g_j^*:=\prod_{\sigma\in G'}g_j^\sigma\in K[\x]$ and let $g_j^\bullet\in K[\x]$ be the Galois completion of $g_j$ for each $j\in\{1,\ldots,r\}$. By the preceding item $(\mr{i})$, $g_j^\bullet$ is irreducible in $K[\x]$ and $g_j^*=u'_j(g_j^\bullet)^{\ell_j}$ for some $u'_j\in K^*$ and $\ell_j\in\N^*$. As $g^\sigma=\prod_{j=1}^r(g_j^\sigma)^{\alpha_j}$ for each $\sigma\in G'$, we deduce $g^*=\prod_{\sigma\in G'}\prod_{j=1}^r(g_j^\sigma)^{\alpha_j}=\prod_{j=1}^r(g_j^*)^{\alpha_j}=u'\prod_{j=1}^r(g_j^\bullet)^{\alpha_j\ell_j}$, where $u':=\prod_{j=1}^r(u'_j)^{\alpha_j}\in K^*$. As $K[\x]$ is a UFD, there exists a surjective map $\eta:\{1,\ldots,r\}\to\{1,\ldots,s\}$ such that $f_{\eta(j)}=g_j^\bullet$ for each $j\in\{1,\ldots,r\}$ (up to multiplication by elements of $K^*$) and $\textstyle\sum_{j\in\eta^{-1}(i)}\ell_j\alpha_j=m_i\,$ for each $i\in\{1,\ldots,s\}$. $\sqbullet$
\end{remark}

\subsection{On the subfield-dimension invariance. Counterexamples via Faltings' theorem}\label{sdimension}
The next result is the version of Corollary \ref{kdim} for real closed fields.

\begin{lem}\label{rkdim}
Let $R|K$ be an extension of fields such that $R$ is a real closed field, and let $\kr$ be the algebraic closure of $K$ in $R$. If $X\subset R^n$ is a $\kr$-algebraic set (for instance, a $K$-algebraic set), then $\dim_R(X)=\dim_K(X)$.
\end{lem}
\begin{proof}
By Corollary \ref{inter}$(\mr{iv})$ and Corollary \ref{dimalg}, we have $\dim_R(X)=\dim_{\kr}(X\cap\krn)$ and $\dim_{\kr}(X\cap\krn)=\dim_K(X\cap\krn)$, respectively. Thus, $\dim_R(X)=\dim_K(X\cap\krn)$. By Corollary \ref{inter}$(\mr{ii})$, we have $\II_\kr(X)=\II_\kr(X\cap\krn)$, so $\II_K(X)=\II_\kr(X)\cap K[\x]=\II_\kr(X\cap\krn)\cap K[\x]=\II_K(X\cap\krn)$ and $\dim_K(X)=\dim_K(X\cap\krn)$. We conclude $\dim_R(X)=\dim_K(X)$, as required.
\end{proof}

Corollary \ref{kdim} and Lemma \ref{rkdim} have the following relevant consequence.

\begin{thm}[Subfield-dimension invariance]\label{dimension}
Let $L|K$ be an extension of fields such that $L$ is either algebraically closed or real closed, and let $X\subset L^n$ be a $K$-algebraic set. Then $\dim_L(X)=\dim_K(X)$.
\end{thm}

\begin{remark}\label{dime}
By the latter result, when we are dealing with algebraically closed fields or real closed fields, the dimension does not depend on the ground subfield we are working with and we will often avoid referring to such ground field. $\sqbullet$
\end{remark}

The invariance of dimension described in the previous theorem extends to other types of extensions of fields. Recall that a field $L$ is said to be real if it admits at least one (total) ordering that is compatible with the field operations or, equivalently, if $-1$ cannot be written as a finite sum of squares in $L$. Suppose $K$ is a real closed field and its extension $L$ is real. As the non-negative elements of a real closed field are squares and the squares in each ordered field are non-negative, we deduce that each ordering of $L$ extends the ordering of $K$.

\begin{lem}\label{acfideal}
Let $L|K$ be an extension of fields such that either $K$ is algebraically closed or $L$ is real and $K$ is real closed. Denote $\ol{L}^\sqbullet$ the algebraic closure of $L$ in the first case and the real closure of $L$ in the second case. Let $X\subset L^n$ be a $K$-algebraic set. Then we have $\II_L(X)=\II_L(X\cap K^n)=\II_K(X)L[\x]$, $\dim_L(X)=\dim_K(X)=\dim_K(X\cap K^n)$ and $X=(X\cap K^n)_{\ol{L}^\sqbullet}\cap L^n$.
\end{lem}
\begin{proof}
Let $Z:=\zcl_{(\ol{L}^\sqbullet)^n}(X)$. By Proposition \ref{prop:zar}, we have $\II_{\ol{L}^\sqbullet}(Z)=\II_L(X)\ol{L}^\sqbullet[\x]$. As $\II_{\ol{L}^\sqbullet}(Z)=\II_{\ol{L}^\sqbullet}(X)$, we also have $\II_{\ol{L}^\sqbullet}(X)=\II_L(X)\ol{L}^\sqbullet[\x]$.

If $L$ is real and $K$ is real closed, then $\II_K(X)$ is a real ideal of $K[\x]$ (in the usual sense of \cite[Def.4.1.3]{bcr}). To prove this fact, consider $p_1,\ldots,p_\ell\in K[\x]$ such that $p_1^2+\ldots+p_\ell^2\in\II_K(X)$. Then $p_1(x)^2+\ldots+p_\ell(x)^2=0$ for each $x\in X$. As $p_1(x),\ldots,p_\ell(x)\in L$ and $L$ is real, we deduce $p_1(x)=\ldots=p_\ell(x)=0$ for each $x\in X$, so $p_1,\ldots,p_\ell\in\II_K(X)$ and $\II_K(X)$ is a real ideal of $K[\x]$.

Define $S:=X\cap K^n$ and observe that $S=\ZZ_L(\II_K(X))\cap K^n=\ZZ_K(\II_K(X))$. If $K$ is algebraically closed, Hilbert's Nullstellensatz implies that $\II_K(S)=\II_K(\ZZ_K(\II_K(X)))=\II_K(X)$, because $\II_K(X)$ is a radical ideal of $K[\x]$. If $K$ is real closed, Real Nullstellensatz \cite[Thm.4.1.4]{bcr} implies that $\II_K(S)=\II_K(\ZZ_K(\II_K(X)))=\II_K(X)$, because $\II_K(X)$ is a real ideal of $K[\x]$.

Using Proposition \ref{prop:zar} again, we deduce that $\II_{\ol{L}^\sqbullet}(S)=\II_K(S)\ol{L}^\sqbullet[\x]=\II_K(X)\ol{L}^\sqbullet[\x]$. Thus, 
$$
\II_{\ol{L}^\sqbullet}(S)=\II_{K}(X)\ol{L}^\sqbullet[\x]\subset\II_L(X)\ol{L}^\sqbullet[\x]=\II_{\ol{L}^\sqbullet}(X)\subset\II_{\ol{L}^\sqbullet}(S).
$$ 
Consequently, $\II_{\ol{L}^\sqbullet}(S)=\II_K(X)\ol{L}^\sqbullet[\x]=\II_L(X)\ol{L}^\sqbullet[\x]$. Intersecting with $L[\x]$ and applying Corollary \ref{k}, we obtain $\II_L(S)=\II_K(X)L[\x]=\II_L(X)$ and in particular $\II_K(S)=\II_K(X)$, so $\dim_K(S)=\dim_K(X)$. By Corollary \ref{dimkl}, it follows that $\dim_L(X)=\dim_K(X)=\dim_K(S)$. As $\II_{\ol{L}^\sqbullet}(S)=\II_K(X)\ol{L}^\sqbullet[\x]$, by Proposition \ref{extension-zar}$(\mr{i})$, we have
\begin{align*}
(X\cap K^n)_{\ol{L}^\sqbullet}\cap L^n&=(S)_{\ol{L}^\sqbullet}\cap L^n=\zcl_{(\ol{L}^\sqbullet)^n}(S)\cap L^n\\
&=\ZZ_{\ol{L}^\sqbullet}(\II_{\ol{L}^\sqbullet}(S))\cap L^n=\ZZ_{\ol{L}^\sqbullet}(\II_K(X))\cap L^n=\ZZ_L(\II_K(X))=X,
\end{align*} 
as required.
\end{proof}

\begin{cor}
Let $L|K$ be an extension of fields and let $\kbar^\sqbullet$ be the algebraic closure of $K$ in~$L$. Suppose that either $\kbar^\sqbullet$ coincides with the algebraic closure of $K$ or $K$ is an ordered field and $\kbar^\sqbullet$ is the real closure of~$K$. If $X\subset L^n$ is a $K$-algebraic set, then $\dim_L(X)=\dim_K(X)$.
\end{cor}
\begin{proof}
As $X\subset L^n$ is also $\kbar^\sqbullet$-algebraic, Lemma \ref{acfideal} implies $\II_L(X)=\II_L(X\cap(\kbar^\sqbullet)^n)$ and $\dim_L(X)=\dim_{\kbar^\sqbullet}(X\cap(\kbar^\sqbullet)^n)$. By Corollary \ref{dimalg}, we also have $\dim_{\kbar^\sqbullet}(X\cap(\kbar^\sqbullet)^n)=\dim_K(X\cap(\kbar^\sqbullet)^n)$, so $\dim_L(X)=\dim_K(X\cap(\kbar^\sqbullet)^n)$. As $\II_K(X)=\II_L(X)\cap K[\x]=\II_L(X\cap (\kbar^\sqbullet)^n)\cap K[\x]=\II_K(X\cap(\kbar^\sqbullet)^n)$, we deduce $\dim_L(X)=\dim_K(X\cap(\kbar^\sqbullet)^n)=\dim_K(X)$, as required. 
\end{proof}

We will further extend Theorem \ref{dimension} to all $K$-semialgebraic sets in Theorem \ref{edimsa}.

\vspace{1em}
Our next purpose is to exhibit examples of extensions of fields $L|K$ and $K$-algebraic sets $X\subset L^n$ such that the dimensions $\dim_L(X)$ and $\dim_K(X)$ are different. As we have already proved in Corollary \ref{cor:dimLleqdimK}, the inequality $\dim_L(X)\leq\dim_K(X)$ always holds. We will provide examples in which the previous inequality is strict. Before doing so, we need a diophantine preface. Let $E$ be a number field, that is, a finite extension of~$\Q$ and let $\ove$ be its algebraic closure. Let $F$ be a homogeneous polynomial in $E[\x_0,\x_1,\x_2]$ that has degree $d\geq4$ and is \textit{nonsingular in $\ove$}, that is, the set of common zeros of $ \frac{\partial F}{\partial\x_0},\frac{\partial F}{\partial\x_1},\frac{\partial F}{\partial\x_2}$ in $\ove^3$ is just the origin. A~main example is $F=\x_0^d+\x_1^d+\x_2^d$. Denote $\PP\ZZ_\ove(F)$ the vanishing set of $F$ in the projective plane $\PP^2(\ove)$ and observe that $\PP\ZZ_\ove(F)$ is a nonsingular curve of genus $\frac{(d-1)(d-2)}{2}\geq2$.

In the celebrated paper~\cite{fa83}, Faltings solved Mordell's conjecture proving that: \textit{The curve $\PP\ZZ_E(F)\subset\PP^2(E)$ has a finite number of points}. In \cite[Thm.3, p.205]{fwgss}, Faltings extends this result as follows: \textit{The same finiteness assertion holds when $E$ is an arbitrary finitely generated extension of $\Q$}.

\begin{prop}\label{faltings}
Let $f$ be a polynomial in $\Q[\x_1,\x_2]$ of degree $d\geq4$ and let $F\in\Q[\x_0,\x_1,\x_2]=\Q[\x_1,\x_2][\x_0]$ be the homogeneization $F:=f(\frac{\x_1}{\x_0},\frac{\x_2}{\x_0})\x_0^d$ of $f$. Suppose that $f$ is not homogeneous in $\Q[\x_1,\x_2]$ and $F$ is nonsingular in $\qbar$ (for instance, $f=1+\x_1^d+\x_2^d$ and $F=\x_0^d+\x_1^d+\x_2^d$). Consider $F$ as an element of the polynomial ring $\Q(\x_1,\x_2)[\x_0]$.

We have:
\begin{itemize}
\item[$(\mr{i})$] $F$ is irreducible in $\Q(\x_1,\x_2)[\x_0]$ and the field $L:=\Q(\x_1,\x_2)[\x_0]/((F)\Q(\x_1,\x_2)[\x_0])$ is a finitely generated extension of $\Q$.
\item[$(\mr{ii})$] If $X\subset L^2$ is the $\Q$-algebraic set $X:=\ZZ_L(f)\subset L^2$, then $\dim_L(X)=0$ and $\dim_\Q(X)=1$.
\end{itemize}
\end{prop}
\begin{proof}
As $F$ is nonsingular in $\qbar$, by Bezout's theorem, $F$ is irreducible in $\qbar[\x_0,\x_1,\x_2]$ and consequently in $\Q[\x_1,\x_2][\x_0]$. In particular, $f$ is irreducible in $\Q[\x_1,\x_2]$. The fact that $f$ is not homogeneous in $\Q[\x_1,\x_2]$ implies that the indeterminate $\x_0$ appears in the expression of $F$, that is, $F$ has positive degree with respect to~$\x_0$. This assures that $F$ is also irreducible as a polynomial in $\Q[\x_1,\x_2][\x_0]$. By Gauss'~lemma, $F$ is irreducible as a polynomial in $\Q(\x_1,\x_2)[\x_0]$ as well. Let $\pi:\Q(\x_1,\x_2)[\x_0]\to L$ be the natural projection and let $\x_i':=\pi(\x_i)$ for each $i\in\{0,1,2\}$. Observe that $L=\Q(\x_0',\x_1',\x_2')$ is a finitely generated field that extends~$\Q$. Moreover, $\x_0'\neq0$ in $L$ and the point $p:=(\x_1'(\x_0')^{-1},\x_2'(\x_0')^{-1})\in L^2$ belongs to $X=\ZZ_L(f)$.

Identify $L^2$ with $\{[x_0,x_1,x_2]\in\PP^2(L): x_0=1\}$ via the chart $(x_1,x_2)\mapsto[1,x_1,x_2]$.

Let $\ol{L}$ be the algebraic closure of $L$. As $\ol{L}|\qbar$ is an extension of algebraically closed fields, the `extension of coefficients' procedure assures that $\ZZ_{\ol{L}}(\frac{\partial F}{\partial\x_0},\frac{\partial F}{\partial\x_1},\frac{\partial F}{\partial\x_2})=(\ZZ_\qbar(\frac{\partial F}{\partial\x_0},\frac{\partial F}{\partial\x_1},\frac{\partial F}{\partial\x_2}))_{\ol{L}}=(\{(0,0,0)\})_{\ol{L}}=\{(0,0,0)\}$. Thus, $F$ is also nonsingular in $\ol{L}$ and irreducible in $\ol{L}[\x_0,\x_1,\x_2]$. In particular, we can apply the mentioned Faltings' theorem \cite[Thm.3, p.205]{fwgss}, obtaining that the set $\PP\ZZ_L(F)\subset\PP^2(L)$ is finite. Thus, the set $X=\ZZ_L(f)=\PP\ZZ_L(F)\cap L^2$ is also finite and $\dim_L(X)=0$.

Observe that the point $[p,1]$ belongs to $\PP^2(\ol{L})\setminus\PP^2(\qbar)$.

Let us show that $\dim_\Q(X)=1$ by proving that $\II_\Q(X)=(f)\Q[\x_1,\x_2]$. Let $g\in\II_\Q(X)$ and let $G$ be its homogeneization in $\Q[\x_0,\x_1,\x_2]$. Suppose that $f$ does not divide $g$ in $\Q[\x_1,\x_2]$. Then $F$ does not divide $G$ in $\Q[\x_0,\x_1,\x_2]$. By Lemma \ref{k0}, $F$ does not divide $G$ in $\ol{L}[\x_0,\x_1,\x_2]$ as well. As both $\qbar\subset\ol{L}$ are algebraically closed fields and $F,G\in\qbar[\x_0,\x_1,\x_2]\subset\ol{L}[\x_0,\x_1,\x_2]$, the common zero set $\PP\ZZ_{\ol{L}}(F,G)$ of $F$ and $G$ in $\PP^2(\ol{L})$ is by Bezout's theorem a finite set (constituted by $d\cdot\deg(G)$ points counted with their intersection multiplicity and) contained in $\PP^2(\qbar)$. As $p\in X\subset\zcl_{\ol{L}}(X)=\ZZ_{\ol{L}}(\II_{\ol{L}}(X))\subset\ZZ_{\ol{L}}(f,g)$, it follows that $[1,p]\in\PP\ZZ_{\ol{L}}(F,G)\subset\PP^2(\qbar)$, which is a contradiction. This proves that $f$ divides $g$ in $\Q[\x_1,\x_2]$. Thus, $\II_\Q(X)=(f)\Q[\x_1,\x_2]$ and $\dim_\Q(X)=2-\mr{ht}((f)\Q[\x_1,\x_2])=1$, as required.
\end{proof}

\subsection{Complexification}\label{fe}

\emph{Throughout this subsection, $R$ is a real closed field, $\ii:=\sqrt{-1}$ and $C:=R[\ii]$ is the algebraic closure of $R$.} In this case, $G:=G(C:R)$ is a group of order~$2$ generated by conjugation involution $\varphi$ that fixes $R$ and maps $\ii$ to $-\ii$, that is, $\varphi:C\to C,\ x+\ii y\mapsto x-\ii y$. Define
$$
\varphi_n:C^n\to C^n,\ (x_1+\ii y_1,\ldots,x_n+\ii y_n)\mapsto(x_1-\ii y_1,\ldots,x_n-\ii y_n).
$$
Algorithm \ref{gc} together with Theorem \ref{thm:gc0} allow us to revisit the complexification of algebraic subsets of $R^n$.

\begin{lem}\label{conjugation}
If $\gta$ is an ideal of $C[\x]$, then
\begin{align}\label{eq:zzC}
\ZZ_C(\gta\cap R[\x])&=\ZZ_C(\gta)\cup\varphi_n(\ZZ_C(\gta))=\zcl_{C^n}^R(\ZZ_C(\gta)),\\
\label{eq:zzR}
\ZZ_R(\gta\cap R[\x])&=\ZZ_C(\gta)\cap R^n.
\end{align}
\end{lem}
\begin{proof}
Equality \eqref{eq:zzC} is an immediate consequence of Theorem \ref{thm:gc0}$(\mr{ii})(\mr{v})$. In addition,
\begin{align*}
\ZZ_R(\gta\cap R[\x])&=\ZZ_C(\gta\cap R[\x])\cap R^n=(\ZZ_C(\gta)\cup\varphi_n(\ZZ_C(\gta)))\cap R^n\\
&=(\ZZ_C(\gta)\cap R^n)\cup\varphi_n(\ZZ_C(\gta)\cap R^n)=\ZZ_C(\gta)\cap R^n.
\end{align*}
This proves \eqref{eq:zzR} and completes the proof.
\end{proof}

If $S\subset R^n$ is an algebraic set, its Zariski closure in $C^n$ is called the \emph{complexification of $S$}. Recall that, if $T$ is an algebraic subset of $C^n$, then $T\cap R^n$ is an algebraic subset of $R^n$ by Corollary \ref{cor-cap}. In addition, Proposition \ref{prop:zar}$(\mr{ii})$ assures that $\zcl_{C^n}(S)\cap R^n=S$ for each algebraic set $S\subset R^n$.

Recall that an ideal $I$ of $R[\x]$ is said to be \emph{real} if, for every finite sequence of polynomials $p_1,\ldots,p_\ell$ in $R[\x]$ such that $\sum_{j=1}^\ell p_j^2\in I$, we have $p_j\in I$ for each $j\in\{1,\ldots,\ell\}$, see \cite[Def.4.1.3]{bcr}.

\begin{prop}\label{rc}
Let $T\subset C^n$ be an algebraic set and let $S:=T\cap R^n$. The following conditions are equivalent.
\begin{itemize}
\item[$(\mr{i})$] $T$ is the complexification of $S$.
\item[$(\mr{ii})$] $\II_C(T)=\II_R(S)C[\x]$.
\item[$(\mr{iii})$] $\II_R(T)=\II_R(S)$.
\item[$(\mr{iv})$] $\II_R(T)$ is a real ideal of $R[\x]$.
\end{itemize}
Moreover, each of the above equivalent conditions implies the following condition $(\mr{v})$ and they are equivalent to~$(\mr{v})$ when $T\subset C^n$ is irreducible:
\begin{itemize}
\item[$(\mr{v})$] $\dim_C(T)=\dim_R(S)$.
\end{itemize}
\end{prop}
\begin{proof}
Implications $(\mr{i})\Longrightarrow(\mr{ii})$ and $(\mr{ii})\Longrightarrow(\mr{iii})$ follow immediately from Proposition \ref{prop:zar}$(\mr{i})$ and Corollary \ref{k}, respectively. 
Let us prove implications $(\mr{iii})\Longleftrightarrow(\mr{iv})\Longrightarrow(\mr{ii})\Longrightarrow(\mr{i})$.

$(\mr{iii})\Longleftrightarrow(\mr{iv})\ $ As $T$ is Zariski closed in $C^n$ and $\II_R(T)=\II_C(T)\cap R[\x]$, equation \eqref{eq:zzR} implies that $\ZZ_R(\II_R(T))=S$. Thus, by the Real Nullstellensatz \cite[Thm.4.1.4]{bcr}, $\II_R(T)$ is real if and only if it coincides with $\II_R(S)$. 

$(\mr{iv})\Longrightarrow(\mr{ii})\ $ Assume that $\II_R(T)$ is real, so it coincides with $\II_R(S)$. Let $\{g_1,\ldots,g_r\}$ be a system of generators of $\II_C(T)$ in $C[\x]$. Write $g_i=a_i+\ii b_i$ for some $a_i,b_i\in R[\x]$ and set $g_i':=a_i-\ii b_i$. Observe that $g_i'g_i=a_i^2+b_i^2\in\II_C(T)\cap R[\x]=\II_R(T)$, so $a_i,b_i\in \II_R(T)=\II_R(S)$. Thus, $a_1,b_1,\ldots,a_r,b_r\in \II_R(S)$ are generators of $\II_C(T)$ in $C[\x]$, so $\II_C(T)\subset \II_R(S)C[\x]$. On the other hand, $\II_C(T)\supset \II_R(T)=\II_R(S)$, so $\II_C(T)\supset \II_R(S)C[\x]$. Thus, $\II_C(T)=\II_R(S)C[\x]$.

$(\mr{ii})\Longrightarrow(\mr{i})\ $ By Proposition \ref{prop:zar}$(\mr{i})$, $\zcl_{C^n}(S)=\ZZ_C(\II_R(S))$. Thus, if $\II_C(T)=\II_R(S)C[\x]$, then $T=\ZZ_C(\II_C(T))=\ZZ_C(\II_R(S))=\zcl_{C^n}(S)$.

Implication $(\mr{i})\Longrightarrow(\mr{v})$ is true by Proposition \ref{prop:zar}$(\mr{iv})$ (even if $T\subset C^n$ is reducible). Let us complete the proof.

$(\mr{v})\Longrightarrow(\mr{i})\ $ Suppose that $T\subset C^n$ is irreducible and $\dim_C(T)=\dim_R(S)$. Set $Z:=\zcl_{C^n}(S)$. Observe that $Z\subset T$, because $S\subset T$. Proposition \ref{prop:zar}$(\mr{iv})$ assures that $\dim_C(Z)=\dim_R(S)$, so $\dim_C(Z)=\dim_C(T)$. As $T\subset C^n$ is irreducible, we conclude $Z=T$, as required.
\end{proof}

\begin{cor}\label{cor:ii-zz}
Let $T\subset C^n$ be an algebraic set and let $S:=T\cap R^n$. Then:
\begin{itemize}
\item[$(\mr{i})$] $\ZZ_R(\II_R(T))=\ZZ_R(\II_R(S))=S$.
\item[$(\mr{ii})$] If the ideal $\II_R(T)$ of $R[\x]$ is real, then $\II_R(T)=\II_R(S)$ and $\dim_C(T)=\dim_R(S)$.
\item[$(\mr{ii}')$] If the ideal $\II_R(T)$ of $R[\x]$ is non-real, then $\II_R(T)\subsetneqq\II_R(S)$. If in addition $T\subset C^n$ is irreducible, then $\dim_C(T)>\dim_R(S)$.
\end{itemize}
\end{cor}
\begin{proof}
As $\II_R(T)=\II_C(T)\cap R[\x]$, $T$ is by hypothesis an algebraic subset of $C^n$ and $S$ is by Corollary \ref{cor-cap} an algebraic subset of $R^n$, item $(\mr{i})$ follows from \eqref{eq:zzR}. Equivalence $(\mr{iii})\Longleftrightarrow(\mr{iv})$ of Proposition \ref{rc} assures that $\II_R(T)=\II_R(S)$ if $\II_R(T)$ is real, and $\II_R(T)\subsetneqq\II_R(S)$ if $\II_R(T)$ is non-real. If $\II_R(T)=\II_R(S)$, then $T$ is the complexification of $S$ and $\dim_C(T)=\dim_R(S)$ by Proposition \ref{prop:zar}$(\mr{iv})$. If $T\subset C^n$ is irreducible and $\II_R(T)\subsetneqq\II_R(S)$, that is, $\II_R(T)$ is non-real, then Corollary \ref{<=dim}$(\mr{i})$ and equivalence $(\mr{iv})\Longleftrightarrow(\mr{v})$ of Proposition \ref{rc} imply that $\dim_C(T)>\dim_R(S)$, as required. 
\end{proof}

Let us present a generalization of Corollary \ref{cor:ii-zz}$(\mr{i})$ we will use in the proof of Proposition~\ref{S_0}.

\begin{lem}\label{lem:ts}
Let $R|K$ be a field extension, let $\kr$ be the algebraic closure of $K$ in $R$, let $\kbar=\kr[\ii]$ be the algebraic closure of $K$ in $C$, let $T\subset C^n$ be a $\kbar$-algebraic set and let $S:=T\cap R^n$. Then $\ZZ_R(\II_\kr(T))=\ZZ_R(\II_\kr(S))=S$. In particular, $S\subset R^n$ is a $\kr$-algebraic~set.
\end{lem}
\begin{proof}
Let $g_1,\ldots,g_s\in\kbar[\x]$ be such that $T=\ZZ_C(g_1,\ldots,g_s)$. Write $g_j=a_j+\ii b_j$ with $a_j,b_j\in\kr[\x]$. As $S=\ZZ_R(a_1,b_1,\ldots,a_s,b_s)$, it follows that $S\subset R^n$ is $\kr$-algebraic, so $S=\ZZ_R(\II_\kr(S))\subset\ZZ_R(\II_\kr(T))$. It remains to show that $\ZZ_R(\II_\kr(T))\subset S$. Let $x\in R^n\setminus S=R^n\setminus T$. As $T\subset C^n$ is $\kbar$-algebraic and $x\not\in T$, there exists $f\in\II_\kbar(T)$ such that $f(x)\neq0$. Write $f=a+\ii b$ with $a,b\in\kr[\x]$ and observe that $h=a^2+b^2=(a+\ii b)(a-\ii b)\in\II_{\kbar}(T)\cap\kr[\x]=\II_{\kr}(T)$. As $a(x),b(x)\in R$ and $a(x)+\ii b(x)=f(x)\neq0$, we deduce $h(x)=a(x)^2+b(x)^2\neq0$, so $x\not\in\ZZ_R(\II_\kr(T))$, as required.
\end{proof}

\subsection{The projective case and the $L|K$-elimination theory}
\label{subsec:proj} \emph{Along this subsection, $L|K$ is an extension of fields such that $L$ is either algebraically closed or real closed.}

Here we extend the notion of $K$-algebraic set and some related concepts from the affine to the projective case. We will follow Mumford's strategy \cite[Sections~2A-2C]{mumford}.

Denote $L[\x_0,\x]:=L[\x_0,\x_1,\ldots,\x_n]$ and define $L[\x_0,\x]_\sfh:=\{f\in L[\x_0,\x]:\text{$f$ is homogeneous}\}$. Consider $K[\x_0,\x]$ as a subset of $L[\x_0,\x]$, and $K[\x_0,\x]_\sfh$ as a subset of $L[\x_0,\x]_\sfh$. Let $\PP^n(L)$ be the usual projective $n$-space over $L$ and let $[x_0,x]=[x_0,x_1,\ldots,x_n]$ be its homogeneous coordinates. Given sets $F\subset L[\x_0,\x]_\sfh$ and $S\subset \PP^n(L)$, define:
\begin{align*}
\PP\ZZ_L(F)&:=\{[x_0,x]\in\PP^n(L): f(x_0,x)=0,\ \forall f\in F\},\\
\PP\II_K(S)&:=\big(\{f\in K[\x_0,\x]_\sfh: f(x_0,x)=0,\ \forall [x_0,x]\in S\}\big)K[\x_0,\x],
\end{align*}
where as usual $(G)K[\x_0,\x]$ denotes the ideal of $K[\x_0,\x]$ generated by a given subset $G$ of $K[\x_0,\x]$. If $F=\{f_1,\ldots,f_s\}\subset L[\x_0,\x]_\sfh$, we set $\PP\ZZ_L(f_1,\ldots,f_s):=\PP\ZZ_L(F)$. Recall that an ideal $\mk{I}$ of $K[\x_0,\x]$ is said to be homogeneous if the following is true: given any $g\in\mk{I}$, if we write $g=\sum_{i=0}^dg_i$ with $d=\deg(g)$ and $g_i\in L[\x_0,\x]_\sfh$ of degree $i$, then each $g_i$ belongs to $\mk{I}$. By Noetherianity, $\mk{I}$ is homogeneous if and only if it admits a finite system of homogeneous generators in $K[\x_0,\x]$. If $\mk{I}$ is a homogeneous ideal of $K[\x_0,\x]$, we define $\PP\ZZ_L(\mk{I})$ as the set $\PP\ZZ_L(\mk{I}\cap K[\x_0,\x]_\sfh)$. Evidently, $\PP\II_K(S)$ is a homogeneous ideal of $K[\x_0,\x]$.

\begin{defn}\label{def:proj-K-alg}
Let $X$ be a subset of $\PP^n(L)$. We say that $X$ is a \emph{$K$-algebraic subset of $\PP^n(L)$}, or $X\subset\PP^n(L)$ is a \emph{$K$-algebraic set}, if $X=\PP\ZZ_L(F)$ for some $F\subset K[\x_0,\x]_\sfh$. $\sqbullet$
\end{defn}

Observe that $X\subset\PP^n(L)$ is $K$-algebraic if and only if $X=\PP\ZZ_L(\PP\II_K(X))$. If $K=L$, an $L$-algebraic subset of $\PP^n(L)$ is a usual algebraic subset of $\PP^n(L)$. As in the affine case, the family of all $K$-algebraic subsets of $\PP^n(L)$ provides a topology on $\PP^n(L)$ coarser than the usual Zariski topology on $\PP^n(L)$. In particular, such a topology is Noetherian.

\begin{defn}\label{def:proj-K-zar}
The \emph{$K$-Zariski topology of $\PP^n(L)$} is the Noetherian topology of $\PP^n(L)$ whose closed sets are the $K$-algebraic subsets of $\PP^n(L)$. We call the open sets of this topology \emph{$K$-Zariski open} subsets of $\PP^n(L)$ and its closed sets \emph{$K$-Zariski closed} subsets of $\PP^n(L)$.

Equip $\PP^n(L)$ with the $K$-Zariski topology. Given a subset $S$ of $\PP^n(L)$, we say that $S$ is \emph{$K$-Zariski locally closed} if it is locally closed in $\PP^n(L)$, that is, $S=S_1\setminus S_2$ for some closed subsets $S_1$ and $S_2$ of $\PP^n(L)$. We also say that $S$ is \emph{$K$-constructible} if it is a Boolean combination of closed subsets of $\PP^n(L)$ or, equivalently, if $S$ is the disjoint union of finitely many $K$-Zariski locally closed subsets of $\PP^n(L)$. 

We call the set $\zcl_{\PP^n(L)}^K(S):=\PP\ZZ_L(\PP\II_K(S))$ the \emph{$K$-Zariski closure of $S$ (in~$\PP^n(L)$)}, that is, the closure of $S$ in $\PP^n(L)$. $\sqbullet$
\end{defn}

As in the affine case, making use of the $K$-Zariski topology of $\PP^n(L)$, one can introduce the concepts of $K$-irreducible $K$-algebraic subset of $\PP^n(L)$ and $K$-irreducible components of a $K$-algebraic subset of $\PP^n(L)$.

For each $i\in\{0,\ldots,n\}$, define the projective hyperplane $H^n_i:=\{[x_0,x]\in\PP^n(L):x_i=0\}$ and its complement $U^n_i:=\PP^n(L)\setminus H^n_i$ in $\PP^n(L)$, and denote $\theta^n_i:U^n_i\to L^n$ the corresponding affine chart
$$\textstyle
\theta^n_i([x_0,x]):=\big(\frac{x_0}{x_i},\ldots,\frac{x_{i-1}}{x_i},\frac{x_{i+1}}{x_i},\ldots,\frac{x_n}{x_i}\big).
$$
The set $\{U^n_i\}_{i=0}^n$ is a $K$-Zariski open cover of $\PP^n(L)$ and $(\theta^n_i)^{-1}(x)=[x_1,\ldots,x_i,1,x_{i+1},\ldots,x_n]$.

\emph{Equip $L^n$ and $\PP^n(L)$ with their respective $K$-Zariski topologies and each $U^n_i$ with the relative topology induced by the $K$-Zariski topology of $\PP^n(L)$.} If $f\in K[\x_0,\x]_\sfh$, then $\theta^n_i(U^n_i\cap\PP\ZZ_L(f))=\ZZ_L(f(\x_1,\ldots,\x_i,1,\x_{i+1},\ldots,\x_n))$, so $\theta^n_i$ is a closed map.

Consider any $p\in L[\x]$ and write $p=\sum_{\ell=0}^ep_\ell$ with $e=\deg(p)$ and $p_\ell\in L[\x]_\sfh$ of degree $\ell$. Denote $p^{\sfh,i}(\x_0,\x)\in K[\x_0,\x]=K[\x_0,\x_1,\ldots,\x_n]$ the $i^{\mr{th}}$-homogenization of $p$, that is,
$$\textstyle
p^{\sfh,i}(\x_0,\x):=\sum_{\ell=0}^e\x_i^{e-\ell}p_\ell(\x_0,\ldots,\x_{i-1},\x_{i+1},\ldots,\x_n).
$$
For short, we set $p^\sfh:=p^{\sfh,0}$ and we simply call $p^\sfh$ \emph{homogenization of $p$}. As $p^{\sfh,i}(x_0,x)=x_i^ep\big(\frac{x_0}{x_i},\ldots,\frac{x_{i-1}}{x_i},\frac{x_{i+1}}{x_i},\ldots,\frac{x_n}{x_i}\big)$ for all $(x_0,x)\in L^{n+1}$ with $x_i\neq0$, we deduce $(\theta^n_i)^{-1}(\ZZ_L(p))=U^n_i\cap\PP\ZZ_L(p^{\sfh,i})$, so $(\theta^n_i)^{-1}$ is also a closed map. It follows that each affine chart $\theta^n_i$ is a homeo\-morphism and hence a set $S\subset\PP^n(L)$ is $K$-algebraic if and only if each $\theta^n_i(U^n_i\cap X)\subset L^n$ is $K$-algebraic. Similarly, if a set $S\subset\PP^n(L)$ is $K$-Zariski locally closed ($K$-constructible) if and only if each $\theta^n_i(U^n_i\cap X)\subset L^n$ is $K$-Zariski locally closed ($K$-constructible, respectively).

\begin{defn}\label{def:proj-K-dim}
Given a subset $S$ of $\PP^n(L)$, we define the \emph{$K$-dimension $\dim_K(S)$ of $S$ (in $\PP^n(L)$)} by setting $\dim_K(S):=\max_{i\in\{0,\ldots,n\}}\{\dim_K(\theta^n_i(U^n_i\cap S))\}$, where $\dim_K(\theta^n_i(U^n_i\cap S))$ is the Krull dimension of the ring $K[\x]/\II_K(\theta^n_i(U^n_i\cap S))$ as in Definition \ref{def:K-dim}. $\sqbullet$
\end{defn}

\begin{remarks}\label{rem265}
$(\mr{i})$ Let $i,j\in\{0,\ldots,n\}$ with $i\neq j$. Suppose that $i<j$. Then $A^n_{ij}:=\theta^n_i(U^n_i\cap U^n_j)$ is equal to the $K$-Zariski open set $\{x_j\neq0\}$ of $L^n$, $B^n_{ij}:=\theta^n_j(U^n_i\cap U^n_j)$ is equal to the $K$-Zariski open set $\{x_{i+1}\neq0\}$ of $L^n$ and the transition map induced by $\theta^n_i$ and $\theta^n_j$, that is,
$$\textstyle
A^n_{ij}\to B^n_{ij},\quad x\mapsto\theta^n_j((\theta^n_i)^{-1}(x))=\big(\frac{x_1}{x_j},\ldots,\frac{x_i}{x_j},\frac{1}{x_j},\frac{x_{i+1}}{x_j},\ldots,\frac{x_{j-1}}{x_j},\frac{x_{j+1}}{x_j},\ldots,\frac{x_n}{x_j}\big),
$$
is a $K$-biregular isomorphism. Similar considerations can be repeated when $i>j$.

$(\mr{ii})$ In Definition \ref{sub-alg-reg} below, given a $K$-algebraic set $X\subset L^n$, we will introduce the notion of $L|K$-local ring $\reg^{L|K}_{X,a}$ of $X$ at each of its points $a$. Such a notion is local with respect to the $K$-Zariski topology of $L^n$ and it is invariant under $K$-biregular isomorphisms, up to ring isomorphisms. In Proposition \ref{prop:XY}$(\mr{iii})$ below, we will see that $\dim_K(X)=\max_{a \in X}\{\dim(\reg^{L|K}_{X,a})\}$.

$(\mr{iii})$ Let $S\subset\PP^n(L)$ be a $K$-constructible set. The preceding items $(\mr{i})$ and $(\mr{ii})$ describe the $K$-Zariski local nature of the notion of $\dim_K(S)$ introduced in Definition \ref{def:proj-K-dim}. In addition, by Remark \ref{dime}, $\dim_K(S)$ coincides with the usual ($L$-)dimension $\dim_L(S)$ of $S$ viewed as an $L$-constructible subset of $\PP^n(L)$. $\sqbullet$
\end{remarks}

Assertion $(\mr{iii})$ of the preceding remark justifies the following notation.

\begin{notation}
If $S\subset\PP^n(L)$ is a $K$-constructible set, we simply say that $S$ has \emph{dimension $\dim(S)$}, where $\dim(S):=\dim_K(S)=\dim_L(S)$.
\end{notation}

\emph{In what follows, we identify $U^n_0=\PP^n(L)\setminus H^n_0$ and $L^n$ via the affine chart $\theta^n_0$. In this way, each subset of $L^n$ is also a subset of $\PP^n(L)$ and $L^n$ equipped with its $K$-Zariski topology is a topological subspace of $\PP^n(L)$ equipped with its $K$-Zariski topology.} Observe that $L^n$ is a $K$-Zariski open subset of $\PP^n(L)$. Moreover, a subset $S$ of $L^n$ is $K$-constructible in $L^n$ in the sense of Definition \ref{def:K-constructible} if and only if it is $K$-constructible in $\PP^n(L)$ in the sense of Definition~\ref{def:proj-K-zar}.

A simple but useful result is the following.

\begin{lem}\label{lem:266}
Let $X\subset L^n$ be a $K$-algebraic set and let $\overline{X}:=\zcl^K_{\PP^n(L)}(X)$ be its $K$-Zariski closure in $\PP^n(L)$. Denote $\mk{J}$ the ideal of $K[\x_0,\x]$ generated by the homogenizations of all polynomials in $\II_K(X)$, that is, $\mk{J}:=(\{g^\sfh\}_{g\in\II_K(X)})K[\x_0,\x]$. Then $\PP\II_K(\overline{X})=\mk{J}$ and $\overline{X}=\PP\ZZ_L(\mk{J})$.

In addition, there exist a finite system of generators $\{g_1,\ldots,g_s\}$ of $\II_K(X)$ in $K[\x]$ such that $\PP\II_K(\overline{X})=(g_1^\sfh,\ldots,g_s^\sfh)L[\x_0,\x]$ and $\overline{X}=\PP\ZZ_L(g_1^\sfh,\ldots,g_s^\sfh)$.
\end{lem}
\begin{proof}
A standard argument works. Let $g\in\II_K(X)$. As $g^\sfh(1,\x)=g(\x)$, $\PP\ZZ_L(g^\sfh)$ is a $K$-Zariski closed subset of $\PP^n(L)$ containing $X$, so $\overline{X}\subset\PP\ZZ_L(g^\sfh)$ and $\mk{J}\subset\PP\II_K(\overline{X})$. Let $F\in\PP\II_K(\overline{X})\cap L[\x_0,\x]_\sfh$ and define $f(\x):=F(1,\x)\in\II_K(X)$. As $F(\x_0,\x)=\x_0^\ell f^\sfh(\x_0,\x)$ for some $\ell\in\N$, we have $F\in\mk{J}$, so $\PP\II_K(\overline{X})\subset\mk{J}$. It follows that $\PP\II_K(\overline{X})=\mk{J}$ and hence $\overline{X}=\PP\ZZ_L(\PP\II_K(\overline{X}))=\PP\ZZ_L(\mk{J})$.

Let $\{G_1,\ldots,G_s\}$ be a finite system of homogeneous generators of $\PP\II_K(\overline{X})$ in $K[\x_0,\x]$. Let $\ell_i\in\N$ be the largest natural number such that the monomial $\x_0^{\ell_i}$ divides $G_i(\x_0,\x)$ in $K[\x_0,\x]$ and let $G'_i(\x_0,\x)\in K[\x_0,\x]_\sfh$ be the unique homogeneous polynomial such that $G_i(\x_0,\x)=\x_0^{\ell_i}G'_i(\x_0,\x)$. As $\PP\II_K(\overline{X})=\PP\II_K(X)$ and $X\cap H^n_0=\varnothing$, the homogeneous polynomial $G'_i$ belongs to $\PP\II_K(\overline{X})$. Therefore, we replace each $G_i$ by $G'_i$ and assume that $\ell_i=0$ for each $i\in\{1,\ldots,s\}$. Each polynomial $g_i(\x):=G_i(1,\x)$ belongs to $\II_K(X)$ and $g_i^\sfh=G_i$. It remains to show that $\II_K(X)=(g_1,\ldots,g_s)K[\x]$. If $f$ is any polynomial in $\II_K(X)$, then $f^\sfh\in\mk{J}=\PP\II_K(\overline{X})$ so $f^\sfh=\sum_{i=1}^sH_iG_i$ for some $H_i\in K[\x_0,\x]_\sfh$, so $f(\x)=f^\sfh(1,\x)=\sum_{i=1}^sH_i(1,\x)g_i(\x)$, as required. 
\end{proof}

\emph{Fix $m\in\N^*$.} Equip $L^n\times L^m=L^{n+m}$ with the $K$-Zariski topology of $L^{n+m}$. Define the $K$-Zariski topologies of $\PP^n(L)\times L^m$, $L^n\times\PP^m(L)$ and $\PP^n(L)\times\PP^m(L)$, and the corresponding concepts of $K$-constructible sets in the usual way as follows:
\begin{itemize}
\item Let $([\x_0,\x],\y)=([\x_0,\ldots,\x_n],(\y_1,\ldots,\y_m))$ be the coordinates of $\PP^n(L)\times L^m$. The closed sets of the $K$-Zariski topology of $\PP^n(L)\times L^m$ are the zero sets of families of polynomials in $K[\y][\x_0,\x]_\sfh$, that is, polynomials in the variables $(\x_0,\x,\y)$ with coefficients in $K$ that are homogeneous of some degree in the variables $(\x_0,\x)$. \emph{Equip $\PP^n(L)\times L^m$ with the $K$-Zariski topology.} Then the topological subspaces $\{U^n_i\times L^m\}_{i=0}^n$ of $\PP^n(L)\times L^m$ form an open cover of $\PP^n(L)\times L^m$ and each map $\theta^n_i\times\mr{id}_{L^m}:U^n_i\times L^m\to L^{n+m}$ is a homeomorphism. 
\item The $K$-Zariski topology of $L^n\times\PP^m(L)$ is similarly defined using coordinates $(\x,[\y_0,\y])=((\x_1,\ldots,\x_n),[\y_0,\ldots,\y_m])$ and polynomials in $K[\x][\y_0,\y]_\sfh$. \emph{Equip $L^n\times\PP^m(L)$ with the $K$-Zariski topology.} Again, the topological subspaces $\{L^n\times U^m_j\}_{j=0}^m$ of $L^n\times\PP^m(L)$ form an open cover of $L^n\times\PP^m(L)$ and each map $\mr{id}_{L^n}\times\theta^m_j:L^n\times U^m_j\to L^{n+m}$ is a homeomorphism.
\item Let $([\x_0,\x],[\y_0,\y])=([\x_0,\ldots,\x_n],[\y_0,\ldots,\y_m])$ be the coordinates of $\PP^n(L)\times\PP^m(L)$. The closed sets of the $K$-Zariski topology of $\PP^n(L)\times\PP^m(L)$ are the zero sets of families of polynomials in $K[\x_0,\x,\y_0,\y]$ that are homogeneous of some bi-degree separately in the variables $[\x_0,\x]$ and $[\y_0,\y]$. \emph{Equip $\PP^n(L)\times\PP^m(L)$ with the $K$-Zariski topology.} Once again, the topological subspaces $\{U^n_i\times U^m_j\}_{i=0,\ldots,n,\,j=0,\ldots,m}$ of $\PP^n(L)\times\PP^m(L)$ form an open cover of $\PP^n(L)\times\PP^m(L)$ and each map $\theta^n_i\times\theta^m_j:U^n_i\times U^m_j\to L^{n+m}$ is a homeomorphism.
\item Let $T$ be one of the three topological spaces $\PP^n(L)\times L^m$, $L^n\times\PP^m(L)$ and $\PP^n(L)\times\PP^m(L)$, and let $S$ be a subset of $T$. We say that: $S$ is \emph{$K$-Zariski locally closed} if it is locally closed in $T$, $S$ is \emph{$K$-constructible} if it is a Boolean combination of closed subsets of $T$ and $S$~is \emph{$K$-irreducible} if it is irreducible as a topological subspace of~$T$. If $S$ is closed in~$T$, then we also say that $S\subset T$ is \emph{$\Q$-algebraic}.
\item As $L^n$ was identified with $U^n_0$ via $\theta^n_0$ and $L^m$ with $U^m_0$ via $\theta^m_0$, the products $L^n\times\PP^m(L)$ and $\PP^n(L)\times L^m$ are topological subspaces of $\PP^n(L)\times\PP^m(L)$. In this way, each ($K$-constructible) subset of either $L^n\times\PP^m(L)$ or $\PP^n(L)\times L^m$ is also a ($K$-constructible) subset of $\PP^n(L)\times\PP^m(L)$.
\end{itemize}

\begin{defn}\label{def267}
Given a subset $S$ of $\PP^n(L)\times\PP^m(L)$, we define the \emph{$K$-dimension $\dim_K(S)$ of~$S$ (in $\PP^n(L)\times\PP^m(L)$)} by $\dim_K(S):=\max_{(i,j)\in\{0,\ldots,n\}\times\{0,\ldots,m\}}\{\dim_K((\theta^n_i\times\theta^m_j)(S\cap(U^n_i\times U^m_j)))\}$. $\sqbullet$
\end{defn}

\begin{remark}\label{rem267}
By Remarks \ref{rem265}$(\mr{i})$ and \ref{dime}, if $S\subset\PP^n(L)\times\PP^m(L)$ is $K$-constructible, then $\dim_K(S)$ coincides with the usual ($L$-)dimension $\dim_L(S)$ of $S$ viewed as a usual ($L$-)constructible subset of $\PP^n(L)\times\PP^m(L)$. $\sqbullet$
\end{remark}

The preceding remark justifies the following notation.

\begin{notation}
If $S\subset\PP^n(L)\times\PP^m(L)$ is a $K$-constructible set, we simply say that $S$ has \emph{dimension $\dim(S)$}, where $\dim(S):=\dim_K(S)=\dim_L(S)$. $\sqbullet$
\end{notation}

The next result is a $L|K$-version of the main theorem of elimination theory and Chevalley's projection theorem. 

\begin{thm}[$L|K$-elimination and Chevalley's $L|K$-projection theorems]\label{thm:K-elimination}
Let $L$ be an algebraically closed field and let $\rho:\PP^n(L)\times\PP^m(L)\to\PP^m(L)$ be the canonical projection onto the second factor. We have:
\begin{itemize}
\item[$(\mr{i})$] $\rho$ is a $K$-Zariski closed map, that is, if $S$ is a $K$-algebraic subset of $\PP^n(L)\times\PP^m(L)$, then $\rho(S)$ is a $K$-algebraic subset of $\PP^m(L)$. The same is true for the canonical projection $\PP^n(L)\times L^m\to L^m$.
\item[$(\mr{ii})$] $\rho$ maps $K$-constructible sets in $K$-constructible sets, that is, if $S$ is a $K$-constructible subset of $\PP^n(L)\times\PP^m(L)$, then $\rho(S)$ is a $K$-constructible subset of $\PP^m(L)$. The same is true for the canonical projection $\PP^n(L)\times L^m\to L^m$.
\end{itemize}
\end{thm}

We postpone the proof to Appendix \ref{appendix:c}.1. In Remark \ref{product-project-spaces}, we extend Theorem \ref{thm:K-elimination} to the case of all finite products of proje\-ctive spaces over $L$.

We will use Theorem \ref{thm:K-elimination} in Subsection \ref{subsec:proj2}.

%%%
\section{Real $K$-algebraic sets}\label{s2}

{\it Along this section, $R$ is a real closed field and $C:=R[\ii]$ is its algebraic closure. As usual, for each algebraic set $S\subset R^n$, the dimension $\dim(S)$ of $S$ is the natural number $\dim_R(S)$, that is, $\dim(S):=\dim_R(S)$. We will use Remark \emph{\ref{dime}} freely.

Let $K$ be an ordered subfield of $R$. Denote $\kbar\subset C$ the algebraic closure of $K$ and $\kr\subset R$ the real closure of $K$, so $\kbar=\kr[\ii]$ and $R\cap\kbar=\kr$. Consider the extensions of fields $R|\kr|K$ and $C|\kbar|K$. A~crucial case concerns $K=\Q$. If such is the case, $\qbar$ is the field of algebraic numbers, whereas $\qr$ is the field of real algebraic numbers.

By Definition \ref{def:K-alg}, a set $Y\subset R^n$ is a \emph{$K$-algebraic set} if $Y=\ZZ_R(F)$ for some $F\subset K[\x]$.}

\subsection{Complex and real Galois completions. $K$-bad points}

\subsubsection{Complex and real Galois completions of a real algebraic set}

We are interested in studying the $K$-Zariski closure of an algebraic set $Y\subset R^n$. By Lemma \ref{big}, we focus on the case where $Y\subset R^n$ is a $\kr$-algebraic set. Let $Z:=\zcl_{C^n}(Y)$ be the complexification of $Y$. By Corollary \ref{inter}$(\mr{ii})$ and Proposition \ref{rc}$(\mr{ii})$, we have $\II_R(Y)=\II_\kr(Y)R[\x]$ and $\II_C(Z)=\II_R(Y)C[\x]=(\II_\kr(Y)R[\x])C[\x]=\II_\kr(Y)C[\x]$, so
\begin{equation}\label{equaz}
\II_C(Z)=\II_\kr(Y)C[\x].
\end{equation}
Thus, $Z\subset C^n$ is $\kbar$-algebraic ($\kr$-algebraic indeed) and we can apply Algorithm \ref{gc} to $X=Z$. 

\begin{defn}
Let $Y\subset R^n$ be a $\kr$-algebraic set and let $Z\subset C^n$ be the complexification of~$Y$. Denote $T\subset C^n$ the Galois completion of $Z$ and define $T^r:=T\cap R^n$. We call $T\subset C^n$ the \emph{complex Galois completion of $Y\subset R^n$} and $T^r\subset R^n$ the \emph{real Galois completion of $Y\subset R^n$ (with respect to the extension of fields $C|K$)}. $\sqbullet$ 
\end{defn}

We show next that $T$ is the $K$-Zariski closure of $Y$ in $C^n$, $T^r$ is the $K$-Zariski closure of $Y$ in $R^n$ and their main properties. We apply Algorithm \ref{gc} to $X=Z$ and keep the notations used in this algorithm and in \eqref{psi}\&\eqref{hatpsi}. We set $G:=G(C:K)$. Recall that, given any ideal $\gta$ of $R[\x]$, the real radical $\sqrt[\mr{r}]{\gta}$ of $\gta$ is the ideal of $R[\x]$ defined by
$$
\sqrt[\mr{r}]{\gta}:=\big\{f\in R[\x]:\text{$f^{2m}+p_1^2+\cdots +p_\ell^2\in\gta$ for some $m\in\N$ and $p_1,\ldots,p_\ell\in R[\x]$}\big\}.
$$

\begin{thm}[Galois completions and $K$-Zariski closures]\label{thm:gc}
Let $Y\subset R^n$ be a $\kr$-algebraic set and let $Z\subset C^n$ be the complexification of $Y$, that is, $Z:=\zcl_{C^n}(Y)$. Let $T\subset C^n$ and $T^r\subset R^n$ be the complex and real Galois completions of $Y\subset R^n$. Let $E|K$ be a finite Galois subextension of $\kbar|K$ that contains all the coefficients of finitely many polynomials $g_1,\ldots,g_r\in\kbar[\x]$ such that $Z=\ZZ_C(g_1,\ldots,g_r)$ (for instance, we can choose by \eqref{equaz} a finite system of generators $\{g_1,\ldots,g_r\}$ of $\II_\kr(Y)$ in $\kr[\x]$) and denote $G':=G(E:K)$.

We have:
\begin{itemize}
\item[$(\mr{i})$] $T=\bigcup_{\psi\in G}\psi_n(Z)$, $T^r=\bigcup_{\psi\in G}(\psi_n(Z)\cap R^n)$ and, for each $\psi\in G$, it holds
$$
\II_C(\psi_n(Z))=\widehat{\psi}(\II_C(Z))
$$
and
$$
\dim(\psi_n(Z)\cap R^n))\leq\dim_C(\psi_n(Z))=\dim_C(Z)=\dim(Y).
$$
\item[$(\mr{i}')$] $T=\bigcup_{\sigma\in G'}Z^\sigma$, $T^r=\bigcup_{\sigma\in G'}(Z^\sigma\cap R^n)$ and, for each $\sigma\in G'$, it holds
$$
\II_C(Z^\sigma)=\widehat{\Phi}_\sigma(\II_C(Z))
$$
and
$$
\dim(Z^\sigma\cap R^n)\leq\dim_C(Z^\sigma)=\dim_C(Z)=\dim(Y),
$$
where $\Phi_\sigma:C\to C$ is a fixed automorphism such that $\Phi_\sigma|_E=\sigma$ (see Lemma \emph{\ref{extension}}) and $\widehat{\Phi}_\sigma:C[\x]\to C[\x]$ is the corresponding ring automorphism, see \eqref{hatpsi}. In particular, $T\subset C^n$ and $T^r\subset R^n$ are algebraic sets and $\dim_C(T)=\dim(T^r)=\dim(Y)$.
\item[$(\mr{ii})$] Given any $\sigma\in G'$, we have that $Z^\sigma\subset C^n$ is irreducible if and only if $Y\subset R^n$ is $\kr$-irreducible (or, equivalently, if $Y\subset R^n$ is irreducible).
\item[$(\mr{iii})$] Let ${\mathfrak H}\subset E[\x]$ be the set of all products of the form $\prod_{\sigma\in G'}h_\sigma$, where $h_\sigma\in\{g_1^\sigma,\ldots,g_r^\sigma\}$ for each $\sigma\in G'$. For each $h\in{\mathfrak H}$, define 
$$
P_h(\t):=\prod_{\tau\in G'}(\t-h^\tau)=\t^d+\sum_{j=1}^d(-1)^jq_{hj}\t^{d-j}\in E[\x][\t],
$$
where $d$ is the order of $G'$. Set ${\mathfrak G}:=\{q_{hj}\}_{h\in{\mathfrak H},\,j\in\{1,\ldots,d\}}$. Then ${\mathfrak G}$ is a subset of $K[\x]$ and $T=\ZZ_C({\mathfrak G})$.
\item[$(\mr{iv})$] $T=\zcl_{C^n}^K(Y)$ and $T^r=\zcl_{R^n}^K(Y)$. In particular, $T\subset C^n$ and $T^r\subset R^n$ are $K$-algebraic sets, and $T=\zcl_{C^n}^K(T^r)$.
\item[$(\mr{v})$] $\II_K(T)=\II_K(T^r)=\II_K(Y)=\sqrt{{\mathfrak G}K[\x]}$, where $\sqrt{{\mathfrak G}K[\x]}$ is the radical of ${\mathfrak G}K[\x]$ in $K[\x]$. Moreover, $\II_C(T)=\II_K(Y)C[\x]$ and $\II_R(T^r)=\sqrt[\mr{r}]{\II_K(Y)R[\x]}=\sqrt[\mr{r}]{{\mathfrak G}R[\x]}$.
\item[$(\mr{vi})$] $\zcl_{C^n}(T^r)\subset T$ and $\II_R(T)=\II_K(T^r)R[\x]$. Moreover, if $G'^*$ is the subset of $G'$ of all $\sigma$ such that $\II_R(Z^\sigma)$ is a non-real ideal of $R[\x]$ and $\gtb:=\bigcap_{\sigma\in G'^*}\II_R(Z^\sigma)$ (where $\gtb:=R[\x]$ if $G'^*=\varnothing$), then $\II_R(T)=\II_R(T^r)\cap \gtb$, the ideal $\gtb$ of $R[\x]$ is radical and $\ZZ_R(\gtb)=\bigcup_{\sigma\in G'^*}(Z^\sigma \cap R^n)\subset T^r$. If in addition the algebraic set $Y\subset R^n$ is $\kr$-irreducible, then $\dim(Z^\sigma\cap R^n)<\dim(T^r)$ for each $\sigma\in G'^*$ and $\dim(\ZZ_R(\gtb))<\dim(T^r)$.
\end{itemize}
\end{thm}
\begin{proof}
$(\mr{i})\,\&\,(\mr{i}')$ By Theorem \ref{thm:gc0}$(\mr{i})(\mr{ii})(\mr{ii}')(\mr{vi})$, we know that $T=\bigcup_{\sigma\in G'}Z^\sigma=\bigcup_{\psi\in G}\psi_n(Z)$, $\II_C(\psi_n(Z))=\widehat{\psi}(\II_C(Z))$, $\II_C(Z^\sigma)=\widehat{\Phi}_\sigma(\II_C(Z))$ and $\dim_C(\psi_n(Z))=\dim_C(Z^\sigma)=\dim_C(Z)=\dim_C(T)$ for each $\psi\in G$ and $\sigma\in G'$. Observe that $T^r=T\cap R^n=\bigcup_{\sigma\in G'}(Z^\sigma\cap R^n)=\bigcup_{\psi\in G}(\psi_n(Z)\cap R^n)$. By Proposition \ref{prop:zar}$(\mr{iv})$ and Corollary \ref{<=dim}$(\mr{i})$, we have: $\dim(\psi_n(Z)\cap R^n)\leq\dim_C(\psi_n(Z))=\dim_C(Z)=\dim(Y)$ and $\dim(Z^\sigma\cap R^n)\leq\dim_C(Z^\sigma)=\dim_C(Z)=\dim(Y)$. Observe that, if $e$ is the identity of $G'$, by Proposition \ref{prop:zar}$(\mr{ii})$, $Z^e\cap R^n=Z\cap R^n=Y$. It follows that $T=\bigcup_{\sigma\in G'}Z^\sigma\subset C^n$ and $T^r=\bigcup_{\sigma\in G'}(Z^\sigma\cap R^n)\subset R^n$ are algebraic sets and
$$\textstyle
\dim(T^r)=\max_{\sigma\in G'}\{\dim(Z^\sigma\cap R^n)\}=\dim(Y)=\dim_C(Z)=\dim_C(T).
$$

$(\mr{ii})$ By Proposition \ref{prop:zar}$(\mr{iii})$ and Corollary \ref{inter}$(\mr{iii})$, $Z\subset C^n$ is irreducible if and only if $Y\subset R^n$ is irreducible or, equivalently, if $Y\subset R^n$ is $\kr$-irreducible. In addition, $Z^\sigma\subset C^n$ is irreducible for $\sigma\in G'$ if and only if so is $Z\subset C^n$ by Theorem \ref{thm:gc0}$(\mr{ii}')$.

$(\mr{iii})$ This item follows from Theorem \ref{thm:gc0}$(\mr{iii})(\mr{iv})$.

$(\mr{iv})$ By Theorem \ref{thm:gc0}$(\mr{v})$, we have 
$$
T=\zcl_{C^n}^K(Z)=\zcl_{C^n}^K(\zcl_{C^n}(Y))=\zcl_{C^n}^K(Y).
$$
As $T=\ZZ_C(\II_K(Y))$, it holds: $T^r=T\cap R^n=\ZZ_C(\II_K(Y))\cap R^n=\ZZ_R(\II_K(Y))=\zcl_{R^n}^K(Y)$. As $Y\subset T^r\subset T=\zcl_{C^n}^K(Y)$, we deduce $\zcl_{C^n}^K(T^r)=T$.

$(\mr{v})$ As $T=\zcl_{C^n}^K(Y)$ and $T^r=\zcl_{R^n}^K(Y)$, we deduce $\II_K(T)=\II_K(Y)=\II_K(T^r)$. By Theorem \ref{thm:gc0}$(\mr{vi})$, we have $\II_K(T)=\sqrt{{\mathfrak G}K[\x]}$ and $\II_C(T)=\II_K(T)C[\x]=\II_K(Y)C[\x]$. In particular, $\II_K(Y)=\sqrt{{\mathfrak G}K[\x]}$. As $T^r=\ZZ_R(\II_K(Y))$, the Real Nullstellensatz \cite[Cor.4.1.8]{bcr} implies $\II_R(T^r)=\sqrt[\mr{r}]{\II_K(Y)R[\x]}=\sqrt[\mr{r}]{\gtd R[\x]}$, where $\gtd:=\sqrt{{\mathfrak G}K[\x]}$. As $\gtd$ is a radical ideal of $K[\x]$, Lemma \ref{radical} assures that $\gtd R[\x]$ is a radical ideal of $R[\x]$. As ${\mathfrak G}R[\x]\subset\gtd R[\x]\subset\sqrt{{\mathfrak G}R[\x]}$, it follows that $\gtd R[\x]=\sqrt{{\mathfrak G}R[\x]}$, so $\II_R(T^r)=\sqrt[\mr{r}]{\sqrt{{\mathfrak G}R[\x]}}=\sqrt[\mr{r}]{{\mathfrak G}R[\x]}$.

$(\mr{vi})$ As $T\subset C^n$ is an algebraic set that contains $T^r$, it holds $\zcl_{C^n}(T^r)\subset T$. Now the equality $\II_C(T)=\II_K(T^r)C[\x]$ proven in $(\mr{v})$ and Corollary \ref{k} imply
$$
\II_R(T)=\II_C(T)\cap R[x]=\II_K(T^r)C[\x]\cap R[x]=(\II_K(T^r)R[\x])C[\x]\cap R[x]=\II_K(T^r)R[\x].
$$
Observe that the ideal $\gtb$ of $R[\x]$ is radical, because each $\II_R(Z^\sigma)$ is radical. By Corollary \ref{cor:ii-zz}$(\mr{i})(\mr{ii})$, we have
$$\textstyle
\ZZ_R(\gtb)=\bigcup_{\sigma\in G'^*}\ZZ_R(\II_R(Z^\sigma))=\bigcup_{\sigma\in G'^*}(Z^\sigma \cap R^n)\subset T\cap R^n=T^r
$$
and $\II_R(Z^\sigma)=\II_R(Z^\sigma\cap R^n)$ for each $\sigma\in G'\setminus G'^*$. As $\II_R(Z^\sigma)\subset\II_R(Z^\sigma\cap R^n)$ for each $\sigma\in G'^*$, we deduce:
\begin{align*}
\gtb\cap\II_R(T^r)&\textstyle=\gtb\cap\II_R(\bigcup_{\sigma\in G'}(Z^\sigma\cap R^n))=\gtb\cap\bigcap_{\sigma\in G'}\II_R(Z^\sigma\cap R^n)\\
&\textstyle=\bigcap_{\sigma\in G'^*}\II_R(Z^\sigma)\cap\big(\bigcap_{\sigma\in G'^*}\II_R(Z^\sigma\cap R^n)\cap\bigcap_{\sigma\in G'\setminus G'^*}\II_R(Z^\sigma\cap R^n)\big)\\
&\textstyle=\big(\bigcap_{\sigma\in G'^*}\II_R(Z^\sigma)\cap\bigcap_{\sigma\in G'^*}\II_R(Z^\sigma\cap R^n)\big)\cap\bigcap_{\sigma\in G\setminus G'^*}\II_R(Z^\sigma\cap R^n)\\
&\textstyle=\bigcap_{\sigma\in G'^*}\II_R(Z^\sigma)\cap\bigcap_{\sigma\in G'\setminus G'^*}\II_R(Z^\sigma)\\
&\textstyle=\bigcap_{\sigma\in G'}\II_R(Z^\sigma)=\II_R(\bigcup_{\sigma\in G'}Z^\sigma)=\II_R(T).
\end{align*}

Finally, suppose that $Y\subset R^n$ is $\kr$-irreducible. By $(\mr{ii})$, each $Z^\sigma\subset C^n$ is irreducible. Thus, Corollary \ref{cor:ii-zz}$(\mr{ii}')$ and item $(\mr{i}')$ assure that $\dim(Z^\sigma\cap R^n)<\dim_C(Z^\sigma)=\dim(T^r)$ for each $\sigma\in G'^*$. It follows that $\dim(\ZZ_R(\gtb))=\max_{\sigma\in G'^*}\{\dim(Z^\sigma \cap R^n)\}<\dim(T^r)$, as required.
\end{proof}

\begin{remarks}\label{rem:ref}
(i) In the preceding item $(\mr{vi})$, the inclusion $\zcl_{C^n}(T^r)\subset T$ may be strict, that is, it may happen that the complexification of $T^r$ is strictly contained in $T$, see Examples \ref{exa:gc}.

(ii) It may happen that $\dim(\ZZ_R(\gtb))=\dim(T^r)$ if $Y\subset R^n$ is not $\kr$-irreducible. Consider $Y:=\ZZ_R(\x_2(\x_2-\sqrt[4]{2}\x_1+\sqrt{2}))\subset R^2$, which has dimension $1$. We have
\begin{align*}
T&=\ZZ_C(\x_2(\x_2-\sqrt[4]{2}\x_1+\sqrt{2})(\x_2+\sqrt[4]{2}\x_1+\sqrt{2})(\x_2-\sqrt[4]{2}\ii\x_1-\sqrt{2})(\x_2+\sqrt[4]{2}\ii\x_1-\sqrt{2})),\\
T^r&=\ZZ_R(\x_2(\x_2-\sqrt[4]{2}\x_1+\sqrt{2})(\x_2+\sqrt[4]{2}\x_1+\sqrt{2}))\cup\{(0,\sqrt{2})\}.
\end{align*}
Consequently,
\begin{align*}
\II_R(T)&=\x_2(\x_2-\sqrt[4]{2}\x_1+\sqrt{2})(\x_2+\sqrt[4]{2}\x_1+\sqrt{2})(\x_2-\sqrt[4]{2}\ii\x_1-\sqrt{2})(\x_2+\sqrt[4]{2}\ii\x_1-\sqrt{2})R[\x_1,\x_2],\\
\II_R(T^r)&=(\x_2(\x_2-\sqrt[4]{2}\x_1+\sqrt{2})(\x_2+\sqrt[4]{2}\x_1+\sqrt{2}))R[\x_1,\x_2]\cap(\x_1,\x_2-\sqrt{2})R[\x_1,\x_2],\\
\gtb&=(\x_2((\x_2-\sqrt{2})^2+\sqrt{2}\x_1^2))R[\x_1,\x_2]
\end{align*}
and $\ZZ_R(\gtb)=\ZZ_R(\x_2)\cup\{(0,\sqrt{2})\}$, which has dimension $1$. $\sqbullet$
\end{remarks}

\begin{defn}
Let $Y$ be a subset of $R^n$ and let $Z:=\zcl_{C^n}(Y)\subset C^n$. We say the set $Y\subset R^n$ is \emph{$G$-stable} if $\psi_n(Z)\cap R^n\subset Y$ for each $\psi\in G$, where $G:=G(C:K)$. $\sqbullet$
\end{defn}

Theorem \ref{thm:gc} has the following immediate consequence:

\begin{cor}[Real $K$-algebraicity and $G$-stability]\label{cor:rk-algebraicity}
A $\kr$-algebraic subset of $R^n$ is $K$-alge\-braic if and only if it is $G$-stable.
\end{cor}

We will apply the latter corollary in Examples \ref{432}$(\mr{v})$.

\subsubsection{Simultaneous Galois completions}\label{gcsequence} Let $(Y_1,\ldots,Y_s)$ be a $s$-tuple of $\kr$-algebraic subsets of~$R^n$. We associate to $(Y_1,\ldots,Y_s)$ the $s$-tuple $(T_1,\ldots,T_s)$ of their complex Galois completions and the $s$-tuple $(T_1^r,\ldots,T_s^r)$ of their real Galois completions. By Theorem \ref{thm:gc}$(\mr{iv})$, $T_i$ is the $K$-Zariski closure of $Y_i$ in $C^n$ and $T_i^r$ the $K$-Zariski closure of $Y$ in $R^n$. To compute $(T_1,\ldots,T_s)$, we apply simultaneously Algorithm \ref{gc} to the complexifications $Z_i\subset C^n$ of all the $Y_i$. Recall that the crucial point is to consider a finite Galois subextension $E|K$ of $\kbar|K$ that contains all the coefficients of polynomials $g_{i1},\ldots,g_{ir_i}$ in $\kbar[\x]$ such that $\ZZ_C(g_{i1},\ldots,g_{ir_i})=Z_i$ for each $i\in\{1,\ldots,s\}$. For instance, we can choose by \eqref{equaz} generators $g_{i1},\ldots,g_{ir_i}$ of $\II_\kr(Y_i)$ in $\kr[\x]$. Once this is done, we set $(T_1^r,\ldots,T_s^r):=(T_1\cap R^n,\ldots,T_s\cap R^n)$. We call $(T_1,\ldots,T_s)$ the \emph{complex Galois completion of $(Y_1,\ldots,Y_s)$} and $(T^r_1,\ldots,T^r_s)$ the \emph{real Galois completion of $(Y_1,\ldots,Y_s)$ (with respect to the extension of fields $C|K$)}.

\subsubsection{Galois presentation of a real $K$-algebraic set and $K$-bad points}
Let $X\subset R^n$ be a $K$-algebraic set. Similarly to what we have done in Subsection \ref{gp0}, we want to find a `minimal' algebraic set $Y\subset R^n$ whose $K$-Zariski closure is $X$. The particular properties of real algebraic sets force us to make a more sophisticated discussion than the one of Subsection \ref{gp0}.
 
\begin{lem}\label{lem:gp}
Let $X\subset R^n$ be a $K$-irreducible $K$-algebraic set of dimension $d$ and let $Y\subset R^n$ be an $R$-irreducible component of $X$ of dimension $d$. We have:
\begin{itemize}
\item[$(\mr{i})$] $Y\subset R^n$ is a $\kr$-irreducible component of $X$.
\item[$(\mr{ii})$] Let $Z\subset C^n$ be the complexification of $Y$. Apply Algorithm \emph{\ref{gc}} to $Z\subset C^n$ and obtain: a finite Galois group $G'$, a family $\{Z^\sigma\}_{\sigma\in G'}$ of algebraic subsets of $C^n$ and the complex and real Galois completions $T$ and $T^r$ of $Y\subset R^n$. Then $T^r=\bigcup_{\sigma\in G'}(Z^\sigma\cap R^n)=X=\zcl_{R^n}^K(Y)$ and $T=\bigcup_{\sigma\in G'}Z^\sigma=\zcl^K_{C^n}(X)=\zcl_{C^n}^K(Y)$.
\item[$(\mr{ii}')$] $T\subset C^n$ is $K$-irreducible.
\item[$(\mr{iii})$] The family $\{Z^\sigma\}_{\sigma\in G'}$ coincides with that of all $C$-irreducible components of $T\subset C^n$. Moreover, $\dim_C(Z^\sigma)=d$ for each $\sigma\in G'$.
\item[$(\mr{iii}')$] If $G'^*$ is the set of all $\sigma\in G'$ such that $\dim(Z^\sigma\cap R^n)<d$, then $\{Z^\sigma\cap R^n\}_{\sigma\in G'\setminus G'^*}$ coincides with the family of all $R$-irreducible components of $X\subset R^n$ of dimension $d$.
\item[$(\mr{iv})$] $Z^\sigma\subset C^n$ is a $\kbar$-algebraic set and $Z^\sigma\cap R^n\subset R^n$ is a $\kr$-algebraic set for each $\sigma\in G'$. 
\end{itemize}
\end{lem}
\begin{proof}
$(\mr{i})$ As $X\subset R^n$ is also $\kr$-algebraic, $Y\subset R^n$ is by Corollary \ref{inter}$(\mr{iii})$ a $\kr$-irreducible component of $X$. In particular, $Z\subset C^n$ is $\kbar$-algebraic and we can apply Algorithm \ref{gc} to $Z$.

$(\mr{ii})$ By Theorem \ref{thm:gc}$(\mr{i}')(\mr{iv})$, $T^r=\zcl_{R^n}^K(Y)\subset\zcl_{R^n}^K(X)=X$ and $d=\dim(Y)=\dim(T^r)$. As $T^r\subset R^n$ and $X\subset R^n$ are $K$-algebraic, Theorem \ref{dimension} assures that $\dim_K(T^r)=\dim(T^r)=d=\dim(X)=\dim_K(X)$. By Lemma \ref{dimirred}, $T^r=X$, because $X\subset R^n$ is a $K$-irreducible $K$-algebraic set. Using again Theorem \ref{thm:gc}$(\mr{i}')(\mr{iv})$, we have $X=T^r=\bigcup_{\sigma\in G'}(Z^\sigma\cap R^n)$ and $T=\bigcup_{\sigma\in G'}Z^\sigma=\zcl_{C^n}^K(Y)=\zcl^K_{C^n}(T^r)=\zcl^K_{C^n}(X)$.

$(\mr{ii}')$ As $T=\zcl_{C^n}^K(X)$, the ideal $\II_K(T)$ coincides with $\II_K(X)$, which is a prime ideal of $K[\x]$ by Lemma \ref{lem:prime}. Thus, Lemma \ref{lem:prime} implies that $T\subset C^n$ is $K$-irreducible.

$(\mr{iii})$ By Theorem \ref{thm:gc}$(\mr{i}')(\mr{ii})$, $T=\bigcup_{\sigma\in G'}Z^\sigma$ and all the $Z^\sigma$ are irreducible algebraic subsets of $C^n$ of the same (complex) dimension $d$. It follows immediately that $\{Z^\sigma\}_{\sigma\in G'}$ is the family of all $C$-irreducible components of $T\subset C^n$.

$(\mr{iii}')$ By Theorem \ref{thm:gc}$(\mr{i}')$ and Corollary \ref{<=dim}$(\mr{ii})$, $Z^\sigma\cap R^n\subset R^n$ has dimension $d$ and is irreducible for each $\sigma\in G'\setminus G'^*$. Now it is enough to observe that $X=T^r=\bigcup_{\sigma\in G'\setminus G'^*}(Z^\sigma\cap R^n)\cup\bigcup_{\sigma\in G'^*}(Z^\sigma\cap R^n)$ and $\dim(\bigcup_{\sigma\in G'^*}(Z^\sigma\cap R^n))<d$.

$(\mr{iv})$ Choose $\sigma\in G'$. By step {\it(3)} of Algorithm \ref{gc}, we have $Z^\sigma=\ZZ_C(h_1,\ldots,h_r)$ for some polynomials $h_1,\ldots,h_r$ in $\kbar[\x]$, so $Z^\sigma\subset C^n$ is $\kbar$-algebraic. For each $j\in\{1,\ldots,r\}$, denote $a_j$ and $b_j$ the unique polynomials in $\kr[\x]$ such that $h_j=a_j+\ii b_j$. As $Z^\sigma\cap R^n=\ZZ_R(a_1,b_1,\ldots,a_r,b_r)$, the set $Z^\sigma\cap R^n\subset R^n$ is $\kr$-algebraic, as required. 
\end{proof}

The preceding lemma allows to introduce the concept of Galois presentation of a real $K$-algebraic set.

\begin{defn}\label{def:gp}
Let $X\subset R^n$ be a $K$-algebraic set, let $(X_1,\ldots,X_s)$ be the $K$-irreducible components of $X$ listed in some order, and let $d_i:=\dim(X_i)$ for each $i\in\{1,\ldots,s\}$. Choose an $R$-irreducible component $Y_i$ of $X_i$ of dimension $d_i$ and denote $Z_i\subset C^n$ the complexification of $Y_i$ for each $i\in\{1,\ldots,s\}$, apply Algorithm \ref{gc} simultaneously to $(Z_1,\ldots,Z_s)$ and obtain: a finite Galois group $G'$ and finite families $\{Z_1^\sigma\}_{\sigma\in G'},\ldots,\{Z_s^\sigma\}_{\sigma\in G'}$ of algebraic subsets of $C^n$. We call the tuple
$$
(Y_1,\ldots,Y_s;G';\{Z_1^\sigma\}_{\sigma\in G'},\ldots,\{Z_s^\sigma\}_{\sigma\in G'})
$$
a \emph{Galois presentation of $X\subset R^n$} and $(Y_1,\ldots,Y_s)$ the \emph{start} of such a presentation. For short, we say that $X=\bigcup_{i=1}^s\bigcup_{\sigma\in G'}(Z_i^\sigma\cap R^n)$ is a Galois presentation of $X$ with start $(Y_1,\ldots,Y_s):=(Z_1^e\cap R^n,\ldots,Z_s^e\cap R^n)$, where $e$ is the identity of $G'$.

If in addition $X\subset R^n$ is $K$-irreducible (that is, $s=1$) and $(Y;G';\{Z^\sigma\}_{\sigma\in G'})$ is a Galois presentation of $X\subset R^n$ with start $Y$, then for short we say that $X=\bigcup_{\sigma\in G'}(Z^\sigma\cap R^n)$ is a Galois presentation of $X\subset R^n$ with start $Y:=Z^e\cap R^n$. $\sqbullet$
\end{defn}

The following result is a version of Lemma \ref{lem:gp} for real $K$-algebraic sets non necessarily $K$-irreducible.

\begin{lem}\label{lem:gp-reducible}
Let $X\subset R^n$ be a $K$-algebraic set of dimension $d$, let $X_1,\ldots,X_s$ be the $K$-irreducible components of $X$, let $Y_i$ be an $R$-irreducible component of $X_i$ of dimension $\dim(X_i)$ for each $i\in\{1,\ldots,s\}$, and let $(Y_1,\ldots,Y_s;G';\{Z_1^\sigma\}_{\sigma\in G'},\ldots,\{Z_s^\sigma\}_{\sigma\in G'})$ be a Galois presentation of $X\subset R^n$. Define: $T_i:=\bigcup_{\sigma\in G'}Z^\sigma_i\subset C^n$ and $T^r_i:=T_i\cap R^n=\bigcup_{\sigma\in G'}(Z^\sigma_i\cap R^n)\subset R^n$ for each $i\in\{1,\ldots,s\}$, $T:=\bigcup_{i=1}^sT_i\subset C^n$ and $T^r:=T\cap R^n=\bigcup_{i=1}^sT^r_i\subset R^n$. We have:
\begin{itemize}
\item[$(\mr{i})$] For each $i\in\{1,\ldots,s\}$, it holds: $T_i$ is the complex Galois completion of $Y_i\subset R^n$ and coincides with $\zcl_{C^n}^K(X_i)$, $T^r_i$ is the real Galois completion of $Y_i\subset R^n$ and coincides with $X_i$, $\{Z^\sigma_i\}_{\sigma\in G'}$ coincides with the family of all $C$-irreducible components of $T_i$, $\dim_C(Z^\sigma_i)=\dim(X_i)$, $Z^\sigma_i\subset C^n$ is $\kbar$-algebraic and $Z^\sigma_i\cap R^n\subset R^n$ is $\kr$-algebraic for each $\sigma\in G'$.
\item[$(\mr{ii})$] The set $T$ is the complex Galois completion of $\bigcup_{i=1}^sY_i\subset R^n$ and coincides with $\zcl_{C^n}^K(X)$. The set $T^r$ is the real Galois completion of $\bigcup_{i=1}^sY_i\subset R^n$ and coincides with $X$.
\item[$(\mr{ii}')$] $T_1,\ldots,T_s$ are the $K$-irreducible components of $T\subset C^n$.
\item[$(\mr{iii})$] $Z^\sigma_i\not\subset Z^\tau_j$ for all $i,j\in\{1,\ldots,s\}$ with $i\neq j$ and for all $\sigma,\tau\in G'$. Thus, $\{Z^\sigma_i\}_{i\in\{1,\ldots,s\},\sigma\in G'}$ is the family of the $C$-irreducible components of $T$.
\item[$(\mr{iii}')$] For each $i\in\{1,\ldots,s\}$, let $G'^*_i$ be the set of all $\sigma\in G'$ such that $\dim(Z_i^\sigma\cap R^n)<\dim(X_i)$. If $i,j\in\{1,\ldots,s\}$ with $i\neq j$, $\sigma\in G'\setminus G'^*_i$ and $\tau\in G'$, then $Z^\sigma_i\cap R^n\not\subset Z^\tau_j\cap R^n$. Hence, the family $\{Z_i^\sigma\cap R^n\}_{i\in\{1,\ldots,s\},\sigma\in G'\setminus G'^*_i}$ is contained in (and possibly coincides with) the family of the $R$-irreducible components of $X$.
\end{itemize}
\end{lem}
\begin{proof}
$(\mr{i})$ This item follows immediately from Lemma \ref{lem:gp} applied to each $X_i$.

$(\mr{ii})$ By $(\mr{i})$ and Theorem \ref{thm:gc}$(\mr{iv})$, we have $T_i=\zcl_{C^n}^K(Y_i)=\zcl_{C^n}^K(X_i)$ and $T^r_i=\zcl_{R^n}^K(Y_i)=X_i$ for each $i\in\{1,\ldots,s\}$, so $T=\bigcup_{i=1}^s\zcl_{C^n}^K(Y_i)=\zcl_{C^n}^K(\bigcup_{i=1}^sY_i)$, $T=\bigcup_{i=1}^s\zcl_{C^n}^K(X_i)=\zcl_{C^n}^K(\bigcup_{i=1}^sX_i)=\zcl_{C^n}^K(X)$, $T^r=\bigcup_{i=1}^s\zcl_{R^n}^K(Y_i)=\zcl_{R^n}^K(\bigcup_{i=1}^sY_i)$ and $T^r=\bigcup_{i=1}^sX_i=X$. In particular, $T=\zcl_{C^n}^K(\bigcup_{i=1}^sY_i)$ and $T^r=\zcl_{R^n}^K(\bigcup_{i=1}^sY_i)$ so Theorem \ref{thm:gc}$(\mr{i})(\mr{iv})$ assures that $T$ and $T^r$ are the complex and real Galois completions of $\bigcup_{i=1}^sY_i$, respectively. 

$(\mr{ii}')$\&$(\mr{iii})$ By Lemma \ref{lem:gp}$(\mr{ii}')$, each $T_i\subset C^n$ is $K$-irreducible. Observe that $T_i\not\subset T_j$ for each $i,j\in\{1,\ldots,s\}$ with $i\neq j$. Otherwise, by $(\mr{i})$, $X_i=T^r_i=T_i\cap R^n\subset T_j\cap R^n=T^r_j=X_j$ for some indexes $i,j$ with $i\neq j$, which is a contradiction. This proves $(\mr{ii}')$. By $(\mr{i})$, $\{Z^\sigma_i\}_{\sigma\in G'}$ is the family of all $C$-irreducible components of $T_i$ for each $i\in\{1,\ldots,s\}$, so $(\mr{iii})$ follows immediately from Corollary \ref{cho} applied to $T$.

$(\mr{iii}')$ Let $i\in\{1,\ldots,s\}$ and let $\sigma\in G'\setminus G'^*_i$. By Corollary \ref{<=dim}, we know that $Z_i^\sigma\cap R^n\subset R^n$ has dimension $\dim(X_i)=\max_{\sigma'\in G'}\dim{(Z_i^{\sigma'}\cap R^n)}$ and is irreducible, and $\zcl_{C^n}(Z_i^\sigma\cap R^n)=Z_i^\sigma$. As $X=\bigcup_{j=1}^s\bigcup_{\tau\in G'}(Z_j^\tau\cap R^n)$, it is enough to show that $Z_i^\sigma \cap R^n\not\subset Z_j^\tau\cap R^n$ for all $j\in \{1,\ldots,s\}\setminus\{i\}$ and $\tau\in G'$. If this would be false, then $Z_i^\sigma \cap R^n\subset Z_j^\tau\cap R^n$ for some $j\in \{1,\ldots,s\}\setminus\{i\}$ and $\tau\in G'$, so $Z_i^\sigma=\zcl_{C^n}(Z_i^\sigma \cap R^n)\subset \zcl_{C_n}(Z_j^\tau\cap R^n)\subset Z_j^\tau$, contradicting $(\mr{iii})$.
\end{proof}

We now introduce the concepts of bad set and bad point of a real $K$-algebraic set.

\begin{defns}\label{def:bad-points}
Let $X\subset R^n$ be a $K$-algebraic set of dimension $d$, let $X_1,\ldots,X_s$ be the $K$-irreducible components of $X$, and let $I$ be the subset of $\{1,\ldots,s\}$ constituted by all indices $i$ such that $\dim(X_i)=d$. For each index $i\in I$, let $T_i:=\zcl_{C^n}^K(X_i)$, let $\{T_{i1},\ldots,T_{is_i}\}$ be the family of all $C$-irreducible components of $T_i$, let $J_i$ be the subset of $\{1,\ldots,s_i\}$ of all indices $j$ such that $\dim(T_{ij}\cap R^n)<d$, and let $B_i:=\bigcup_{j\in J_i}(T_{ij}\cap R^n)$. By Lemma \ref{lem:gp-reducible}$(\mr{i})(\mr{ii})(\mr{iii})$, $\{T_{ij}\}_{i\in I,j\in\{1,\ldots,s_i\}}$ is the family of all $C$-irreducible components of $\zcl^K_{C^n}(X)\subset C^n$ of $C$-dimension~$d$.

We define the \emph{$K$-bad set $B_K(X)$ of $X$} as follows:
$$
B_K(X):=\bigcup_{i\in I}B_i\cup\bigcup_{i\in\{1,\ldots,s\}\setminus I}X_i\,,
$$
where $\bigcup_{i\in\{1,\ldots,s\}\setminus I}X_i:=\varnothing$ if $I=\{1,\ldots,s\}$. The points of $B_K(X)$ are the \emph{$K$-bad points of $X$}. $\sqbullet$
\end{defns}

\begin{remark}\label{rem317}
We keep the notations of the previous definition. Consider a Galois presentation $(Y_1,\ldots,Y_s;G';\{Z_1^\sigma\}_{\sigma\in G'},\ldots,\{Z_s^\sigma\}_{\sigma\in G'})$ of the $K$-algebraic set $X\subset R^n$. By Lemma \ref{lem:gp-reducible}$(\mr{i})$, each $X_i$ is the union of the family $\{Z_i^\sigma\cap R^n\}_{\sigma\in G'}$ and, given any $i\in I$, we have that $\{Z_i^\sigma\}_{\sigma\in G'}=\{T_{i1},\ldots,T_{is_i}\}$, and $B_i$ is the union of the family $\{Z_i^\sigma\cap R^n\}_{\sigma\in G'^*_i}$, where $G'^*_i$ is the subset of $G'$ of all $\sigma$ such that $\dim(Z_i^\sigma \cap R^n)<d$. Recall that, by Lemma \ref{lem:gp-reducible}$(\mr{ii})(\mr{iii})$, $\{Z^\sigma_i\}_{i\in\{1,\ldots,s\},\sigma\in G'}$ is the family of the $C$-irreducible components of $\zcl_{C^n}^K(X)\subset C^n$. If $\widetilde{I}$ denotes the subset of $\{1,\ldots,s\}\times G'$ of all pairs $(i,\sigma)$ such that $\dim(Z^\sigma_i\cap R^n)<d$, we have
$$\textstyle
B_K(X)=\bigcup_{i\in I}\bigcup_{\sigma\in G'^*_i}(Z^\sigma_i\cap R^n)\cup\bigcup_{i\in\{1,\ldots,s\}\setminus I}\bigcup_{\sigma\in G'}(Z^\sigma_i\cap R^n)=\bigcup_{(i,\sigma)\in\widetilde{I}}(Z^\sigma_i\cap R^n)\,.
$$

Lemma \ref{lem:gp-reducible}$(\mr{i})$ implies that each set $Z^\sigma_i\cap R^n\subset R^n$ is $\kr$-algebraic, so $B_K(X)\subset R^n$ is $\kr$-algebraic and has dimension $<\dim(X)$. $\sqbullet$
\end{remark}

\begin{example}
Let $K:=\Q$. Consider the $\Q$-irreducible $\Q$-algebraic set $X:=\ZZ_R(\x^3-2)=\{\sqrt[3]{2}\}\subset R$. The polynomial $\x^3-2\in\Q[\x]$ is irreducible and its roots in $C=R[\ii]$ are $\sqrt[3]{2}$, $\sqrt[3]{2}(\frac{1}{2}+\frac{\sqrt{3}}{2}\ii)$ and $\sqrt[3]{2}(\frac{1}{2}-\frac{\sqrt{3}}{2}\ii)$. Consequently, $T:=\big\{\sqrt[3]{2},\sqrt[3]{2}(\frac{1}{2}+\frac{\sqrt{3}}{2}\ii),\sqrt[3]{2}(\frac{1}{2}-\frac{\sqrt{3}}{2}\ii)\big\}$ is the complex Galois completion $\zcl^\Q_C(X)$ of $X\subset R$. The $C$-irreducible components of $T$ are $T_1:=\{\sqrt[3]{2}\}$, $T_2:=\big\{\sqrt[3]{2}(\frac{1}{2}+\frac{\sqrt{3}}{2}\ii)\big\}$ and $T_3:=\big\{\sqrt[3]{2}(\frac{1}{2}-\frac{\sqrt{3}}{2}\ii)\big\}$. As $\dim(X)=0$, we have $B_\Q(X)=\big(\big\{\sqrt[3]{2}(\frac{1}{2}+\frac{\sqrt{3}}{2}\ii)\big\}\cap R\big)\cup\big(\big\{\sqrt[3]{2}(\frac{1}{2}-\frac{\sqrt{3}}{2}\ii)\big\}\cap R\big)=\varnothing$. $\sqbullet$
\end{example}

In Examples \ref{exa:gc}, we will compute explicitly the complex and real Galois completions of some $\kr$-algebraic sets, together with the $K$-bad sets of the real Galois completions.

%%%
\subsubsection{Real clustering phenomenon} The next result extends Theorem \ref{thm:238} to the real case.

\begin{thm}[{$R|E|K$-clustering phenomenon}] \label{thm:R238}
Let $X\subset R^n$ be a $K$-irreducible $K$-algebraic set of dimension $d$, let $Y_1,\ldots,Y_s$ be the $R$-irreducible components of $X\subset R^n$ of dimension $d$, let $E$ be an ordered subfield of $R$ containing $K$, and let $W_1,\ldots,W_t$ be the $E$-irreducible components of $X\subset R^n$ of dimension $d$. Then $t\leq s$ and there exists a partition of $\{1,\ldots,s\}=J_1\sqcup\cdots\sqcup J_t$ such that $\{Y_j:j\in J_i\}$ is the family of $R$-irreducible components of $W_i\subset R^n$ of dimension $d$ for each $i\in\{1,\ldots,t\}$.
\end{thm}
\begin{proof}
Let $T:=\zcl_{C^n}^K(X)\subset C^n$ and let $Y_1^C,\ldots,Y_{s'}^C$ be the $C$-irreducible components of $T\subset C^n$. Applying Lemma \ref{lem:gp}$(\mr{ii})(\mr{ii}')(\mr{iii})$ to an $R$-irreducible component of $X\subset R^n$ of dimension $d$, we deduce that $T\cap R^n=X$, $T\subset C^n$ is $K$-irreducible and $\dim_C(Y_j^C)=d$ for each $j\in\{1,\ldots,s'\}$. Let $W_1^C,\ldots,W_{t'}^C$ be the $E$-irreducible components of $T\subset C^n$. By Theorem \ref{thm:238}, $t'\leq s'$ and there exists a partition of $\{1,\ldots,s'\}=J'_1\sqcup\cdots\sqcup J'_{t'}$ such that $W_i^C=\bigcup_{j\in J'_i}Y_j^C$ is the decomposition of $W_i^C$ into its $C$-irreducible components for each $i\in\{1,\ldots,t'\}$. Define $Y_j:=Y_j^C\cap R^n$ for each $j\in\{1,\ldots,s'\}$ and $W_i:=W_i^C\cap R^n=\bigcup_{j\in J'_i}Y_j$ for each $i\in\{1,\ldots,t'\}$. 

By Corollary \ref{<=dim}$(\mr{i})$, we have $\dim(Y_j)\leq d$ for each $j\in\{1,\ldots,s'\}$. Rearranging the indices $j$ if necessary, we can assume that there exists $s\in\{1,\ldots,s'\}$ such that $\dim(Y_j)=d$ for each $j\in\{1,\ldots,s\}$ and $\dim(Y_j)<d$ otherwise. Set $Y^*:=\bigcup_{j=s+1}^{s'}Y_j$ if $s<s'$ and $Y^*:=\varnothing$ if $s=s'$. Evidently, $Y^*\subset R^n$ is an algebraic set of dimension $<d$. By Corollary \ref{<=dim}$(\mr{ii})$, if $j\in\{1,\ldots,s\}$, then $Y_j\subset R^n$ is irreducible (of dimension $d$) and $\zcl_{C^n}(Y_j)=Y_j^C$. The latter equality implies that $Y_j\not\subset Y_{j'}$ for all $j'\in\{1,\ldots,s\}\setminus\{j\}$, otherwise $Y_j^C\subset Y_{j'}^C$, which is a contradiction. As $X=T\cap R^n=Y^*\cup\bigcup_{j=1}^sY_j$ and $\dim(Y^*)<d$, we deduce that $Y_1,\ldots,Y_s$ are the $R$-irreducible components of $X\subset R^n$ of dimension $d$. 

Observe that $\dim(W_i)=\max_{j\in J'_i}\dim(Y_j)\leq d$ for each $i\in\{1,\ldots,t'\}$. Rearranging the indices $i$ if necessary, we can assume that there exists $t\in\{1,\ldots,t'\}$ such that $\dim(W_i)=d$ for each $i\in\{1,\ldots,t\}$ and $\dim(W_i)<d$ otherwise. Define $W^*:=\bigcup_{i=t+1}^{t'}W_i$ if $t<t'$ and $W^*:=\varnothing$ if $t=t'$. Evidently, $W^*\subset R^n$ is $E$-algebraic and $\dim(W^*)<d$. Choose $i\in\{1,\ldots,t\}$. Define $J_i:=J'_i\cap\{1,\ldots,s\}=\{j\in J'_i:\dim(Y_j)=d\}$. As $W_i=\bigcup_{j\in J'_i}Y_j$ and $\dim(W_i)=d$, it follows that $J_i\neq\varnothing$. By Proposition \ref{rc}$(\mr{v})$, we have $\dim_C(\zcl_{C^n}(W_i))=\dim(W_i)=d$. As $\zcl_{C^n}(W_i)\subset\zcl_{C^n}^E(W_i)\subset W_i^C$, we deduce $\dim_C(\zcl_{C^n}^E(W_i))=d=\dim_C(W_i^C)$, so $\zcl_{C^n}^E(W_i)=W_i^C$ by Lemma \ref{dimirred}. The latter equality assures that $\II_E(W_i)=\II_E(W_i^C)$, which is a prime ideal of $E[\x]$. Thus, $W_i\subset R^n$ is $E$-irreducible by Lemma \ref{lem:prime}. In addition, $W_i\not\subset W_{i'}$ for all $i,i'\in\{1,\ldots,t\}$ with $i\neq i'$, otherwise $W_i^C\subset W_{i'}^C$, which is a contradiction. As $X=T\cap R^n=W^*\cup\bigcup_{i=1}^tW_i$ and $\dim(W^*)<d$, we deduce that $W_1,\ldots,W_t$ are the $E$-irreducible components of $X\subset R^n$ of dimension $d$. Finally, observe that: $W_i=\bigcup_{j\in J_i}Y_j\cup\bigcup_{j\in J'_i\setminus J_i}Y_j$, $Y_j\subset R^n$ is irreducible and of dimension $d$ for each $j\in J_i$, $Y_j\not\subset Y_{j'}$ for all $j,j'\in J_i$ with $j\neq j'$ and $\dim(\bigcup_{j\in J'_i\setminus J_i}Y_j)<d$. Thus, $\{Y_j\}_{j\in J_i}$ is the family of irreducible components of $W_i\subset R^n$ of dimension $d$, as required.
\end{proof}

\subsection{Real $K$-geometric polynomials and hypersurfaces}\label{rKgh}
Consider a polynomial $g\in\kr[\x]$ and set $Y:=\ZZ_R(g)\subset R^n$. We want to compute the complex and real Galois completions of $Y\subset R^n$ using only the polynomial $g$, as we have done in Algorithm \ref{gcpol} for the complex case. Equation \eqref{equaz} and Theorem \ref{thm:gc} suggest to determine when $\II_\kr(Y)=(g)\kr[\x]$, that is, $\II_\kr(\ZZ_R(g))=(g)\kr[\x]$. More generally, we wonder when a polynomial $f\in K[\x]$ has the property $\II_K(\ZZ_R(f))=(f)K[\x]$. If we change $R$ by an algebraically closed field $L$, this problem is trivial because $C|K$-Nullstellensatz (Corollary \ref{kreliablec}) assures that each radical ideal $\gta$ of $K[\x]$ satisfies $\II_K(\ZZ_L(\gta))=\gta$, so each square-free polynomial $f\in K[\x]$ satisfies $\II_K(\ZZ_L(f))=(f)K[\x]$. In the real closed case, the situation is different and multifaceted. We begin with the following definitions.

\begin{defns} \label{321}
Let $f\in K[\x]$. We say that $f$ is \emph{$K$-geometric in $R^n$} if $\II_K(\ZZ_R(f))=(f)K[\x]$. A polynomial is \emph{geometric in $R^n$} if it is $R$-geometric in $R^n$.

A set $X\subset R^n$ is a \emph{$K$-geometric (algebraic) hypersurface} of $R^n$ if $X=\ZZ_R(f)$ for some $K$-geometric polynomial $f$ in $R^n$ or, equivalently, if $X\subset R^n$ is $K$-algebraic and $\II_K(X)$ is a principal ideal of $K[\x]$. A subset of $R^n$ is a \emph{geometric (algebraic) hypersurface} of $R^n$ if it is a $R$-geometric hypersurface of $R^n$. $\sqbullet$
\end{defns}

As zero ideals are always radical, if a polynomial $f\in K[\x]$ is $K$-geometric in $R^n$, then $f$ is square-free. In addition, if $g$ is also $K$-geometric in $R^n$ and $\ZZ_R(f)=\ZZ_R(g)$, then $f$ and $g$ are associated. Observe that $f:=\x_1^2+1$ is a square-free polynomial in $K[\x]$, which is not $K$-geometric in~$R^n$. 

\begin{lem}\label{univ}
$(\mr{i})$ Let $L|R$ be an extension of real closed fields and let $f\in K[\x]$. Then $f$ is $K$-geometric in $R^n$ if and only if $f$ is $K$-geometric in $L^n$. In particular, if $K=R$, the polynomial $f$ is geometric in $R^n$ if and only if it is $R$-geometric in $L^n$.

$(\mr{ii})$ Let $R|E|K$ be an extension of fields. If $g\in K[\x]\subset E[\x]$ is $E$-geometric in $R^n$, then $g$ is also $K$-geometric in $R^n$.
\end{lem}
\begin{proof}
$(\mr{i})$As $\ZZ_L(f)=(\ZZ_R(f))_L$, Proposition \ref{extension-zar}$(\mr{ii})$ assures that $\II_R(\ZZ_L(f))=\II_R(\ZZ_R(f))$ so $\II_K(\ZZ_L(f))=\II_R(\ZZ_L(f))\cap K[\x]=\II_R(\ZZ_R(f))\cap K[\x]=\II_K(\ZZ_R(f))$.

$(\mr{ii})$ This item follows immediately from Corollary \ref{k}:
$$
\II_{K}(\ZZ_R(g))=\II_E(\ZZ_R(g))\cap K[\x]=(g)E[\x]\cap K[\x]=((g)K[\x])E[\x]\cap K[\x]=(g)K[\x],
$$
as required.
\end{proof}

\begin{remark}\label{2210}
The converse of Lemma \ref{univ}$(\mr{ii})$ is false in general: there are extensions of fields $R|E|K$ and polynomials $g$ in $K[\x]$, which are $K$-geometric in $R^n$, but not $E$-geometric in $R^n$. See Corollary \ref{corro} for a result in the opposite direction.

A simple counterexample in the case $R|E|K=R|R|\Q$ is the following. Set $g:=\x_1^3-2\in\Q[\x_1]$. We have: $\II_\Q(\ZZ_R(g))=\II_\Q(\{\sqrt[3]{2}\})=(g)\Q[\x_1]$, $\II_R(\ZZ_R(g))=\II_R(\{\sqrt[3]{2}\})=(\x_1-\sqrt[3]{2})R[\x_1]\neq (g)R[\x_1]$, so $g$ is $\Q$-geometric in $R$, but not ($R$-)geometric in $R$.

In Examples \ref{exa:gc}, we will see two additional counterexamples in the case $R|E|K=R|R|\Q$. More precisely, we will construct two polynomials $g^\bullet,q^\bullet\in\Q[\x]:=\Q[\x_1,\x_2,\x_3]$ with the following properties:
\begin{itemize}
\item $g^\bullet$ is $\Q$-geometric but not ($R$-)geometric in $R^3$, so we have $\II_R(\ZZ_R(g^\bullet))\neq(g^\bullet)R[\x]=\II_\Q(\ZZ_R(g^\bullet))R[\x]$. In addition, $g^\bullet$ is irreducible in $\Q[\x]$ and reducible in $R[\x]$, and $\ZZ_R(g^\bullet)$ is a $\Q$-irreducible $\Q$-geometric hypersurface of $R^3$, but not a ($R$-)geometric hypersurface of $R^3$.
\item $q^\bullet$ is $\Q$-geometric but not ($R$-)geometric in $R^3$, so we have $\II_R(\ZZ_R(q^\bullet))\neq(q^\bullet)R[\x]=\II_\Q(\ZZ_R(q^\bullet))R[\x]$. In addition, $q^\bullet$ is reducible in $\Q[\x]$ and $\ZZ_R(q^\bullet)$ is a $\Q$-reducible $\Q$-geometric hypersurface of $R^3$. Actually, $\ZZ_R(q^\bullet)$ is also a ($R$-)geometric hypersurface of $R^3$, because there exists a ($R$-)geometric polynomial $Q$ in $R^3$ that divides $q^\bullet$ and has the same zero set of $q^\bullet$ in $R^3$. $\sqbullet$
\end{itemize}
\end{remark}

A polynomial $f\in R[\x]$ {\em changes sign in $R^n$} if there exist $x,y\in R^n$ such that $f(x)f(y)<0$ and we say that $f$ {\em has a nonsingular zero in $R^n$} if there exists $x\in\ZZ_R(f)$ such that $\nabla f(x)=\big(\frac{\partial f}{\partial \x_1}(x),\ldots,\frac{\partial f}{\partial \x_n}(x)\big)\neq0$.

\begin{prop}[$K$-geometric polynomials]\label{prop:hyper}
Let $f\in K[\x]$ be a square-free polynomial and let $f=f_1\cdots f_r$ be its factorization in $K[\x]$. The following conditions are equivalent.
\begin{itemize}
\item[$(\mr{i})$] $f$ is $K$-geometric in $R^n$.
\item[$(\mr{i}')$] $\sqrt[r]{(f)R[\x]}\cap K[\x]=(f)K[\x]$.
\item[$(\mr{ii})$] $f_i$ is $K$-geometric in $R^n$ for each $i\in\{1,\ldots,r\}$.
\item[$(\mr{iii})$] $f_i$ has a nonsingular zero in $R^n$ for each $i\in\{1,\ldots,r\}$.
\item[$(\mr{iv})$] $f_i$ changes sign in $R^n$ for each $i\in\{1,\ldots,r\}$.
\item[$(\mr{v})$] $\dim(\ZZ_R(f_i))=n-1$ for each $i\in\{1,\ldots,r\}$.
\end{itemize}
 
If $f$ is $K$-geometric in $R^n$, then the sets $\ZZ_R(f_1),\ldots,\ZZ_R(f_r)$ are $K$-irreducible $K$-geometric hypersurfaces of $R^n$ and they are the $K$-irreducible components of the $K$-geometric hypersurface $\ZZ_R(f)$ of $R^n$.
\end{prop}
\begin{proof}
$(\mr{i})\Longleftrightarrow(\mr{i}')\ $ $\II_R(\ZZ_R(f))=\sqrt[r]{(f)R[\x]}$ by the Real Nullstellensatz \cite[Cor.4.1.8]{bcr}. As $\II_K(\ZZ_R(f))=\II_R(\ZZ_R(f))\cap K[\x]$, $f$ is $K$-geometric in $R^n$ if and only if $\sqrt[r]{(f)R[\x]}\cap K[\x]=(f)K[\x]$.

$(\mr{i})\Longrightarrow(\mr{ii})\ $ Suppose there exists $i\in\{1,\ldots,r\}$ such that $\II_K(\ZZ_R(f_i))\supsetneqq(f_i)K[\x]$, that is, there exists $g\in K[\x]$ vanishing on $\ZZ_R(f_i)$, which is not divisible by $f_i$. Thus, the polynomial $f_1\cdots f_{i-1}gf_{i+1}\cdots f_r$ belongs to $\II_K(\ZZ_R(f))\setminus(f)K[\x]$, which is impossible by $(\mr{i})$.

$(\mr{ii})\Longrightarrow(\mr{i})\ $ We have $\II_K(\ZZ_R(f))=\bigcap_{i=1}^r\II_K(\ZZ_R(f_i))=\bigcap_{i=1}^r(f_i)K[\x]=(f)K[\x]$.

Let us fix an index $i\in\{1,\ldots,r\}$.

$(\mr{ii})\Longrightarrow(\mr{iii})\ $ If $\nabla f_i$ vanishes identically on $\ZZ_R(f_i)$, then $f_i$ divides each partial derivative $\frac{\partial f_i}{\partial \x_j}$, so $\frac{\partial f_i}{\partial\x_j}=0$ for $j=1,\ldots,n$. Thus, $f_i$ is constant, which is impossible, because $f_i$ is irreducible in $K[\x]$.

$(\mr{iii})\Longrightarrow(\mr{iv})\ $ If $\nabla f_i(x)\neq0$ for some $x\in\ZZ_R(f_i)$, the polynomial function $\phi:R\to R$, defined by $f(t):=f_i(x+t\nabla f_i(x))$, vanishes at $0$ and is strictly increasing locally at $0$, because $\phi'(0)=\sum_{j=1}^n\big(\frac{\partial f_i}{\partial\x_j}(x)\big)^2>0$.

$(\mr{iv})\Longrightarrow(\mr{v})\ $ Use \cite[Lem.4.5.2]{bcr}.

$(\mr{v})\Longrightarrow(\mr{ii})\ $ Let $\gta:=\II_K(\ZZ_R(f_i))$. Observe that $n-1=\dim_K(\ZZ_R(f_i))=\dim(K[\x]/\gta)$. As $\gta$ is a radical ideal of $K[\x]$, there exists a finite family of prime ideals $\gtp_1,\ldots,\gtp_t$ of $K[\x]$ such that $\gta=\gtp_1\cap\ldots\cap\gtp_t$. Thus, $n-1=\dim(K[\x]/\gta)=\dim(K[\x]/\gtp_j)$ for some $j\in\{1,\ldots,t\}$. By \eqref{eisenbud}, $\mr{ht}(\gtp_j)=\dim(K[\x])-\dim(K[\x]/\gtp_j)=n-(n-1)=1$. As $\mr{ht}((f_i)K[\x])=1=\mr{ht}(\gtp_j)$, $(f_i)K[\x]\subset\gta\subset\gtp_j$ and $(f_i)K[\x]$ is a prime ideal of $K[\x]$, it follows $(f_i)K[\x]=\gta=\gtp_j$, so $(f_i)K[\x]=\II_K(\ZZ_R(f_i))$. Thus, $f_i$ is $K$-geometric in $R^n$.

Let us prove the last part of the statement. Assume that $f=f_1\cdots f_r$ is $K$-geometric in $R^n$. Thus, each $f_i$ is $K$-geometric in $R^n$, so $\ZZ_R(f_i)$ is a $K$-geometric hypersurface of $R^n$. As $\II_K(\ZZ_R(f_i))=(f_i)K[\x]$ is a prime ideal of $K[\x]$, Lemma \ref{lem:prime} implies that $\ZZ_R(f_i)\subset R^n$ is $K$-irreducible. In addition, $\ZZ_R(f_i)\not\subset\ZZ_R(f_j)$ for all $i,j\in\{1,\ldots,r\}$ with $i\neq j$ because, otherwise, $f_j\in\II_K(\ZZ_R(f_i))=(f_i)K[\x]$ for some $i\neq j$, so $f_i$ divides $f_j$, which is a contradiction. 
\end{proof}

A first immediate consequence of the preceding result is the following:

\begin{cor}[Characterization of $K$-geometric polynomials in $R^n$]\label{cor:char-geometric}
A polynomial $f\in K[\x]$ is $K$-geometric in $R^n$ if and only if it is square-free in $K[\x]$ and its factorization $f=f_1\cdots f_r$ in $K[\x]$ satisfies one of the five equivalent conditions $(\mr{ii})$, $(\mr{iii})$, $(\mr{iv})$ and $(\mr{v})$ stated in Proposition $\ref{prop:hyper}$.
\end{cor}

\begin{remarks}\label{irre}
$(\mr{i})$ Let $R|E$ be an extension of real closed fields, let $f\in E[\x]$ and let $f=f_1^{m_1}\cdots f_r^{m_r}$ be the factorization of $f$ as a polynomial in $E[\x]$. Then $f=f_1^{m_1}\cdots f_r^{m_r}$ is also the factorization of $f$ as a polynomial in $R[\x]$. This is a direct consequence of Tarski-Seidenberg`s principle. A~sketch of the proof is the following. As $E[\x]$ and $R[\x]$ are UFD, it is enough to check the following assertion: if $f$ is irreducible as a polynomial of $E[\x]$, then $f$ is also irreducible as a polynomial in $R[\x]$. Let $d$ be the degree of $f$, let $L=R$ or $L=E$, let $L^N$ be the affine space that parametrizes of all coefficients of polynomials in $L[\x]$ of degree $\leq d$, and let $I^L$ the subset of $L^N$ of all coefficients of irreducible polynomials in $L[\x]$ of degree $\leq d$. The definition of irreducible polynomial and Tarski-Seidenberg's principle imply that $I^L\subset L^N$ is semialgebraic for $L=R$ and $L=E$, and $I^R$ is the extension of coefficients of $I^E\subset E^N$ to $R$, that is, $I^R=(I^E)_R\subset R^N$. Thus, we have $f\in I^E=I^R\cap E^N\subset I^R$, as required.

$(\mr{ii})$ Let $L|F$ be an extension of algebraically closed fields, let $f\in F[\x]$ and let $f=f_1^{m_1}\cdots f_r^{m_r}$ be the factorization of $f$ as a polynomial in $F[\x]$. Then, repeating the previous argument, but changing `semialgebraic' by `constructible', one deduces that $f=f_1^{m_1}\cdots f_r^{m_r}$ is also the factorization of $f$ as a polynomial in $L[\x]$. $\sqbullet$
\end{remarks}

Combining Corollary \ref{cor:char-geometric} with Remarks \ref{irre}$(\mr{i})$, we deduce:

\begin{cor}\label{corro}
If $f\in K[\x]$ and the ordered subfield $K$ of $R$ is real closed, then $f$ is $K$-geometric in $R^n$ if and only if $f$ is geometric in $R^n$.
\end{cor}
\begin{proof}
By Lemma \ref{univ}$(\mr{ii})$, it is enough to show that, if $f$ is $K$-geometric in $R^n$, then $f$ is also ($R$-)geometric in $R^n$. By Corollary \ref{cor:char-geometric}, $f$ is square-free in $K[\x]$ and, if $f=f_1\cdots f_r$ is the factorization in $K[\x]$, then each $f_i$ satisfies, for example, condition $(\mr{iv})$ of Proposition \ref{prop:hyper}, that is, $f_i$ changes sign in $R^n$ (which does not depend on $K$). Remarks \ref{irre}$(\mr{i})$ assures that $f=f_1\cdots f_r$ is also the factorization of $f$ in $R[\x]$, so Corollary \ref{cor:char-geometric} also implies that $f$ is ($R$-)geometric in $R^n$.
\end{proof}

The next result is another consequence of Propostion \ref{prop:hyper}.

\begin{prop}[Characterization of $K$-geometric hypersurfaces of $R^n$]\label{3111}
A $K$-algebraic set $X\subset R^n$ is a $K$-geometric hypersurface of $R^n$ if and only if the dimension of each $K$-irreducible component of $X$ is $n-1$.
\end{prop}
\begin{proof}
The `only if' implication follows immediately from Proposition \ref{prop:hyper}. Let us prove the `if' implication. Suppose $X_1,\ldots,X_r$ are the $K$-irreducible components of $X$ and $\dim(X_i)=n-1$ for each $i\in\{1,\ldots,r\}$. As $\II_K(X_i)$ is a nonzero prime ideal of $K[\x]$, there exists an irreducible polynomial $f_i\in\II_K(X_i)$, so $X_i\subset\ZZ_R(f_i)\subsetneqq R^n$. It follows that $n-1=\dim(X_i)\leq\dim(\ZZ_R(f_i))\leq n-1$. Thus, $\dim(\ZZ_R(f_i))=n-1$. By Proposition \ref{prop:hyper}, $\ZZ_R(f_i)$ is a $K$-irreducible $K$-geometric hypersurface of $R^n$. By Theorem \ref{dimension} and Lemma \ref{dimirred}, we have $\dim_K(X_i)=n-1=\dim_K(\ZZ_R(f_i))$ and $X_i=\ZZ_R(f_i)$. Define $f:=f_1\cdots f_r\in K[\x]$. Observe that $\ZZ_R(f)=X$ and the polynomials $f_i$ and $f_j$ are non-associated if $i\neq j$ because, otherwise, $X_i=\ZZ_R(f_i)=\ZZ_R(f_j)=X_j$, which is a contradiction. Thus, $f=f_1\cdots f_r$ is the factorization of $f$ in $K[\x]$. By implication $(\mr{v})\Longrightarrow(\mr{i})$ in Proposition \ref{prop:hyper}, $f$ is a $K$-geometric polynomial in $R^n$, so $X=\ZZ_R(f)$ is a $K$-geometric hypersurface of $R^n$, as required.
\end{proof}

\begin{remark}\label{krkr}
Let $g\in\kr[\x]$ be a $\kr$-geometric polynomial in $R^n$, let $g=g_1\cdots g_r$ be the factorization of $g$ in $\kr[\x]$ (recall that $g$ is square-free) and let $Z\subset C^n$ be the complexification of $\ZZ_R(g)\subset R^n$. \emph{Then $\II_C(Z)=(g)C[\x]$, $Z=\ZZ_C(g)$ and $g=g_1\cdots g_r$ is also the factorization of $g$ in $\kbar[\x]$ and in $C[\x]$. In particular, $g$ is square-free in $\kbar[\x]$ and in $C[\x]$ (see also Remark \ref{squaresquarefree}). Moreover, $g$ is irreducible in $\kr[\x]$ if and only if it is irreducible in $\kbar[\x]$ or, equivalently, if it is irreducible in $C[\x]$.}

Let us prove these assertions. By \eqref{equaz}, $\II_C(Z)=\II_\kr(\ZZ_R(g))C[\x]=(g)C[\x]$, so $Z=\ZZ_C(g)$. Let $i\in\{1,\ldots,r\}$. By Remarks \ref{irre}$(\mr{ii})$, we only need to show that $g_i$ is irreducible as a polynomial in $\kbar[\x]$. By Lemma \ref{univ} and Proposition \ref{prop:hyper}, $g_i\in\kr[\x]$ is geometric in $(\kr)^n$ and $\ZZ_\kr(g_i)\subset (\kr)^n$ is irreducible. Let $V_i\subset\kbar^n$ be the complexification of $\ZZ_\kr(g_i)\subset(\kr)^n$. By Proposition \ref{prop:zar}$(\mr{i})(\mr{iii})$, $V_i\subset\kbar^n$ is irreducible and the ideal $\II_\kbar(V_i)=\II_\kr(\ZZ_\kr(g_i))\kbar[\x]=(g_i)\kbar[\x]$ of $\kbar[\x]$ is prime. Thus, $g_i$ is irreducible in $\kbar[\x]$, as required. $\sqbullet$
\end{remark}

Let $g\in\kr[\x]$ be a $\kr$-geometric polynomial in $R^n$ and let $g=g_1\cdots g_r$ be its factorization in $\kr[\x]$. As $g\in\kbar[\x]$, we can speak about the Galois completion $g^\bullet$ of $g$ (see Definition~\ref{gbullet}). To compute $g^\bullet$, as $g=g_1\cdots g_r$ is also the factorization of $g$ in $\kbar[\x]$ by Remark \ref{krkr}, we can apply Algorithm \ref{gcpol} to $g$ choosing in step {\it(1)} a finite Galois extension $E|K$ such that $E$ contains all the coefficients of $g_1,\ldots,g_r$ (and hence of the coefficients of $g$). Now, an imme\-diate consequence of Theorem \ref{thm:gc}, Remarks \ref{welld-sfree}$\,$\&$\,$\ref{gcpolr}$\,$\&$\,$\ref{irre}$\,$\&$\,$\ref{krkr} and \eqref{22}\&\eqref{23}\&\eqref{22b} is the following.

\begin{cor}\label{3113}
Let $g\in\kr[\x]$ be a $\kr$-geometric polynomial in $R^n$ and let $g^\bullet\in K[\x]$ be the Galois completion of $g$. We have:
\begin{itemize}
\item[$(\mr{i})$] $\ZZ_C(g^\bullet)=\zcl_{C^n}^K(\ZZ_R(g))$ and $\ZZ_R(g^\bullet)=\zcl_{R^n}^K(\ZZ_R(g))$ are the complex and real Galois completions of $\ZZ_R(g)\subset R^n$. In addition, we have $\ZZ_C(g^\bullet)=\zcl_{C^n}^K(\ZZ_R(g^\bullet))$.
\item[$(\mr{ii})$] $\II_K(\ZZ_R(g))=\II_K(\ZZ_C(g^\bullet))=\II_K(\ZZ_R(g^\bullet))=(g^\bullet)K[\x]$. In particular, the latter equality assures that $g^\bullet$ is a $K$-geometric polynomial in $R^n$.
\item[$(\mr{ii}')$] $\II_C(\ZZ_C(g^\bullet))=(g^\bullet)C[\x]$ so $g^\bullet$ is square-free as a polynomial in $C[\x]$ (see also Remark \ref{squaresquarefree}).
\item[$(\mr{iii})$] $\zcl_{C^n}(\ZZ_R(g^\bullet))\subset \ZZ_C(g^\bullet)$ and $\II_R(\ZZ_C(g^\bullet))=\II_K(\ZZ_R(g^\bullet))R[\x]$.
\item[$(\mr{iv})$] If $g=g_1\cdots g_r$ is the factorization of $g$ in $\kr[\x]$ and $g^\bullet=f_1\cdots f_s$ is the factorization of $g^\bullet$ in $K[\x]$, then there exists a surjective map $\eta:\{1,\ldots,r\}\to\{1,\ldots,s\}$ such that $f_{\eta(j)}$ is the Galois completion of $g_j$ for each $j\in\{1,\ldots,r\}$.
\item[$(\mr{iv}')$] If $E|K$ is a finite Galois extension that contains all the coefficients of $g_1,\ldots,g_r$ and we set $G':=G(E:K)$ and $\textstyle g^*:=\prod_{\sigma\in G'}g^\sigma\in K[\x]$, then $ g^*=u\prod_{j=1}^rf_{\eta(j)}^{\ell_j}\,$ for some $u\in K^*$ and $\ell_j\in\N^*$.
\item[$(\mr{iv}'')$] If $g$ is irreducible in $\kr[\x]$ (so $r=1$), then $g^\bullet$ is irreducible in $K[\x]$ and $g^*=u(g^\bullet)^\ell$ for some $u\in K^*$ and $\ell\in\N^*$.
\item[$(\mr{v})$] Suppose again that $g$ is irreducible in $\kr[\x]$ (and so in $\kbar[\x]$), consider again a finite Galois extension $E|K$ that contains all the coefficients of $g$ and set $G':=G(E:K)$. Let $\mc{F}$ be a subset of $G'$ such that $\{g^\sigma:\sigma\in\mc{F}\}=\{g^\sigma:\sigma\in G'\}$ and $g^\sigma$ and $g^\tau$ are non-associated in $E[\x]$ for each $\sigma,\tau\in\mc{F}$ with $\sigma\neq\tau$. Then $g^\bullet=v\prod_{\sigma\in\mc{F}}g^\sigma$ is the factorization of $g^\bullet$ in $\kbar[\x]$ and in $C[\x]$, for some $v\in\kbar^*$.
\end{itemize}
\end{cor}

We refer the reader to Corollary \ref{3113b} for an additional property of $g^\bullet$.

Now, we compute explicitly the Galois completions of two $\qr$-geometric polynomials $g$ and $p$ in $R^3$ by means of Algorithm \ref{gcpol}.

\begin{examples}\label{exa:gc}
$(\mr{i})$ Let $\qr[\x]:=\qr[\x_1,\x_2,\x_3]$ and let $g:=\x_1+\sqrt{2}\x_2+\sqrt[4]{2}\x_3\in\qr[\x]$. By Proposition \ref{prop:hyper}, $g$ is $\qr$-geometric in $R^3$. Let us compute the Galois completion $g^\bullet$ of $g$. Let $N:=\Q(\sqrt{2},\sqrt[4]{2})=\Q(\sqrt[4]{2})$, let $E|\Q$ be the Galois closure of $N|\Q$ and let $G':=G(E:\Q)$. Observe that $\t^4-2$ is the minimal polynomial of $\sqrt[4]{2}$ over $\Q$ and $\{\sqrt[4]{2},\ii\sqrt[4]{2},-\sqrt[4]{2},-\ii\sqrt[4]{2}\}$ is the set of its roots, where $\ii:=\sqrt{-1}$. It follows that $E=\Q(\sqrt[4]{2},\ii\sqrt[4]{2})=\Q(\sqrt[4]{2},\ii)=N(\ii)$ and $[E:\Q]=[E:N][N:\Q]=2\cdot4=8$, so $G'$ is a group of order $8$ and in fact it is the dihedral group $D_4$. The elements of $G'$ are the automorphisms $\sigma_{ab}$ of $E$, where $a\in\{0,1,2,3\}$ and $b\in\{0,1\}$, such that $\sigma_{ab}(\sqrt[4]{2})=\ii^a\sqrt[4]{2}$ and $\sigma_{ab}(\ii)=(-1)^b\ii$. As a consequence, $\sigma_{ab}(\sqrt{2})=\sigma_{ab}((\sqrt[4]{2})^2)=(\sigma_{ab}(\sqrt[4]{2}))^2=(-1)^a\sqrt{2}$. We have:
\begin{itemize}
\item $g^{\sigma_{00}}=g^{\sigma_{01}}=g=\x_1+\sqrt{2}\x_2+\sqrt[4]{2}\x_3$,
\item $g^{\sigma_{10}}=g^{\sigma_{11}}=\x_1-\sqrt{2}\x_2+\ii\sqrt[4]{2}\x_3$,
\item $g^{\sigma_{20}}=g^{\sigma_{21}}=\x_1+\sqrt{2}\x_2-\sqrt[4]{2}\x_3$,
\item $g^{\sigma_{30}}=g^{\sigma_{31}}=\x_1-\sqrt{2}\x_2-\ii\sqrt[4]{2}\x_3$.
\end{itemize}

Define $g^*:=\prod_{\sigma\in G'}g^\sigma\in\Q[\x]$. It holds
\begin{align*}
g^*\textstyle=\prod_{b=0}^1\prod_{a=0}^3g^{\sigma_{ab}}=\big(\prod_{a=0}^3g^{\sigma_{a0}}\big)^2=(\x_1^4-4\x_1^2\x_2^2+8\x_1\x_2\x_3^2+4\x_2^4-2\x_3^4)^2.
\end{align*}

Let $h:=\x_1^4-4\x_1^2\x_2^2+8\x_1\x_2\x_3^2+4\x_2^4-2\x_3^4\in\Q[\x]$ and let $g^\bullet\in\Q[\x]$ be a Galois completion of $g$. By Corollary \ref{3113}$(\mr{iv}'')$, $g^\bullet$ is irreducible in $\Q[\x]$ and $h^2=g^*=u(g^\bullet)^\ell$ for some $u\in\Q^*$ and $\ell\in\N^*$, so $h=u'(g^\bullet)^m$ for some $u'\in\Q^*$ and $m\in\N^*$. As $m|\deg(h)=4$, we have that $m\in\{1,2,4\}$. Observe that $m$ cannot be an even number. Otherwise, $h$ would be a square in $\Q[\x]$ up to a nonzero multiplicative constant. Thus, the coefficients of $h$ corresponding to monomials $\x_1^4$ and $\x_3^4$ should have the same sign, which is false. This proves that $m=1$, so $h=u'g^\bullet$. Renaming $u'g^\bullet$ as $g^\bullet$, we can assume that $g^\bullet=h$, so
$$
g^\bullet=\x_1^4-4\x_1^2\x_2^2+8\x_1\x_2\x_3^2+4\x_2^4-2\x_3^4\in\Q[\x]
$$
is the Galois completion of $g$.

Define: $Y:=\ZZ_R(g)$, $X:=\ZZ_R(g^\bullet)$, $X^C:=\ZZ_C(g^\bullet)$ and $X_a:=\ZZ_R(g^{\sigma_{a0}})$ and $X_a^C:=\ZZ_C(g^{\sigma_{a0}})$ for each $a\in\{0,1,2,3\}$. By Corollary \ref{3113}$(\mr{i})(\mr{ii})$, $X$ is the real Galois completion of $Y\subset R^3$, $X^C$ is the complex Galois completion of $Y\subset R^3$, and $g^\bullet$ is a $\Q$-geometric polynomial in~$R^3$. By Proposition \ref{prop:hyper}, $X$ is a $\Q$-irreducible $\Q$-geometric hypersurface of $R^3$.

As $g^\bullet=\prod_{a=0}^3g^{\sigma_{a0}}$, $X^C$ is an algebraic hypersurface of $C^3$ whose irreducible components are $\{X_a^C\}_{a=0}^3$ and $X=\bigcup_{a=0}^3X_a$, where
\begin{itemize}
\item $X_0=X_0^C\cap R^3=\{x\in R^3:x_1+\sqrt{2}x_2+\sqrt[4]{2}x_3=0\}=Y$,
\item $X_1=X_1^C\cap R^3=\{x\in R^3:x_1-\sqrt{2}x_2+\ii\sqrt[4]{2}x_3=0\}=\{x\in R^3:x_1-\sqrt{2}x_2=0,x_3=0\}$,
\item $X_2=X_2^C\cap R^3=\{x\in R^3:x_1+\sqrt{2}x_2-\sqrt[4]{2}x_3=0\}$,
\item $X_3=X_3^C\cap R^3=\{x\in R^3:x_1-\sqrt{2}x_2-\ii\sqrt[4]{2}x_3=0\}=\{x\in R^3:x_1-\sqrt{2}x_2=0,x_3=0\}$.
\end{itemize}

It follows that $X=X_0\cup X_1\cup X_2$ and, as an algebraic subset of $R^3$, $X$ has three irreducible components: the two planes $X_0$ and $X_2$ of $R^3$, and the line $X_1=X_3$ of $R^3$. As $\dim(X_1)=1\neq2$, by Proposition \ref{3111}, $X$ is not a $R$-geometric hypersurface of $R^3$ and so $g^\bullet$ is not $R$-geometric in $R^3$. As $g^\bullet$ is $\Q$-geometric but not $R$-geometric in $R^3$, we have $\II_R(\ZZ_R(g^\bullet))\neq(g^\bullet)R[\x]=\II_\Q(\ZZ_R(g^\bullet))R[\x]$. Moreover, the complexification of $X$ is strictly contained in $X^C$, because $\zcl_{C^3}(X_1)\subsetneqq X_1^C$ and $\zcl_{C^3}(X_1)\subsetneqq X_3^C$, so $\zcl_{C^3}(X)=X_0^C\cup\zcl_{C^3}(X_1)\cup X_2^C\subsetneqq\bigcup_{a=0}^3X_a^C=X^C$.

By Remark \ref{rem317}, the $\Q$-bad set $B_\Q(X)$ of $X$ is equal to the line $X_1$ of $R^3$. See Figure \ref{im:poly2}.

%%%
\begin{figure}[!ht]
\begin{center}
\begin{tikzpicture}[scale=0.32]

\draw[fill=blue!60,opacity=0.30,draw] (-6,-2.5) -- (-8,-6.5) -- (-4.58,-5.1) -- (-3,-8) -- (9,-3) -- (6,2.5) -- (8,6.5) -- (4.58,5.1) -- (3,8) -- (-9,3) -- (-6,-2.5);

\draw[blue,line width=0.01mm,dashed] (-6,-2.5) -- (6,2.5) -- (8,6.5) -- (-4,1.5) -- (-6,-2.5);
\draw[blue,line width=0.01mm,dashed] (-6,-2.5) -- (6,2.5) -- (4,-1.5) -- (-8,-6.5) -- (-6,-2.5);
\draw[blue,line width=0.01mm,dashed] (-6,-2.5) -- (-3,-8) -- (9,-3) -- (6,2.5) -- (-6,-2.5);
\draw[blue,line width=0.01mm,dashed] (-6,-2.5) -- (-9,3) -- (3,8) -- (6,2.5) -- (-6,-2.5);

\draw[blue,line width=0.05mm] (6,2.5) -- (9,-3) -- (-3,-8) -- (-9,3) -- (3,8) -- (4.58,5.1);
\draw[blue,line width=0.05mm] (6,2.5) -- (8,6.5) -- (-4,1.5) -- (-6,-2.5) -- (0,0);
\draw[blue,line width=0.05mm] (-6,-2.5) -- (-8,-6.5) -- (-4.6,-5.085);
\draw[blue,line width=0.05mm] (-6,-2.5) -- (6,2.5);

\draw[blue,line width=0.35mm] (-4.15,4.98) -- (-5.5,6.6);
\draw[blue,line width=0.15mm,dashed] (0,0) -- (-4.2,5.04);
\draw[blue,line width=0.35mm] (5.5,-6.6) -- (0,0);

\draw[line width=0.25mm] (-9.5,0) -- (-7.35,0);
\draw[line width=0.25mm,dashed] (-7.1,0) -- (3.375,0);
\draw[->,line width=0.25mm] (0,0) -- (9,0);
\draw[line width=0.25mm] (0,-9) -- (0,-6.75);
\draw[line width=0.25mm,dashed] (0,-7.8) -- (0,3);
\draw[->,line width=0.25mm] (0,3.15) -- (0,9.6);
\draw[line width=0.25mm] (7,7) -- (5.45,5.45);
\draw[line width=0.25mm,dashed] (5.22,5.22) -- (0,0);
\draw[->,line width=0.25mm] (0,0) -- (-7,-7);

\draw (-7,-7.5) node{\small$\!\!\!\x_1$};
\draw (9.45,-0.75) node{\small$\;\,\x_2$};
\draw (0.75,9.6) node{\small$\;\;\,\x_3$};
\draw (3.9,8.7) node{\small\color{blue}$X_0$};
\draw (8.9,7.1) node{\small\color{blue}$X_2$};
\draw (6,-7.45) node{\small\color{blue}$X_1$};
\draw (8.9,7.1) node{\small\color{blue}$X_2$};
\draw (-9,5.2) node{\color{blue}$X$};
\end{tikzpicture}
\end{center}
\caption{The real Galois completion $X$ of $Y=X_0$ and its three $R$-irreducible components: the planes $X_0,X_2$ and the line $X_1$ of $R^3$. The set $X\subset R^3$ coincides with the $\Q$-Zariski closure of $Y$ and is $\Q$-irreducible.}
\label{im:poly2}
\end{figure}
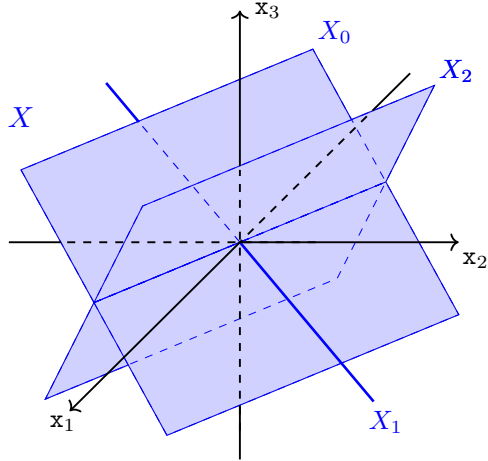

%%%

$(\mr{ii})$ Let $g:=\x_1+\sqrt{2}\x_2+\sqrt[4]{2}\x_3\in\qr[\x]:=\Q[\x_1,\x_2,\x_3]$ be the polynomial of the preceding item~$(\mr{i})$. Consider the polynomials $p:=\x_1+\sqrt{2}\x_2+\x_3\in\qr[\x]$ and $q:=gp$. By Proposition \ref{prop:hyper}, $q$ is $\qr$-geometric in $R^3$. Let us compute the Galois completion $q^\bullet$ of $q$. We keep the notations of item (i). For each $a\in\{0,1,2,3\}$ and for each $b\in\{0,1\}$, we have
$$
p^{\sigma_{ab}}=\x_1+(-1)^a\sqrt{2}\x_2+\x_3.
$$

Define the polynomials $p^\bullet:=p^{\sigma_{00}}p^{\sigma_{10}}\in\qr[\x]$ and $q^*:=\prod_{\sigma\in G'}q^\sigma\in\Q[\x]$. It holds
$$
p^\bullet=\x_1^2+2\x_1\x_3-2\x_2^2+\x_3^2\in\Q[\x],
$$
\begin{align*}
q^*&\textstyle=(\prod_{\sigma\in G'}g^\sigma)(\prod_{\sigma\in G'}p^\sigma)=(g^\bullet)^2(\prod_{a=0}^3p^{\sigma_{a0}})^2=(g^\bullet)^2(p^\bullet)^4\\
&=(\x_1^4-4\x_1^2\x_2^2+8\x_1\x_2\x_3^2+4\x_2^4-2\x_3^4)^2(\x_1^2+
2\x_1\x_3-2\x_2^2+\x_3^2)^4,
\end{align*}
where the latter product is the factorization of $q^*$ in $\Q[\x]$. By Algorithm \ref{gcpol}{\it(2)(3)}, the polynomial
$$
q^\bullet:=g^\bullet p^\bullet=(\x_1^4-4\x_1^2\x_2^2+8\x_1\x_2\x_3^2+4\x_2^4-2\x_3^4)(\x_1^2+2\x_1\x_3-2\x_2^2+\x_3^2)\in\Q[\x]
$$
is the Galois completion of $q$. Observe that $q^\bullet=g^\bullet p^\bullet$ is the factorization of $q^\bullet$ in $\Q[\x]$.

Define: $W:=\ZZ_R(p)$, $U:=\ZZ_R(p^\bullet)$, $U_a:=\ZZ_R(p^{\sigma_{a0}})$ and $U_a^C:=\ZZ_C(p^{\sigma_{a0}})$ for each $a\in\{0,1\}$, $V:=\ZZ_R(q)=Y\cup W$, $S:=\ZZ_R(q^\bullet)=X\cup U$ and $S^C:=\ZZ_C(q^\bullet)$. Observe that the line $X_1=X_3$ is contained in the plane $U_1=\{x\in R^3:x_1-\sqrt{2}x_2+x_3=0\}$ of $R^3$. We have:
\begin{itemize}
\item $U$ is the real Galois completion of $W\subset R^3$, is a $\Q$-geometric and $R$-geometric hypersurface of $R^3$ and is $\Q$-irreducible but $R$-reducible with the two planes $U_0=W$ and $U_1$ of $R^2$ as its $R$-irreducible components.
\item $S$ is the real Galois completion of $V\subset R^3$, which is a $\Q$-geometric hypersurface of $R^3$ having $X$ and $U$ as $\Q$-irreducible components. Moreover, $S=X_0\cup X_2\cup U_0\cup U_1$, so it is a $R$-geometric hypersurface of $R^3$ having the four distinct planes $X_0$, $X_2$, $U_0$ and $U_1$ as $R$-irreducible components. See Figure \ref{im:poly3}.

%%%
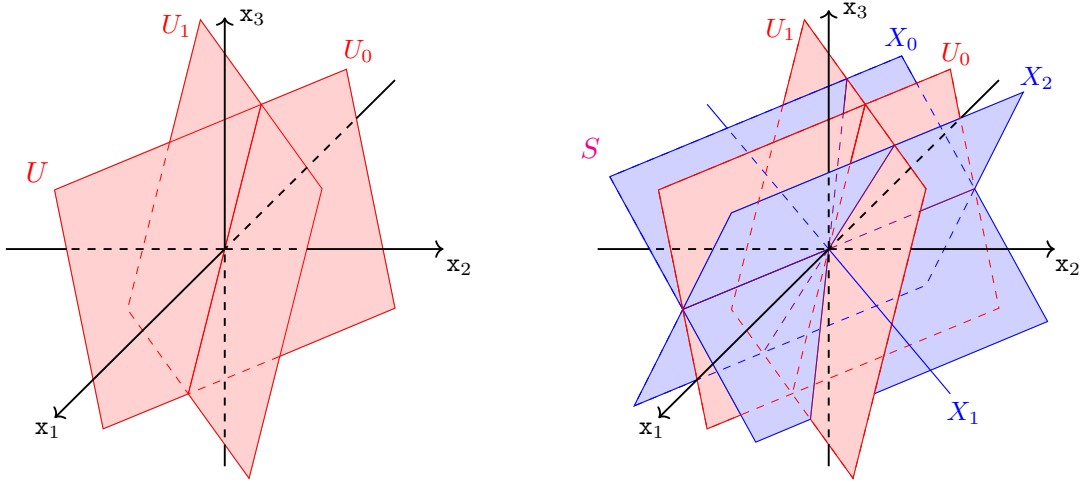
\begin{figure}[ht]
\begin{minipage}[b]{0.475\linewidth}
\centering
\begin{tikzpicture}[scale=0.32]

\draw[fill=red!60,opacity=0.30,draw] (-5,-7.45) -- (-1.5,-6) -- (1,-9.5) -- (2.25,-4.45) -- (7,-2.45) -- (5,7.45) -- (1.5,6) -- (-1,9.5) -- (-2.25,4.45) -- (-7,2.45) -- (-5,-7.45);

\draw[red,line width=0.05mm] (-7,2.45) -- (5,7.45) -- (7,-2.45) -- (2.25,-4.45);
\draw[red,line width=0.05mm,dashed] (2.25,-4.45) -- (-1.5,-6);
\draw[red,line width=0.05mm] (-1.5,-6) -- (-5,-7.45) -- (-7,2.45);
\draw[red,line width=0.05mm] (1.5,6) -- (-1.5,-6) -- (1,-9.5) -- (4,2.5) -- (1.5,6);
\draw[red,line width=0.05mm] (1.5,6) -- (-1,9.5) -- (-2.275,4.4);
\draw[red,dashed,line width=0.05mm] (-2.275,4.4) -- (-4,-2.5) -- (-1.5,-6);
\draw[red,line width=0.05mm] (1.5,6) -- (-1.5,-6);

\draw[line width=0.25mm,dashed] (0,0) -- (3.175,0);

\draw[line width=0.25mm] (-9,0) -- (-6.5,0);
\draw[line width=0.25mm,dashed] (-6.3,0) -- (0,0);
\draw[->,line width=0.25mm] (3.375,0) -- (9,0);
\draw[line width=0.25mm] (0,-9) -- (0,-8.1);
\draw[line width=0.25mm,dashed] (0,-7.8) -- (0,0);
\draw[->,line width=0.25mm] (0,0) -- (0,9.6);
\draw[line width=0.25mm] (7,7) -- (5.41,5.41);
\draw[line width=0.25mm,dashed] (5.25,5.25) -- (0,0);
\draw[->,line width=0.25mm] (0,0) -- (-7,-7);

\draw (-7,-7.5) node{\small$\!\!\!\x_1$};
\draw (9.25,-0.75) node{\small$\;\;\,\x_2$};
\draw (0.75,9.6) node{\small$\;\;\,\x_3$};
\draw (-1.8,9.25) node{\small\color{red}$\!\!U_1$};
\draw (5.5,8.2) node{\small\color{red}$U_0$};
\draw (-7.7,3.2) node{\color{red}$U$};

\end{tikzpicture}

\end{minipage}
\begin{minipage}[b]{0.475\linewidth}
\centering

\begin{tikzpicture}[scale=0.32]

\draw[fill=blue!60,opacity=0.30,draw] (-3,-8) -- (-0.75,-7.05) -- (0,0) -- (-6,-2.5) -- (-3,-8);
\draw[fill=blue!60,opacity=0.30,draw] (0,0) -- (-6,-2.5) -- (-4,1.5) -- (2.7,4.30) -- (0,0);
\draw[fill=blue!60,opacity=0.30,draw] (6,2.5) -- (9,-3) -- (1.86,-6) -- (3.75,1.58) -- (6,2.5);
\draw[fill=blue!60,opacity=0.30,draw] (6,2.5) -- (3.75,1.58) -- (4,2.5) -- (2.7,4.30) -- (8,6.5) -- (6,2.5);
\draw[fill=blue!60,opacity=0.30,draw] (-6,-2.5) -- (-8,-6.5) -- (-5.415,-5.415) -- (-6,-2.5);
\draw[fill=blue!60,opacity=0.30,draw] (0.75,7.05) -- (-9,3) -- (-7.4,0) -- (-6,-2.5) -- (-7,2.45) -- (0.6,5.6) -- (0.75,7.05);
\draw[fill=blue!60,opacity=0.30,draw] (0.75,7.05) -- (1.5,6) -- (3.62,6.875) -- (3,8) -- (0.75,7.05);

\draw[fill=red!60,opacity=0.30,draw] (-0.75,-7.05) -- (0,0) -- (2.7,4.30) -- (4,2.5) -- (1,-9.5) -- (-0.75,-7.05);
\draw[fill=red!60,opacity=0.30,draw] (-6,-2.5) -- (-5,-7.45) -- (-3.61,-6.875) -- (-6,-2.5);
\draw[fill=red!60,opacity=0.30,draw] (-6,-2.5) -- (-7,2.45) -- (1.5,6) -- (0.9,3.54) -- (-4,1.5) -- (-6,-2.5);
\draw[fill=red!60,opacity=0.30,draw] (1.5,6) -- (0.75,7.05) -- (0.6,5.6) -- (1.5,6);
\draw[fill=red!60,opacity=0.30,draw] (0.75,7.05) -- (-1,9.5) -- (-1.875,5.975) -- (0.75,7.05);
\draw[fill=red!60,opacity=0.30,draw] (1.5,6) -- (0.9,3.54) -- (2.7,4.30) -- (1.5,6);
\draw[fill=red!60,opacity=0.30,draw] (1.5,6) -- (2.7,4.30) -- (5.4,5.4) -- (5,7.45) -- (1.5,6);

\draw[blue,line width=0.01mm,dashed] (6,2.5) -- (4,-1.5) -- (-8,-6.5) -- (-6,-2.5);
\draw[blue,line width=0.01mm,dashed] (-6,-2.5) -- (-9,3) -- (3,8) -- (6,2.5);

\draw[blue,line width=0.05mm] (6,2.5) -- (9,-3) -- (1.86,-6);
\draw[blue,line width=0.05mm] (-0.75,-7.05) -- (-3,-8) -- (-9,3) -- (3,8) -- (3.65,6.81);
\draw[blue,line width=0.05mm] (6,2.5) -- (8,6.5) -- (-4,1.5) -- (-6,-2.5) -- (0,0);
\draw[blue,line width=0.05mm] (-6,-2.5) -- (-8,-6.5) -- (-5.415,-5.415);
\draw[violet,line width=0.07mm,dashed] (4,1.666) -- (0,0);

\draw[red,line width=0.01mm,dashed] (-6,-2.5) -- (-7,2.45) -- (5,7.45) -- (6,2.5);
\draw[red,line width=0.05mm,dashed] (-6,-2.5) -- (-5,-7.45) -- (7,-2.45) -- (6,2.5);
\draw[red,line width=0.01mm,dashed] (1.5,6) -- (-1.5,-6) -- (1,-9.5) -- (4,2.5) -- (1.5,6);
\draw[red,line width=0.01mm,dashed] (1.5,6) -- (-1,9.5) -- (-4,-2.5) -- (-1.5,-6);

\draw[red,line width=0.05mm] (0.887,3.55) -- (1.5,6) -- (4,2.5) -- (1,-9.5) -- (-0.75,-7.05);
\draw[red,line width=0.05mm] (5.4,5.4) -- (5,7.45) -- (-7,2.45) -- (-5,-7.45) -- (-3.6,-6.87);
\draw[red,line width=0.05mm] (1.5,6) -- (-1,9.5) -- (-1.875,6);
\draw[red,line width=0.05mm] (0.887,3.55) -- (1.5,6);

\draw[blue,line width=0.15mm] (-4.17,5.002) -- (-5,6);
\draw[blue,line width=0.10mm,dashed] (0,0) -- (-4.2,5.04);
\draw[blue,line width=0.15mm] (5,-6) -- (0,0);

\draw[violet,line width=0.05mm] (2.7,4.30) -- (0,0);
\draw[violet,line width=0.01mm,dashed] (0,0) -- (-2.7,-4.30);
\draw[violet,line width=0.01mm,dashed] (0.75,7.05) -- (0,0);
\draw[violet,line width=0.05mm] (0,0) -- (-0.75,-7.05);
\draw[violet,line width=0.05mm] (0.6,5.6) -- (0.75,7.05);
\draw[violet,line width=0.05mm] (-6,-2.5) -- (0,0);
\draw[violet,line width=0.05mm] (4,1.666) -- (6,2.5);

\draw[line width=0.25mm] (-9.5,0) -- (-7.38,0);
\draw[line width=0.25mm,dashed] (-7.1,0) -- (0,0);
\draw[line width=0.25mm,dashed] (0,0) -- (3.175,0);
\draw[->,line width=0.25mm] (3.375,0) -- (9.3,0);
\draw[line width=0.25mm] (0,-9) -- (0,-8.08);
\draw[line width=0.25mm,dashed] (0,-7.8) -- (0,0);
\draw[line width=0.25mm,dashed] (0,0) -- (0,3.175);
\draw[->,line width=0.25mm] (0,3.175) -- (0,9.9);
\draw[line width=0.25mm] (7,7) -- (5.45,5.45);
\draw[line width=0.25mm,dashed] (5.25,5.25) -- (0,0);
\draw[->,line width=0.25mm] (0,0) -- (-7,-7);

\draw (-7,-7.5) node{\small$\!\!\!\x_1$};
\draw (9.45,-0.75) node{\small$\;\;\,\x_2$};
\draw (0.5,9.9) node{\small$\;\;\;\;\x_3$};
\draw (3,8.7) node{\small\color{blue}$X_0$};
\draw (5.55,-6.8) node{\small\color{blue}$X_1$};
\draw (8.25,7.1) node{\small\color{blue}$\;\;X_2$};
\draw (-1.7,9.1) node{\small\color{red}$\!\!\!U_1$};
\draw (5.25,8.1) node{\small\color{red}$U_0$};
\draw (-9.8,4.2) node{\color{magenta}$S$};

\end{tikzpicture}
\end{minipage}
\caption{The real Galois completion $U$ of $W=U_0$ on the left and the real Galois completion $S$ of $V=X_0\cup U_0$ on the right.}
\label{im:poly3}
\end{figure}

\item By Corollary \ref{3113}$(\mr{ii})$, $q^\bullet$ is $\Q$-geometric in $R^3$. Define the polynomials $P,Q\in R[\x]$ as
$$
P:=g^{\sigma_{10}}g^{\sigma_{30}}=(\x_1-\sqrt{2}\x_2)^2+\sqrt{2}x_3^2
$$
and
$$
Q:=g^{\sigma_{00}}g^{\sigma_{20}}p^{\sigma_{00}}p^{\sigma_{10}}=(\x_1^2+2\sqrt{2}\x_1\x_2+2\x_2^2-\sqrt{2}\x_3^2)(\x_1^2+2\x_1\x_3-2\x_2^2+\x_3^2).
$$
Observe that $q^\bullet=Pg^{\sigma_{00}}g^{\sigma_{20}}p^{\sigma_{00}}p^{\sigma_{10}}$ and $Q=g^{\sigma_{00}}g^{\sigma_{20}}p^{\sigma_{00}}p^{\sigma_{10}}$ are the factorizations of $q^\bullet$ and $Q$ in $R[\x]$, respectively. Moreover, $\ZZ_R(P)=X_1\subsetneqq U_1=\ZZ_R(p^{\sigma_{10}})$. As $\dim(X_1)=1\neq2$, by Proposition \ref{prop:hyper}, $q^\bullet$ is not $R$-geometric in $R^3$. However, $Q$ is $R$-geometric in $R^3$ and $\II_R(S)=(Q)R[\x]$, so $S=\ZZ_R(Q)$ is a $R$-geometric hypersurface of $R^3$.
\item $S^C$ is the complex Galois completion of $V\subset R^3$ and it is an algebraic hypersurface of $C^3$ having the six complex planes $X_0^C$, $X_1^C$, $X_2^C$, $X_3^C$, $U_0^C$ and $U_1^C$ of $C^3$ as irreducible components.
\end{itemize}

The complexification of $S$ is strictly contained in $S^C$, because $\zcl_{C^3}(S)=X_0^C\cup X_2^C\cup U_0^C\cup U_1^C\subsetneqq S^C$. Moreover, $S$ contains an embedded component of the real part of $S^C$ in the sense that the intersections of $R^3$ with the irreducible components $X_1^C$ and $X_3^C$ of $S^C$ are both equal to the line $X_1=X_3$, which is strictly contained in the irreducible component $U_1$ of $S$.

As $\ZZ_R(p^\bullet)$ has no $\Q$-bad points, we have $B_\Q(S)=B_\Q(X)=X_1$. $\sqbullet$
\end{examples}

The next result concerns the complex Galois completion and the $K$-bad set of a $K$-irreducible $K$-geometric hypersurface.

\begin{prop}[$K$-irreducible $K$-geometric hypersurfaces]\label{irre-sf}
Let $f\in K[\x]$ be a polynomial that is irreducible in $K[\x]$ and $K$-geometric in $R^n$, and let $f=f_1\cdots f_r$ be the factorization of $f$ in $\kr[\x]$ (which is square-free by Remark \ref{squaresquarefree}). We have:
\begin{itemize}
\item[$(\mr{i})$] The set $I$ of all indices $i\in\{1,\ldots,r\}$ such that $f_i$ is $\kr$-geometric in $R^n$ is non-empty. For each $i\in I$, the Galois completion of $f_i$ coincides with $f$ (up to a multiplicative constant in $K^*$). 
\item[$(\mr{ii})$] $\ZZ_C(f)=\zcl^K_{C^n}(\ZZ_R(f))$.
\item[$(\mr{iii})$] If $J$ is the set of all indices $j\in\{1,\ldots,r\}$ such that $\dim(Z_R(f_j))<n-1$, then $J=\{1,\ldots,r\}\setminus I$ and $B_K(\ZZ_R(f))=\bigcup_{j\in J}Z_R(f_j)$.
\end{itemize}
\end{prop}
\begin{proof}
By Proposition \ref{prop:hyper}, $\dim(\ZZ_R(f))=n-1$ and $I$ coincides with the set of all indices $i\in\{1,\ldots,r\}$ such that $\dim(\ZZ_R(f_i))=n-1$ (so $J=\{1,\ldots,r\}\setminus I$). As $\ZZ_R(f)=\bigcup_{i=1}^r\ZZ_R(f_i)$, it follows that $I\neq\varnothing$. Let $i\in I$ and let $f_i^\bullet\in K[\x]$ be the Galois completion of $f_i$. By Corollary \ref{3113}$(\mr{ii})$, $f\in\II_K(\ZZ_R(f_i))=(f_i^\bullet)K[\x]$ so $f_i^\bullet$ divides $f$ in $K[\x]$. As $f$ is irreducible in $K[\x]$, we deduce $f=f_i^\bullet$ (up to a multiplicative constant in $K^*$). By Corollary \ref{3113}$(\mr{i})$, we have $\ZZ_C(f)=\ZZ_C(f_i^\bullet)=\zcl^K_{C^n}(\ZZ_R(f_i^\bullet))=\zcl^K_{C^n}(\ZZ_R(f))$. This proves $(\mr{i})$ and $(\mr{ii})$.

Let us prove $(\mr{iii})$. By Remark \ref{irre}$(\mr{i})$, $f=f_1\cdots f_r$ is the factorization of $f$ in $R[\x]$. By Remark \ref{krkr}, $f_i$ is irreducible in $C[\x]$ for each $i\in I$. Rearranging the indices, if necessary, we can assume that there exists $s,t\in\{1,\ldots,r\}$ such that $t\leq s$, $I=\{1,\ldots,t\}$, the polynomials $f_1,\ldots,f_s$ are irreducible in $C[\x]$ and $f_{s+1},\ldots,f_r$ are reducible in $C[\x]$, where the latter polynomials are omitted if $s=r$. Thus, $J=\{t+1,\ldots,r\}$ if $t<r$ and $J=\varnothing$ if $t=s=r$. As the Galois group $G(C,R)\cong\Z_2$ is generated by the complex conjugation $\alpha+\ii\beta\mapsto\ol{\alpha+\ii\beta}:=\alpha-\ii\beta$ for $\alpha,\beta\in R$, if $j>s$, then the irreducible polynomial $f_j\in R[\x]$ factorizes in $C[\x]$ as follows: $f_j=g_j\ol{g_j}$, where $g_j=\sum_\nu a_\nu\x^\nu$ is an irreducible polynomial in $C[\x]$, $\ol{g_j}:=\sum_\nu\ol{a_\nu}\x^\nu$ and $g_j$ and $\ol{g_j}$ are non-associated. Thus, $f=f_1\cdots f_sg_{s+1}\ol{g_{s+1}}\cdots g_r\ol{g_r}$ is the factorization of $f$ in $C[\x]$ and $\ZZ_C(f_1),\ldots,\ZZ_C(f_s)$, $\ZZ_C(g_{s+1})$, $\ZZ_C(\ol{g_{s+1}}),\cdots,\ZZ_C(g_r)$, $\ZZ_C(\ol{g_r})$ are the $C$-irreducible components of $\ZZ_C(f)\subset C^n$. 

Suppose $s<r$, pick $j\in\{s+1,\ldots,r\}$ and write $g_j=a_j+\ii b_j$ with $a_j,b_j\in R[\x]$. As $f_j=g_j\ol{g_j}=a_j^2+b_j^2$, we have $\ZZ_R(f_j)=\ZZ_R(a_j,b_j)=\ZZ_C(g_j)\cap R^n=\ZZ_C(\ol{g_j})\cap R^n=\ZZ_C(g_j,\ol{g_j})\cap R^n$, which has dimension $\leq n-2$, because $g_j,\ol{g_j}$ are irreducible and they are non-associated. Then, by Remark \ref{rem317}, it holds:
\begin{align*}
B_K(\ZZ_R(f))&\textstyle=\bigcup_{j=t+1}^s(\ZZ_C(f_i)\cap R^n)\cup\bigcup_{j=s+1}^r(\ZZ_C(g_j)\cap R^n)\cup\bigcup_{j=s+1}^r(\ZZ_C(\overline{g_j})\cap R^n)\\
&\textstyle=\bigcup_{j=t+1}^s\ZZ_R(f_j)\cup\bigcup_{j=s+1}^r\ZZ_R(f_j)=\bigcup_{j\in J}\ZZ_R(f_j).
\end{align*}
This proves $(\mr{iii})$ and completes the proof.
\end{proof}

We will explicitly apply Proposition \ref{irre-sf} in Examples \ref{432}$(\mr{i})$.

\subsection{Zero ideals of real $K$-algebraic sets.}
Equation \eqref{equaz}, Theorem \ref{thm:gc} and Lemma \ref{lem:gp} allow to compute the zero ideal $\II_K(X)$ of a $K$-algebraic set $X\subset R^n$.

\begin{alg}\label{gc2}
The algorithm works as follows:
\begin{itemize}
\item[(0)] Start with a $K$-algebraic set $X\subset R^n$.
\item[(1)] Decompose $X\subset R^n$ as the union of its $K$-irreducible components $X_1,\ldots,X_s$. Choose a $\kr$-irreducible component $Y_i$ of $X_i$ of dimension $\dim(X_i)$ and a finite system of generators $\{g_{i1},\ldots,g_{ir_i}\}$ of $\II_\kr(Y_i)$ in $\kr[\x]$ for each $i\in\{1,\ldots,s\}$.
\item[(2)] Choose a finite Galois extension $E|K$ that contains all the coefficients of polynomials $g_{i1},\ldots,g_{ir_i}$ for each $i\in\{1,\ldots,s\}$. Set $G':=G(E:K)$.
\item[(3)] For each $i\in\{1,\ldots,s\}$, compute the sets ${\mathfrak H}_i\subset E[\x]$ and afterwards ${\mathfrak G}_i\subset K[\x]$ as in Theorem \emph{\ref{thm:gc}$(\mr{iii})$} using $G'$ and the polynomials $g_{i1},\ldots,g_{ir_i}$.
\item[(4)] $\II_K(X_i)=\sqrt{{\mathfrak G}_iK[\x]}$ for each $i\in\{1,\ldots,s\}$, and $\II_K(X)=\bigcap_{i=1}^s\II_K(X_i)$.
\end{itemize}
\end{alg}

\begin{remark}[{$K$-geometric hypersurfaces}]
Assume now that the $K$-algebraic set $X\subset R^n$ we are considering is a $K$-geometric hypersurface of $R^n$ or, equivalently, $\II_K(X)=(g^\bullet)R[\x]$ for some $g^\bullet\in K[\x]$. By Proposition \ref{3111}, this is equivalent to say that each $K$-irreducible component $X_i$ of $X$ has dimension $n-1$. In this situation, to compute $\II_K(X)$, we can also proceed as follows. For each $i\in\{1,\ldots,s\}$, choose a $\kr$-irreducible (or equivalently $R$-irreducible) component $Y_i$ of $X_i\subset R^n$ of dimension $n-1$, and an irreducible polynomial $g_i\in\kr[\x]$ such that $\II_R(Y_i)=(g_i)R[\x]$ (hence $Y_i=\ZZ_R(g_i)$). Then, compute the Galois completion $g_i^\bullet\in K[\x]$ of $g_i$ via Algorithm \ref{gcpol}. By Lemma \ref{lem:gp} and Corollary \ref{3113}$(\mr{i})(\mr{ii})(\mr{iv}'')$, we have $X_i=\zcl^K_{R^n}(Y_i)$, $\II_K(X_i)=(g_i^\bullet)K[\x]$ and $g_i^\bullet$ is irreducible in $K[\x]$ for each $i\in\{1,\ldots,s\}$. Thus, $g_i^\bullet$ and $g_j^\bullet$ are non-associated for all $i,j\in\{1,\ldots,s\}$ with $i\neq j$, and $\II_K(X)=\bigcap_{i=1}^s\II_K(X_i)=(g^\bullet)K[\x]$, where $g^\bullet:=g_1^\bullet\cdots g_s^\bullet\in K[\x]$. Observe that $\II_K(\ZZ_R(g_1\cdots g_s))=(g^\bullet)K[\x]$, because $\ZZ_R(g_1\cdots g_s)=\bigcup_{i=1}^sY_i$ and $X=\zcl^K_{R^n}(\bigcup_{i=1}^sY_i)$. It follows that: $g^\bullet$ is the Galois completion of the product polynomial $g_1\cdots g_s\in\kr[\x]$ and $g^\bullet=g_1^\bullet\cdots g_s^\bullet$ is the factorization of $g^\bullet$ in $K[\x]$. $\sqbullet$
\end{remark}

\subsection{$K$-reliability and real algebraic sets defined over $K$}\label{rasdok}
{\it Recall that $R$ is a real closed field, $C:=R[\ii]$ is its algebraic closure, and $K$ is an arbitrary ordered subfield of $R$, endowed with the ordering induced by that of~$R$. Denote $\kbar\subset C$ the algebraic closure of $K$, and $\kr\subset R$ the real closure of $K$, so $\kbar=\kr[\ii]$ and $R\cap\kbar=\kr$.

Although the ordered subfield $K$ of $R$ is arbitrarily chosen, here we are mainly interested in the case where $K$ is not real closed, that is, $K\neq\kr$, so the `extension of coefficients' procedure is not available. The main example to keep in mind is $K=\Q$. Observe that the field $\Q$ has a unique ordering, the usual one. 

Given any algebraic set $S\subset R^n$, we set $\dim(S):=\dim_R(S)$ (see Remark $\ref{dime}$)}.

By Corollaries \ref{kreliablec} and \ref{inter}$(\mr{ii})$, if either $S$ is a $K$-algebraic subset of $C^n$ or $X$ is a $K$-algebraic subset of $R^n$ and $K\subset R$ is real closed, then $S\subset C^n$ and $X\subset R^n$ are `defined over $K$' in the sense that $\II_C(S)=\II_K(S)C[\x]$ and $\II_R(X)=\II_K(X)R[\x]$. If $K\subset R$ is not real closed, the situation is not so favorable, as the following simple example shows: if we set $K:=\Q$ and $X:=\ZZ_R(\x_1^3-2)=\{\sqrt[3]{2}\}\subset R$, then $\II_K(X)=(\x_1^3-2)K[\x_1]$ and $\II_R(X)=(\x_1-\sqrt[3]{2})R[\x_1]$, so $\II_R(X)\neq\II_K(X)R[\x_1]$.

In this subsection, we will study the $K$-algebraic subsets $X$ of $R^n$ that satisfy the equality $\II_R(X)=\II_K(X)R[\x]$. We begin introducing the concept of reliable family of polynomials in $K[\x]$. First, we need a preliminary notion.

\begin{defn}
If $F$ is an ordered field, we say that \emph{$F$ contains $K$} if $F|K$ is an extension of fields and the ordering of $F$ extends the one of $K$. $\sqbullet$
\end{defn}

By assumption, the real closed field $R$ and its ordered subfield $\kr$ both contain $K$. Moreover, if an ordered field $F$ contains $K$, then $F$ also contains $\kr$. 

\begin{defn}
Let $\{f_1,\ldots,f_r\}$ be a (finite) subset of $K[\x]$. The family $\{f_1,\ldots,f_r\}$ of polynomials in $K[\x]$ is \emph{$K$-reliable} if, for each real closed field $F$ that contains $K$, we have
$$
\II_F(\ZZ_F(f_1,\ldots,f_r))=(f_1,\ldots,f_r)F[\x].
$$
In case $K=\Q$, the $\Q$-reliable family $\{f_1,\ldots,f_r\}$ will be called only \emph{reliable}.

If $f\in K[\x]$, we say that $f$ is \emph{$K$-reliable} if so is the corresponding singleton $\{f\}$. We say that $f$ is \emph{reliable} if it is $\Q$-reliable. $\sqbullet$
\end{defn}

Observe that a polynomial $f\in K[\x]$ is $K$-reliable if and only if it is geometric in $F^n$ for each real closed field $F$ that contains $K$ (see Definition \ref{321}). Moreover, as each real closed field contains $\Q$, a family $\{f_1,\ldots,f_r\}\subset\Q[\x]$ is reliable if and only if it is geometric in $F^n$ for each real closed field $F$. Thus, if the family $\{f_1,\ldots,f_r\}\subset\Q[\x]$ is reliable, the same family is $K$-reliable for every ordered field $K$ (that always contains $\Q$).

It is worth noting that, combining Lemma \ref{univ} and Corollary \ref{corro}, we immediately deduce:

\begin{cor}\label{cor243}
A polynomial in $K[\x]$ is $K$-reliable if and only if it is geometric in $(\kr)^n$ or, equivalently, if it is geometric in $F^n$ for at least one real closed field $F$ that contains $K$.
\end{cor}

If in the previous corollary we choose $F=R$, we obtain:
\begin{cor}\label{cor243'}
A polynomial $f\in K[\x]$ is $K$-reliable if and only if it is geometric in $R^n$, that is, $\II_R(\ZZ_R(f))=(f)R[\x]$.
\end{cor}

The next result generalizes the latter corollary to all finite family of polynomials in $K[\x]$. 

\begin{prop}\label{clue}
Let $\{f_1,\ldots,f_r\}$ be a subset of $K[\x]$. The following conditions are equivalent:
\begin{itemize}
\item[$(\mr{i})$] The family $\{f_1,\ldots,f_r\}$ is $K$-reliable.
\item[$(\mr{i}')$] There exists a real closed field $F$ that contains $K$ such that 
$$
\II_F(\ZZ_F(f_1,\ldots,f_r))=(f_1,\ldots,f_r)F[\x].
$$
\item[$(\mr{ii})$] For each real closed field $F$ that contains $K$, the ideal $(f_1,\ldots,f_r)F[\x]$ of $F[\x]$ is real or, equivalently, $(f_1,\ldots,f_r)F[\x]=\II_F(X)$ for some algebraic set $X\subset F^n$.
\item[$(\mr{ii}')$] There exists a real closed field $F$ that contains $K$ such that the ideal $(f_1,\ldots,f_r)F[\x]$ of $F[\x]$ is real or, equivalently, $(f_1,\ldots,f_r)F[\x]=\II_F(X)$ for some algebraic set $X\subset F^n$.
\end{itemize}
\end{prop}
\begin{proof}
Equivalences $(\mr{i})\Longleftrightarrow(\mr{ii})$ and $(\mr{i}')\Longleftrightarrow(\mr{ii}')$ are straightforward but we include next a brief explanation. The equivalence of the conditions `the ideal $(f_1,\ldots,f_r)F[\x]$ of $F[\x]$ is real' and `$(f_1,\ldots,f_r)F[\x]=\II_F(X)$ for some algebraic set $X\subset F^n$' is an immediate consequence of Real Nullstellensatz \cite[Thm.4.1.4]{bcr}. Now, if $\II_F(\ZZ_F(f_1,\ldots,f_r))=(f_1,\ldots,f_r)F[\x]$, define $X:=\ZZ_F(f_1,\ldots,f_r)$. If $\II_F(X)=(f_1,\ldots,f_r)F[\x]$ with $X\subset F^n$ algebraic, it is enough to observe that $X=\ZZ_F(f_1,\ldots,f_r)$.

Implication $(\mr{i})\Longrightarrow(\mr{i}')$ is evident. Let us prove the converse implication $(\mr{i}')\Longrightarrow(\mr{i})$.

Let $F$ be a real closed field that contains $K$ such that $\II_F(\ZZ_F(f_1,\ldots,f_r))=(f_1,\ldots,f_r)F[\x]$. Observe that $F$ also contains $\kr$. Let us show that
\begin{equation}\label{equa}
\II_\kr(\ZZ_\kr(f_1,\ldots,f_r))=(f_1,\ldots,f_r)\kr[\x].
\end{equation}
We only have to check the inclusion `$\subset$', as the converse inclusion always holds. Let $f\in \II_\kr(\ZZ_{\kr}(f_1,\ldots,f_r))$, that is, $\ZZ_\kr(f_1,\ldots,f_r)\subset \ZZ_\kr(f)$. We extend coefficients to $F$ and obtain
$$
\ZZ_F(f_1,\ldots,f_r)=(\ZZ_{\kr}(f_1,\ldots,f_r))_F\subset(\ZZ_{\kr}(f))_F=\ZZ_F(f),
$$
so $f\in \II_F(\ZZ_F(f_1,\ldots,f_r))=(f_1,\ldots,f_r)F[\x]$. Thus, by Corollary~\ref{k},
$$
f\in(f_1,\ldots,f_r)F[\x]\cap\kr[\x]
=(f_1,\ldots,f_r)\kr[\x].
$$
We conclude $\II_{\kr}(\ZZ_{\kr}(f_1,\ldots,f_r))\subset(f_1,\ldots,f_r)\kr[\x]$, so \eqref{equa} is proved. 

By \eqref{equa} and Proposition \ref{prop:zar}$(\mr{i})$, we have
\begin{equation}\label{equa2}
\zcl_{\kbarn}(\ZZ_\kr(f_1,\ldots,f_r))=\ZZ_\kbar(f_1,\ldots,f_r).
\end{equation}

Let now $F_1$ be a real closed field that contains $K$. Let us check: $\II_{F_1}(\ZZ_{F_1}(f_1,\ldots,f_r))=(f_1,\ldots,f_r)F_1[\x]$. 

Let $C_1:=F_1[\ii]$ be the algebraic closure of $F_1$. Observe that $\kbar\subset C_1$. By \eqref{equa2}, Corollary \ref{LK-zar} and Proposition \ref{extension-zar}$(\mr{i})$, we have
\begin{align*}
\zcl_{C_1^n}(\ZZ_\kbar(f_1,\ldots,f_r))&=\zcl_{C_1^n}(\zcl_{\kbarn}(\ZZ_\kr(f_1,\ldots,f_r)))=\zcl_{C_1^n}(\ZZ_\kr(f_1,\ldots,f_r)),\\
\zcl_{C_1^n}(\ZZ_\kbar(f_1,\ldots,f_r))&=(\ZZ_\kbar(f_1,\ldots,f_r))_{C_1}=\ZZ_{C_1}(f_1,\ldots,f_r).
\end{align*}
Consequently,
\begin{equation}\label{equa4}
\zcl_{C_1^n}(\ZZ_\kbar(f_1,\ldots,f_r))=\zcl_{C_1^n}(\ZZ_\kr(f_1,\ldots,f_r))=\ZZ_{C_1}(f_1,\ldots,f_r).
\end{equation}

Let us prove that
\begin{equation}\label{equa3}
\II_{C_1}(\ZZ_{\kbar}(f_1,\ldots,f_r))=(f_1,\ldots,f_r)C_1[\x].
\end{equation}
By \eqref{equa}, $(f_1,\ldots,f_r)\kr[\x]$ is a radical ideal of $\kr[\x]$, so \eqref{equa4} and Corollary \ref{kreliablec} assure that
$$
\II_{C_1}(\ZZ_\kbar(f_1,\ldots,f_r))=\II_{C_1}(\ZZ_{C_1}(f_1,\ldots,f_r))=((f_1,\ldots,f_r)\kr[\x])C_1[\x]=(f_1,\ldots,f_r)C_1[\x]
$$
and \eqref{equa3} is proved.

Let $f\in \II_{F_1}(\ZZ_{F_1}(f_1,\ldots,f_r))$. Then
$$
\ZZ_\kr(f_1,\ldots,f_r)\subset(\ZZ_\kr(f_1,\ldots,f_r))_{F_1}=\ZZ_{F_1}(f_1,\ldots,f_r)\subset \ZZ_{F_1}(f)\subset \ZZ_{C_1}(f).
$$
By \eqref{equa4}, we deduce $\ZZ_{C_1}(f_1,\ldots,f_r)\subset \ZZ_{C_1}(f)$, that is, $f\in\II_{C_1}(\ZZ_{C_1}(f_1,\ldots,f_r))=(f_1,\ldots,f_r)C_1[\x]$ by \eqref{equa3}. By Corollary \ref{k}, we deduce
$$
f\in(f_1,\ldots,f_r)C_1[\x]\cap F_1[\x]=(f_1,\ldots,f_r)F_1[\x],
$$
so $\II_{F_1}(\ZZ_{F_1}(f_1,\ldots,f_r))=(f_1,\ldots,f_r)F_1[\x]$.

Consequently, the family $\{f_1,\ldots,f_r\}$ is $K$-reliable, as required.
\end{proof}

An immediate consequence of Proposition \ref{clue} when $K=\Q$ is the following:

\begin{cor}
Let $\{f_1,\ldots,f_r\}$ be a subset of $\Q[\x]$. The following conditions are equivalent.
\begin{itemize}
\item[$(\mr{i})$] The family $\{f_1,\ldots,f_r\}$ is reliable.
\item[$(\mr{ii})$] $\II_\qr(\ZZ_\qr(f_1,\ldots,f_r))=(f_1,\ldots,f_r)\qr[\x]$.
\item[$(\mr{iii})$] $\II_\R(\ZZ_\R(f_1,\ldots,f_r))=(f_1,\ldots,f_r)\R[\x]$.
\end{itemize}
\end{cor}

Let us now look at the relationship between the notions of $K$-geometric and $K$-reliable polynomials. 

\begin{lem}\label{no-news}
$(\mr{i})$ If a polynomial in $K[\x]$ is $K$-reliable, then it is also $K$-geometric in~$R^n$.

$(\mr{ii})$ If the ordered subfield $K$ of $R$ is real closed, then a polynomial in $K[\x]$ is $K$-reliable if and only if it is $K$-geometric in $R^n$.
\end{lem}
\begin{proof}
$(\mr{i})$ If $f\in K[\x]$ is $K$-reliable, then $f$ is geometric in $R^n$ by definition, and hence $K$-geometric in $R^n$ by Lemma \ref{univ}$(\mr{ii})$.

$(\mr{ii})$ Straightforward from Corollaries \ref{corro} and \ref{cor243'}.
\end{proof}

Under appropriate hypotheses, item $(\mr{ii})$ of the preceding result extends to the case where $K$ is not real closed. 

\begin{lem}\label{remrem}
Suppose that the ordered subfield $K$ of $R$ is not real closed. If $f$ is a square-free polynomial in $K[\x]$ such that its factorization $f=f_1\cdots f_r$ in $K[\x]$ is also its factori\-zation in $R[\x]$ (that is, each $f_i\in K[\x]$ is irreducible as a polynomial in $R[\x]$, and the $f_i$ are non-associated in $K[\x]$), then $f$ is $K$-reliable if and only if it is $K$-geometric in $R^n$.

In particular, if $f\in K[\x]$ is irreducible as a polynomial in $R[\x]$, then $f$ is $K$-reliable if and only if it is $K$-geometric in $R^n$.
\end{lem}
\begin{proof}
The `only if' implication is true by Lemma \ref{no-news}$(\mr{i})$. Suppose that $f\in K[\x]$ is $K$-geometric in $R^n$. We have to show that $f$ is $K$-reliable. By Corollary \ref{cor:char-geometric}, each $f_i$ satisfies condition $(\mr{v})$ of Proposition \ref{prop:hyper}, that is, $\dim(\ZZ_R(f_i))=n-1$. As $f$ is a square-free polynomial as a polynomial in $R[\x]$ and $f=f_1\cdots f_r$ is its factorization in $R[\x]$, by the mentioned condition~$(\mr{v})$, Corollary \ref{cor:char-geometric} also implies that $f$ is ($R$-)geometric in $R^n$. By Corollary \ref{cor243'}, we deduce that $f$ is $K$-reliable.
\end{proof}

We provide some examples of $K$-geometric polynomials, which are not $K$-reliable. By Lemma \ref{no-news}$(\mr{ii})$, if $K$ is real closed, a polynomial in $K[\x]$ is $K$-reliable if and only if it is $K$-geometric in $R^n$. Thus, we consider the case $K:=\Q$. 

\begin{examples}\label{rem:polyn-reliable}
$(\mr{i})$ As we have already seen previously, the polynomial $g=\x_1^3-2\in\Q[\x_1]$ is $\Q$-geometric in $R$, but not reliable, because $\II_\Q(\ZZ_R(g))=\II_\Q(\{\sqrt[3]{2}\})=(g)\Q[\x_1]$ and $\II_R(\ZZ_R(g))=\II_R(\{\sqrt[3]{2}\})=(\x_1-\sqrt[3]{2})R[\x_1]\neq (g)R[\x_1]$.

A similar example is $G:=\x_1^3-2\x_2^3\in\Q[\x]:=\Q[\x_1,\x_2]$. As $G$ is irreducible in $\Q[\x]$ and $G(-1,0)G(1,0)=-1<0$, Proposition \ref{prop:hyper} implies that $G$ is $\Q$-geometric in $R^2$. Let us show that $G$ is not geometric in $(\qr)^n$ and therefore not reliable either. Set $G_1:=\x_1-\sqrt[3]{2}\x_2\in\qr[\x]$ and $G_2:=\x_1^2+\sqrt[3]{2}\x_1\x_2+\sqrt[3]{4}\x_2^2\in\qr[\x]$. It is immediate to verify that $G=G_1G_2$ is the factorization of $G$ in $\qr[\x]$. As $\ZZ_\qr(G_2)=\{(0,0)\}$, we have $\dim_\qr(\ZZ_\qr(G_2))=0< 2-1$. By Proposition \ref{prop:hyper}, $G$ is not geometric in $(\qr)^n$.

$(\mr{ii})$ The polynomials $g^\bullet$ and $q^\bullet=g^\bullet p^\bullet\in\Q[\x_1,\x_2,\x_3]$ of Examples \ref{exa:gc} are $\Q$-geometric but not geometric in $R^3$, so they are not reliable. Recall that:
\begin{itemize}
\item $X=\ZZ_R(g^\bullet)$ is a $\Q$-irreducible $\Q$-geometric hypersurface of $R^3$ but not a $R$-geometric hypersurface of $R^3$, having two planes $X_0$ and $X_2$, and a line $X_1$ of $R^3$ as irreducible components.
\item $S=\ZZ_R(q^\bullet)=X\cup\ZZ_R(p^\bullet)$ is a $\Q$-geometric and $R$-geometric hypersurface of~$R^3$, having $X$ and $\ZZ_R(p^\bullet)$ as $\Q$-irreducible components. Moreover, $\ZZ_R(p^\bullet)$ has two planes $U_0$ and $U_1$ of $R^3$ as irreducible components, and the planes $X_0$, $X_2$, $U_0$ and $U_1$ are the irreducible components of $S$ (because these planes are pairwise distinct and $U_1$ contains $X_1$). $\sqbullet$
\end{itemize}
\end{examples}

The following definition is a special case of Definition 3 on page 30 of \cite{to2}, due to Tognoli.

\begin{defn}\label{def:overK}
Let $X\subset R^n$ be an algebraic set. We say that $X\subset R^n$ is \emph{defined over~$K$} if $\II_R(X)=\II_K(X)R[\x]$. Moreover, we say that $X$ is a \emph{geometric hypersurface of $R^n$ defined over~$K$} if it is a geometric hypersurface of $R^n$ (in the sense of Definition \ref{321}), and the algebraic set $X\subset R^n$ is defined over $K$. $\sqbullet$
\end{defn}
 
\begin{remarks}\label{exa:2410}
$(\mr{i})$ If an algebraic set $X\subset R^n$ is defined over $K$, then it is also $K$-algebraic, because $X=\ZZ_R(\II_R(X))=\ZZ_R(\II_K(X))$.

$(\mr{i}')$ If an algebraic set $X\subset R^n$ is defined over $K$ and $R|E|K$ is an extension of fields, then $X\subset R^n$ is also an algebraic set defined over $E$, because $\II_R(X)=\II_K(X)R[\x]\subset\II_E(X)R[\x]\subset\II_R(X)$ so $\II_R(X)=\II_E(X)R[\x]$.

$(\mr{ii})$ If a polynomial $f\in K[\x]$ is $K$-reliable (hence $K$-geometric in $R^n$ by Proposition \ref{no-news}$(\mr{i})$), then the algebraic set $\ZZ_R(f)\subset R^n$ is defined over $K$, indeed it is a geometric hypersurface of $R^n$ defined over $K$. See Proposition \ref{hyperdefK} for the converse implication.

$(\mr{ii}')$ If a polynomial $f\in K[\x]$ is $K$-geometric in $R^n$ but not ($R$-)geometric in $R^n$ (or equiva\-lently, not $K$-reliable), then the algebraic set $\ZZ_R(f)\subset R^n$ is not defined over $K$. Otherwise, $\II_R(\ZZ_R(f))=\II_K(\ZZ_R(f))R[\x]=((f)K[\x])R[\x]=(f)R[\x]$ and hence $f$ would be geometric in $R^n$ and $K$-reliable by Corollary \ref{cor243'}.

$(\mr{iii})$ As we have already mentioned, by Corollary \ref{inter}$(\mr{ii})$, if the ordered subfield $K$ of $R$ is real closed, then a subset of $R^n$ is $K$-algebraic if and only if it is defined over $K$.

$(\mr{iii}')$ If $K$ is not real closed, for instance $K=\Q$, there exist $K$-algebraic sets that are not defined over $K$. The polynomials $g$, $G$, $g^\bullet$ and $q^\bullet$ considered in Examples \ref{rem:polyn-reliable} are $\Q$-geometric but not reliable. Thus, by $(\mr{ii}')$, their zero sets are $\Q$-algebraic sets that are not defined over $\Q$. More precisely, we have:
\begin{itemize}
\item The singleton $\{\sqrt[3]{2}\}=\ZZ_R(g)\subset R$, the line $\{\x_1-\sqrt[3]{2}\x_2=0\}=\ZZ_R(G)\subset R^2$ and the union of four planes $\ZZ_R(q^\bullet)\subset R^3$ are examples of $\Q$-geometric and geometric hypersurfaces that are not defined over $\Q$.
\item $\ZZ_R(g^\bullet)$ is a $\Q$-geometric hypersurface of $R^3$, which is not a geometric hypersurface of $R^3$ and it is not defined over $\Q$.
\end{itemize}

Other examples of $\Q$-algebraic sets not defined over $\Q$ will be described in Examples \ref{354}.

$(\mr{iv})$ Let $k\geq3$, let $G_k:=\x_1^{2k}+\x_2^{2k}-2^k\in\Q[\x]:=\Q[\x_1,\x_2]$ and let $F_k:=\ZZ_R(G_k)$ be the corresponding Fermat curve in $R^2$. As $G_k$ is irreducible in $R[\x]$ and $G_k(0,0)G_k(2,0)<0$, by Proposition \ref{prop:hyper}, $G_k$ is geometric in $R^2$. By Corollary \ref{cor243'}, we know that $G_k$ is reliable, so $F_k\subset R^2$ is defined over~$\Q$. Thus, the Fermat curve $F_k\subset R^2$ is an example of irreducible geometric hypersurface of $R^2$ defined over $\Q$ without rational points. $\sqbullet$
\end{remarks}

\begin{prop}\label{hyperdefK}
A set $X\subset R^n$ is a geometric hypersurface of $R^n$ defined over $K$ if and only if there exists a $K$-reliable polynomial $f\in K[\x]$ such that $X=\ZZ_R(f)$ or, equivalently, if $X\subset R^n$ is an algebraic set and $\II_R(X)$ is a principal ideal of $R[\x]$ generated by a polynomial in $K[\x]$.

In particular, a geometric hypersurface of $R^n$ defined over $K$ is also a $K$-geometric hypersurface of $R^n$ (the converse of the latter assertion is false, see Remarks \ref{exa:2410}$(\mr{iii}')$).
\end{prop}
\begin{proof}
First, observe that the conditions ``there exists a $K$-reliable polynomial $f\in K[\x]$ such that $X=\ZZ_R(f)$'' and ``$X\subset R^n$ is an algebraic set and $\II_R(X)$ is a principal ideal of $R[\x]$ generated by a polynomial in $K[\x]$'' are equivalent by Corollary \ref{cor243'}. By Remark \ref{exa:2410}$(\mr{ii})$, it is enough to show the `only if' implication. Suppose that $X$ is a geometric hypersurface of $R^n$ defined over $K$, then a $K$-algebraic subset of $R^n$. Let $X_1,\ldots,X_s$ be the ($R$-)irreducible components of $X\subset R^n$. As $X$ is a geometric hypersurface of $R^n$, by Proposition \ref{prop:hyper}, each $X_i$ has dimension $n-1$.

Let us follow an argument used in the proof of Lemma \ref{lem:238}. Let $W_1,\ldots,W_t$ be the $K$-irreducible components of $X\subset R^n$. For each $i\in\{1,\ldots,t\}$, let $W_{i,1},\ldots,W_{i,T_i}$ be the $(R\text{-})$irreducible components of $W_i\subset R^n$. As $\bigcup_{i=1}^t\bigcup_{\ell=1}^{T_i}W_{i,\ell}=X$, by the uniqueness of the decomposition into irreducible components of $X=\bigcup_{j=1}^sX_j\subset R^n$, we have that $\{X_1,\ldots,X_s\}\subset\bigcup_{i=1}^t\{W_{i,1},\ldots,W_{i,T_i}\}$. Let us prove that, for each $i\in\{1,\ldots,t\}$, there exist $j_i\in\{1,\ldots,s\}$ and $\ell_i\in\{1,\ldots,T_i\}$ such that $W_{i,\ell_i}=X_{j_i}$. Otherwise, there exists $I\in\{1,\ldots,t\}$ such that $\{X_1,\ldots,X_s\}\cap\{W_{I,1},\ldots,W_{I,t_I}\}=\varnothing$, so $\{X_1,\ldots,X_s\}\subset\bigcup_{i\in\{1,\ldots,t\}\setminus\{I\}}\{W_{i,1},\ldots,W_{i,T_i}\}$, so 
$
\textstyle
X=\bigcup_{j=1}^sX_j\subset\bigcup_{i\in\{1,\ldots,t\}\setminus\{I\}}\bigcup_{\ell=1}^{T_i}W_{i,\ell}=\bigcup_{i\in\{1,\ldots,t\}\setminus\{I\}}W_i\subsetneqq X,
$
which is a contradiction.

Consequently, for each $i\in\{1,\ldots,t\}$, it holds $\dim(W_i)\geq\dim(W_{i,\ell_i})=\dim(X_{j_i})=n-1$, so $\dim(W_i)=n-1$. By Proposition \ref{3111}, we deduce that $X$ is a $K$-geometric hypersurface of $R^n$. Thus, $X=\ZZ_R(f)$ and $\II_K(X)=(f)K[\x]$ for some polynomial $f\in K[\x]$. As $X\subset R^n$ is defined over~$K$, we have $\II_R(X)=\II_K(X)R[\x]=((f)K[\x])R[\x]=(f)R[\x]$, so $f$ is geometric in $R^n$. By Corollary \ref{cor243'}, $f\in K[\x]$ is $K$-reliable, as required.
\end{proof}

The next result provides some characterizations of real algebraic sets defined over $K$ that generalize the previous proposition.

\begin{thm}\label{thm:48}
Let $X\subset R^n$ be a $K$-algebraic set. The following conditions are equivalent:
\begin{itemize}
\item[$(\mr{i})$] $X\subset R^n$ is defined over $K$.
\item[$(\mr{ii})$] $\II_R(X)=(f_1,\ldots,f_r)R[\x]$ for some $f_1,\ldots,f_r\in K[\x]$.
\item[$(\mr{iii})$] The complex Galois completion $T$ of $X\subset R^n$ coincides with the complexification of $X\subset R^n$, that is, $T=\zcl_{C^n}^K(X)=\zcl_{C^n}(X)$.
\item[$(\mr{iii}')$] $\II_R(X)=\II_R(T)$.
\item[$(\mr{iv})$] There exists a $K$-reliable family $\{f_1,\ldots,f_r\}\subset K[\x]$ such that $X=\ZZ_R(f_1,\ldots,f_r)$.
\item[$(\mr{v})$] $X\cap\krn\subset\krn$ is defined over $K$.
\end{itemize}
\end{thm}
\begin{proof}
$(\mr{i})\Longrightarrow(\mr{ii})$ This follows immediately from the Noetherianity of $K[\x]$: it is enough to consider a finite system $\{f_1,\ldots,f_r\}$ of generators of $\II_K(X)$ in $K[\x]$.

$(\mr{ii})\Longrightarrow(\mr{i})$ Assume that $\II_R(X)=(f_1,\ldots,f_r)R[\x]$ for some $f_1,\ldots,f_r\in K[\x]$. By Corollary \ref{k}, we have:
$$
\II_K(X)=\II_R(X)\cap K[\x]=(f_1,\ldots,f_r)R[\x]\cap K[\x]=(f_1,\ldots,f_r)K[\x].
$$
Consequently, 
$$
\II_K(X)R[\x]=((f_1,\ldots,f_r)K[\x])R[\x]=(f_1,\ldots,f_r)R[\x]=\II_R(X).
$$

$(\mr{i})\Longleftrightarrow(\mr{iii})\Longleftrightarrow(\mr{iii}')$ Let $T^r:=T\cap R^n$ be the real Galois completion of $X\subset R^n$. Theorem \ref{thm:gc}$(\mr{iv})$ assures that $T^r=\zcl_{R^n}^K(X)=X$. By Theorem \ref{thm:gc}$(\mr{vi})$, $\II_R(T)=\II_K(T^r)R[\x]=\II_K(X)R[\x]$. By equivalence $(\mr{i})\Longleftrightarrow(\mr{iii})$ of Proposition \ref{rc}, $T$ is the complexification of $T^r=X\subset R^n$ if and only if $\II_R(X)=\II_R(T)$. Thus, $\II_R(X)=\II_K(X)R[\x]$ if and only if $\II_R(X)=\II_R(T)$ or, equivalently, if $T=\zcl_{C^n}(X)$.

$(\mr{ii})\Longrightarrow(\mr{iv})$ This follows immediately from implication $(\mr{ii}')\Longrightarrow(\mr{i})$ of Proposition \ref{clue}.

$(\mr{iv})\Longrightarrow(\mr{ii})$ If $X=\ZZ_R(f_1,\ldots,f_r)$ for some $K$-reliable family $\{f_1,\ldots,f_r\}\subset K[\x]$, then $\II_R(X)=\II_R(\ZZ_R(f_1,\ldots,f_r))=(f_1,\ldots,f_r)R[\x]$.

$(\mr{iv})\Longrightarrow(\mr{v})$ Observe that $X\cap\krn=\ZZ_\kr(f_1,\ldots,f_r)$ and $\{f_1,\ldots,f_r\}\subset K[\x]$ is a $K$-reliable family. Then $X\cap\krn\subset\krn$ is defined over $K$ by the previous implication $(\mr{iv})\Longrightarrow(\mr{i})$.

$(\mr{v})\Longrightarrow(\mr{iv})$ Assume that $X':=X\cap\krn\subset\krn$ is defined over $K$ or, equivalently, $X'=\ZZ_\kr(f_1,\ldots,f_r)$ for some $K$-reliable family $\{f_1,\ldots,f_r\}\subset K[\x]$. As $X=(X')_R=(\ZZ_\kr(f_1,\ldots,f_r))_R=\ZZ_R(f_1,\ldots,f_r)$, we obtain $(\mr{iv})$.
\end{proof}

\begin{cor}
Let $\{f_1,\ldots,f_r\}\subset K[\x]$ be a $K$-reliable family. Then $\ZZ_C(f_1,\ldots,f_r)=\zcl_{C^n}(\ZZ_R(f_1,\ldots,f_r))=\zcl_{C^n}^K(\ZZ_R(f_1,\ldots,f_r))$ and $\II_C(\ZZ_C(f_1,\ldots,f_r))=(f_1,\ldots,f_r)C[\x]$.
\end{cor}
\begin{proof}
Let $X:=\ZZ_R(f_1,\ldots,f_r)\subset R^n$ and $T:=\zcl_{C^n}^K(X)\subset C^n$. Observe that $\II_R(X)=(f_1,\ldots,f_r)R[\x]$, because the family $\{f_1,\ldots,f_r\}\subset K[\x]$ is $K$-reliable. By Theorem \ref{thm:48} and Proposition \ref{rc}, we have $T=\zcl_{C^n}(X)$ and $\II_C(T)=\II_R(X)C[\x]=(f_1,\ldots,f_r)C[\x]$, so $T=\ZZ_C(f_1,\ldots,f_r)$, as required. 
\end{proof}

We conclude this part with a result that will be used in Subsection \ref{s51}.

\begin{lem}\label{lem:xLn}
Let $X\subset R^n$ be an algebraic set and let $X_C:=\zcl_{C^n}(X)\subset C^n$ be the complexification of $X$. The following conditions are equivalent:
\begin{itemize}
\item[$(\mr{i})$] $X\subset R^n$ is $R$-irreducible and defined over $K$.
\item[$(\mr{ii})$] $X_C\subset C^n$ is $C$-irreducible and $\II_C(X_C)=\II_K(X)C[\x]$.
\item[$(\mr{ii}')$] $X_C\subset C^n$ is $C$-irreducible and $K$-algebraic.
\end{itemize}
\end{lem}
\begin{proof}
$(\mr{i})\Longrightarrow(\mr{ii})$ By Proposition \ref{prop:zar}$(\mr{i})(\mr{iii})$, $X_C\subset C^n$ is $C$-irreducible and $\II_C(X_C)=\II_R(X)C[\x]=(\II_K(X)R[\x])C[\x]=\II_K(X)C[\x]$.

$(\mr{ii})\Longrightarrow(\mr{ii}')$ This holds, because $X_C=\ZZ_C(\II_K(X))\subset C^n$.

$(\mr{ii}')\Longrightarrow(\mr{ii})$ By Corollary \ref{kreliablec}, we know that $\II_C(X_C)=\II_K(X_C)C[\x]$. As $X\subset X_C\subset\zcl_{C^n}^K(X)$, we deduce $\II_K(X_C)=\II_K(X)$, so $\II_C(X_C)=\II_K(X)C[\x]$. 

$(\mr{ii})\Longrightarrow(\mr{i})$ By Proposition \ref{prop:zar}$(\mr{iii})$ and implication $(\mr{i})\Longrightarrow(\mr{iii})$ of Proposition \ref{rc}, we have that $X\subset R^n$ is $R$-irreducible and $\II_R(X)=\II_R(X_C)$. By Corollary \ref{k}, we deduce $\II_R(X)=\II_R(X_C)=\II_C(X_C)\cap R[\x]=\II_K(X)C[\x]\cap R[\x]=(\II_K(X)R[\x])C[\x]\cap R[\x]=\II_K(X)R[\x]$, so $X\subset R^n$ is defined over $K$, as required.
\end{proof}

\subsection{Detecting real algebraic sets defined over $K$ via Galois completions} \label{rasdok1} It is natural to ask whether it is possible to determine when a real $K$-algebraic set is a real algebraic~set defined over $K$, making use of computational tools such as the complex and real Galois completions (Algorithm \ref{gc} and Theorem \ref{thm:gc}). We have a complete affirmative answer.

We consider first the $K$-irreducible case.

\begin{thm}\label{irredq}
Let $X\subset R^n$ be a $K$-irreducible $K$-algebraic set of dimension $d$ and let $X=\bigcup_{\sigma\in G'}(Z^\sigma\cap R^n)$ be a Galois presentation of $X$ (see Definition {\em\ref{def:gp}}). The following conditions are equivalent:
\begin{itemize}
\item[$(\mr{i})$] $X\subset R^n$ is defined over $K$.
\item[$(\mr{ii})$] $\II_R(Z^\sigma)$ is a real ideal of $R[\x]$ for each $\sigma\in G'$.
\item[$(\mr{iii})$] $Z^\sigma\subset C^n$ is the complexification of $Z^\sigma\cap R^n$ for each $\sigma\in G'$.
\item[$(\mr{iv})$] $\dim(Z^\sigma\cap R^n)=\dim_C(Z^\sigma)=d\,$ for each $\sigma\in G'$. 
\end{itemize}
\end{thm}
\begin{proof}
Let $Y$ be the start of the Galois presentation $X=\bigcup_{\sigma\in G'}(Z^\sigma\cap R^n)$ of $X$, let $T$ be the complex Galois completion of $Y\subset R^n$ and let $T^r:=T\cap R^n$ be the real Galois completion of $Y\subset R^n$. Recall that, by Lemma \ref{lem:gp}$(\mr{ii})$, $T=\zcl_{C^n}^K(X)$ and $T^r=X$.

$(\mr{i})\Longrightarrow(\mr{ii})$ Suppose $X\subset R^n$ is defined over $K$. By Theorem \ref{thm:48}, $\II_R(X)=\II_R(T)$ and $T\subset C^n$ is the complexification of $X$. As $T=\bigcup_{\sigma\in G'}Z^\sigma$,
\begin{equation}\label{II}
\textstyle
\II_R(X)=\bigcap_{\sigma\in G'}\II_R(Z^\sigma).
\end{equation}
As $Y\subset R^n$ is irreducible, Proposition \ref{prop:zar}$(\mr{iii})$ assures that the complexification $Z\subset C^n$ of $Y$ is irreducible as well, that is, $\II_C(Z)$ is a prime ideal of $C[\x]$. Consequently, by Theorem \ref{thm:gc}$(\mr{ii})$, each ideal $\II_C(Z^\sigma)$ of $C[\x]$ is prime, so each ideal $\II_R(Z^\sigma)=\II_C(Z^\sigma)\cap R[\x]$ of $R[\x]$ is prime as well. By \eqref{II} and \cite[Prop.1.11.ii)]{am}, the minimal prime ideals of $\II_R(X)$ in $R[\x]$ is a subfamily of $\{\II_R(Z^\sigma):\sigma\in G'\}$. As $\II_R(X)$ is a real ideal of $R[\x]$, by \cite[Lem.4.1.5]{bcr}, each minimal prime ideal associated to $\II_R(X)$ is also real. Define
$$
F_0:=\big\{\sigma\in G': \text{$\II_R(Z^\sigma)$ is a real ideal of $R[\x]$}\big\},
$$
and choose a subset $F$ of $F_0$ such that $\{\II_R(Z^\sigma):\sigma\in F\}=\{\II_R(Z^\sigma):\sigma\in F_0\}$ and $\II_R(Z^\sigma)\neq\II_R(Z^\tau)$ for all $\sigma,\tau\in F$ with $\sigma\neq\tau$.
In particular,
$$\textstyle
\II_R(X)=\bigcap_{\sigma\in F}\II_R(Z^\sigma).
$$
If $\II_R(Z^\tau)$ is not a real ideal of $R[\x]$ for some $\tau\in G'$, then it is not a minimal prime ideal of $\II_R(X)$ and
$$\textstyle
\bigcap_{\sigma\in F}\II_R(Z^\sigma)\subset\II_R(Z^\tau).
$$
By \cite[Prop.1.11]{am}, there exists $\sigma\in F$ such that $\II_R(Z^\sigma)\subset\II_R(Z^\tau)$. Thus, by implications $(\mr{iv})\Longrightarrow(\mr{iii})\Longrightarrow(\mr{ii})$ of Proposition \ref{rc}, we deduce that
$$
\II_C(Z^\sigma)=\II_R(Z^\sigma)C[\x]\subset\II_R(Z^\tau)C[\x]\subset\II_C(Z^\tau),
$$
so $Z^\tau\subset Z^\sigma$. By Theorem \ref{thm:gc}$(\mr{i}')$, we know that $\dim_C(Z^\tau)=d=\dim_C(Z^\sigma)$. As $Z^\sigma\subset C^n$ is irreducible, we have $Z^\tau=Z^\sigma$, which is a contradiction because $\II_R(Z^\sigma)$ is a real ideal of $R[\x]$, whereas $\II_R(Z^\tau)$ is not.

This means that $F_0=G'$ and $\II_R(Z^\sigma)$ is a real prime ideal of $R[\x]$ for each $\sigma\in G'$.

$(\mr{ii})\Longleftrightarrow(\mr{iii})\Longleftrightarrow(\mr{iv})$ As $Z^\sigma\subset C^n$ is irreducible and $\dim_C(Z^\sigma)=d$, these equivalences follow directly from equivalences $(\mr{i})\Longleftrightarrow(\mr{iv})\Longleftrightarrow(\mr{v})$ of Proposition \ref{rc}.

$(\mr{iv})\Longrightarrow(\mr{i})$ Suppose next $\dim(Z^\sigma\cap R^n)=d=\dim_C(Z^\sigma)$ for each $\sigma\in G'$. Implication $(\mr{v})\Longrightarrow(\mr{iii})$ of Proposition \ref{rc} assures that $\II_R(Z^\sigma\cap R^n)=\II_R(Z^\sigma)$ for each $\sigma\in G'$. As $X=\bigcup_{\sigma\in G'}(Z^\sigma\cap R^n)$,
we deduce 
$$\textstyle
\II_R(X)=\bigcap_{\sigma\in G'}\II_R(Z^\sigma\cap R^n)=\bigcap_{\sigma\in G'}\II_R(Z^\sigma)=\II_R(T).
$$
Using Theorem \ref{thm:gc}$(\mr{vi})$, we have $\II_R(T)=\II_K(T^r)R[\x]=\II_K(X)R[\x]$, so $\II_R(X)=\II_K(X)R[\x]$, as required.
\end{proof}

Next result reduces to the $K$-irreducible case the problem of determining which real $K$-algebraic sets are real algebraic sets defined over $K$.

\begin{thm}\label{defined}
Let $X\subset R^n$ be a $K$-algebraic set and let $X_1,\ldots,X_s$ be the $K$-irreducible components of $X$. Then $X\subset R^n$ is defined over $K$ if and only if each $X_i\subset R^n$ is defined over~$K$.
\end{thm}
\begin{proof}
We prove first the `if' implication. Suppose that each $X_i\subset R^n$ is defined over $K$, that is, $\II_R(X_i)=\II_K(X_i)R[\x]=\II_K(X_i)\otimes_KR$. We have by \cite[Ch2.\S7.Cor.\ to Prop.14, p.306]{b}: 
\begin{align*}
\II_R(X)&\textstyle=\bigcap_{i=1}^s\II_R(X_i)=\bigcap_{i=1}^s\big(\II_K(X_i)\otimes_KR\big)\\
&\textstyle=\big(\bigcap_{i=1}^s\II_K(X_i)\big)\otimes_KR=\II_K(X)\otimes_KR=\II_K(X)R[\x].
\end{align*}
Alternatively, we can prove the previous fact as follows. By Corollary \ref{lem:a}, we have
$$
\II_R(X)=\bigcap_{i=1}^s\II_R(X_i)=\bigcap_{i=1}^s\II_K(X_i)R[\x]=\Big(\bigcap_{i=1}^s\II_K(X_i)\Big)R[\x]=\II_K(X)R[\x].
$$

Let us prove the `only if' implication. We adapt to the present situation the argument already used to prove implication $(\mr{i})\Longrightarrow(\mr{ii})$ of Theorem \ref{irredq}. Assume that $X\subset R^n$ is defined over~$K$, that is, $\II_R(X)=\II_K(X)R[\x]$. Let $X=\bigcup_{i=1}^s\bigcup_{\sigma\in G'}(Z_i^\sigma\cap R^n)$ be a Galois presentation of $X$ with start $(Y_1,\ldots,Y_s)$ (see Definition \ref{def:gp}), where each $Y_i$ is an $R$-irreducible component of $X_i\subset R^n$ of dimension $\dim(X_i)$ for each $i\in\{1,\ldots,s\}$. By Lemma \ref{lem:gp-reducible}, if $T$ and $T^r$ are the complex and real Galois completions of $\bigcup_{i=1}^sY_i\subset R^n$, then
\begin{align*}
&\textstyle T=\bigcup_{i=1}^s T_i=\bigcup_{i=1}^s\bigcup_{\sigma\in G'}Z_i^\sigma,\\
&\textstyle T^r=\bigcup_{i=1}^s T^r_i=\bigcup_{i=1}^s X_i=X,
\end{align*}
where $T_i$ and $T_i^r=X_i$ are the complex and real Galois completions of $Y_i\subset R^n$. By Theorem \ref{thm:gc}$(\mr{vi})$, $\II_R(T)=\II_K(T^r)R[\x]=\II_K(X)R[\x]$. As $\II_R(X)=\II_K(X)R[\x]$, we deduce $\II_R(X)=\II_R(T)$, so
$$\textstyle
\II_R(X)=\bigcap_{i=1}^s\bigcap_{\sigma\in G'}\II_R(Z_i^\sigma).
$$

As each $Y_i\subset R^n$ is irreducible, each $Z_i^\sigma\subset{C^n}$ is irreducible too and each ideal $\II_C(Z_i^\sigma)$ of $C[\x]$ is prime, so each ideal $\II_R(Z_i^\sigma)=\II_C(Z_i^\sigma)\cap R[\x]$ of $R[\x]$ is prime as well. As $\II_R(X)$ is a real ideal of $R[\x]$, each minimal prime ideal associated to $\II_R(X)$ is also real. Define
$$
F_0:=\big\{(i,\sigma)\in\{1,\ldots,s\}\times G'\,:\,\text{$\II_R(Z_i^\sigma)$ is a real ideal of $R[\x]$}\big\},
$$
and choose a subset $F$ of $F_0$ such that $\{\II_R(Z_i^\sigma)\,:\,(i,\sigma)\in F\}=\{\II_R(Z_i^\sigma)\,:\,(i,\sigma)\in F_0\}$ and $\II_R(Z_i^\sigma)\neq\II_R(Z_j^\tau)$ for each $(i,\sigma),(j,\tau)\in F$ with $(i,\sigma)\neq(j,\tau)$. In particular,
\begin{equation}\label{equa37}\textstyle
\II_R(X)=\bigcap_{(i,\sigma)\in F}\II_R(Z_i^\sigma).
\end{equation}
If $\II_R(Z_j^\tau)$ is not a real ideal of $R[\x]$ for some $(j,\tau)\in\{1,\ldots,s\}\times G'$, then it is not a minimal prime ideal of $\II_R(X)$ and
$$\textstyle
\bigcap_{(i,\sigma)\in F}\II_R(Z_i^\sigma)\subset\II_R(Z_j^\tau).
$$
Thus, by \cite[Prop.1.11]{am}, there exists $(i,\sigma)\in F$ such that $\II_R(Z_i^\sigma)\subset\II_R(Z_j^\tau)$, so Proposition \ref{rc} implies
$$
\II_C(Z_i^\sigma)=\II_R(Z_i^\sigma)C[\x]\subset\II_R(Z_j^\tau)C[\x]\subset\II_C(Z_j^\tau),
$$
so $Z_j^\tau\subset Z_i^\sigma$. By Lemma \ref{lem:gp-reducible}$(\mr{iii})$, we deduce $i=j$, so $Z_i^\tau\subset Z_i^\sigma$. As $\dim_C(Z_i^\tau)=\dim(X_i)=\dim_C(Z_i^\sigma)$, we have $Z_i^\tau= Z_i^\sigma$, which is a contradiction because $\II_R(Z_i^\sigma)$ is a real ideal of $R[\x]$, whereas $\II_R(Z_i^\tau)$ is not. 

Consequently, $F_0=\{1,\ldots,s\}\times G'$, that is, $\II_R(Z_i^\sigma)$ is a real prime ideal of $R[\x]$ for each $(i,\sigma)\in\{1,\ldots,s\}\times G'$. By implication $(\mr{ii})\Longrightarrow(\mr{i})$ of Theorem \ref{irredq}, each $X_i\subset R^n$ is defined over $K$, as required.
\end{proof}

The latter two theorems have the following immediate consequence. 

\begin{thm}\label{dq}
Let $X\subset R^n$ be a $K$-algebraic set and let $X=\bigcup_{i=1}^s\bigcup_{\sigma\in G'}(Z_i^\sigma\cap R^n)$ be a Galois presentation of $X$. The following conditions are equivalent:
\begin{itemize}
\item[$(\mr{i})$] $X\subset R^n$ is defined over $K$.
\item[$(\mr{ii})$] $\II_R(Z_i^\sigma)$ is a real ideal of $R[\x]$ for each $i\in\{1,\ldots,s\}$ and $\sigma\in G'$.
\item[$(\mr{iii})$] $Z_i^\sigma\subset C^n$ is the complexification of $Z_i^\sigma\cap R^n$ for each $i\in\{1,\ldots,s\}$ and $\sigma\in G'$.
\item[$(\mr{iv})$] $\dim(Z_i^\sigma\cap R^n)=\dim_C(Z_i^\sigma)$ for each $i\in\{1,\ldots,s\}$ and $\sigma\in G'$.
\end{itemize}
\end{thm}

\begin{remark}
Recall that, by Lemma \ref{lem:gp-reducible}$(\mr{i})(\mr{ii})(\mr{iii})$, $\{Z^\sigma_i\}_{\sigma\in G'}$ is the family of all $C$-irreducible components of $\zcl_{C^n}^K(X_i)\subset C^n$ and $\{Z^\sigma_i\}_{i\in\{1,\ldots,s\},\sigma\in G'}$ is the family of the $C$-irreducible components of $\zcl^K_{C^n}(X)\subset C^n$. $\sqbullet$ 
\end{remark}

Another consequence is the following.

\begin{cor}\label{cor:2261}
Let $X\subset R^n$ be a $K$-algebraic set of dimension $d$, let $X_1,\ldots,X_s$ be the $K$-irreducible components of $X$, let $I$ be the subset of $\{1,\ldots,s\}$ constituted by all indices $i$ such that $\dim(X_i)=d$, and let $B_K(X)$ be the $K$-bad set of $X\subset R^n$ (see Definition $\ref{def:bad-points}$). If $X\subset R^n$ is defined over $K$, then
$$
B_K(X)=\bigcup_{i\in\{1,\ldots,s\}\setminus I}X_i.
$$

If in particular the $K$-irreducible components of $X\subset R^n$ have all the same dimension (which happens for instance if $X$ is $K$-irreducible), then $B_K(X)$ is empty.
\end{cor}
\begin{proof}
By Theorem \ref{defined}, each $X_i\subset R^n$ is defined over $K$. By Remark \ref{rem317} and equivalence $(\mr{i})\Longleftrightarrow(\mr{iv})$ of Theorem \ref{irredq}, we have that $B_K(X_i)=\varnothing$ for each $i\in I$, so $B_K(X)=\bigcup_{i\in I}B_K(X_i)\cup\bigcup_{i\in\{1,\ldots,s\}\setminus I}X_i=\bigcup_{i\in\{1,\ldots,s\}\setminus I}X_i$, as required.
\end{proof}

\begin{remarks}\label{rem:252}
$(\mr{i})$ If $R|K$ is an extension of real closed fields and if $X\subset R^n$ is a $K$-algebraic set of dimension $d$, then Corollary \ref{inter}$(\mr{ii})$ assures that $X\subset R^n$ is defined over $K$, so Corollary \ref{cor:2261} implies that $B_K(X)$ is contained in the union of all the $K$-irreducible components of $X$ of dimension $<d$. If in particular the $K$-irreducible components of $X\subset R^n$ have all the same dimension, then $B_K(X)$ is empty by Corollary \ref{cor:2261}.

$(\mr{ii})$ The converse of Corollary \ref{cor:2261} is false in general, even in the $K$-irreducible case. Indeed, the singleton $\{\sqrt[3]{2}\}\subset R$ is $\Q$-irreducible $\Q$-algebraic, has empty $\Q$-bad set, but it is not defined over $\Q$. $\sqbullet$ 
\end{remarks}

\subsection{Underlying real structures.}\label{urs}
\emph{Recall that $C|R|K$ is a given extension of fields, where $R$ is a real closed field, $C=R[\ii]$ is the algebraic closure of $R$ and $K$ is an ordered subfield of $R$.}

Let us fix some notations. Write a point $z=(z_1,\ldots,z_n)$ of $C^n$ as $z=x+\ii y$, where $x=(x_1,\ldots,x_n),y=(y_1,\ldots,y_n)\in R^n$. This allows us to identify $C^n$ with $R^n\times R^n=R^{2n}$. Define $\ol{w}:=\alpha-\ii\beta$ for all $w=\alpha+\ii\beta\in C$ with $\alpha,\beta\in R$, and set $\ol{z}:=(\ol{z_1},\ldots,\ol{z_n})=x-\ii y\in C^n$. Denote $\varphi:C\to C$ the usual conjugation involution $\varphi(w):=\ol{w}$, whose fixed field is $R$.

Define the indeterminates $\x:=(\x_1,\ldots,\x_n)$, $\y:=(\y_1,\ldots,\y_n)$, $\z:=\x+\ii\y=(\z_1,\ldots,\z_n)$ where $\z_k:=\x_k+\ii\y_k$, and $\ol{\z}:=\x-\ii\y=(\ol{\z_1},\ldots,\ol{\z_n})$ where $\ol{\z_k}:=\x_k-\ii\y_k$. As $\x=\frac{1}{2}(\z+\ol{\z})$ and $\y=-\frac{\ii}{2}(\z-\ol{\z})$, the ring homomorphism $C[\x,\y]\to C[\z,\ol{\z}]$, sending each $\x_k$ to $\frac{1}{2}(\z_k+\ol{\z_k})$ and each $\y_k$ to $-\frac{\ii}{2}(\z_k-\ol{\z_k})$, is actually a ring isomorphism. Thus, we can identify $C[\x,\y]$ with $C[\z,\ol{\z}]$, and consequently $R[\x,\y]$ with a subring of $C[\z,\ol{\z}]$. Observe that $C[\z]$ is also a subring of $C[\z,\ol{\z}]$. Define the ring automorphism $\ol{\varphi}:C[\z,\ol{\z}]\to C[\z,\ol\z]$ as
$$
\textstyle
\ol{\varphi}(\sum_{\nu,\mu}a_{\nu,\mu}\z^\nu\ol{\z}^\mu):=\sum_{\nu,\mu}\varphi(a_{\nu,\mu})\ol{\z}^\nu\z^\mu=\sum_{\nu,\mu}\ol{a_{\nu,\mu}}\z^\mu\ol{\z}^\nu,
$$
which is an involution. As an $R$-vector space, $C[\z,\ol{\z}]$ decomposes as the direct sum of its $R$-vector subspaces $R[\x,\y]$ and $\ii R[\x,\y]$, that is, $C[\z,\ol{\z}]=R[\x,\y]\oplus\ii R[\x,\y]$. More precisely, each $f\in C[\z,\ol{\z}]$ can be uniquely written as $f=a+\ii b$, where $a,b\in R[\x,\y]$ can be computed via $\ol{\varphi}$ as follows:
\begin{equation}\label{qqq}
\textstyle
a:=\frac{1}{2}(f+\ol{\varphi}(f))
\quad\text{ and }\quad
b:=-\frac{\ii}{2}(f-\ol{\varphi}(f)).
\end{equation}
Observe that $R[\x,\y]$ is the fixed point set of $\ol{\varphi}$.

\begin{defns}[Underlying real structure]
Let $X\subset C^n$ be an algebraic set. Define $X^R:=\{(x,y)\in R^n\times R^n=R^{2n}: x+\ii y\in X\}$. We call $X^R\subset R^{2n}$ the \emph{underlying real structure of~$X$}. Given an ideal $\gta$ of $C[\z]$, we define the ideal $\gta^R:=(\gta\cup\ol{\varphi}(\gta))C[\z,\ol{\z}]\cap R[\x,\y]$ of $R[\x,\y]$. $\sqbullet$
\end{defns}

The following result allows us to understand ($K[\ii]$-)algebraic subsets of $C^n$ as ($K$-)algebraic subsets of $R^{2n}$.

\begin{lem}\label{zeroset}
Let $X\subset C^n$ be an algebraic set and let $\gta$ be an ideal of $C[\z]$ such that $X=\ZZ_C(\gta)$. Then $X^R\subset R^{2n}$ is an algebraic set such that $X^R=\ZZ_R(\gta^R)$. In particular, if $X\subset C^n$ is $K[\ii]$-algebraic, then $X^R\subset R^{2n}$ is $K$-algebraic. 
\end{lem}
\begin{proof}
Let $\{g_1,\ldots,g_r\}\subset C[\z]$ be a system of generators of $\gta$ in $C[\z]$. Write $g_k=a_k+\ii b_k$, where $a_k,b_k\in R[\x,\y]$ are defined as in \eqref{qqq} by $a_k:=\frac{1}{2}(g_k+\ol{\varphi}(g_k))$ and $b_k:=-\frac{\ii}{2}(g_k-\ol{\varphi}(g_k))$. The family of polynomials $\{a_1,\ldots,a_r,b_1,\ldots,b_r\}$ constitutes both a system of generators of $(\gta\cup\ol{\varphi}(\gta))C[\z,\ol{\z}]$ in $C[\z,\ol{\z}]$ and a subset of $\gta^R$.

Let us see that $\{a_1,\ldots,a_r,b_1,\ldots,b_r\}$ is also a system of generators of $\gta^R$ in $R[\x,\y]$. If $f\in\gta^R$, then $f\in(\gta\cup\ol{\varphi}(\gta))C[\z,\ol{\z}]$, so $f=\sum_{k=1}^r(u_{k,1}+\ii u_{k,2})a_k+\sum_{k=1}^r(v_{k,1}+\ii v_{k,2})b_k$ for some $u_{k,1},u_{k,2},v_{k,1},v_{k,2}\in R[\x,\y]$. As $f\in R[\x,\y]$, we have $f=\sum_{k=1}^ru_{k,1}a_k+\sum_{k=1}^rv_{k,1}b_k$.

It holds:
\begin{equation*}
\begin{split}
X&=\{x+\ii y\in C^n:g_1(x+\ii y)=0,\ldots,g_r(x+\ii y)=0\}\\
&=\{x+\ii y\in C^n:a_1(x,y)+\ii b_1(x,y)=0,\ldots,a_r(x,y)+\ii b_r(x,y)=0\}\\
&=\{x+\ii y\in C^n:a_1(x,y)=0,\ldots,a_r(x,y)=0,b_1(x,y)=0,\ldots,b_r(x,y)=0\},
\end{split}
\end{equation*}
so $X^R=\{(x,y)\in R^{2n}:x+\ii y\in X\}=\ZZ_R(\gta^R)$. Observe that, if $g_1,\ldots,g_r\in K[\ii][\z]$, then $a_1,\ldots,a_r,b_1,\ldots,b_r\in K[\x,\y]$, as required.
\end{proof}

\begin{remark}\label{olvarphi}
Let $\gta$ be an ideal of $C[\z]$ and let $\{g_1,\ldots,g_r\}\subset C[\z]$ be a system of generators of $\gta$ in $C[\z]$. Write each $g_k=a_k+\ii b_k$ with $a_k,b_k\in R[\x,\y]$. The preceding proof implies that $\{a_1,\ldots,a_r,b_1,\ldots,b_r\}$ is both a system of generators of $(\gta\cup\ol{\varphi}(\gta))C[\z,\ol{\z}]$ in $C[\z,\ol{\z}]$ and of $\gta^R$ in $R[\x,\y]$. $\sqbullet$ 
\end{remark}

Let $X\subset C^n$ be a $K[\ii]$-algebraic set and let $X^R\subset R^{2n}$ be its underlying real stru\-cture, which is a $K$-algebraic subset of $R^{2n}$ by the previous lemma. As $C$ is algebraically closed, Corollary \ref{kreliablec} implies that $\II_C(X)=\II_K(X)C[\z]$. It is now natural to ask whether $\II_R(X^R)=\II_K(X^R)R[\x,\y]$, that is, $X^R\subset R^{2n}$ is an algebraic set defined over $K$. We will see in Examples \ref{354} that even in very simple cases the answer is negative. Something similar happens with the concept of coherence: there exist complex analytic sets (which are always coherent) whose corresponding underlying real structures are not coherent \cite[Prop.III.2.15]{gmt}.

Before presenting the mentioned examples, we need to study some basic properties of the underlying real structure of a complex algebraic set.

Denote $L$ either $R$ or $C$ for a while. Let $Y\subset L^n$ be an algebraic set of dimension $d$ and let $p\in Y$. Recall that $p$ is a \emph{nonsingular point of $Y$ of dimension $d$} if $L[\x]_{\gtn_p}/\II_L(X)L[\x]_{\gtn_p}$ is a regular local ring of dimension $d$. Here $\gtn_p$ denotes the maximal ideal of $L[\x]$ associated to $p\in L^n$. Consider a system of generators $\{f_1,\ldots,f_r\}$ of $\II_L(X)$ in $L[\x]$. By the Jacobian criterion, $p$~is a nonsingular point of $Y$ of dimension $d$ if and only if the Jacobian matrix $G:=\big(\frac{\partial f_k}{\partial\x_j}(p)\big)_{k=1,\ldots,r,\,j=1,\ldots,n}$ has rank $n-d$, that is, ${\rm rk}(G)=n-d$ (see \cite[Ch.I.\S.5]{ha} for $L=C$ and \cite[Prop.3.3.10]{bcr} for $L=R$).

\begin{thm}\label{urs0}
Let $X\subset C^n$ be an algebraic set of dimension $d$. We have:
\begin{itemize}
\item[$(\mr{i})$] $\dim_R(X^R)=2d$.
\item[$(\mr{ii})$] Let $p\in X$. Write $p=p_1+\ii p_2$ with $p_1,p_2\in R^n$. Then $p$ is a nonsingular point of $X\subset C^n$ of dimension $d$ if and only if $(p_1,p_2)$ is a nonsingular point of $X^R\subset R^{2n}$ of dimension $2d$.
\item[$(\mr{iii})$] If $X\subset C^n$ is $K[\ii]$-algebraic, then there exist $f_1,\ldots,f_s\in K[\x,\y]$ such that $\II_C(X)^R=(f_1,\ldots,f_s)R[\x,\y]$.
\item[$(\mr{iii}')$] If $X\subset C^n$ is ($C$-)irreducible, then $X^R\subset R^{2n}$ is ($R$-)irreducible and $\II_R(X^R)=\II_C(X)^R$. If in addition $\{a_1+\ii b_1,\ldots,a_r+\ii b_r\}$ is a system of generators of the ideal $\II_C(X)$ in $C[\z]$ with $a_1,b_1,\ldots,a_r,b_r\in R[\x,\y]$, then $\{a_1,\ldots,a_r,b_1,\ldots,b_r\}$ is a system of generators of $\II_R(X^R)$ in $R[\x,\y]$.
\item[$(\mr{iii}'')$] If $X\subset C^n$ is ($C$-)irreducible and $K[\ii]$-algebraic, then the algebraic set $X^R\subset R^{2n}$ is defined over $K$. 
\end{itemize}
\end{thm}
\begin{proof}
Before proving the statement, we include a preliminary discussion to ease the proof.

Write $\gta:=\II_C(X)$ and let $\{g_1,\ldots,g_r\}\subset C[\z]$ be a system of generators of $\gta$ in $C[\z]$. Write each $g_k=a_k+\ii b_k$ with $a_k,b_k\in R[\x,\y]$. By Remark \ref{olvarphi}, the set $\{a_1,\ldots,a_r,b_1,\ldots,b_r\}$ is both a system of generators of $(\gta\cup\ol{\varphi}(\gta))C[\z,\ol{\z}]$ in $C[\z,\ol{\z}]$ and of $\gta^R$ in $R[\x,\y]$. By Lemma \ref{zeroset} and the Real Nullstellensatz, it holds $\II_R(X^R)=\sqrt[r]{\gta^R}$. 

Let $p=p_1+\ii p_2$ be a point of $X\subset C^n$ and let $(p_1,p_2)$ be the corresponding point of $X^R\subset R^{2n}$. Recall that $\frac{\partial}{\partial\z_j}=\frac{1}{2}\big(\frac{\partial}{\partial\x_j}-\ii\frac{\partial}{\partial\y_j}\big)$ and $\frac{\partial}{\partial\ol\z_j}=\frac{1}{2}\big(\frac{\partial}{\partial\x_j}+\ii\frac{\partial}{\partial\y_j}\big)$. A direct computation shows that $\frac{\partial\ol{\varphi}(g_k)}{\partial\ol{\z_j}}=\ol{\varphi}(\frac{\partial g_k}{\partial\z_j})$ and $\frac{\partial\ol{\varphi}(g_k)}{\partial\z_j}=\ol{\varphi}(\frac{\partial g_k}{\partial\ol{\z_j}})$, so $\frac{\partial\ol{\varphi}(g_k)}{\partial\ol{\z_j}}(p)=\varphi\big(\frac{\partial g_k}{\partial\z_j}(p)\big)$ and $\frac{\partial\ol{\varphi}(g_k)}{\partial\z_j}(p)=\varphi(\frac{\partial g_k}{\partial\ol{\z_j}}(p))$. By Cauchy-Riemann's conditions, we have $\frac{\partial g_k}{\partial\ol{\z_j}}=0$, $\frac{\partial\ol{\varphi}(g_k)}{\partial\z_j}=0$ and $\frac{\partial g_k}{\partial\z_j}=\frac{\partial a_k}{\partial\x_j}-\ii\frac{\partial a_k}{\partial\y_j}=\frac{\partial b_k}{\partial\y_j}+\ii\frac{\partial b_k}{\partial\x_j}$. Consider the following $(r\times n)$-matrices $G$ and $G_\varphi$, $(2r\times 2n)$-matrices $A$ and $B$, $(2r\times 2r)$-matrix $M$ and $(2n\times 2n)$-matrix $N$:
\begin{align*}
G&\textstyle:=\big(\frac{\partial g_k}{\partial\z_j}(p)\big)_{k=1,\ldots,r,\,j=1,\ldots,n}\,,\quad G_\varphi:=\Big(\varphi\big(\frac{\partial g_k}{\partial\z_j}(p)\big)\Big)_{k=1,\ldots,r,\,j=1,\ldots,n}\,,\\
A&:=\left(\begin{array}{c|c}
\frac{\partial g_k}{\partial\z_j}(p)&\frac{\partial g_k}{\partial\ol{\z}_j}(p)\\
\hline
\frac{\partial\ol{\varphi}(g_k)}{\partial\z_j}(p)&\frac{\partial\ol{\varphi}(g_k)}{\partial\ol{\z}_j}(p)
\end{array}\right)_{k=1,\ldots,r,\ j=1,\ldots,n}=
\left(\begin{array}{c|c}
G&0\\
\hline
0&G_\varphi\end{array}\right)\,,\\
B&:=\left(\begin{array}{c|c}
\frac{\partial a_k}{\partial\x_j}(p_1,p_2)&\frac{\partial a_k}{\partial\y_j}(p_1,p_2)\\
\hline
\frac{\partial b_k}{\partial\x_j}(p_1,p_2)&\frac{\partial b_k}{\partial\y_j}(p_1,p_2)
\end{array}\right)_{k=1,\ldots,r,\ j=1,\ldots,n}\,,\\
M&:=\left(\begin{array}{c|c}
I_r&\ii I_r\\
\hline
I_r&-\ii I_r
\end{array}\right)\quad\text{ and }\quad N:=\frac{1}{2}\left(\begin{array}{c|c}
I_n&I_n\\
\hline
-\ii I_n&\ii I_n
\end{array}\right),
\end{align*}
where $I_r$ and $I_n$ are the $(r\times r)$- and $(n\times n)$-identity matrices, respectively. Observe that $B$ is the Jacobian matrix at $(p_1,p_2)\in R^{2n}$ of the polynomial map
$$
R^{2n}\to R^{2r},\; (x,y)\mapsto(a_1(x,y),\ldots,a_r(x,y),b_1(x,y),\ldots,b_r(x,y))
$$
and $G$ is the Jacobian matrix at $p\in C^n$ of the polynomial map $C^n\to C^n:z\mapsto(g_1(z),\ldots,g_r(z))$. The matrices $N$ and $M$ are invertible and $MBN=A$. Consequently, ${\rm rk}(B)={\rm rk}(A)={\rm rk}(G)+{\rm rk}(G_\varphi)=2\,{\rm rk}(G)$, so
\begin{equation}\label{vvv}
{\rm rk}(B)=2\,{\rm rk}(G).
\end{equation}

$(\mr{i})$ By \eqref{eisenbud}, $\hgt(\gta)=n-\dim_C(X)=n-d$ and $\dim_R(X^R)=2n-\hgt(\sqrt[r]{\gta^R})$. Thus, to prove that $\dim_R(X^R)=2d$, we have to check: $\hgt(\sqrt[r]{\gta^R})=2\hgt(\gta)$.

By \cite[Ex.13.13]{e}, it holds
\begin{equation}\label{step1}
\dim(C[\z]/\gta\otimes_CC[\ol{\z}]/\ol{\varphi}(\gta))=\dim(C[\z]/\gta)+\dim(C[\ol{\z}]/\ol{\varphi}(\gta))=2\dim(C[\z]/\gta)=2d.
\end{equation}

As $C[\z]\otimes_CC[\ol{\z}]\cong C[\z,\ol{\z}]$, we deduce by \cite[Ex.1.3]{l}
\begin{equation*}
C[\z]/\gta\otimes_CC[\ol{\z}]/\ol{\varphi}(\gta)\cong C[\z,\ol{\z}]/((\gta\cup\ol{\varphi}(\gta))C[\z,\ol{\z}]).
\end{equation*}

As $C[\z,\ol{\z}]\cong C[\x,\y]$ and $C=R[\ii]$, we also deduce
\begin{align}
(R[\x,\y]/\gta^R)[\ii]&\cong(R[\x,\y]/\gta^R)\otimes_RC\cong C[\x,\y]/(\gta^RC[\x,\y])\nonumber\\
\label{step3}
&\cong C[\z,\ol{\z}]/(\gta^RC[\z,\ol{\z}])=C[\z,\ol{\z}]/((\gta\cup\ol{\varphi}(\gta))C[\z,\ol{\z}])\\
&\cong C[\z]/\gta\otimes_CC[\ol{\z}]/\ol{\varphi}(\gta)\nonumber.
\end{align}

By \eqref{step1} and \eqref{step3}, we conclude that $\dim((R[\x,\y]/\gta^R)[\ii])=2d$ and
$$
2n-\hgt(\gta^R)=\dim(R[\x,\y]/\gta^R)=\dim((R[\x,\y]/\gta^R)[\ii])=2d=2n-2\hgt(\gta),
$$
so $\hgt(\gta^R)=2\hgt(\gta)=2n-2d$. Thus, it remains to check: $\hgt(\sqrt[r]{\gta^R})=\hgt(\gta^R)$.

Suppose that $p=p_1+\ii p_2$ is a nonsingular point of $X\subset C^n$ of dimension $d$. By the Jacobian criterion, we have ${\rm rk}(G)=n-d$. Rearranging the indices if necessary, we can assume that the first $(n-d)$-rows of the matrix $G$ are linearly independent in $C^n$. Thus, there exists a Zariski open neighborhood $U$ of $p$ in $C^n$ such that
\begin{equation}\label{xr}
\ZZ_C(g_1,\ldots,g_{n-d})\cap U=X\cap U.
\end{equation}

Define $U^R:=\{(x,y)\in R^{2n}: x+\ii y\in U\}$. By Lemma \ref{zeroset}, $U^R$ is a Zariski open neighborhood of $(p_1,p_2)$ in $R^{2n}$. Equation \eqref{xr} implies that
$$
X^R\cap U^R=\{(x,y)\in U^R: a_1(x,y)=0,\ldots,a_{n-d}(x,y)=0,b_1(x,y)=0,\ldots,b_{n-d}(x,y)=0\}.
$$
Repeating the argument used to prove \eqref{vvv}, we deduce that the rank of the Jacobian matrix at $(p_1,p_2)\in R^{2n}$ of the polynomial map
$$
R^{2n}\to R^{2n-2d},\; (x,y)\mapsto(a_1(x,y),\ldots,a_{n-d}(x,y),b_1(x,y),\ldots,b_{n-d}(x,y))
$$
is equal to $2n-2d$. By \cite[Prop.3.3.10]{bcr}, $(p_1,p_2)\in X^R$ is a nonsingular point of dimension~$2d$. Thus, $\dim_R(X^R)\geq 2d$, so $\hgt(\gta^R)\leq\hgt(\sqrt[r]{\gta^R})=2n-\dim_R(X^R)\leq 2n-2d=\hgt(\gta^R)$ so $\hgt(\sqrt[r]{\gta^R})=\hgt(\gta^R)$, as required.

$(\mr{ii})$ This item follows immediately from $(\mr{i})$, the Jacobian criterion and \eqref{vvv}. 

$(\mr{iii})$ Suppose that $X\subset C^n$ is $K[\ii]$-algebraic. Let $\{c_1+\ii d_1,\ldots,c_h+\ii d_h\}$ be a system of generators of $\II_{K[\ii]}(X)$ in $K[\ii][\z]$ with $c_1,d_1,\ldots,c_h,d_h\in K[\x,\y]$. By Corollary \ref{kreliablec} and Remark \ref{olvarphi}, we have $\II_C(X)=\II_{K[\ii]}(X)C[\z]=(c_1+\ii d_1,\ldots,c_h+\ii d_h)C[\z]$ and $\II_C(X)^R=(c_1,\ldots,c_h,d_1,\ldots,d_h)R[\x,\y]$. It is now sufficient to set $s:=2h$ and $f_1:=c_1,\ldots,f_h:=c_h$, $f_{h+1}:=d_1,\ldots,f_{2h}:=d_h$.

$(\mr{iii}')$ Suppose that $X\subset C^n$ is irreducible or, equivalently, $\gta=\II_C(X)$ is a prime ideal of $C[\z]$. Observe that the quotient field $\qf(C[\z]/\gta)$ is separable over $C$ (because it has characteristic $0$) and the relative separable algebraic closure of $C$ in $\qf(C[\z]/\gta)$ is $C$ (because $C$ is algebraically closed). By \cite[Ch.V.\S17, Cor.to Prop.1]{b}, the tensor product $\qf(C[\z]/\gta)\otimes_C\qf(C[\ol{\z}]/\ol{\varphi}(\gta))$ is an integral domain. As by \eqref{step3}
$$
R[\x,\y]/\gta^R\hookrightarrow(R[\x,\y]/\gta^R)[\ii]\cong C[\z]/\gta\otimes_CC[\ol{\z}]/\ol{\varphi}(\gta)\hookrightarrow\qf(C[\z]/\gta)\otimes_C\qf(C[\ol{\z}]/\ol{\varphi}(\gta)),
$$
we conclude that $R[\x,\y]/\gta^R$ is also an integral domain, so the ideal $\gta^R$ of $R[\x,\y]$ is prime. We have just proven that $\II_R(X^R)=\sqrt[r]{\gta^R}$ and $\hgt(\gta^R)=\hgt(\sqrt[r]{\gta^R})$. As $\gta^R$ is prime, we deduce that $\gta^R=\sqrt[r]{\gta^R}=\II_R(X^R)$. By Remark \ref{olvarphi}, $\{a_1,b_1,\ldots,a_r,b_r\}$ is a system of generators of $\gta^R=\II_R(X^R)$ in $R[\x,\y]$.

$(\mr{iii}'')$ Suppose that $X\subset C^n$ is both irreducible and $K[\ii]$-algebraic. By $(\mr{iii})$ and $(\mr{iii}')$, we have $\II_R(X^R)=\II_C(X)^R=(f_1,\ldots,f_s)R[\x,\y]$ for some $f_1,\ldots,f_s\in K[\x,\y]$. By implication $(\mr{ii})\Longrightarrow(\mr{i})$ of Theorem \ref{thm:48}, we deduce that $X^R\subset R^{2n}$ is defined over $K$, as required.
\end{proof}

We are ready to present two examples of reducible algebraic curves of $C^2$, the first singular and the second nonsingular, which are $\Q$-irreducible $\Q$-algebraic but whose underlying real structures are not defined over $\Q$. We will use Algorithm \ref{gc}, Theorem \ref{thm:gc}, simultaneous Galois completions (as described in Subsection \ref{gcsequence}) and Lemma \ref{lem:gp-reducible} freely.

The following examples also show that the statements $(\mr{iii}')$ and $(\mr{iii}'')$ of Theorem \ref{urs0} are not true in general for reducible algebraic set $X\subset C^n$.

\begin{examples}\label{354}
$(\mr{i})$ Let $X\subset C^2$ be the algebraic curve defined by
$$
X:=\{(z_1,z_2)\in C^2:z_2^3-2z_1^3=0\}.
$$
Observe that $X$ is singular at the origin $O$ of $C^2$ and $\II_C(X)=(\z_2^3-2\z_1^3)C[\z_1,\z_2]$. In addition, $X\subset C^2$ is $\Q$-algebraic and $\Q$-irreducible. Define $w:=-\frac{1}{2}+\ii\frac{\sqrt{3}}{2}\in C$ and the polynomials $g_1,g_2,g_3\in C[\z_1,\z_2]$ as follows:
$$
g_1:=\z_2-\sqrt[3]{2}\z_1,\qquad g_2:=\z_2-\sqrt[3]{2}w\z_1,\qquad g_3:=\z_2-\sqrt[3]{2}w^2\z_1.
$$

We have $X=X_1\cup X_2\cup X_3$, where $X_i:=\{(z_1,z_2)\in C^2:\, g_i(z_1,z_2)=0\}$.

Observe that $X_1,X_2,X_3$ are the ($C$-)irreducible components of $X\subset C^2$ and each $X_i$ has complex dimension equal to $1$. Write $\z_k:=\x_k+\ii\y_k$ for each $k\in\{1,2\}$. For each $i\in\{1,2,3\}$, let $a_i,b_i\in R[\x_1,\x_2,\y_1,\y_2]$ be such that $g_i=a_i+\ii b_i$.

Consider the $\Q$-algebraic set $X^R\subset R^4$. One can check that
$$
X^R=\{(x_1,x_2,y_1,y_2)\in R^4:\, x_2^3-3x_2y_2^2-2x_1^3+6x_1y_1^2=0,y_2^3-3x_2^2y_2-2y_1^3+6x_1^2y_1=0\}
$$
and
\begin{align*}
&a_1=\x_2-\sqrt[3]{2}\x_1, \qquad b_1=\y_2-\sqrt[3]{2}\y_1,\\
&a_2=\x_2+\tfrac{\sqrt[3]{2}}{2}\x_1+\tfrac{\sqrt[3]{2}\sqrt{3}}{2}\y_1,\qquad b_2=\y_2-\tfrac{\sqrt[3]{2}\sqrt{3}}{2}\x_1+\tfrac{\sqrt[3]{2}}{2}\y_1,\\
&a_3=\x_2+\tfrac{\sqrt[3]{2}}{2}\x_1-\tfrac{\sqrt[3]{2}\sqrt{3}}{2}\y_1,\qquad b_3=\y_2+\tfrac{\sqrt[3]{2}\sqrt{3}}{2}\x_1+\tfrac{\sqrt[3]{2}}{2}\y_1.
\end{align*}

Observe that $X^R=X_1^R\cup X_2^R\cup X_3^R$, where $X_i^R\subset R^4$ is the underlying real structure of $X_i\subset C^2$. We have:
$$
X_i^R:=\{(x_1,x_2,y_1,y_2)\in R^4:\, a_i(x_1,x_2,y_1,y_2)=0,\,b_i(x_1,x_2,y_1,y_2)=0\}.
$$
By Theorem \ref{urs0}$(\mr{i})(\mr{iii}')$, each $X_i^R\subset R^4$ is irreducible and has real dimension $2$. In particular, $X_1^R$, $X_2^R$ and $X_3^R$ are the irreducible components of $X^R\subset R^4$. 

Define $Z_i:=\zcl_{C^4}(X_i^R)$ for each $i\in\{1,2,3\}$. Using again Theorem \ref{urs0}$(\mr{iii}')$, we know that $\{a_i,b_i\}$ is a system of generators of $\II_R(X_i^R)$ in $R[\x_1,\x_2,\y_1,\y_2]$, so it is also a system of generators of $\II_C(Z_i)$ in $C[\x_1,\x_2,\y_1,\y_2]$ by Proposition \ref{prop:zar}$(\mr{i})$. In particular, we have:
\begin{equation}\label{zi}
a_i,b_i\in\qbar[\x_1,\x_2,\y_1,\y_2]
\quad\text{and}\quad
Z_i=\ZZ_C(a_i,b_i)\subset C^4.
\end{equation}

By \eqref{zi}, we can compute the simultaneous Galois completions of $(X_1^R,X_2^R,X_3^R)$, making use of the polynomials $a_1,b_1,a_2,b_2,a_3,b_3$.

Consider the Galois extension $E|\Q$, where $E:=\Q(\sqrt[3]{2},\sqrt{3},\ii)$. Observe that $E$ contains all the coefficients of $a_1,b_1,a_2,b_2,a_3,b_3$. The Galois group $G':=G(E:\Q)$ is isomorphic to the dihedral group $D_6$ and its elements are the automorphisms $\sigma_{a,b,c}:E\to E$ given by $\sigma_{a,b,c}(\sqrt[3]{2})=\sqrt[3]{2}w^a$, $\sigma_{a,b,c}(\sqrt{3})=(-1)^b\sqrt{3}$ and $\sigma_{a,b,c}(\ii)=(-1)^c\ii$ for each $a\in\{0,1,2\}$ and $b,c\in\{0,1\}$. It follows that $\zcl_{C^4}^\Q(X_1^R)=Z_{10}\cup Z_{11}\cup Z_{12}$, where
$$
Z_{1a}:=\{(x_1,x_2,y_1,y_2)\in C^4:\, x_2-\sqrt[3]{2}w^ax_1=0,y_2-\sqrt[3]{2}w^ay_1=0\}.
$$
Thus, $Z_{10}=Z_1$. Moreover, $\zcl_{C^4}^\Q(X_k^R)=\bigcup_{a=0}^2\bigcup_{b=0}^1Z_{kab}$ for each $k\in\{2,3\}$, where
\begin{align*}
Z_{2ab}:=\big\{(x_1,x_2,y_1,y_2)\in C^4:\, &x_2+\tfrac{\sqrt[3]{2}w^a}{2}x_1+\tfrac{\sqrt[3]{2}w^a(-1)^b\sqrt{3}}{2}y_1=0,\\
&y_2-\tfrac{\sqrt[3]{2}w^a(-1)^b\sqrt{3}}{2}x_1+\tfrac{\sqrt[3]{2}w^a}{2}y_1=0\big\},\\
Z_{3ab}:=\big\{(x_1,x_2,y_1,y_2)\in C^4:\, &x_2+\tfrac{\sqrt[3]{2}w^a}{2}x_1-\tfrac{\sqrt[3]{2}w^a(-1)^b\sqrt{3}}{2}y_1=0,\\
&y_2+\tfrac{\sqrt[3]{2}w^a(-1)^b\sqrt{3}}{2}x_1+\tfrac{\sqrt[3]{2}w^a}{2}y_1=0\big\}.
\end{align*}
Observe that $Z_{200}=Z_{301}=Z_2$, $Z_{210}=Z_{311}$, $Z_{220}=Z_{321}$, $Z_{201}=Z_{300}=Z_3$, $Z_{211}=Z_{310}$ and $Z_{221}=Z_{320}$, so $\zcl_{C^4}^\Q(X_2^R)=\zcl_{C^4}^\Q(X_3^R)$ and $\zcl_{C^4}^\Q(X^R)=\zcl_{C^4}^\Q(X_1^R)\cup\zcl_{C^4}^\Q(X_2^R)=\bigcup_{a=0}^2Z_{1a}\cup\bigcup_{a=0}^2\bigcup_{b=0}^1Z_{2ab}$. The sets $Z_{10}$, $Z_{11}$, $Z_{12}$, $Z_{200}$, $Z_{210}$, $Z_{220}$, $Z_{201}$, $Z_{211}$ and $Z_{221}$ are pairwise distinct and none of these sets is included into another. In addition, we have:
\begin{align*}
&Z_{10}\cap R^4=X_1^R,\qquad Z_{200}\cap R^4=X_2^R, \qquad Z_{201}\cap R^4=X_3^R
\end{align*}
and
$$
Z_{1a}\cap R^4=Z_{2ab}\cap R^4=\{O\}
$$
for each $a\in\{1,2\}$ and $b\in\{0,1\}$. Recall that $O$ denotes the origin of $R^4=C^2$.

As $O\in X_1^R\cap X_2^R\cap X_3^R$, we deduce $\zcl_{C^4}^\Q(X_1^R)\cap R^4=X_1^R$ and $\zcl_{C^4}^\Q(X_2^R)\cap R^4=X_2^R\cup X_3^R$, so $X_1^R$ and $X_2^R\cup X_3^R$ are the $\Q$-irreducible components of $X^R\subset R^4$.

Observe that $\dim(Z_{1a}\cap R^4)=0<2=\dim_C(Z_{1a})=\dim(X_1^R)$ and $\dim(Z_{2ab}\cap R^4)=0<2=\dim_C(Z_{2ab})=\dim(X_2^R\cup X_3^R)$ if $a\in\{1,2\}$ and $b\in\{0,1\}$. By the equivalence $(\mr{i})\Longleftrightarrow(\mr{iv})$ of Theorem \ref{dq}, we conclude that $X^R\subset R^4$ is not an algebraic set defined over $\Q$. The same conclusion follows also from equivalence $(\mr{i})\Longleftrightarrow(\mr{iii})$ of Theorem \ref{thm:48} and the fact that $\zcl^\Q_{C^4}(X^R)=\bigcup_{a=0}^2Z_{1a}\cup\bigcup_{a=0}^2\bigcup_{b=0}^1Z_{2ab}\supsetneqq Z_{10}\cup Z_{200}\cup Z_{201}=\zcl_{C^4}(X^R)$.

Observe that the algebraic curve $X\subset C^2$ is an example of reducible complex algebraic set that does not satisfy the statements $(\mr{iii}')$ and $(\mr{iii}'')$ of Theorem \ref{urs0} in the sense that: $X\subset C^2$ is $\Q[\ii]$-algebraic (indeed, $\Q$-algebraic) but $X^R\subset R^4$ is not defined over $\Q$ and $\II_R(X^R)\neq\II_C(X)^R$ (or, better, $\II_R(X^R)=\sqrt[r]{\II_C(X)^R}\supsetneqq\II_C(X)^R$ by Lemma \ref{zeroset}). To prove the latter strict inclusion, it is enough to observe that, if $\II_R(X^R)=\II_C(X)^R$, then $X^R\subset R^4$ would be defined over $\Q$ by Theorem \ref{urs0}$(\mr{iii})$ and implication $(\mr{ii})\Longrightarrow(\mr{i})$ of Theorem \ref{thm:48}.

$(\mr{ii})$ Let $V\subset C^2$ be the $\Q$-irreducible $\Q$-algebraic set defined by
$$
V:=\{(z_1,z_2)\in C^2:(z_1^2+z_2^2)^3-2=0\}.
$$
Observe that $V$ is nonsingular, as an algebraic curve of $ C^2$. Let us proceed as in the previous example. Consider the following polynomials $h_1,h_2,h_3\in C[\z_1,\z_2]$:
$$
h_1:=\z_1^2+\z_2^2-\sqrt[3]{2},\qquad h_2:=\z_1^2+\z_2^2-\sqrt[3]{2}w,\qquad h_3:=\z_1^2+\z_2^2-\sqrt[3]{2}w^2,
$$
where $w:=-\frac{1}{2}+\ii\frac{\sqrt{3}}{2}$ as above. Set $V_i:=\{(z_1,z_2)\in C^2:\, h_i(z_1,z_2)=0\}$ for each $i\in\{1,2,3\}$. Evidently, $V=V_1\cup V_2\cup V_3$ is the decomposition of $V\subset\C^2$ into its irreducible components and each $V_i$ has complex dimension~$1$. Write $h_i=c_i+\ii d_i$ with $c_i,d_i\in R[\x_1,\x_2,\y_1,\y_2]$. Explicit computations provides:

\begin{align*}
&c_1=\x_1^2-\y_1^2+\x_2^2-\y_2^2-\sqrt[3]{2}, \qquad d_1=2\x_1\y_1+2\x_2\y_2,\\
&c_2=\x_1^2-\y_1^2+\x_2^2-\y_2^2+\tfrac{\sqrt[3]{2}}{2},\qquad d_2=2\x_1\y_1+2\x_2\y_2-\tfrac{\sqrt[3]{2}\sqrt{3}}{2},\\
&c_3=\x_1^2-\y_1^2+\x_2^2-\y_2^2+\tfrac{\sqrt[3]{2}}{2},\qquad d_3=2\x_1\y_1+2\x_2\y_2+\tfrac{\sqrt[3]{2}\sqrt{3}}{2}.
\end{align*}

Consider the $\Q$-algebraic set $V^R\subset R^4$. For each $i\in\{1,2,3\}$, define $V_i^R:=\ZZ_R(c_i,d_i)\subset R^4$ and $\Zz_i:=\zcl_{C^4}(V_i^R)$. The sets $V_1^R$, $V_2^R$ and $V_3^R$ are the irreducible components of $V^R\subset R^4$, and have real dimension $2$. By Theorem \ref{urs0}$(\mr{iii}')$ and Proposition \ref{prop:zar}$(\mr{i})$, we have that $\II_C(\Zz_i)$ is generated by $c_i,d_i$ in $C[\x_1,\x_2,\y_1,\y_2]$, so $\Zz_i=\ZZ_C(c_i,d_i)\subset C^4$. In addition, $c_i,d_i$ belong to $\qbar[\x_1,\x_2,\y_1,\y_2]$. Thus, we can compute the simultaneous Galois completions of $(V_1^R,V_2^R,V_3^R)$, making use of the polynomials $c_1,d_1,c_2,d_2,c_3,d_3$.

Consider again the Galois extension $E|\Q$, where $E:=\Q(\sqrt[3]{2},\sqrt{3},\ii)$ contains all the coefficients of $c_1,d_1,c_2,d_2,c_3,d_3$. We have: $\zcl_{C^4}^\Q(V_1^R)=\bigcup_{a=0}^2\Zz_{1a}$ and $\zcl_{C^4}^\Q(V_k^R)=\bigcup_{a=0}^2\bigcup_{b=0}^1\Zz_{kab}$ for each $k\in\{2,3\}$, where
\begin{align*}
\Zz_{1a}&:=\{(x_1,x_2,y_1,y_2)\in C^4:\, x_1^2-y_1^2+x_2^2-y_2^2-\sqrt[3]{2}w^a=x_1y_1+x_2y_2=0\},\\
\Zz_{2ab}&:=\big\{(x_1,x_2,y_1,y_2)\in C^4:\,x_1^2-y_1^2+x_2^2-y_2^2+\tfrac{\sqrt[3]{2}w^a}{2}=2x_1y_1+2x_2y_2-\tfrac{\sqrt[3]{2}w^a(-1)^b\sqrt{3}}{2}=0\big\},\\
\Zz_{3ab}&:=\big\{(x_1,x_2,y_1,y_2)\in C^4:\,x_1^2-y_1^2+x_2^2-y_2^2+\tfrac{\sqrt[3]{2}w^a}{2}=2x_1y_1+2x_2y_2+\tfrac{\sqrt[3]{2}w^a(-1)^b\sqrt{3}}{2}=0\big\},
\end{align*}
for each $a\in\{0,1,2\}$ and $b\in\{0,1\}$. Observe that $\Zz_{2a0}=\Zz_{3a1}$ and $\Zz_{2a1}=\Zz_{3a0}$, so $\zcl_{C^4}^\Q(V_2^R)=\zcl_{C^4}^\Q(V_3^R)$ and $\zcl_{C^4}^\Q(V^R)=\zcl_{C^4}^\Q(V_1^R)\cup\zcl_{C^4}^\Q(V_2^R)=\bigcup_{a=0}^2\Zz_{1a}\cup\bigcup_{a=0}^2\bigcup_{b=0}^1\Zz_{2ab}$. Moreover, the algebraic sets $\Zz_{1a}$ for $a\in\{0,1,2\}$ and $\Zz_{2ab}$ for $a\in\{0,1,2\},b\in\{0,1\}$ are pairwise distinct and none of these sets is included into another. In addition, $\Zz_{10}\cap R^4=V_1^R$, $\Zz_{200}\cap R^4=V_2^R$, $\Zz_{201}\cap R^4=V_3^R$ and $\Zz_{1a}\cap R^4=\Zz_{2ab}\cap R^4=\varnothing$ for each $a\in\{1,2\}$ and $b\in\{0,1\}$. It follows that $\zcl_{C^4}^\Q(V_1^R)\cap R^4=V_1^R$ and $\zcl_{C^4}^\Q(V_2^R)\cap R^4=V_2^R\cup V_3^R$ (because $\zcl_{C^4}^\Q(V_2^R)=\zcl_{C^4}^\Q(V_3^R)$), so $V_1^R$ and $V_2^R\cup V_3^R$ are the $\Q$-irreducible components of $V^R\subset R^4$. As $\zcl^\Q_{C^4}(V^R)=\bigcup_{a=0}^2\Zz_{1a}\cup\bigcup_{a=0}^2\bigcup_{b=0}^1\Zz_{2ab}\supsetneqq\Zz_{10}\cup\Zz_{200}\cup\Zz_{201}=\zcl_{C^4}(V^R)$, equivalence $(\mr{i})\Longleftrightarrow(\mr{iii})$ of Theorem \ref{thm:48} implies that the algebraic set $V^R\subset R^4$ is not defined over $\Q$.

The algebraic curve $V\subset C^2$ provides an example of a reducible nonsingular complex algebraic set that does not satisfy the statements $(\mr{iii}')$ and $(\mr{iii}'')$ of Theorem \ref{urs0}. $\sqbullet$
\end{examples}

%%%
\section{$E|K$-local notions in $L^n$}\label{s3}

\emph{Along this section, $L|E|K$ is an extension of fields such that $L$ is either algebraically closed or real closed. Denote $\ove^\sqbullet$ the algebraic closure of $E$ in $L$. If $L$ is algebraically closed, then $\ove^\sqbullet=\ol{E}$ is the algebraic closure of $E$. If $L$ is real closed, then we assume that both $E$ and $K$ are endowed with the ordering inherited from the one of $L$ (that is, both are ordered subfields of $L$), so $\ove^\sqbullet=\ove^r$ is the real closure of $E$. Evidently, $\ove^\sqbullet=L$ if $E=L$.}

\emph{Although $E$ is chosen arbitrarily with the property $K\subset E\subset L$, the extreme cases $E=L$ and $E=K$ are of particular importance, as we shall see in Subsections \ref{subsec:L|K-E|K}, \ref{subsec:R|R-R|K}, \ref{subsec:proj2} and \ref{nash-tognoli-Q}}.

\emph{We will work with the extensions of fields $L|\ove^\sqbullet|E|K$. We will use Remark \emph{\ref{dime}} freely, and write $\dim(S):=\dim_L(S)$ for each algebraic set $S\subset L^n$}.

Let $a=(a_1,\ldots,a_n)\in L^n$, let ${\rm ev}_a:E[\x]\to L,\ f\mapsto f(a)$ be the evaluation homomorphism, let $E[a]:=E[a_1,\ldots,a_n]$ be the image of ${\rm ev}_a$, let $E(a):=E(a_1,\ldots,a_n)$ be the quotient field of the integral domain $E[a]$ and let $\gtn_a:=\II_E(\{a\})=\{f\in E[\x]:f(a)=0\}$ be the kernel of ${\rm ev}_a$. Observe that $\gtn_a$ is a prime ideal of $E[\x]$, and $\gtn_a=(0)$ if and only if $a_1,\ldots,a_n$ are algebraically independent over $E$. Moreover, $\gtn_a$ is maximal if and only if $E[a]\cong E[\x]/\gtn_a$ has dimension~$0$. By \cite[(14.G) Cor.1, p.91]{m}, the latter assertion is equivalent to the fact that $E(a)$ has transcendence degree $0$ over $E$, that is, $a_1,\ldots,a_n$ are algebraic over $E$ or, equivalently, $a\in(\ove^\sqbullet)^n$. Thus, we have:
\begin{equation}\label{maximal-e}
\text{\it $\gtn_a$ is a maximal ideal of $E[\x]$ if and only if $a\in(\ove^\sqbullet)^n$.}
\end{equation}

In what follows, we identify $E[\x]$ with a subring of $E[\x]_{\gtn_a}$ via the embedding $f\mapsto\frac{f}{1}$.

\subsection{Local rings, localizations and completions}
\label{subsec:local}

Next we focus on the points of $(\ove^\sqbullet)^n$.

\begin{defn}[$E|K$-local rings]\label{sub-alg-reg}
Let $X\subset L^n$ be a $K$-algebraic set and let $a\in X\cap(\ove^\sqbullet)^n$. We define the \emph{$E|K$-local ring $\reg^{E|K}_{X,a}$ of $X$ at $a$} as
$$
\reg^{E|K}_{X,a}:=E[\x]_{\gtn_a}/(\II_K(X)E[\x]_{\gtn_a}),
$$
where $\gtn_a:=\II_E(\{a\})=\{f\in E[\x]:f(a)=0\}$.

If $E=L$, we also introduce the following simplified notation that we will use a few times later: we define the \emph{$K$-local ring $\reg^K_{X,a}$ of $X$ at $a\in X$} as $\reg^K_{X,a}:=\reg^{L|K}_{X,a}=L[\x]_{\gtn_a}/(\II_K(X)L[\x]_{\gtn_a})$, where $\gtn_a:=\{f\in L[\x]:f(a)=0\}$. $\sqbullet$
\end{defn}

\begin{remark}\label{loc-bcr}
If $L=E=K$ (so $K$ is either an algebraically closed field or a real closed field), then $\reg^K_{X,a}$ is the usual local ring $\reg_{X,a}$ of the algebraic set $X$ of $L^n$ at $a$ for each $a\in X$. $\sqbullet$
\end{remark}

The notion of $E|K$-local ring is invariant under the action of certain Galois groups.

\begin{lem}\label{inv-E|K}
Let $X\subset L^n$ be a $K$-algebraic set and let $a=(a_1,\ldots,a_n)\in X\cap(\ove^\sqbullet)^n$. We~have:
\begin{itemize}
\item[$(\mr{i})$] Suppose that $L$ is an algebraically closed field, so $\ove^\sqbullet=\ove$. Let $E'|E$ be a finite Galois subextension of $\ove|E$ that contains $a_1,\ldots,a_n$, let $G':=G(E':E)$ and let $a^\sigma:=(\sigma(a_1),\ldots,\sigma(a_n))\in\ove^n$ for each $\sigma\in G'$. Then
$$\textstyle
\zcl^E_{L^n}(\{a\})=\bigcup_{\psi\in G(L:E)}\{\psi_n(a)\}=\bigcup_{\sigma\in G'}\{a^\sigma\}\subset X\cap\ove^n,
$$
where $\psi_n:L^n\to L^n$ is defined as in \eqref{psi}. In addition, $\zcl^E_{L^n}(\{a\})$ is a finite set,
$$
\text{$\II_E(\{b\})=\II_E(\{a\})\,$ and $\,\reg^{E|K}_{X,b}=\reg^{E|K}_{X,a}\,$ for each $\,b\in \zcl^E_{L^n}(\{a\})$.}
$$
\item[$(\mr{ii})$] Suppose that $L$ is a real closed field, so $\ove^\sqbullet=\ove^r$. Let $E'|E$ be a finite Galois subextension of $\ove^r[\ii]|E$ that contains $a_1,\ldots,a_n$, let $G':=G(E':E)$ and let $a^\sigma:=(\sigma(a_1),\ldots,\sigma(a_n))\in(\ove^r[\ii])^n$ for each $\sigma\in G'$. Then
 $$\textstyle
 \zcl^E_{L^n}(\{a\})=L^n\cap\bigcup_{\psi\in G(L[\ii]:E)}\{\psi_n(a)\}=L^n\cap\bigcup_{\sigma\in G'}\{a^\sigma\}\subset X\cap(\ove^r)^n,
 $$
 where $\psi_n:L[\ii]^n\to L[\ii]^n$ is defined as in \eqref{psi}. In addition, $\zcl^E_{L^n}(\{a\})$ is a finite set,
 $$
 \text{$\II_E(\{b\})=\II_E(\{a\})\,$ and $\,\reg^{E|K}_{X,b}=\reg^{E|K}_{X,a}\,$ for each $\,b\in \zcl^E_{L^n}(\{a\})$.}
 $$
\end{itemize}
\end{lem}
\begin{proof}
$(\mr{i})$ Set $g_i(\x):=\x_i-a_i\in\ove[\x]$ for each $i\in\{1,\ldots,n\}$. Choose a finite Galois subextension $E'|E$ of $\ove|E$ that contains $a_1,\ldots,a_n$ (so all the coefficients of $g_1,\ldots,g_n$), define $G':=G(E':E)$ and apply Algorithm \ref{gc} to $\{a\}=\ZZ_L(g_1,\ldots,g_n)$. By Theorem \ref{thm:gc0}$(\mr{i})(\mr{ii})(\mr{v})$, we know that
$$\textstyle
\zcl^E_{L^n}(\{a\})=\bigcup_{\psi\in G(L:E)}\{\psi_n(a)\}=\bigcup_{\sigma\in G'}\{a^\sigma\}\subset (E')^n\subset\ove^n,
$$
which is a finite set. As $X\subset L^n$ is $K$-algebraic, it is also $E$-algebraic, so $\zcl^E_{L^n}(\{a\})\subset X$.

Pick any $b\in \zcl^E_{L^n}(\{a\})$. As $\{a\}\subset L^n$ is an $\ove$-irreducible $\ove$-algebraic set, by Lemma \ref{lem:irreducibility} (applied to $Y:=\{a\}\subset L^n$ with the extension of fields $L|\ove|E$), we deduce that $\zcl^E_{L^n}(\{a\})\subset L^n$ is $E$-irreducible. Lemma \ref{lem:gp0}$(\mr{ii})$ now implies that $\zcl^E_{L^n}(\{b\})=\zcl^E_{L^n}(\{a\})$ or, equivalently, $\II_E(\{b\})=\II_E(\{a\})$. Consequently, $\reg^{E|K}_{X,b}=\reg^{E|K}_{X,a}$.

$(\mr{ii})$ Let us adapt the previous part of the proof to the present situation. Define $g_i(\x):=\x_i-a_i\in\ove^r[\x]$ for each $i\in\{1,\ldots,n\}$, and observe that $\{g_1,\ldots,g_n\}$ is a system of generators of $\II_{\ove^r}(\{a\})$ in $\ove^r[\x]$. Choose a finite Galois subextension $E'|E$ of $\ove^r[\ii]|E$ that contains $a_1,\ldots,a_n$ (so all the coefficients of $g_1,\ldots,g_n$) and set $G':=G(E':E)$. By Theorem \ref{thm:gc}$(\mr{i})(\mr{i}')(\mr{iv})$,
$$\textstyle
 \zcl^E_{L^n}(\{a\})=L^n\cap\bigcup_{\psi\in G(L[\ii]:E)}\{\psi_n(a)\}=L^n\cap\bigcup_{\sigma\in G'}\{a^\sigma\}\subset L^n\cap(\ove^r[\ii])^n=(\ove^r)^n.
$$
As $X\subset L^n$ is $E$-algebraic, we have $\zcl^E_{L^n}(\{a\})\subset X$.

Let $b\in \zcl^E_{L^n}(\{a\})$. Observe that $\{a\}\subset L^n$ is an $\ove^r$-irreducible $\ove^r$-algebraic set. By Lemma \ref{lem:irreducibility} (applied to $Y:=\{a\}\subset L^n$ with the extension of fields $L|\ove^r|E$), $\zcl^E_{L^n}(\{a\})\subset L^n$ is $E$-irreducible. Lemma \ref{lem:gp}$(\mr{ii})$ implies that $\zcl^E_{L^n}(\{b\})=\zcl^E_{L^n}(\{a\})$ or, equivalently, $\II_E(\{b\})=\II_E(\{a\})$. Consequently, $\reg^{E|K}_{X,b}=\reg^{E|K}_{X,a}$.
\end{proof}

\emph{Pick $a\in(\ove^\sqbullet)^n$, so $E(a)=E[a]$}. If $P(\t)=\sum_{i=0}^db_i\t^i$ is a polynomial in $E[a][t]$, then there exists $f\in E[\x,\t]=E[\x_1,\ldots,\x_n,\t]$ such that $P(\t)=f(a,\t)$ in $E(a)[\t]$ and $\deg(P)$ equals the degree $\deg_{\t}(f)$ of $f$ with respect to $\t$: it is enough to choose polynomials $B_i\in E[\x]$ such that $B_i(a)=b_i$ and define $f:=\sum_{i=0}^dB_i\t^i$.

Let us construct a suitable minimal system of generators of $\gtn_a=\II_E(\{a\})$ in $E[\x]$ (see also \cite[\S3]{con}).

\begin{lem}\label{lem:gtn_a}
Let $a=(a_1,\ldots,a_n)\in(\ove^\sqbullet)^n$. For each $i\in\{1,\ldots,n\}$, denote $E_i$ the field $E(a_1,\ldots,a_{i-1})$ (where $E_1:=E$) and let $P_i\in E_i[\t]$ be the minimal polynomial of $a_i$ over $E_i$. Choose polynomials $f_i\in E[\x_1,\ldots,\x_i]\subset E[\x]$ such that $f_1(\t)=P_1(\t)$ and $f_i$ is monic with respect to $\x_i$ and $f_i(a_1,\ldots,a_{i-1},\t)=P_i(\t)$ for $i\geq2$ (so $\deg(P_i)=\deg_{\x_i}(f)$). Then $\gtn_a=(f_1,\ldots,f_n)E[\x]$.
\end{lem}
\begin{proof}
The inclusion `$\supset$' is clear. To prove the converse inclusion, we proceed by induction on $n$. If $n=1$, then $P_1$ generates $\gtn_a$, because $P_1$ is the minimal polynomial of $a=(a_1)$ over~$E$. Assume the result true for $n-1\geq1$ and let us check that it is also true for $n$. Set $m:=\deg(P_n)=\deg_{\x_n}(f_n)$, $\x':=(\x_1,\ldots,\x_{n-1})$ and $a':=(a_1,\ldots,a_{n-1})$. Pick $f\in\gtn_a$, consider $f$ and $f_n$ as polynomials in $E[\x'][\x_n]$ and divide $f$ by $f_n$ obtaining $f=qf_n+r$, where $q,r\in E[\x'][\x_n]$ are polynomials such that $\deg_{\x_n}(r)<m$. Write $r=\sum_{k=0}^{m-1}b_k(\x')\x_n^k$. We have
$$\textstyle
0=f(a)=r(a)=\sum_{k=0}^{m-1}b_k(a')a_n^k.
$$
As $P_n$ is the minimal polynomial of $a_n$ over $E_n$ and $\deg_{\x_n}(r)<m=\deg(P_n)$, we deduce $r(a',\x_n)=0$, so $b_k(a')=0$ for all $k\in\{0,\ldots,m-1\}$. By induction hypothesis, each $b_k$ belongs to $\gtn_{a'}=(f_1,\ldots,f_{n-1})E[\x']$, so $r=\sum_{k=0}^{m-1}b_k(\x')\x_n^k\in(f_1,\ldots,f_{n-1})E[\x]$ and $f=qf_n+r\in(f_1,\ldots,f_{n-1},f_n)E[\x]$. Consequently, $\gtn_a=(f_1,\ldots,f_n)E[\x]$, as required.
\end{proof}

\begin{remark}\label{F}
The $(n\times n)$-matrix $F:=\big(\frac{\partial f_i}{\partial\x_j}(a)\big)_{i,j=1,\ldots,n}$ with coefficients in the field $E[a]$ is invertible. Indeed, $F$ is a lower triangular matrix whose $i$-th diagonal element is equal to $\frac{\partial f_i}{\partial\x_i}(a)=\frac{\partial P_i}{\partial\t}(a_i)\neq0$. $\sqbullet$ 
\end{remark}

\begin{cor}\label{regular}
If $a\in(\ove^\sqbullet)^n$, then $E[\x]_{\gtn_a}$ is a regular local ring of dimension~$n$, and the polynomials $f_1,\ldots,f_n\in E[\x]$ defined in the statement of Lemma \ref{lem:gtn_a} form a regular system of parameters of $E[\x]_{\gtn_a}$.
\end{cor}
\begin{proof}
As $\gtn_a$ is generated by the $n$ elements $f_1,\ldots,f_n$ by Lemma \ref{lem:gtn_a}, it is enough to show that $\hgt(\gtn_a)\geq n$ (see \cite[\S11]{am}). For each $k\in\{1,\ldots,n\}$, define $\gtp_k:=(f_1,\ldots,f_k)E[\x]$. As $(0)\subsetneqq\gtp_1\subsetneqq\cdots\subsetneqq\gtp_n$, to finish, it is enough to show that each $\gtp_k$ is a prime ideal of $E[\x]$. This is evident for $k=n$, because $\gtp_n=\gtn_a$.
 
Write $a=(a_1,\ldots,a_n)\in(\ove^\sqbullet)^n$. Suppose $n\geq2$, choose $k\in\{1,\ldots,n-1\}$ and denote $a':=(a_1,\ldots,a_k)$, $\x':=(\x_1,\ldots,\x_k)$ and $\x'':=(\x_{k+1},\ldots,\x_n)$. Consider the evaluation homomorphism $\varphi_k:E[\x]\to E[a'][\x''],\ f\mapsto f(a',\x'')$. As $E[a'][\x'']$ is an integral domain and $\varphi_k$ is surjective, it is enough to prove that $\ker(\varphi_k)=\gtp_k$. The inclusion `$\supset$' is clear. To show the converse, pick $f\in\ker(\varphi_k)$ and write $f=\sum_\nu b_\nu(\x'')^\nu$, where $b_\nu\in E[\x']$. As $f\in\ker(\varphi_k)$, we have $0=f(a',\x'')=\sum_\nu b_\nu(a')(\x'')^\nu$, so $b_\nu(a')=0$ for each $\nu$. Using Lemma \ref{lem:gtn_a} again, we deduce $b_\nu\in(f_1,\ldots,f_k)E[\x']\subset\gtp_k$, whence $f\in\gtp_k$, as required.
\end{proof} 

A usual tool to deal with regular local rings is completion. Let us analyze what happens in our setting.

\begin{cor}[Completion]\label{c}
Let $a\in(\ove^\sqbullet)^n$ and let $\gtn_a:=\II_E(\{a\})\subset E[\x]$. We have:
\begin{itemize}
\item[$(\mr{i})$] Let $\gtN_a:=\gtn_a E[\x]_{\gtn_a}$ and $\kappa:=E[\x]_{\gtn_a}/\gtN_a$ be the maximal ideal and the residue field of $E[\x]_{\gtn_a}$, respectively. Then the map $\psi:\kappa\to E[a]$, defined by $\psi\big(\frac{f}{g}+\gtN_a\big):=f(a)(g(a))^{-1}$, is a field isomorphism.
\item[$(\mr{ii})$] The completion $\widehat{E[\x]_{\gtn_a}}$ of $E[\x]_{\gtn_a}$ is isomorphic to $E[a][[\x]]$. In fact, if $\{f_1,\ldots,f_n\}$ is a minimal system of generators of $\gtn_a$ in $E[\x]$, then there exists an isomorphism $\varphi:\widehat{E[\x]_{\gtn_a}}\to E[a][[\x]]$ such that $\varphi(f_i)=\x_i$ for each $i\in\{1,\ldots,n\}$.
\end{itemize}
\end{cor}
\begin{proof}
$(\mr{i})$ By \eqref{maximal-e}, $\gtn_a$ is a maximal ideal of $E[\x]$ and the residue class ring $E[\x]/\gtn_a\cong E[a]$ is a field. Consider the evaluation homomorphism $E[\x]_{\gtn_a}\to E[a], \frac{f}{g}\mapsto f(a)(g(a))^{-1}$, which is well-defined because $E[a]$ is a field and surjective because $E[\x]\subset E[\x]_{\gtn_a}$. Its kernel is $\gtN_a$ so $\psi$ is a field isomorphism.

$(\mr{ii})$ As $E[\x]_{\gtn_a}$ is by Corollary \ref{regular} a regular local ring of dimension $n$, its completion $\widehat{E[\x]_{\gtn_a}}$ is also a regular local ring of dimension $n$ (see \cite[(24.D), p.175]{m}). As $E\subset E[\x]_{\gtn_a}$, Cohen's structure theorem (of completion \cite[Thm.7.7]{e}) states that $\widehat{E[\x]_{\gtn_a}}$ is isomorphic to $k[[\x]]\cong E[a][[\x]]$ via an isomorphism that maps $f_i$ to $\x_i$ for each $i\in\{1,\ldots,n\}$.
\end{proof}

\emph{Fix $a\in(\ove^\sqbullet)^n$ and an ideal $\gta$ of $K[\x]$. Suppose that $\gta\subset\gtn_a=\II_E(\{a\})$, that is, each polynomial in $\gta$ vanishes at $a$}. From the set-theoretic point of view, the latter inclusion makes sense because $K[\x]\subset E[\x]$, so both $\gta$ and $\gtn_a$ are subsets of $E[\x]$.

\emph{Define $A:=E[\x]_{\gtn_a}/(\gta E[\x]_{\gtn_a})$}. Let us analyze the structure of its completion $\widehat{A}$. To that end, we introduce first some preliminary notations.

Let $\varphi:\widehat{E[\x]_{\gtn_a}}\to E[a][[\x]]$ be an isomorphism, whose existence is assured by Corollary \ref{c}$(\mr{ii})$. Write $a\in(\ove^\sqbullet)^n$ explicitly as $a=(a_1,\ldots,a_n)$. Let $\eta:E[a][\x]\to E[a][\x]$ be the ring isomorphism that fixes $E[a]$ and maps each $\x_i$ to $\x_i+a_i$. Let $\gtQ_a:=(\x_1-a_1,\ldots,\x_n-a_n)E[a][\x]$ be the maximal ideal of $E[a][\x]$ associated to $a$, and $\gtQ_0:=(\x_1,\ldots,\x_n)E[a][\x]$ be the maximal ideal of $E[a][\x]$ associated to the origin. Fix an $E$-isomorphism $\psi:\widehat{E[\x]_{\gtQ_a}}\to E[a][[\x]]$, whose existence is assured by Corollary \ref{c}(ii). Observe that $\varphi|_{E[\x]}=\psi|_{E[\x]}$.

\begin{lem}[Structure of completion]\label{completion}
Let $g_1,\ldots,g_s\in\gta$ be polynomials that generate $\gta$ in $K[\x]$ $($and hence $\gta E[\x]_{\gtn_a}$ in $E[\x]_{\gtn_a})$. Define $h_i(\x):=\eta(g_i(\x))=g_i(\x+a)\in E[a][\x]$ for each $i\in\{1,\ldots,s\}$. Define 
$$
B:=E[a][\x]_{\gtQ_a}/((g_1,\ldots,g_s)E[a][\x]_{\gtQ_a})\quad\text{and}\quad C:=E[a][\x]_{\gtQ_0}/((h_1,\ldots,h_s)E[a][\x]_{\gtQ_0}).
$$
Recall that $A:=E[\x]_{\gtn_a}/(\gta E[\x]_{\gtn_a})$ and denote $\widehat{B}$ and $\widehat{C}$ the completions of $B$ and $C$. We have:
\begin{itemize}
\item[$(\mr{i})$] $\widehat{A}\cong E[a][[\x]]/(\varphi(\gta)E[a][[\x]])=E[a][[\x]]/((\varphi(g_1),\ldots,\varphi(g_s))E[a][[\x]])$.
\item[$(\mr{ii})$] $\widehat{B}\cong E[a][[\x]]/((\varphi(g_1),\ldots,\varphi(g_s))E[a][[\x]])$.
\item[$(\mr{iii})$] $\widehat{C}\cong E[a][[\x]]/((h_1,\ldots,h_s)E[a][[\x]])$.
\item[$(\mr{iv})$] $\widehat{A}\cong\widehat{B}\cong\widehat{C}$.
\end{itemize}
\end{lem}
\begin{proof}
By \cite[10.12-10.15]{am} applied to $0\to\gta E[\x]_{\gtn_a}\to E[\x]_{\gtn_a}\to E[\x]_{\gtn_a}/\gta E[\x]_{\gtn_a}\to0$, we have $\widehat{A}\cong\widehat{E[\x]_{\gtn_a}}/(\gta\widehat{E[\x]_{\gtn_a}})$, so using the isomorphism $\varphi$, we deduce $(\mr{i})$. The proof of $(\mr{ii})$ is analogous using the $E$-isomorphism $\psi$ in this case. To prove $(\mr{iii})$, observe first that $h_i(0)=g_i(a)=0$ for each $i\in\{1,\ldots,s\}$. Next we use the system of generators $\{\x_1,\ldots,\x_n\}$ of $\gtQ_0$ to construct an isomorphism between the completion $\widehat{E[a][\x]_{\gtQ_0}}$ and $E[a][[\x]]$. Then we proceed analogously to the proofs of $(\mr{i})$ and $(\mr{ii})$ using \cite[10.12-10.15]{am}. 

As $\eta$ induces an isomorphism between $B$ and $C$, we conclude from (i)-(iii) and using that $\varphi|_{E[\x]}=\psi|_{E[\x]}$ that $\widehat{A}\cong\widehat{B}\cong\widehat{C}$, as required. 
\end{proof}

\begin{remarks}
We keep the notations already introduced in Lemma \ref{completion}.

$(\mr{i})$ Since $\gtQ_a\cap E[\x]=\gtn_a$, the map $E[\x]_{\gtn_a}\to E[a][\x]_{\gtQ_a}$, $\frac{f}{g}\mapsto\frac{f}{g}$ is a well-defined injective homomorphism, so we can identify $E[\x]_{\gtn_a}$ with a subring of $E[a][\x]_{\gtQ_a}$.

$(\mr{ii})$ If $\gta$ is a radical ideal of $K[\x]$, then $\gta E[a][\x]_{\gtQ_a}\cap E[\x]_{\gtn_a}=\gta E[\x]_{\gtn_a}$. The inclusion `$\supset$' is clear. Let us check the converse inclusion. Let $\frac{f}{g}\in\gta E[a][\x]_{\gtQ_a}\cap E[\x]_{\gtn_a}$, where $f,g\in E[\x]$ and $g\not\in\gtn_a$. By Lemma \ref{radical}, $\gta E[\x]$ is a radical ideal of $E[\x]$, so there exist prime ideals $\gtp_1,\ldots,\gtp_r$ of $E[\x]$ such that $\gta E[\x]=\gtp_1\cap\cdots\cap\gtp_r$. We want to prove that, after multiplying $f$ by a suitable polynomial of $E[\x]\setminus\gtn_a$, we can assume that $f\in\gta E[\x]$ or, equivalently, that $f\in\gtp_i$ for all $i\in\{1,\ldots,r\}$. Suppose that $f\not\in\gtp_i$ for some $i\in\{1,\ldots,r\}$ and $\gtp_i\not\subset\gtn_a$. Then we take $p_i\in\gtp_i\setminus\gtn_a$, so $gp_i\in E[\x]\setminus\gtn_a$, $fp_i\in\gtp_i$ and $\frac{fp_i}{gp_i}=\frac{f}{g}$. Thus, substituting $f$ by $fp_i$, we can assume that $f\in\gtp_i$. We can repeat the same procedure with all the indices $i\in\{1,\ldots,r\}$ such that $f\not\in\gtp_i$ and $\gtp_i\not\subset\gtn_a$. Suppose there exists an index $i\in\{1,\ldots,r\}$ such that $f\not\in\gtp_i$ and $\gtp_i\subset\gtn_a$. Let us prove that this cannot happen. As $\frac{f}{g}\in\gta E[a][\x]_{\gtQ_a}\subset\gtp_iE[a][\x]_{\gtQ_a}$, there exists $P\in E[a][\x]\setminus\gtQ_a$ such that $Pf\in\gtp_iE[a][\x]$. Consider the algebraic closure $\ol{E}$ of $E$ and observe that $\ol{E}$ contains $E[a]$. Define $X:=\ZZ_{\ol{E}}(\gtp_iE[a][\x])=\ZZ_{\ol{E}}(\gtp_i)$. By Corollary \ref{kreliablec}, $\II_{\ol{E}}(X)=\gtp_i\ol{E}[\x]$ and $\II_E(X)=\gtp_i$, so $X\subset\ol{E}^n$ is an $E$-irreducible $E$-algebraic set. Let $X_1,\ldots,X_s$ be the $\ol{E}$-irreducible components of~$X$. As $a\in X$, rearranging the indices if necessary, we can assume that $a\in X_1$. As $P(a)\neq0$, we have $P\not\in\II_{\ol{E}}(X_1)$. As $Pf\in\gtp_iE[a][\x]\subset\gtp_i\ol{E}[\x]=\II_{\ol{E}}(X)\subset\II_{\ol{E}}(X_1)$ and $P\not\in\II_{\ol{E}}(X_1)$, we deduce that $f\in\II_{\ol{E}}(X_1)\cap E[\x]$, because $\II_{\ol{E}}(X_1)$ is a prime ideal of $\ol{E}[\x]$. By Lemma \ref{lem:gp0}, we know that $X=\zcl^E_{\ol{E}^n}(X_1)$, so $f\in\II_E(X)=\gtp_i$, which is a contradiction. Consequently, we can assume that $f\in\bigcap_{i=1}^r\gtp_i=\gta E[\x]$, so $\frac{f}{g}\in\gta E[\x]_{\gtn_a}$, as required.

$(\mr{iii})$ If $\gta$ is a radical ideal of $K[\x]$, then $(\mr{i})$ and $(\mr{ii})$ assure that the map $E[\x]_{\gtn_a}/\gta E[\x]_{\gtn_a}\to E[a][\x]_{\gtQ_a}/\gta E[a][\x]_{\gtQ_a}$, $\frac{f}{g}+\gta E[\x]_{\gtn_a}\mapsto\frac{f}{g}+\gta E[a][\x]_{\gtQ_a}$ is a well-defined injective homomorphism.~$\sqbullet$
\end{remarks}

\subsection{Zariski tangent spaces} 
Let us introduce the notion of $E|K$-Zariski tangent space.

\begin{defn}[$E|K$-Zariski tangent space]\label{ek-Zar-tang}
Let $X\subset L^n$ be a $K$-algebraic set and let $a\in X\cap(\ove^\sqbullet)^n$. We define the \textit{$E|K$-Zariski tangent space $T^{E|K}_a(X)$ of $X$ at $a$} as the following $E[a]$-vector subspace of $E[a]^n$:
\begin{equation}\label{eq:E|K-Zariskitangspace}
T^{E|K}_a(X):=\big\{v\in E[a]^n:\langle\nabla g(a),v\rangle=0\,\text{ for all $g\in\II_K(X)$}\big\},
\end{equation}
where $\langle\nabla g(a),v\rangle:=\sum_{j=1}^n\frac{\partial g}{\partial \x_j}(a)v_j$ is the usual matrix product between $\nabla g(a)$ and the column vector $v=(v_1,\ldots,v_n)^T\in E[a]^n$ (recall that the exponent $^T$ indicates the transpose operation).

If $E=L$, we also introduce the following simplified notation that we will use a few times later: we define the \emph{$K$-Zariski tangent space $T^K_a(X)$ of $X$ at $a\in X$} as $T^K_a(X):=T^{L|K}_a(X)$. $\sqbullet$
\end{defn}

\begin{remarks}\label{tang-bcr}
Let $X\subset L^n$ be a $K$-algebraic set. Recall that $X\subset L^n$ is also ($L$-)algebraic. 

$(\mr{i})$ If $a\in X\cap(\ove^\sqbullet)^n$ and $\{g_1,\ldots,g_s\}$ is a system of generators of $\II_K(X)$ in $K[\x]$, then
$$
T^{E|K}_a(X):=\big\{v\in E[a]^n:\langle\nabla g_1(a),v\rangle=0,\ldots,\langle\nabla g_s(a),v\rangle=0\big\}.
$$

$(\mr{i}')$ The nature of $T^{E|K}_a(X)$ is $K$-Zariski local in the following sense: if $U$ is a $K$-Zariski open neighborhood of $a$ in $X\subset L^n$, then $v\in E[a]^n$ belongs to $T^{E|K}_a(X)$ if and only if $\langle\nabla g(a),v\rangle=0$ for all $g\in\II_K(U)$. Indeed, if $g\in\II_K(U)$, then there exists $h\in K[\x]$ such that $h(a)\neq0$ and $\ZZ_L(h)\supset X\setminus U$, so $hg\in\II_K(X)$ and $\nabla(hg)(a)=h(a)\nabla g(a)$.

$(\mr{ii})$ $T^{E|K}_a(X)=T^{L|K}_a(X)\cap E[a]^n$ for each $a\in X\cap(\ove^\sqbullet)^n$.

$(\mr{iii})$ For each $a\in X$, the $L$-Zariski tangent space $T^L_a(X)=T^{L|L}_a(X)$ of $X$ at $a$ coincides with the usual Zariski tangent space $T_a(X)$ of the algebraic set $X\subset L^n$ at $a$, see \cite[Ch.I.\S.5]{ha} and \cite[Section 3.3]{bcr}. As $\II_K(X)\subset\II_L(X)$, we deduce that $T_a(X)=T^L_a(X)$ is a $L$-vector subspace of $T^K_a(X)=T^{L|K}_a(X)$ for each $a\in X$.

If $L$ is an algebraically closed field, then Corollary \ref{kreliablec} assures that $\II_L(X)=\II_K(X)L[\x]$, so $T_a(X)=T^L_a(X)=T^K_a(X)$ for each $a\in X$, and $T^{K|K}_a(X)=T_a(X)\cap K[a]^n$ for each $a\in X\cap(\kbar^\sqbullet)^n$. If $L$ is a real closed field $R$ and $a\in X$, the inclusion $T_a(X)=T^R_a(X)\subset T^K_a(X)$ may be strict.

For instance, if $K=\Q$, $X$ is the line $\ZZ_R(\x_1-\sqrt[3]{2}\x_2)=\ZZ_R(\x_1^3-2\x_2^3)$ of $R^2$ and $a:=(0,0)$, then $T_a(X)=T^R_a(X)=X\subsetneqq R^2=T^\Q_a(X)$. Moreover, $T^{\Q|\Q}_a(X)=T^\Q_a(X)\cap\Q^2=\Q^2$. $\sqbullet$
\end{remarks}

Let $a\in(\ove^\sqbullet)^n$ and let $\gtn_a:=\II_E(\{a\})\subset E[\x]$. Consider the local ring $E[\x]_{\gtn_a}$, its maximal ideal $\gtN_a:=\gtn_a E[\x]_{\gtn_a}$ and its residue field $k:=E[\x]_{\gtn_a}/\gtN_a$. By Corollary \ref{c}$(\mr{i})$, we know that $k$ is isomorphic to $E[a]$ via the field isomorphism $\frac{f}{g}+\gtN_a\mapsto f(a)(g(a))^{-1}$. Consequently, we can understand the $k$-vector space $\gtN_a/\gtN_a^2$ as an $E[a]$-vector space. It is enough to define the scalar multiplication $\ast:E[a]\times\gtN_a/\gtN_a^2\to\gtN_a/\gtN_a^2$ as follows: given any $b\in E[a]$ and $\frac{p}{q}+\gtN_a^2\in\gtN_a/\gtN_a^2$, choose $f\in E[\x]$ in such a way that $b=f(a)$, and set 
\begin{equation}\label{eq:star}
\textstyle
b\ast(\frac{p}{q}+\gtN_a^2):=\frac{fp}{q}+\gtN_a^2.
\end{equation}

Suppose that $a\in X\cap(\ove^\sqbullet)^n$, consider the $E|K$-local ring $\reg^{E|K}_{X,a}=E[\x]_{\gtn_a}/(\II_K(X)E[\x]_{\gtn_a})$ of $X$ at $a$, denote $\gtM_a:=\gtN_a/(\II_K(X)E[\x]_{\gtn_a})$ its maximal ideal and recall that its residue field $\kappa:=\reg^{E|K}_{X,a}/\gtM_a$ is isomorphic to $E[a]$ (via the field isomorphism $(\frac{f}{g}+\II_K(X)E[\x]_{\gtn_a})+\gtM_a\mapsto f(a)(g(a))^{-1}$). As above, we interpret the $\kappa$-vector space $\gtM_a/\gtM_a^2$ as an $E[a]$-vector space by means of the following scalar multiplication, which we again denote $\ast$:
$$\textstyle
b\ast\big((\frac{p}{q}+\II_K(X)E[\x]_{\gtn_a})+\gtM_a^2\big)\mapsto (\frac{fp}{q}+\II_K(X)E[\x]_{\gtn_a})+\gtM_a^2
$$
if $b\in E[a]$ and $f$ is a polynomial in $E[\x]$ such that $b=f(a)$.

\begin{lem}\label{451}
Let $X\subset L^n$ be a $K$-algebraic set, let $a\in X\cap(\ove^\sqbullet)^n$ and let $\gtM_a$ be the maximal ideal of $\reg^{E|K}_{X,a}$. Denote $(T^{E|K}_a(X))^\wedge$ the dual of the $E[a]$-vector space $T^{E|K}_a(X)$. Then the map $\Psi_a:\gtM_a/\gtM_a^2\to(T^{E|K}_a(X))^\wedge$, defined by
$$
\textstyle
\Psi_a\big((\frac{p}{q}+\II_K(X)E[\x]_{\gtn_a})+\gtM_a^2\big):=\big(T^{E|K}_a(X)\ni v\mapsto q(a)^{-1}\langle\nabla p(a),v\rangle\in E[a]\big),
$$
is a well-defined $E[a]$-linear isomorphism.
\end{lem}

We postpone the proof to Appendix \ref{appendix:c}.2.

\subsection{Regular maps and differentials}

Consider again $a\in(\ove^\sqbullet)^n$ and the maximal ideal $\gtn_a=\II_E(\{a\})$ of $E[\x]$. For each $i\in\{1,\ldots,n\}$, we define the function $\frac{\partial}{\partial\x_i}\big|_a:E[\x]_{\gtn_a}\to E[a]$ by
$$
\textstyle
\frac{\partial}{\partial\x_i}\big|_a\big(\frac{p}{q}\big):=q(a)^{-2}\big(q(a)\frac{\partial p}{\partial\x_i}(a)-p(a)\frac{\partial q}{\partial\x_i}(a)\big).
$$
This function is a well-defined abelian homomorphism. Given $m\in\N^*$, we denote $\mc{M}_{m,n}(E[a])$ the set of all $(m\times n)$-matrices with coefficients in the field $E[a]$ and $J_a:(E[\x]_{\gtn_a})^m\to\mc{M}_{m,n}(E[a])$ the abelian homomorphism
$$
\textstyle
J_a\big(\frac{p_1}{q_1},\ldots,\frac{p_m}{q_m}\big):=\big(\frac{\partial}{\partial\x_j}\big|_a\big(\frac{p_i}{q_i}\big)\big)_{i=1,\ldots,m,\,j=1,\ldots,n}\,.
$$
Observe that the $i^{\mr{th}}$-row of $J_a(\frac{p_1}{q_1},\ldots,\frac{p_m}{q_m}\big)$ is equal to $q_i(a)^{-2}\big(q_i(a)\nabla p_i(a)-p_i(a)\nabla q_i(a)\big)$. Therefore, given a $K$-algebraic set $X\subset L^n$ and assuming that $a\in X\cap(\ove^\sqbullet)^n$, it holds:
\begin{equation}\label{eq:TEK}
\textstyle
J_a\big(\frac{p_1}{q_1},\ldots,\frac{p_m}{q_m}\big)v=0\quad\text{if $p_1,\ldots,p_m\in\II_K(X)$ and $v\in T^{E|K}_a(X)$}.
\end{equation}

The next definition is a reformulation of Definition 4 on page 30 of \cite{to2}.

\begin{defn}\label{def:331}
Let $X$ be a subset of $L^n$, let $Y$ be a subset of $L^m$ and let $f:X\to Y$ be a map. Given a point $a\in X$, we say that $f$ is \emph{$K$-regular at $a$} if there exist a $K$-Zariski open neighborhood $U$ of $a$ in $L^n$ and polynomials $p_1,\ldots,p_m,q_1,\ldots,q_m\in K[\x]$ such that $q_1(x)\neq0,\ldots,q_m(x)\neq0$ and $f(x)=\big(\frac{p_1(x)}{q_1(x)},\ldots,\frac{p_m(x)}{q_m(x)}\big)$ for each $x\in X\cap U$. If this is true, we write: \emph{$f=\big(\frac{p_1}{q_1},\ldots,\frac{p_m}{q_m}\big)$ on $X\cap U$}. We say that $f$ is \emph{$K$-regular} if $f$ is $K$-regular at each point of $X$. We say that $f$ is a \emph{$K$-biregular isomorphism} if $f$ is bijective and both $f$ and $f^{-1}$ are $K$-regular.

We denote $\reg^K(X,Y)$ the set of all $K$-regular maps from $X\subset L^n$ to $Y\subset L^m$. $\sqbullet$
\end{defn}

\begin{remark}\label{rem:reg}
In the above definition, by reducing all fractions to a common denominator and shrinking $U$ around $a$ if necessary, we can assume that all polynomials $q_i$ are equal to the same polynomial $q\in K[\x]$ and $U=L^n\setminus \ZZ_L(q)$. $\sqbullet$
\end{remark}

In Definition \ref{def:331-projection}, we will extend the affine notion of $K$-regular map given earlier to the projective setting.

\textit{Suppose that the sets $X\subset L^n$ and $Y\subset L^m$ are $K$-algebraic, the map $f:X\to Y$ is $K$-regular at a point $a\in X\cap(\ove^\sqbullet)^n$ and $f=\big(\frac{p_1}{q_1},\ldots,\frac{p_m}{q_m}\big)$ on some $K$-Zariski open neighborhood of $a$ in $X\subset L^n$. Consider $v\in T^{E|K}_a(X)$}. By \eqref{eq:TEK}, the vector $J_a\big(\frac{p_1}{q_1},\ldots,\frac{p_m}{q_m}\big)v$ of $E[a]^m$ depends only on $f$ and not on the chosen polynomials $p_1,\ldots,p_m,q_1,\ldots,q_m\in K[\x]$ that describe it $K$-Zariski locally at $a$. Thus, we define $J_a(f)v:=J_a\big(\frac{p_1}{q_1},\ldots,\frac{p_m}{q_m}\big)v$. In addition, if $h\in K[\y]:=K[\y_1,\ldots,\y_m]$, a direct (polynomial) computation provides
\begin{equation}\label{composition-jacobians}
\textstyle
\langle\nabla h(f(a)),J_a(f)v\rangle=J_a\big(h\circ f\big)v\;\text{ in $E[a]$},
\end{equation}
where $h\circ f:X\to L$ is the function defined by $(h\circ f)(x):=h(f(x))$ for each $x\in X$, which is $K$-regular at $a$. Thus, by \eqref{eq:TEK}, we have
\begin{equation}\label{nabla}
\textstyle
\langle\nabla h(f(a)),J_a(f)v\rangle=0 \quad \text{if $h\in\II_K(Y)$ and $v\in T^{E|K}_a(X)$.}
\end{equation}

Observe that $f(a)\in Y\cap(\ove^\sqbullet)^m$. In addition, $E[f(a)]$ is a subfield of $E[a]$, which is in general different to $E[a]$, so the vector $J_a(f)v$ of $E[a]^m$ may not belong to the vector subspace $T^{E|K}_{f(a)}(Y)$ of $E[f(a)]^m$, as the following example shows. 

\begin{remark}
Let $E=K:=\Q$, let $g:=\x_1-\x_2^2\in\Q[\x_1,\x_2]$, let $X:=\ZZ_L(g)\subset L^2$, let $Y:=L$, let $f:X\to Y$ be the $\Q$-regular map $f(x_1,x_2):=x_1$ and let $a:=(\sqrt{2},\sqrt[4]{2})\in X\cap(\qbar^\sqbullet)^2$. We~have:
\begin{itemize}
\item $f(a)=\sqrt{2}$ so $\Q[f(a)]=\Q[\sqrt{2}]\subsetneqq\Q[\sqrt[4]{2}]=\Q[a]$;
\item $\II_\Q(X)=(g)\Q[\x_1,\x_2]$ by Corollary \ref{kreliablec} when $L$ is algebraically closed and by Proposition \ref{prop:hyper} when $L$ is real closed;
\item $\nabla g(a)=(1,-2\sqrt[4]{2})$, $T^{\Q|\Q}_a(X)=\{v\in\Q[\sqrt[4]{2}]^2:\langle\nabla g(a),v\rangle=0\}=\{(v_1,v_2)^T\in\Q[\sqrt[4]{2}]^2:v_1-2v_2\sqrt[4]{2}=0\}$, and $T^{\Q|\Q}_{f(a)}(Y)=\Q[\sqrt{2}]$;
\item $v:=(2\sqrt[4]{2},1)^T\in T^{\Q|\Q}_a(X)$, but $J_a(f)=J_a\big(\frac{\x_1}{1}\big)v=2\sqrt[4]{2}\not\in\Q[\sqrt{2}]=T^{\Q|\Q}_{f(a)}(Y)$. $\sqbullet$
\end{itemize}
\end{remark}

If $E[a]=E[f(a)]$ (or, equivalently, $a\in E[f(a)]$), equation \eqref{nabla} assures that $J_a(f)v\in T^{E|K}_{f(a)}(Y)$ for each $v\in T^{E|K}_a(X)$. The last assertion allows us to define the $E|K$-differential of a $K$-regular map.

\begin{defn}\label{def:differential}
Let $X\subset L^n$ and $Y\subset L^m$ be $K$-algebraic sets, let $a\in X\cap(\ove^\sqbullet)^n$ and let $f:X\to Y$ be a map, which is $K$-regular at $a$. If $E[a]=E[f(a)]$, then we define the \emph{$E|K$-differential $d^{E|K}_af:T^{E|K}_a(X)\to T^{E|K}_{f(a)}(Y)$ of $f$ at $a$} as the following $E[a]$-linear map:
$$\textstyle
d^{E|K}_af(v):=J_a(f)v
$$
for each $v\in T^{E|K}_a(X)$.

If $E=L$, we also introduce the following simplified notation that we will use a few times later: we define the \emph{$K$-differential $d^K_a f:T^K_a(X)\to T^K_{f(a)}(Y)$ of $f$ at $a\in X$} as $d^K_a f:=d^{L|K}_a f$. $\sqbullet$
\end{defn}

\begin{remarks}\label{rem:435}
$(\mr{i})$ The equality $E[a]=E[f(a)]$ holds if and only if $a\in E[f(a)]$, that is, if there exist polynomials $h_1,\ldots,h_n\in E[\y]$ such that $a=(h_1,\ldots,h_n)(f(a))$. This is always satisfied if $E=L$, because $L[a]=L[f(a)]=L$.

$(\mr{ii})$ For each $a\in X$, the $L$-differential $d^L_a f:T^L_a(X)\to T^L_{f(a)}(Y)$ of $f$ at $a$ is well-defined and coincides with the usual differential $d_a f:T_a(X)\to T_{f(a)}(Y)$ of the map $f:X\to Y$ between the $(L)$-algebraic sets $X\subset L^n$ and $Y\subset L^m$, which is ($L$-)regular at $a$. We refer the reader to \cite[Ch.2.\S1.3]{sha} for a detailed study of the usual differential.
 $\sqbullet$ 
\end{remarks}

Consider again a map $f:X\to Y$ between $K$-algebraic sets $X\subset L^n$ and $Y\subset L^m$, and a point $a\in X\cap(\ove^\sqbullet)^n$ such that $f$ is $K$-regular at $a$. Set $b:=f(a)\in Y\cap(\ove^\sqbullet)^m$. Let $\gtM_a:=\gtn_aE[\x]_{\gtn_a}/(\II_K(X)E[\x]_{\gtn_a})$ be the maximal ideal of $\reg^{E|K}_{X,a}$ and let $\gtM_b:=\gtn_bE[\y]_{\gtn_b}/(\II_K(Y)E[\y]_{\gtn_b})$ be the maximal ideal of $\reg^{E|K}_{Y,b}$, where $\gtn_a:=\II_E(\{a\})=\{\ell\in E[\x]:\ell(a)=0\}$, $E[\y]:=E[\y_1,\ldots,\y_m]$, and $\gtn_b:=\II_E(\{b\})=\{\ell\in E[\y]:\ell(b)=0\}$.

Given an element $\frac{h}{k}$ of $E[\y]_{\gtn_b}$, we define the element $\frac{h}{k}\circ f$ of $E[\x]_{\gtn_a}$ in the usual way as follows. First, write $f=\big(\frac{p_1}{q},\ldots,\frac{p_m}{q}\big)$ with a common denominator $q$ on some $K$-Zariski open neighborhood of $a$ in $X\subset L^n$, and $h=\sum_{i=0}^dh_i$ and $k=\sum_{i=0}^ek_i$, where $h_i,k_i\in K[\y]$ are homogeneous polynomials of degree $i$, $h_d\neq0$ and $k_e\neq0$. Then define
\begin{equation*}
\textstyle
\frac{h}{k}\circ f:=\frac{q(\x)^e\sum_{i=0}^dh_i(p_1(\x),\ldots,p_m(\x))q(\x)^{d-i}}{q(\x)^d\sum_{i=0}^ek_i(p_1(\x),\ldots,p_m(\x))q(\x)^{e-i}}.
\end{equation*}

Define the pullback homomorphism $f^*:\reg^{E|K}_{Y,b}\to\reg^{E|K}_{X,a}$ and the pullback homomorphism $f^{**}:\gtM_b/\gtM_b^2\to\gtM_a/\gtM_a^2$ induced by $f$ as follows:
\begin{align}\label{eqf*}
\textstyle
f^*(\frac{h}{k}+\II_K(Y)E[\y]_{\gtn_b})&\textstyle:=\frac{h}{k}\circ f+\II_K(X)E[\x]_{\gtn_a},\\
\label{eqf**}\textstyle
f^{**}\big((\frac{h}{k}+\II_K(Y)E[\y]_{\gtn_b})+\gtM_b^2\big)&\textstyle:=(\frac{h}{k}\circ f+\II_K(X)E[\x]_{\gtn_a})+\gtM_a^2.
\end{align}

If $E[a]=E[b]$, we equip $\gtM_a/\gtM_a^2$ and $\gtM_b/\gtM_b^2$ with their structures of $E[a]$-vector space and $f^{**}$ is an $E[a]$-linear map.

From Definition \ref{def:differential} and equality \eqref{composition-jacobians}, we immediately deduce the following result:

\begin{lem}\label{455}
Let $X\subset L^n$ and $Y\subset L^m$ be $K$-algebraic sets, let $f\in\reg^K(X,Y)$, let $a\in X\cap(\ove^\sqbullet)^n$, let $b:=f(a)\in Y\cap(\ove^\sqbullet)^m$, let $\gtM_a$ be the maximal ideal of $\reg^{E|K}_{X,a}$ and let $\gtM_b$ be the maximal ideal of $\reg^{E|K}_{Y,b}$. Define the $E[a]$-linear isomorphism $\Psi_a:\gtM_a/\gtM_a^2\to(T^{E|K}_a(X))^\wedge$ and the $E[b]$-linear isomorphism $\Psi_b:\gtM_b/\gtM_b^2\to(T^{E|K}_b(Y))^\wedge$ as in the statement of Lemma \ref{451}. Then, if $E[a]=E[b]$, the following diagram of $E[a]$-linear maps makes sense and is commutative:
\begin{center}
\begin{tikzpicture}

\node (A) at (0,0) [] {$(T^{E|K}_b(Y))^\wedge$}; 

\node (B) at (4,0) [] {$(T^{E|K}_a(X))^\wedge$}; 

\node (A2) at (0,-2) [] {$\gtM_b/\gtM_b^2$}; 

\node (B2) at (4,-2) [] {$\gtM_a/\gtM_a^2$}; 

\draw [->] (A) --node[above, ]{\scriptsize $(d^{E|K}_a f)^\wedge$} (B);

\draw [->] (A2) --node[left, ]{\scriptsize $\Psi_b$} (A);

\draw [->] (B2) --node[right, ]{\scriptsize $\Psi_a$} (B);

\draw [->] (A2) --node[above, ]{\scriptsize $f^{**}$} (B2);

\node at (2,-.9) [] {$\circlearrowleft$};

\node at (.2,-1) [] {\scriptsize $\cong$};

\node at (3.8,-1) [] {\scriptsize $\cong$};
\end{tikzpicture}
\end{center}
\end{lem}

We conclude this section with a result, which we will use in Subsection \ref{subsec:Whitney}: it states that, if $L$ is a real closed field, then each $K$-regular map can be expressed as a global fraction.

\begin{lem}\label{lem437}
Let $X$ be a subset of $L^n$, let $Y$ be a subset of $L^m$ and let $f\in\reg^K(X,Y)$. If $L$ is a real closed field, then there exist polynomials $p_1,\ldots,p_m,q\in K[\x]$ such that $X\cap \ZZ_L(q)=\varnothing$ and $f(x)=\big(\frac{p_1(x)}{q(x)},\ldots,\frac{p_m(x)}{q(x)}\big)$ for each $x\in X$.
\end{lem}
\begin{proof}
A standard argument works. Evidently, we can assume $m=1$ and $Y=L$. As $f:X\to L$ is $K$-regular, for each $a\in X$, there exist polynomials $p_a,q_a\in K[\x]$ such that $q_a(a)\neq0$ and $f(x)=\frac{p_a(x)}{q_a(x)}$ for all $x\in X\setminus\ZZ_L(q_a)$. The set $X$ equipped with the $K$-Zariski topology is Noetherian, hence compact. Thus, there exist $a_1,\ldots,a_\ell\in X$ such that $X=\bigcup_{i=1}^\ell(X\setminus\ZZ_L(q_{a_i}))$. If we define the polynomials $p,q\in K[\x]$ by $p:=\sum_{i=1}^\ell q_{a_i}p_{a_i}$ and $q:=\sum_{i=1}^\ell q_{a_i}^2$, then $X\cap\ZZ_L(q)=\varnothing$ and $\frac{p(x)}{q(x)}=f(x)$ for all $x\in X$, as required.
\end{proof}

%%%
\section{$E|K$-nonsingular and $E|K$-singular points of $K$-algebraic subsets of $L^n$}\label{s4}

In this section, as in the previous one, we assume that $L|E|K$ is an extension of fields such that $L$ is either algebraically closed or real closed, and $\ove^\sqbullet$ is the algebraic closure of $E$ in $L$. If $L$ is algebraically closed, then $\ove^\sqbullet=\ol{E}$ is the algebraic closure of $E$. If $L$ is real closed, then we assume that both $E$ and $K$ are ordered subfields of $L$, so $\ove^\sqbullet=\ove^r$ is the real closure of $E$. Evidently, $\ove^\sqbullet=L$ if $E=L$.

\emph{Although $E$ is chosen arbitrarily with the property $K\subset E\subset L$, the extreme cases $E=L$ and $E=K$ are of particular importance, as we shall see in Subsections \ref{subsec:L|K-E|K}, \ref{subsec:R|R-R|K}, \ref{subsec:proj2} and \ref{nash-tognoli-Q}.}

\emph{We will use Remark \emph{\ref{dime}} freely, and set $\dim(S):=\dim_L(S)$ for each algebraic set $S\subset L^n$.}

\subsection{Nonsingular and singular points. The Jacobian criterion}
We introduce the notions of $E|K$-nonsingular and $E|K$-singular points of a $K$-algebraic subset of $L^n$. We then prove a related Jacobian criterion.

Recall that, given a point $a\in L^n$, the ideal $\gtn_a:=\II_E(\{a\})=\{f\in E[\x]:f(a)=0\}$ of $E[\x]$ is by \eqref{maximal-e} maximal if and only if $a\in(\ove^\sqbullet)^n$. Moreover, if $a\in(\ove^\sqbullet)^n$, the subring $E[a]$ of $L$ coincides with the subfield $E(a)$ of $L$. For these reasons, we focus on the points in $(\ove^\sqbullet)^n$.

\begin{defn}[$E|K$-nonsingular points]\label{E|K-regular}
Let $X\subset L^n$ be a $K$-algebraic set, let $a\in X\cap(\ove^\sqbullet)^n$ and let $\reg^{E|K}_{X,a}=E[\x]_{\gtn_a}/(\II_K(X)E[\x]_{\gtn_a})$ be the $E|K$-local ring of $X$ at $a$. Set $d:=\dim(X)$.

We say that $a$ is a \emph{$E|K$-nonsingular point of $X$ of dimension~$e$} if $\reg^{E|K}_{X,a}$ is a regular local ring of dimension $e$. An \emph{$E|K$-nonsingular point of $X$} is a $E|K$-nonsingular point of $X$ of dimension $d$. We denote $\Reg^{E|K}(X,e)$ the set of all $E|K$-nonsingular points of~$X$ of dimension $e$, and $\Reg^{E|K}(X):=\Reg^{E|K}(X,d)$ the set of all $E|K$-nonsingular points of~$X$. We say that $\Reg^{E|K}(X)$ is the \emph{$E|K$-nonsingular locus of $X$}. We define the \emph{$E|K$-singular locus $\sing^{E|K}(X)$ of $X$} by $\sing^{E|K}(X):=(X\cap(\ove^\sqbullet)^n)\setminus\Reg^{E|K}(X)$ and we call a point of this set an \emph{$E|K$-singular point of $X$.}

If $E=L$, we also introduce the following simplified notations that we will use a few times later: we define a \emph{$K$-nonsingular point of $X$ (of dimension~$e$)} as an $L|K$-nonsingular point of $X$ (of dimension $e$), $\Reg^K(X,e):=\Reg^{L|K}(X,e)$, $\Reg^K(X):=\Reg^{L|K}(X)$, $\sing^K(X):=\Sing^{L|K}(X)$ and a \emph{$K$-singular point of $X$} as an $L|K$-singular point of~$X$. We say that the $K$-algebraic set $X\subset L^n$ is \emph{$K$-nonsingular} if each of its points is $K$-nonsingular, that is, $\Reg^K(X)=X$. $\sqbullet$
\end{defn}

\begin{remark}\label{regolare}
If $L=E=K$ (so $K$ is either an algebraically closed field or a real closed field), then a $K$-nonsingular point $a$ of the algebraic set $X\subset L^n$ of dimension $e$ is a usual nonsingular point of $X$ of dimension $e$ for each $a\in X$, $\Reg^K(X)$ is the usual set $\Reg(X)$ of nonsingular points of $X$ and $\sing^K(X)$ is the usual set $\sing(X)$ of singular points of $X$, see \cite[Ch.I.\S.5]{ha} and \cite[Sect.3.3]{bcr}. $\sqbullet$
\end{remark}

\begin{remark}[Galois group action invariance]
If $X\subset L^n$ is a $K$-algebraic set, the sets $\Reg^{E|K}(X,e)$ and the set $\sing^{E|K}(X)$ are invariant under the action of the Galois group $G(L:E)$ when $L$ is algebraically closed, whereas they are invariant under the action of $G(L[\ii]:E)$ when $L$ is real closed. This follows immediately from Lemma \ref{inv-E|K}. $\sqbullet$
\end{remark}

\begin{remark}[Easy comparison]\label{reg-bcr}
Consider again a $K$-algebraic set $X\subset L^n$. Under certain conditions on $L|E|K$, it is straightforward to compare nonsingular and singular loci of $X$.

Suppose that $\II_L(X)=\II_E(X)L[\x]$. This happens, for instance, if $L$ is algebraically closed (Corollary \ref{kreliablec}), if $L$ is a real closed field and $X\subset L^n$ is defined over $K$ (Definition \ref{def:overK} \& Remark \ref{exa:2410}$(\mr{i}')$) or $L|K$ is an extension of real closed fields (Corollary \ref{inter}$(\mr{ii})$), because in all the previous cases $\II_L(X)=\II_K(X)L[\x]$, so $\II_L(X)=\II_E(X)L[\x]$. By Definition~\ref{sub-alg-reg}, $\reg^{L|K}_{X,a}\cong\reg^{E|K}_{X,a}\otimes_EL$ for each $a\in X\cap(\ove^\sqbullet)^n$. By base change \cite[(33E), Lem.4, p.253]{m}, if $a\in X\cap(\ove^\sqbullet)^n$, the local ring $\reg^{L|K}_{X,a}$ is regular if and only if so is $\reg^{E|K}_{X,a}$. Therefore, $\Reg^{E|K}(X,e)=\Reg^{L|K}(X,e)\cap(\ove^\sqbullet)^n$ for each $e\in\N$.

In addition, we have $\reg_{X,a}=\reg^{L|L}_{X,a}=\reg^{L|E}_{X,a}$ for each $a\in X$. By Definition \ref{sub-alg-reg} and Remark \ref{regolare}, $\Reg^{L|L}(X,e)=\Reg^{L|E}(X,e)$ for each $e\in\N$, $\Reg(X)=\Reg^{L|L}(X)=\Reg^{L|E}(X)$ and $\Sing(X)=\Sing^{L|L}(X)=\Sing^{L|E}(X)$.

The reader observes that the previous `easy comparison' cases do not include the case of main interest of this paper, that is, when $L$ and $K$ are ordered fields, the ordering of $L$ extends the one of~$K$, $L$ is a real closed field but $K$ is not.

For comparison results in the main case, we refer to Subsections \ref{subsec:L|K-E|K} and \ref{subsec:R|R-R|K}. $\sqbullet$
\end{remark}

\begin{remark}
Let $X\subset L^n$ and $Y\subset L^m$ be $K$-algebraic sets. If 
$f:X\to Y$ is a $K$-biregular isomorphism (see Definition \ref{def:331}), then $\Reg^{E|K}(Y,e)=f(\Reg^{E|K}(X,e))$ for each $e\in\N$. In particular, $\Reg^{E|K}(Y)=f(\Reg^{E|K}(X))$ and $\Sing^{E|K}(Y)=f(\Sing^{E|K}(X))$. To prove this, it is enough to observe that $f(X\cap(\ove^\sqbullet)^n)=Y\cap(\ove^\sqbullet)^m$ and the pullback homomorphism $f^*:\reg^{E|K}_{Y,f(a)}\to\reg^{E|K}_{X,a}$ defined as in \eqref{eqf*} is a ring isomorphism. $\sqbullet$
\end{remark}

\emph{Pick $a\in(\ove^\sqbullet)^n$.} Let $s\in\N^*$ and let $\mc{M}_{s,n}(E[a])$ be the set of all $(s\times n)$-matrices with coefficients in the field $E[a]$. As in Subsection \ref{urs}, if $M\in\mc{M}_{s,n}(E[a])$, we denote ${\rm rk}(M)$ the rank of~$M$. Sometimes, for short, we will use the symbol ${\rm rk}\,M$ instead of ${\rm rk}(M)$.

\emph{Pick an ideal $\gta$ of $K[\x]$ such that $\gta\subset\gtn_a:=\II_E(\{a\})$.} We define the \emph{rank ${\rm rk}_a(\gta)$ of $\gta$ at $a$} as
\begin{equation}\label{eq:rk0}
{\rm rk}_a(\gta):={\rm rk}\Big(\frac{\partial g_i}{\partial\x_j}(a)\Big)_{i=1,\ldots,s,\, j=1,\ldots,n}\ ,
\end{equation}
where $\{g_1,\ldots,g_s\}$ is a system of generators of $\gta$ in $K[\x]$. As $\gta\subset\gtn_a$, the rank ${\rm rk}_a(\gta)$ is well-defined, that is, it does not depend on the chosen system of generators.

Let $\gtN_a:=\gtn_a E[\x]_{\gtn_a}$ be the maximal ideal of $E[\x]_{\gtn_a}$ and let $\gtN_a/\gtN_a^2$ be the $E[a]$-vector space, whose scalar multiplication was defined in \eqref{eq:star}. Observe that $(\gta E[\x]_{\gtn_a}+\gtN_a^2)/\gtN_a^2$ is an $E[a]$-vector subspace of $\gtN_a/\gtN_a^2$, so we can speak about the dimension $\dim_{E[a]}\big((\gta E[\x]_{\gtn_a}+\gtN_a^2)/\gtN_a^2\big)$ of $(\gta E[\x]_{\gtn_a}+\gtN_a^2)/\gtN_a^2$ as an $E[a]$-vector space.

\begin{lem}\label{rk}
If $a$ is a point of $(\ove^\sqbullet)^n$, $\gta$ is an ideal of $K[\x]$ such that $\gta\subset\gtn_a$ and $r:=\dim_{E[a]}\big((\gta E[\x]_{\gtn_a}+\gtN_a^2)/\gtN_a^2\big)$, then
\begin{equation}\label{eq:rk}
r={\rm rk}_a(\gta).
\end{equation}

In particular, if $\{g_1,\ldots,g_s\}$ is a system of generators of $\gta$ in $K[\x]$, then $r\leq\min\{s,n\}$.
\end{lem}
\begin{proof}
Let $\{g_1,\ldots,g_s\}$ be a system of generators of $\gta$ in $K[\x]$ and let $f_1,\ldots,f_n\in E[\x]$ be the polynomials defined in the statement of Lemma \ref{lem:gtn_a}, which constitute a system of generators of $\gtn_a$. As each $g_i$ belongs to $\gtn_a$, by Lemma \ref{lem:gtn_a}, there exist $h_{i1},\ldots,h_{in}\in E[\x]$ such that $g_i=\sum_{j=1}^nh_{ij}f_j$. It follows that $\{g_1+\gtN_a^2,\ldots,g_s+\gtN_a^2\}$ is a system of generators of $(\gta E[\x]_{\gtn_a}+\gtN_a^2)/\gtN_a^2$ and $g_i+\gtN_a^2=\sum_{j=1}^n(h_{ij}f_j+\gtN_a^2)=\sum_{j=1}^nh_{ij}(a)\ast(f_j+\gtN_a^2)$. By Corollary \ref{regular}, $\{f_1,\ldots,f_n\}$ is a regular system of parameters of $E[\x]_{\gtn_a}$, so by \cite[Cor.2, p.302]{zs2}, $\{f_1+\gtN_a^2,\ldots,f_n+\gtN_a^2\}$ is an $E[a]$-vector basis of $\gtN_a/\gtN_a^2$. Define $H:=(h_{ij}(a))_{i=1,\ldots,s,\, j=1,\ldots,n}\in\mc{M}_{s,n}(E[a])$, $G:=\big(\frac{\partial g_i}{\partial\x_j}(a)\big)_{i=1,\ldots,s,\, j=1,\ldots,n}\in\mc{M}_{s,n}(E[a])$ and $F:=\big(\frac{\partial f_i}{\partial\x_j}(a)\big)_{i,j=1,\ldots,n}\in\mc{M}_{n,n}(E[a])$. As $\nabla g_i(a)=\sum_{j=1}^nh_{ij}(a)\nabla f_j(a)$, we deduce $G=HF$. By Remark \ref{F}, the matrix $F$ is invertible. Thus, ${\rm rk}_a(\gta)={\rm rk}(G)={\rm rk}(H)=\dim_{E[a]}((\gta E[\x]_{\gtn_a}+\gtN_a^2)/\gtN_a^2)=r$, as required.
\end{proof}

As $\dim(E[\x]_{\gtn_a})=n$ (by Corollary \ref{regular}) and $E[\x]_{\gtn_a}$ is a localization of a ring of polynomials, we have
\begin{equation} \label{eq:ht}
\hgt(\gta E[\x]_{\gtn_a})=n-\dim(E[\x]_{\gtn_a}/(\gta E[\x]_{\gtn_a})).
\end{equation}
Indeed, by \cite[Prop.3.11\;\&\;Ch.3,Ex.4]{am}, $\hgt(\gta E[\x]_{\gtn_a})=\hgt(\gta E[\x])$ and $\dim(E[\x]_{\gtn_a}/(\gta E[\x]_{\gtn_a}))=\dim((E[\x]/\gta E[\x])_{\gtn_a/\gta E[\x]})$. By \cite[Prop.3.11 \& Thm.11.25]{am}, we conclude
\begin{align*}
\dim((E[\x]/\gta E[\x])_{\gtn_a/\gta E[\x]})&=\hgt(\gtn_a/\gta E[\x])=\hgt(\gtn_a)-\hgt(\gta E[\x])\\
&=\dim(E[\x])-\hgt(\gta E[\x]_{\gtn_a})=n-\hgt(\gta E[\x]_{\gtn_a}),
\end{align*}
so equation \eqref{eq:ht} holds.

The next result is a $E|K$-version of a classical result: here the point is that we consider an ideal $\gta$ of $K[\x]$ and we analyze the regularity of the ring $E[\x]_{\gtn_a}/(\gta E[\x]_{\gtn_a})$ for a point $a\in(\ove^\sqbullet)^n$.

\begin{lem}\label{localregular}
Let $a\in(\ove^\sqbullet)^n$, let $\gta$ be an ideal of $K[\x]$ such that $\gta\subset\gtn_a$, let $r:={\rm rk}_a(\gta)$, and let $A:=E[\x]_{\gtn_a}/(\gta E[\x]_{\gtn_a})$. The following four conditions are equivalent.
\begin{itemize}
\item[$(\mr{i})$] The local ring $A$ is regular.
\item[$(\mr{ii})$] If $r=0$, then $\gta=(0)$. If $r>0$, then there exist $r$ elements $g_1,\ldots,g_r$ of $\gta$ such that $\gta E[\x]_{\gtn_a}=(g_1,\ldots,g_r)E[\x]_{\gtn_a}$.
\item[$(\mr{iii})$] $\dim(A)=n-r$.
\item[$(\mr{\mr{iv}})$] $\mr{ht}(\gta E[\x]_{\gtn_a})=r$.
\end{itemize}

In addition, if $A$ is not regular, then the minimal cardinality of a system of generators of $\gta E[\x]_{\gtn_a}$ in $E[\x]_{\gtn_a}$ is $>r$, $\dim(A)<n-r$ and $\hgt(\gta E[\x]_{\gtn_a})>r$.
\end{lem}

We postpone the proof of the previous result to Appendix \ref{appendix:c}.3. As we shall see, such a proof will crucially use Theorem \ref{thm:ZS2-Thm26}, which is a specific version of a well-known characterization of regular local residue class rings \cite[Thm.26, p.303]{zs2}.

\vspace{3mm}
Recall that $L|E|K$ is a given extension of fields such that $L$ is either algebraically closed or real closed, and $\ove^\sqbullet$ is the algebraic closure of $E$ in $L$.

If $L$ is a real closed field, we can speak about \emph{the (unique) Euclidean topology of $L^n$}, that is, the topology induced by the ordering of $L$ (via the topological product if $n\geq2$). In this situation, $E$ inherits the ordering from $L$ and $\ove^\sqbullet$ coincides with the real closure $\ove^r$ of $E$. Therefore, the Euclidean topology of $(\ove^r)^n$ is equal to the relative topology inherited from the Euclidean topology of $L^n$, because the ordering of $\ove^r$ is the restriction to $\ove^r$ of the ordering of $L$.

In the algebraically closed case, we can also speak about Euclidean topology. In fact, if $L$ is algebraically closed, then $\ove^\sqbullet$ is the algebraic closure $\ove$ of $E$, and there exist an extension of real closed fields $F|R$ such that $F$ is a subfield of $L$, $R$ is a subfield of $\ove$, $L=F[\ii]$ and $\ove=R[\ii]$. Thus, we can equip $L^n$ with the Euclidean topology of $L^n=F^{2n}$ and $\ove^n$ with the Euclidean topology of $\ove^n=R^{2n}$. Evidently, the Euclidean topology of $\ove^n=R^{2n}$ is again the relative topology inherited from the Euclidean topology of $L^n=F^{2n}$. The extension $F|R$ is not uniquely determined by $L|\ove$, so $L^n$ and $\ove^n$ may have several different Euclidean topologies, all of which are finer than the $L$-Zariski and $\ove$-Zariski topologies, respectively. The reason is that, in each of these topologies, the singletons are closed and the polynomial functions $L^n=F^{2n}\to L=F^2$ and $\ove^n=R^{2n}\to\ove=R^2$ are continuous. For more details on Euclidean topologies of $L^n$ with $L$ algebraically closed, see Appendix~\ref{appendix}.

\emph{In the following, if $L$ is an algebraically closed field, we will assume that an extension of real closed fields $F|R$ has been fixed, together with the corresponding Euclidean topologies of $L^n$ and $\ove^n$ induced by the identifications $L^n=F^{2n}$ and $\ove^n=R^{2n}$. In this way, we can speak of ($F$-)semialgebraic subsets of $L^n=F^{2n}$ and of ($R$-)semialgebraic subsets of $\ove^n=R^{2n}$.}

\textit{Consider again a point $a\in(\ove^\sqbullet)^n$}. Recall that $\gtn_a$ denotes the maximal ideal $\{f\in E[\x]:f(a)=0\}$ of $E[\x]$. Define the prime ideal $\gtn_{K,a}:=\{f\in K[\x]:f(a)=0\}=\gtn_a\cap K[\x]$ of $K[\x]$ and identify 
$K[\x]_{\gtn_{K,a}}$ with a subring of $E[\x]_{\gtn_a}$ in the natural way.

The next result is a `localized' version of Corollary \ref{k}.

\begin{lem}\label{lem:222loc}
Let $\gta$ be an ideal of $K[\x]$ and let $a\in\ZZ_{\ove^\sqbullet}(\gta)\subset(\ove^\sqbullet)^n$. Let $g_1,\ldots,g_r\in K[\x]$ be generators of the ideal $\gta E[\x]_{\gtn_a}$ of $E[\x]_{\gtn_a}$. Then $g_1,\ldots,g_r$ are also generators of the ideal $\gta E[\x]_{\gtn_a}\cap K[\x]_{\gtn_{K,a}}$ of $K[\x]_{\gtn_{K,a}}$. In addition, we have:
\begin{equation}\label{equa518}
\gta E[\x]_{\gtn_a}\cap K[\x]_{\gtn_{K,a}}=\gta K[\x]_{\gtn_{K,a}}.
\end{equation}
\end{lem}
\begin{proof}
Let $\frac{f}{g}\in\gta E[\x]_{\gtn_a}\cap K[\x]_{\gtn_{K,a}}$ with $f,g\in K[\x]$ and $g(a)\neq0$. As $f=\frac{f}{g}g\in \gta E[\x]_{\gtn_a}$, there exist $h_0,\ldots,h_r\in E[\x]$ such that $h_0(a)\neq0$ and $h_0f=\sum_{i=1}^rh_ig_i$ in $E[\x]$. Let $\Bb:=\{u_j\}_{j\in J}$ be a basis of $E$ as a $K$-vector space. By Lemma \ref{k0}, we can write each $h_i$ as follows: $h_i=\sum_{j\in J}u_jh_{ij}$, where each $h_{ij}$ belongs to $K[\x]$ and only finitely many of these polynomials are non-zero. We have:
$$
\sum_{j\in J}u_j(h_{0j}f)=h_0f=\sum_{i=1}^rh_ig_i=\sum_{j\in J}u_j\Big(\sum_{i=1}^rh_{ij}g_i\Big).
$$
As $h_0(a)\neq0$, there exists an index $j_0\in J$ such that $h_{0j_0}(a)\neq0$. Consequently, by the uniqueness assertion in Lemma \ref{k0}, we deduce $h_{0j_0}f=\sum_{i=1}^rh_{ij_0}g_i$, so $f\in\gta K[\x]_{\gtn_{K,a}}$. It follows that $\frac{f}{g}\in\gta K[\x]_{\gtn_{K,a}}$ and $\gta E[\x]_{\gtn_a}\cap K[\x]_{\gtn_{K,a}}\subset\gta K[\x]_{\gtn_{K,a}}$. As the converse inclusion is clear, we conclude that equation \eqref{equa518} holds. As $\frac{f}{g}=\sum_{i=1}^r\frac{h_{ij_0}}{h_{0j_0}g}g_i$ and $\frac{h_{ij_0}}{h_{0j_0}g}\in K[\x]_{\gtn_{K,a}}$ for all $i\in\{1,\ldots,r\}$, we deduce that $g_1,\ldots,g_r$ are generators of the ideal $\gta E[\x]_{\gtn_a}\cap K[\x]_{\gtn_{K,a}}$ of $K[\x]_{\gtn_{K,a}}$, as required.
\end{proof}

A weak version of the previous result was first used in the proof \cite[Lem.2.15]{GS}.

To prove the following result, we adapt the strategy used in the proof of \cite[Prop.3.3.8 \& Prop.3.3.10]{bcr} to the present situation. 

\begin{thm}[$E|K$-Jacobian criterion]\label{Q-jacobian}
Let $a\in(\ove^\sqbullet)^n\subset L^n$, let $\gta$ be an ideal of $K[\x]$ such that $\gta\subset\gtn_a:=\{f\in E[\x]:f(a)=0\}$, and let $e\in\{0,\ldots,n\}$. If $e=n$, then the local ring $E[\x]_{\gtn_a}/(\gta E[\x]_{\gtn_a})$ is regular of dimension $e$ if and only if $\gta=(0)$. If $e<n$, then the following conditions are equivalent.
\begin{itemize}
\item[$(\mr{i})$] The local ring $E[\x]_{\gtn_a}/(\gta E[\x]_{\gtn_a})$ is regular of dimension $e$.
\item[$(\mr{ii})$] There exist polynomials $f_1,\ldots,f_{n-e}\in\gta$ and a $K$-Zariski open  neighborhood $U$ of $a$ in $(\ove^\sqbullet)^n$ such that
$$ 
{\rm rk}\Big(\frac{\partial f_i}{\partial\x_j}(a)\Big)_{i=1,\ldots,n-e,\, j=1,\ldots,n}=n-e
$$
and
$$
\ZZ_{\ove^\sqbullet}(\gta)\cap U=\ZZ_{\ove^\sqbullet}(f_1,\ldots,f_{n-e})\cap U.
$$
\item[$(\mr{ii}')$] There exist polynomials $f_1,\ldots,f_{n-e}\in\gta$ and an Euclidean open neighborhood $V$ of $a$ in $(\ove^\sqbullet)^n$ such that
$$ 
{\rm rk}\Big(\frac{\partial f_i}{\partial\x_j}(a)\Big)_{i=1,\ldots,n-e,\, j=1,\ldots,n}=n-e
$$
and
$$
\ZZ_{\ove^\sqbullet}(\gta)\cap V=\ZZ_{\ove^\sqbullet}(f_1,\ldots,f_{n-e})\cap V.
$$
\end{itemize}
\end{thm}
\begin{proof}
The case $e=n$ follows immediately from Lemma \ref{localregular}. Suppose $e<n$.

$(\mr{i})\Longrightarrow(\mr{ii})$ Assume that $E[\x]_{\gtn_a}/(\gta E[\x]_{\gtn_a})$ is a regular local ring of dimension $e$. By Lemma \ref{localregular}, we know that ${\rm rk}_a(\gta)=n-e$ and there exist $f_1,\ldots,f_{n-e}\in\gta$ such that $\gta E[\x]_{\gtn_a}=(f_1,\ldots,f_{n-e})E[\x]_{\gtn_a}=((f_1,\ldots,f_{n-e})K[\x])E[\x]_{\gtn_a}$, so
$$
(\gta E[\x]_{\gtn_a}+\gtN_a^2)/\gtN_a^2=(((f_1,\ldots,f_{n-e})K[\x])E[\x]_{\gtn_a}+\gtN_a^2)/\gtN_a^2.
$$
By Lemma \ref{rk}, we deduce $n-e={\rm rk}_a(\gta)={\rm rk}_a((f_1,\ldots,f_{n-e})K[\x])$, so
$$
{\rm rk}\Big(\frac{\partial f_i}{\partial\x_j}(a)\Big)_{i=1,\ldots,n-e,\, j=1,\ldots,n}=n-e.
$$

Let $\{g_1,\ldots,g_s\}$ be a finite system of generators of $\gta$ in $K[\x]$ (and hence of $\gta E[\x]$ in $E[\x]$). As $\gta E[\x]_{\gtn_a}=(f_1,\ldots,f_{n-e})E[\x]_{\gtn_a}$, by Lemma \ref{lem:222loc}, there exist polynomials $\{k_{ti}\}_{t=1,\ldots,s,\, i=1,\ldots,n-e}\subset K[\x]$ and $h\in K[\x]\setminus\gtn_{K,a}$ such that $g_t=\sum_{i=1}^s\frac{k_{ti}}{h}f_i$ in $K[\x]_{\gtn_{K,a}}$ for all $t\in\{1,\ldots,s\}$. Consequently, $\gta K[\x]_h=(f_1,\ldots,f_{n-e})K[\x]_h$. Thus, the set $U:=(\ove^\sqbullet)^n\,\setminus\,\ZZ_{\ove^\sqbullet}(h)$ is an
$K$-Zariski open neighborhood of $a$ in $(\ove^\sqbullet)^n$ such that
$$
\textstyle
\ZZ_{\ove^\sqbullet}(\gta)\cap U=\ZZ_{\ove^\sqbullet}(f_1,\ldots,f_{n-e})\cap U.
$$

$(\mr{ii})\Longrightarrow(\mr{ii}')$ This implication follows immediately from the fact that the Euclidean topology of $(\ove^\sqbullet)^n$ is finer than the $K$-Zariski topology of $(\ove^\sqbullet)^n$.

$(\mr{ii}')\Longrightarrow(\mr{i})$ Let $f_1,\ldots,f_{n-e}$ and $V$ be as in $(\mr{ii}')$. Write $a:=(a_1,\ldots,a_n)$. Rearranging the indices if necessary, we may assume
$$
\textstyle
\det\big(\frac{\partial f_i}{\partial\x_j}(a)\big)_{i=1,\ldots,n-e,\,j=e+1,\ldots,n}\neq0,
$$
so the polynomial map 
$$
(\ove^\sqbullet)^n\to(\ove^\sqbullet)^n,\ x:=(x_1,\ldots,x_n)\mapsto(x_1-a_1,\ldots,x_e-a_e,f_1(x),\ldots,f_{n-e}(x))
$$
defines by the inverse function theorem a Nash diffeomorphism $\varphi:W\to U_0$ (either complex or real, depending on whether $L$ is algebraically closed or real closed) between a semialgebraic Euclidean open neighborhood $W\subset(\ove^\sqbullet)^n$ of the origin and a semialgebraic Euclidean open neighborhood $U_0$ of $a$ in $(\ove^\sqbullet)^n$ such that $U_0\subset V$. Thus, if $O$ is the origin of $(\ove^\sqbullet)^{n-e}$, then
$$
\varphi(((\ove^\sqbullet)^e\times\{O\})\cap W)=\ZZ_{\ove^\sqbullet}(f_1,\ldots,f_{n-e})\cap U_0=\ZZ_{\ove^\sqbullet}(\gta)\cap U_0
$$
so $e=\dim_{\ove^\sqbullet}(\ZZ_{\ove^\sqbullet}(\gta)\cap U_0)$. Let $X_1,\ldots,X_r$ be the $E$-irreducible components of $\ZZ_{\ove^\sqbullet}(\gta)$ that contains $a$ and let $X:=\bigcup_{i=1}^rX_i\subset(\ove^\sqbullet)^n$. Shrinking $U_0$ if necessary, we may assume that $\ZZ_{\ove^\sqbullet}(\gta)\cap U_0\subset X$, so $e=\dim_{\ove^\sqbullet}(\ZZ_{\ove^\sqbullet}(\gta)\cap U_0)\leq\dim_{\ove^\sqbullet}(X)$. By Theorem \ref{dimension} and \eqref{eisenbud}, we have $\dim_{\ove^\sqbullet}(X)=\dim_E(X)$ and
\begin{align*}
\hgt(\II_E(X))&=\dim(E[\x])-\dim(E[\x]/\II_E(X))=n-\dim_E(X)\\
&=n-\dim_{\ove^\sqbullet}(X)\leq n-\dim_{\ove^\sqbullet}(\ZZ_{\ove^\sqbullet}(\gta)\cap U_0)=n-e.
\end{align*}

Consequently,
\begin{equation}\label{1}
\hgt(\II_E(X))\leq n-e.
\end{equation}
Let $X_{r+1},\ldots,X_m$ be the $E$-irreducible components of $\ZZ_{\ove^\sqbullet}(\gta)$ that do not contain $a$ and let $\gtq_i:=\II_E(X_i)$ for each $i\in\{1,\ldots,m\}$. We have $\gtq_i\not\subset\gtn_a$ for each $i\in\{r+1,\ldots,m\}$, so $\gtq_iE[\x]_{\gtn_a}=E[\x]_{\gtn_a}$ for each $i\in\{r+1,\ldots,m\}$. Thus, by Lemma \ref{lem:a},
\begin{equation}\label{1.5}
\textstyle
\gta E[\x]_{\gtn_a}\subset\II_E(\ZZ_{\ove^\sqbullet}(\gta))E[\x]_{\gtn_a}=\bigcap_{i=1}^r(\gtq_iE[\x]_{\gtn_a})=\II_E(X)E[\x]_{\gtn_a}.
\end{equation}

In addition, $\II_E(X)=\bigcap_{i=1}^r\II_E(X_i)=\bigcap_{i=1}^r\gtq_i$, so
\begin{equation}\label{equa55b}
\textstyle
\hgt(\bigcap_{i=1}^r\gtq_i)=\hgt(\II_E(X))\leq n-e
\end{equation}

By hypothesis, the rank of $\big(\frac{\partial f_i}{\partial\x_j}(a)\big)_{i=1,\ldots,n-e,\, j=1,\ldots,n}$ is $n-e$. Equivalences $(\mr{i})\Longleftrightarrow(\mr{ii})\Longleftrightarrow(\mr{iv})$ of Lemma \ref{localregular} applied to the ideal $(f_1,\ldots,f_{n-e})K[\x]$ of $K[\x]$ assures that the ring $E[\x]_{\gtn_a}/((f_1,\ldots,f_{n-e})E[\x]_{\gtn_a})$ is regular of dimension $e$ and
\begin{equation}\label{2}
n-e=\hgt((f_1,\ldots,f_{n-e})E[\x]_{\gtn_a}).
\end{equation}
As all regular local rings are integral domains \cite[Cor.1, p.302]{zs2}, $(f_1,\ldots,f_{n-e})E[\x]_{\gtn_a}$ is a prime ideal of $E[\x]_{\gtn_a}$. By \eqref{1.5}, we have $(f_1,\ldots,f_{n-e})E[\x]_{\gtn_a}\subset\gta E[\x]_{\gtn_a}\subset\II_E(X)E[\x]_{\gtn_a}=\bigcap_{i=1}^r(\gtq_iE[\x]_{\gtn_a})$,
so
\begin{equation}\label{2.5}
\textstyle
\mr{ht}((f_1,\ldots,f_{n-e})E[\x]_{\gtn_a})\leq\mr{ht}(\gta E[\x]_{\gtn_a})\leq\min_{i\in\{1,\ldots,r\}}\{\mr{ht}(\gtq_iE[\x]_{\gtn_a})\}.
\end{equation}
As $\gtq_i\subset\gtn_a$ for each $i\in\{1,\ldots,r\}$, $\mr{ht}(\gtq_iE[\x]_{\gtn_a})=\mr{ht}(\gtq_i)$ for each $i\in\{1,\ldots,r\}$ by \cite[Cor.3.13]{am}.

Thus, by \eqref{1}, \eqref{equa55b}, \eqref{2} and \eqref{2.5},
\begin{align*}\label{123}
n-e&\textstyle=\hgt((f_1,\ldots,f_{n-e})E[\x]_{\gtn_a})\leq\mr{ht}(\gta E[\x]_{\gtn_a})\leq\min_{i\in\{1,\ldots,r\}}\{\mr{ht}(\gtq_iE[\x]_{\gtn_a})\}\\
&\textstyle=\min_{i\in\{1,\ldots,r\}}\{\mr{ht}(\gtq_i)\}=\mr{ht}(\gtq_1\cap\ldots\cap\gtq_r)=\mr{ht}(\II_E(X))\leq n-e.
\end{align*}
Consequently,
$$
\textstyle
\hgt((f_1,\ldots,f_{n-e})E[\x]_{\gtn_a})=\mr{ht}(\gta E[\x]_{\gtn_a})=n-e.
$$
As $(f_1,\ldots,f_{n-e})E[\x]_{\gtn_a}$ is a prime ideal of $E[\x]_{\gtn_a}$, we deduce $(f_1,\ldots,f_{n-e})E[\x]_{\gtn_a}=\gta E[\x]_{\gtn_a}$ and $A=E[\x]_{\gtn_a}/(\gta E[\x]_{\gtn_a})=E[\x]_{\gtn_a}/((f_1,\ldots,f_{n-e})E[\x]_{\gtn_a})$ is a regular local ring of dimension~$e$, as required.
\end{proof}

\begin{remark}\label{rmk-jc}
A consequence of the proof of implication $(\mr{ii}')\Longrightarrow(\mr{i})$ (setting $c:=n-e$) is that, if there exist polynomials $f_1,\ldots,f_c\in\gta$ such that ${\rm rk}\big(\frac{\partial f_i}{\partial\x_j}(a)\big)_{i=1,\ldots,c,\, j=1,\ldots,n}=c$ and
$\ZZ_{\ove^\sqbullet}(\gta)\cap V=\ZZ_{\ove^\sqbullet}(f_1,\ldots,f_c)\cap V$ for some Euclidean open neighborhood $V$ of $a$ in $(\ove^\sqbullet)^n$, then
$$
\hgt(\gta E[\x]_{\gtn_a})=c
\quad \text{and} \quad
\gta E[\x]_{\gtn_a}=(f_1,\ldots,f_c)E[\x]_{\gtn_a}. \; \text{ $\sqbullet$}
$$
\end{remark}

\subsection{Structure of nonsingular and singular loci}
Recall that we are working with the extension of fields $L|\ove^\sqbullet|E|K$, where $L$ is either algebraically closed or real closed, and $\ove^\sqbullet$ is the algebraic closure of $E$ in $L$. 

\begin{lem}\label{lem:integraldom}
Let $Y\subset(\ove^\sqbullet)^n$ be a $K$-algebraic set and let $a\in Y$. If $\reg^{E|K}_{Y,a}$ is an integral domain, then $a$ belongs only to one of the $K$-irreducible components of $Y\subset(\ove^\sqbullet)^n$.
\end{lem}

\begin{proof}
Let $Y_1,\ldots,Y_s$ be the $K$-irreducible components of $Y\subset(\ove^\sqbullet)^n$. Suppose that $a\in Y_i\cap Y_j$ for some pair of indices $i,j\in\{1,\ldots,s\}$ with $i\neq j$. Let us show that $\reg^{E|K}_{Y,a}$ is not an integral domain.

Set $Y':=\bigcup_{k\in\{1,\ldots,s\}\setminus\{i\}}Y_k$. Pick $b_j\in Y_j\setminus Y_i$, $b_i\in Y_i\setminus Y'$ and $f,g\in K[\x]$ such that $f(b_j)\neq0$, $Y_i\subset\ZZ_{\ove^\sqbullet}(f)$, $g(b_i)\neq0$ and $Y'\subset\ZZ_{\ove^\sqbullet}(g)$. Consider $\alpha,\beta\in\reg^{E|K}_{Y,a}$ given by $\alpha:=f+\II_K(Y)E[\x]_{\gtn_a}$ and $\beta:=g+\II_K(Y)E[\x]_{\gtn_a}$. As $\alpha\beta=0$ in $\reg^{E|K}_{Y,a}$, it is enough to prove that both $\alpha\neq0$ and $\beta\neq0$, that is, $f,g\not\in\II_K(Y)E[\x]_{\gtn_a}$.

Suppose $f\in\II_K(Y)E[\x]_{\gtn_a}$. By Lemma \ref{lem:222loc}, $f\in\II_K(Y)E[\x]_{\gtn_a}\cap K[\x]_{\gtn_{K,a}}= \II_K(Y)K[\x]_{\gtn_{K,a}}$ so there exists $h\in K[\x]$ such that $h(a)\neq0$ and $hf\in\II_K(Y)$. Consequently,
$$
b_j\in Y_j\setminus\ZZ_{\ove^\sqbullet}(f)\subset\ZZ_{\ove^\sqbullet}(h).
$$
As $Y_j\setminus\ZZ_{\ove^\sqbullet}(f)$ is a non-empty $K$-Zariski open subset of the $K$-irreducible $K$-algebraic set $Y_j\subset(\ove^\sqbullet)^n$, it is $K$-Zariski dense in $Y_j$, that is, $Y_j\setminus\ZZ_{\ove^\sqbullet}(f)$ is dense in $Y_j$ with respect to its $K$-Zariski topology. Thus, $a\in Y_j\subset\ZZ_{\ove^\sqbullet}(h)$ so $h(a)=0$, which is a contradiction. This proves that $f\not\in\II_K(Y)E[\x]_{\gtn_a}$. An analogous argument shows (using now that $Y_i\subset(\ove^\sqbullet)^n$ is a $K$-irreducible $K$-algebraic set) that $g\not\in\II_K(Y)E[\x]_{\gtn_a}$, as required.
\end{proof}

The previous result is an improved version of \cite[Lem.2.15]{GS}.

By Corollary \ref{kreliablec} (the $C|K$-Nullstellensatz), all prime ideals $\gtp$ of $K[\x]$ have the zero pro\-perty with respect to~$\ove$ in the sense that $\II_K(\ZZ_\ove(\gtp))=\gtp$. If $L$ is a real closed field, then $E$ inherits the ordering from $L$, $\ove^\sqbullet$ is equal to the real closure $\ove^r$ of $E$, but not all the prime ideals $\gtp$ of $K[\x]$ have the zero property with respect to~$\ove^r$. Thus, it may happen that $\II_K(\ZZ_{\ove^r}(\gtp))$ is strictly larger than $\gtp$ for some prime ideal $\gtp$ of $K[\x]$. For instance, if $\gtp:=(\x_1^2+1)K[\x]$, then $\II_K(\ZZ_{\ove^r}(\gtp))=K[\x]\supsetneqq\gtp$.

The following result provides a characterization of all the prime ideals of $K[\x]$ that have the zero property with respect to $\ove^\sqbullet$, and describes important properties of the $E|K$-local rings of a $K$-irreducible $K$-algebraic subset of $(\ove^\sqbullet)^n$.

\begin{thm}\label{apl1}
Let $\gtp$ be a prime ideal of $K[\x]$ and let $Y:=\ZZ_{\ove^\sqbullet}(\gtp)\subset(\ove^\sqbullet)^n$. Denote $S$ the set of all points $a\in Y$ such that the local ring $E[\x]_{\gtn_a}/(\gtp E[\x]_{\gtn_a})$ is not regular. We have:
\begin{itemize}
\item[$(\mr{i})$] $\II_K(Y)=\gtp$ if and only if $S\subsetneqq Y$, that is, there exists $a\in Y$ such that the local ring $E[\x]_{\gtn_a}/(\gtp E[\x]_{\gtn_a})$ is regular.
\item[$(\mr{ii})$] Assume that $\II_K(Y)=\gtp$, so $Y\neq\varnothing$ and $E[\x]_{\gtn_a}/(\gtp E[\x]_{\gtn_a})=\reg^{E|K}_{Y,a}$ for each $a\in Y$. Define $c:=\max_{a\in Y}\{{\rm rk}_a(\gtp)\}$. Then we have:
 \begin{itemize}
\item[$\bullet$] $\dim(\reg^{E|K}_{Y,a})=\dim(Y)$ for each $a\in Y$,
\item[$\bullet$] $\dim(Y)=n-c$,
\item[$\bullet$] $Y\subset(\ove^\sqbullet)^n$ is a $K$-irreducible algebraic set and $S$ is a $K$-algebraic subset of $Y$ that we can describe as $S=\{a\in Y:{\rm rk}_a(\gtp)<c\}$ and $\dim(S)<\dim(Y)$.
\end{itemize}
\end{itemize}
\end{thm}
\begin{proof} 

\emph{Assume $\II_K(Y)=\gtp$}. As $\gtp\subset K[\x]$ is prime, $Y\subset(\ove^\sqbullet)^n$ is non-empty and $K$-irreducible by Lemma \ref{lem:prime}. \emph{We claim that $\hgt(\gtp E[\x]_{\gtn_a})=\hgt(\gtp)$ for each $a\in Y$}.

Let $\ove$ be the algebraic closure of $E$ and let $T:=\zcl_{\ove^n}^K(Y)\subset\ove^n$. Evidently, $T=Y$ if $\ove^\sqbullet=\ove$. Observe that $\II_K(T)=\II_K(Y)=\gtp$, so $T\subset\ove^n$ is $K$-irreducible. Let $\gtp_1,\ldots,\gtp_t\subset E[\x]$ be the minimal prime ideals associated to $\gtp E[\x]$, and let $W_i:=\ZZ_\ove(\gtp_i)\subset\ove^n$ for each $i\in\{1,\ldots,t\}$. Lemma \ref{lem:238} implies that $\II_E(T)=\gtp E[\x]$, $W_1,\ldots,W_t$ are the $E$-irreducible components of $T\subset\ove^n$ and, for each $i\in\{1,\ldots,t\}$, $\II_E(W_i)=\gtp_i$ and $\hgt(\gtp_i)=\hgt(\gtp)$. Let $a\in Y\subset T$. Rearranging the indices if necessary, we can assume that $W_1,\ldots,W_u$ are the $E$-irreducible components of $T\subset\ove^n$ that contains $a$, for some $u\in\{1,\ldots,t\}$. Thus, $\gtp_i\subset\gtn_a$ if $i\in\{1,\ldots,u\}$ and $\gtp_i\not\subset\gtn_a$ otherwise. Define $\gtq_i:=\gtp_iE[\x]_{\gtn_a}$ for each $i\in\{1,\ldots,u\}$. By Lemma \ref{lem:a} and \cite[Cor.3.13]{am}, we have:
\begin{itemize}
\item $\gtp E[\x]_{\gtn_a}=(\gtp E[\x])E[\x]_{\gtn_a}=(\bigcap_{i=1}^t\gtp_i)E[\x]_{\gtn_a}=\bigcap_{i=1}^t(\gtp_iE[\x]_{\gtn_a})=\bigcap_{i=1}^u\gtq_i$,
\item $\gtq_1,\ldots,\gtq_u\subset E[\x]_{\gtn_a}$ are the minimal prime ideals associated to $\gtp E[\x]_{\gtn_a}$,
\item $\hgt(\gtq_i)=\hgt(\gtp_i)=\hgt(\gtp)$ for each $i\in\{1,\ldots,u\}$.
\end{itemize}
We deduce that $\hgt(\gtp E[\x]_{\gtn_a})=\min_{i\in\{1,\ldots,u\}}\{\hgt(\gtq_i)\}=\hgt(\gtp)$, as claimed.

Thus, $\dim(\reg^{E|K}_{Y,a})=\dim(Y)$ for each $a \in Y$. Indeed, by \eqref{eq:ht}, we have:
$$
\dim(\reg^{E|K}_{Y,a})=n-\hgt(\gtp E[\x]_{\gtn_a})=n-\hgt(\gtp)=\dim_K(Y)=\dim(Y).
$$

Let $\{g_1,\ldots,g_s\}$ be a system of generators of $\gtp$ in $K[\x]$. As $\gtp\subset K[\x]$ is prime and $K$ has characteristic zero, the rank $r$ of the matrix $\big(\frac{\partial g_i}{\partial\x_j}+\gtp\big)_{i=1,\ldots,s,\,j=1,\ldots,n}$ with coefficients in the field of fractions of $K[\x]/\gtp$ is by \cite[Chap.10, \S 14, Thm.1]{hp} equal to $n-\dim_K(Y)=n-\dim(Y)$. Observe that ${\rm rk}_a(\gtp)\leq r$ for each $a\in Y$ and there exist points $a\in Y$ such that ${\rm rk}_a(\gtp)=r$. We deduce that $c=r$, so $c=n-\dim(Y)$ and $\dim(\reg^{E|K}_{Y,a})=\dim(Y)=n-c$ for each $a\in Y$.

Lemma \ref{localregular} assures that $\reg^{E|K}_{Y,a}$ is not regular if and only if $n-c=\dim(\reg^{E|K}_{Y,a})<n-{\rm rk}_a(\gtp)$ or, equivalently, ${\rm rk}_a(\gtp)<c$. Thus, $S$ is a proper $K$-algebraic subset of $Y$ that coincides with the set $\{a\in Y:{\rm rk}_a(\gtp)<c\}$. By Lemma \ref{dimirred}, $\dim(S)<\dim(Y)$, so $S\subsetneq Y$.

If ($L$ is algebraically closed and hence) $\ove^\sqbullet=\ove$, then $\II_K(Y)=\gtp$, so in the following we assume: ($L$ is real closed,) {\em $\ove^\sqbullet=\ove^r$ and there exists $a\in Y=\ZZ_{\ove^r}(\gtp)\subset(\ove^r)^n$ such that $E[\x]_{\gtn_a}/(\gtp E[\x]_{\gtn_a})$ is a regular local ring, that is, $a\in Y\setminus S$.} 

Let $\ell:={\rm rk}_a(\gtp)$. As the local ring $E[\x]_{\gtn_a}/(\gtp E[\x]_{\gtn_a})$ is regular, by implication $(\mr{i})\Longrightarrow(\mr{ii})$ of Lemma \ref{localregular}, there exist $p_1,\ldots,p_\ell\in\gtp$ such that $\gtp E[\x]_{\gtn_a}=(p_1,\ldots,p_\ell)E[\x]_{\gtn_a}$. Complete $\{p_1,\ldots,p_\ell\}$ to a system of generators $\{p_1,\ldots,p_\ell,p_{\ell+1},\ldots,p_m\}$ of $\gtp$ in $K[\x]$, and define $J:=\{\ell+1,\ldots,m\}$ if $m>\ell$ and $J:=\varnothing$ if $m=\ell$. By Lemma \ref{lem:222loc}, $p_1,\ldots,p_\ell$ generate $\gtp E[\x]_{\gtn_a}\cap K[\x]_{\gtn_{K,a}}=\gtp K[\x]_{\gtn_{K,a}}$. Observe that $p_{\ell+1},\ldots,p_m\in \gtp K[\x]_{\gtn_{K,a}}$ if $m>\ell$. Thus, for each $k\in J$ and $i\in\{1,\ldots,\ell\}$, there exist $h_k,h_{ik}\in K[\x]$ such that $h_k(a)\neq0$ and
\begin{equation}\label{generators}
h_kp_k=\sum_{i=1}^\ell h_{ik}p_i.
\end{equation}

Let $\kr\subset\ove^r$ be the real closure of $K$. By Tarski-Seidenberg's principle, there exists a point $b\in\krn$ such that $p_1(b)=0,\ldots,p_m(b)=0$, ${\rm rk}_b(\gtp)=\ell$ and $h_k(b)\neq0$ for each $k\in J$. By \eqref{generators}, $\{p_1,\ldots,p_\ell\}$ is a system of generators of the ideal $\gtp K[\x]_{\gtn_b}$ of the ring $K[\x]_{\gtn_b}$. By implication $(\mr{ii})\Longrightarrow(\mr{i})$ of Lemma \ref{localregular}, $K[\x]_{\gtn_b}/(\gtp K[\x]_{\gtn_b})$ is a regular local ring of dimension $n-\ell$ and $\hgt(\gtp K[\x]_{\gtn_b})=\ell$. As $Y\cap\krn=\ZZ_\kr(p_1,\ldots,p_m)$, \eqref{generators} holds and $\ell={\rm rk}_b(\gtp)$, we can apply equivalence $(\mr{i})\Longleftrightarrow(\mr{iii})$ of \cite[Prop.3.3.10]{bcr}, obtaining that $b$ is a nonsingular point of $Y\cap(\ol{K}^r)^n$ of dimension $n-\ell$. Thus, by \cite[Prop.3.3.11]{bcr}, there exists a Nash submanifold $M$ of $(\ol{K}^r)^n$ of dimension $n-\ell$ such that $b\in M\subset Y\cap(\ol{K}^r)^n$. Corollary \ref{inter}$(\mr{iv})$ implies
$$
n-\ell=\dim(M)\leq\dim_{\kr}(Y\cap(\kr)^n)=\dim(Y)=\dim_K(Y)=n-\hgt(\II_K(Y))
$$
and consequently
$$
\hgt(\II_K(Y))\leq \ell=\hgt(\gtp K[\x]_{\gtn_b})=\hgt(\gtp).
$$
The latter equality follows from \cite[Cor.3.13]{am}, because $\gtp$ is a prime ideal of $K[\x]$ contained in $\gtn_b$. As $\gtp\subset\II_K(Y)$, we conclude $\mr{ht}(\gtp)=\mr{ht}(\II_K(Y))$. As $\gtp\subset K[\x]$ is prime, $\II_K(Y)=\gtp$, as required.
\end{proof}

We apply the previous results to study the concepts of nonsingular and singular points introduced in Definition \ref{E|K-regular}.

\begin{prop}\label{prop:XY}
Let $X\subset L^n$ be a $K$-algebraic set and let $X_1,\ldots,X_s$ be the $K$-irreducible components of $X$. Denote $Y:=X\cap(\ove^\sqbullet)^n$ and $Y_i:=X_i\cap(\ove^\sqbullet)^n$ for each $i\in\{1,\ldots,s\}$. We have:
\begin{itemize}
\item[$(\mr{i})$] $Y\subset(\ove^\sqbullet)^n$ is a $K$-algebraic set, $X=Y_L=\zcl_{L^n}(Y)=\zcl_{L^n}^{\ove^\sqbullet}(Y)$, $\II_{\ove^\sqbullet}(Y)=\II_{\ove^\sqbullet}(X)$ and $\dim(Y)=\dim(X)$. In particular, $X=\zcl_{L^n}^K(Y)$ and $\II_K(Y)=\II_K(X)$.
\item[$(\mr{ii})$] $Y_1,\ldots,Y_s$ are the $K$-irreducible components of $Y\subset(\ove^\sqbullet)^n$. In particular, $X\subset L^n$ is $K$-irreducible if and only if so is $Y\subset(\ove^\sqbullet)^n$. In addition, $X_i=(Y_i)_L=\zcl_{L^n}(Y_i)=\zcl_{L^n}^{\ove^\sqbullet}(Y_i)=\zcl_{L^n}^K(Y_i)$ and $\dim(Y_i)=\dim(X_i)$ for each $i\in\{1,\ldots,s\}$.
\item[$(\mr{iii})$] Let $a\in Y$ and let $I_a$ be the subset of $\{1,\ldots,s\}$ of all indices $i$ such that $a\in Y_i$. Then $\reg^{E|K}_{Y,a}=\reg^{E|K}_{X,a}$ and $\dim(\reg^{E|K}_{Y,a})=\max_{i\in I_a}\{\dim(Y_i)\}$. In particular, $\dim(X)=\max_{a\in Y}\{\dim(\reg^{E|K}_{X,a})\}$.
\item[$(\mr{iv})$] Let $a\in Y$ be such that $\reg^{E|K}_{X,a}$ is an integral domain. Then there exists an index $i\in\{1,\ldots,s\}$ such that $a\in X_i\setminus\bigcup_{j\in\{1,\ldots,s\}\setminus\{i\}}X_j$.
\item[$(\mr{v})$] A point $a\in Y$ is $E|K$-nonsingular of some dimension $e$ if and only if $a$ is a $E|K$-nonsingular point of $X$ of the same dimension $e$. Thus, $\Reg^{E|K}(Y,e)=\Reg^{E|K}(X,e)$ for each $e\in\N$, $\Reg^{E|K}(Y)=\Reg^{E|K}(X)$ and $\mr{Sing}^{E|K}(Y)=\mr{Sing}^{E|K}(X)$.
\end{itemize}
\end{prop}
\begin{proof}
$(\mr{i})$ Corollary \ref{inter}$(\mr{i})(\mr{ii})(\mr{iv})$ implies that $X=\zcl_{L^n}(Y)=\zcl_{L^n}^{\ove^\sqbullet}(Y)$, $\II_{\ove^\sqbullet}(Y)=\II_{\ove^\sqbullet}(X)$ and $\dim(X)=\dim_L(X)=\dim_{\ove^\sqbullet}(Y)=\dim(Y)$. As $K\subset\ove^\sqbullet$ and $X\subset L^n$ is $K$-algebraic, we have $X=\zcl_{L^n}^{\ove^\sqbullet}(Y)\subset\zcl_{L^n}^K(Y)\subset X$, so $X=\zcl_{L^n}^K(Y)$ and $\II_K(Y)=\II_K(X)$.

$(\mr{ii})$ Fix $i\in\{1,\ldots,s\}$. By $(\mr{i})$, we have $X_i=(Y_i)_L=\zcl_{L^n}(Y_i)=\zcl_{L^n}^{\ove^\sqbullet}(Y_i)=\zcl_{L^n}^K(Y_i)$, $\dim(Y_i)=\dim(X_i)$ and $\II_K(Y_i)=\II_K(X_i)$ is a prime ideal of $K[\x]$. By Lemma \ref{lem:prime}, $Y_i\subset(\ove^\sqbullet)^n$ is $K$-irreducible. Choose $j\in\{1,\ldots,s\}\setminus\{i\}$. Then $Y_i\not\subset Y_j$ because otherwise, $Y_i\subset Y_j$ and $X_i=\zcl_{L^n}(Y_i)\subset\zcl_{L^n}(Y_j)=X_j$, which is a contradiction. As $Y=\bigcup_{i=1}^sY_i$, the $Y_i$ are the $K$-irreducible components of $Y\subset(\ove^\sqbullet)^n$.

$(\mr{iii})$ Fix $a\in Y$. As $\II_K(Y)=\II_K(X)$, it follows: $\reg^{E|K}_{Y,a}=\reg^{E|K}_{X,a}$. If $Y\subset(\ove^\sqbullet)^n$ is $K$-irreducible, Theorem \ref{apl1}$(\mr{ii})$ assures that $\dim(\reg^{E|K}_{Y,a})=\dim(Y)$. 

Suppose $Y\subset(\ove^\sqbullet)^n$ is $K$-reducible with $K$-irreducible components $Y_1,\ldots,Y_s$. As $\II_K(Y)=\bigcap_{i=1}^s\II_K(Y_i)$, Lemma \ref{lem:a} implies
$$\textstyle
\II_K(Y)E[\x]_{\gtn_a}=\bigcap_{i=1}^s(\II_K(Y_i)E[\x]_{\gtn_a}).
$$
As $\II_K(Y_i)E[\x]_{\gtn_a}=E[\x]_{\gtn_a}$ for each $i\in\{1,\ldots,s\}\setminus I_a$, we have 
$$\textstyle
\II_K(Y)E[\x]_{\gtn_a}=\bigcap_{i\in I_a}(\II_K(Y_i)E[\x]_{\gtn_a}).
$$
We deduce: $\mr{ht}(\II_K(Y)E[\x]_{\gtn_a})=\min_{i\in I_a}\{\mr{ht}(\II_K(Y_i)E[\x]_{\gtn_a})\}$. By \eqref{eq:ht}, it follows that
\begin{align*}
\dim(\reg^{E|K}_{Y,a})&\textstyle=n-\mr{ht}(\II_K(Y)E[\x]_{\gtn_a})=\max_{i\in I_a}\{n-\mr{ht}(\II_K(Y_i)E[\x]_{\gtn_a})\}\\
&\textstyle=\max_{i\in I_a}\{\dim(\reg^{E|K}_{Y_i,a})\}=\max_{i\in I_a}\{\dim(Y_i)\}.
\end{align*}

$(\mr{iv})$ This follows from $(\mr{ii})$, equality $\reg^{E|K}_{Y,a}=\reg^{E|K}_{X,a}$ and Lemma \ref{lem:integraldom}.

$(\mr{v})$ This item is straightforward, because $\reg^{E|K}_{X,a}=\reg^{E|K}_{Y,a}$ for each $a\in Y$ and $\dim(X)=\dim(Y)$.
\end{proof}

Next three results deal with the structure of $E|K$-nonsingular loci of $K$-algebraic subsets of $L^n$.

\begin{prop}\label{29}
Let $X\subset L^n$ be a $K$-irreducible $K$-algebraic set of dimension $d$. We have:
\begin{itemize}
\item[$(\mr{i})$] If $a\in X\cap(\ove^\sqbullet)^n$ is a $E|K$-nonsingular point of $X$ of some dimension $e$, then $e=d$.
\item[$(\mr{ii})$] $\Sing^{E|K}(X)$ is a $K$-algebraic subset of $(\ove^\sqbullet)^n$ of dimension $<d$, so $\Reg^{E|K}(X)$ is a non-empty $K$-Zariski open subset of $X\cap(\ove^\sqbullet)^n\subset(\ove^\sqbullet)^n$. If $\{g_1,\ldots,g_s\}$ is a system of generators of $\II_K(X)$ in $K[\x]$, then
\begin{equation}\label{regge}
\textstyle
\Reg^{E|K}(X)=\big\{a\in X\cap(\ove^\sqbullet)^n: {\rm rk}\big(\frac{\partial g_i}{\partial\x_j}(a)\big)_{i=1,\ldots,s,\, j=1\ldots,n}=n-d\big\}.
\end{equation}
and
\begin{equation}\label{singe}
\textstyle
\Sing^{E|K}(X)=\big\{a\in X\cap(\ove^\sqbullet)^n: {\rm rk}\big(\frac{\partial g_i}{\partial\x_j}(a)\big)_{i=1,\ldots,s,\, j=1\ldots,n}<n-d\big\}.
\end{equation}

\item[$(\mr{iii})$] If $d<n$, then a point $a\in X\cap(\ove^\sqbullet)^n$ belongs to $\Reg^{E|K}(X)$ if and only if there exist $f_1,\ldots,f_{n-d}\in\II_K(X)$ and a $K$-Zariski open neighborhood $U$ of $a$ in $(\ove^\sqbullet)^n$ such that
$$\textstyle
\text{${\rm rk}\big(\frac{\partial f_i}{\partial\x_j}(a)\big)_{i=1,\ldots,n-d,\, j=1,\ldots,n}=n-d\;$ and $\;X\cap U=\ZZ_{\ove^\sqbullet}(f_1,\ldots,f_{n-d})\cap U$}.
$$
\item[$(\mr{iii}')$] If $d<n$, then a point $a\in X\cap(\ove^\sqbullet)^n$ belongs to $\Reg^{E|K}(X)$ if and only if there exist $f_1,\ldots,f_{n-d}\in\II_K(X)$ and an Euclidean open neighborhood $V$ of $a$ in $(\ove^\sqbullet)^n$ such that
 $$\textstyle
\text{${\rm rk}\big(\frac{\partial f_i}{\partial\x_j}(a)\big)_{i=1,\ldots,n-d,\, j=1,\ldots,n}=n-d\;$ and $\;X\cap V=\ZZ_{\ove^\sqbullet}(f_1,\ldots,f_{n-d})\cap V$.}
$$
\end{itemize}
\end{prop}
\begin{proof}
By Lemma \ref{lem:prime}, the ideal $\gtp:=\II_K(X)$ of $K[\x]$ is prime. Thus, Proposition \ref{prop:XY}$(\mr{v})$ and Theorem \ref{apl1}$(\mr{ii})$ assure that $\max_{a\in X\cap(\ove^\sqbullet)^n}\{{\rm rk}_a(\gtp)\}=n-d$ and a point $a\in X\cap(\ove^\sqbullet)^n$ is a $E|K$-nonsingular point of $X$ of some dimension $e$ if and only if $e=d$ and ${\rm rk}_a(\gtp)=n-d$. In addition, $\Sing^{E|K}(X)=\{a\in X\cap(\ove^\sqbullet)^n: {\rm rk}_a(\gtp)<n-d\}$ is a $K$-algebraic subset of $X\cap(\ove^\sqbullet)^n$ and has of dimension $<d$ by Lemma \ref{dimirred}. This proves $(\mr{i})$ and $(\mr{ii})$. Items $(\mr{iii})$ and $(\mr{iii}')$ are immediate consequences of the $E|K$-Jacobian criterion Theorem \ref{Q-jacobian}.
\end{proof}

\begin{remark}\label{rem526}
We restate Proposition \ref{29}$(\mr{ii})$ in terms of $E|K$-Zariski tangent spaces. Indeed, $\dim_{E[a]}(T^{E|K}_a(X))=n-{\rm rk}\big(\frac{\partial g_i}{\partial\x_j}(a)\big)_{i=1,\ldots,s,\, j=1\ldots,n}$ for each $a\in X\cap(\ove^\sqbullet)^n$, so \eqref{regge} is equivalent to the equality
$$
\Reg^{E|K}(X)=\big\{a\in X\cap(\ove^\sqbullet)^n:\dim_{E[a]}(T^{E|K}_a(X))=\dim(X)\big\}
$$
and \eqref{singe} to the equality
$$
\Sing^{E|K}(X)=\big\{a\in X\cap(\ove^\sqbullet)^n:\dim_{E[a]}(T^{E|K}_a(X))>\dim(X)\big\}.\;\,\sqbullet
$$
\end{remark}

\begin{prop}\label{30}
Let $X\subset L^n$ be a $K$-algebraic set, let $X_1,\ldots,X_s$ be the $K$-irreducible components of $X$, let $a\in X\cap(\ove^\sqbullet)^n$ and let $e\in\{0,\ldots,n\}$. If $e=n$, then $a$ is a $E|K$-nonsingular point of $X$ of dimension $e$ if and only if $X=L^n$. If $e<n$, then the following conditions are equivalent.
\begin{itemize}
\item[$(\mr{i})$] $a$ is a $E|K$-nonsingular point of $X$ of dimension $e$.
\item[$(\mr{ii})$] There exists an index $i\in\{1,\ldots,s\}$ such that $a\in X_i\setminus\bigcup_{j\in\{1,\ldots,s\}\setminus\{i\}}X_j$, $a\in\Reg^{E|K}(X_i)$ and $\dim(X_i)=e$.
\item[$(\mr{iii})$] There exist $f_1,\ldots,f_{n-e}\in\II_K(X)$ and a $K$-Zariski open neighborhood $U$ of $a$ in $(\ove^\sqbullet)^n$ such that ${\rm rk}\big(\frac{\partial f_i}{\partial\x_j}(a)\big)_{i=1,\ldots,n-e,\, j=1,\ldots,n}=n-e$ and $X\cap U=\ZZ_{\ove^\sqbullet}(f_1,\ldots,f_{n-e})\cap U$.
\item[$(\mr{iii}')$] There exist $f_1,\ldots,f_{n-e}\in\II_K(X)$ and an Euclidean open neighborhood $V$ of $a$ in $(\ove^\sqbullet)^n$ such that ${\rm rk}\big(\frac{\partial f_i}{\partial\x_j}(a)\big)_{i=1,\ldots,n-e,\, j=1,\ldots,n}=n-e$ and $X\cap V=\ZZ_{\ove^\sqbullet}(f_1,\ldots,f_{n-e})\cap V$.
\end{itemize}
\end{prop}
\begin{proof}
If $e=n$, the result follows from the first part of Theorem \ref{Q-jacobian}. Suppose $e<n$.

$(\mr{i})\Longrightarrow(\mr{ii})$ If $\reg^{E|K}_{X,a}$ is a regular local ring of dimension $e$, then $\reg^{E|K}_{X,a}$ is an integral domain (\cite[Cor.1, p.302]{zs2}), so $a$ belongs to only one $K$-irreducible component $X_i$ of $X\subset L^n$ by Proposition \ref{prop:XY}$(\mr{iv})$. As $\reg^{E|K}_{X,a}\cong\reg^{E|K}_{X_i,a}$, $a$ is a $E|K$-nonsingular point of $X_i\subset L^n$ of dimension~$e$. By Proposition \ref{29}$(\mr{i})$, it holds: $e=\dim(X_i)$, so $a\in\Reg^{E|K}(X_i)$.

$(\mr{ii})\Longrightarrow(\mr{iii})$ By Proposition \ref{29}$(\mr{iii})$ applied to $X_i\subset L^n$, there exist $g_1,\ldots,g_{n-e}\in\II_K(X_i)$ and a $K$-Zariski open neighborhood $U$ of $a$ in $(\ove^\sqbullet)^n$ such that ${\rm rk}\big(\frac{\partial g_i}{\partial\x_j}(a)\big)_{i=1,\ldots,n-e,\, j=1,\ldots,n}=n-e$ and $X_i\cap U=\ZZ_{\ove^\sqbullet}(g_1,\ldots,g_{n-e})\cap U$. Set $Z:=\bigcup_{j\in\{1,\ldots,s\}\setminus\{i\}}X_j$ and choose a polynomial $h\in K[\x]$ such that $a\not\in\ZZ_{\ove^\sqbullet}(h)$ and $Z\cap(\ove^\sqbullet)^n\subset\ZZ_{\ove^\sqbullet}(h)$. Substituting $U$ by $U\setminus\ZZ_{\ove^\sqbullet}(h)$, we can assume that $U\cap Z=\varnothing$. To finish, it is enough to define $f_i:=g_ih$ for each $i\in\{1,\ldots,n-e\}$.

Implication $(\mr{iii})\Longrightarrow(\mr{iii}')$ follows immediately from the fact that the Euclidean topology is finer than the $K$-Zariski topology of $(\ove^\sqbullet)^n$. Implication $(\mr{iii}')\Longrightarrow(\mr{i})$ is a straightforward consequence of Theorem \ref{Q-jacobian}.
\end{proof}

\begin{remarks}\label{rem528}
$(\mr{i})$ Let $X\subset L^n$ be a $K$-algebraic set, let $a\in X\cap(\ove^\sqbullet)^n$ and let $e\in\{0,\ldots,n-1\}$. Suppose that there exist $f_1,\ldots,f_{n-e}\in\II_K(X)$ and an Euclidean open neighborhood $V$ of $a$ in $(\ove^\sqbullet)^n$ such that ${\rm rk}\big(\frac{\partial f_i}{\partial\x_j}(a)\big)_{i=1,\ldots,n-e,\, j=1,\ldots,n}=n-e$ and $X\cap V=\ZZ_{\ove^\sqbullet}(f_1,\ldots,f_{n-e})\cap V$. By Proposition \ref{30}, Remark \ref{rem526} and the $K$-Zariski local nature in $R^n$ of the $E|K$-Zariski tangent space $T^{E|K}_a(X)$ (see Definition \ref{ek-Zar-tang} and Remarks \ref{tang-bcr}$(\mr{i}')$), we have:
$$
T^{E|K}_a(X)=\{v\in E[a]^n:\langle\nabla f_1(a),v\rangle=0,\ldots,\langle\nabla f_{n-e}(a),v\rangle=0\}.
$$
Indeed, if $X_1,\ldots,X_s$ are the $K$-irreducible components of $X\subset L^n$, then there exists a unique index $i\in\{1,\ldots,s\}$ such that $a\in\Reg^{E|K}(X_i)\setminus\bigcup_{j\in\{1,\ldots,s\}\setminus\{i\}}X_j$ and $\dim(X_i)=e$. Thus, $T^{E|K}_a(X)=T^{E|K}_a(X_i)$ is contained in $T^*:=\{v\in E[a]^n:\langle\nabla f_1(a),v\rangle=0,\ldots,\langle\nabla f_{n-e}(a),v\rangle=0\}$ and both $T^{E|K}_a(X)$ and $T^*$ have $E[a]$-vector dimension equal to $e$, so they coincide.

$(\mr{ii})$ If $X\subset L^n$ is a $K$-algebraic set and $a$ is a $E|K$-nonsingular point of $X$ of some dimension $e$, then $\dim_{E[a]}(T^{E|K}_a(X))=e$. This follows immediately from Proposition \ref{30}$(\mr{ii})$, the $K$-Zariski local nature of $T^{E|K}_a(X)$ (see Remarks \ref{tang-bcr}$(\mr{i}')$) and Remark \ref{rem526}.
 $\sqbullet$ 
\end{remarks}

\begin{cor}\label{30'}
Let $X\subset L^n$ be a $K$-algebraic set of dimension $d<n$, let $X_1,\ldots,X_s$ be the $K$-irre\-ducible components of $X$ and, for each $e\in\{0,\ldots,d\}$, let $I_e$ be the set of all indices $i\in\{1,\ldots,s\}$ such that $\dim(X_i)=e$. Set $Y_i:=X_i\cap(\ove^\sqbullet)^n$ for each $i\in\{1,\ldots,s\}$. We have:
\begin{itemize}
\item[$(\mr{i})$] $\Reg^{E|K}(X_i)\not\subset\bigcup_{j\in\{1,\ldots,s\}\setminus\{i\}}X_j\,$ for each $i\in\{1,\ldots,s\}$.
\item[$(\mr{ii})$] $\Reg^{E|K}(X,e)=\bigsqcup_{i\in I_e}\!\big(\Reg^{E|K}(X_i)\setminus\bigcup_{j\in\{1,\ldots,s\}\setminus\{i\}}X_j\big)$ for all $e\in\{0,\ldots,d\}$. In particular, $\Reg^{E|K}(X,e)\neq\varnothing$ if and only if $I_e\neq\varnothing$. 
\item[$(\mr{iii})$] $\mr{Sing}^{E|K}(X)$ is a $K$-algebraic subset of $(\ove^\sqbullet)^n$ of dimension $<d$, and $\Reg^{E|K}(X)$ is a non-empty $K$-Zariski open subset of $X\cap(\ove^\sqbullet)^n\subset(\ove^\sqbullet)^n$ such that
\begin{equation}\label{eq:regd}
\textstyle
\Reg^{E|K}(X)=\bigsqcup_{i\in I_d}\!\big(\Reg^{E|K}(X_i)\setminus\bigcup_{j\in\{1,\ldots,s\}\setminus\{i\}}X_j\big).
\end{equation}
In addition, it holds
\begin{equation}\label{eq:singd}
\textstyle
\Sing^{E|K}(X)=\bigcup_{i\in I_d}\Sing^{E|K}(X_i)\cup\bigcup_{i,j\in I_d,i\neq j}(Y_i\cap Y_j)\cup\bigcup_{i\in\{1,\ldots,s\}\setminus I_d}Y_i.
\end{equation}
\end{itemize}
\end{cor}
\begin{proof}
$(\mr{i})$ Suppose that $\Reg^{E|K}(X_i)\subset\bigcup_{j\in\{1,\ldots,s\}\setminus\{i\}}X_j$ for some $i\in\{1,\ldots,s\}$. By Propositions \ref{prop:XY}$(\mr{ii})$ and \ref{29}$(\mr{ii})$, $\Reg^{E|K}(X_i)$ is a non-empty $K$-Zariski open subset of the $K$-irreducible $K$-algebraic set $Y_i\subset(\ove^\sqbullet)^n$, and $\zcl_{L^n}^K(Y_i)=X_i$. As $\Reg^{E|K}(X_i)$ is $K$-Zariski dense in $Y_i\subset(\ove^\sqbullet)^n$, we deduce $X_i=\zcl_{L^n}^K(\Reg^{E|K}(X_i))\subset\bigcup_{j\in\{1,\ldots,s\}\setminus\{i\}}X_j$, which is a contradiction.

$(\mr{ii})$ This follows immediately from $(\mr{i})$ and equivalence $(\mr{i})\Longleftrightarrow(\mr{ii})$ of Proposition \ref{30}.

$(\mr{iii})$ Applying $(\mr{ii})$ with $e=d$, we obtain formulas \eqref{eq:regd}. By Proposition \ref{29}$(\mr{ii})$, the set $\Sing^{E|K}(X_i)\subset(\ove^\sqbullet)^n$ is $K$-algebraic and of dimension $<d$ for each $i\in I_d$. By Proposition \ref{prop:XY}$(\mr{i})(\mr{ii})$ and Lemma \ref{dimirred}, each set $Y_i\subset(\ove^\sqbullet)^n$ is $K$-algebraic, $\dim(Y_i\cap Y_j)<d$ for each $i,j\in I_d$ with $i\neq j$, and $\dim(Y_i)<d$ for each $i\in\{1,\ldots,s\}\setminus I_d$. Thus, keeping in mind that $\Sing^{E|K}(X)=(X\cap(\ove^\sqbullet)^n)\setminus\Reg^{E|K}(X)$, we deduce formula \eqref{eq:singd}.

By \eqref{eq:singd}, the set $\Sing^{E|K}(X)\subset(\ove^\sqbullet)^n$ is also $K$-algebraic and of dimension $<d$. In particular, $\Sing^{E|K}(X)$ is a proper subset of $X\cap(\ove^\sqbullet)^n$, so $\Reg^{E|K}(X)$ is a non-empty $K$-Zariski open subset of $X\cap(\ove^\sqbullet)^n\subset(\ove^\sqbullet)^n$, as required.
\end{proof}

The next result concerns the $K$-algebraicity of the difference between two appropriate $K$-algebraic sets, and generalizes both \cite[Prop.3.3.17]{bcr} and \cite[Prop.2.14]{GS}.

\begin{cor}\label{K-difference}
Let $X\subset L^n$ be a $K$-algebraic set of dimension $d$, and let $Z\subset L^n$ be a $K$-algebraic set such that $Z\subset X$. Suppose that one of the following two conditions is satisfied:
\begin{itemize}
\item[$(\mr{i})$] $\dim(\reg^{E|K}_{Z,a})=d$ and $\reg^{E|K}_{X,a}$ is an integral domain for each $a\in Z\cap(\ove^\sqbullet)^n$.
\item[$(\mr{ii})$] $\dim(Z)=d$ and $Z\cap(\ove^\sqbullet)^n=\Reg^{E|K}(Z)\subset\Reg^{E|K}(X)$.
\end{itemize} 
Then $X\setminus Z\subset L^n$ is a $K$-algebraic set. In addition, if $Z\cap(\ove^\sqbullet)^n\neq\Reg^{E|K}(X)$, then we have $\Reg^{E|K}(X\setminus Z)=\Reg^{E|K}(X)\setminus Z$.
\end{cor}
\begin{proof}
We show first that condition $(\mr{ii})$ implies condition $(\mr{i})$. Fix $a\in Z\cap(\ove^\sqbullet)^n$. By $(\mr{ii})$, $a$ belongs to both $\Reg^{E|K}(Z,d)$ and $\Reg^{E|K}(X,d)$, so $\reg^{E|K}_{Z,a}$ and $\reg^{E|K}_{X,a}$ are regular local rings of dimension $d$. In particular, $\dim(\reg^{E|K}_{Z,a})=d$ and $\reg^{E|K}_{X,a}$ is an integral domain, because so are all regular local rings (see \cite[Cor.1, p.302]{zs2}).

As $(\mr{ii})\Longrightarrow(\mr{i})$, we assume that $(\mr{i})$ is satisfied and prove that $X\setminus Z\subset L^n$ is $K$-algebraic.

If $Z=\varnothing$, the result is clear, so we suppose $Z\neq\varnothing$. Let $X_1,\ldots,X_s$ be the $K$-irreducible components of $X\subset L^n$. Define the $K$-algebraic set $Z_i\subset L^n$ as $Z_i:=X_i\cap Z$ for each $i\in\{1,\ldots,s\}$ and denote $I$ the subset of $\{1,\ldots,s\}$ constituted by all indices $i$ such that $Z_i\neq\varnothing$. Fix $i\in I$.

Let us prove: $Z_i\cap X_j=\varnothing$ for each $j\in\{1,\ldots,s\}\setminus\{i\}$. Otherwise, $Z_i\cap X_j\neq\varnothing$ for some $j\in\{1,\ldots,s\}\setminus\{i\}$. By Corollary \ref{inter}$(\mr{i})$, $Z_i\cap X_j\cap(\ove^\sqbullet)^n\neq\varnothing$. Choose a point $a$ in $Z_i\cap X_j\cap(\ove^\sqbullet)^n\subset X\cap(\ove^\sqbullet)^n$. As $Z_i\subset X_i$, we have $a\in X_i\cap X_j$. Proposition \ref{prop:XY}$(\mr{iv})$ implies that $\reg^{E|K}_{X,a}$ is not an integral domain, which contradicts $(\mr{i})$.

We now prove that $Z_i=X_i$. As $Z_i\neq\varnothing$, using again Corollary \ref{inter}$(\mr{i})$, we have $Z_i\cap(\ove^\sqbullet)^n\neq\varnothing$. Choose $b\in Z_i\cap(\ove^\sqbullet)^n$. As $b\in Z_i\setminus\bigcup_{j\in\{1,\ldots,s\}\setminus\{i\}}X_j=Z\setminus\bigcup_{j\in\{1,\ldots,s\}\setminus\{i\}}Z_j$, it follows that $\reg^{E|K}_{Z_i,b}\cong\reg^{E|K}_{Z,b}$. In particular, $\dim(\reg^{E|K}_{Z_i,b})=d$. By Proposition \ref{prop:XY}$(\mr{iii})$, we deduce $d\leq\dim(Z_i)\leq\dim(X_i)\leq\dim(X)=d$ so $\dim(Z_i)=d=\dim(X_i)$. As $Z_i\subset L^n$ is a $K$-algebraic set, $X_i\subset L^n$ is a $K$-irreducible $K$-algebraic set, $Z_i\subset X_i$ and $\dim(Z_i)=\dim(X_i)$, we conclude that $Z_i=X_i$ by Lemma \ref{dimirred}.

We have already proved $X_i\cap X_j=\varnothing$ if $i,j\in I$ and $i\neq j$. Thus, $Z=\bigcup_{i=1}^sZ_i=\bigsqcup_{i\in I}X_i$ and $X_i\cap\bigcup_{j\in \{1,\ldots,s\}\setminus I}X_j=\varnothing$ for each $i\in I$, so $X\setminus Z$ is the $K$-algebraic set $\bigcup_{j\in \{1,\ldots,s\}\setminus I}X_j\subset L^n$.

Finally, suppose $Z\cap(\ove^\sqbullet)^n\neq\Reg^{E|K}(X)$. Pick a point $p\in\Reg^{E|K}(X)\setminus(Z\cap(\ove^\sqbullet)^n)=\Reg^{E|K}(X)\cap(X\setminus Z)=\Reg^{E|K}(X)\cap\bigcup_{j\in\{1,\ldots,s\}\setminus I}X_j$. By Proposition \ref{30}$(\mr{ii})$, there exists a (unique) $i\in\{1,\ldots,s\}\setminus I$ such that $p\in X_i$ and $\dim(X_i)=d$, so $d=\dim(X_i)\leq\dim(X\setminus Z)\leq\dim(X)=d$ and $\dim(X\setminus Z)=d$. As $X_j$ for $j\in\{1,\ldots,s\}\setminus I$ are the $K$-irreducible components of $X\setminus Z\subset L^n$, by \eqref{eq:regd}, we deduce $\Reg^{E|K}(X)\setminus Z=\Reg^{E|K}(X,d)\setminus Z=\Reg^{E|K}(X\setminus Z,d)=\Reg^{E|K}(X\setminus Z)$, as required.
\end{proof}

\begin{remark}
In the previous statement, $E$ can be chosen arbitrarily among the fields such that $L|E|K$ is an extension of fields, for example $E=K$ or $E=L$. In particular, if $E=L$, $\dim(Z)=d$ and $Z=\Reg^{L|K}(Z)\subsetneqq\Reg^{L|K}(X)$, then $X\setminus Z\subset L^n$ is a $K$-algebraic set and $\Reg^{L|K}(X\setminus Z)=\Reg^{L|K}(X)\setminus Z$. $\sqbullet$
\end{remark}

We conclude this subsection with two applications: the first to real $K$-geometric polynomials and the second to real Galois completions of $\kr$-geometric polynomials (see Definitions \ref{321} and \ref{gbullet}, respectively).

\begin{cor}\label{cor:kpol}
Suppose that $L$ is a real closed field $R$. Let $f\in K[\x]$ be a $K$-geometric polynomial in $R^n$ and let $X:=\ZZ_R(f)$ be the corresponding $K$-geometric hypersurface of $R^n$. Then we have
$$
\Sing^{E|K}(X)=\{x\in X\cap(\ove^r)^n:\nabla f(x)=0\}
$$
or, equivalently,
$$
\Reg^{E|K}(X)=\{x\in X\cap(\ove^r)^n:\nabla f(x)\neq0\}.
$$
\end{cor}
\begin{proof}
Let $f=f_1\cdots f_s$ be the factorization of $f\in K[\x]$, which is square-free. Define $X_i:=\ZZ_R(f_i)\subset R^n$ for each $i\in\{1,\ldots,s\}$. By Proposition \ref{prop:hyper}, $X_1,\ldots,X_s$ are the $K$-irreducible components of $X\subset R^n$, and $\dim(X_i)=n-1$ and $\II_K(X_i)=(f_i)K[\x]$ for each $i\in\{1,\ldots,s\}$.

Let us prove that $\Reg^{E|K}(X)=\{x\in X\cap(\ove^r)^n:\nabla f(x)\neq0\}$. Pick $x\in\Reg^{E|K}(X)$. By Propositions \ref{30}$(\mr{ii})$ and \ref{29}$(\mr{ii})$, $x$ belongs to a unique $K$-irreducible component $X_j$ of $X$ and $\nabla f_j(x)\neq0$. Thus, if we set $g:=\prod_{i\in\{1,\ldots,s\}\setminus\{j\}}f_i$, then $f=f_jg$ and Leibniz's rule implies $\nabla f(x)=g(x)\nabla f_j(x)+f_j(x)\nabla g(x)=g(x)\nabla f_j(x)\neq0$. Consider now a point $x\in X\cap(\ove^r)^n$ such that $\nabla f(x)\neq0$. We have to show that $x\not\in\Sing^{E|K}(X)$. Suppose this is false, so $x\in\Sing^{E|K}(X)$. Using again Propositions \ref{29}$(\mr{ii})$ and \ref{30}$(\mr{ii})$, we have that either $x$ belongs to some $X_j$ and $\nabla f_j(x)=0$ (that is, $x\in\Sing^{E|K}(X_j)$) or $x$ belongs to some intersection $X_j\cap X_k$ for some $j\neq k$. In both cases, Leibniz's rule implies $\nabla f(x)=0$, which is a contradiction.
\end{proof}

\begin{cor}\label{3113b}
Suppose that $L=E$ is a real closed field $R$. Let $g\in\kr[\x]$ be a $\kr$-geometric polynomial in $R^n$ and let $g^\bullet\in K[\x]$ be the Galois completion of $g$. Consider the real Galois completion $\ZZ_R(g^\bullet)$ of $\ZZ_R(g)\subset R^n$. Then, in the statement of Corollary \ref{3113}, we can add the following item:
\begin{itemize}
\item[$(\mr{vi})$] $\Sing^{R|K}(\ZZ_R(g^\bullet))=\{x\in\ZZ_R(g^\bullet):\nabla g^\bullet(x)=0\}$.
\end{itemize}
\end{cor}
\begin{proof}
It follows immediately from Corollary \ref{3113}$(\mr{ii})$ and Corollary \ref{cor:kpol}.
\end{proof}

%%%
\subsection{Comparison of $L|K$-nonsingular and $E|K$-nonsingular loci}\label{subsec:L|K-E|K} We now compare $L|K$-nonsingular and $E|K$-nonsingular loci of a $K$-algebraic subset of $L^n$ (see Remark \ref{reg-bcr}).

\begin{thm}
\label{L|K-E|K}
Let $X\subset L^n$ be a $K$-algebraic set. We have:
\begin{itemize}
\item[$(\mr{i})$] $\mr{Sing}^{L|K}(X)\subset L^n$ is the extension of coefficients of $\mr{Sing}^{E|K}(X)\subset(\ove^\sqbullet)^n$ {(from $\ove^\sqbullet$)} to $L$, that is,
 $$
 (\Sing^{E|K}(X))_L=\Sing^{L|K}(X).
 $$
 In particular, $\Sing^{E|K}(X)=\Sing^{L|K}(X)\cap(\ove^\sqbullet)^n$ and $\mr{Sing}^{L|K}(X)=\varnothing$ if and only if $\mr{Sing}^{E|K}(X)=\varnothing$.
\item[$(\mr{ii})$] $(\Reg^{E|K}(X,e))_L=\Reg^{L|K}(X,e)$ and $\Reg^{E|K}(X,e)=\Reg^{L|K}(X,e)\cap(\ove^\sqbullet)^n$ for each $e\in\N$. In particular, $(\Reg^{E|K}(X))_L=\Reg^{L|K}(X)$, $\Reg^{E|K}(X)=\Reg^{L|K}(X)\cap(\ove^\sqbullet)^n$ and $\Reg^{L|K}(X)=X$ if and only if $\Reg^{E|K}(X)=X\cap(\ove^\sqbullet)^n$.
\end{itemize}
\end{thm}
\begin{proof}
Let $d:=\dim(X)$, let $X_1,\ldots,X_s$ be the $K$-irreducible components of $X\subset L^n$ and let $I$ be the subset of $\{1,\ldots,s\}$ of all indices $k$ such that $\dim(X_k)=d$. For each $k\in I$, let $\{g_{k1},\ldots,g_{ks_k}\}$ be a system of generators of $\II_K(X_k)$ in $K[\x]$ and let $\{D_{k1},\ldots,D_{kt_k}\}\subset K[\x]$ be the set of the determinants of all $(n-d)\times(n-d)$-submatrices of $\big(\frac{\partial g_{ki}}{\partial\x_j}\big)_{i=1,\ldots,s_k,\; j=1,\ldots,n}$. Set $Y_k:=X_k\cap(\ove^\sqbullet)^n$ for each $k\in\{1,\ldots,s\}$. By Proposition \ref{prop:XY}$(\mr{ii})$, \eqref{singe} and \eqref{eq:singd}, we have: $X_k=(Y_k)_L$ for each $k\in\{1,\ldots,s\}$, $\Sing^{E|K}(X_k)\subset(\ove^\sqbullet)^n$ is a $K$-algebraic set such that
$$
\Sing^{E|K}(X_k)=Y_k\cap\ZZ_{\ove^\sqbullet}(D_{k1},\ldots,D_{kt_k})
$$
for each $k\in I$, and
$$
\textstyle
\Sing^{E|K}(X)=\bigcup_{k\in I}\mr{Sing}^{E|K}(X_k)\cup\bigcup_{k,h\in I,\,k\neq h}(Y_k\cap Y_h)\cup\bigcup_{k\in\{1,\ldots,s\}\setminus I}Y_k.
$$

Repeating the previous discussion with $E=L$, we obtain:
\begin{align*}
\Sing^{L|K}(X_k)&=X_k\cap\ZZ_L(D_{k1},\ldots,D_{kt_k})=(Y_k)_L\cap(\ZZ_{\ove^\sqbullet}(D_{k1},\ldots,D_{kt_k}))_L\\
&=\big(Y_k\cap\ZZ_{\ove^\sqbullet}(D_{k1},\ldots,D_{kt_k})\big)_L=(\Sing^{E|K}(X_k))_L
\end{align*}
for each $k\in I$, and
\begin{align*}
\Sing^{L|K}(X)&\textstyle=\bigcup_{k\in I}\mr{Sing}^{L|K}(X_k)\cup\bigcup_{k,h\in I,\,k\neq h}(X_k\cap X_h)\cup\bigcup_{k\in\{1,\ldots,s\}\setminus I}X_k\\
&\textstyle=\bigcup_{k\in I}(\Sing^{E|K}(X_k))_L\cup\bigcup_{k,h\in I,\,k\neq h}(Y_k\cap Y_h)_L\cup\bigcup_{k\in\{1,\ldots,s\}\setminus I}(Y_k)_L\\
&\textstyle=\left(\bigcup_{k\in I}\Sing^{E|K}(X_k)\cup\bigcup_{k,h\in I,\,k\neq h}(Y_k\cap Y_h)\cup\bigcup_{k\in\{1,\ldots,s\}\setminus I}Y_k\right)_L\\
&=(\Sing^{E|K}(X))_L.
\end{align*}

This proves $(\mr{i})$. Let us show $(\mr{ii})$. By $(\mr{i})$, we know that $(\Sing^{E|K}(X_k))_L=\Sing^{L|K}(X_k)$ for each $k\in\{1,\ldots,s\}$. Let $J:=\bigcup_{k=1}^s\{\dim(X_k)\}$. If $e\in\N\setminus J$, Corollary \ref{30'}$(\mr{ii})$ assures that $\Reg^{E|K}(X,e)=\Reg^{L|K}(X,e)=\varnothing$. If $e\in J$ and $I_e$ denotes the subset of $\{1,\ldots,s\}$ of all indices $k$ such that $\dim(X_k)=e$, then formula \eqref{eq:regd} implies that
$$\textstyle
\Reg^{E|K}(X,e)=\bigsqcup_{k\in I_e}\left(\big(Y_k\setminus\Sing^{E|K}(X_k)\big)\setminus\bigcup_{j\in\{1,\ldots,s\}\setminus\{k\}}Y_j\right)
$$
and
$$\textstyle
\Reg^{L|K}(X,e)=\bigsqcup_{k\in I_e}\left(\big(X_k\setminus\Sing^{L|K}(X_k)\big)\setminus\bigcup_{j\in\{1,\ldots,s\}\setminus\{k\}}X_j\right).
$$
Consequently,
$$\textstyle
(\Reg^{E|K}(X,e))_L=\bigsqcup_{k\in I_e}\left(\big((Y_k)_L\setminus(\Sing^{E|K}(X_k))_L\big)\setminus\bigcup_{j\in\{1,\ldots,s\}\setminus\{k\}}(Y_j)_L\right)=\Reg^{L|K}(X,e),
$$
as required.
\end{proof}

\subsection{Comparison of $L|K$-nonsingular and usual $L|L$-nonsingular loci: the real case}\label{subsec:R|R-R|K}
\emph{Recall that $L|K$ is a fixed extension of fields such that $L$ is either algebraically closed or real closed.}

Let $X\subset L^n$ be a $K$-algebraic set. Our next goal is to compare the usual nonsingular and singular loci $\Reg(X)=\Reg^{L|L}(X)$ and $\Sing(X)=\Sing^{L|L}(X)$ of $X$, considered as an algebraic subset of~$L^n$, with the $L|K$-nonsingular and $L|K$-singular loci $\Reg^{L|K}(X)$ and $\Sing^{L|K}(X)$ of $X$, considered as a $K$-algebraic subset of $L^n$.

This task is straightforward in the case where either $L$ is an algebraically closed field or $L$ is a real closed field and $X\subset L^n$ is defined over $K$ or $L|K$ is an extension of real closed fields, as we explained in Remark \ref{reg-bcr}. In particular, if $L=R$ is real closed and $X\subset R^n$ is defined over $K$, then $\Reg(X)=\Reg^{R|R}(X)$ is equal to $\Reg^{R|K}(X)$, and $\Sing(X)=\Sing^{R|R}(X)$ to $\Sing^{R|K}(X)$. Moreover, by Remark \ref{rem:252}$(\mr{i})$ and equation \eqref{eq:singd} (with $L=E=R$), we have that the $K$-bad set $B_K(X)$ of $X$ is contained in $\Sing(X)$ and is empty when $X\subset R^n$ is $K$-irreducible.

\emph{In the rest of this subsection, we assume that $L$ is a real closed field $R$, $K$ is an ordered subfield of $R$, and $C$ denotes the algebraic closure $R[\ii]$ of $R$}.

In the next result, we establish the relationship between $\Reg(X)$ and $\Reg^{R|K}(X)$, and between $\Sing(X)$ and $\Sing^{R|K}(X)$, in the case where $X\subset R^n$ is $K$-algebraic, but not defined over $K$. In particular, $K$ is not real closed by Corollary \ref{inter}$(\mr{ii})$. Recall that, by Definition \ref{E|K-regular}, $\Reg^{R|K}(X)$ and $\Sing^{R|K}(X)$ can also be denoted $\Reg^K(X)$ and $\Sing^K(X)$, respectively.

\begin{thm}
\label{regreg}
Let $X\subset R^n$ be a $K$-algebraic set of dimension $d$. Suppose that $X\subset R^n$ is not defined over $K$. We have:
\begin{itemize}
\item[$(\mr{i})$] $\Sing^{R|K}(X)$ is a $K$-algebraic subset of $R^n$ of dimension $<d$, $\Sing(X)$ and $B_K(X)$ are $\kr$-algebraic subsets of $R^n$ of dimension $<d$, and
 $$
 \Sing^{R|K}(X)=\Sing(X)\cup B_K(X).
 $$
\item[$(\mr{ii})$] $\Reg^{R|K}(X)$ is a non-empty $K$-Zariski open subset of $X$, $\Reg(X)$ is a non-empty $\kr$-Zariski open subset of $X$, and
 $$
 \Reg^{R|K}(X)=\Reg(X)\setminus B_K(X).
 $$
\end{itemize}
\end{thm}
\begin{proof}
By Corollary \ref{30'}$(\mr{iii})$, $\Sing^{R|K}(X)\subset R^n$ is a $K$-algebraic set of dimension $<d$ and $\Reg^{R|K}(X)=X\setminus\Sing^{R|K}(X)$ is a non-empty $K$-Zariski open subset of $X\subset R^n$.

Let us show that $\Sing(X)\subset R^n$ is a $\kr$-algebraic set of dimension $<d$.

Let $V_1,\ldots,V_t$ be the $R$-irreducible components of $X\subset R^n$. As $X\subset R^n$ is $\kr$-algebraic, Corollary \ref{inter}$(\mr{ii})(\mr{iii})$ implies that $V_1,\ldots,V_t$ are also the $\kr$-irreducible components of $X\subset R^n$ and $\II_R(V_h)=\II_\kr(V_h)R[\x]$ for each $h\in\{1,\ldots,t\}$. Let $H$ be the subset of $\{1,\ldots,t\}$ of all indices $h$ such that $\dim(V_h)=d$. For each $h\in H$, choose a system of generators $\{v_{h,1},\ldots,v_{h,t_h}\}\subset\kr[\x]$ of $\II_\kr(V_h)$ in $\kr[\x]$ (and therefore of $\II_R(V_h)$ in $R[\x]$). Thus, $\Sing(V_h)$ is the algebraic subset of $R^n$ of dimension $<d$ given by
\begin{equation}\label{sing1}
\textstyle\Sing(V_h)=\big\{a\in V_h:{\rm rk}\big(\frac{\partial v_{h,i}}{\partial\x_j}(a)\big)_{i=1,\ldots,t_h,\,j=1,\ldots,n}<n-d\big\}
\end{equation}
(this fact is well-known and included in Proposition \ref{29}$(\mr{ii})$),
and
\begin{equation}\label{sing2}
\textstyle\Sing(X)=\bigcup_{h\in H}\Sing(V_h)\cup\bigcup_{h,j\in H, h\neq j}(V_h\cap V_j)\cup\bigcup_{h\in\{1,\ldots,t\}\setminus H}V_h
\end{equation}
(this fact is well-known and included in Proposition \ref{29}$(\mr{i})$ and equivalence $(\mr{i})\Longleftrightarrow(\mr{ii})$ of Proposition \ref{30}). Equations \eqref{sing1}\&\eqref{sing2} and Lemma \ref{dimirred} imply immediately that $\Sing(X)\subset R^n$ is $\kr$-algebraic and of dimension $<d$. As a consequence, $\Reg(X)=X\setminus\Sing(X)$ is a non-empty $\kr$-Zariski open subset of $X\subset R^n$.

Let $X_1,\ldots,X_s$ be the $K$-irreducible components of $X\subset R^n$. Choose an $R$-irreducible component $Y_i$ of $X_i\subset R^n$ of dimension $\dim(X_i)$ for each $i\in\{1,\ldots,s\}$, and a Galois presentation $(Y_1,\ldots,Y_s;G';\{Z_1^\sigma\}_{\sigma\in G'},\ldots,\{Z_s^\sigma\}_{\sigma\in G'})$ of $X\subset R^n$, see Definition \ref{def:gp}. Denote $I$ the subset of $\{1,\ldots,s\}$ of all indices $i$ such that $\dim(X_i)=d$. For each $i\in I$, choose a subset $F_i$ of $G'$ such that $\{Z^\sigma_i: \sigma\in F_i\}=\{Z^\sigma_i: \sigma\in G'\}$ and $Z^\sigma_i\neq Z^\tau_i$ if $\sigma,\tau\in F_i$ and $\sigma\neq\tau$, and set $F^*_i:=\{\sigma\in F_i: \dim(Z^\sigma_i\cap R^n)<d\}$. Define:
$$\textstyle
\mc{F}:=\bigsqcup_{i\in I}(\{i\}\times F_i), \quad \mc{F}^*:=\bigsqcup_{i\in I}(\{i\}\times F^*_i) \quad\text{and}\quad \mc{G}:=(\{1,\ldots,s\}\setminus I)\times G'.
$$
By Lemma \ref{lem:gp-reducible}$(\mr{i})(\mr{iii})$, we have that $\dim_C(Z^\sigma_i)=d$ for each $(i,\sigma)\in\mc{F}$, and $Z^\sigma_i\neq Z^\tau_j$ for each $(i,\sigma),(j,\tau)\in\mc{F}$ with $(i,\sigma)\neq(j,\tau)$. In addition, by Lemma \ref{lem:gp-reducible}$(\mr{ii})$,
\begin{equation}\label{T}
\textstyle
T:=\bigcup_{(i,\sigma)\in\mc{F}}Z^\sigma_i\cup\bigcup_{(i,\sigma)\in\mc{G}}Z^\sigma_i
\end{equation}
is the complex Galois completion of $\bigcup_{i=1}^sY_i\subset R^n$ and the set $T^r:=T\cap R^n$ is the real Galois completion of $\bigcup_{i=1}^sY_i\subset R^n$, which coincides with~$X$. Thus,
\begin{equation}\label{X}
\textstyle
X=T^r=\bigcup_{(i,\sigma)\in\mc{F}}(Z^\sigma_i\cap R^n)\cup\bigcup_{(i,\sigma)\in\mc{G}}(Z^\sigma_i\cap R^n)\,.
\end{equation}
By Remark \ref{rem317}, we know that $B_K(X)\subset R^n$ is a $\kr$-algebraic set of dimension $<d$ and 
\begin{equation}\label{B}
\textstyle
B_K(X)=\bigcup_{(i,\sigma)\in\mc{F}^*}(Z^\sigma_i\cap R^n)\cup\bigcup_{(i,\sigma)\in\mc{G}}(Z^\sigma_i\cap R^n)\,.
\end{equation}

As $\Sing^{R|K}(X)=X\setminus\Reg^{R|K}(X)$ and $\Sing(X)=X\setminus\Reg(X)$, it remains to prove the equality: $\Reg^{R|K}(X)=\Reg(X)\setminus B_K(X)$.

The proof is conducted in several steps:

{\sc Step I.} {\em Zero ideals of Galois completions.} By Theorem \ref{thm:gc}$(\mr{vi})$, we have the equalities $\II_R(T)=\II_K(T^r)R[\x]=\II_K(X)R[\x]$. Thus, by \eqref{T} and \eqref{X},
\begin{equation}\label{iqr}
\textstyle
\II_K(X)R[\x]=\bigcap_{(i,\sigma)\in\mc{F}}\II_R(Z^\sigma_i)\cap\bigcap_{(i,\sigma)\in\mc{G}}\II_R(Z^\sigma_i)
\end{equation}
and
\begin{equation}\label{iqr2}
\textstyle
\II_R(X)=\bigcap_{(i,\sigma)\in\mc{F}}\II_R(Z^\sigma_i\cap R^n)\cap\bigcap_{(i,\sigma)\in\mc{G}}\II_R(Z^\sigma_i\cap R^n)\,.
\end{equation}

{\sc Step II.} {\em $R$-irreducible components of $X$ of maximal dimension.} By Lemma \ref{lem:gp-reducible}$(\mr{i})$, each $Z^\sigma_i\subset C^n$ is $C$-irreducible, so $\II_C(Z^\sigma_i)$ is a prime ideal of $C[\x]$ and $\II_R(Z^\sigma_i)=\II_C(Z^\sigma_i)\cap R[\x]$ is a prime ideal of $R[\x]$ as well. By Proposition \ref{rc}, $\II_R(Z^\sigma_i)\subsetneqq\II_R(Z^\sigma_i\cap R^n)$ for each $(i,\sigma)\in\mc{F}^*$. Moreover, if $(i,\sigma)\in\mc{F}\setminus\mc{F}^*$, then $\II_R (Z^\sigma_i\cap R^n)=\II_R(Z^\sigma_i)$ (so $\II_R (Z^\sigma_i\cap R^n)\subset R[\x]$ is prime), $Z^\sigma_i\cap R^n$ is an $R$-irreducible algebraic subset of $R^n$ of dimension $d$, and $\zcl_{C^n}(Z^\sigma_i\cap R^n)=Z^\sigma_i$. The latter equality implies that $Z^\sigma_i\cap R^n\neq Z^\tau_j\cap R^n$ if $(i,\sigma),(j,\tau)\in\mc{F}\setminus\mc{F}^*$ with $(i,\sigma)\neq(j,\tau)$. Observe that $\dim(Z^\sigma_i\cap R^n)<d$ for each $(i,\sigma)\in\mc{F}^*\cup\mc{G}$. Thus, by \eqref{X}, we deduce that $\{Z^\sigma_i\cap R^n\}_{(i,\sigma)\in\mc{F}\setminus\mc{F}^*}$ is the family of all $R$-irreducible components of $X\subset R^n$ of dimension~$d$.

{\sc Step III.} {\em The inclusion $\Reg(X)\setminus B_K(X)\subset\Reg^{R|K}(X)$.} Let $a\in\Reg(X)\setminus B_K(X)$. By \eqref{X} and \eqref{B}, there exists an index $(p,\xi)\in\mc{F}\setminus\mc{F}^*$ such that $a\in Z^\xi_p\cap R^n$. As $a\in\Reg(X)$ and $\{Z^\sigma_i\cap R^n\}_{(i,\sigma)\in\mc{F}\setminus\mc{F}^*}$ is the family of all $R$-irreducible components of $X\subset R^n$ of dimension $d=\dim(X)$, such an index $(p,\xi)\in\mc{F}\setminus\mc{F}_*$ is unique, that is, $a\not\in\bigcup_{(i,\sigma)\in(\mc{F}\setminus\mc{F}^*)\setminus\{(p,\xi)\}}(Z^\sigma_i\cap R^n)$. As $a\not\in B_K(X)$, we deduce by \eqref{B}
\begin{equation*}\label{aaa}\textstyle
a\in(Z^\xi_p\cap R^n)\setminus\bigcup_{(i,\sigma)\in(\mc{F}\cup\mc{G})\setminus\{(p,\xi)\}}(Z^\sigma_i\cap R^n)\,.
\end{equation*}
Let $(i,\sigma)\in(\mc{F}\cup\mc{G})\setminus\{(p,\xi)\}$. As $a\not\in Z^\sigma_i\cap R^n$ (hence $a\not\in Z^\sigma_i$, because $a\in\Reg(X)\subset R^n$), there exists a polynomial $h\in \II_C(Z^\sigma_i)$ such that $h(a)\neq0$. Let $h_1,h_2\in R[\x]$ be such that $h=h_1+\ii h_2$ and define $\ol{h}:=h_1-\ii h_2\in C[\x]$. As $a\in R^n$ and $h(a)=h_1(a)+\ii h_2(a)\neq0$, we have $\ol{h}(a)=h_1(a)-\ii h_2(a)\neq0$. It follows that $h\ol{h}\in \II_C(Z^\sigma_i)\cap R[\x]=\II_R(Z^\sigma_i)$ and $(h\ol{h})(a)\neq0$, so $\II_R(Z^\sigma_i)\not\subset\gtn_a:=\{f\in R[\x]:f(a)=0\}$. Consequently, $\II_R(Z^\sigma_i\cap R^n)\not\subset\gtn_a$ for each $(i,\sigma)\in(\mc{F}\cup\mc{G})\setminus\{(p,\xi)\}$. Recall that $\II_R(Z^\xi_p)=\II_R(Z^\xi_p\cap R^n)$, because $(p,\xi)\in\mc{F}\setminus\mc{F}^*$. By \eqref{iqr}, \eqref{iqr2} and Corollary \ref{lem:a}, we deduce
$$
\II_K(X)R[\x]_{\gtn_a}=\II_R(Z^\xi_p)R[\x]_{\gtn_a}=\II_R(Z^\xi_p\cap R^n)R[\x]_{\gtn_a}=\II_R(X)R[\x]_{\gtn_a}, 
$$
so $\reg^{R|K}_{X,a}=\reg_{X,a}$. As $\reg^{R|K}_{X,a}=\reg_{X,a}$ is a regular local ring of dimension $d$, we deduce $a\in\Reg^{R|K}(X)$.

{\sc Step IV.} {\em The inclusion $\Reg^{R|K}(X)\subset\Reg(X)\setminus B_K(X)$.} Pick a point $a\in\Reg^{R|K}(X)$. Implication $(\mr{i})\Longrightarrow(\mr{iii})$ of Proposition \ref{30} assures the existence of $f_1,\ldots,f_{n-d}\in\II_K(X)$ and a ($R$-)Zariski open neighborhood $U$ of $a$ in $R^n$ such that ${\rm rk}\big(\frac{\partial f_i}{\partial\x_j}(a)\big)_{i=1,\ldots,n-d,\, j=1,\ldots,n}=n-d$ and $X\cap U=\ZZ_R(f_1,\ldots,f_{n-d})\cap U$. As $f_1,\ldots,f_{n-d}\in\II_R(X)$, the converse implication $(\mr{iii})\Longrightarrow(\mr{i})$ of Proposition \ref{30} implies that $a\in\Reg(X)$. In addition, by implication $(\mr{i})\Longrightarrow(\mr{ii})$ of the same Proposition \ref{30}, we have $a\not\in\bigcup_{i\in\{1,\ldots,s\}\setminus I}X_i=\bigcup_{(i,\sigma)\in\mc{G}}(Z^\sigma_i\cap R^n)$. The latter equality holds by Lemma \ref{lem:gp}(iii) applied to each $X_i$ with $i\in\{1,\ldots,s\}\setminus I$. As $a\in\Reg^{R|K}(X)\subset R^n$, we deduce:
\begin{equation}\label{QQ}
\textstyle
a\not\in\bigcup_{(i,\sigma)\in\mc{G}}Z^\sigma_i\,.
\end{equation}

We have to show that $a\not\in B_K(X)$. By \eqref{B} and \eqref{QQ}, it is enough to prove
\begin{equation}\label{QQQQ}
\textstyle
a\not\in\bigcup_{(i,\sigma)\in\mc{F}^*}(Z^\sigma_i\cap R^n)\,.
\end{equation}

Suppose \eqref{QQQQ} is not true, that is, $a\in Z^\tau_h\cap R^n$ for some $h\in I$ and $\tau\in F^*_h$. As $\dim(Z^\tau_h\cap R^n)<d=\dim_C(Z^\tau_h)$, by Proposition \ref{rc}, the ideal $\II_R(Z^\tau_h)$ of $R[\x]$ is non-real. As $\II_R(Z^\tau_h)\subset\gtn_a$ and $\II_R(Z^\tau_h)\subset R[\x]$ is prime, we deduce that
\begin{equation}\label{eqeq}
\text{the ideal $\II_R(Z^\tau_h)R[\x]_{\gtn_a}$ of $R[\x]_{\gtn_a}$ is (prime but) non-real.}
\end{equation}

As the local ring $\reg^{R|K}_{X,a}=R[\x]_{\gtn_a}/(\II_K(X)R[\x]_{\gtn_a})$ is regular, it is an integral domain. Therefore, $\II_K(X)R[\x]_{\gtn_a}$ is a prime ideal of $R[\x]_{\gtn_a}$. Thus, there exists only one minimal prime ideal $\gtp$ of $R[\x]$ associated to $\II_K(X)R[\x]$ such that $\gtp\subset\gtn_a$, so 
\begin{equation}\label{4048}
\II_K(X)R[\x]_{\gtn_a}=\gtp R[\x]_{\gtn_a}.
\end{equation}

Let $\gtq_1,\ldots,\gtq_u$ be the minimal prime ideals of $R[\x]$ associated to $\sqrt[r]{\gtp}$. By \cite[Lem.4.1.5]{bcr}, each $\gtq_k$ is a real ideal of $R[\x]$ and $\sqrt[r]{\gtp}$ is a radical ideal of $R[\x]$. In particular, $\sqrt[r]{\gtp}=\bigcap_{k=1}^u\gtq_k\subset\gtn_a$. By \cite[Prop.1.11.ii)]{am}, we may assume that there exists $v\in\{1,\ldots,u\}$ such that $\gtq_1,\ldots,\gtq_v\subset\gtn_a$, and $\gtq_j\not\subset\gtn_a$ for each $j\in J:=\{v+1,\ldots,u\}$, where $J:=\varnothing$ when $v=u$. Rearranging the indices if necessary, we can also assume that $\hgt(\gtq_1)=\min\{\hgt(\gtq_1),\ldots,\hgt(\gtq_v)\}$. Consequently, each prime ideal $\gtq_kR[\x]_{\gtn_a}$ of $R[\x]_{\gtn_a}$ is real of the same height as $\gtq_k$ for each $k\in\{1,\ldots,v\}$ and $\sqrt[r]{\gtp}R[\x]_{\gtn_a}=\bigcap_{k=1}^v(\gtq_kR[\x]_{\gtn_a})$ by Corollary \ref{lem:a}. Therefore, it holds:
\begin{equation}\label{htrR}
\hgt(\sqrt[r]{\gtp}R[\x]_{\gtn_a})=\hgt(\gtq_1 R[\x]_{\gtn_a})=\hgt(\gtq_1).
\end{equation}
Moreover, replacing $U$ with the Zariski open neighborhood $U\setminus\bigcup_{j\in J}\ZZ_R(\gtq_j)$ of $a$ in $R^n$, we can assume that
\begin{equation}\label{eq:zzu}
\textstyle
\ZZ_R(\sqrt[r]{\gtp})\cap U=\ZZ_R\big(\bigcap_{k=1}^v\gtq_k\big)\cap U.
\end{equation}

By Remark \ref{rmk-jc} and \eqref{4048}, we have
\begin{equation}\label{eq:n-d}
n-d=\hgt(\II_K(X)R[\x]_{\gtn_a})=\hgt(\gtp R[\x]_{\gtn_a})
\end{equation}
and
$$
(f_1,\ldots,f_{n-d}) R[\x]_{\gtn_a}=\II_K(X)R[\x]_{\gtn_a}=\gtp R[\x]_{\gtn_a}.
$$
As $(f_1,\ldots,f_{n-d}) R[\x]_{\gtn_a}=\gtp R[\x]_{\gtn_a}$, shrinking $U$ around $a$ if necessary, we can assume
$$
\ZZ_R(f_1,\ldots,f_{n-d})\cap U=\ZZ_R(\gtp)\cap U=\ZZ_R(\sqrt[r]{\gtp})\cap U.
$$
By \cite[Prop.3.3.11]{bcr}, shrinking $U$ further around $a$ if necessary, we can also assume that $\ZZ_R(\sqrt[r]{\gtp})\cap U$ is a Nash submanifold of $U$ of dimension $d$. Combining this fact with \eqref{htrR}, \eqref{eq:zzu} and \eqref{eq:n-d}, we obtain: 
\begin{align*}
n-\hgt(\gtp R[\x]_{\gtn_a})&\textstyle=d=\dim(\ZZ_R(\sqrt[r]{\gtp})\cap U)=\dim(\ZZ_R(\bigcap_{k=1}^v\gtq_k)\cap U)\\
&\textstyle\leq\dim(\ZZ_R(\bigcap_{k=1}^v\gtq_k))=n-\hgt(\bigcap_{k=1}^v\gtq_k)=n-\hgt(\gtq_1)\\
&\textstyle =n-\hgt(\gtq_1R[\x]_{\gtn_a})=n-\hgt(\sqrt[r]{\gtp}R[\x]_{\gtn_a})\leq n-\hgt(\gtp R[\x]_{\gtn_a}),
\end{align*}
so $\hgt(\gtp R[\x]_{\gtn_a})=\hgt(\gtq_1R[\x]_{\gtn_a})$. As $\gtp R[\x]_{\gtn_a}$ is a prime ideal of $R[\x]_{\gtn_a}$ and $\gtp R[\x]_{\gtn_a}\subset\gtq_1R[\x]_{\gtn_a}$, it follows that $\II_K(X)R[\x]_{\gtn_a}=\gtp R[\x]_{\gtn_a}=\gtq_1R[\x]_{\gtn_a}$. Thus, the prime ideal $\II_K(X)R[\x]_{\gtn_a}$ of $R[\x]_{\gtn_a}$ is in addition real.

Set $\mc{F}_a:=\{(i,\sigma)\in\mc{F}:a\in Z^\sigma_i\}$ and observe that $(h,\tau)\in\mc{F}_a$. Using again \eqref{iqr}, \eqref{QQ} and Corollary \ref{lem:a}, we have $\bigcap_{(i,\sigma)\in\mc{F}_a}\II_R(Z^\sigma_i)R[\x]_{\gtn_a}=\II_K(X)R[\x]_{\gtn_a}$. As $\II_K(X)R[\x]_{\gtn_a}$ is a prime ideal of $R[\x]_{\gtn_a}$, by \cite[Prop.1.11.ii)]{am}, there exists $(i_0,\sigma_0)\in\mc{F}_a$ such that $\II_R(Z^{\sigma_0}_{i_0})R[\x]_{\gtn_a}=\II_K(X)R[\x]_{\gtn_a}$. As the ideal $\II_K(X)R[\x]_{\gtn_a}$ of $R[\x]_{\gtn_a}$ is real, whereas
$\II_R(Z^\tau_h)R[\x]_{\gtn_a}$ is not by \eqref{eqeq}, we deduce that $(i_0,\sigma_0)\neq(h,\tau)$ and
$$\textstyle
\II_R(Z^{\sigma_0}_{i_0})R[\x]_{\gtn_a}=\II_K(X)R[\x]_{\gtn_a}=\bigcap_{(i,\sigma)\in\mc{F}_a}\II_R(Z^\sigma_i)R[\x]_{\gtn_a}\subset\II_R(Z^\tau_h)R[\x]_{\gtn_a},
$$
so $\II_R(Z^{\sigma_0}_{i_0})R[\x]_{\gtn_a}\subsetneqq\II_R(Z^\tau_h)R[\x]_{\gtn_a}$. As $\II_R(Z^{\sigma_0}_{i_0})$ and $\II_R(Z^\tau_h)$ are prime ideals of $R[\x]$ contained in $\gtn_a$, $\II_R(Z^{\sigma_0}_{i_0})R[\x]_{\gtn_a}=\II_K(X)R[\x]_{\gtn_a}$ is a real ideal of $R[\x]_{\gtn_a}$ and $\II_R(Z^{\sigma_0}_{i_0})R[\x]_{\gtn_a}\subsetneqq\II_R(Z^\tau_h)R[\x]_{\gtn_a}$, it follows that $\II_R(Z^{\tau_0}_{i_0})$ is a real ideal of $R[\x]$ and $\II_R(Z^{\sigma_0}_{i_0})\subsetneqq \II_R(Z^\tau_h)$ as well. Thus, Proposition \ref{rc} implies
$$
\II_C(Z^{\sigma_0}_{i_0})=\II_R(Z^{\sigma_0}_{i_0})C[\x]\subset\II_R(Z^\tau_h)C[\x]\subset\II_C(Z^\tau_h),
$$
so $Z^\tau_h\subset Z^{\sigma_0}_{i_0}$. We know that $\dim_C(Z^\tau_h)=d=\dim_C(Z^{\sigma_0}_{i_0})$. As $Z^{\sigma_0}_{i_0}\subset C^n$ is $C$-irreducible, we have $Z^{\sigma_0}_{i_0}=Z^\tau_h$, which is a contradiction because $\II_R(Z^{\sigma_0}_{i_0})$ is a real ideal of $R[\x]$, whereas $\II_R(Z^\tau_h)$ is not. Consequently, $a\not\in B_K(X)$, as required.
\end{proof}

In the next five examples, we will see several behaviors of the sets $\Sing(X)$, $B_K(X)$ and, consequently, $\Sing^{R|K}(X)$ in the case $K=\Q$. In particular, the fifth example shows that Theorem \ref{regreg} is sharp in the following sense: there exist $\Q$-algebraic sets $X\subset R^n$ such that $\Sing(X)\subset R^n$ and $B_\Q(X)\subset R^n$ are $\qr$-algebraic, whereas they are not $\Q$-algebraic and their union $\Sing^{R|\Q}(X)=\Sing^\Q(X)\subset R^n$ is $\Q$-algebraic, as predicted by the aforementioned theorem.

\begin{examples}\label{432}
$(\mr{i})$ The singleton $X:=\{\sqrt[3]{2}\}\subset R$ is an example of $\Q$-irreducible $\Q$-algebraic set, not defined over $\Q$ such that 
$$
\Sing(X)=\varnothing, \qquad B_\Q(X)=\varnothing, \qquad \Sing^{R|\Q}(X)=\varnothing.
$$
$(\mr{ii})$ Let $G:=\x_1^3-2\x_2^3\in\Q[\x]:=\Q[\x_1,\x_2]$, let $g:=\x_1-\sqrt[3]{2}\x_2\in \qr[\x]$ and let $X:=\ZZ_R(G)=\ZZ_R(g)\subset R^2$. Observe that $G$ and $g$ are irreducible, the former in $\Q[\x]$ and the latter in $R[\x]$. By Proposition \ref{prop:hyper}, $G$ is $\Q$-geometric in $R^2$ and $g$ is ($R$-)geometric in $R^2$, that is, $\II_\Q(X)=(G)\Q[\x]$ and $\II_R(X)=(g)R[\x]$. Moreover, $X\subset R^2$ has dimension $1$ and is irreducible, so $\Q$-irreducible. By Corollary \ref{cor:kpol},
$$
\Sing^{R|\Q}(X)=\{x\in X:\nabla G(x)=0\}=\{(0,0)\},
$$
whereas
$$
\Sing(X)=\{x\in X:\nabla g(x)=0\}=\varnothing.
$$

Let $T:=\zcl_{C^2}^\Q(X)\subset C^2$. As $\II_\Q(T)=\II_\Q(X)=(G)\Q[\x]$, we have $T=\ZZ_C(G)$. Thus, if we define $w:=-\frac{1}{2}+\frac{\sqrt{3}}{2}\ii$ and $T_k:=\ZZ_C(\x_1-\sqrt[3]{2}w^k\x_2)\subset C^2$ for each $k\in\{0,1,2\}$, then $T_0$, $T_1$ and $T_2$ are the $C$-irreducible components of $T\subset C^2$. As $T_0\cap R^2=X$ and $T_1\cap R^2=T_2\cap R^2=\{(0,0)\}$, we deduce that $B_\Q(X)=\bigcup_{k=1}^2(T_k\cap R^2)=\{(0,0)\}$. The set $B_\Q(X)$ can also be computed using Proposition \ref{irre-sf}$(\mr{iii})$. Indeed, $G(\x)=g(\x)(\x_1^2+\sqrt[3]{2}\x_1\x_2+\sqrt[3]{4}\x_2^2)$ is the factorization of $f$ in~$\qr[\x]$, where $\ZZ_R(g)\subset R^2$ is a line and $\ZZ_R(\x_1^2+\sqrt[3]{2}\x_1\x_2+\sqrt[3]{4}\x_2^2)=\{(0,0)\}$. Consequently, $B_\Q(X)=\{(0,0)\}$. Summarizing,
$$
\Sing(X)=\varnothing, \qquad B_\Q(X)=\{(0,0)\}, \qquad \Sing^{R|\Q}(X)=\{(0,0)\}.
$$
Therefore, $X\subset R^2$ is nonsingular as an algebraic set, but it is singular as a $\Q$-algebraic set.

$(\mr{iii})$ Let $f:=(\x_1^2-\x_2^2)^3+2\x_3^6\in\Q[\x]:=\Q[\x_1,\x_2,\x_3]$, let $g:=\x_1^2-\x_2^2+\sqrt[3]{2}\x_3^2\in\qr[\x]$ and let $X:=\ZZ_R(f)=\ZZ_R(g)\subset R^3$. The polynomial $f$ is irreducible in $\Q[\x]$ and $\Q$-geometric in $R^3$, and the polynomial $g$ is irreducible in $R[\x]$ and geometric in $R^3$. Thus, $X\subset R^3$ is irreducible (therefore $\Q$-irreducible), $\dim(X)=2$ and it holds:
$$
\Sing^{R|\Q}(X)=\{x\in X:\nabla f(x)=0\}=\ZZ_R(\x_1^2-\x_2^2,\x_3)
$$
and
$$
\Sing(X)=\{x\in X:\nabla g(x)=0\}=\{(0,0,0)\}.
$$

Let $T:=\zcl_{C^3}^\Q(X)\subset C^3$. As $\II_\Q(T)=\II_\Q(X)=(f)\Q[\x]$ and $f=\prod_{k=0}^2(\x_1^2-\x_2^2+\sqrt[3]{2}w^k\x_3^2)$ is the factorization of $f$ in $C[\x]$, we deduce that $T_k:=\ZZ_C(\x_1^2-\x_2^2+\sqrt[3]{2}w^k\x_3^2)\subset C^3$ for $k\in\{0,1,2\}$ are the $C$-irreducible components of $T\subset C^3$. Observe that $T_0\cap R^3=X$, and $T_1\cap R^3=T_2\cap R^3=\ZZ_R(\x_1^2-\x_2^2,\x_3)$ has dimension~$1$, so $
B_\Q(X)=\bigcup_{k=1}^2(T_k\cap R^3)=\ZZ_R(\x_1^2-\x_2^2,\x_3)$. We have just proved that
$$
\Sing(X)=\{(0,0,0)\}, \qquad B_\Q(X)=\ZZ_R(\x_1^2-\x_2^2,\x_3), \qquad \Sing^{R|\Q}(X)=\ZZ_R(\x_1^2-\x_2^2,\x_3).
$$

$(\mr{iv})$ Let $g:=\x_1+\sqrt{2}\x_2+\sqrt[4]{2}\x_3\in\qr[\x]:=\qr[\x_1,\x_2,\x_3]$, let $X_0$ be the plane $\ZZ_R(g)$ of $R^3$ and let $X:=\zcl_{R^3}^\Q(X_0)$ be the real Galois completion of $X_0\subset R^3$. Consider the line $X_1$ and the plane $X_2$ of $R^3$ given by
$$
X_1:=\{x\in R^3:x_1-\sqrt{2}x_2=0,x_3=0\}, \qquad X_2:=\{x\in R^3:x_1+\sqrt{2}x_2-\sqrt[4]{2}x_3=0\}.
$$
Define also the line $X_1'$ of $R^3$ by
$$
X_1':=X_0\cap X_2=\{x\in R^3:x_1+\sqrt{2}x_2=0,x_3=0\}.
$$

In Examples \ref{exa:gc}$(\mr{i})$, we proved: $X$ is a $\Q$-irreducible $\Q$-algebraic hypersurface of $R^3$, $B_\Q(X)$ is equal to $X_1$ and $X$ is a reducible algebraic subset of $R^3$, whose irreducible components are the planes $X_0$, $X_2$ and the line $X_1$. We deduce that $\Sing(X)$ is the union of the lines $X_1$ and $X_1'$. By Theorem \ref{regreg}$(\mr{i})$, we have $\Sing^{R|\Q}(X)=\Sing(X)\cup B_\Q(X)=X_1\cup X_1'$. The latter equality can also been obtained using Corollary \ref{3113b}: $\Sing^{R|\Q}(X)=\{x\in X:\nabla g^\bullet(x)=0\}=X_1\cup X_1'$, where $g^\bullet$ is the Galois completion of $g$ (see Examples \ref{exa:gc}$(\mr{i})$ for the explicit expression of $g^\bullet$). Consequently,
$$
\Sing(X)=X_1\cup X_1', \qquad B_\Q(X)=X_1, \qquad \Sing^{R|\Q}(X)=X_1\cup X_1',
$$
see Figure \ref{im:poly2'}. Observe that $B_\Q(X)=X_1\subset R^3$ is $\qr$-algebraic, but not $\Q$-algebraic, because its $\Q$-Zariski closure in $R^3$ is $\ZZ_R(\x_1^2-2\x_2^2,\x_3)$, which in this case coincides with $\Sing(X)=\Sing^{R|\Q}(X)$.

%%%
\begin{figure}[!ht]
\begin{center}
\begin{tikzpicture}[scale=0.32]

\draw[fill=blue!60,opacity=0.30,draw] (-6,-2.5) -- (-8,-6.5) -- (-4.58,-5.1) -- (-3,-8) -- (9,-3) -- (6,2.5) -- (8,6.5) -- (4.58,5.1) -- (3,8) -- (-9,3) -- (-6,-2.5);

\draw[blue,line width=0.01mm,dashed] (-6,-2.5) -- (6,2.5) -- (8,6.5) -- (-4,1.5) -- (-6,-2.5);
\draw[blue,line width=0.01mm,dashed] (-6,-2.5) -- (6,2.5) -- (4,-1.5) -- (-8,-6.5) -- (-6,-2.5);
\draw[blue,line width=0.01mm,dashed] (-6,-2.5) -- (-3,-8) -- (9,-3) -- (6,2.5) -- (-6,-2.5);
\draw[blue,line width=0.01mm,dashed] (-6,-2.5) -- (-9,3) -- (3,8) -- (6,2.5) -- (-6,-2.5);

\draw[blue,line width=0.05mm] (6,2.5) -- (9,-3) -- (-3,-8) -- (-9,3) -- (3,8) -- (4.58,5.1);
\draw[blue,line width=0.05mm] (6,2.5) -- (8,6.5) -- (-4,1.5) -- (-6,-2.5) -- (0,0);
\draw[blue,line width=0.05mm] (-6,-2.5) -- (-8,-6.5) -- (-4.6,-5.085);
\draw[red,line width=0.35mm] (-8.4,-3.5) -- (8.4,3.5);

\draw[magenta,line width=0.35mm] (-4.15,4.98) -- (-5.5,6.6);
\draw[magenta,line width=0.35mm,dashed] (0,0) -- (-4.2,5.04);
\draw[magenta,line width=0.35mm] (5.5,-6.6) -- (0,0);

\draw[line width=0.25mm] (-9.5,0) -- (-7.35,0);
\draw[line width=0.25mm,dashed] (-7.1,0) -- (3.375,0);
\draw[->,line width=0.25mm] (0,0) -- (9,0);
\draw[line width=0.25mm] (0,-9) -- (0,-6.75);
\draw[line width=0.25mm,dashed] (0,-7.8) -- (0,3);
\draw[->,line width=0.25mm] (0,3.15) -- (0,9.6);
\draw[line width=0.25mm] (7,7) -- (5.45,5.45);
\draw[line width=0.25mm,dashed] (5.22,5.22) -- (0,0);
\draw[->,line width=0.25mm] (0,0) -- (-7,-7);

\draw (-7,-7.5) node{\small$\!\!\!\x_1$};
\draw (9.45,-0.75) node{\small$\;\,\x_2$};
\draw (0.75,9.6) node{\small$\;\;\,\x_3$};
\draw (3.9,8.7) node{\small\color{blue}$X_0$};
\draw (8.9,7.1) node{\small\color{blue}$X_2$};
\draw (8.6,-7.6) node{\small\color{blue}$X_1=B_\Q(X)$};
\draw (8.9,7.1) node{\small\color{blue}$X_2$};
\draw (-9,5.2) node{\color{blue}$X$};
\draw (-9.5,-4.4) node{\color{red}$X_1'$};
\end{tikzpicture}
\end{center}
\caption{The $\Q$-algebraic set $X:=\zcl_{R^3}^\Q(X_0)=X_0\cup X_1\cup X_2\subset R^3$ such that $\Sing(X)=X_1\cup X_1'$, $B_\Q(X)=X_1$ and $\Sing^{R|\Q}(X)=X_1\cup X_1'$.}
\label{im:poly2'}
\end{figure}
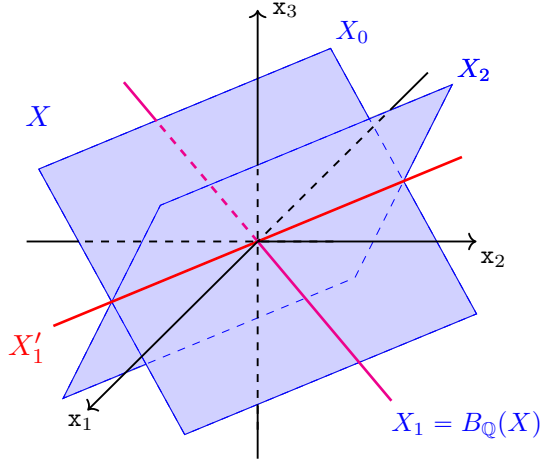
%%%

$(\mr{v})$ Consider again the polynomial $g\in\qr[\x]$, the planes $X_0=\ZZ_R(g)$ and $X_2$, the lines $X_1$ and $X_1'=X_0\cap X_2$ of $R^3$, and $X=\zcl_{R^3}^\Q(X_0)\subset R^3$ of the previous item.

Let $p,q\in\qr[\x]$ be the $\qr$-geometric polynomials in $R^3$ defined by
$$
p:=\x_1+\sqrt{2}\x_2+\x_3, \qquad q:=gp,
$$
let $U_0:=\ZZ_R(p)$ and let $U:=\zcl_{R^3}^\Q(U_0)$ and $S:=\zcl_{R^3}^\Q(X_0\cup U_0)=X\cup U$ be the real Galois completions of $U_0\subset R^3$ and $X_0\cup U_0=\ZZ_R(q)\subset R^3$, respectively. Consider the plane:
$$
U_1:=\{x\in R^3:x_1-\sqrt{2}x_2+x_3=0\}.
$$
In Examples \ref{exa:gc}$(\mr{ii})$, we proved: $U=U_0\cup U_1$, $X$ and $U$ are the two $\Q$-irreducible components of $S$, $S=X_0\cup X_2\cup U_0\cup U_1$ (because $S=X\cup U$, $X=X_0\cup X_2\cup X_1$ and $X_1\subset U_1$) and $B_\Q(S)=X_1$. As $S\subset R^3$ is the union of the four distinct planes $X_0$, $X_2$, $U_0$ and $U_1$, its singular locus $\Sing(S)$ is the union of the six intersections of pairs of these planes. Actually, $\Sing(S)$ is the union of the following four distinct lines $r_1,r_2,r_3,r_4$: 
\begin{align*}
X_0\cap X_2&=X_0\cap U_0=X_2\cap U_0=X_1'=:r_1,\\
U_0\cap U_1&=\{x\in R^3:x_2=0,x_1+x_3=0\}=:r_2,\\
X_0\cap U_1&=\{x\in R^3:x_1+\sqrt{2}x_2+\sqrt[4]{2}x_3=0,x_1-\sqrt{2}x_2+x_3=0\}=:r_3,\\
X_2\cap U_1&=\{x\in R^3:x_1+\sqrt{2}x_2-\sqrt[4]{2}x_3=0,x_1-\sqrt{2}x_2+x_3=0\}=:r_4,
\end{align*}

Denote $r_5$ the line $X_1$. Theorem \ref{regreg} implies that $\Sing^{R|\Q}(S)=\Sing(S)\cup B_\Q(S)$ is equal to the union of the five distinct lines $r_1,\ldots,r_5$. Again, the same equality can be obtained using Corollary \ref{3113b}: $\Sing^{R|\Q}(S)=\{x\in X:\nabla q^\bullet(x)=0\}=\bigcup_{i=1}^5r_i$, where $q^\bullet$ is the Galois completion of $q$ (see Examples \ref{exa:gc}$(\mr{ii})$ for the explicit expression of $q^\bullet$). Consequently,
$$
\Sing(S)=r_1\cup r_2\cup r_3\cup r_4, \qquad B_\Q(S)=r_5, \qquad \Sing^{R|\Q}(S)=r_1\cup r_2\cup r_3\cup r_4\cup r_5,
$$
see Figure \ref{im:poly2''}.

%%%
\begin{figure}[!ht]
\begin{center}
\begin{tikzpicture}[scale=0.32]

\draw[fill=blue!60,opacity=0.30,draw] (-3,-8) -- (-0.75,-7.05) -- (0,0) -- (-6,-2.5) -- (-3,-8);
\draw[fill=blue!60,opacity=0.30,draw] (0,0) -- (-6,-2.5) -- (-4,1.5) -- (2.7,4.30) -- (0,0);
\draw[fill=blue!60,opacity=0.30,draw] (6,2.5) -- (9,-3) -- (1.86,-6) -- (3.75,1.58) -- (6,2.5);
\draw[fill=blue!60,opacity=0.30,draw] (6,2.5) -- (3.75,1.58) -- (4,2.5) -- (2.7,4.30) -- (8,6.5) -- (6,2.5);
\draw[fill=blue!60,opacity=0.30,draw] (-6,-2.5) -- (-8,-6.5) -- (-5.415,-5.415) -- (-6,-2.5);
\draw[fill=blue!60,opacity=0.30,draw] (0.75,7.05) -- (-9,3) -- (-7.4,0) -- (-6,-2.5) -- (-7,2.45) -- (0.6,5.6) -- (0.75,7.05);
\draw[fill=blue!60,opacity=0.30,draw] (0.75,7.05) -- (1.5,6) -- (3.62,6.875) -- (3,8) -- (0.75,7.05);

\draw[fill=red!60,opacity=0.30,draw] (-0.75,-7.05) -- (0,0) -- (2.7,4.30) -- (4,2.5) -- (1,-9.5) -- (-0.75,-7.05);
\draw[fill=red!60,opacity=0.30,draw] (-6,-2.5) -- (-5,-7.45) -- (-3.61,-6.875) -- (-6,-2.5);
\draw[fill=red!60,opacity=0.30,draw] (-6,-2.5) -- (-7,2.45) -- (1.5,6) -- (0.9,3.54) -- (-4,1.5) -- (-6,-2.5);
\draw[fill=red!60,opacity=0.30,draw] (1.5,6) -- (0.75,7.05) -- (0.6,5.6) -- (1.5,6);
\draw[fill=red!60,opacity=0.30,draw] (0.75,7.05) -- (-1,9.5) -- (-1.875,5.975) -- (0.75,7.05);
\draw[fill=red!60,opacity=0.30,draw] (1.5,6) -- (0.9,3.54) -- (2.7,4.30) -- (1.5,6);
\draw[fill=red!60,opacity=0.30,draw] (1.5,6) -- (2.7,4.30) -- (5.4,5.4) -- (5,7.45) -- (1.5,6);

\draw[blue,line width=0.01mm,dashed] (6,2.5) -- (4,-1.5) -- (-8,-6.5) -- (-6,-2.5);
\draw[blue,line width=0.01mm,dashed] (-6,-2.5) -- (-9,3) -- (3,8) -- (6,2.5);

\draw[blue,line width=0.05mm] (6,2.5) -- (9,-3) -- (1.86,-6);
\draw[blue,line width=0.05mm] (-0.75,-7.05) -- (-3,-8) -- (-9,3) -- (3,8) -- (3.65,6.81);
\draw[blue,line width=0.05mm] (6,2.5) -- (8,6.5) -- (-4,1.5) -- (-6,-2.5) -- (0,0);
\draw[blue,line width=0.05mm] (-6,-2.5) -- (-8,-6.5) -- (-5.415,-5.415);
\draw[red,line width=0.35mm,dashed] (4.05,1.68) -- (0,0);

\draw[red,line width=0.01mm,dashed] (-6,-2.5) -- (-7,2.45) -- (5,7.45) -- (6,2.5);
\draw[red,line width=0.05mm,dashed] (-6,-2.5) -- (-5,-7.45) -- (7,-2.45) -- (6,2.5);
\draw[red,line width=0.35mm,dashed] (1.45,5.8) -- (-1.8,-7.2);
\draw[red,line width=0.01mm,dashed] (-1.5,-6) -- (1,-9.5) -- (4,2.5) -- (1.5,6);
\draw[red,line width=0.01mm,dashed] (1.5,6) -- (-1.5,-6) -- (1,-9.5) -- (4,2.5) -- (1.5,6);
\draw[red,line width=0.01mm,dashed] (1.5,6) -- (-1,9.5) -- (-4,-2.5) -- (-1.5,-6);

\draw[red,line width=0.05mm] (0.887,3.55) -- (1.5,6) -- (4,2.5) -- (1,-9.5) -- (-0.75,-7.05);
\draw[red,line width=0.05mm] (5.4,5.4) -- (5,7.45) -- (-7,2.45) -- (-5,-7.45) -- (-3.6,-6.87);
\draw[red,line width=0.05mm] (1.5,6) -- (-1,9.5) -- (-1.875,6);
\draw[red,line width=0.35mm] (0.887,3.55) -- (2.35,9.4);

\draw[blue,line width=0.35mm] (-4.17,5.002) -- (-5,6);
\draw[blue,line width=0.35mm,dashed] (0,0) -- (-4.2,5.04);
\draw[blue,line width=0.35mm] (5,-6) -- (0,0);

\draw[red,line width=0.35mm] (5.13,8.2) -- (0,0);
\draw[red,line width=0.35mm,dashed] (0,0) -- (-3.25,-5.2);
\draw[red,line width=0.35mm,dashed] (0.75,7.05) -- (0,0);
\draw[red,line width=0.35mm] (0,0) -- (-0.91,-8.5);
\draw[red,line width=0.35mm] (0.6,5.6) -- (1.06,10);
\draw[red,line width=0.35mm] (-8.2,-3.42) -- (0,0);
\draw[red,line width=0.35mm] (4.05,1.68) -- (8.2,3.42);

\draw[line width=0.25mm] (-9.5,0) -- (-7.38,0);
\draw[line width=0.25mm,dashed] (-7.1,0) -- (0,0);
\draw[line width=0.25mm,dashed] (0,0) -- (3.175,0);
\draw[->,line width=0.25mm] (3.375,0) -- (9.3,0);
\draw[line width=0.25mm] (0,-9) -- (0,-8.08);
\draw[line width=0.25mm,dashed] (0,-7.8) -- (0,0);
\draw[line width=0.25mm,dashed] (0,0) -- (0,3.175);
\draw[->,line width=0.25mm] (0,3.175) -- (0,9.9);
\draw[line width=0.25mm] (7,7) -- (5.45,5.45);
\draw[line width=0.25mm,dashed] (5.25,5.25) -- (0,0);
\draw[->,line width=0.25mm] (0,0) -- (-7,-7);

\draw (-6.8,-7.8) node{\small$\!\!\!\x_1$};
\draw (9.45,-0.75) node{\small$\;\;\,\x_2$};
\draw (-0.4,10.5) node{\small$\x_3$};
\draw (-2.2,-8.6) node{\small\color{blue}$X_0$};
\draw (9,-6.8) node{\small\color{blue}$X_1=B_\Q(S)=r_5$};
\draw (-8.6,-5.4) node{\small\color{blue}$X_2$};
\draw (1.7,-10.3) node{\small\color{red}$U_1$};
\draw (-4.5,-8.2) node{\small\color{red}$U_0$};
\draw (-9.8,4.2) node{\color{magenta}$S$};
\draw (1.4,10.5) node{\small\color{red}$r_3$};
\draw (2.9,9.8) node{\small\color{red}$r_2$};
\draw (5.7,8.6) node{\small\color{red}$r_4$};
\draw (10.15,4.05) node{\small\color{red}$X_1'=r_1$};
\end{tikzpicture}
\end{center}
\caption{The $\Q$-algebraic set $S:=\zcl_{R^3}^\Q(X_0\cup U_0)=X_0\cup X_2\cup U_0\cup U_1\subset R^3$ such that $\Sing(S)=r_1\cup r_2\cup r_3\cup r_4$, $B_\Q(S)=r_5$, and $\Sing^{R|\Q}(S)=r_1\cup r_2\cup r_3\cup r_4\cup r_5$.}
\label{im:poly2''}
\end{figure}
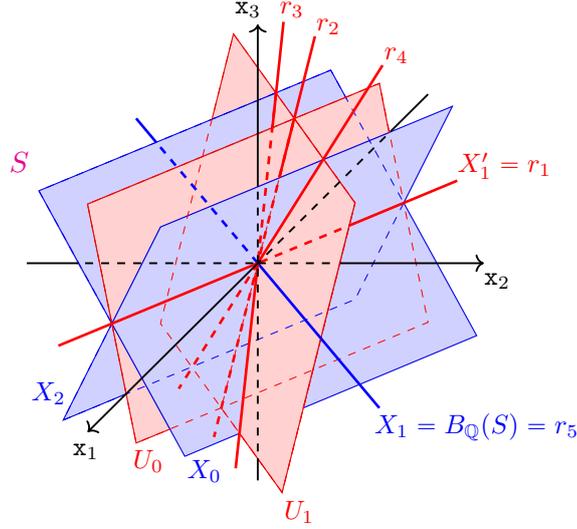

Observe that $B_\Q(S)=r_5\subset R^3$ is $\qr$-algebraic, but not $\Q$-algebraic. The same is true for $\Sing(X)=r_1\cup r_2\cup r_3\cup r_4$. To prove the latter two assertions, we can use Corollary \ref{cor:rk-algebraicity}. Define:
\begin{align*}
\mc{Z}&:=\zcl_{C^3}(r_1\cup r_2\cup r_3\cup r_4),\\
Z_1&:=\zcl_{C^3}(r_1)=\{x\in C^3:x_1+\sqrt{2}x_2=0,x_3=0\}\subset\mc{Z},\\
Z_5&:=\zcl_{C^3}(r_5)=\{x\in C^3:x_1-\sqrt{2}x_2=0,x_3=0\}.
\end{align*}
Let $E:=\Q(\sqrt[4]{2},\ii)$ and let $G':=G(E:\Q)=\{\sigma_{ab}\}_{a\in\{0,1,2,3\},b\in\{0,1\}}$ be as in Examples \ref{exa:gc}, so $\sigma_{ab}(\sqrt[4]{2})=\ii^a\sqrt[4]{2}$ and $\sigma_{ab}(\ii)=(-1)^b\ii$. Choose an automorphism in $G:=G(C:\Q)$ that extends $\sigma_{ab}$ (see Lemma \ref{extension}) and denote it again $\sigma_{ab}$. Define $(\sigma_{ab})_3$ and $\widehat{\sigma_{ab}}$ as in \eqref{psi} and \eqref{hatpsi}. Observe that $\sigma_{10}(\sqrt{2})=-\sqrt{2}$, so $\widehat{\sigma_{10}}(\x_1\pm\sqrt{2}\x_2)=\x_1\mp\sqrt{2}\x_2$ and $\widehat{\sigma_{10}}(\x_3)=\x_3$. By Theorem \ref{thm:gc0}$(\mr{i})$, we have:
\begin{align*}
(\sigma_{10})_3(Z_5)\cap R^3&=\{x\in R^3:\widehat{\sigma_{10}}(x_1-\sqrt{2}x_2)=0,\widehat{\sigma_{10}}(x_3)=0\}=r_1=Z_1\cap R^3,\\
(\sigma_{10})_3(Z_1)\cap R^3&=\{x\in R^3:\widehat{\sigma_{10}}(x_1+\sqrt{2}x_2)=0,\widehat{\sigma_{10}}(x_3)=0\}=r_5=Z_5\cap R^3.
\end{align*}
It follows that $(\sigma_{10})_3(Z_5)\cap R^3=r_1\not\subset r_5$ and $(\sigma_{10})_3(Z_1)\cap R^3=r_5\not\subset r_1\cup r_2\cup r_3\cup r_4$, so $(\sigma_{10})_3(\mc{Z})\cap R^3\not\subset r_1\cup r_2\cup r_3\cup r_4$ as well. We deduce that $B_\Q(S)=r_5$ and $\Sing(X)=r_1\cup r_2\cup r_3\cup r_4$ are both not $G$-stable, so they are not $\Q$-algebraic by Corollary \ref{cor:rk-algebraicity}.

Observe that the line $r_2$ of $R^2$ is $\Q$-algebraic and irreducible (so $\Q$-irreducible). By Theorem \ref{thm:gc}$(\mr{i}')(\mr{iv})$, we know that $\zcl_{R^3}^\Q(r_h)=\bigcup_{a\in\{0,1,2,3\},b\in\{0,1\}}(\sigma_{ab})_3(\zcl_{C^3}(r_h))$ for each $h\in\{1,3,4,5\}$, so $\zcl_{R^3}^\Q(r_1)=\zcl_{R^3}^\Q(r_5)=r_1\cup r_5$ and $\zcl_{R^3}^\Q(r_3)=\zcl_{R^3}^\Q(r_4)=r_3\cup r_4$. As the lines $r_1$ and $r_3$ of $R^2$ are irreducible (so $\qr$-irreducible), Lemma \ref{lem:irreducibility} assures that the $\Q$-algebraic sets $r_1\cup r_5\subset R^3$ and $r_3\cup r_4\subset R^3$ are $\Q$-irreducible. This proves that the sets $r_1\cup r_5$, $r_3\cup r_4$ and $r_2$ are the $\Q$-irreducible components of $\Sing^{R|\Q}(S)\subset R^3$. $\sqbullet$
\end{examples}

The next result is a criterion for a nonsingular point to be $K$-nonsingular, see Definitions \ref{ek-Zar-tang} and \ref{E|K-regular}. 

\begin{prop}
Let $X\subset R^n$ be a $K$-algebraic set and let $a\in\Reg(X)$. Then $a\in\Reg^K(X)$ if and only if $T^K_a(X)=T_a(X)$. 
\end{prop}
\begin{proof}
Let $d:=\dim(X)$. If $d=n$, then $X=R^n$ and the statement is obvious. Suppose $d<n$.

If $a\in\Reg^K(X)$, then Remarks \ref{rem528}$(\mr{ii})$ implies that $\dim_{R}(T_a(X))=d=\dim_{R}(T^K_a(X))$. As $T^K_a(X)\supset T_a(X)$, we have $T^K_a(X)=T_a(X)$.

Suppose that $T^K_a(X)=T_a(X)$, so $\dim_R(T^K_a(X))=d$. Choose $f_1,\ldots,f_{n-d}\in\II_K(X)$ such that $T^K_a(X)=\{v\in R^n:\langle\nabla f_1(a),v\rangle=0,\ldots,\langle\nabla f_{n-d}(a),v\rangle=0\}$ and set $X':=\ZZ_R(f_1,\ldots,f_{n-d})\subset R^n$. Equip each subset of $R^n$ with the Euclidean topology. By the Implicit Function Theorem \cite[Cor.2.9.8]{bcr}, there exists an open neighborhood $V$ of $a$ in $R^n$ such that the intersections $X\cap V$ and $X'\cap V$ are semialgebraically connected Nash submanifolds of $R^n$ of the same dimension $d$ and also closed subsets of $V$. As $X\cap V\subset X'\cap V$, it follows that $X\cap V=X'\cap V$, so $a\in\Reg^K(X)$ by the equivalence $(\mr{i})\Longleftrightarrow(\mr{iii}')$ of Proposition \ref{30}.
\end{proof}

A $K=\Q$ version of the latter result was proved in \cite[Prop.2.13]{GS} by a slightly different argument.

We conclude this section by comparing the ideals $\II_K(X)R[\x]$ and $\II_R(X)$ of $R[\x]$ for each $K$-algebraic set $X\subset R^n$ off the $K$-bad set $B_K(X)$. Given $f\in R[\x]$, we denote $R[\x]_f$ the quotient ring of $R[\x]$ with respect to its multiplicative subset $\{f^n\}_{n\in\N}$.

\begin{prop}\label{S_0}
If $X\subset R^n$ is a $K$-algebraic set, 
then there exists $f\in\kr[\x]$ such that $\ZZ_R(f)=B_K(X)$ 
and $\II_K(X)R[\x]_f=\II_R(X)R[\x]_f$.
\end{prop}
\begin{proof}
We keep the notations used in the proof of Theorem \ref{regreg}. Let $(i,\sigma)\in\mc{F}^*\cup\mc{G}$. By Lemmas \ref{lem:gp-reducible}$(\mr{i})$ and \ref{lem:ts}, we know that $Z_i^\sigma\subset C^n$ is $\kbar$-algebraic and $\ZZ_R(\II_\kr(Z_i^\sigma))=Z_i^\sigma\cap R^n$. Thus, there exists $f_{i,\sigma}\in\II_\kr(Z_i^\sigma)\subset\II_R(Z_i^\sigma)$ such that $\ZZ_R(f_{i,\sigma})=Z_i^\sigma\cap R^n$. Define $f:=\prod_{(i,\sigma)\in\mc{F}^*\cup\mc{G}}f_{i,\sigma}\in\kr[\x]$. Observe that $f\in\bigcap_{(i,\sigma)\in\mc{F}^*\cup\mc{G}}\II_R(Z_i^\sigma)\subset\bigcap_{(i,\sigma)\in\mc{F}^*\cup\mc{G}}\II_R(Z_i^\sigma\cap R^n)$ and $\ZZ_R(f)=B_K(X)$. As $\dim(B_K(X))<\dim(X)$ and $\dim(Z_i^\sigma\cap R^n)=\dim(X)$ for each $(i,\sigma)\in\mc{F}\setminus\mc{F}^*$, we deduce $f\not\in \II_R(Z_i^\sigma)$ for each $(i,\sigma)\in\mc{F}\setminus\mc{F}^*$. Recall that $\II_R(Z^\sigma_i)=\II_R (Z^\sigma_i\cap R^n)$ for each $(i,\sigma)\in\mc{F}\setminus\mc{F}^*$ (see \textsc{Step II} of the proof of Theorem \ref{regreg}). By \eqref{iqr}, \eqref{iqr2} and \cite[Cor.3.4.ii)]{am}, we have:
\begin{align*}
\II_K(X)R[\x]_f&\textstyle=\bigcap_{(i,\sigma)\in\mc{F}\cup\mc{G}}\II_R(Z_i^\sigma)R[\x]_f=\bigcap_{(i,\sigma)\in\mc{F}\setminus\mc{F}^*}\II_R(Z_i^\sigma) R[\x]_f\\
&\textstyle=\bigcap_{(i,\sigma)\in\mc{F}\setminus\mc{F}^*}\II_R(Z_i^\sigma\cap R^n)R[\x]_f=\bigcap_{(i,\sigma)\in\mc{F}\cup\mc{G}}\II_R(Z_i^\sigma\cap R^n)R[\x]_f\\
&=\II_R(X)R[\x]_f,
\end{align*}
as required.
\end{proof}

\begin{remarks}
Let $X\subset R^n$ be a $K$-algebraic set and let $f\in\kr[\x]$ be such that $\ZZ_R(f)=B_K(X)$ and $\II_K(X)R[\x]_f=\II_R(X)R[\x]_f$, as in the previous proposition.

$(\mr{i})$ Let $V_1,\ldots,V_h$ be the ($R$-)irreducible components of $X\subset R^n$ of dimension $\dim(X)$. We have that $\{\II_R(V_i)R[\x]_f\}_{i=1}^h$ is the family of minimal prime ideals of $R[\x]_f$ associated to the ideal $\II_K(X)R[\x]_f=\II_R(X)R[\x]_f$. Indeed, by the preceding proof and by \textsc{Step II} of the proof of Theorem \ref{regreg}, we know that $\II_K(X)R[\x]_f=\II_R(X)R[\x]_f=\bigcap_{(i,\sigma)\in\mc{F}\setminus\mc{F}^*}\II_R(Z_i^\sigma\cap R^n)R[\x]_f$ and $\{Z^\sigma_i\cap R^n\}_{(i,\sigma)\in\mc{F}\setminus\mc{F}^*}$ is the family of irreducible components of $X\subset R^n$ of dimension $\dim(X)$, that is, $\{Z^\sigma_i\cap R^n\}_{(i,\sigma)\in\mc{F}\setminus\mc{F}^*}=\{V_1,\ldots,V_h\}$.

$(\mr{ii})$ Suppose that $B_K(X)=\varnothing$, so $f$ never vanishes on $R^n$. Denote $\mc{R}(X)$ the usual ring of regular functions on the ($R$-)algebraic set $X\subset R^n$, that is, the functions $g:X\to R$ of the form $g(x)=p(x)q(x)^{-1}$ for all $x\in X$, where $p,q\in R[\x]$ and $q$ never vanishes on $X$ (see \cite[{Sec.3.2}]{bcr}). Then $R[\x]_f/(\II_K(X)R[\x]_f)=R[\x]_f/(\II_R(X)R[\x]_f)$ can be identified with a subring of $\mc{R}(X)$: it is enough to consider the natural ring embedding that maps $\frac{g}{f^m}+\II_R(X)R[\x]_f$ to the regular function $X\to R:x\mapsto g(x)f(x)^{-m}$. $\sqbullet$
\end{remarks}

%%%

\section{Applications of the theory: some examples}\label{s5}
In this section, we present four applications of the theory we have developed above.

\subsection{$K$-rational points of a real $K$-algebraic set and their $K$-Zariski density}\label{s51}
\emph{Let $R$ be a real closed field and let $K$ be an ordered subfield of $R$ (that is, $K$ is a subfield of $R$ endowed with the ordering induced by that of~$R$). Denote $C:=R[\ii]$ the algebraic closure of~$R$, $\kr\subset R$ the real closure of $K$ and $\kbar:=\kr[\ii]\subset C$ the algebraic closure of $K$. Assume that $K$ is not a real closed field, that is, $K\subsetneqq\kr$. We will use Remark \emph{\ref{dime}} freely, and set $\dim(S):=\dim_R(S)$ for each algebraic set $S\subset R^n$.}

Let us borrow a classical notion from diophantine geometry \cite[p.2]{lang}.

\begin{defn}
Let $X\subset R^n$ be a $K$-algebraic set. We say that a point $x$ of $X$ is a \emph{$K$-rational point of $X$} if it belongs to $X(K):=X\cap K^n$. $\sqbullet$
\end{defn}

In the present $R|K$-context, the main diophantine problem is to determine the set $X(K)$ of the $K$-rational points of a $K$-algebraic set $X\subset R^n$. Part of this problem involves estimating the size of $X(K)$ within $X$. A subset $S$ of $X$ is \textit{$K$-Zariski dense in $X$} if its $K$-Zariski closure in $R^n$ coincides with $X$, that is, $\zcl_{R^n}^K(S)=X$. The next result is a version of \cite[{Prop.C.5, $(\mr{i})\Longrightarrow(\mr{ii})$}]{GS}.

\begin{lem}\label{lem:enrico}
Let $X\subset R^n$ be a $K$-algebraic set. If $X(K)$ is $K$-Zariski dense in $X$, then $X\subset R^n$ is defined over $K$, that is, $\II_R(X)=\II_K(X)R[\x]$ (see Definition \ref{def:overK}). Equivalently, if $X\subset R^n$ is not defined over $K$, then $X(K)$ is not $K$-Zariski dense in $X$. 
\end{lem}
\begin{proof}
We follow the proof of \cite[{Prop.C.5, $(\mr{i})\Longrightarrow(\mr{ii})$}]{GS}. Let $f \in\II_R(X)$ and let $\{u_j\}_{j\in J}$ be a basis of $R$ as a $K$-vector space. By Lemma \ref{k0}, the polynomial $f\in R[\x]$ can be written uniquely as $f=\sum_{j\in J}u_jf_j$, where $f_j\in K[\x]$ and only finitely many of the polynomials $f_j$ are nonzero. If $x\in X(K)$, then $0=f(x)=\sum_{j\in J}u_jf_j(x)$ and $f_j(x)\in K$ for all $j\in J$, so $f_j(x)=0$. As $X(K)$ is $K$-Zariski dense in $X$, each polynomial $f_j\in K[\x]$ belongs to $\II_K(X)$, so $f\in\II_K(X)R[\x]$. Consequently, $\II_R(X)=\II_K(X)\R[\x]$, as required. 
\end{proof}

\begin{remark}
Let $F_h\subset R^2$ be the Fermat curve $\ZZ_R(\x_1^{2h}+\x_2^{2h}-2^{h})$ with $h\geq3$. The previous result cannot be used to prove that $F_h(\Q)$ is not $\Q$-Zariski dense in $F_h$ and is therefore finite. Indeed, by Proposition \ref{prop:hyper}, the polynomial $\x_1^{2h}+\x_2^{2h}-2^{h}\in\Q[\x_1,\x_2]$ is both $\Q$-geometric and geometric in $R^2$ so $F_h\subset R^2$ is defined over $\Q$. $\sqbullet$ 
\end{remark}

\begin{remark}
If $K$ is a real closed subfield of $R$, then $X(K)$ is $K$-Zariski dense in $X\subset R^n$ for each $K$-algebraic set $X\subset R^n$. This is due to Corollary \ref{inter}$(\mr{i})$. $\sqbullet$
\end{remark}

The next two results concern the $K$-Zariski density of $X(K)$ in $X$ when $X\subset R^n$ is $K$-irreducible. In the first, we will use the concept of Galois presentation $X=\bigcup_{\sigma\in G'}(Z^\sigma\cap R^n)$ of $X\subset R^n$, see Definition $\ref{def:gp}$. Recall that, by Lemma \ref{lem:gp}$(\mr{ii})(\mr{iii})(\mr{iv})$, $\{Z^\sigma\}_{\sigma\in G'}$ is the family (with possible repetitions) of the $C$-irreducible components of the $K$-Zariski closure of $X$ in $C^n$, each $Z^\sigma\subset C^n$ is $\kbar$-algebraic and each $Z^\sigma\cap R^n\subset R^n$ is $\kr$-algebraic.

\begin{prop}\label{prop:diophantine1}
Let $X\subset R^n$ be a $K$-irreducible $K$-algebraic set. If $X=\bigcup_{\sigma\in G'}(Z^\sigma\cap R^n)$ is a Galois presentation of $X\subset R^n$, then
\begin{equation}\label{equa61}
\textstyle
X(K)\subset\bigcap_{\sigma\in G'}(Z^\sigma\cap R^n).
\end{equation}
If in addition $X\subset R^n$ is $R$-reducible (or, equivalently, $\kr$-reducible), then
$$
\dim_R(X(K))\leq\dim_K(X(K))<\dim(X).
$$ 
\end{prop}
\begin{proof}
Let $d:=\dim(X)$, let $e$ be the identity of $G'$ and let $Z:=Z^e\subset C^n$. Recall that $G'$ denotes the Galois group of a finite Galois sub\-extension $E|K$ of $\kbar|K$ that contains all the coefficients of polynomials $g_1,\ldots,g_r\in\kbar[\x]$ such that $Z=\ZZ_C(g_1,\ldots,g_r)\subset C^n$. Moreover, we have $Z^\sigma=\ZZ_C(g_1^\sigma,\ldots,g_r^\sigma)$ for each $\sigma\in G'$, where $g_i^\sigma:=\sum_\nu\sigma(a_{i\nu})\x^\nu\in E[\x]$ if $g_i:=\sum_\nu a_{i\nu}\x^\nu\in E[\x]$.

Pick $x=(x_1,\ldots,x_n)\in X(K)$. As $X(K)=\bigcup_{\sigma\in G'}(Z^\sigma\cap K^n)$, $x$ belongs to $Z^\tau\cap K^n$ for some $\tau\in G'$. Pick any $\sigma\in G'$ and set $\sigma':=\sigma\circ\tau^{-1}\in G'$. Observe that $\sigma'(x_i)=x_i$ for $i=1,\ldots,n$, so
$$
0=\sigma'(g_i^\tau(x))=g_i^{\sigma'\tau}(\sigma'(x_1),\ldots,\sigma'(x_n))=g_i^\sigma(x)
$$
for each $i\in\{1,\ldots,r\}$. Thus, $x\in Z^\sigma\cap R^n$ for each $\sigma\in G'$, so \eqref{equa61} holds.

Suppose now $X\subset R^n$ is $R$-reducible (or, equivalently, $\kr$-reducible by Corollary \ref{inter}$(\mr{iii})$). Define $X_K:=\zcl_{R^n}^K(X(K))\subset R^n$. Let $G'^*$ be the set of all $\sigma\in G'$ such that $\dim(Z^\sigma\cap R^n)<d$. We distinguish two cases.

\textsc{Case 1.} If $G'^*=\varnothing$, then $\{Z^\sigma\cap R^n\}_{\sigma\in G'}$ is by Lemma \ref{lem:gp}$(\mr{iii}')$ the family of all $R$-irreducible components of $X\subset R^n$, each of which has dimension $d$. As $X\subset R^n$ is $R$-reducible, there exist $\sigma_1,\sigma_2\in G'$ such that $Z^{\sigma_1}\cap R^n\neq Z^{\sigma_2}\cap R^n$. Define $V:=(Z^{\sigma_1}\cap R^n)\cap(Z^{\sigma_2}\cap R^n)\subset R^n$, which has dimension $<d$, and $V_K:=\zcl_{R^n}^K(V)\subset R^n$. As $V\subset R^n$ is a $\kr$-algebraic set, Theo\-rem \ref{thm:gc}$(\mr{i}')(\mr{iv})$ assures that $\dim(V_K)=\dim(V)<d$. By \eqref{equa61}, we have $X(K)\subset V$, so $X_K\subset V_K$ and $\dim_K(X(K))=\dim_K(X_K)\leq\dim_K(V_K)=\dim(V_K)<d$.

\textsc{Case 2.} Suppose now that $G'^*\neq\varnothing$. Choose $\xi\in G'^*$ and set $W:=Z^\xi\cap R^n$ and $W_K:=\zcl_{L^n}^K(W)$. As $X(K)\subset W$, we have $X_K\subset W_K$. Repeating the preceding argument, we deduce $\dim_K(X(K))=\dim_K(X_K)\leq\dim_K(W_K)=\dim(W_K)=\dim(W)<d$.

This proves that $\dim_K(X(K))<d$. By Remark \ref{rem:dimLleqdimK}, we also know that $\dim_R(X(K))\leq\dim_K(X(K))<d$, as required.
\end{proof}

\begin{prop}
Let $X\subset R^n$ be a $K$-irreducible $K$-algebraic set and let $X_C:=\zcl_{C^n}(X)\subset C^n$ be the complexification of $X$. If $X(K)$ is $K$-Zariski dense in $X$, then $X$ satisfies the following three equivalent conditions:
\begin{itemize}
\item[$(\mr{i})$] $X\subset R^n$ is $R$-irreducible and defined over $K$.
\item[$(\mr{ii})$] $X_C\subset C^n$ is $C$-irreducible and $\II_C(X_C)=\II_K(X)C[\x]$.
\item[$(\mr{ii}')$] $X_C\subset C^n$ is $C$-irreducible and $K$-algebraic.
\end{itemize}
\end{prop}
\begin{proof}
Assume $X\subset R^n$ is a $K$-irreducible $K$-algebraic set such that $X(K)$ is $K$-Zariski dense in $X$. Let $d:=\dim(X)$. By Lemma \ref{lem:xLn}, the three conditions are equivalent. By Proposition \ref{prop:diophantine1}, $X\subset R^n$ is $R$-irreducible. Otherwise, $d=\dim_K(X)=\dim_K(X(K))<d$, which is a contradiction. By Lemma \ref{lem:enrico}, $X\subset R^n$ is defined over $K$.
\end{proof}

\subsection{Whitney regular $K$-algebraic stratifications}\label{subsec:Whitney}
\textit{Let $R$ be a real closed field and let $K$ be an ordered subfield of $R$. The main example to keep in mind is the one in which $K=\Q$. Endow $R^n$ with the Euclidean topology and each of its subsets with the relative topology.}

During the Spanish-Polish Mathematical Meeting held in 2023 in {\L}\'od\'z, we presented a first draft of this paper. On that occasion, Wies{\l}aw Paw{\l}ucki asked: (1) whether there exists a notion of stratification for $K$-algebraic subsets of $R^n$ that is natural in the context of $R|K$-algebraic geometry, and (2) whether each $K$-algebraic subset of $R^n$ admits such a stratification that is Whitney regular. We present below an affirmative solution to this problem.

\subsubsection{$K$-semialgebraic sets: some basic properties and the subfield-dimension invariance} 

Let~us introduce the concept of $K$-semialgebraic set. See also \cite[Ch.6.\S7]{abr}.

\begin{defn}
Let $S$ be a subset of $R^n$. We say that $S$ is a \emph{$K$-semialgebraic subset of $R^n$} or $S\subset R^n$ is a \emph{$K$-semialgebraic set} if it admits a description as a finite Boolean combination of polynomial equalities and inequalities with coefficients in $K$ or, equivalently, if $S=\bigcup_{i=1}^s\bigcap_{j=1}^{r_i}\{x\in R^n:f_{ij}\ast_{ij}0\}$ for some $s,r_1,\ldots,r_s\in\N^*$, $f_{ij}\in K[\x]$ and $\ast_{ij}\in\{>,=\}$.~$\sqbullet$
\end{defn}

The family of $K$-semialgebraic subsets of $R^n$ is the smallest family of subsets of $R^n$ containing all sets of the form $\{x\in R^n:f(x)>0\}$, where $f\in K[\x]$, and closed under taking finite intersections, finite unions and complements. Observe that an $R$-semialgebraic subset of $R^n$ is a usual semialgebraic subset of $R^n$.

Our next task is to generalize some basic results concerning the usual semialgebraic sets to the $K$-semialgebraic sets. To do this, we will follow Chapter 2 of \cite{bcr}.

The following result is straightforward.

\begin{lem}\label{form}
Each $K$-semialgebraic subset of $R^n$ can be written as a finite union of $K$-semialgebraic sets of the form $\{x\in R^n:f(x)=0,g_1(x)>0,\ldots,g_m(x)>0\}$ for some $m\in\N^*$ and $f,g_1,\ldots,g_m\in K[\x]$. We call these $K$-semialgebraic sets {\em basic $K$-semialgebraic sets}.
\end{lem}

We adapt \cite[Thm.2.2.1]{bcr} to obtain the following result.

\begin{thm}[Projection of $K$-semialgebraic sets]\label{projsa}
Let $S\subset R^{n+1}$ be a $K$-semialgebraic set and let $\pi:R^{n+1}\to R^n$ be the projection onto the first $n$ coordinates. Then $\pi(S)$ is a $K$-semialgebraic subset of $R^n$.
\end{thm}
\begin{proof}
By Lemma \ref{form}, we can assume that $S$ is a basic $K$-semialgebraic set, that is, 
$$
S=\{(x,x_{n+1})\in R^n\times R=R^{n+1}: f(x,x_{n+1})=0,\ g_1(x,x_{n+1})>0,\ldots,g_m(x,x_{n+1})>0\}
$$
for some $m\in\N^*$ and $f,g_1,\ldots,g_m\in K[\x,\x_{n+1}]$. By \cite[Cor.1.4.7]{bcr}, there exists a Boolean combination ${\mathcal B}(\x)$ of polynomial equalities and inequalities in the indeterminates $\x=(\x_1,\ldots,\x_n)$ and coefficients in $K$ such that, for each $x\in R^n$, the system $f(x,\x_{n+1})=0,g_j(x,\x_{n+1})>0$ for $j=1,\ldots,m$ has a solution $x_{n+1}\in R$ if and only if ${\mathcal B}(x)$ holds in $R$. As the set of all $x\in R^n$ satisfying ${\mathcal B}(x)$ is $K$-semialgebraic, the result is proved.
\end{proof}

\begin{defn}
A \emph{first-order $K$-formula of the language of ordered fields with parameters in $R$} is a formula written with a finite number of conjunctions, disjunctions, negations, and universal or existential quantifiers on variables, starting from atomic $K$-formulas which are formulas of the kind $f(x)=0$ or $g(x)>0$, where $f$ and $g$ are polynomials in $K[\x]$. The free variables of a $K$-formula are those variables of the polynomials appearing in the formula, which are not quantified. $\sqbullet$
\end{defn}

By the very definition, each $K$-semialgebraic set is described by quantifier free first-order $K$-formulas of the language of ordered fields with parameters in $R$. The properties of stabi\-lity of $K$-semialgebraic sets under taking finite intersections, finite unions, complements and projections can be expressed as described in the next result, which is an adaptation of \cite[Prop.2.2.4]{bcr} (see also \cite[\S1]{jrs}).

\begin{lem}\label{firstorder}
Let $\Phi(\x)$ be a first-order $K$-formula of the language of ordered fields with parame\-ters in $R$ and with free variables $\x=(\x_1,\ldots,\x_n)$. Then $\{x\in R^n:\Phi(x)\}$ is a $K$-semialgebraic subset of $R^n$.
\end{lem}
\begin{proof}
One can proceed by induction on the number of steps used to construct the $K$-formula, starting from atomic $K$-formulas. As the family of $K$-semialgebraic subsets of $R^n$ are stable under finite intersections, finite unions and complements, the cases in which $\Phi(\x)$ is a conjunction or a disjunction or a negation are trivial. If $\Phi(\x)$ is of the form $\exists\x_{n+1}\ \Psi(\x,\x_{n+1})$ and the set $S:=\{(x,x_{n+1})\in R^{n+1}:\Psi(x,x_{n+1})\}$ is $K$-semialgebraic, then $\{x\in R^n:\Phi(x)\}$ is the projection of $S$ onto the first $n$ coordinates, which is $K$-semialgebraic by Theorem \ref{projsa}. Finally, the case of universal quantifier reduces to that of existential quantifier, because ``$\forall\x_{n+1}\ldots$'' is equivalent to ``not $\exists\x_{n+1}$ not$\ldots$''.
\end{proof}

Let us give the definition of $K$-semialgebraic map.

\begin{defn}
We say that a map $f:S\to T$ between $K$-semialgebraic sets $S\subset R^n$ and $T\subset R^m$ is \emph{$K$-semialgebraic} if its graph is a $K$-semialgebraic subset of $R^{n+m}=R^n\times R^m$. $\sqbullet$
\end{defn}

We adapt the first part of \cite[Prop.2.2.7]{bcr} to obtain the following result.

\begin{prop}\label{prop625}
Let $S\subset R^n$ and $T\subset R^m$ be $K$-semialgebraic sets, and let $f:S\to T$ be a $K$-semialgebraic map. If $A\subset R^n$ is a $K$-semialgebraic set contained in $S$, then $f(A)\subset R^m$ is a $K$-semialgebraic set.
\end{prop}
\begin{proof}
Let $\pi_2:R^{n+m}=R^n\times R^m\to R^m$ be the projection $\pi_2(x,y):=y$ and let $\Gamma\subset R^{n+m}$ be the graph of $f$. As $f(A)=\pi_2(\Gamma\cap(A\times R^m))$, the result follows from Theorem \ref{projsa}.
\end{proof}

If $S$ is a subset of $R^n$ and $x=(x_1,\ldots,x_n)$ is a vector of $R^n$, we denote $\cl_{R^n}(S)$ the (Euclidean) closure of $S$ in $R^n$ and we set $\|x\|:=\sqrt{x_1^2+\cdots+x_n^2}$.

We adapt \cite[Prop.2.2.2]{bcr} to obtain the following result.

\begin{prop}\label{clos}
If $S\subset R^n$ is $K$-semialgebraic, then $\cl_{R^n}(S)\subset R^n$ is also $K$-semialgebraic. 
\end{prop}
\begin{proof}
Let $\pi_1:R^{2n+1}=R^n\times R^n \times R\to R^{n+1}=R^n\times R$ and $\pi_2:R^{n+1}=R^n\times R\to R^n$ be 
the projections $\pi_1(x,y,t):=(x,t)$ and $\pi_2(x,t):=x$. Define the $K$-semialgebraic sets $V\subset R^{2n+1}$ and $H\subset R^{n+1}$ as follows:
$$
V:=\{(x,y,t)\in R^{2n+1}:t>0,y\in S,\|y-x\|^2<t\}
$$
and
$$
H:=\{(x,t)\in R^{n+1}:t>0\}.
$$
It holds $\cl_{R^n}(S)=R^n\setminus(\pi_2(H\setminus\pi_1(V)))$, so $\cl_{R^n}(S)\subset R^n$ is $K$-semialgebraic by Proposition \ref{prop625}, as required.
\end{proof}

Let us introduce the concept of $K$-semialgebraically connected $K$-semialgebraic set.

\begin{defn}
Let $S\subset R^n$ be a $K$-semialgebraic set. We say that $S$ is \emph{$K$-semialgebraically connected} if there do not exist two $K$-semialgebraic sets $C_1\subset R^n$ and $C_2\subset R^n$ such that $C_1$ and $C_2$ are disjoint proper closed subsets of $S$ and $C_1\cup C_2=S$. $\sqbullet$ 
\end{defn}

If $K=R$, then $K$-semialgebraically connected $K$-semialgebraic sets coincide with the usual semialgebraically connected semialgebraic sets (see \cite[Def.2.4.2]{bcr}). If $S\subset R^n$ is a usual semialgebraic set, then \cite[Thm.2.4.4]{bcr} assures that $S$ is the disjoint union of finitely many semialgebraically connected semialgebraic sets, which are both closed and open in $S$ and are uniquely determined by $S$. They are called \emph{semialgebraically connected components of $S$}. If in addition $S\subset R^n$ is $K$-semialgebraic, then the semialgebraically connected components of $S$ are also $K$-semialgebraic by \cite[Main Thm, p.122]{jrs} (see also \cite[pp.168-169]{abr}). This result can be completed as follows.

\begin{prop}\label{K-semialg-connected}
Let $S\subset R^n$ be a $K$-semialgebraic set. We have:
\begin{itemize}
\item[$(\mr{i})$] Let $S_1,\ldots,S_\ell$ be the $R$-semialgebraically connected components of the set $S\subset R^n$, viewed as a usual $R$-semialgebraic set. Then each set $S_i\subset R^n$ is $K$-semialgebraic and $K$-semi\-algebraically connected. We say that $S_1,\ldots,S_\ell$ are the \emph{$K$-semialgebraically conne\-cted components of $S\subset R^n$}.
\item[$(\mr{i}')$] The $R$-semialgebraically connected components and the $K$-semialgebraically connected components of $S\subset R^n$ coincide.
\item[$(\mr{ii})$] $S\subset R^n$ is $R$-semialgebraically connected if and only if it is $K$-semialgebraically connected.
\end{itemize}
\end{prop}
\begin{proof}
$(\mr{i})$ As we said, each $S_i\subset R^n$ is $K$-semialgebraic by \cite[Main Thm, p.122]{jrs}. As $S_i\subset R^n$ is $R$-semialgebraically connected, it is also $K$-semialgebraically connected.

$(\mr{i}')$ follows immediately from $(\mr{i})$.

$(\mr{ii})$ The `only if' implication is evident. Let us prove the `if' implication. Suppose that $S\subset R^n$ is $K$-semialgebraically connected but not $R$-semialgebraically connected. Then $S\subset R^n$ has finitely many and at least two $R$-semialgebraically connected components, which are closed in $S$ and $K$-semialgebraic by $(\mr{i})$. Thus, $S\subset R^n$ is not $K$-semialgebraically connected, which is a contradi\-ction.
\end{proof}

\begin{remark}
Let $S\subset R^n$ be a $K$-semialgebraic set and let $S_1,\ldots,S_\ell$ be its $K$-semialgebraically connected components. The $K$-semialgebraically connectedness of each subset $S_i$ of $S$ is maxi\-mal in the following sense: if $C\subset R^n$ is a $K$-semialgebraically connected $K$-semialgebraic set such that $C\subset S$ and $C\cap S_i\neq\varnothing$ for some $i\in\{1,\ldots,\ell\}$, then $C\subset S_i$. Indeed, if $C\not\subset S_i$, then the sets $C_1:=C\cap S_i\subset R^n$ and $C_2:=C\cap\bigcup_{j\in\{1,\ldots,\ell\}\setminus\{i\}}S_j\subset R^n$ would be $K$-semialgebraic, disjoint, properly contained and closed in $C$ such that $C_1\cup C_2=C$. This is a contradiction, because $C\subset R^n$ is $K$-semialgebraically connected. $\sqbullet$
\end{remark}

In the next result, we see that the $K$-dimension and the $R$-dimension of a $K$-semialgebraic set in the sense of Definition \ref{def:K-dim} coincide.

\begin{thm}\label{edimsa}
If $S\subset R^n$ is a $K$-semialgebraic set, then $\dim_K(S)=\dim_R(S)$.
\end{thm}
\begin{proof}
Let $\kr\subset R$ be the real closure of $K$ and let $T:=S\cap(\kr)^n\subset(\kr)^n$. Observe that $S=T_R$ is the extension of coefficients of $T$ to $R$. Define $X:=\zcl_{(\kr)^n}(T)\subset(\ol{K}^r)^n$ and $Y:=\zcl_{R^n}(T)\subset R^n$. By Corollary \ref{LK-zar} and Proposition \ref{extension-zar}$(\mr{i})$, we have $Y=\zcl_{R^n}(X)=X_R$. In~particular, $Y\subset R^n$ is $\kr$-algebraic. As $T\subset S=T_R\subset X_R=Y=\zcl_{R^n}(T)$, we deduce that $\zcl_{R^n}(S)=Y$.

Set $P:=\zcl_{R^n}^K(Y)\subset R^n$. Observe that $P=\zcl_{R^n}^K(\zcl_{R^n}(S))=\zcl_{R^n}^K(S)$. By Theorem \ref{thm:gc}$(\mr{i}')(\mr{iv})$, we have $\dim_R(P)=\dim_R(Y)$. By Lemma \ref{rkdim}, we have
$$
\dim_R(S)=\dim_R(Y)=\dim_R(P)=\dim_K(P)=\dim_K(S),
$$ 
as required.
\end{proof}

The previous proposition justifies the following notation, which we will use throughout the rest of this section.

\begin{notation}\label{enotsa}
If $S\subset R^n$ is a $K$-semialgebraic set, we simply say that $S$ has \emph{dimension $\dim(S)$}, where $\dim(S):=\dim_R(S)=\dim_K(S)$.
\end{notation}

\subsubsection{$K$-algebraic partial manifolds}

Taking inspiration from Whitney's seminal paper \cite{wh}, we propose the following definition.

\begin{defn}\label{def:Q-alg-partial-subman}
Let $M$ be a subset of $R^n$. We say that $M$ is a \emph{$K$-algebraic partial submani\-fold of $R^n$} or $M\subset R^n$ is a \emph{$K$-algebraic partial manifold} if $M$ is a $K$-semialgebraic subset of~$R^n$ and either $M$ is open in $R^n$ or there exists $m\in\{0,\ldots,n-1\}$ with the following property: for each $p\in M$, there exist polynomials $f_1,\ldots,f_{n-m}\in\II_K(M)$ and an open neighborhood $U$ of $p$ in $R^n$ such that the gradients $\nabla f_1(p),\ldots,\nabla f_{n-m}(p)$ are linearly independent in $R^n$ and $M\cap U=\ZZ_R(f_1,\ldots,f_{n-m})\cap U$. $\sqbullet$
\end{defn}

\begin{remarks}\label{rem6314}
$(\mr{i})$ Let $M\subset R^n$ be a $K$-algebraic partial manifold and let $M_K:=\zcl_{R^n}^K(M)$. Suppose that $M$ is not an open subset of $R^n$ and let $m\in\{0,\ldots,n-1\}$ be as in Definition \ref{def:Q-alg-partial-subman}. Let us prove that $M$ is open in $\Reg^K(M_K)$ and $m=\dim(M)$, where $\dim(M)$ is understood in the sense of Notation \ref{enotsa}.

Let $p\in M$, let $f_1,\ldots,f_{n-m}\in\II_K(M)=\II_K(M_K)$ and let $U$ be an open neighborhood of $p$ in $R^n$ such that the gradients $\nabla f_1(p),\ldots,\nabla f_{n-m}(p)$ are linearly independent in $R^n$ and $M\cap U=\ZZ_R(f_1,\ldots,f_{n-m})\cap U$. We have $M\cap U\subset M_K\cap U\subset\ZZ_R(f_1,\ldots,f_{n-m})\cap U=M\cap U$, so $M\cap U=M_K\cap U=\ZZ_R(f_1,\ldots,f_{n-m})\cap U$. As $p\in M_K\cap U\subset M$, it follows that $M$ is open in $M_K$, because it is a neighborhood of each of its points. By Definition \ref{E|K-regular} and equivalence $(\mr{i})\Longleftrightarrow(\mr{iii}')$ of Proposition \ref{30}, we deduce that $p\in\Reg^K(M_K,m)$, so $M\subset\Reg^K(M_K,m)$. As $M$ is $K$-Zariski dense in $M_K\subset R^n$ and $M\subset\Reg^K(M_K,m)$, we have by Corollary \ref{30'}$(\mr{ii})$ that $m=\dim(M_K)=\dim_K(M)=\dim(M)$. We have thus proved that $m=\dim(M)$ and also that $M$ is an open subset of $\Reg^K(M_K)$.

$(\mr{ii})$ Let $M$ be a $K$-semialgebraic subset of $R^n$ and let $M_K:=\zcl_{R^n}^K(M)$. A consequence of the previous item is that the following assertions are equivalent:
\begin{itemize}
\item[$\bullet$] $M\subset R^n$ is a $K$-algebraic partial manifold.
\item[$\bullet$] $M$ is an open subset of the $K$-nonsingular locus $\Reg^K(M_K)$ of $M_K$.
\item[$\bullet$] There exists a $K$-algebraic set $X\subset R^n$ such that $M$ is an open subset of $\Reg^K(X)$.
\end{itemize}

$(\mr{iii})$ By the Implicit Function Theorem \cite[Cor.2.9.8]{bcr}, a $K$-algebraic partial submanifold $M$ of $R^n$ of dimension $m$ is also a Nash submanifold of $R^n$ of dimension $m$ in the usual sense of \cite[Def.2.9.9]{bcr}. Let $p\in M$, let $M_K:=\zcl_{R^n}^K(M)$, let $f_1,\ldots,f_{n-m}\in\II_K(M)$ and let $U$ be an open neighborhood of $p$ in $R^n$ such that the gradients $\nabla f_1(p),\ldots,\nabla f_{n-m}(p)$ are linearly independent in $R^n$ and $M\cap U=M_K\cap U=\ZZ_R(f_1,\ldots,f_{n-m})\cap U$. By Definition \ref{ek-Zar-tang} and Remarks \ref{rem528}$(\mr{i})$, we know that
$$
T^K_p(M_K)=\{v\in R^n:\langle\nabla f_1(p),v\rangle=0,\ldots,\langle\nabla f_{n-m}(p),v\rangle=0\},
$$
so
$$
T^K_p(M_K)=T_p(M),
$$
where $T_p(M)$ is the usual tangent space at $p$ of the Nash submanifold $M$ of $R^n$.~$\sqbullet$
\end{remarks}

\begin{remark}\label{rem6316}
Each $K$-algebraic set $X\subset R^n$ can be written as a disjoint finite union of $K$-algebraic partial submanifolds of $R^n$. Let $d:=\dim(X)$. By Corollary \ref{30'}$(\mr{iii})$, we know that the $K$-singular locus $\Sing^K(X)$ of $X$ is $K$-algebraic in $R^n$ and of dimension $<d$, so we can define the family $\Sing^K_X:=\{M_i\}_{i=0}^{d'}$ of $K$-algebraic partial submanifolds of $R^n$ inductively as follows: $M'_0:=X$, $M'_i:=\Sing^K(M'_{i-1})$ for $i\geq1$, $M_i:=M'_i\setminus M'_{i+1}=\Reg^K(M'_i)$ for $i\in\{0,\ldots,d\}$ and $d'$ is the smallest index $i\in\{0,\ldots,d\}$ such that $M'_{i+1}=\varnothing$.

For example, if $p:=\x_2^2-\x_3\x_1^2\in K[\x_1,\x_2,\x_3]$ and $X$ is the Whitney umbrella $\ZZ_R(p)\subset R^3$, then Corollary \ref{cor:kpol} implies that $\Sing^K_X=\{x\in X:\nabla p(x)=0\}$, so $\Sing^K(X)$ coincides with the $x_3$-axis $Z$ of $R^3$ and $\Sing^K_X=\{X\setminus Z,Z\}$. $\sqbullet$
\end{remark}

\subsubsection{$K$-algebraic stratifications and Whitney's regularity} Let us introduce the concepts of $K$-algebraic stratification and Whitney regular $K$-algebraic stratification.

\begin{defn}\label{def6317}
Let $S\subset R^n$ be a $K$-semialgebraic set. A \emph{$K$-algebraic stratification of $S$} is a finite partition $\{M_i\}_{i\in I}$ of $S$ with the following properties:
\begin{itemize}
\item[$(\mr{i})$] Each $M_i\subset R^n$ is a $K$-semialgebraically connected $K$-algebraic partial mani\-fold.
\item[$(\mr{ii})$] If $M_j\cap\cl_{R^n}(M_i)\neq\varnothing$ for some $i,j\in I$ with $i\neq j$, then $M_j\subset\cl_{R^n}(M_i)$.
\end{itemize} 
The sets $M_i$ are called the \emph{strata} of the $K$-algebraic stratification $\{M_i\}_{i\in I}$ and, if $m:=\dim(M_i)$, then $M_i$ is said to be a \emph{$d$-stratum} of $\{M_i\}_{i\in I}$.

Given a finite family $\{S_\lambda\}_{\lambda\in\Lambda}$ of $K$-semialgebraic subsets of $R^n$ contained in $S$, we say that the $K$-algebraic stratification $\{M_i\}_{i\in I}$ is \emph{compatible with $\{S_\lambda\}_{\lambda\in\Lambda}$} if each $S_\lambda$ is the union of some strata of $\{M_i\}_{i\in I}$.

Given another $K$-algebraic stratification $\{N_q\}_{q\in Q}$ of $S\subset R^n$, we say that $\{N_q\}_{q\in Q}$ is \emph{finer than $\{M_i\}_{i\in I}$} if each stratum $M_i$ is the union of some strata $N_q$. $\sqbullet$
\end{defn}

\begin{remark}
Condition $(\mr{ii})$ in the previous definition is called \emph{frontier condition}. Such a condition can be equivalently restated as follows: ``If $M_j\cap\cl_{R^n}(M_i)\neq\varnothing$ for some $i,j\in I$ with $i\neq j$, then $M_j\subset\cl_{R^n}(M_i)$ and $\dim(M_j)<\dim(M_i)$''. Indeed, $M_j\subset\cl_{R^n}(M_i)\setminus M_i$ and $\dim(\cl_{R^n}(M_i)\setminus M_i)<\dim(M_i)$ by \cite[Prop.2.8.13]{bcr}. $\sqbullet$
\end{remark}

By Proposition \ref{K-semialg-connected}$(\mr{ii})$ and Remarks \ref{rem6314}, a $K$-algebraic stratification of a $K$-semialgebraic set $S\subset R^n$ is also a Nash stratification of the set $S\subset R^n$, viewed as a usual $R$-semialgebraic set, in the sense of \cite[Def.9.1.7]{bcr}. If $S$ is equal to a $K$-algebraic subset $X$ of $R^n$, the partition $\Sing^K_X$ defined in Remark \ref{rem6316} may not be a $K$-algebraic stratification: for example, if $X$ is the Whitney umbrella $\ZZ_R(\x_2^2-\x_3\x_1^2)\subset R^3$, then $\Sing^K_X=\{X\setminus Z,Z\}$, where $Z$ is the $x_3$-axis of $R^3$, is not a $K$-algebraic stratification, because it does not satisfy $(\mr{i})$: $X\setminus Z$ has the two $K$-semialgebraically connected components $X_+:=X\cap\{x_1>0\}$ and $X_-:=X\cap\{x_1<0\}$. Partition $\{X_+,X_-,Z\}$ is not a $K$-algebraic stratification too, because it satisfies $(\mr{i})$, but it does not satisfy $(\mr{ii})$.

\vspace{.5em}
{\textit{Fix $m\in\{0,\ldots,n\}$}. Identify the affine space $\mc{M}_{n,n}(R)$ of all $(n\times n)$-matrices with coefficients in $R$ with $R^{n^2}$ and the Grassmannian $\G_{n,m}(R)$ of all $m$-dimensional $R$-vector subspaces of $R^n$ with the $K$-algebraic subset of $R^{n^2}$ of all matrices $X$ such that $X^T=X$, $X^2=X$ and $\mr{tr}(X)=m$, that is,
$$
\G_{n,m}(R)=\big\{X\in\R^{n^2}:X^T=X,X^2=X,\mr{tr}(X)=m\big\}.
$$
This identification is done via the map that maps an element $H$ of $\G_{n,m}(R)$ onto the matrix of the orthogonal projection onto $H$. Along this section, $\PP^{n-1}(R)=\G_{n,1}(R)$ represents the $1$-dimensional $R$-vector subspaces (lines through the origin) of $R^n$. An explicit algebraic embedding of $\PP^{n-1}(R)$ in $\mc{M}_{n,n}(R)=R^{n^2}$ is
$$
x:=[x_1,\ldots,x_n]\mapsto\left(\frac{x_ix_j}{\sum_{i=1}^nx_i^2}\right)_{1\leq i,j\leq n}=\frac{x\cdot x^T}{\|x\|^2},
$$
see \cite[p.72]{bcr}.

Consider $R^n\times\G_{n,m}(R)$ as a subset of $R^{n+n^2}=R^n\times R^{n^2}$ and $R^n\times\G_{n,m}(R)\times R^n\times \PP^{n-1}(R)$ as a subset of $R^{2n+2n^2}=R^n\times R^{n^2}\times R^n\times R^{n^2}$. Denote $\cl_1$ and $\cl_2$ the (Euclidean) closure operators of $R^{n+n^2}$ and 
$R^{2n+2n^2}$, respectively. Consider the elements of $R^n$ as column vectors.

The next definition extends \cite[Def.9.7.3]{bcr} to the $R|K$-algebraic setting.

\begin{defn}\label{def:K-Whitney}
Let $M$ and $N$ be two non-empty disjoint $K$-semialgebraically connected $K$-algebraic partial submanifolds of $R^n$ such that $\dim(M)=m$ and $N\subset\cl_{R^n}(M)$. Define:
\begin{align*}
&F_a(M):=\{(x,\tau)\in R^n\times{\mathbb G}_{n,m}(R):x\in M, \tau=T_x(M)\},\\
&F_b(M,N):=\{(x,\tau,y,\alpha)\in F_a(M)\times R^n\times\PP^{n-1}(R):y\in N, \alpha=[x-y]\}.
\end{align*}
Given a point $y$ of $N$, we say that:
\begin{itemize}
\item[$(\mr{i})$] The pair \emph{$(M,N)$ satisfies condition $a$ at $y$} if, for all $\tau\in\G_{n,m}(R)$ such that $(y,\tau)\in\cl_1(F_a(M))$, we have $T_y(N)\subset\tau$.
\item[$(\mr{ii})$] The pair \emph{$(M,N)$ satisfies condition $b$ at $y$} if, for all $\tau\in\G_{n,m}(R)$ and all $\delta\in\PP^{n-1}(R)$ such that $(y,\tau,y,\delta)\in\cl_2(F_b(M,N))$, we have $\delta\subset\tau$.
\end{itemize}
Recall that $\tau\in\G_{n,m}(R)$ is an $m$-dimensional $R$-vector subspace of $R^n$ and $\delta\in\PP^{n-1}(R)$ is a $1$-dimensional $R$-vector subspace of $R^n$.
 $\sqbullet$
\end{defn}

The definition of Whitney regular $K$-algebraic stratification reads as follows.

\begin{defn}
Let $S\subset R^n$ be a $K$-semialgebraic set and let $\{M_i\}_{i\in I}$ be a $K$-algebraic stratification of $S$. We say that $\{M_i\}_{i\in I}$ is \emph{Whitney regular} if, for each $i,j\in I$ with $i\neq j$ and $\varnothing\neq M_j\subset\cl_{R^n}(M_i)$, the pair $(M_i,M_j)$ satisfies both conditions $a$ and $b$ at each point of $M_j$. $\sqbullet$
\end{defn}

\begin{remark}
Let $C:=R[\ii]$ be the algebraic closure of $R$, let $X_C$ be the complex Whitney umbrella $\ZZ_C(\x_2^2-\x_3\x_1^2)\subset C^3$ and let $Z_C$ be the $x_3$-axis of $C^3$. Consider $X_C$ and $Z_C$ as real algebraic subsets of $C^3=R^6$. Then the family $\{X_C\setminus Z_C,Z_C\}$ is a $K$-algebraic stratification of $Z_C\subset R^6$ (because $Z_C$ is semialgebraically connected as it is a plane and $X_C\setminus Z_C$ is semialgebraically connected as it is the regular part of a $C$-irreducible $C$-algebraic set), which is not Whitney regular, because the pair $(X_C\setminus Z_C,Z_C)$ does not satisfy condition~$a$ at the origin of $R^6$. Consider the complex line $L_C=\ZZ_C(\x_2,\x_3)\subset X_C$ of $C^3$. The tangent plane to the points of $L_C\setminus\{(0,0,0)\}$ is $\Pi_C=\ZZ_C(\x_3)$. Thus, $((0,0,0),\Pi_C)\in\cl_1(F_a(X_C\setminus Z_C))$, but $T_{(0,0,0)}Z_C=Z_C\not\subset\Pi_C$. $\sqbullet$
\end{remark}

The main result of this subsection is an $R|K$-version of \cite[Prop.9.1.8\;\&\;Thm.9.7.11]{bcr}.

\begin{thm}\label{thm:R|K-Whitney}
Let $S\subset R^n$ be a $K$-semialgebraic set. We have:
\begin{itemize}
\item[$(\mr{i})$] Let $\{S_\lambda\}_{\lambda\in\Lambda}$ be any (possibly empty) finite family of $K$-semialgebraic subsets of $R^n$ contained in $S$. Then there exists a $K$-algebraic stratification $\{M_i\}_{i\in I}$ of $S$ compatible with $\{S_\lambda\}_{\lambda\in\Lambda}$.
\item[$(\mr{ii})$] For each $K$-algebraic stratification $\{M_i\}_{i\in I}$ of $S$, there exists a Whitney regular $K$-algebraic stratification $\{N_q\}_{q\in Q}$ of $S$ finer than $\{M_i\}_{i\in I}$.
\end{itemize}
\end{thm}

To prove this theorem, we need some preparation.

The following result generalizes \cite[Prop.9.7.4\&Thm.9.7.5]{bcr}.

\begin{prop}\label{fafb}
Let $M$, $N$, $m$, $F_a(M)$ and $F_b(M,N)$ be as in Definition \ref{def:K-Whitney}, let $S_a(M,N)$ be the subset of $N$ of all points $y$ such that $(M,N)$ does not satisfy condition $a$ at $y$, and let $S_b(M,N)$ be the subset of $N$ of all points $y$ such that $(M,N)$ does not satisfy condition $b$ at $y$. We have:
\begin{itemize}
\item[$(\mr{i})$] The sets $F_a(M)\subset R^{n+n^2}$ and $F_b(M,N)\subset R^{2n+2n^2}$ are $K$-semialgebraic.
\item[$(\mr{ii})$] The sets $S_a(M,N)\subset R^n$ and $S_b(M,N)\subset R^n$ are $K$-semialgebraic, $\dim(S_a(M,N))<\dim(N)$ and $\dim(S_b(M,N))<\dim(N)$.
\end{itemize}
\end{prop}
\begin{proof}
If $m=n$, then $M$ is a non-empty open subset of $R^n$, $T_x(M)$ is equal to the $(n\times n)$-identity matrix $I_n$ for all $x\in M$, $F_a(M)=M\times\{I_n\}$, $F_b(M,N)$ is the graph of the $K$-regular map $F_a(M)\times N\ni(x,I_n,y)\mapsto (x-y)(x-y)^T\|x-y\|^{-2}\in R^{n^2}$, $S_a(M,N)=S_b(M,N)=\varnothing$ and so the result is proven.

Suppose $m<n$.

$(\mr{i})$ Denote $\beta:M\to\G_{n,m}(R)$ the map defined by $\beta(x):=T_x(M)$. Let us prove that $\beta$ is $K$-regular. Let $p\in M$ and let $M_K:=\zcl_{R^n}^K(M)\subset R^n$. By Remarks \ref{rem6314}$(\mr{i})(\mr{ii})$ and implication $(\mr{i})\Longrightarrow(\mr{iii})$ of Proposition \ref{30}, we have that $\dim(M_K)=m$, $p\in M\subset\Reg^{R|K}(M_K)$ and there exist a $K$-Zariski open neighborhood $W$ of $p$ in $R^n$ and polynomials $f_1,\ldots,f_{n-m}\in\II_K(M)=\II_K(M_K)$ such that the gradients $\nabla f_1(p),\ldots,\nabla f_{n-m}(p)$ are linearly independent in $R^n$ and $M_K\cap W=\ZZ_R(f_1,\ldots,f_{n-m})\cap W$. Shrinking $W$ around $p$ if necessary, we can assume that the gradients $\nabla f_1(x),\ldots,\nabla f_{n-m}(x)$ are linearly independent in $R^n$ for all $x\in W$. By Remarks \ref{rem6314}$(\mr{iii})$, $T_x(M)=\{v\in R^n:\langle\nabla f_1(x),v\rangle=0,\ldots,\langle\nabla f_{n-m}(x),v\rangle=0\}$. Let $x\in W$. Define the matrices $A(x)\in\mc{M}_{n-m,n}(R)$ and $B\in\mc{M}_{n,n}(R)$ by
$$\textstyle
A(x):=\big(\frac{\partial f_i}{\partial\x_j}(x)\big)_{i=1,\ldots,n-m,\,j=1,\ldots,n}
$$
and
$$
B(x):=I_n-A(x)^T(A(x)A(x)^T)^{-1}A(x).
$$
Let us check: $B(x)$ is the matrix of the orthogonal projection of $R^n$ onto $T_x(M)$.

As the gradients $\nabla f_1(x),\ldots,\nabla f_{n-m}(x)$ are linearly independent in $R^n$, the matrix $A(x)A(x)^T$ has rank $n-m$. Observe that $B(x)v=v$ for all $v\in T_x(M)$ (so ${\rm rk}(B(x))\geq m=\dim(T_x(M))$) and $B(x)A(x)^T=0$, so ${\rm rk}(B(x))+n-m\leq n$. We deduce ${\rm rk}(B(x))=m$ and $\tr(B(x))=m$, because $B(x)^T=B(x)$ (so it is diagonalizable) and $B(x)^2=B(x)$ (so its eigenvalues are $0,1$). Consequently, $B(x)$ is the matrix of the orthogonal projection of $R^n$ onto $T_x(M)$. 

Thus, if $x\in M\cap W$, then the matrix $B(x)$ is the element of $\G_{n,m}(R)$ representing $T_x(M)$. It follows that $\beta(x)=B(x)$ for all $x\in M\cap W$, so $\beta$ is $K$-regular.

By Lemma \ref{lem437}, there exist polynomials $p_{ij}\in K[\x]$ for $i,j=1,\ldots,n$ and $q\in K[\x]$ such that $\ZZ_R(q)\cap M=\varnothing$ and $\beta(x)=\big(\frac{p_{ij}(x)}{q(x)}\big)_{i,j=1,\ldots,n}$ for each $x\in M$. As $F_a(M)$ is the graph of~$\beta$, we have that $F_a(M)=\{(x,\tau)\in R^{n+n^2}:x\in M,q(x)\tau=(p_{ij}(x))_{i,j=1,\ldots,n}\}$, so $F_a(M)\subset R^{n+n^2}$ is $K$-semialgebraic. As $F_b(M,N)$ is the graph of the $K$-regular map $F_a(M)\times N\ni(x,\tau,y)\mapsto (x-y)(x-y)^T\|x-y\|^{-2}\in R^{n^2}$, $F_b(M,N)\subset R^{2n+2n^2}$ is also $K$-semialgebraic.

$(\mr{ii})$ Let $\pi_1:R^{n+n^2}\to R^n$ be the projection $\pi_1(x,\tau):=x$, let $e:=\dim(N)$ and let $\{g_1,\ldots,g_s\}$ be a system of generators of $\II_K(N)$ in $K[\x]$. For all $(y,\tau)\in R^n\times R^{n^2}$, define $P(y,\tau)$ as the matrix in $\mc{M}_{n+s,n}(R)$ whose first $n$ rows are the rows of $I_n-\tau$ (here we understand $\tau$ as a matrix) and whose last $s$ rows are $\nabla g_1(y),\ldots,\nabla g_s(y)$. We have: $\cl_1(F_a(M))\subset\cl_{R^n}(M)\times\G_{n,m}(R)$,
\begin{align*}
S_a(M,N)&=\{y\in N:\exists \tau\in\G_{n,m}(R),\text{ $(y,\tau)\in\cl_1(F_a(M))$ and $T_y(N)\not\subset\tau$}\}\\
&=\{y\in N:\exists \tau\in\G_{n,m}(R),\text{ $(y,\tau)\in\cl_1(F_a(M))$ and $\dim(T_y(N)\cap\tau)<e$}\},
\end{align*}
(here we understand $\tau$ as an $m$-dimensional $R$-vector subspace of $R^n$) and 
$$
T_y(N)\cap\tau=\{v\in R^n:P(y,\tau)v=0\}
$$
(here we understand $\tau$ with both descriptions). Consequently,
$$
S_a(M,N)=N\cap\pi_1\big(\cl_1(F_a(M))\cap\{(y,\tau)\in R^{n+n^2}:{\rm rk}(P(y,\tau))>n-e\}\big)
$$
(we understand here $\tau$ as a $R$-vector subspace of $R^n$), so $S_a(M,N)\subset R^{n+n^2}$ is $K$-semialgebraic by Propositions \ref{prop625} and \ref{clos}.

Let $\pi_2:R^{2n+2n^2}\to R^n$ be the projection $\pi_2(x,\tau,y,\delta):=x$. For all $(\tau,\delta)\in R^{n^2}\times R^{n^2}$, define $Q(\tau,\delta)$ as the matrix in $\mc{M}_{2n,n}(R)$ whose first $n$ rows are the rows of $I_n-\tau$ (here we understand $\tau$ as a matrix) and whose last $n$ rows are the rows of $I_n-\delta$ (here we understand $\delta$ as a matrix). Define the $K$-algebraic set $\Delta^*\subset R^{2n+2n^2}$ by $\Delta^*:=\{(x,\tau,y,\delta)\in R^{2n+2n^2}:x=y\}$. We have: $\cl_2(F_b(M,N))\subset\cl_{R^n}(M)\times\G_{n,m}(R)\times R^n\times\PP^{n-1}(R)$,
\begin{align*}
&S_b(M,N)=\{y\in N:\exists (\tau,\delta)\in\G_{n,m}(R)\times\PP^{n-1}(R),\text{ $(y,\tau,y,\delta)\in\cl_2(F_b(M,N))$ and $\delta\not\subset\tau$}\}\\
&=\{y\in N:\exists (\tau,\delta)\in\G_{n,m}(R)\times\PP^{n-1}(R),\text{ $(y,\tau,y,\delta)\in\cl_2(F_b(M,N))$ and $\dim(\delta\cap\tau)=0$}\},
\end{align*}
and $\delta\cap\tau=\{v\in R^n:Q(\tau,\delta)v=0\}$ (here we understand $\tau,\delta$ with both descriptions). Consequently,
$$
S_b(M,N)=N\cap\pi_2\big(\cl_2(F_b(M,N))\cap\Delta^*\cap\{(x,\tau,y,\delta)\in R^{2n+2n^2}:{\rm rk}(Q(\tau,\delta))=n\}\big)
$$
(we understand here $\delta,\tau$ as $R$-vector subspaces of $R^n$), so $S_b(M,N)\subset R^{2n+2n^2}$ is $K$-semialgebraic by Propositions \ref{prop625} and \ref{clos}.

Finally, by Proposition \ref{K-semialg-connected}$(\mr{iii})$, we know that $M$ and $N$ are $R$-semialgebraically connected Nash submanifolds of $R^n$, so we can apply \cite[Thm.9.7.5]{bcr} obtaining $\dim_R(S_a(M,N))<e$ and $\dim_R(S_b(M,N))<e$. By Theorem \ref{edimsa} and Notation \ref{enotsa}, we also have $\dim(S_a(M,N))=\dim_R(S_a(M,N))<e$ and $\dim(S_b(M,N))=\dim_R(S_b(M,N))<e$, as required.
\end{proof}

\begin{remark}\label{rem:sacsb}
Let $M$ and $N$ be two non-empty disjoint $K$-semialgebraically connected $K$-algebraic partial submanifolds of $R^n$ such that $N\subset\cl_{R^n}(M)$. As $M$ and $N$ are $R$-semialge\-braically connected Nash submanifolds of $R^n$, we can apply \cite[Prop.9.7.7]{bcr} obtaining that $S_a(M,N)\subset S_b(M,N)$. $\sqbullet$
\end{remark}

\begin{lem}\label{strata}
Let $S\subset R^n$ be a non-empty $K$-semialgebraic set of dimension $d$. Then there exists a $K$-algebraic partial manifold $M\subset R^n$ of dimension $d$ contained in $S$ such that $\dim(\cl_{R^n}(S)\setminus M)<d$ and $M$ is open both in $\cl_{R^n}(S)$ and in $\zcl_{R^n}^K(S)$ (thus in $S$).
\end{lem}
\begin{proof}
Let $Z:=\zcl_{R^n}^K(S)\subset R^n$ and let $Y:=\Sing^K(Z)$. Observe that $\dim(Z)=\dim(S)=d$ and $Y\subset R^n$ is a $K$-algebraic set of dimension $<d$. 

Set $S_0:=\cl_{R^n}(S)\setminus\cl_{R^n}(\cl_{R^n}(S)\setminus S)=S\setminus\cl_{R^n}(\cl_{R^n}(S)\setminus S)$. The set $S_0$ is $K$-semialgebraic in $R^n$, is open in $\cl_{R^n}(S)$ and 
$$
\cl_{R^n}(S)\setminus S_0\subset\cl_{R^n}(\cl_{R^n}(S)\setminus S)\subset\zcl_{R^n}(\cl_{R^n}(S)\setminus S)$$ 
so, by \cite[Prop.2.8.13]{bcr}, we have
$$
\dim(\cl_{R^n}(S)\setminus S_0)\leq\dim(\cl_{R^n}(\cl_{R^n}(S)\setminus S))=\dim(\cl_{R^n}(S)\setminus S)<d.
$$

Set $S_1:=S_0\setminus Y$. The set $S_1$ is $K$-semialgebraic in $R^n$, open in $S_0$ (thus in $\cl_{R^n}(S)$) and
$$
\cl_{R^n}(S)\setminus S_1=(\cl_{R^n}(S)\setminus S_0)\sqcup(S_0\setminus S_1)\subset\cl_{R^n}(\cl_{R^n}(S)\setminus S)\cup Y,
$$
so $\dim(\cl_{R^n}(S)\setminus S_1)<d$.

Set $S_2:=\cl_{R^n}(Z\setminus S_1)\subset R^n$, which satisfies
\begin{align*}
S_2&=\cl_{R^n}((Z\setminus S_0)\cup Y)=\cl_{R^n}((Z\setminus\cl_{R^n}(S))\cup\cl_{R^n}(\cl_{R^n}(S)\setminus S)\cup Y)\\
&=\cl_{R^n}(Z\setminus\cl_{R^n}(S))\cup\cl_{R^n}(\cl_{R^n}(S)\setminus S)\cup Y.
\end{align*} 

The set $S_2\subset R^n$ is $K$-semialgebraic in $R^n$, closed in $Z$ and $S_2\cap S_1=\cl_{R^n}(Z\setminus S_1)\setminus(Z\setminus S_1)$ (because $S_1=Z\setminus(Z\setminus S_1)$), so, by \cite[Prop.2.8.13]{bcr}, we have
$$
\dim(S_2\cap S_1)<\dim(Z\setminus S_1)\leq d.
$$

Finally, set $M:=S_1\setminus S_2$, which satisfies 
$$
M=\cl_{R^n}(S)\setminus(\cl_{R^n}(Z\setminus\cl_{R^n}(S))\cup\cl_{R^n}(\cl_{R^n}(S)\setminus S)\cup Y)
$$
The set $M$ is $K$-semialgebraic in $R^n$ and is open in $S_1$ (thus in $\cl_{R^n}(S)$). We have: $\cl_{R^n}(S)\setminus M=(\cl_{R^n}(S)\setminus S_1)\sqcup(S_2\cap S_1)$, so $\dim(\cl_{R^n}(S)\setminus M)<d=\dim(S)=\dim(\cl_{R^n}(S))$, so $M\neq\varnothing$. We claim: $M$ is open in $Z$.

Indeed,
$$
Z\setminus M=(Z\setminus\cl_{R^n}(S))\cup\cl_{R^n}(Z\setminus\cl_{R^n}(S))\cup\cl_{R^n}(\cl_{R^n}(S)\setminus S)\cup Y=S_2
$$ 
(because $Z\setminus\cl_{R^n}(S)\subset\cl_{R^n}(Z\setminus\cl_{R^n}(S))$), which is a closed subset of $Z$, so $M=Z\setminus S_2$ is open in $Z$. It follows that $M=Z\setminus S_2$ is also a non-empty open subset of $Z\setminus Y=\Reg^K(Z)$ (because $Y\subset S_2$), so $M$ is by Remark \ref{rem6314}(ii) a $K$-algebraic partial submanifold of $R^n$ of dimension $d$, as required.
\end{proof}

\begin{lem}\label{strat}
Let $S\subset R^n$ be a non-empty $K$-semialgebraic set and let $\{S_1,\ldots,S_r\}$ be a non-empty finite family of (possibly empty) $K$-semialge\-braic subsets of~$R^n$ such that $S=\bigcup_{i=1}^rS_i$. Then there exists a $K$-algebraic stratification of $S$ compatible with $\{S_1,\ldots,S_r\}$.
\end{lem} 
\begin{proof}
We proceed by induction on $\dim(S)\geq0$.

If $\dim(S)=0$, then $S$ is a finite set and its singletons are its $K$-semialgebraically connected components so they are $K$-semialgebraic. It follows that the family of all the singletons of $S$ is a $K$-algebraic stratification of $S$ compatible with $\{S_1,\ldots,S_r\}$.

Let $d\in\N^*$. Suppose the result is true for $K$-semialgebraic sets of dimension $<d$. We have to check that it is also true for $K$-semialgebraic sets of dimension $d$. Denote $\mc{P}(r)$ the power set of $\{1,\ldots,r\}$ and, for each $J\in\mc{P}(r)\setminus\{\varnothing\}$, define $S_J:=(\bigcap_{j\in J}S_j)\setminus(\bigcup_{j\in\{1,\ldots,r\}\setminus J}S_j)$. Observe that $S=\bigcup_{J\in\mc{P}(r)\setminus\{\varnothing\}}S_J$, $S_{J_1}\cap S_{J_2}=\varnothing$ for all $J_1,J_2\in\mc{P}(r)\setminus\{\varnothing\}$ with $J_1\neq J_2$, and $S_J\subset S_i$ for all $J\in\mc{P}(r)\setminus\{\varnothing\}$ and $i\in\{1,\ldots,r\}$ with $S_J\cap S_i\neq\varnothing$. Consequently, it is enough to prove the result for the finite family $\{S_J\}_{J\in\mc{P}(r)\setminus\{\varnothing\}}$ of $K$-semialgebraic subsets of $R^n$. In other words, we can assume from the beginning that $S_i\cap S_j=\varnothing$ for all $i,j\in\{1,\ldots,r\}$ with $i\neq j$. 

Rearranging the indices if necessary, we can also assume that there exists $\ell\in\{1,\ldots,r\}$ such that $S_1,\ldots,S_\ell$ have dimension $d$ and the $S_i$ for $i>\ell$ have dimension $<d$ (if there exists any). By Lemma \ref{strata}, for each $i\in\{1,\ldots,\ell\}$, there exists a $K$-algebraic partial manifold $M'_i\subset R^n$ of dimension~$d$ contained in $S_i$ such that $\dim(\cl_{R^n}(S_i)\setminus M'_i)<d$ and $M'_i$ is open in $\cl_{R^n}(S_i)$. In particular, $\cl_{R^n}(S_i)\setminus M'_i$ is closed in $R^n$.

Define the closed $K$-semialgebraic set $C\subset R^n$ of dimension $<d$ by
$$\textstyle
C:=\bigcup_{i=1}^\ell(\cl_{R^n}(S_i)\setminus M'_i)\cup\bigcup_{i>\ell}^r\cl_{R^n}(S_i),
$$
where $\bigcup_{i>\ell}^r\cl_{R^n}(S_i):=\varnothing$ if $\ell=r$, and $M_i:=M'_i\setminus C$ for each $i\in\{1,\ldots,\ell\}$.

Pick $i\in\{1,\ldots,\ell\}$. For each $j\in\{1,\ldots,\ell\}\setminus\{i\}$, we have:
$$
M_i\cap\cl_{R^n}(M_j)\subset(M'_i\setminus C)\cap\cl_{R^n}(S_j)=M'_i\cap(\cl_{R^n}(S_j)\setminus C)\subset M'_i\cap M'_j\subset S_i\cap S_j=\varnothing.
$$
In addition, $M_i$ is open in $M'_i$ (thus in $\cl_{R^n}(S_i)$), and $\cl_{R^n}(S_i)\setminus M_i\subset(\cl_{R^n}(S_i)\setminus M'_i)\cup C$, so $\dim(\cl_{R^n}(S_i)\setminus M_i)<d$. Consequently, $M_i\subset R^n$ is a non-empty $K$-algebraic partial manifold of dimension~$d$.

Observe that $M_i$ is also open in $\cl_{R^n}(S)$. Indeed, $M_i\cap\cl_{R^n}(S_j)=(M'_i\setminus C)\cap\cl_{R^n}(S_j)=M'_i\cap(\cl_{R^n}(S_j)\setminus C)\subset M'_i\cap M'_j\subset S_i\cap S_j=\varnothing$ if $j\in\{1,\ldots,\ell\}\setminus\{i\}$, and $M_i\cap\cl_{R^n}(S_j)=M'_i\cap(\cl_{R^n}(S_j)\setminus C)=M'_i\cap\varnothing=\varnothing$ if $j>\ell$. Thus, 
$$\textstyle
\cl_{R^n}(S)\setminus M_i=(\cl_{R^n}(S_i)\setminus M_i)\cup\bigcup_{j\in\{1,\ldots,r\}\setminus\{i\}}\cl_{R^n}(S_j), 
$$
which is closed in $\cl_{R^n}(S)$. As $M_i\subset S_i\subset S$, $M_i$ is also open in $S$. 

Define $S_{r+1}:=S\setminus\bigcup_{i=1}^\ell M_i$. Observe that $S_{r+1}$ is closed in $S$, $K$-semialgebraic in~$R^n$, and it is the disjoint union of the $S_i\setminus M_i$ for $i\in\{1,\ldots,\ell\}$ and the $S_i$ for $i>\ell$, that is,
$$\textstyle
S_{r+1}=\big(\bigsqcup_{i=1}^\ell(S_i\setminus M_i)\big)\sqcup\big(\bigsqcup_{i>\ell}^rS_i\big),
$$
Moreover, it holds
$$\textstyle
S_{r+1}=S\setminus\bigcup_{i=1}^\ell(M'_i\setminus C)=(S\setminus\bigcup_{i=1}^\ell M'_i)\cup(S\cap C)\subset\bigcup_{i=1}^\ell(S_i\setminus M'_i)\cup C\subset C,
$$
so $\dim(S_{r+1})<d$.

Let $\{N_q\}_{q\in Q}$ denote the finite family of all $K$-semialgebraically connected components of $M_1,\ldots,M_\ell$, which are $K$-algebraic partial submanifolds of $R^n$ of dimension~$d$. Observe that $S$ is the disjoint union of $S_{r+1}$ and the $N_q$, that is,
$$\textstyle
S=\big(\bigsqcup_{q\in Q}N_q\big)\sqcup S_{r+1}=\big(\bigsqcup_{q\in Q}N_q\big)\sqcup\big(\bigsqcup_{i=1}^\ell(S_i\setminus M_i)\big)\sqcup\big(\bigsqcup_{j>\ell}^rS_j\big).
$$

Define $T_p:=S\cap(\cl_{R^n}(N_p)\setminus N_p)$ for each $p\in Q$. Observe that each $T_p$ is contained in $S_{r+1}$, because $\{N_q\}_{q\in Q}$ is the family of all $K$-semialgebraically connected components of $\bigcup_{i=1}^\ell M_i$, so $N_p\cap\cl_{R^n}(N_q)=\varnothing$ for all $q\in Q\setminus\{p\}$.

Define the family $\mc{F}$ of $K$-semialgebraic subsets of $R^n$ by
$$
\mc{F}:=\{T_p\}_{p\in Q}\cup\{S_i\setminus M_i\}_{i=1}^\ell\cup\{S_i\}_{i>\ell}^r,
$$
where $\{S_i\}_{i>\ell}^r$ is the empty family if $\ell=r$.

As $\dim(S_{r+1})<d$, by induction hypothesis, there exists a $K$-algebraic stratification $\{E_g\}_{g\in G}$ of $S_{r+1}$ compatible with $\mc{F}$. Define ${\mathfrak E}:=\{N_q\}_{q\in Q}\cup\{E_g\}_{g\in G}$.

We claim: \emph{${\mathfrak E}$ is a $K$-algebraic stratification of $S$ compatible with $\{S_1,\ldots,S_r\}$.} As ${\mathfrak E}$ is a partition of $S$ whose elements are $K$-semialgebraically connected $K$-algebraic partial submanifolds of $R^n$ and the family $\mc{F}$ contains $\{S_i\setminus M_i\}_{i=1}^\ell\cup\{S_i\}_{i>\ell}^r$ (so $\mathfrak E$ is compatible with $\{S_1,\ldots,S_r\}$), it is enough to prove that ${\mathfrak E}$ satisfies the frontier condition (ii) of Definition \ref{def6317}.

Let $\Gamma_1,\Gamma_2\in{\mathfrak E}$ be such that $\Gamma_1\cap\cl_{R^n}(\Gamma_2)\neq\varnothing$. We have to check that $\Gamma_1\subset\cl_{R^n}(\Gamma_2)$.

Let us distinguish three cases:

\textsc{Case 1.} If both $\Gamma_1$ and $\Gamma_2$ belong to $\{N_q\}_{q\in Q}$ or $\{E_g\}_{g\in G}$, then the inclusion $\Gamma_1\subset\cl_{R^n}(\Gamma_2)$ holds, because $\{N_q\}_{q\in Q}$ are $\{E_g\}_{g\in G}$ are $K$-algebraic stratifications.

\textsc{Case 2.} Suppose that $\Gamma_1=N_p$ and $\Gamma_2=E_g$ for some $p\in Q$ and $g\in G$, so $N_p\cap\cl_{R^n}(E_g)\neq\varnothing$. This situation cannot occur, because $E_g\subset S_{r+1}$ and $S_{r+1}$ is a closed subset of $S$ disjoint to $\bigcup_{q\in Q}N_q=\bigcup_{i=1}^\ell M_i$.

\textsc{Case 3.} Suppose that $\Gamma_1=E_g$ and $\Gamma_2=N_p$ for some $g\in G$ and $p\in Q$, so $E_g\cap\cl_{R^n}(N_p)\neq\varnothing$. As $E_g\subset S\setminus N_p$, we deduce $E_g\cap T_p\neq\varnothing$, so $E_g\subset T_p\subset\cl_{R^n}(N_p)$ (because $\{E_g\}_{g\in G}$ is a $K$-algebraic stratification of $S_{r+1}$ compatible with ${\mathcal F}$ and $T_p\in{\mathcal F}$), as required.
\end{proof}

We are ready to prove Theorem \ref{thm:R|K-Whitney}.

\begin{proof}[Proof of Theorem \em \ref{thm:R|K-Whitney}]
$(\mr{i})$ Apply Lemma \ref{strat} with $\{S_1,\dots,S_r\}:=\{S_\lambda\}_{\lambda\in\Lambda}\cup\{S\}$.

$(\mr{ii})$ We adapt the proof of \cite[Thn.9.7.11]{bcr} to the present situation. Let $d:=\dim(S)$. We can assume $d>0$, otherwise the result is obvious. Let us prove by induction on $k=d,\ldots,0$ that, after refining the $K$-algebraic stratification $\{M_i\}_{i\in I}$, we may assume that the following condition $(C_k)$ holds:
\begin{itemize}
\item[$(C_k)$] \emph{Each pair of distinct strata $(M_i,M_j)$ with $M_i\subset\cl_{R^n}(M_j)$ and $\dim(M_j)>k$ satisfies condition $b$ at every point of $M_j$.}
\end{itemize}

By Remark \ref{rem:sacsb}, all the pairs $(M_i,M_j)$ under the hypothesis of $(C_k)$ also satisfy condition $a$ at every point of~$M_j$.

Condition $(C_d)$ is clear: there is nothing to check. Suppose $d>0$ and $C_k$ is true for some $k\in\{1,\ldots,d\}$. For each $k$-stratum $M_j$, define
$$
T_j:=M_j\cap\cl_{R^n}\!\!\Big(\bigcup\big\{S_b(M_i,M_j):M_j\subset\cl_{R^n}(M_i),i\neq j\big\}\Big).
$$
By Proposition \ref{fafb}$(\mr{ii})$, we have $\dim(T_j)<k$. By Lemma \ref{strat}, there exists a $K$-algebraic stratification of the union of the $T_j$ and the strata $M_h$ of dimension $<k$ compatible with the $T_j$ and such strata~$M_h$. Thus, we obtain a $K$-algebraic stratification of $S$ finer than $\{M_i\}_{i\in I}$ with the same strata of dimension $>k$ as $\{M_i\}_{i\in I}$ and whose strata of dimension $k$ are the $K$-semialgebraically connected components of the difference $M_j\setminus T_j$. This new stratification satisfies $(C_{k-1})$ and proceeding inductively until $d=0$ the proof is complete. 
\end{proof}

For Whitney regular stratifications, their variants and applications, we refer to \cite{bcr,vdD,tr1,tr2,paru,pawlucki,tv} and the references mentioned therein.

\subsection{The $L|K$-generic projection theorem}\label{subsec:proj2}
\textit{Throughout this subsection, $L|K$ is an extension of fields such that $L$ is either algebraically closed or real closed, and $\kbar^\sqbullet$ is the algebraic closure of $K$ in~$L$.} We will use Remark \ref{dime} freely, and set $\dim(S):=\dim_L(S)$ for each algebraic set $S\subset L^n$.

Let $n,r\in\N^*$ be such that $r<n$. Identify $L^r$ and $L^{n-r}$ with the $L$-vector subspaces $L^r\times\{0\}$ and $\{0\}\times L^{n-r}$ of $L^n=L^r\times L^{n-r}$, respectively. Recall that the exponent $^T$ indicates the transpose operation. Denote $x':=(x_1,\ldots,x_r)^T$ the coordinates of $L^r$ and $x'':=(x_{r+1},\ldots,x_n)^T$ the coordinates of $L^{n-r}$, so $x=(x',x'')=(x_1,\ldots,x_n)^T$. Note that here, for convenience, we consider the coordinates $x'$, $x''$, and then $x$ as column vectors. Indicate $\mc{M}_{r,n-r}(L)$ the $L$-vector space of all $r\times(n-r)$-matrices with coefficients in $L$, and identify $\mc{M}_{r,n-r}(L)$ with $L^{r(n-r)}$.

\begin{defn}\label{V_AP_A}
For each $A\in\mc{M}_{r,n-r}(L)$, we define the $L$-vector subspace $V_A$ of $L^n$ by
$$
V_A:=\{(x',x'')\in L^n:x'=Ax''\}
$$
and the $L$-linear map $\pi_A:L^n\to L^r$ by
$$
\pi_A(x',x''):=x'-Ax''.\text{ $\sqbullet$}
$$
\end{defn}

Observe that $\{V_A\}_{A\in\mc{M}_{r,n-r}(L)}$ is the family of all $(n-r)$-dimensional $L$-vector subspaces of $L^n$ that are transverse to $L^r$, and $\pi_A$ is the (affine) projection of $L^n$ onto $L^r$ in the direction of $V_A$, that is, for each $x=(x',x'')\in L^n$, $\pi_A(x)=x'-Ax''$ corresponds to the unique point of the intersection $L^r\cap(x+V_A)$, because $(x',x'')+(A(-x''),-x'')=(x'-Ax'',0)$.

Denote $\mc{M}_{r,n-r}(K)$ the $K$-vector space of all $r\times(n-r)$-matrices with coefficients in $K$ and identify it with $K^{r(n-r)}$. Thus, we can write $\mc{M}_{r,n-r}(K)=K^{r(n-r)}\subset L^{r(n-r)}=\mc{M}_{r,n-r}(L)$ and, for each $A\in\mc{M}_{r,n-r}(K)$, we consider the $L$-vector subspace $V_A$ of $L^n$ and the $L$-linear map $\pi_A:L^n\to L^r$.

\begin{defn}\label{def:embedding-dim}
Let $X\subset L^n$ be a $K$-algebraic set. For each $a\in X$, we define the \emph{$L|K$-embedding dimension $e^{L|K}_a(X)$ of $X\subset L^n$ at $a$} as the $L$-vector dimension of the $L|K$-Zariski tangent space $T^{L|K}_a(X)$, that is, $e^{L|K}_a(X):=\dim_L(T^{L|K}_a(X))$ (see Definition \ref{ek-Zar-tang}). We call \emph{$L|K$-embedded dimension of $X\subset L^n$} the natural number $e^{L|K}(X):=\max_{a\in X}\{e^{L|K}_a(X)\}$. $\sqbullet$
\end{defn}

\begin{remark}\label{rem633}
By Remarks \ref{tang-bcr}$(\mr{iii})$, if $L$ is algebraically closed, $T^{L|K}_a(X)$ coincides with the usual Zariski tangent space $T_a(X)=T^{L|L}_a(X)$ of the set $X\subset L^n$, viewed as a usual ($L$-)algebraic subset of $L^n$, so $e^{L|K}(X)$ is equal to the usual embedding dimension $e(X):=e^{L|L}(X)$ of $X\subset L^n$. If~$L$ is a real closed field $R$, then $T_a(X)$ can be strictly contained in $T^{R|K}_a(X)$, and $e(X)$ strictly smaller than $e^{R|K}(X)$. For instance, if $K=\Q$ and $X$ is the line $\ZZ_R(\x_1-\sqrt[3]{2}\x_2)=\ZZ_R(\x_1^3-2\x_2^3)$ of $R^2$, then $e(X)=1<2=e^{R|\Q}(X)$. $\sqbullet$
\end{remark}

The next result is our $L|K$-generic projection theorem. In its statement, we will mention the concepts of $K$-biregular isomorphism and $K$-nonsingular/$K$-singular loci introduced in Definitions \ref{def:331} and \ref{E|K-regular}.

\begin{thm}[$L|K$-generic projection theorem]\label{emb1}
Let $X\subset L^n$ be a $K$-algebraic set of dimension $d$ and let $r:=\max\{e^{L|K}(X)+d-1,2d+1\}$. If $r<n$, then there exists a non-empty $K$-Zariski open subset $\Omega$ of $\mc{M}_{r,n-r}(K)=K^{r(n-r)}$ such that, for each $A\in\Omega$, we have:
\begin{itemize}
\item[$(\mr{i})$] $\pi_A(X)\subset L^r$ is $K$-algebraic.
\item[$(\mr{ii})$] The restriction of $\pi_A$ from $X\subset L^n$ to $\pi_A(X)\subset L^r$ is a $K$-biregular isomorphism. In particular, $\Reg^K(\pi_A(X))=\pi_A(\Reg^K(X))$ and $\Sing^K(\pi_A(X))=\pi_A(\Sing^K(X))$.
\end{itemize}
\end{thm}

An immediate consequence is the following.

\begin{cor}\label{cor:emb1}
Let $X\subset L^n$ be a $K$-nonsingular $K$-algebraic set of dimension $d$ and let $r:=2d+1$. If $r<n$, then there exists a non-empty $K$-Zariski open subset $\Omega$ of $\mc{M}_{r,n-r}(K)=K^{r(n-r)}$ such that, for each $A\in\Omega$, we have:
\begin{itemize}
\item[$(\mr{i})$] $\pi_A(X)\subset L^r$ is a $K$-nonsingular $K$-algebraic set.
\item[$(\mr{ii})$] The restriction of $\pi_A$ from $X$ to $\pi_A(X)$ is a $K$-biregular isomorphism.
\end{itemize}
\end{cor}
\begin{proof}
Apply Theorem \ref{emb1} to $X$. As $X\subset R^n$ is $K$-nonsingular, we have $r=2d+1$, because $e^{L|K}(X)=d$ by Remark \ref{rem528}$(\mr{ii})$. In addition, $\Reg^K(\pi_A(X))=\pi_A(\Reg^K(X))=\pi_A(X)$, so the $K$-algebraic set $\pi_A(X)\subset R^r$ is also $K$-nonsingular.
\end{proof}

To prove Theorem \ref{emb1}, we will adapt the classical argument used in the standard $L|L$-case to the present situation. As we will see, this is not an easy task. In fact, it will require a specific preparation based on several concepts and results that were first introduced and proved in the previous sections. To enable the reader to follow the proof of the $L|K$-generic projection theorem more quickly, we will postpone the proofs of some preliminary results to Appendix \ref{appendix:c}.4.

If $n\geq2$, we denote $\tilde{\x}$ the indeterminates $(\x_1,\ldots,\x_{n-1})$, $\tilde{x}$ the coordinates $(x_1,\ldots,x_{n-1})^T$ of $L^{n-1}$ and $\Pi:L^n\to L^{n-1}$ the projection map onto the first $n-1$ coordinates, that is, $\Pi(x):=\tilde{x}$ for each $x\in L^n$.

\begin{lem}\label{kJP}
Suppose that $L$ is algebraically closed and $n\geq2$. Let $\gta$ be an ideal of $K[\x]$ that contains a monic polynomial $\ell(\x):=\x_n^d+a_{n-1}(\tilde{\x})\x_n^{d-1}+\cdots+a_1(\tilde{\x})\x_n+a_0(\tilde{\x})\in K[\tilde{\x}][\x_n]=K[\x]$ with respect to $\x_n$. Then it holds:
$$
\Pi(\ZZ_L(\gta))=\ZZ_L(\gta\cap K[\tilde{\x}]).
$$
\end{lem}
\begin{proof}
If $K=L$, this result reduces to the well-known Project Theorem \cite[Thm.2.2.8]{jp}. Thus, applying \cite[Thm.2.2.8]{jp} to the ideal $\gta L[\x]$ of $L[\x]$, we obtain:
$$
\Pi(\ZZ_L(\gta))=\Pi(\ZZ_L(\gta L[\x]))=\ZZ_L((\gta L[\x])\cap L[\tilde{\x}]).
$$
By Corollary \ref{intkx}, it holds $(\gta L[\x])\cap L[\tilde{\x}]=(\gta\cap K[\tilde{\x}])L[\tilde{\x}]$, so
$$
\Pi(\ZZ_L(\gta))=\ZZ_L((\gta\cap K[\tilde{\x}])L[\tilde{\x}])=\ZZ_L(\gta\cap K[\tilde{\x}]),
$$
as required.
\end{proof}

Suppose we have two $K$-algebraic sets $X\subset L^n$ and $Y\subset L^m$, and a $K$-regular map $f:X\to Y$ (see Definition \ref{def:331}). Of course, $X\subset L^n$ and $Y\subset L^m$ are also usual ($L$-)algebraic sets, and $f:X\to Y$ is also a usual ($L$-)regular map. Thus, for each $a\in X$, we can speak about the usual differential $d_af:T_a(X)\to T_{f(a)}(Y)$ of $f$ at $a$ between the usual Zariski tangent spaces. Such differential coincides with the $L$-differential $d^L_af:T^L_a(X)\to T^L_{f(a)}(Y)$, see Remarks \ref{rem:435}$(\mr{ii})$.

\begin{thm}\label{thmproj}
Suppose that $L$ is algebraically closed and $n\geq2$. Consider a $K$-algebraic set $X\subset L^n$ and assume that $\II_K(X)$ contains a monic polynomial $\ell(\x):=\x_n^d+a_{n-1}(\tilde{\x})\x_n^{d-1}+\cdots+a_1(\tilde{\x})\x_n+a_0(\tilde{\x})\in K[\tilde{\x}][\x_n]=K[\x]$ with respect to $\x_n$. Set $Y:=\Pi(X)$ and denote $\pi:X\to Y$ the restriction of $\Pi$ from $X$ to $Y$. Then $Y\subset L^{n-1}$ is a $K$-algebraic set, $\dim(Y)=\dim(X)$ and $\pi$ is a $K$-Zariski closed map, that is, the image under $\pi$ of each $K$-Zariski closed subset of $X$ is a $K$-Zariski closed subset of $Y$. Moreover, if we pick a point $a\in X\cap(\kbar^\sqbullet)^n$ and we set $b:=\pi(a)\in Y\cap(\kbar^\sqbullet)^{n-1}$, we have the following properties:
\begin{itemize}
\item[$(\mr{i})$] If $\pi^{-1}(b)=\{a\}$, then the homomorphism $\pi^*:\reg^{K|K}_{Y,b}\to\reg^{K|K}_{X,a}$ is finite.
\item[$(\mr{ii})$] If $\pi^{-1}(b)=\{a\}$ and $d_a\pi:T_a(X)\to T_b(Y)$ is injective, then:
 \begin{itemize}
\item[$(\mr{a})$] $K[a]=K[b]$, the differential $d^{K|K}_a\pi:T^{K|K}_a(X)\to T^{K|K}_b(Y)$ is well-defined and it is injective;
\item[$(\mr{b})$] $\pi^*:\reg^{K|K}_{Y,b}\to\reg^{K|K}_{X,a}$ is a ring isomorphism;
\item[$(\mr{c})$] there exists a $K$-Zariski open neighborhood $U$ of $a$ in $X$ such that $\pi(U)$ is a $K$-Zariski open neighborhood of $b$ in $Y$ and the restriction of $\pi$ from $U$ to $\pi(U)$ is a $K$-biregular isomorphism;
\item[$(\mr{d})$] $d_a\pi$ is a $L$-linear isomorphism and $d_a^{K|K}\pi$ is a $K[a]$-linear isomorphism.
 \end{itemize}
 \end{itemize}
In addition, we have:
\begin{itemize} 
\item[$(\mr{iii})$] Let $V$ be a $K$-Zariski open subset of $X$ and let $\pi|_V:V\to\pi(V)$ be the corresponding restriction of $\pi$. If $\pi^{-1}(\pi(c))=\{c\}$ and the differential $d_c\pi:T_c(X)\to T_{\pi(c)}(Y)$ is injective for each $c\in V\cap(\kbar^\sqbullet)^n$, then $\pi(V)$ is a $K$-Zariski open subset of $Y$ and $\pi|_V$ is a $K$-biregular isomorphism.
\item[$(\mr{iv})$] The map $\pi:X\to Y$ is a $K$-biregular isomorphism if and only if $\pi$ is injective and its differential $d_c\pi:T_c(X)\to T_{\pi(c)}(Y)$ is injective for each $c\in X\cap(\kbar^\sqbullet)^n$.
\end{itemize}
\end{thm}
\begin{proof}
 The proof will use Lemma \ref{kJP} and results from Sections \ref{s1} and \ref{s3}: see Appendix \ref{appendix:c}.4.
\end{proof}

As in Subsection \ref{subsec:proj}, for each $i\in\{0,\ldots,n\}$, we define $H^n_i:=\{[x_0,x]\in\PP^n(L):x_i=0\}$, $U^n_i:=\PP^n(L)\setminus H^n_i$ and $\theta^n_i:U^n_i\to L^n$ by $\theta^n_i([x_0,x]):=\big(\frac{x_0}{x_i},\ldots,\frac{x_{i-1}}{x_i},\frac{x_{i+1}}{x_i},\ldots,\frac{x_n}{x_i}\big)$. Identify $\PP^n(L)\setminus H^n_0$ with $L^n$ via $\theta^n_0$.} By Remarks \ref{rem265}$(\mr{iii})$, if $S\subset\PP^n(L)$ is $K$-constructible, then $\dim_K(S)=\dim_L(S)$. In particular, this is true when $S\subset\PP^n(L)$ is $K$-algebraic. \textit{In the remaining part of this subsection, we write $\dim(S):=\dim_K(S)=\dim_L(S)$ for each $K$-constructible set $S\subset\PP^n(L)$.}

By Definition \ref{V_AP_A}, given a matrix $A\in\mc{M}_{r,n-r}(L)$, $V_A$ is the $L$-vector subspace of $L^n$ defined by the $L$-linear equations $x'=Ax''$ and $\pi_A:L^n\to L^r$ is the corresponding projection $\pi(x)=x'-Ax''$ in the direction of $V_A$, where $x':=(x_1,\ldots,x_r)^T$ and $x'':=(x_{r+1},\ldots,x_n)^T$ are the coordinates of $L^r$ and $L^{n-r}$, respectively.

\begin{notation}\label{notaz526}
For each $A\in\mc{M}_{r,n-r}(L)$, we define the $(n-r-1)$-dimensional projective subspace $\widehat{V}_A$ of $\PP^n(L)$ by
$$
\widehat{V}_A:=H^n_0\cap\zcl_{\PP^n(L)}^L(V_A)=\{[0,x',x'']\in\PP^n(L):x'=Ax''\},
$$
where $\zcl_{\PP^n(L)}^L(V_A)$ is the usual ($L$-)Zariski closure of $V_A$ in~$\PP^n(L)$ (see Definition \ref{def:proj-K-zar}).
\end{notation}

Theorem \ref{thmproj} can be generalize as follows.

\begin{thm}\label{corproj}
Suppose that $L$ is algebraically closed. Let $X\subset L^n$ be a $K$-algebraic set, let $\overline{X}$ be the $K$-Zariski closure of $X$ in $\PP^n(L)$, let $r\in\N^*$ with $r<n$ and let $A\in\mc{M}_{r,n-r}(K)$ be such that $\overline{X}\cap\widehat{V}_A=\varnothing$. Consider the projection $\pi_A:L^n\to L^r$, set $Y:=\pi_A(X)$ and denote $\pi:X\to Y$ the restriction of $\pi_A$ from $X$ to $Y$. Then $Y\subset L^r$ is a $K$-algebraic set, $\dim(Y)=\dim(X)$ and $\pi$ is a $K$-Zariski closed map. Moreover, if we pick a point $a\in X\cap(\kbar^\sqbullet)^n$ and we set $b:=\pi(a)\in Y\cap(\kbar^\sqbullet)^r$, then items $(\mr{i})$ and $(\mr{ii})$ stated in Theorem \ref{thmproj} hold. In addition, items $(\mr{iii})$ and $(\mr{iv})$ stated in Theorem \ref{thmproj} are also true.
\end{thm}
\begin{proof}
See Appendix \ref{appendix:c}.4.
\end{proof}

The next result is a consequence of Theorem \ref{thm:K-elimination}$(\mr{i})$ and the fact that the set $K^n$ is dense in $L^n$ with respect to the usual ($L$-)Zariski topology: recall that $K$ has characteristic zero, so it is infinite.

\begin{lem}\label{618}
Suppose that $L$ is algebraically closed. Let $r\in\N^*$ and let $W\subset\PP^n(L)$ be a $K$-algebraic set such that $r<n$, $W\subset H^n_0$ and $\dim(W)\leq r-1$. Define $\Omega:=\{A\in\mc{M}_{r,n-r}(K):W\cap\widehat{V}_A=\varnothing\}$. Then $\Omega$ is a non-empty $K$-Zariski open subset of $\mc{M}_{r,n-r}(K)=K^{r(n-r)}$.
\end{lem}
\begin{proof}
See Appendix \ref{appendix:c}.4.
\end{proof}

Let us extend the affine notion of $K$-regular map introduced in Definition \ref{def:331} to the proj\-ective case.

\begin{defn} \label{def:331-projection}
Let $X$ be a subset of $\PP^n(L)$ and let $Y$ be a subset of $\PP^m(L)$. We endow $X$ and $Y$ with the relative topology induced by the $K$-Zariski topology of $\PP^n(L)$ and $\PP^m(L)$, respectively. Let $f:X\to Y$ be a continuous map. Let $a\in X$, let $i\in\{0,\dots,n\}$ and let $j\in\{0,\ldots,m\}$ be such that $a\in U^n_i$ and $f(a)\in U^m_j$. Define the map $f_{ij}:\theta^n_i(X\cap U^n_i\cap f^{-1}(Y\cap U^m_j))\to\theta^m_j(Y\cap U^m_j)$ by $f_{ij}(x)=\theta^m_j(f((\theta^n_i)^{-1}(x)))$. We say that $f$ is \emph{$K$-regular at $a$} if $f_{ij}$ is $K$-regular at $a$ in the sense of Definition \ref{def:331}. We say that $f$ is \emph{$K$-regular} if $f$ is $K$-regular at each point of $X$. $\sqbullet$
\end{defn}

By Remarks \ref{rem265}$(\mr{i})$, the above definition is well-posed, that is, it does not depend on the chosen indices $i$ and $j$. The next result is a consequence of Theorem \ref{thm:K-elimination}$(\mr{ii})$.

\begin{prop}\label{prop-added}
Suppose that $L$ is algebraically closed. If $X\subset\PP^n(L)$ is a $K$-constructible set and $f:X\to\PP^m(L)$ is a $K$-regular map, then $f(X)\subset\PP^m(L)$ is also a $K$-constructible set.
\end{prop}
\begin{proof}
Let $\Gamma$ be the graph of $f$ and let $\rho:\PP^n(L)\times\PP^m(L)\to\PP^m(L)$ be the projection $(\alpha,\beta)\mapsto\beta$. As $X\subset\PP^n(L)$ is $K$-constructible and $f$ is $K$-regular, we have that $\Gamma\subset\PP^n(L)\times\PP^m(L)$ is $K$-constructible, so $f(X)=\rho(\Gamma)\subset\PP^m(L)$ is also $K$-constructible by Theorem \ref{thm:K-elimination}$(\mr{ii})$.
\end{proof}

We need one last preliminary result.

\begin{lem}\label{dimcompl}
Let $R$ be a real closed field, let $C:=R[\ii]$ be its algebraic closure, let $X\subset R^n$ be a $K$-algebraic set and let $X_C$ be the $K$-Zariski closure of $X$ in $C^n$. Then there exists a $K$-Zariski open neighborhood $V$ of $X$ in $X_C$ such that $\max_{a\in V}\{e^{C|K}_a(X_C)\}=e^{R|K}(X)$.
\end{lem}
\begin{proof}
By Corollary \ref{kreliablec}, $\II_C(X_C)=\II_K(X)C[\x]$, so $e^{R|K}_a(X)=\dim_R(T^{R|K}_a(X))$ is equal to $e^{C|K}_a(X_C)=\dim_C(T^{C|K}_a(X_C))$ for each $a\in X$. Let $\{g_1,\ldots,g_s\}$ be a system of generators of $\II_K(X)$ in $K[\x]$, let $J_a\in\mc{M}_{s,n}(K)$ be the matrix $J_a:=\big(\frac{\partial g_i}{\partial \x_j}(a)\big)_{i=1,\ldots,s,\,j=1,\ldots,n}$ for each $a\in X_C$, let $N:=e^{R|K}(X)$ and let $D:=\{a\in X_C:\dim_C(T^{C|K}_a(X_C))>N\}=\{a\in X_C:{\rm rk}(J_a)<n-N\}$. As $D\subset C^n$ is $K$-algebraic and $D\cap X=\varnothing$, the set $V:=X_C\setminus D$ has the required property. 
\end{proof}

We are ready to prove the $L|K$-generic projection theorem. 

\begin{proof}[Proof of Theorem \em \ref{emb1}]
Let $X\subset L^n$ be a $K$-algebraic set of dimension $d$. Suppose that $r:=\max\{e^{L|K}(X)+d-1,2d+1\}<n$. Let $\{g_1,\ldots,g_s\}$ be a system of generators of $\II_K(X)$ in $K[\x]$. 

\emph{First, we assume that $L$ is algebraically closed.}

Define the following subsets of $L^{2n}$:
\begin{align*}
\Delta_X&:=\{(x,y)\in X\times X:x=y\},\\
S_X&:=(X\times X)\setminus\Delta_X,\\
TX&:=\{(x,v)\in X\times L^n:\langle\nabla g_i(x),v\rangle=0\,\text{ for all $i=1,\ldots,s$}\},\\
T^*X&:=TX\setminus(X\times\{0\}),\\
Z^0_X&:=T^*X\cap(\Reg^{L|K}(X)\times L^n),\\
Z^1_X&:=T^*X\cap(\Sing^{L|K}(X)\times L^n)
\end{align*}
and the following maps:
\begin{align*}
&F:S_X\to H^n_0, \quad F(x,y):=[0,x-y],\\
&G:T^*X\to H^n_0, \quad G(x,v):=[v].
\end{align*}
Observe that $T^*X=Z^0_X\cup Z^1_X$. Keeping in mind Corollary \ref{30'}$(\mr{iii})$, we see that $S_X$, $Z^0_X$, $Z^1_X$ and $T^*X$ are $K$-constructible subsets of~$L^{2n}$ and $\dim(S_X)=\dim_L(S_X)=2d$, $\dim(Z^0_X)=\dim_L(Z^0_X)=2d$ and
$$
\dim(Z^1_X)=\dim_L(Z^1_X)\leq e^{L|K}(X)+\dim_L(\Sing^{L|K}(X))\leq e^{L|K}(X)+d-1,
$$
so $\dim(T^*X)=\dim_L(T^*X)=\dim_L(Z^0_X\cup Z^1_X)\leq\max\{2d,e^{L|K}(X)+d-1\}\leq r$. As $F$ and $G$ are $K$-regular maps, their graphs are $K$-constructible subsets of $L^{2n}\times H^n_0=L^{2n}\times\PP^{n-1}(L)\subset\PP^{2n}(L)\times\PP^{n-1}(L)$. By Proposition \ref{prop-added}, we have that $F(S_X)$ and $G(T^*X)$ are $K$-constructible subsets of $H^n_0=\PP^{n-1}(L)$. In addition, we have $\dim(F(S_X))=\dim_L(F(S_X))\leq 2d\leq r-1$ and $\dim(G(T^*X))=\dim_L(G(T^*X))\leq r-1$ (recall that we have taken projective points $[v]\in H^n_0$ associated to the corresponding non-zero vectors $v\in L^n$). Consequently, $F(S_X)\cup G(T^*X)$ is a $K$-constructible subset of $H^n_0$ of dimension $\leq r-1$. Let $\overline{X}$ be the $K$-Zariski closure of $X$ in $\PP^n(L)$, which has dimension $d$ and none of its $K$-irreducible components is contained in $H^n_0$.

Indeed, observe that $X$ is contained in the union of those $K$-irreducible components of $\ol{X}$ that are not contained in $H^n_0$. As $\ol{X}$ is the $K$-Zariski closure of $X$ in $\PP^n(L)$, none of the $K$-irreducible components of $\ol{X}$ is contained in $H^n_0$. As none of the $K$-irreducible components of $\ol{X}$ is contained in $H^n_0$ and $H_0^n$ is $K$-irreducible, we deduce $\dim(\overline{X}\cap H^n_0)\leq d-1<r-1$. 

Let $W$ be the $K$-Zariski closure of the union $(\overline{X}\cap H^n_0)\cup F(S_X)\cup G(T^*X)$ in $H^n_0$ (or, equivalently, in $\PP^n(L)$). As $\dim(W)\leq r-1$, we have by Lemma \ref{618} that $\Omega:=\{A\in\mc{M}_{r,n-r}(K):W\cap\widehat{V}_A=\varnothing\}$ is a non-empty ($K$-)Zariski open subset of $\mc{M}_{r,n-r}(K)$. By Theorem \ref{corproj} (especially item $(\mr{iv})$), each matrix $A\in\Omega$ has the required properties.

\emph{Next, we assume that $L$ is a real closed field $R$. Set $C:=R[\ii]$}

Let $X_C:=\zcl_{C^n}^K(X)\subset C^n$. By Lemma \ref{dimcompl}, there exists a $K$-Zariski open neighborhood $V$ of $X$ in $X_C$ such that
\begin{equation}\label{maxV}
\textstyle
\max_{a\in V}\big\{e^{C|K}_a(X_C)\big\}=e^{R|K}(X).
\end{equation}
As $\II_K(X_C)=\II_K(X)$, we have $\dim_C(X_C)=\dim_K(X_C)=\dim_K(X)=\dim(X)=d$. Define $\Delta_{X_C}$, $S_{X_C}$, $TX_C$, $T^*X_C$, $Z^0_{X_C}$ and $Z^1_{X_C}$ as above replacing $X$ with $X_C$, and define
$$
T^*V:=T^*X_C\cap(V\times L^n),\quad Z^0_V:=Z^0_{X_C}\cap(V\times L^n),\quad
Z^1_V:=Z^1_{X_C}\cap(V\times L^n),
$$
so $T^*V=Z^0_V\cup Z^1_V$. We also define the $K$-regular maps $F:S_{X_C}\to H^n_0=\PP^{n-1}(C)$ and $G:T^*V\to H^n_0=\PP^{n-1}(C)$ as above, that is, $F(x,y):=[0,x-y]$ and $G(x,v):=[v]$. Let $\overline{X_C}$ be the $K$-Zariski closure of $X_C$ in $\PP^n(C)$, and let $W'$ and $W''$ be the $K$-Zariski closures in $H^n_0=\PP^{n-1}(C)$ of the $K$-constructible sets $F(S_{X_C})$ and $G(T^*V)$, respectively. We have: $\dim_C(\overline{X_C}\cap H^n_0)=d-1<r-1$, $\dim_C(F(S_{X_C}))\leq 2d$, $\dim_C(Z^0_V)=2d$ and, by \eqref{maxV}, $\dim_C(Z^1_V)\leq e^{R|K}(X)+d-1$ so $\dim_C(T^*V)\leq r$ and $\dim_C(G(T^*V))\leq r-1$. It follows that the $K$-algebraic subset $W:=(\overline{X_C}\cap H^n_0)\cup W'\cup W''$ of $H^n_0=\PP^{n-1}(C)$ has $C$-dimension $\leq r-1$. As above, by Lemma \ref{618}, the set $\Omega:=\{A\in\mc{M}_{r,n-r}(K):W\cap\widehat{V}_A=\varnothing\}$ is non-empty and ($K$-)Zariski open in $\mc{M}_{r,n-r}(K)$.

Let $A\in\Omega$ and let $\pi_{C,A}:C^n\to C^r$ be the projection $\pi_{C,A}(x)=x'-Ax''$. By Theorem \ref{corproj} (especially item $(\mr{iii})$), $Y_C:=\pi_{C,A}(X_C)\subset C^r$ is $K$-algebraic, $\pi_{C,A}(V)$ is a $K$-Zariski open subset of $Y_C$ and the restriction $\pi':V\to\pi_{C,A}(V)$ of $\pi_{C,A}$ is a $K$-biregular isomorphism. Observe that $X_C\cap R^n=\ZZ_C(\II_K(X_C))\cap R^n=\ZZ_R(\II_K(X))=X$, $V\cap R^n=V\cap X_C\cap R^n=V\cap X=X$, so $\pi_{C,A}(X)=\pi_{C,A}(X_C\cap R^n)=\pi'(V\cap R^n)$. As $A\in\mc{M}_{r,n-r}(K)\subset\mc{M}_{r,n-r}(R)$, the projection $\pi_{C,A}$ commutes with the conjugations of $C^n$ and $C^r$. Combining the latter fact with the injectivity of $\pi_{C,A}$ and the equality $X_C\cap R^n=X$, we deduce that $\pi_{C,A}(X)$ is equal to the $K$-algebraic set $Y:=Y_C\cap R^r\subset R^r$, so the restriction of $\pi'$ from $X$ to $Y$ is a $K$-biregular isomorphism, as required.
\end{proof}

\subsection{The Nash-Tognoli theorem over $\Q$ and its version for isolated singularities}\label{nash-tognoli-Q} \textit{Endow $\R^n$ with the Euclidean topology and each of its subsets with the relative topology. `Manifold' means `non-empty manifold without bounda\-ry', and `smooth(ly)' means `of class $\cinfty$'.}

Given a smooth manifold $M$, we denote $\cinfty(M,\R^m)$ the set of smooth maps from $M$ to $\R^m$ endowed with the usual weak $\cinfty$ topology, see \cite[p.36]{hirsch}.

The celebrated Nash-Tognoli theorem \cite{nash,to1} asserts the following: \textit{Every compact smooth manifold $M$ is smoothly diffeomorphic to a nonsingular algebraic subset $M'$ of some $\R^n$, called algebraic model of $M$. More precisely, if $M$ has dimension $d$ and $\psi:M\to\R^{2d+1}$ is a smooth embedding (whose existence is guaranteed by Whitney's embedding theorem), then $\psi$ can be approximated in $\cinfty(M,\R^{2d+1})$ by an arbitrarily close smooth embedding $\phi:M\to\R^{2d+1}$ whose image $M':=\phi(M)$ is a nonsingular algebraic subset of~$\R^{2d+1}$.}

There is an extensive literature dedicated to this fundamental result. See the books \cite[Sect.II.8]{akbking:tras}, \cite[Ch.14]{bcr}, the papers \cite{Be2022,GT2017,kucharz2011}, the surveys \cite[Sect.1]{delellis} and \cite[Sects.1\&2]{kollar2017} and the numerous references therein.

Let $K=\qr$ or $\Q$. According to Definition \ref{E|K-regular} with $L|E|K=\R|\R|K$, a $K$-algebraic set $X\subset\R^n$ is $K$-nonsingular if the local ring $\reg^{\R|K}_{X,a}:=\R[\x]_{\gtn_a}/(\II_K(X)\R[\x]_{\gtn_a})$ is regular of dimension $d:=\dim(X)$ for each $a\in X$, that is, $\Reg^\Q(X)=X$. By Proposition \ref{30}, this is equivalent to required that either $d=n$ (so $X=\R^n$) or $d<n$ and, for each $a\in X$, there exist $f_1,\ldots,f_{n-d}\in\II_K(X)$ and an open neighborhood $V$ of $a$ in $\R^n$ such that the gradients $\nabla f_1(a),\ldots,\nabla f_{n-d}(a)$ are linearly independent in $\R^n$, and $X\cap V=\{x\in V:f_1(x)=0,\ldots,f_{n-d}(x)=0\}$.

The field $\qr$ of real algebraic numbers is a real closed subfield of $\R$ (the smallest one), so the `extension of coefficients' procedure from $\qr$ to $\R$ is available and therefore the $\R|\qr$-version of Hardt's trivialization theorem for Nash manifolds \cite[Thm.A]{CS1992} is also available. An immediate consequence is the following improved version of the Nash-Tognoli theorem, which we can call `Nash-Tognoli theorem over the real algebraic numbers': \textit{The algebraic model $M'\subset\R^{2d+1}$ of $M$ can be chosen to be a $\qr$-nonsingular $\qr$-algebraic set}.

The field $\Q$ of rational numbers is not a real closed field, so we cannot use the `extension of coefficients' procedure to prove that $M$ has an algebraic model $M'$ in some $\R^n$ that is $\Q$-algebraic.

Let $X\subset\R^n$ be an algebraic set. Describe $X$ as the set of solutions of a fixed polynomial system $f_1=0,\ldots,f_s=0$ in $\R^n$, that is, $X=\{x\in\R^n:f_1(a,x)=\ldots=f_s(a,x)=0\}$, where $f_1,\ldots,f_s\in\Z[\mathtt{a}_1,\ldots,\mathtt{a}_m,\x_1,\ldots,\x_n]$ and $a=(a_1,\ldots,a_m)\in\R^m$ is the vector of the coefficients of the polynomials $f_i$, ordered in some way. In \cite[Thm.11 \& Rmk.13]{PR2020}, making use of Zariski equisingular deformations of the coefficients $a=(a_1,\ldots,a_m)$, Parusi\'{n}ski and Rond proved that $X\subset\R^n$ is homeomorphic to a $\Q$-algebraic subset of $\R^n$, if the field extension of $\Q$ obtained by adding $a_1,\ldots,a_m$ is purely transcendental. One could therefore hope to apply this technique to $X:=M'$, where $M'\subset\R^{2d+1}$ is an algebraic model of $M$ given by the Nash-Tognoli theorem, in order to obtain an algebraic model of $M$ that is $\Q$-nonsingular $\Q$-algebraic or at least $\Q$-algebraic up to homeomorphism. This approach does not work. The reason is that, even when $X$ is a compact nonsingular algebraic hypersurface of $\R^n$, the above Zariski equisingular deformations of the coefficients $a=(a_1,\ldots,a_m)$ preserve the polynomial relations over $\Q$ satisfied by $a_1,\ldots,a_m$. In general, the problem of making real algebraic sets $\Q$-algebraic up to homeomorphism is open, even for germs, see \cite[Open problems 1\,\&\,2, pp.199-200]{P2021}. See also \cite{PP,Ro,Te90} and the references mentioned therein for the related problem of making real algebraic sets $\Q$-algebraic up to equisingular deformation.

Here we present the main results of \cite{GS} that give a complete affirmative solution to the previous open problem for all nonsingular real algebraic sets and, more generally, for all real algebraic sets with isolated singularities, that is, with finitely many singularities. These results were obtained by carefully extending to the $\R|\Q$-algebraic setting the techniques of Nash and Tognoli \cite{nash,to1} in the nonsingular case, and the ones of Akbulut and King \cite{ak1981,akbking:tras} in the case of isolated singularities. This is not an easy task. To do so, the authors of \cite{GS} used the $L|K$-algebraic geometry developed in previous sections of this work, mainly in the case $L|K=\R|\Q$, as well as further developments of the $\R|\Q$-algebraic geometry obtained~in~\cite{GS}.

\begin{remark}[{\cite[Rem.1.4]{GS}}]\label{rem1}
Recall that $\Reg^\Q(X)\subset\Reg(X)=\Reg^\R(X)$ for each $\Q$-algebraic set $X\subset\R^n$ (see Remark \ref{regolare}) and this inclusion can be strict. In Examples~\ref{432}$(\mr{ii})(\mr{iii})(\mr{v})$, we have already seen some examples of noncompact $\Q$-algebraic sets $X\subset\R^n$ with $\Reg^\Q(X)\subsetneqq\Reg(X)$. Below we present a couple of compact examples.

Let $f_1:=\x_1^2+\x_2^2-\sqrt[3]{2}\x_1\in\R [\x]:=\R[\x_1,\x_2]$, let $g_1:=(\x_1^2+\x_2^2)^3-2\x_1^3\in\Q[\x]$ and let $C_1$ be the $\Q$-algebraic circumference $C_1:=\ZZ_\R(f_1)=\ZZ_\R(g_1)$ of $\R^2$. By Proposition \ref{prop:hyper}, we have that $\II_\R(C_1)=(f_1)\R[\x]$, $\II_\Q(C_1)=(g_1)\Q[\x]$ and $\Reg^\Q(C_1)=\{x\in C_1:\nabla g_1(x)=0\}=C_1\,\setminus\{(0,0)\}\subsetneqq C_1=\{x\in C_1:\nabla f_1(x)=0\}=\Reg(C_1)$, so $C_1$ is nonsingular but not $\Q$-nonsingular.

Let $f_2:=\x_2^2+\sqrt[3]{2}\x_1(\x_1-1)^3\in\R[\x]$, let $g_2:=\x_2^6+2\x_1^3(\x_1-1)^9\in\Q[\x]$ and let $C_2$ be the $\Q$-algebraic curve $C_2:=\ZZ_\R(f_2)=\ZZ_\R(g_2)$ of $\R^2$, which is homeomorphic to $C_1$ and has a cusp at $(1,0)$. Using again Proposition \ref{prop:hyper}, we have $\Reg^\Q(C_2)=\{x\in C_2:\nabla g_2(x)=0\}=C_2\setminus\{(0,0),(1,0)\}\subsetneqq C_2\setminus\{(1,0)\}=\{x\in C_2:\nabla f_2(x)=0\}=\Reg(C_2)$. $\sqbullet$ 
\end{remark}

The concept of $\Q$-algebraic model of a smooth manifold is the following:

\begin{defn}[{\cite[Def.1.5]{GS}}]
Let $M$ be a smooth manifold and let $M'\subset\R^m$ be a $\Q$-nonsingular $\Q$-algebraic set. If $M$ is smoothly diffeomorphic to $M'$, then we say that $M'$ is a \emph{$\Q$-algebraic model of~$M$}.~$\sqbullet$
\end{defn}The celebrated Nash-Tognoli theorem \cite{nash,to1} asserts the following: \textit{Every compact smooth manifold $M$ is smoothly diffeomorphic to a nonsingular algebraic subset $M'$ of some $\R^n$, called algebraic model of $M$. More precisely, if $M$ has dimension $d$ and $\psi:M\to\R^{2d+1}$ is a smooth embedding (whose existence is guaranteed by Whitney's embedding theorem), then $\psi$ can be approximated in $\cinfty(M,\R^{2d+1})$ by an arbitrarily close smooth embedding $\phi:M\to\R^{2d+1}$ whose image $M':=\phi(M)$ is a nonsingular algebraic subset of~$\R^{2d+1}$.}

The `Nash-Tognoli theorem over the rationals' is true.

\begin{thm}[{\cite[Thm.1.7]{GS}}]\label{thm:NTQ}
Every compact smooth manifold has a $\Q$-algebraic model.

More precisely, if $M$ is a compact smooth manifold of dimension $d$, $\psi:M\to\R^{2d+1}$ is a smooth embedding and $\,\mc{V}$ is a neighborhood of $\psi$ in $\cinfty(M,\R^{2d+1})$, then there exists a smooth embedding $\phi:M\to\R^{2d+1}$ belonging to $\mc{V}$ such that $\phi(M)\subset\R^{2d+1}$ is a $\Q$-nonsingular $\Q$-algebraic set.
\end{thm}

Let us introduce the notion of $\Q$-determined $\Q$-algebraic set.

\begin{defn}[{\cite[Def.1.9]{GS}}]\label{Q-determined}
A $\Q$-algebraic set $X\subset\R^n$ is said to be \emph{$\Q$-determined} if $\Reg^\Q(X)=\Reg(X)$. $\sqbullet$ 
\end{defn}

Observe that, if a $\Q$-algebraic set $X\subset\R^n$ is $\Q$-determined, then $\Sing(X)=X\setminus\Reg(X)$ is equal to $\Sing^\Q(X)=X\setminus\Reg^\Q(X)$. The curves $C_1$ and $C_2$ of $\R^2$ defined in Remark \ref{rem1} are examples of $\Q$-algebraic sets that are not $\Q$-determined, the first is nonsingular and the second has only one singular point.

For every $n,m\in\N$ with $m>n$, we identify $\R^n$ with the subset $\R^n\times\{0\}$ of $\R^n\times\R^{m-n}=\R^m$, so we can write $\R^n\subset\R^m$ and every subset of $\R^n$ is also a subset of $\R^m$.

Let $S$ be a topological space. We denote $\czero(S,\R^m)$ the set of continuous maps from $S$ to~$\R^m$, endowed with the usual compact-open topology. A map $\phi\in\czero(S,\R^m)$ is a continuous embedding if it is a homeomorphism onto its image.

\begin{thm}[{\cite[Thm.1.10\,\&\,Rem.1.11$(\mr{ii})$]{GS}}]\label{thm:main}
Every real algebraic set with isolated singularities is semialgebraically homeo\-morphic to a $\Q$-determined $\Q$-algebraic set with isolated singularities. 

More precisely, the following holds. Let $X$ be an algebraic subset of $\R^n$ of dimension $d$ with isolated singularities, let $m:=n+2d+4$, and let $i:X\hookrightarrow\R^m$ and $j:\Reg(X)\hookrightarrow\R^m$ be the inclusions of $X$ and $\Reg(X)$ in $\R^m$. Then, for every neighborhood $\mc{U}$ of $i$ in $\czero(X,\R^m)$ and for every neighborhood $\mc{V}$ of $j$ in $\cinfty(\Reg(X),\R^m)$, there exists a semialgebraic continuous embedding $\phi:X\to\R^m$ with the following properties:
\begin{itemize}
 \item[$(\mr{i})$] $X':=\phi(X)$ is a $\Q$-determined $\Q$-algebraic subset of $\R^m$ such that $\Sing(X')=\phi(\Sing(X))$ or, equivalently, $\phi(\Reg(X))=\Reg(X')$.
 \item[$(\mr{ii})$] The restriction $\phi|_{\Reg(X)}:\Reg(X)\to\R^m$ of $\phi$ to $\Reg(X)$ is a Nash embedding. In particular, the Nash manifolds $\Reg(X)$ and $\Reg(X')$ are Nash diffeomorphic.
 \item[$(\mr{iii})$] $\phi\in\mc{U}$ and $\phi|_{\Reg(X)}\in\mc{V}$.
\end{itemize}
\end{thm}

If we are willing to lose approximation properties $(\mr{iii})$, then we can drop the embedding dimension of $X'$.

\begin{thm}[{\cite[Thm.1.12]{GS}}]\label{thm:main-2}
Every real algebraic set $X$ of dimension $d$ with isolated singularities is semi\-algebraically homeomorphic to a $\Q$-determined $\Q$-algebraic subset $X'$ of $\R^{2d+5}$.

More precisely, there exists a semialgebraic homeomorphism $\eta:X\to X'$ such that $\eta(\Reg(X))=\Reg(X')$ and the restriction of $\eta$ from $\Reg(X)$ to $\Reg(X')$ is a Nash diffeomorphism.
\end{thm}

A result of Shiota \cite[Rem.VI.2.11, p.208]{Sh} asserts that every noncompact Nash submanifold of some $\R^n$ is Nash diffeomorphic to a nonsingular real algebraic set. Combining this result, Theorem \ref{thm:NTQ}, the noncompact nonsingular cases of Theorems \ref{thm:main} and \ref{thm:main-2}, and the $L|K$-generic projection Corollary \ref{cor:emb1} with $L|K=\R|\Q$, we immediately obtain:

\begin{cor}[{\cite[Cor.1.13]{GS}}]\label{cor:main-3}
We have:
\begin{itemize}
 \item[$(\mr{i})$] Every (possibly noncompact) Nash submanifold $M$ of some $\R^n$ of dimension $d$ is Nash diffeomorphic to a $\Q$-nonsingular $\Q$-algebraic subset of $\R^{2d+1}$.
 \item[$(\mr{ii})$] If $M$ is a noncompact nonsingular algebraic subset of $\R^n$ of dimension $d$, $m:=n+2d+4$ and $\mc{V}$ is a neighborhood of the inclusion map $M\hookrightarrow\R^m$ in $\cinfty(M,\R^m)$, then there exists a Nash embedding $\phi:M\to\R^m$ belonging to $\mc{V}$ such that $\phi(M)\subset\R^m$ is a $\Q$-nonsingular $\Q$-algebraic set.
\end{itemize}
\end{cor}

Given an algebraic set $X\subset\R^n$ and $a\in X$, the local dimension $\dim(X_a)$ of $X$ at $a$ is the dimension of a small enough semialgebraic neighborhood of $a$ in~$X$ (see \cite[Def.2.8.11]{bcr}). Another consequence of Theorem \ref{thm:main-2} is the following result concerning real algebraic set germs with an isolated singularity.

\begin{thm}[{\cite[Thm.1.14]{GS}}]\label{thm:main-germs}
Every real algebraic set germ with an isolated singularity is semialgebraically homeomorphic to a real $\Q$-determined $\Q$-algebraic set germ with an isolated singularity. 

More precisely, the following holds. Let $X\subset\R^n$ be an algebraic set containing the origin~$O$ of $\R^n$, and let $d:=\dim(X_O)$. Suppose that $O$ is an isolated point of $\Sing(X)$. Then there exist a compact $\Q$-determined $\Q$-algebraic set $X'\subset\R^{2d+4}$ that contains the origin $O'$ of $\R^{2d+4}$, a semialgebraic open neighborhood $U$ of $O$ in $X$, a semialgebraic open neighborhood $U'$ of $O'$ in $X'$ and a semialgebraic homeomorphism $\phi:U\to U'$ such that $\Sing(X)\cap U=\{O\}$, $\Sing(X')\cap U'=\{O'\}$, $\phi(O)=O'$ and the restriction of $\phi$ from $U\setminus\{O\}$ to $U'\setminus\{O'\}$ is a Nash diffeomorphism.
\end{thm}

For further results and remarks, we refer the readers to \cite{GS}.

%%%

\addtocontents{toc}{\protect\setcounter{tocdepth}{1}}

\appendix

\section{Extension of coefficients and Euclidean topologies in the algebraically closed case}\label{appendix}

\subsection{Extension of coefficients and Euclidean topologies: from the real case to the complex case}\label{appendixA1}

It is well-known that the nontrivial subgroups of finite order in the group ${\rm Aut}(C)$ of automorphisms of an algebraically closed field $C$ are subgroups of order $2$, that is, subgroups generated by an involution \cite{as} or \cite[Ch.VI.\S11.Thm.17]{j}. In addition, the fixed field $R$ of such involution is a real closed subfield of $C$ and $C\cong R[\ii]$. In \cite{ba}, it is shown that there always exist involutions in ${\rm Aut}(C)$ (recall that we are assuming that $C$ has characteristic zero, as all the fields involved in this paper). In \cite{sch}, the author studies the family of conjugacy classes of the involutions in ${\rm Aut}(C)$, that is, the family of non-equivalent real structures inside~$C$. By \cite{a,ba}, this family depends on the transcendence degree of $C$ over the prime field $\Q$. In particular, all the involutions of $C$ are conjugate with each other if and only if $C$ coincides with the algebraic closure $\qbar$ of $\Q$. Thus, $\qbar$ is the unique algebraically closed field that has only one real structure inside, that is, the field $\qr=\qbar\cap\R$ of real algebraic numbers.

Extension of coefficients between algebraically closed fields or real closed fields is an important tool in complex and real algebraic geometry, see \cite[\S3]{dm} and \cite[\S5]{bcr}. The algebraically closed case reduces to the real closed one. Let $L|C$ be an extension of algebraically closed fields, let $\varphi:C\to C$ be an involution and let $R$ be the fixed subfield of~$\varphi$, so $C\cong R[\ii]$. By \cite[Lem.2]{sch}, there exists an involution $\psi:L\to L$ that extends $\varphi$. If $F$ denotes the fixed subfield of $\psi$, then $L\cong F[\ii]$, $F|R$ is an extension of real closed fields and we can perform the extension of coefficients from $C$ to $L$ via the extension of coefficients from $R$ to~$F$.

Consider again an algebraically closed field $C$ and a real closed subfield $R$ of $C$ such that $C\cong R[\ii]$. The topology of the real closed field $R$ is the one induced by its unique ordering and the topology induced on each $R^n$ is the product topology. We can endow $C^n$ with the Euclidean topology of $C^n=R^{2n}$. The real closed field $R$ is in general not uniquely determined by $C$ and $C^n$ may have several different Euclidean topologies, as we will see below. All of these Euclidean topologies are finer than the usual $C$-Zariski topology of $C^n$, because the singletons are closed and the polynomial functions $C^n=R^{2n}\to C=R^2$ are continuous with respect to each Euclidean topology.

\subsection{Non-homeomorphic Euclidean topologies on $\C$}\label{appendixA2} The Euclidean topology on the field $\C$ of complex numbers induced by the field $\R$ of real numbers is the usual one, which is locally compact. We provide below an example of Euclidean topology on $\C$, which is not locally compact and therefore not even homeomorphic to the usual Euclidean topology. This example was suggested to us by Elias Baro. 

Consider the field $\Q((\t))$ of formal meromorphic series with coefficients in $\Q$. As a set, $\Q((\t))$ has the following form
$$
\textstyle
\Q((\t))=\big\{\sum_{i\geq k}a_i\t^i: k\in\Z,\, \text{$a_i\in\Q$ for each $i\geq k$},\, a_k\neq0\big\}\cup\{0\},
$$
so it has the same cardinality of $\R$ (thus of $\C$). The algebraic closure of $\Q((\t))$ is the field $\qbar((\t^*))$ of Puiseux series with coefficients in $\qbar$, see \cite[p.98]{walker}. The cardinality of $\ol{\Q}((\t^*))$ is the same as the cardinality of $\Q((\t))$, so $\C$ and $\qbar((\t^*))$ are algebraic closed fields of the same non-countable cardinality. By Steinitz's Theorem \cite[\S23.6\;\&\;\S24, p.301]{steinitz}, $\C$ and $\qbar((\t^*))$ are isomorphic, so $\C$ is the algebraic closure of both $\R$ and the real closure of $\Q((\t))$, which is the field $R:=\qr((\t^*))$ of Puiseux series with coefficients in~$\qr$. Recall that a non-zero element $\xi$ of $R$ is positive if $\xi=\sum_{i\geq k}a_i\t^{i/q}$ with $k\in\Z$, $q\in\N^*$, $a_i\in\qr$ for each $i\geq k$ and $a_k>0$ (in $\qr$).

Let us prove that the topology on $\C$ induced by $R$ is not locally compact. To do so, it is enough to show that $R$ is not locally compact or, equivalently, the interval $[0,1]$ of $R$ is not compact. As usual, if $\xi,\zeta\in R$ with $\xi<\zeta$, we denote $(\xi,\zeta)$ and $[\xi,\zeta]$ the corresponding open and closed intervals of $\qr((\t^*))$. Endow $[0,1]$ with the relative topology induced by that of $R$. For each $\rho\in[0,1]$, we define an open neighborhood $U_\rho$ of $\rho$ in $[0,1]$ as follows. If $\rho\in[0,1]\cap\qr$, we set $U_\rho:=(\rho-\t,\rho+\t)\cap[0,1]$. If $\rho\in[0,1]\setminus\qr$, then $\rho=a_0+\sum_{i\geq k}a_i\t^{i/q}$ for some $k\geq1$, $q\in\N^*$, $a_0\in\qr\cap[0,1]$, $a_i\in\qr$ for each $i\geq k$ and $a_k\neq0$ (because if $a_i=0$ for $i>0$, then $\rho\in\qr$). We set $U_\rho:=(\rho-\t^{(k+1)/q},\rho+\t^{(k+1)/q})\subset[0,1]$. By construction, $\{U_\rho\}_{\rho\in[0,1]}$ is an open cover of $[0,1]$, $U_\rho\cap([0,1]\cap\qr)=\{\rho\}$ for each $\rho\in[0,1]\cap\qr$ and $U_\rho\cap([0,1]\cap\qr)=\varnothing$ for each $\rho\in[0,1]\setminus\qr$. Therefore, each cover of $[0,1]$ extracted from $\{U_\rho\}_{\rho\in[0,1]}$ has to contain the infinite family $\{U_\rho\}_{\rho\in[0,1]\cap\qr}$. Consequently, the closed interval $[0,1]$ of $R$ is not compact, as required.

For recent applications of Puiseux series in real algebraic geometry, we refer to \cite{fs,fy}.

%%%
\section{Laksov's Nullstellensatz}\label{appendix:b}

\subsection{Laksov's radical ideal and Nullstellens\"atze}
In this paper, the main purpose is to study the geometry of $K$-algebraic subsets of $L^n$, where $L|K$ is an extension of fields. One main tool concerns the use of Nullstellens\"atze like Hilbert's Nullstellensatz and Real Nullstellensatz. Laksov analyzes in \cite{la} the Nullstellens\"atze in terms of radical operations that generalize the classical radical ideal, see also \cite{lw}.

Hilbert's Nullstellensatz is a crucial tool to compute zero ideals, when we are dealing with algebraically closed fields. This involves the use of the classical {\em radical ideal} $\sqrt{\gta}$ of a given ideal $\gta$ of $K[\x]$. Namely, if $L$ is algebraically closed and $\gta$ is an ideal of $K[\x]$, then $\I_K(\ZZ_L(\gta))=\sqrt{\gta}$, where
$$
\sqrt{\gta}:=\big\{f\in K[\x]:\exists \ell\geq1,\ f^\ell\in\gta\big\}.
$$
We presented this version of Hilbert's Nullstellensatz in Corollary \ref{kreliablec}. Roughly speaking, the operator $\sqrt{\cdot}$ `only erases' multiplicities.

When dealing with real closed fields, it appears a more technical ideal $\sqrt[r]{\gta}$, called the {\em real radical ideal} of $\gta$, firstly introduced and developed by Dubois-Risler-Stengle \cite{d,r,st}. Its properties are `worse' than the ones of the classical radical ideal, but provides a good description of zero ideals of zero sets of ideals of $K[\x]$, where $K$ is an ordered subfield of a real closed field~$L$. Namely, if $\gta$ is an ideal of $K[\x]$, we have $\I_K(Z_L(\gta))=
\sqrt[r]{\gta}$, where
\begin{align*}
\textstyle\sqrt[r]{\gta}:=\big\{f\in K[\x]:&\ \exists m\geq0,\ \exists a_1,\ldots,a_m\in K \text{ with }a_1>0,\ldots,a_m>0,\\
&\ \textstyle\exists f_1,\ldots,f_m\in K[\x],\ \exists \ell\geq1,f^{2\ell}+\sum_{i=1}^ma_if_i^2\in\gta\big\}.
\end{align*}
Roughly speaking, the operator $\sqrt[r]{\cdot}$ `breaks sums of squares scaled by positive elements of $K$' into their addends and `erases' multiplicities.

When dealing with an arbitrary algebraic extension of fields $L|K$, the {\em Atkins-Gianni-Tognoli-Laksov radical ideal $\sqrt[L]{\gta}$} of $\gta$ appears \cite{agt,la,lw}.

Let us recall some concepts introduced in \cite{la}, using slightly different notations. Let $m\in\N$, let $K[\y_0,\y_1,\ldots,\y_m]_\sfh$ be the set of all homogeneous polynomials in $K[\y_0,\ldots,\y_m]$ and let $P(m)$ be the set of all polynomials in $K[\y_0,\y_1,\ldots,\y_m]_\sfh$ whose zero set in $L^{m+1}$ is contained in $\ZZ_L(\y_0)$, that is,
$$
P(m):=\big\{F\in K[\y_0,\y_1,\ldots,\y_m]_\sfh:\ZZ_L(F)\subset\ZZ_L(\y_0)\big\}.
$$
Observe that $F\in P(m)$ if and only if $F(1,\y_1,\ldots,\y_m)$ has empty zero set. Identify $K[\y_0,\ldots,\y_m]$ with a subset of $K[\y_0,\ldots,\y_m,\y_{m+1}]$ in the natural way. Let $K[(\y_i)_{i\in\N}]:=\bigcup_{m\in\N}K[\y_0,\ldots,\y_m]$ and let $P$ be the subset of $K[(\y_i)_{i\in\N}]$ defined by $P:=\bigcup_{m\in\N}P(m)$. 

Let $\gta$ be an ideal of $K[\x]$. Define the {\em $L$-radical ideal $\sqrt[L]{\gta}$ of $\gta$} as:
$$
\sqrt[L]{\gta}:=\big\{f\in K[\x]:\exists m\geq0,\ \exists F\in P(m),\ \exists f_1,\ldots,f_m\in K[\x],\ F(f,f_1,\ldots,f_m)\in\gta\big\}.
$$
By \cite[{Prop.2$\&$3}]{la}, $\sqrt[L]{\gta}$ is an ideal of $K[\x]$ that contains the radical ideal $\sqrt{\gta}$ and $\sqrt[L]{\sqrt[L]{\gta}}=\sqrt[L]{\gta}$, that is, the $L$-radical ideal of an ideal is an $L$-radical ideal.

The main two results of \cite{la} are the following:

\begin{thm}[Weak Hilbert's $L$-Nullstellensatz]
Let $L|K$ be an algebraic extension of fields and let $\gta$ be an ideal of $K[\x]$. Then $1\in\sqrt[L]{\gta}$ if and only if $\ZZ_L(\gta)=\varnothing$.
\end{thm}

\begin{thm}[Hilbert's $L$-Nullstellensatz]\label{thm:laksov}
Let $L|K$ be an algebraic extension of fields. Then $\II_K(\ZZ_L(\gta))=\sqrt[L]{\gta}$ for each ideal $\gta$ of $K[\x]$.
\end{thm}

\textit{Fix an algebraic extension of fields $L|K$.}

If $L$ is algebraically closed, then it is possible to describe the $L$-radical ideal $\sqrt[L]{\gta}$ of $\gta$ by the subfamily $P_H=\bigcup_{m\in\N}P_H(m)$ of $P$, where $P_H(m):=\{\y_0^\ell:\ell\geq1\}\subset P(m)$ for each $m\in\N$. Indeed, we have:
$$
\sqrt[L]{\gta}=\sqrt{\gta}=\big\{f\in K[\x]:\exists m\geq0,\ \exists F\in P_H(m),\ \exists f_1,\ldots,f_m\in K[\x],\ F(f,f_1,\ldots,f_m)\in\gta\big\}.
$$

Suppose $L$ is a real closed field and $K$ is an ordered subfield of $L$. In this case, as discussed in \cite[Ex.4.1(b)]{lw}, we can consider the subfamily $P_{DRS}=\bigcup_{m\in\N}P_{DRS}(m)$ of $P$, where
\begin{align*}
\textstyle P_{DRS}(m):=\big\{F\in K[\y_0,\ldots,\y_m]:&\ \exists q\geq0,\ \exists a_1,\ldots,a_q\in K \text{ with }a_1>0,\ldots,a_q>0,\\
&\ \text{$\exists G_1,\ldots,G_q\in K[\y_0,\ldots,\y_m]_\sfh$ of degree $\ell$ with}\\
&\ G_1(1,0,\ldots,0)=0,\ldots,G_q(1,0,\ldots,0)=0,\\
&\textstyle\ F=\y_0^{2d}+\sum_{i=1}^qa_iG_i(\y_0,\ldots,\y_m)^2\big\}.
\end{align*}
As the polynomial $\widehat{F}:=\y_0^{2d}+\sum_{i=1}^ma_i\y_i^2\y_{m+1}^{2d-2}$ belongs to $P_{DRS}(m+1)$ and $\widehat{F}(\y_0,\ldots,\y_m,1)=\y_0^{2\ell}+\sum_{i=1}^ma_i\y_i^2$, it is possible to describe $\sqrt[L]{\gta}$ by the subfamily $P_{DRS}$ of $P$:
$$
\sqrt[L]{\gta}=\sqrt[r]{\gta}=\big\{f\in K[\x]:\exists m\geq0,\ \exists F\in P_{DRS}(m),\ \exists f_1,\ldots,f_m\in K[\x],\ F(f,f_1,\ldots,f_m)\in\gta\big\}.
$$

In view of the above alternative (sub)families, one can wonder whether in the definition of $\sqrt[L]{\gta}$ it is possible to replace the family $P$ by its subfamily $P_0=\bigcup_{m\in\N}P_0(m)$, where
\begin{equation}\label{equab1}
P_0(m):=\big\{F\in K[\y_0,\y_1,\ldots,\y_m]_\sfh:\ZZ_L(F)=\{0\}\big\}.
\end{equation}
In \cite[\S5, Ex.2]{la}, Laksov shows that in general $P_0$ is not enough to define $\sqrt[L]{\gta}$. Namely, if $L=K:=\Q$ and $\gta:=(\x_1^3+\x_2^3+3\x_3^3)\Q[\x_1,\x_2,\x_3]$, then $\ZZ_\Q(\gta)=\{(q,-q,0)\in\Q^3:q\in\Q\}$, $\sqrt[L]{\gta}=(\x_1+\x_2,\x_3)\Q[\x_1,\x_2,\x_3]$ and the ideal $I_\gta$ of $K[\x]$ obtained by replacing $P(m)$ with $P_0(m)$ in the definition of $\sqrt[L]{\gta}$, that is,
$$
I_\gta:=\big\{f\in K[\x]:\exists m\geq0,\ \exists F\in P_0(m),\ \exists f_1,\ldots,f_m\in K[\x],\ F(f,f_1,\ldots,f_m)\in\gta\big\}
$$
is strictly contained in $\sqrt[L]{\gta}$, because it turns out that $I_\gta$ does not contain non-zero linear forms.

\subsection{Radical operations and Nullstellens\"atze.} In \cite{lw}, the authors study how the subfamilies of $P$ produce different types of radical operations. Let $\gtA$ be the set of all ideals of $K[\x]$. A map ${\mathfrak R}:\gtA\to\gtA$ is a {\em radical operation} if, for each $\gta\in\gtA$, it holds:\begin{itemize}
\item[$(\mr{i})$] ${\mathfrak R}(\gta)$ is the intersection of all prime ideals $\gtp$ of $K[\x]$ that contains $\gta$ and satisfy ${\mathfrak R}(\gtp)=\gtp$.
\item[$(\mr{ii})$] ${\mathfrak R}({\mathfrak R}(\gta))={\mathfrak R}(\gta)$.
\item[$(\mr{iii})$] $\sqrt{\gta}\subset{\mathfrak R}(\gta)$.
\end{itemize}
An example of a radical operation is the `$K$-ideal of $L$-zeros' operation, that is, 
$$
{\mathfrak R}_\ZZ^{L|K}(\gta):=\II_K(\ZZ_L(\gta))\ \text{ for each $\gta\in\gtA$.}
$$
Evidently, ${\mathfrak R}_\ZZ^{L|K}$ satisfies conditions $(\mr{ii})$ and $(\mr{iii})$. The fact that ${\mathfrak R}_\ZZ^{L|K}$ satisfies also condition $(\mr{ii})$ can be proved using Lemmas \ref{lem:prime} and \ref{lem:irred}.

Let $Q=\bigcup_{m\in\N}Q(m)$ be a subfamily of $P=\bigcup_{m\in\N}P(m)$ with $Q(m)\subset P(m)$ for each $m\in\N$. Define the map ${\mathfrak R}_Q:\gtA\to\gtA$ associated to $Q$ by
$$
{\mathfrak R}_Q(\gta):=\big\{f\in K[\x]:\exists m\geq0,\ \exists F\in Q(m),\ \exists f_1,\ldots,f_m\in K[\x],\ F(f,f_1,\ldots,f_m)\in\gta\big\}.
$$ 
A {\em Nullstellensatz for the algebraic extension of fields $L|K$} consists on finding such a subfamily $Q$ of $P$ such that ${\mathfrak R}_\ZZ^{L|K}={\mathfrak R}_Q$.

\subsection{Laksov's Criterion.} In \cite{lw}, the authors provide sufficient conditions for a subfamily $G$ of $P$ to define a radical operation. Namely,

\begin{crit}\label{lw0}
Let $Q=\bigcup_{m\in\N}Q(m)$ be a subfamily of $P=\bigcup_{m\in\N}P(m)$ with $Q(m)\subset P(m)$ for each $m\in\N$ such that:
\begin{itemize}
\item[$(\mr{C}1)$] Each $Q(m)$ contains polynomials of degree two. 
\item[$(\mr{C}2)$] If $m,m'\in\N$, $F\in G(m)$, $G\in G(m')$ and $S\in K[\y_0,\ldots,\y_{m+1}]_\sfh$ with $\deg(S)=\deg(F)-1$, we have $G(F,\y_{m+2}S,\ldots,\y_{m+m'+1}S)\in Q(m+m'+1)$.
\end{itemize}
Then ${\mathfrak R}_Q$ is a radical operation.
\end{crit}

The whole family $P$, that provides Atkins-Gianni-Tognoli-Laksov radical ideal and Hilbert's $L$-Nullstellensatz (by Laksov), satisfies both conditions of Criterion \ref{lw0}. The same happens with the subfamilies $P_H$ and $P_{DRS}$ of $P$, that provide the classical radical ideal involved in Hilbert's Nullstellensatz and the real radical ideal involved in Real Nullstellensatz, respectively.

If one considers the subfamily $P_0$ of $P$ introduced in \eqref{equab1}, then ${\mathfrak R}_{P_0}\neq{\mathfrak R}_\ZZ^{L|K}$ when $L$ is not algebraically closed \cite[Ex.7.2]{lw}. Consider the maps $\sqrt{{\mathfrak R}_{P_0}}:\gtA\to\gtA$ and ${\mathfrak R}_{P_0}^2:\gtA\to\gtA$ defined by $\sqrt{{\mathfrak R}_{P_0}}(\gta):=\sqrt{{\mathfrak R}_{P_0}(\gta)}$ and ${\mathfrak R}^2_{P_0}:={\mathfrak R}_{P_0}\circ{\mathfrak R}_{P_0}$. It is not know under which conditions on the extension of fields $L|K$ the maps $\sqrt{{\mathfrak R}_{P_0}}$ and ${\mathfrak R}_{P_0}\circ{\mathfrak R}_{P_0}$ define radical operations.

For each $s\geq 3$, we define inductively the map ${\mathfrak R}_{P_0}^s:\gtA\to\gtA$ by ${\mathfrak R}_{P_0}^s:={\mathfrak R}_{P_0}\circ{\mathfrak R}_{P_0}^{s-1}$. Denote ${\mathfrak R}_{P_0}^\infty:\gtA\to\gtA$ the map ${\mathfrak R}_{P_0}^\infty(\gta):=\bigcup_{s\geq1}{\mathfrak R}_{P_0}^s(\gta)$. One can define a subfamily $P^\infty_0=\bigcup_{m\in\N}P^\infty_0(m)$ of $P$ such that $P^\infty_0$ satisfies the two conditions of Criterion \ref{lw0} and ${\mathfrak R}_{P_0}^\infty={\mathfrak R}_{P^\infty_0}$. It follows that ${\mathfrak R}_{P_0}^\infty$ is a radical operation. 

In addition, the authors of \cite{lw} introduce the subfamily $P_0^*:=\{P_0^*(m)\}_{m\in\N}$ of $P$, defining $P_0^*(m)$ as the set of all polynomials in $K[\y_0,\y_1,\ldots,\y_m]$ that vanish at the origin of $L^{m+1}$ and are quasi-homogeneous in some blocks of variables $\{\y_0,\ldots,\y_{m_1}\},\ldots,\{\y_{m_{\ell-1}+1},\ldots,\y_{m_\ell}\}$ of some weights $w_1,\ldots,w_\ell$ and $m_1<\ldots<m_\ell:=m$. This subfamily of $P$ contains $P^\infty_0$, it satisfies the two conditions of Criterion \ref{lw0}, so it defines a radical operation. We have:
$$
\sqrt{{\mathfrak R}_{P_0}}(\gta)\subset{\mathfrak R}_{P_0}^2(\gta)\subset{\mathfrak R}^\infty_{P_0}(\gta)\subset{\mathfrak R}_{P^*_0}(\gta)\subset{\mathfrak R}_\ZZ^{L|K}(\gta)
$$
for each ideal $\gta$ of $K[\x]$. The main problem here concerning the Nullstellensatz is to determine the algebraic extension of fields $L|K$ such that one of the following equalities hold:
$$
\sqrt{{\mathfrak R}_{P_0}}={\mathfrak R}_\ZZ^{L|K},\qquad
{\mathfrak R}_{P_0}^2={\mathfrak R}_\ZZ^{L|K},\qquad
{\mathfrak R}^\infty_{P_0}={\mathfrak R}_\ZZ^{L|K},\qquad
{\mathfrak R}_{P^*_0}={\mathfrak R}_\ZZ^{L|K}.
$$

For recent developments concerning Laksov's Nullstellensatz, see \cite{lu} and the references mentioned therein.

%%%
\section{Proofs of results \ref{thm:K-elimination}, \ref{451}, \ref{localregular}, \ref{thmproj}, \ref{corproj} and \ref{618}}\label{appendix:c}

\subsection{Proof of Theorem \ref{thm:K-elimination}} For the reader's convenience, we rewrite the statement of this result here: {\it Let $L$ be an algebraically closed field and let $\rho:\PP^n(L)\times\PP^m(L)\to\PP^m(L)$ be the canonical projection onto the second factor. We have:
\begin{itemize}
\item[$(\mr{i})$] $\rho$ is a $K$-Zariski closed map, that is, if $S$ is a $K$-algebraic subset of $\PP^n(L)\times\PP^m(L)$, then $\rho(S)$ is a $K$-algebraic subset of $\PP^m(L)$. The same is true for the canonical projection $\PP^n(L)\times L^m\to L^m$.
\item[$(\mr{ii})$] $\rho$ maps $K$-construcible sets in $K$-construcible sets, that is, if $S$ is a $K$-constructible subset of $\PP^n(L)\times\PP^m(L)$, then $\rho(S)$ is a $K$-constructible subset of $\PP^m(L)$. The same is true for the canonical projection $\PP^n(L)\times L^m\to L^m$.
\end{itemize}}
\begin{proof}
$(\mr{i})$ Let us adapt the proof of Theorem (2.23) of \cite[pp.33-34]{mumford} to our setting. As the result is local with respect to the $K$-Zariski topology, the statement is equivalent to show that, if $S\subset\PP^n(L)\times L^m$ is $K$-Zariski closed, then $\rho(S)\subset L^m$ is $K$-Zariski closed as well. Let $f_1(\x_0,\x,\y),\ldots,f_\ell(\x_0,\x,\y)\in K[\y][\x_0,\x]_\sfh$ be such that $S$ is the common zero set of $f_1,\ldots,f_\ell$. Denote $d_i$ the degree of $f_i$ with respect to the variables $(\x_0,\x)$. By the projective version of Nullstellensatz, we have that $q\in L^m\setminus\rho(S)$ if and only if the radical of $(f_1(\x_0,\x,q),\ldots,f_\ell(\x_0,\x,q))L[\x_0,\x]$ in $L[\x_0,\x]$ contains $(\x_0,\ldots,\x_n)L[\x_0,\x]$. This is equivalent to the existence of $d\in\N^*$ such that
\begin{equation}\label{d*}
((\x_0,\ldots,\x_n)L[\x_0,\x])^d\subset(f_1(\x_0,\x,q),\ldots,f_\ell(\x_0,\x,q))L[\x_0,\x].
\end{equation}
If we set $\Omega_d:=\{q\in L^m:\text{\eqref{d*} is true}\}$ for each $d\in\N^*$, it holds $L^m\setminus\rho(S)=\bigcup_{d\in\N^*}\Omega_d$. It is enough to show that each $\Omega_d\subset L^m$ is $K$-Zariski open. Fix $d\in\N^*$. For each $t\in\Z$, define $V_t$ as the null $L$-vector space if $t<0$ and the $L$-vector subspace of all homogeneous polynomials in $L[\x_0,\x]_\sfh$ of degree $t$ if $t\geq0$. We endow $V_t$ with the standard ($L$-vector) basis constituted by the monomials in the indeterminates $(\x_0,\x)$ of degree $t$. For each $q\in L^m$, consider the $L$-linear map $T_d(q):\bigoplus_{i=1}^sV_{d-d_i}\to V_d$ defined by $T_d(q)(g_1,\ldots,g_s):=\sum_{i=1}^sg_i(\x_0,\x)f_i(\x_0,\x,q)$, define $m_d:=\dim_L(\bigoplus_{i=1}^sV_{d-d_i})$ and $n_d:=\dim_L(V_d)$, and denote $T^{(d)}(q)$ the $(n_d\times m_d)$-matrix of $T_d(q)$ with respect to the fixed canonical bases. Thus, a point $q\in L^m$ belongs to $\Omega_d$ if and only if $T_d(q)$ is surjective or, equivalently, if $\Omega_d=\{q\in L^m:{\rm rk}(T^{(d)}(q))=n_d\}$. Up to this point, the proof is almost identical to that of Theorem (2.23) of \cite[p.34]{mumford} mentioned above. The point now is that, because each polynomial $f_i(\x_0,\x,\y)$ has coefficients in $K$, the entries of $T^{(d)}(q)$ are polynomials in $K[q]$, so $\Omega_d\subset L^m$ is $K$-Zariski open.

$(\mr{ii})$ Let us follow the proof of Proposition (2.31) of \cite[pp.37-38]{mumford}. In addition to appropriately adapting the mentioned proof, we use crucially the fact that the subfield $K$ of $L$ is infinite. Again, the result is local with respect to the $K$-Zariski topology, so the statement is equivalent to prove that, if $S\subset L^n\times L^m=L^{n+m}$ is $K$-constructible, then $\rho(S)\subset L^m$ is $K$-constructible as well. By induction on $n\in\N^*$, we can assume $n=1$.

Suppose for a moment that the following assertion is true:
\begin{itemize}
\item[$(\ast)$] If $P\subset L^{1+m}$ is a non-empty $K$-irreducible $K$-Zariski locally closed set, then there exists a non-empty $K$-Zariski open subset $\Omega$ of $Q:=\zcl^K_{L^m}(\rho(P))$ such that $\Omega\subset\rho(P)$.
\end{itemize}

Fix a $K$-constructible set $S\subset L^{1+m}$. If $S=\varnothing$, then $\rho(S)=\varnothing$, which is $K$-constructible. Suppose $S\neq\varnothing$. Let us prove by induction on $e:=\dim_K(\rho(S))$ that $\rho(S)\subset L^m$ is $K$-constructible. Write $S=\bigcup_{i=1}^a P_i$, where $\{P_i\}_{i=1}^a$ is a finite family of non-empty $K$-irreducible $K$-Zariski locally closed subsets of $L^{1+m}$. By $(\ast)$, there exists a non-empty $K$-Zariski open subset $\Omega_i$ of $Q_i:=\zcl^K_{L^m}(\rho(P_i))$ such that $\Omega_i\subset\rho(P_i)\subset Q_i$. Observe that $\rho(S)=\bigcup_{i=1}^a\rho(P_i)$ and $e=\max_{i\in\{1,\ldots,a\}}\{\dim_K(Q_i)\}$. Moreover, $Q_i\subset L^m$ is $K$-irreducible, because so is $P_i$ in $L^{1+m}$. If $e=0$, then each $Q_i\subset L^m$ has zero $K$-dimension so, by Lemma \ref{dimirred}, we have $\Omega_i=\rho(P_i)=Q_i$ and $\rho(S)\subset L^m$ coincides with the $K$-algebraic set $\bigcup_{i=1}^aQ_i$, which of course is $K$-constructible. Suppose $e>0$. Rearranging the indices if necessary, we can assume that there exists $b\in\{1,\ldots,a\}$ such that $\dim_K(Q_i)=e$ for $i\leq b$ and $\dim_K(Q_i)<e$ for $i>b$ (where the last condition is omitted when $b=a$). By induction hypothesis, $\rho(P_i)\subset L^m$ is $K$-constructible for $i>b$. Let $i\in\{1,\ldots,b\}$. Using $(\ast)$ again, we can find a non-empty $K$-Zariski open subset $\Omega_i$ of $Q_i\subset L^m$ such that $\Omega_i\subset\rho(P_i)\subset Q_i$. 
Observe that $Q_i\setminus\Omega_i\subset L^m$ is $K$-algebraic and $\dim_K(Q_i\setminus\Omega_i)<\dim_K(Q_i)=e$ by Lemma \ref{dimirred}. Define the $K$-constructible set $P'_i:=P_i\setminus\rho^{-1}(\Omega_i)\subset L^{1+m}$. As $\rho(P'_i)\subset\rho(P_i)\setminus\Omega_i\subset Q_i\setminus\Omega_i$, we have $\dim_K(\rho(P'_i))<e$, so $\rho(P'_i)\subset L^m$ is $K$-constructible by induction. As $\Omega_i\subset\rho(P_i)$, we have $\rho(P_i)=\Omega_i\sqcup\rho(P'_i)$ so $\rho(P_i)\subset L^m$ is $K$-constructible as well. This proves that $\rho(S)\subset L^m$ is $K$-constructible, as required.
 
We prove next $(\ast)$. Denote $\widehat{P}$ the $K$-Zariski closure of $P$ in $\PP^1(L)\times L^m$. By $(\mr{i})$, we have $\rho(\widehat{P})=Q$. We distinguish two cases: $\widehat{P}=\PP^1(L)\times Q$ or $\widehat{P}\neq\PP^1(L)\times Q$ (and in this second case $\dim_K(\widehat{P})=\dim_L(\widehat{P})=\dim_L(Q)=\dim_K(Q)$).

\textsc{Case 1:} First, suppose $\widehat{P}=\PP^1(L)\times Q$. Pick $t_0\in L\subset\PP^1(L)$ and $x_0\in L^m$ such that $(t_0,x_0)\in P$. Denote $P_0$ the subset of $L$ such that $P\cap(L\times\{x_0\})=P_0\times\{x_0\}$. Observe that $P$ is also $K$-Zariski locally closed in $\PP^1(L)\times L^m$, so $P$ is $K$-Zariski open in $\widehat{P}=\PP^1(L)\times Q$. It follows that $P_0$ is $L$-Zariski open in $L$, it contains $t_0$, so $L\setminus P_0$ is a finite set, because $L\setminus P_0$ is a proper $L$-Zariski closed subset of $L$. We would like that $t_0$ belongs to $K$ (instead to merely~$L$), so let us change $t_0$ by an element $t^*\in K\subset L$. The intersection $K\cap P_0\neq\varnothing$, because $K$ is infinite and $L\setminus P_0$ is finite. Choose $t^*\in K\cap P_0$ and denote $Q^*$ the subset of $L^m$ such that $P\cap(\{t^*\}\times L^m)=\{t^*\}\times Q^*$. Observe that $Q^*$ is non-empty, because it contains $x_0$. As $t^*\in K$ and $P$ is a $K$-Zariski open subset of $\widehat{P}=\PP^1(L)\times Q$, we deduce that $Q^*$ is a non-empty $K$-Zariski open subset of $Q$ contained in $\rho(P)$.

\textsc{Case 2:} Next, suppose that $\dim_K(P)=\dim_K(\widehat{P})=\dim_K(Q)=:D$. Consider the $K$-Zariski closed set $\partial P:=\widehat{P}\setminus P\subset \PP^1(L)\times L^m$. Observe that $Q=\rho(\widehat{P})=\rho(P)\cup\rho(\partial P)$ so $Q\setminus\rho(\partial P)\subset\rho(P)$. As $\partial P$ is a proper $K$-Zariski closed subset of the $K$-irreducible $K$-Zariski closed set $\widehat{P}\subset\PP^1(L)\times L^m$, we deduce applying Lemma \ref{dimirred} via the two homeomorphisms $\{\theta^1_i\times\mr{id}_{L^m}\}_{i=0}^1$ that $\dim_K(\partial P)<D$. By $(\mr{i})$, $\rho(\partial P)$ is a $K$-algebraic subset of $Q\subset L^m$. As $\dim_K(\rho(\partial P))=\dim_L(\rho(\partial P))\leq\dim_L(\partial P)=\dim_K(\partial P)<D$, we deduce that $Q\setminus\rho(\partial P)$ is a non-empty $K$-Zariski open subset of $Q\subset L^m$ contained in $\rho(P)$, as required.
\end{proof}

\begin{remark}\label{rem:pitfall1}
The proof of item $(\mr{ii})$ of the previous theorem contains an additional argument, not needed in the classical proof of Proposition (2.31) of \cite[pp.37-38]{mumford}, that we followed finding the point $t^*\in K\cap P_0$ in {\sc Case 1}. The existence of $t^*$ follows from the fact that the (infinite) subfield $K$ of~$L$ is dense in $L$ with respect to the usual $L$-Zariski topology (see also Lemma \ref{K-dense-L}). $\sqbullet$
\end{remark}

\begin{remark}\label{product-project-spaces}
We can easily extend Theorem \ref{thm:K-elimination} to the case of all finite products of proje\-ctive spaces over $L$. Let $n_1,\ldots,n_\ell\in\N^*$ with $\ell\geq2$ and let $[\x^i]=[\x^i_0,\ldots,\x^i_{n_i}]$ be the homogeneous coordinates of $\PP^{n_i}(L)$. Let $S$ be a subset of $\PP^{n_1}(L)\times\ldots\times\PP^{n_\ell}(L)$. We say that $S$ is $K$-algebraic if it is the zero set of a family of polynomials in $K[\x^1,\ldots,\x^\ell]=K[\x^1_0,\ldots,\x^\ell_{n_\ell}]$ that are homogeneous separately in the variables $\x^i=(\x^i_0,\ldots,\x^i_{n_i})$ for all $i\in\{1,\ldots,\ell\}$. We also say that $S$ is $K$-constructible if it is a Boolean combination of $K$-algebraic subsets of $\PP^{n_1}(L)\times\ldots\times\PP^{n_\ell}(L)$. The $K$-algebraic subsets of $\PP^{n_1}(L)\times\ldots\times\PP^{n_\ell}(L)$ are the closed sets of a topology we call $K$-Zariski topology of $\PP^{n_1}(L)\times\ldots\times\PP^{n_\ell}(L)$. 

Recall that, for each $n\in\N^*$ and for each $j\in\{0,\ldots,n\}$, $\theta^n_j:U^n_j\to L^n$ is the affine chart $\theta^n_i([x_0,x]):=\big(\frac{x_0}{x_i},\ldots,\frac{x_{i-1}}{x_i},\frac{x_{i+1}}{x_i},\ldots,\frac{x_n}{x_i}\big)$, where $U^n_i:=\{[x_0,x]\in\PP^n(L):x_i\neq0\}$.

If $S\subset\PP^{n_1}(L)\times\ldots\times\PP^{n_\ell}(L)$ is $K$-constructible, then we define the $K$-dimension $\dim_K(S)$ of $S$ as the maximum over all the multi-indices $(j_1,\ldots,j_\ell)$ with $j_i\in\{0,\ldots,n_i\}$ of the $K$-dimension of the $K$-constructible subset $(\theta^{n_1}_{j_1}\times\ldots\times\theta^{n_\ell}_{j_\ell})(S\cap(U^{n_1}_{j_1}\times\ldots\times U^{n_\ell}_{j_\ell}))$ of $L^{n_1+\ldots+n_\ell}$. As in the case $\ell=2$, the nature of $\dim_K(S)$ is $K$-Zariski local and $\dim_K(S)$ coincides with the dimension $\dim_L(S)$ of $S$ viewed as a usual ($L$-)constructible subset of $\PP^{n_1}(L)\times\ldots\times\PP^{n_\ell}(L)$.

Equip each product of the form $\PP^{n_1}(L)\times\ldots\times\PP^{n_\ell}(L)$ with the $K$-Zariski topology and denote $\rho:\PP^{n_1}(L)\times\ldots\times\PP^{n_\ell}(L)\to\PP^{n_2}(L)\times\ldots\times\PP^{n_\ell}(L)$ the projection $(\alpha_1,\ldots,\alpha_\ell)\mapsto(\alpha_2,\ldots,\alpha_\ell)$. The preceding proof of Theorem \ref{thm:K-elimination} with suitable typographic changes proves that $\rho$ is $K$-Zariski closed and maps $K$-constructible sets onto $K$-constructible sets. By induction, the same is true for all the projections $\PP^{n_1}(L)\times\ldots\times\PP^{n_\ell}(L)\to\PP^{n_h}(L)\times\ldots\times\PP^{n_\ell}(L)$, $(\alpha_1,\ldots,\alpha_\ell)\mapsto(\alpha_h,\ldots,\alpha_\ell)$ with $h\in\{2,\ldots,\ell\}$. $\sqbullet$
\end{remark}

\subsection{Proof of Lemma \ref{451}}
The statement reads as follows: {\it Let $X\subset L^n$ be a $K$-algebraic set, let $a\in X\cap(\ove^\sqbullet)^n$ and let $\gtM_a$ be the maximal ideal of $\reg^{E|K}_{X,a}$. Denote $(T^{E|K}_a(X))^\wedge$ the dual of the $E[a]$-vector space $T^{E|K}_a(X)$. Then the map $\Psi_a:\gtM_a/\gtM_a^2\to(T^{E|K}_a(X))^\wedge$, defined by
$$
\textstyle
\Psi_a\big((\frac{p}{q}+\II_K(X)E[\x]_{\gtn_a})+\gtM_a^2\big):=\big(T^{E|K}_a(X)\ni v\mapsto q(a)^{-1}\langle\nabla p(a),v\rangle\in E[a]\big),
$$
is a well-defined $E[a]$-linear isomorphism.}

\vspace{3mm}
To prove this result, we adapt the standard argument to the present situation. To do so, we have to combine the new Corollary \ref{regular} with the classic \cite[Cor.2, p.302]{zs2}.

\begin{proof}
First, consider the map $\xi':\gtN_a\to\gtM_a/\gtM_a^2$ defined by $\xi'(\frac{p}{q}):=(\frac{p}{q}+\II_K(X)E[\x]_{\gtn_a})+\gtM_a^2$. Observe that $\xi'$ is a surjective homomorphism between abelian groups. The kernel of $\xi'$ is equal to $\II_K(X)E[\x]_{\gtn_a}+\gtN_a^2$. Thus, $\xi'$ induces the abelian group isomorphism $\xi:\gtN_a/(\II_K(X)E[\x]_{\gtn_a}+\gtN_a^2)\to\gtM_a/\gtM_a^2$ defined by
$$
\textstyle
\xi(\frac{p}{q}+(\II_K(X)E[\x]_{\gtn_a}+\gtN_a^2)):=(\frac{p}{q}+\II_K(X)E[\x]_{\gtn_a})+\gtM_a^2.
$$ 
We equip the abelian group $\gtN_a/(\II_K(X)E[\x]_{\gtn_a}+\gtN_a^2)$ with a structure of $E[a]$-vector space by mean of a scalar multiplication similar to the one given in \eqref{eq:star}, which we again denote~$\ast$:
$$\textstyle
b\ast\big(\frac{p}{q}+(\II_K(X)E[\x]_{\gtn_a}+\gtN_a^2)\big)\mapsto \frac{fp}{q}+(\II_K(X)E[\x]_{\gtn_a}+\gtN_a^2)
$$
if $b\in E[a]$ and $f$ is a polynomial in $E[\x]$ such that $b=f(a)$. One verifies that $\xi$ is an $E[a]$-linear isomorphism. 

Consider now the $E[a]$-linear map $\theta:\gtN_a/\gtN_a^2\to(E[a]^n)^\wedge$ defined by
$$
\textstyle
\theta(\frac{p}{q}+\gtN_a^2)=\big(E[a]^n\ni v\mapsto q(a)^{-1}\langle\nabla p(a),v\rangle\in E[a]\big).
$$
By Corollary \ref{regular}, the polynomials $\{f_1,\ldots,f_n\}\subset E[\x]$ defined in the statement of Lemma \ref{lem:gtn_a} form a regular system of parameters of $E[\x]_{\gtn_a}$. Thus, \cite[Cor.2, p.302]{zs2} assures that $\{f_1+\gtN_a^2,\ldots,f_n+\gtN_a^2\}$ is an $E[a]$-vector basis of $\gtN_a/\gtN_a^2$, and Remark \ref{F} implies that $\theta$ is an $E[a]$-linear isomorphism. In particular, we have:
\begin{equation}\label{order2}
\text{An element $\textstyle\frac{p}{q}\in\gtN_a$ belongs to $\gtN_a^2$ if and only if $\nabla p(a)=0$.} 
\end{equation}

Let $i:T^{E|K}_a(X)\hookrightarrow E[a]^n$ be the inclusion map, let $i^\wedge:(E[a]^n)^\wedge\to(T^{E|K}_a(X))^\wedge$ be its dual map and let $\Theta':\gtN_a/\gtN_a^2\to(T^{E|K}_a(X))^\wedge$ be the composition $\Theta':=i^\wedge\circ\theta$. Observe that
$$
\textstyle
\ker(\Theta')=\{\frac{p}{q}+\gtN_a^2\in\gtN_a/\gtN_a^2:\langle\nabla p(a),v\rangle=0,\ \forall v\in T^{E|K}_a(X)\}.
$$
Let us show the following equality:
\begin{equation}\label{eq:kerTheta'}
\ker(\Theta')=(\II_K(X)E[\x]_{\gtn_a}+\gtN_a^2)/\gtN_a^2.
\end{equation}
It is evident that $(\II_K(X)E[\x]_{\gtn_a}+\gtN_a^2)/\gtN_a^2\subset\ker(\Theta')$. Let us prove the converse inclusion. Let $\{g_1,\ldots,g_s\}$ be a system of generators of $\II_K(X)$ in $K[\x]$, let $\ell:={\rm rk}\big(\frac{\partial g_i}{\partial\x_j}(a)\big)_{i=1,\ldots,s,\,j=1,\ldots,n}$, let $W_1$ be the $E[a]$-vector subspace of $(E[a])^n$ generated by $\nabla g_1(a),\ldots,\nabla g_s(a)$ and let $W_2$ be the $E[a]$-vector subspace of $(E[a])^n$ defined by 
$$
\textstyle
W_2:=\{(w_1,\ldots,w_n)\in(E[a])^n:\sum_{i=1}^nw_iv_i=0, \forall(v_1,\ldots,v_n)\in T^{E|K}_a(X)\}.
$$
It holds: $\dim_{E[a]}(W_1)=\ell$, $\dim_{E[a]}(T^{E|K}_a(X))=n-\ell$, $\dim_{E[a]}(W_2)=n-(n-\ell)=\ell$ and $W_1\subset W_2$, so $W_1=W_2$. Consider $\frac{p}{q}+\gtN_a^2\in \ker(\Theta')$, that is, $\nabla p(a)\in W_2$. As $W_2=W_1$, there exist $\lambda_1,\ldots,\lambda_s\in E[a]$ such that $\nabla p(a)=\sum_{i=1}^s\lambda_i\nabla g_i(a)$. Choose $f_i\in E[\x]$ such that $f_i(a)=\lambda_i$ for each $i\in\{1,\ldots,s\}$, and define $\frac{f}{q}\in \II_K(X)E[\x]_{\gtn_a}$ by setting $f:=\sum_{i=1}^sf_ig_i$. As $\nabla(p-f)(a)=0$, \eqref{order2} assures that $\frac{p}{q}-\frac{f}{q}=\frac{p-f}{q}\in\gtN_a^2$ so $\frac{p}{q}\in\II_K(X)E[\x]_{\gtn_a}+\gtN_a^2$. This proves that $\ker(\Theta')\subset(\II_K(X)E[\x]_{\gtn_a}+\gtN_a^2)/\gtN_a^2$ and hence \eqref{eq:kerTheta'} holds. 

By \eqref{eq:kerTheta'}, $\Theta'$ induces the $E[a]$-linear isomorphism 
$$
\textstyle
\Theta:\gtN_a/(\II_K(X)E[\x]_{\gtn_a}+\gtN_a^2)\simeq(\gtN_a/\gtN_a^2)\big/((\II_K(X)E[\x]_{\gtn_a}+\gtN_a^2)/\gtN_a^2)\to(T^{E|K}_a(X))^\wedge
$$
defined by
$$
\textstyle
\Theta(\frac{p}{q}+(\II_K(X)E[\x]_{\gtn_a}+\gtN_a^2))=\big(E[a]^n\ni v\mapsto q(a)^{-1}\langle\nabla p(a),v\rangle\in E[a]\big).
$$

As $\Psi_a$ is the composition $\Theta\circ\xi^{-1}$ of two $E[a]$-linear isomorphisms, so is $\Psi_a$, as required. 
\end{proof}

\subsection{Proof of Lemma \ref{localregular}}
Let us recall the statement: {\it Let $a\in(\ove^\sqbullet)^n$, let $\gta$ be an ideal of $K[\x]$ such that $\gta\subset\gtn_a$, let $r:={\rm rk}_a(\gta)$, and let $A:=E[\x]_{\gtn_a}/(\gta E[\x]_{\gtn_a})$. The following four conditions are equivalent.
\begin{itemize}
\item[$(\mr{i})$] The local ring $A$ is regular.
\item[$(\mr{ii})$] There exist $r$ elements $g_1,\ldots,g_r$ of $\gta$ such that $\gta E[\x]_{\gtn_a}=(g_1,\ldots,g_r)E[\x]_{\gtn_a}$.
\item[$(\mr{iii})$] $\dim(A)=n-r$.
\item[$(\mr{\mr{iv}})$] $\mr{ht}(\gta E[\x]_{\gtn_a})=r$.
\end{itemize}

In addition, if $A$ is not regular, then the minimal cardinality of a system of generators of $\gta E[\x]_{\gtn_a}$ in $E[\x]_{\gtn_a}$ is $>r$, $\dim(A)<n-r$ and $\hgt(\gta E[\x]_{\gtn_a})>r$.
}

\vspace{3mm}
To prove this result, we need the following accurate version of a well-known result on regular local residue class rings \cite[Thm.26, p.303]{zs2}.

\begin{thm}\label{thm:ZS2-Thm26}
Let $B$ be a regular local ring of dimension $d$ with maximal ideal $\gtm$, let $\kappa:=B/\gtm$ be the residue field of $B$, let $\gtb$ be a proper ideal of $B$ and let $r$ be the dimension of $(\gtb+\gtm^2)/\gtm^2$ as a $\kappa$-vector subspace of $\gtm/\gtm^2$. Choose an arbitrary system of generators $\{g_1,\ldots,g_s\}$ of $\gtb$ in $B$. Rearranging the indices if necessary, we can assume $\{g_1+\gtm^2,\ldots,g_r+\gtm^2\}$ is a $\kappa$-vector basis of $(\gtb+\gtm^2)/\gtm^2$. Then the following assertions are equivalent.
\begin{itemize}
\item[$(\mr{i})$] The local ring $B/\gtb$ is regular.
\item[$(\mr{ii})$] $\gtb=(g_1,\ldots,g_r)B$ (that is, $\gtb$ is generated by $g_1,\ldots,g_r$ in $B$).
 \end{itemize}
In addition, if conditions $(\mr{i})$ and $(\mr{ii})$ are satisfied, then $\dim(B/\gtb)=d-r$.
\end{thm}
\begin{proof}
One can follow the proof of \cite[Thm.26, p.303]{zs2}. We have to pay attention to the following point. Assume that $B/\gtb$ is regular. Let $\gtb':=(g_1,\ldots,g_r)B$, let $\tilde{\gtm}=\gtm/\gtb$ be the maximal ideal of $B/\gtb$ and let $\varphi:\gtm/\gtm^2\to\tilde{\gtm}/\tilde{\gtm}^2=\gtm/(\gtb+\gtm^2)$ be the canonical surjective $\kappa=(B/\gtb)/\tilde{\gtm}$-linear map. As $\ker(\varphi)=(\gtb+\gtm^2)/\gtm^2$, the set $\{g_1+\gtm^2,\ldots,g_r+\gtm^2\}$ is a $\kappa$-vector basis of $\ker(\varphi)$. Thus, the last part of the mentioned proof of \cite[Thm.26, p.303]{zs2} implies that $\gtb=\gtb'$, as required.
\end{proof}

\begin{proof}[Proof of Lemma \em \ref{localregular}]
Choose a system of generators $\{g_1,\ldots,g_s\}$ of $\gta$ in $K[\x]$, where $s\geq r$. Observe that $\{g_1,\ldots,g_s\}$ is also a system of generators of $\gta E[\x]_{\gtn_a}$ in $E[\x]_{\gtn_a}$. By Lemma \ref{rk}, we can assume that $\{g_1+\gtN_a^2,\ldots,g_r+\gtN_a^2\}$ is a basis of the $E[a]$-vector space $(\gta E[\x]_{\gtn_a}+\gtN_a^2)/\gtN_a^2$. By Theorem \ref{thm:ZS2-Thm26}, if $A$ is regular, $\gta E[\x]_{\gtn_a}=(g_1,\ldots,g_r)E[\x]_{\gtn_a}$. In particular, as $\dim(E[\x]_{\gtn_a})=n$ (by Corollary \ref{regular}), we have $\dim(A)=\dim(E[\x]_{\gtn_a})-r=n-r$. Suppose now that $\gta E[\x]_{\gtn_a}=(g'_1,\ldots,g'_r)E[\x]_{\gtn_a}$ for some $g'_1,\ldots,g'_r\in\gta$. Using again Lemma \ref{rk}, we have that $\{g'_1+\gtN_a^2,\ldots,g'_r+\gtN_a^2\}$ is a $E[a]$-vector basis of $(\gta E[\x]_{\gtn_a}+\gtN_a^2)/\gtN_a^2$. By Theorem \ref{thm:ZS2-Thm26}, the local ring $A$ is regular. This proves equivalence $(\mr{i})\Longleftrightarrow(\mr{ii})$ and implication $(\mr{i})\Longrightarrow(\mr{iii})$.

Equivalence $(\mr{iii})\Longleftrightarrow(\mr{iv})$ follows immediately from \eqref{eq:ht}.

It remains to show implication $(\mr{iii})\Longrightarrow(\mr{i})$. Assume $\dim(A)=n-r$ and that $A$ is not regular. By Lemma \ref{rk}, we have $r\leq\min\{s,n\}$. Define $h_i(\x):=g_i(\x+a)\in E[a][\x]$ for each $i\in\{1,\ldots,s\}$. By Lemma \ref{completion}, the completion of $A$ is
$$
\widehat{A}\cong E[a][[\x]]/((h_1,\ldots,h_s)E[a][[\x]]).
$$
As a local ring is regular if and only if so is its completion (see \cite[(24.D), p.175]{m}), we have that $\widehat{A}$ is not regular. Moreover, $\dim(A)=\dim(\widehat{A})$. Rearranging the indices if necessary, we may assume
$$
\det\Big(\frac{\partial h_i}{\partial\x_j}(0)\Big)_{i,j=1,\ldots,r}=\det\Big(\frac{\partial g_i}{\partial\x_j}(a)\Big)_{i,j=1,\ldots,r}\neq0.
$$
After a substitution automorphism of $E[a][[\x]]$ that maps $\sum_{j=1}^n\frac{\partial h_i}{\partial\x_j}(0)\x_j$ to $\x_i$ for each $i\in\{1,\ldots,r\}$ and $\x_i$ to $\x_i$ for each $i\in\{r+1,\ldots,n\}$ (see \cite[Cor.2, p.137]{zs2}), we may also assume that $\big(\frac{\partial h_i}{\partial\x_j}(0)\big)_{i,j=1,\ldots,r}$ is the $(r\times r)$-identity matrix. Consider the ideal ${\mathfrak h}:=(h_1,\ldots,h_s)E[a][[\x]]$ of $E[a][[\x]]$ and let $\gtm$ be the maximal ideal of $E[a][[\x]]$. By Weierstrass' preparation theorem \cite[Cor.1, p.145]{zs2}, we may assume that $h_1$ (understood as generator of ${\mathfrak h}$) has the following form: $h_1=\x_1+h_1'$, where $h_1'\in\gtm^2\cap E[a][[\x_2,\ldots,\x_n]]$. Dividing each $h_2,\ldots,h_r$ by $h_1$ via Weierstrass' division theorem \cite[Thm.5, pp.139-140]{zs2}, we may assume that $h_2,\ldots,h_r$ (understood as generators of ${\mathfrak h}$) belong to $E[a][[\x_2,\ldots,\x_n]]$. Proceeding recursively using both theorems, we may also assume for each $j\in\{1,\ldots,r\}$ that $h_j=\x_j+h_j'$ for some $h_j'\in\gtm^2\cap E[a][[\x_{j+1},\ldots,\x_n]]$ (where $E[a][[\x_{j+1},\ldots,\x_n]]:=E[a]$ if $j=r=n$). If $r<s$, dividing each $h_{r+1},\ldots,h_s$ recursively by $h_1,\ldots,h_r$, we may assume $h_{r+1},\ldots,h_s\in E[a][[\x_{r+1},\ldots,\x_n]]$. After division, it may happen that some of the $h_j$ or even all are zero for each $j\in\{r+1,\ldots,s\}$. As the rank of the matrix $\big(\frac{\partial h_i}{\partial\x_j}(0)\big)_{i=1,\ldots,s,\,j=1,\ldots,n}$ is equal to $r$, we deduce $h_{r+1},\ldots,h_s\in\gtm^2$. Consequently, if we set
$$
\gtb:=
\begin{cases}
(h_{r+1},\ldots,h_s)E[a][[\x_{r+1},\ldots,\x_n]]&\text{if $r<s$},\\
(0)&\text{if $r=s$,}
\end{cases}
$$
then
\begin{equation}\label{eq:m^2}
\gtb\subset\gtm^2\cap E[a][[\x_{r+1},\ldots,\x_n]].
\end{equation}
Thus, using the substitution automorphism of $E[a][[\x]]$ that maps $h_i$ to $\x_i$ for each $i\in\{1,\ldots,r\}$ and $\x_i$ to $\x_i$ for each $i\in\{r+1,\ldots,n\}$, we obtain
\begin{equation}\label{eq:A}
\widehat{A}\cong E[a][[\x]]/((h_1,\ldots,h_s)E[a][[\x]])\cong E[a][[\x_{r+1},\ldots,\x_n]]/\gtb.
\end{equation}
As $\widehat{A}$ is not regular, by \eqref{eq:m^2}, \eqref{eq:A} and Theorem \ref{thm:ZS2-Thm26}, we deduce that $\gtb\neq(0)$. It follows that $r<s$, and also $r<n$ (otherwise, $r=n<s$ so $h_{r+1},\ldots,h_s\in\gtn_a\cap E[a]=\{0\}$ and $\gtb=(0)$). Observe that $E[a][[\x_{r+1},\ldots,\x_n]]/\gtb$ is a local ring whose maximal ideal $\gtd$ is generated by $\x_{r+1}+\gtb,\ldots,\x_n+\gtb$. Thus, the minimum number $\delta(\gtd)$ of elements we need to generate $\gtd$ is $\leq n-r$. As $\widehat{A}\cong E[a][[\x_{r+1},\ldots,\x_n]]/\gtb$ is not regular, we have $n-r=\dim(A)=\dim(\widehat{A})<\delta(\gtd)\leq n-r$ (see \cite[\S11]{am}), which is a contradiction.
\end{proof}

\subsection{Proofs of Theorems \ref{thmproj}, \ref{corproj} and Lemma \ref{618}} %

\emph{In this subsection, $L|K$ denotes an extension of fields such that $L$ is either algebraically closed or real closed, and $\kbar^\sqbullet$ is the algebraic closure of $K$ in~$L$. We will use Remark \emph{\ref{dime}} freely, and set $\dim(S):=\dim_L(S)$ for each algebraic set $S\subset L^n$.}

If $n\geq2$, we define $\tilde{\x}:=(\x_1,\ldots,\x_{n-1})$, $\tilde{x}:=(x_1,\ldots,x_{n-1})^T\in L^{n-1}$ and $\Pi:L^n\to L^{n-1}$ by $\Pi(x):=\tilde{x}$. Let us recall the statement of the first result we want to prove.

\vspace{2mm}
\noindent \textbf{Theorem \ref{thmproj}.} {\it Suppose that $L$ is algebraically closed and $n\geq2$. Consider a $K$-algebraic set $X\subset L^n$ and assume that $\II_K(X)$ contains a monic polynomial 
$$
\ell(\x):=\x_n^d+a_{n-1}(\tilde{\x})\x_n^{d-1}+\cdots+a_1(\tilde{\x})\x_n+a_0(\tilde{\x})\in K[\tilde{\x}][\x_n]=K[\x]
$$ 
with respect to $\x_n$. Set $Y:=\Pi(X)$ and denote $\pi:X\to Y$ the restriction of $\Pi$ from $X$ to $Y$. Then $Y\subset L^{n-1}$ is a $K$-algebraic set, $\dim(Y)=\dim(X)$ and $\pi$ is a $K$-Zariski closed map, that is, the image under $\pi$ of each $K$-Zariski closed subset of $X$ is a $K$-Zariski closed subset of $Y$. Moreover, if we pick a point $a\in X\cap(\kbar^\sqbullet)^n$ and we set $b:=\pi(a)\in Y\cap(\kbar^\sqbullet)^{n-1}$, we have the following properties:
\begin{itemize}
\item[$(\mr{i})$] If $\pi^{-1}(b)=\{a\}$, then the homomorphism $\pi^*:\reg^{K|K}_{Y,b}\to\reg^{K|K}_{X,a}$ is finite.
\item[$(\mr{ii})$] If $\pi^{-1}(b)=\{a\}$ and $d_a\pi:T_a(X)\to T_b(Y)$ is injective, then:
 \begin{itemize}
\item[$(\mr{a})$] $K[a]=K[b]$, the differential $d^{K|K}_a\pi:T^{K|K}_a(X)\to T^{K|K}_b(Y)$ is well-defined and it is injective;
\item[$(\mr{b})$] $\pi^*:\reg^{K|K}_{Y,b}\to\reg^{K|K}_{X,a}$ is a ring isomorphism;
\item[$(\mr{c})$] there exists a $K$-Zariski open neighborhood $U$ of $a$ in $X$ such that $\pi(U)$ is a $K$-Zariski open neighborhood of $b$ in $Y$ and the restriction of $\pi$ from $U$ to $\pi(U)$ is a $K$-biregular isomorphism;
\item[$(\mr{d})$] $d_a\pi$ is a $L$-linear isomorphism and $d_a^{K|K}\pi$ is a $K[a]$-linear isomorphism.
\end{itemize}
\end{itemize}
In addition, we have:
\begin{itemize} 
\item[$(\mr{iii})$] Let $V$ be a $K$-Zariski open subset of $X$ and let $\pi|_V:V\to\pi(V)$ be the corresponding restriction of $\pi$. If $\pi^{-1}(\pi(c))=\{c\}$ and the differential $d_c\pi:T_c(X)\to T_{\pi(c)}(Y)$ is injective for each $c\in V\cap(\kbar^\sqbullet)^n$, then $\pi(V)$ is a $K$-Zariski open subset of $Y$ and $\pi|_V$ is a $K$-biregular isomorphism.
\item[$(\mr{iv})$] The map $\pi:X\to Y$ is a $K$-biregular isomorphism if and only if $\pi$ is injective and its differential $d_c\pi:T_c(X)\to T_{\pi(c)}(Y)$ is injective for each $c\in X\cap(\kbar^\sqbullet)^n$.
\end{itemize}}

\begin{remark}
The condition `{\em $\,\pi^{-1}(\pi(c))=\{c\}$ for each $c\in V\cap(\kbar^\sqbullet)^n\,$}' is equivalent to the following one: `{\em$\,\pi^{-1}(\pi(V\cap(\kbar^\sqbullet)^n))=V\cap(\kbar^\sqbullet)^n$ and $\pi|_{V\cap(\kbar^\sqbullet)^n}:V\cap(\kbar^\sqbullet)^n\to Y\cap(\kbar^\sqbullet)^{n-1}$ is injective}$\,$'. $\sqbullet$
\end{remark}

\begin{proof}[Proof of Theorem \em\ref{thmproj}]
Let $Z$ be a $K$-Zariski closed subset of $X\subset L^n$. As $\ell\in\II_K(X)\subset\II_K(Z)$, we can apply Lemma \ref{kJP} to $\gta:=\II_K(Z)$ obtaining $\Pi(Z)=\Pi(\ZZ_L(\II_K(Z)))=\ZZ_L(\II_K(Z)\cap K[\tilde{\x}])$. Thus, $\Pi(Z)\subset L^{n-1}$ is $K$-algebraic. In particular, the set $Y=\Pi(X)\subset L^{n-1}$ is $K$-algebraic and the map $\pi$ is $K$-Zariski closed. Moreover, the existence of the polynomial $\ell$ in $\II_K(X)$ implies that the fibers of $\pi$ are finite, so $\dim(Y)=\dim_L(Y)=\dim_L(X)=\dim(X)$. 

$(\mr{i})$ Recall that $\reg^{K|K}_{X,a}=K[\x]_{\gtn_a}/(\II_K(X)K[\x]_{\gtn_a})$ is the $K$-local ring of $X\subset L^n$ at $a$, where $\gtn_a=\{f\in K[\x]:f(a)=0\}$. Consider an arbitrary element $\frac{f}{g}+\II_K(X)K[\x]_{\gtn_a}$ of $\reg^{K|K}_{X,a}$. As $g$ is a polynomial in $K[\x]$ with $g(a)\neq0$ and $\pi^{-1}(b)=\{a\}$ by hypothesis, it follows that $b\not\in\pi(\ZZ_L(g)\cap X)$. As $\pi(\ZZ_L(g)\cap X)$ is $K$-Zariski closed in $Y\subset L^{n-1}$, there exists $k\in K[\tilde{\x}]$ such that $k(b)\neq0$ and $\pi(\ZZ_L(g)\cap X)\subset\ZZ_L(k)$. Consider the polynomial $k\circ\pi\in K[\x]$. Observe that $(k\circ\pi)(a)=k(b)\neq0$ and $\ZZ_L(g)\cap X\subset\ZZ_L(k\circ\pi)$. The latter inclusion allows to define the regular function $1/g:X\setminus\ZZ_L(k\circ\pi)\to L$ by $(1/g)(x):=g(x)^{-1}$. As $X\cap\ZZ_L(g)$ is an algebraic subset of $L^n$ and the field $L$ is algebraically closed, Hilbert's Nullstellensatz implies the existence of an integer $s\geq1$ such that $(k\circ\pi)^s\in\{g\}L[\x]+\II_L(X)$. Thus, the product $(k\circ\pi)^s(1/g)$ defines a polynomial function on $X$ and there exists $h\in L[\x]$ such that $(k\circ\pi)^s-hg\in\II_L(X)$. Let $\Bb=\{u_j\}_{j\in J}$ be a basis of $L$ as a $K$-vector space such that $u_{j_0}=1$ for some $j_0\in J$, and let $\{g_1,\ldots,g_r\}$ be a system of generators of $\II_K(X)$ in $K[\x]$. By Corollary \ref{kreliablec}, we have $\II_L(X)=\II_K(X)L[\x]$, so $(k\circ\pi)^s-hg=\sum_{i=1}^rf_ig_i$ for some $f_1,\ldots,f_r\in L[\x]$. By Lemma \ref{k0}, we can write $h=\sum_{j\in J}u_jh_j$ and $f_i=\sum_{j\in J}u_jf_{ij}$, where $h_j,f_{ij}\in K[\x]$ and only finitely many of these polynomials are nonzero. As $k\circ\pi$ and $g$ belong to $K[\x]$, we have:
\begin{align*}\textstyle
u_{j_0}((k\circ\pi)^s-h_{j_0}g)-\sum_{j\in J\setminus\{j_0\}}u_j(h_jg)&\textstyle=(k\circ\pi)^s-hg=\sum_{i=1}^rf_ig_i\\
&\textstyle=\sum_{i=1}^r\big(\sum_{j\in J}u_jf_{ij}\big)g_i=\sum_{j\in J}u_j\big(\sum_{i=1}^rf_{ij}g_i\big).
\end{align*}
By the uniqueness part of Lemma~\ref{k0}, we deduce $(k\circ\pi)^s-h_{j_0}g=\sum_{i=1}^rf_{ij_0}g_i\in\II_K(X)$, so
\begin{equation}\label{shafarevich1}
\textstyle
\frac{f}{g}+\II_K(X)K[\x]_{\gtn_a}=\frac{h_{j_0}f}{(k\circ\pi)^s}+\II_K(X)K[\x]_{\gtn_a}=\pi^*(\frac{1}{k^s})(h_{j_0}f+\II_K(X)K[\x]_{\gtn_a})\quad\text{in $\reg^{K|K}_{X,a}$.}
\end{equation}

Let us now consider $h_{j_0}f\in K[\x]$ as a polynomial in $K[\tilde{\x}][\x_n]$ and divide it by $\ell(\x)=\x_n^d+a_{n-1}(\tilde{\x})\x_n^{d-1}+\cdots+a_1(\tilde{\x})\x_n+a_0(\tilde{\x})\in K[\tilde{\x}][\x_n]=K[\x]$, obtaining $(h_{j_0}f)(\x)=q(\x)\ell(\x)+\sum_{i=0}^{d-1}b_i(\tilde{\x})\x_n^i$ for some (unique) $q\in K[\x]$ and $b_0,\ldots,b_{d-1}\in K[\tilde{\x}]\subset K[\x]$. As $\ell\in\II_K(X)$, we deduce:
\begin{align}
h_{j_0}f+\II_K(X)K[\x]_{\gtn_a}&\textstyle=\sum_{i=0}^{d-1}(b_i(\tilde{\x})+\II_K(X)K[\x]_{\gtn_a})(\x_n^i+\II_K(X)K[\x]_{\gtn_a})\nonumber\\
&\textstyle=\sum_{i=0}^{d-1}\pi^*(b_i)(\x_n^i+\II_K(X)K[\x]_{\gtn_a})\quad\text{in $\reg^{K|K}_{X,a}$.}\label{shafarevich2}
\end{align}
By \eqref{shafarevich1} and \eqref{shafarevich2}, we have:
$$
\textstyle
\frac{f}{g}+\II_K(X)K[\x]_{\gtn_a}=\sum_{i=0}^{d-1}\pi^*(\frac{b_i}{k^s})(\x_n^i+\II_K(X)K[\x]_{\gtn_a})\quad\text{in $\reg^{K|K}_{X,a}$.}
$$
This proves that $\reg^{K|K}_{X,a}=\sum_{i=0}^{d-1}\pi^*(\reg^{K|K}_{Y,b})(\x_n^i+\II_K(X)K[\x]_{\gtn_a})$, so the homomorphism $\pi^*$ is finite.

$(\mr{ii})(\mr{a})$ As $b\in(\kbar^\sqbullet)^{n-1}$, $K[b]$ coincides with the field $K(b)$. Set $\gtb:=\{f(b,\x_n)\in K[b][\x_n]:f\in\II_K(X)\}\neq(0)$, because $\ell\in\II_K(X)$. Let us prove that $\gtb$ is an ideal of $K[b][\x_n]$. Evidently, $\gtb$ is closed under addition. If $g\in K[b][\x_n]$, there exists $G\in K[\x]$ such that $g(\x_n)=G(b,\x_n)$. Thus, for each $f\in\I_K(X)$, we have $g(\x_n)f(b,\x_n)=G(b,\x_n)f(b,\x_n)=(Gf)(b,\x_n)\in\gtb$, because $Gf\in\I_K(X)$. As $K[b][\x_n]$ is a PID, $\gtb$ is generated by an element $h(b,\x_n)\neq0$ for some $h\in\II_K(X)$. Let $a_n\in\kbar^\sqbullet$ be such that $a=(b,a_n)$. By hypothesis, $(\{b\}\times L)\cap X=\{a\}$, so $\{x_n\in L:h(b,x_n)=0\}=\ZZ_L(\gtb)=\{a_n\}$ and $h(b,\x_n)=c(\x_n-a_n)^m$ for some $c\in L\setminus\{0\}$ and $m\in\N^*$. In fact, $c$ belongs to $K[b]\setminus\{0\}$, because $h(b,\x_n)$ and $(\x_n-a_n)^m$ are both polynomials in $K[b][\x_n]$, and $(\x_n-a_n)^m$ is monic.

Let us prove that $m=1$. Suppose $m\geq2$. Choose an arbitrary $f\in\I_K(X)$ and write $f(b,\x_n)=H(b,\x_n)h(b,\x_n)=cH(b,\x_n)(\x_n-a_n)^m$ for some $H\in K[\x]$. Observe that
$$
\textstyle
\frac{\partial f}{\partial \x_n}(b,\x_n)=c\frac{\partial H}{\partial \x_n}(b,\x_n)(\x_n-a_n)^m+cH(b,\x_n)m(\x_n-a_n)^{m-1},
$$
so $\frac{\partial f}{\partial\x_n}(a)=0$. This means that $(0,\ldots,0,1)^T\in T^{L|K}_a(X)$. As $L$ is algebraically closed, Remarks \ref{tang-bcr}$(\mr{iii})$ assures that $T^{L|K}_a(X)=T_a(X)$, so $(0,\ldots,0,1)^T\in \ker(d_a\pi)$. This contradicts the injectivity of $d_a\pi$. Consequently, $m=1$ and $h(b,\x_n)=c(\x_n-a_n)$. As $h(b,\x_n)\in K[b][\x_n]$ and $c\in K[b]\setminus\{0\}$, we deduce $a_n\in K[b]$, so $K[a]=K[b]$.

As $b=\pi(a)$ and $K[a]=K[b]$, the differential $d^{K|K}_a\pi:T^{K|K}_a(X)\to T^{K|K}_b(Y)$ is well-defined (according to Definition \ref{def:differential}). As $L$ is algebraically closed, using again Remarks \ref{tang-bcr}$(\mr{iii})$, we know that $T^{K|K}_a(X)=T_a(X)\cap K[a]^n$ so $\ker(d^{K|K}_a\pi)=\ker(d_a\pi)\cap K[a]^n=\{0\}$ and $d^{K|K}_a\pi$ is injective.

$(\mr{ii})(\mr{b})$ Recall that $\reg^{K|K}_{Y,b}=K[\tilde{\x}]_{\gtn_b}/(\II_K(Y)K[\tilde{\x}]_{\gtn_b})$ is the $K$-local ring of $Y$ at $b$, where $\gtn_b=\{h\in K[\tilde{\x}]:h(b)=0\}$. Let $\gtM_a:=\gtn_aK[\x]_{\gtn_a}/(\II_K(X)K[\x]_{\gtn_a})$ be the maximal ideal of $\reg^{K|K}_{X,a}$ and let $\gtM_b:=\gtn_bK[\tilde{\x}]_{\gtn_b}/(\II_K(Y)K[\tilde{\x}]_{\gtn_b})$ be the maximal ideal of $\reg^{K|K}_{Y,b}$. Consider the pullback homomorphism $\pi^*:\reg^{K|K}_{Y,b}\to\reg^{K|K}_{X,a}$ induced by $\pi$. We will prove that $\pi^*$ is a ring isomorphism.

\textit{Let us see that $\pi^*$ is injective}. Let $\frac{h}{k}+\II_K(Y)K[\tilde{\x}]_{\gtn_b}\in\ker(\pi^*)$, that is, $\frac{h\circ\pi}{k\circ\pi}\in\II_K(X)K[\x]_{\gtn_a}$. This means that there exists a $K$-Zariski closed subset $P$ of $X\subset L^n$ such that $a\not\in P$ and the polynomial $h\circ\pi\in K[\x]$ vanishes on $X\setminus P$. We have that $\pi(P)$ is a $K$-Zariski closed subset of $Y\subset L^{n-1}$. Moreover, $b=\pi(a)\not\in\pi(P)$, because $\pi^{-1}(b)=\{a\}$ and $a\not\in P$. Observe that $Y=\pi(X)=\pi(X\setminus P)\cup\pi(P)$, so $\pi(X\setminus P)\supset Y\setminus\pi(P)$. As $h$ vanishes on $\pi(X\setminus P)$, $h$ also vanishes on the $K$-Zariski open neighborhood $Y\setminus\pi(P)$ of $b$ in $Y\subset L^{n-1}$, so $\frac{h}{k}\in\I_K(Y)K[\tilde{\x}]_{\gtn_b}$. This proves the injectivity of $\pi^*$.

\textit{Let us show that $\pi^*$ is surjective}. By $(\mr{ii})(\mr{a})$, $d^{K|K}_a\pi$ is well-defined and it is injective, so its dual map $(d^{K|K}_a\pi)^\wedge$ is surjective. By Lemma \ref{455}, the pullback $K[a]$-linear map $\pi^{**}:\gtM_b/\gtM_b^2\to\gtM_a/\gtM_a^2$ induced by $\pi$ is surjective as well. It follows that $\gtM_a=\pi^*(\gtM_b)\reg^{K|K}_{X,a}+\gtM_a^2$. Thus, $\gtM_a=\pi^*(\gtM_b)\reg^{K|K}_{X,a}+(\pi^*(\gtM_b)\reg^{K|K}_{X,a}+\gtM_a^2)\gtM_a=\pi^*(\gtM_b)\reg^{K|K}_{X,a}+\gtM_a^3$ and by induction $\gtM_a=\pi^*(\gtM_b)\reg^{K|K}_{X,a}+\gtM_a^\mu$ for all $\mu\geq2$. As $\reg^{K|K}_{X,a}$ is a Noetherian local ring, Krull's theorem \cite[Cor.10.20]{am} assures that $\bigcap_{\mu\geq2}(\pi^*(\gtM_b)\reg^{K|K}_{X,a}+\gtM_a^\mu)=\pi^*(\gtM_b)\reg^{K|K}_{X,a}$, so
\begin{equation}\label{krull}
\pi^*(\gtM_b)\reg^{K|K}_{X,a}=\gtM_a.
\end{equation}

Recall that the residue field $\reg^{K|K}_{Y,b}/\gtM_b$ of $\reg^{K|K}_{Y,b}$ is isomorphic to $K[b]$ via the field isomorphism $(\frac{h}{k}+\II_K(Y)K[\tilde{\x}]_{\gtn_b})+\gtM_b\mapsto h(b)(k(b))^{-1}$. Pick an arbitrary element $\frac{f}{g}+\II_K(X)K[\x]_{\gtn_a}$ of $\reg^{K|K}_{X,a}$. By $(\mr{ii})(\mr{a})$, we know that $K[a]=K[b]$, so $f(a)(g(a))^{-1}\in K[b]$. Thus, there exists $h\in K[\tilde{\x}]$ such that $h(b)=f(a)(g(a))^{-1}$. Consequently, $(h\circ\pi)+\II_K(X)K[\x]_{\gtn_a}$ is an element of $\pi^*(\reg^{K|K}_{Y,b})$ such that
$$\textstyle
\frac{f}{g}-(h\circ\pi)+\II_K(X)K[\x]_{\gtn_a}\in\gtM_a.
$$
This proves that $\reg^{K|K}_{X,a}=\pi^*(\reg^{K|K}_{Y,b})+\gtM_a$. Combining the latter equation with \eqref{krull}, we deduce:
\begin{equation}\label{nakayama}
\reg^{K|K}_{X,a}=\pi^*(\reg^{K|K}_{Y,b})+\pi^*(\gtM_b)\reg^{K|K}_{X,a}.
\end{equation}
As the homomorphism $\pi^*$ is finite by $(\mr{i})$, Nakayama's lemma \cite[Cor.2.7]{am} assures that $\reg^{K|K}_{X,a}=\pi^*(\reg^{K|K}_{Y,b})$, that is, $\pi^*$ is surjective, as required. 

$(\mr{ii})(\mr{c})$ By $(\mr{ii})(\mr{b})$, $\pi^*$ is surjective so there exist $v,w\in K[\tilde{\x}]$ such that $w(b)\neq0$ and 
$$
\textstyle
\frac{v\circ \pi}{w\circ\pi}+\II_K(X)K[\x]_{\gtn_a}=\pi^*\big(\frac{v}{w}+\II_K(Y)K[\tilde{\x}]_{\gtn_b}\big)=\x_n+\II_K(X)K[\x]_{\gtn_a}.
$$
Thus, there exists a polynomial $t\in K[\x]$ such that $t(a)\neq0$ and $t(\x)((v\circ\pi)(\x)-\x_n(w\circ\pi)(\x))\in\II_K(X)$. Set $D:=X\cap\ZZ_L(t)\subset L^n$. Observe that $X\setminus D$ is a $K$-Zariski open neighborhood of $a$ in $X\subset L^n$ and $x_n=v(x')(w(x'))^{-1}$ for all $(x',x_n)\in X\setminus D$. Moreover, $\pi(D)$ is a $K$-Zariski closed subset of $Y\subset L^{n-1}$. As $a\not\in D$ and $\pi^{-1}(b)=\{a\}$ by hypothesis, it follows that $b\not\in\pi(D)$ and $\pi^{-1}(\pi(D))$ is a $K$-Zariski closed subset of $X\subset L^n$ that does not contain $a$. Set $U:=X\setminus\pi^{-1}(\pi(D))$. Observe that $\pi(U)=Y\setminus\pi(D)$ is a $K$-Zariski open neighborhood of $b$ in $Y$ and the restriction $\pi':U\to\pi(U)$ of $\pi$ is bijective with inverse $(\pi')^{-1}(x')=(x',v(x')(w(x'))^{-1})$, so $\pi'$ is a $K$-biregular isomorphism.

$(\mr{ii})(\mr{d})$ This item follows immediately from $(\mr{ii})(\mr{c})$.

$(\mr{iii})$ Suppose that $\pi^{-1}(\pi(c))=\{c\}$ and the differential $d_c\pi$ is also injective for all $c\in V\cap(\kbar^\sqbullet)^n$. By the extension of coefficients and $(\mr{ii})(\mr{c})$, for each $c\in V\cap(\kbar^\sqbullet)^n$, there exists a $K$-Zariski open neighborhood $U^c$ of $c$ in $V\cap(\kbar^\sqbullet)^n$ such that $\pi(U^c)$ is a $K$-Zariski open neighborhood of $\pi(c)$ in $Y\cap(\kbar^\sqbullet)^{n-1}$ and $\pi|_{U^c}:U^c\to\pi(U^c)$ is a $K$-biregular isomorphism. As $\pi(V\cap(\kbar^\sqbullet)^n)=\bigcup_{c\in V\cap(\kbar^\sqbullet)^n}\pi(U^c)$, we deduce that $\pi(V\cap(\kbar^\sqbullet)^n)$ is a $K$-Zariski open subset of $Y\cap(\kbar^\sqbullet)^{n-1}$.

We claim that $\pi(V\cap(\kbar^\sqbullet)^n)=\pi(V)\cap(\kbar^\sqbullet)^{n-1}$. As the inclusion $\pi(V\cap(\kbar^\sqbullet)^n)\subset \pi(V)\cap(\kbar^\sqbullet)^{n-1}$ is evident, we only have to show that $\pi(V)\cap(\kbar^\sqbullet)^{n-1}\subset\pi(V\cap(\kbar^\sqbullet)^n)$. Let $x\in V$ be such that $\pi(x)\in(\kbar^\sqbullet)^{n-1}$. By the extension of coefficients, we have
$$
\{y\in V\cap(\kbar^\sqbullet)^n:\pi(y)=\pi(x)\}_L=(\pi|_V)^{-1}(\pi(x)).
$$
As $x\in(\pi|_V)^{-1}(\pi(x))$, we have $(\pi|_V)^{-1}(\pi(x))\neq\varnothing$, so $\{y\in V\cap(\kbar^\sqbullet)^n:\pi(y)=\pi(x)\}\neq\varnothing$ or, equivalently, there exists $c\in V\cap(\kbar^\sqbullet)^n$ such that $\pi(c)=\pi(x)$. This proves that $\pi(V)\cap(\kbar^\sqbullet)^{n-1}\subset\pi(V\cap(\kbar^\sqbullet)^n)$, so $\pi(V\cap(\kbar^\sqbullet)^n)=\pi(V)\cap(\kbar^\sqbullet)^{n-1}$, as claimed.

The equality $\pi(V\cap(\kbar^\sqbullet)^n)=\pi(V)\cap(\kbar^\sqbullet)^{n-1}$ implies that the restriction $\rho:V\cap(\kbar^\sqbullet)^n\to\pi(V)\cap(\kbar^\sqbullet)^n$ of $\pi$ is bijective. Moreover, $\rho^{-1}:\pi(V)\cap(\kbar^\sqbullet)^{n-1}\to V\cap(\kbar^\sqbullet)^n$ is $K$-regular (because $K$-Zariski locally it coincides with the maps $(\pi|_{U^c})^{-1}$). Using the extension of coefficients again, we deduce that $(\pi|_V)^{-1}=(\rho^{-1})_L$ is also $K$-regular, so $\pi|_V$ is a $K$-biregular isomorphism.

$(\mr{iv})$ If $\pi$ is a $K$-biregular isomorphism, then it is evident that $\pi$ and the differentials $\{d_c\pi\}_{c\in X}$ are injective. Conversely, if $\pi$ and the differentials $\{d_c\pi\}_{c\in X\cap(\kbar^\sqbullet)^n}$ are injective, then $(\mr{iii})$ implies that $\pi$ is a $K$-biregular isomorphism, as required.
\end{proof}

Let us rewrite the statement of the second result we want to prove. Recall that the symbol $\widehat{V}_A$ was defined in Notation \ref{notaz526}. Moreover, if $r<n$, then $x':=(x_1,\ldots,x_r)^T$ and $x'':=(x_{r+1},\ldots,x_n)^T$ are the coordinates of $L^r$ and $L^{n-r}$, respectively.

\vspace{2mm}

\noindent \textbf{Theorem \ref{corproj}.} {\it Suppose that $L$ is algebraically closed. Let $X\subset L^n$ be a $K$-algebraic set, let $\overline{X}$ be the $K$-Zariski closure of $X$ in $\PP^n(L)$, let $r\in\N^*$ with $r<n$ and let $A\in\mc{M}_{r,n-r}(K)$ be such that $\overline{X}\cap\widehat{V}_A=\varnothing$. Consider the projection $\pi_A:L^n\to L^r$, set $Y:=\pi_A(X)$ and denote $\pi:X\to Y$ the restriction of $\pi_A$ from $X$ to $Y$. Then $Y\subset L^r$ is a $K$-algebraic set, $\dim(Y)=\dim(X)$ and $\pi$ is a $K$-Zariski closed map. Moreover, if we pick a point $a\in X\cap(\kbar^\sqbullet)^n$ and we set $b:=\pi(a)\in Y\cap(\kbar^\sqbullet)^r$, then the items $(\mr{i})$ and $(\mr{ii})$ stated in Theorem \ref{thmproj} hold. In addition, the items $(\mr{iii})$ and $(\mr{iv})$ stated in Theorem \ref{thmproj} are also true.}
\begin{proof}
Consider the projective transformation $\psi:\PP^n(L)\to\PP^n(L)$ defined by $\psi([x_0,x',x'']):=[x_0,x'-Ax'',x'']$. Observe that $\psi(L^n)=L^n$ and $\psi(V_A)=V_O=L^{n-r}$, where $O$ denotes the zero matrix in $\mc{M}_{r,n-r}(K)$. In addition, as the entries of $A$ are in $K$, $\psi$ is a $K$-biregular isomorphism. Thus, replacing $X\subset L^n$ with $\psi(X)\subset L^n$ and $A$ with the zero matrix $O$, we can assume that $V_A=L^{n-r}$ and $\pi_A=\Pi_r$, where $\Pi_r:L^n\to L^r$ is the projection onto the first $r$ coordinates, that is, $\Pi_r(x)=x'$.

Let us now prove the statement by induction on $n-r\in\{1,\ldots,n-1\}$. Let $\Pi_{n-1}:L^n\to L^{n-1}$ be the projection onto the first $n-1$ coordinates, let $X':=\Pi_{n-1}(X)\subset L^{n-1}$, let $\overline{X'}$ be the $K$-Zariski closure of $X'$ in $\PP^{n-1}(L)$, let $\rho:X\to X'$ be the restriction of $\Pi_{n-1}$ from $X$ to $X'$, and let $\eta:X'\to Y$ be the projection $\eta(x_1,\ldots,x_{n-1}):=(x_1,\ldots,x_r)$. As $[0,\ldots,0,1]\not\in\overline{X}$, by Lemma \ref{lem:266}, there exists a polynomial $\ell\in\II_K(X)$ such that $s:=\ell^\sfh(0,\ldots,0,1)\in K\setminus\{0\}$, where $\ell^\sfh$ is the homogeneization of $\ell$. Thus, replacing $\ell$ with $s^{-1}\ell$, we can assume that $\ell$ is of the form $\ell(\x)=\x_n^d+a_{n-1}(\tilde{\x})\x_n^{d-1}+\cdots+a_1(\tilde{\x})\x_n+a_0(\tilde{\x})\in K[\tilde{\x}][\x_n]=K[\x]$, where $\tilde{\x}=(\x_1,\ldots,\x_{n-1})$. By Theorem \ref{thmproj}, the (base) case $n-r=1$ of the inductive process is proven. In particular, $X'\subset L^{n-1}$ is a $K$-algebraic set, $\dim(X')=\dim(X)$, $L^{n-1-r}=\Pi_{n-1}(L^{n-r})$ and, by construction, it holds $\overline{X'}\cap\widehat{L}^{n-1-r}=\varnothing$, where $\widehat{L}^{n-1-r}:=\{x_0=0,\ldots,x_r=0\}\subset\PP^{n-1}(L)$.

Suppose $n-r\geq2$. Let us prove items $(\mr{i})$-$(\mr{iv})$ as stated in Theorem \ref{thmproj}. Let $a\in X\cap(\kbar^\sqbullet)^n$, let $a':=\rho(a)\in X'\cap(\kbar^\sqbullet)^{n-1}$ and let $b:=\pi(a)\in Y\cap(\kbar^\sqbullet)^r$. Observe that $\pi=\eta\circ\rho$ and $d_a\pi=d_{a'}\eta\circ d_a\rho$. We have:

$(\mr{i})$ If $\pi^{-1}(b)=\{a\}$, then $\eta^{-1}(b)=\{a'\}$ so, by the inductive hypothesis, the homomorphism $\eta^*:\reg^{K|K}_{Y,b}\to\reg^{K|K}_{X',a'}$ is finite. As $\rho^{-1}(a')=\{a\}$, (by the case $n-r=1$) we deduce that $\rho^*:\reg^{K|K}_{X',a'}\to\reg^{K|K}_{X,a}$ is finite, so $\pi^*=\rho^*\circ\eta^*$ is also finite.

$(\mr{ii})$ Suppose that $\pi^{-1}(b)=\{a\}$ and $d_a\pi$ is injective. As $\pi^{-1}(b)=\rho^{-1}(\eta^{-1}(b))=\{a\}$, $\rho(a)=a'$ and $\eta(a')=b$, it follows that $\eta^{-1}(b)=\{a'\}$, $\rho^{-1}(a')=\{a\}$. In addition, $d_a\rho$ is injective. As $d_a\rho$ is a $L$-linear isomorphism by the case $n-r=1$, we deduce that $d_{a'}\eta$ is also injective. Now, $K[a']=K[b]$ by the inductive hypothesis, $d^K_{a'}\eta$ is well-defined and it is a $K[a']$-linear isomorphism, $\eta^*$ is a ring isomorphism, there exists a $K$-Zariski open neighborhood $U'$ of $a'$ in $X'$ such that $\eta(U')$ is a $K$-Zariski open neighborhood of $b$ in $Y$ and the restriction of $\eta$ from $U'$ to $\eta(U')$ is a $K$-biregular isomorphism and $d_a\eta$ is a $L$-linear isomorphism. By the case $n-r=1$, we also deduce: $K[a]=K[a']$ so $K[a]=K[b]$, $d^{K|K}_a\pi=d^{K|K}_a\rho\circ d^{K|K}_{a'}\eta$ is well-defined and it is a $K[a]$-linear isomorphism, $\pi^*=\rho^*\circ\eta^*$ is a ring isomorphism and, shrinking $U'$ around $a'$ if necessary and calling $U:=\rho^{-1}(U')$, we have that $U$ is a $K$-Zariski open neighborhood of $a$ in $X$ and the restriction of $\pi=\eta\circ\rho$ from $U$ to $\pi(U)=\eta(U')$ is a $K$-biregular isomorphism.

$(\mr{iii})$ Let $V$ be a given $K$-Zariski open subset of $X$ and let $\rho|:V\to\rho(V)$, $\eta|:\rho(V)\to\pi(V)$ and $\pi|:V\to\pi(V)$ be the corresponding restrictions of $\rho$, $\eta$ and $\pi$, respectively. Suppose that $\pi^{-1}(\pi(c))=\{c\}$ (so $\rho^{-1}(\eta^{-1}(\eta(\rho(c))))=\{c\}$) and $d_c\pi$ is injective for each $c\in V\cap(\kbar^\sqbullet)^n$. Then $\rho^{-1}(\rho(c))=\{c\}$ and $d_c\rho$ is injective for each $c\in V\cap(\kbar^\sqbullet)^n$. By the case $n-r=1$, $\rho(V)$ is a $K$-Zariski open subset of $X'$ and $\rho|$ is a $K$-biregular isomorphism. Thus, $\eta^{-1}(\eta(e))=\{e\}$ and $d_e\eta$ is injective for each $e\in \rho(V\cap(\kbar^\sqbullet)^n)=\rho(V)\cap(\kbar^\sqbullet)^{n-1}$. By the inductive hypothesis, $\eta(\rho(V))$ is a $K$-Zariski open subset of $Y$ and $\eta|$ is a $K$-biregular isomorphism. Consequently, $\pi(V)=\eta(\rho(V))$ is a $K$-Zariski open subset of $Y$ and $\pi|=\eta|\circ\rho|$ is also a $K$-biregular isomorphism.

$(\mr{iv})$ It follows immediately from the preceding item with $V=X$, as required.
\end{proof}

Recall that the fixed fields $L$ and $K$ are of characteristic zero, so they are infinite. The next result is well-known. We include the proof for the sake of completeness. 

\begin{lem}\label{K-dense-L}
The set $K^n$ is dense in $L^n$ with respect to the usual ($L$-)Zariski topology. In particular, if $\Omega$ is a non-empty $K$-Zariski open subset of $L^n$, then $\Omega \cap K^n$ is a non-empty ($K$-)Zariski open subset of $K^n$.
\end{lem}
\begin{proof}
A standard argument works. Let $\Omega'$ be a non-empty ($L$-)Zariski open subset of $L^n$. Let us prove by induction on $n\in\N^*$ that $\Omega'\cap K^n\neq\varnothing$. If $n=1$, then $L\setminus\Omega'$ is finite, so $\Omega'\cap K\neq\varnothing$ because $K$ is infinite. Suppose $n\geq2$. Pick a point $a\in\Omega'$ and write $a=(a',a_n)\in L^{n-1}\times L=L^n$. By induction hypothesis, there exists $b'\in K^{n-1}$ such that $(b',a_n)\in\Omega'\cap(L^{n-1}\times\{a_n\})$. By the case $n=1$, there also exists $b'_n\in K$ such that $(b',b'_n)\in\Omega'\cap(\{b'\}\times K)$. Consequently, $(b',b'_n)\in\Omega'\cap K^n$, so $\Omega'\cap K^n\neq\varnothing$. Finally, if $\Omega'$ is $K$-Zariski open in $L^n$, then the intersection $\Omega'\cap K^n$ is non-empty and ($K$-)Zariski open in $K^n$.
\end{proof}

The third result to be proved is the following.

\vspace{2mm}
\noindent \textbf{Lemma \ref{618}.} {\it Suppose that $L$ is algebraically closed. Let $r\in\N^*$ and let $W\subset\PP^n(L)$ be a $K$-algebraic set such that $r<n$, $W\subset H^n_0$ and $\dim(W)\leq r-1$. Define $\Omega:=\{A\in\mc{M}_{r,n-r}(K):W\cap\widehat{V}_A=\varnothing\}$. Then $\Omega$ is a non-empty $K$-Zariski open subset of $\mc{M}_{r,n-r}(K)=K^{r(n-r)}$.}

\begin{proof}
Let $\Omega_L:=\{A\in\mc{M}_{r,n-r}(L):W\cap\widehat{V}_A=\varnothing\}$. As $\Omega=\Omega_L\cap\mc{M}_{r,n-r}(K)$, by Lemma \ref{K-dense-L}, it is enough to prove that $\Omega_L$ is a non-empty $K$-Zariski open subset of $\mc{M}_{r,n-r}(L)=L^{r(n-r)}$.

Identify $H^n_0\subset\PP^n(L)$ with $\PP^{n-1}(L)$ and its homogeneous coordinates $[0,x]:=[0,x',x'']$ with $[x]:=[x',x'']$, respectively. Consider $W$ as a $K$-algebraic subset of $\PP^{n-1}(L)$, and choose homogeneous polynomials $g_1,\ldots,g_s\in K[\x]_\sfh$ such that $W=\PP\ZZ_L(\{g_1\ldots,g_s\})\subset\PP^{n-1}(L)$. Consider also the following $K$-Zariski closed subset $\EuScript{Q}$ of $\PP^{n-1}(L)\times\mc{M}_{r,n-r}(L)=\PP^{n-1}(L)\times L^{r(n-r)}$:
$$
\textstyle
\EuScript{Q}:=\{([x],A)\in\PP^{n-1}(L)\times\mc{M}_{r,n-r}(L):x'=Ax'',g_1(x)=0,\ldots,g_s(x)=0\}.
$$
Let $\rho:\PP^{n-1}(L)\times\mc{M}_{r(n-r)}(L)\to\mc{M}_{r(n-r)}(L)$ be the canonical projection onto the second factor. As $\Omega_L=\mc{M}_{r(n-r)}(L)\setminus\rho(\EuScript{Q})$, we deduce that $\Omega_L$ is $K$-Zariski open in $\mc{M}_{r,n-r}(L)$ by Theorem \ref{thm:K-elimination}$(\mr{i})$.

It remains to show that $\Omega_L$ is non-empty. As $W$ is a usual ($L$-)algebraic subset of $\PP^{n-1}(L)$ of dimension $\leq r-1$, the fact that $\Omega_L\neq\varnothing$ follows from classical results. However, we provide an explicit proof of this fact here.

Let us prove by induction on $n-r\in\{1,\ldots,n-1\}$ that there exists $A\in\mc{M}_{r,n-r}(L)$ such that $W\cap\widehat{V}_A\neq\varnothing$. Define the algebraic set $Z\subset\PP^{n-1}(L)$ by $Z:=W\cup H_n^n$ where $H_n^n:=\{[x]\in\PP^{n-1}(L):x_n=0\}$. As $r-1<n-1$, it holds $\dim(Z)=n-2$ so $Z\subsetneqq\PP^{n-1}(L)$. Pick a point $[a_1,\ldots,a_{n-1},1]\in\PP^{n-1}(L)\setminus Z$ and define the column matrix $A_1:=(a_1,\ldots,a_{n-1})^T\in\mc{M}_{n-1,1}(L)$. Equip $H^n_n\simeq\PP^{n-1}(L)$ with the homogeneous coordinates $[x_0,\ldots,x_{n-1}]$. Let $x^*:=(x_1,\ldots,x_{n-1})^T$ be the coordinates of $L^{n-1}$, let $\pi_{A_1}:L^n\to L^{n-1}$ be the projection $\pi_{A_1}(x^*,x_n):=x^*-A_1x_n$ in the direction of $V_{A_1}$, let $p\in\PP^n(L)$ be the point $p:=[0,a_1,\ldots,a_{n-1},1]$ and let $\widehat{\pi}_{A_1}:\PP^n(L)\setminus\{p\}\to H^n_n$ be the projection of center $p$ defined as $\widehat{\pi}_{A_1}([x_0,x^*,x_n]):=[x_0,\pi_{A_1}(x^*,x_n)]$. As $\widehat{V}_{A_1}=\{p\}$ and $p\not\in W$, we have that $W\cap\widehat{V}_{A_1}=\varnothing$ and the base case $n-r=1$ is proven.

Observe that $\pi_{A_1}$ coincides with the restriction of $\widehat{\pi}_{A_1}$ to $L^n=(\PP^n(L)\setminus\{p\})\setminus H^n_0$ onto $L^{n-1}=H^n_n\setminus H^n_0$. Define $W':=\widehat{\pi}_{A_1}(W)$. As $\widehat{\pi}_{A_1}|_{W}$ is a ($L$-)Zariski closed map and no projective line of $\PP^n(L)$ passing through the point $p$ is contained in $W$ (because $p\not\in W$ and $W$ is Zariski closed in $\PP^{n-1}(L)$), we have that $W'$ is a Zariski closed subset of $H^n_n\simeq\PP^{n-1}(L)$ such that $\dim(W')=\dim(W)\leq r-1$. In addition, $W'\subset H^n_n\cap H^n_0$, so by inductive hypothesis there exist $A_2\in\mc{M}_{r,n-1-r}(L)$ such that $W'\cap\widehat{V}_{A_2}=\varnothing$, where $\widehat{V}_{A_2}:=\{[0,x',x''']\in H^n_n:x'=A_2x'''\}$ and $x''':=(x_{r+1},\ldots,x_{n-1})^T$. By construction,
\begin{equation}\label{eq:equa}
W\cap(\{p\}\cup(\widehat{\pi}_{A_1})^{-1}(\widehat{V}_{A_2}))=\varnothing.
\end{equation}
Let $\pi_{A_2}:L^{n-1}\to L^r$ be the projection $\pi_{A_2}(x',x'''):=x'-A_2x'''$ in the direction of $V_{A_2}$. Let us compute the matrix $M$ associated to the ($L$-linear) projection $\pi_{A_2}\circ\pi_{A_1}:L^n\to L^r$ (with respect to the standard bases of $L^n$ and $L^r$). Define $A_{1,1}:=(a_1,\ldots,a_r)^T\in\mc{M}_{r,1}(L)$, $A_{1,2}:=(a_{r+1},\ldots,a_{n-1})^T\in\mc{M}_{n-1-r,1}(L)$ and $B:=-A_{1,1}+A_2A_{1,2}\in\mc{M}_{r,1}(L)$. Consider the following block matrix product:
\begin{align}\label{eq:equa-A}
M&=
\textstyle
\left(
\begin{array}{c|c}
I_{r} & -A_2
\end{array}
\right)
\left(
\begin{array}{c|c}
I_{n-1} & -A_1
\end{array}
\right)
=
\left(
\begin{array}{c|c}
I_r & -A_2
\end{array}
\right)
\left(
\begin{array}{c|c|c}
I_r & 0 & -A_{1,1}\\
\hline
0 &I_{n-1-r} & -A_{1,2}
\end{array}
\right)\nonumber\\
&=
\left(
\begin{array}{c|c|c}
I_r & -A_2 & B
\end{array}
\right),
\end{align}
where $I_h$ denotes the $(h\times h)$-identity matrix. Define $A:=\big(-A_2\,|\,B\big)\in\mc{M}_{r,n-r}(L)$. By \eqref{eq:equa-A}, $\widehat{V}_A=\{p\}\cup(\widehat{\pi}_{A_1})^{-1}(\widehat{V}_{A_2})$, whereas $W\cap\widehat{V}_A=\varnothing$ by \eqref{eq:equa}, as required. 
\end{proof}

%%%

\bibliographystyle{amsalpha}

\begin{thebibliography}{FWGSS}

\bibitem[AGT]{agt} W.A.\ Adkins, P. Gianni, A. Tognoli: A Nullstellensatz for an algebraically non-closed field. {\em Boll. Un. Mat. Ital. B} (5) {\bf15} (1978), no.1, 338--343.

\bibitem[AK1]{ak1981} S.\ Akbulut, H.\ King: \emph{The topology of real algebraic sets with isolated singularities}, Annals of Math. 113 (1981) 425-446.

\bibitem[AK2]{akbking:tras} S.\ Akbulut, H.\ King: \emph{Topology of Real Algebraic Sets}, Springer-Verlag New York, Mathematical Sciences Research Institute Publications, Volume 25  (1992).

\bibitem[ABR]{abr} C. Andradas, L. Br\"ocker, J.M. Ruiz: Constructible sets in real geometry. {\em Ergebnisse der Mathematik und ihrer Grenzgebiete} (3) {\bf33}. Springer-Verlag, Berlin, (1996).

\bibitem[A]{a} E. Artin: Kennzeichnung des K\"orpers der reellen algebraischen Zahlen. {\em Abh. Math. Sem. Univ. Hamburg} {\bf3} (1924), no. 1, 319--323. 

\bibitem[AS]{as} E. Artin, O. Schreier: Eine Kennzeichnung der reell abgeschlossenen K\"orper. {\em Abh. Math. Sem. Univ. Hamburg} {\bf5} (1927), no. 1, 225--231. 

\bibitem[AM]{am} M.F. Atiyah, I.G. MacDonald: Introduction to commutative algebra. \em Addison-Wesley Publishing Co., Reading, Mass.-London-Don Mills, Ont. \em (1969).

\bibitem[B]{ba} R. Baer: Die Automorphismengruppe eines algebraisch abgeschlossenen K\"orpers der Charakterkistik 0. {\em Math. Z.} {\bf117} (1970), 7--17.

\bibitem[BT]{bt} E. Ballico, A. Tognoli: Algebraic models defined over \textbf{Q} of differential manifolds. \textit{Geom. Dedicata} 42 (1992), no. 2, 155--161.

\bibitem[Be1]{be} O. Benoist: On the bad points of positive semidefinite polynomials. \textit{Math. Z.} 300 (2022), no. 4, 3383--3403.

\bibitem[Be2]{Be2022} O. Benoist: \emph{On the subvarieties with nonsingular real loci of a real algebraic variety}, Geom. Topol. 28 (2024), no. 4, 1693-1725.

\bibitem[BCR]{bcr} J. Bochnak, M. Coste, M.-F. Roy: Real algebraic geometry. {\em Ergeb. Math. } {\bf 36}, Springer-Verlag, Berlin (1998).

\bibitem[Bo1]{b0} N. Bourbaki: Algebra I. Chapters 1--3. Translated from the French. Reprint of the 1989 English translation {\em Elem. Math.} (Berlin) Springer-Verlag, Berlin, (1998).  

\bibitem[Bo2]{b} N. Bourbaki: Algebra II. Chapters 4--7. (edition by P. M. Cohn and J. Howie.) Reprint of the 1990 English edition. {\em Elements of Mathematics} (Berlin). Springer-Verlag, Berlin (2003). 

\bibitem[BV]{bv} S. Boyd, L. Vandenberghe: Convex optimization. \textit{Cambridge University Press, Cambridge}, (2004).

\bibitem[C]{c} H. Cartan: Vari\'et\'es analytiques r\'eelles et vari\'et\'es analytiques complexes. {\em Bull. Soc. Math. France} {\bf85} (1957), 77--99.

\bibitem[Co]{con} K. Conrad: Maximal ideals in polynomial rings.\\ 
{\tt https://kconrad.math.uconn.edu/blurbs/ringtheory/maxideal-polyring.pdf}

\bibitem[CS]{CS1992} M.\ Coste, M.\ Shiota: \emph{Nash triviality in families of Nash manifolds}, Invent. Math., 108 (1992), 349-368.

\bibitem[DeL]{delellis} C.\ De Lellis: \emph{The Masterpieces of John Forbes Nash Jr.}, Pages 391-499 in the book \emph{The Abel Prize 2013--2017}. Edited by Helge Holden and Ragni Piene. Springer, Cham, 2019. xi+774 pp.

\bibitem[vdD]{vdD} L. van den Dries: Tame topology and o-minimal structures. London Mathematical Society Lecture Note Series, {\bf248}. \textit{Cambridge University Press}, Cambridge, (1998).

\bibitem[D]{d} D.W. Dubois: A nullstellensatz for ordered fields. {\em Ark. Mat.} {\bf8} (1969), 111--114.

\bibitem[E]{e} D. Eisenbud: Commutative algebra. With a view toward algebraic geometry. \em Graduate Texts in Math. \em {\bf150}. Springer-Verlag, New York (1995).

\bibitem[F]{fa83} G. Faltings: Endlichkeitss\"atze f\"ur abelsche Variet\"aten \"uber Zahlk\"orpern. (German) [Finiteness theorems for abelian varieties over number fields] Invent. Math. 73 (1983), no. 3, 349--366.

\bibitem[FWGSS]{fwgss} G. Faltings, G. W\"ustholz, F. Grunewald, N. Schappacher, U. Stuhler: Rational points. Papers from the seminar held at the Max-Planck-Institut f\"ur Mathematik, Bonn/Wuppertal, 1983/1984. With an appendix by W\"ustholz {\em Aspects Math.}, E6 Friedr. Vieweg \& Sohn, Braunschweig, (1992).

\bibitem[Fe]{fe} J.F. Fernando: Representation of positive semidefinite elements as sum of squares in $2$-dimensional local rings. \textit{Rev. R. Acad. Cienc. Exactas F\'{\i}s. Nat. Ser. A Mat. RACSAM} {\bf116} (2022), no. 2, Paper No. 59, 65 pp.

\bibitem[FS]{fs} G. Fichou, M. Shiota: Real Milnor fibres and Puiseux series. \textit{Ann. Sci. \'Ec. Norm. Sup\'er.} (4) 50 (2017), no. 5, 1205--1240.

\bibitem[FY]{fy} G. Fichou, Y. Yin: Motivic integration and Milnor fiber. \textit{J. Eur. Math. Soc. (JEMS)} 24 (2022), no. 5, 1617--1678.

\bibitem[GS]{GS} R. Ghiloni, E. Savi: The Nash-Tognoli theorem over the rationals and its version for isolated singularities. arXiv:2302.04142v2

\bibitem[GT]{GT2017} R.\ Ghiloni, A.\ Tancredi: \emph{Algebraicity of Nash sets and of their asymmetric cobordism.} J. Eur. Math. Soc. (JEMS) 19 (2017), no. 2, 507-529.

\bibitem[GP]{gp} G.M. Greuel, G. Pfister: A Singular introduction to commutative algebra. Second, extended edition. With contributions by Olaf Bachmann, Christoph Lossen and Hans Sch\"onemann. With 1 CD-ROM (Windows, Macintosh and UNIX). Springer, Berlin, (2008).

\bibitem[Gr]{g} A. Grothendieck: \'El\'ements de g\'eom\'etrie alg\'ebrique : IV. \'Etude locale des sch\'emas et des morphismes de sch\'emas, I. {\em Inst. Hautes \'Etudes Sci. Publ. Math.}, {\bf20} (1964), 259 pp.

\bibitem[GMT]{gmt} F. Guaraldo, P. Macr\`{\i}, A. Tancredi: Topics on real analytic spaces. {\em Advanced Lectures in Mathematics}. Friedr. Vieweg \& Sohn, Braunschweig, (1986).

\bibitem[H]{h} J. Harris: Algebraic geometry. A first course. {\em Graduate Texts in Mathematics}, {\bf133}. Springer-Verlag, New York (1995).

\bibitem[Ha]{ha} R. Hartshorne: Algebraic geometry. {\em Graduate Texts in Mathematics}, {\bf52} Springer-Verlag, New York-Heidelberg (1977).

\bibitem[HRS]{jrs} J. Heintz, M.-F. Roy, P. Solern\'o: Description of the connected components of a semialgebraic set in single exponential time. {\em Discrete Comput. Geom.} {\bf11} (1994), no.2, 121--140.

\bibitem[Hilb]{hil} D. Hilbert: Ueber die vollen Invariantensysteme. {\em Math. Ann.} {\bf42} (1893), no. 3, 313--373.

\bibitem[Hi]{hi} C.J. Hillar: Sums of squares over totally real fields are rational sums of squares. \textit{Proc. Amer. Math. Soc.} 137(2009), no.3, 921--930.

\bibitem[Hir]{hirsch} M.W.\ Hirsch: \emph{Differential Topology}, Springer-Verlag, Berlin Heidelberg, Volume 33, (1976).

\bibitem[HP]{hp} W.H.D. Hodge, D. Pedoe: Methods of algebraic geometry. Cambridge: Cambridge Univ. Press (1953).

\bibitem[J]{j} N. Jacobson: Lectures in abstract algebra. III. Theory of fields and Galois theory. Second corrected printing. {\em Graduate Texts in Mathematics}, {\bf32} Springer-Verlag, New York-Heidelberg, (1975).

\bibitem[JP]{jp} T. de Jong, G. Pfister: Local analytic geometry. Basic theory and applications. {\em Advanced Lectures in Mathematics}. Friedr. Vieweg \& Sohn, Braunschweig (2000). 

\bibitem[KLYZ]{klyz} E.L. Kaltofen, B. Li, Z. Yang, L. Zhi: Exact certification in global polynomial optimization via sums-of-squares of rational functions with rational coefficients. \textit{J. Symbolic Comput.} 47 (2012), no. 1, 1--15.

\bibitem[K]{k} M. Kato: Extension and contraction of ideals in polynomial rings. {\tt https://math.stackexchange.com/\\questions/83731/extension-and-contraction-of-ideals-in-polynomial-rings}

\bibitem[Ko]{kollar2017} J.\ Koll\'ar: \emph{Nash's work in algebraic geometry}. Bull. Amer. Math. Soc. (N.S.) 54 (2017), no. 2, 307-324.

\bibitem[Ku]{kucharz2011} W.\ Kucharz: \emph{Algebraicity of cycles on smooth manifolds}. Selecta Math. (N.S.) 17 (2011), no. 4, 855-878.

\bibitem[L]{la} D. Laksov: Radicals and Hilbert Nullstellensatz for not necessarily algebraically closed fields. {\em Enseign. Math.} (2) {\bf33} (1987), no.3-4, 323--338.

\bibitem[LW]{lw} D. Laksov, K. Westin: Nullstellens\"atze; conjectures and counterexamples. Real analytic and algebraic geometry (Trento, 1988), 178--190. {\em Lecture Notes in Math.}, {\bf1420} Springer-Verlag, Berlin, (1990).

\bibitem[La]{lang} S. Lang: Number theory III. Diophantine geometry. {\em Encyclopaedia of Mathematical Sciences}, {\bf60}. Springer-Verlag, Berlin, (1991).

\bibitem[Lap]{laplagne} S. Laplagne: Sum of squares decomposition of positive polynomials with rational coefficients. \textit{J. Pure Appl. Algebra} 228 (2024), no. 7, Paper No. 107631, 8 pp.

\bibitem[Li]{l} Q. Liu: Algebraic geometry and arithmetic curves. Translated from the French by Reinie Ern\'e. {\em Oxford Graduate Texts in Mathematics}, {\bf6}. Oxford Science Publications. Oxford University Press, Oxford, (2002).

\bibitem[Lu]{lu} R. Ludhani: Projective Nullstellensatz for not necessarily algebraically closed fields, arXiv:2411.06325

\bibitem[MSV]{msv} V. Magron, M. Safey El Din, T. Vu: Sum of squares decompositions of polynomials over their gradient ideals with rational coefficients. \textit{SIAM J. Optim.} 33 (2023), no. 1, 63--88.

\bibitem[M]{dm} D. Marker: Model theory. An introduction. {\em Graduate Texts in Mathematics}, {\bf217}. Springer-Verlag, New York (2002). 

\bibitem[Mt]{m} H. Matsumura: Commutative Algebra (2nd edition). {\em Math. Lecture Note Series} {\bf 56}, Benjamin, London-Amsterdam-Tokyo (1980).

\bibitem[Mu]{mumford} D. Mumford: Algebraic geometry. I. Complex projective varieties. Grundlehren der Mathematischen Wissenschaften, {\bf221}. Springer-Verlag, Berlin-New York, (1976).

\bibitem[NS]{ns} S. Naldi, R. Sinn: Conic programming: infeasibility certificates and projective geometry. \textit{J. Pure Appl. Algebra} {\bf225} (2021), no. 7, Paper No. 106605, 21 pp.

\bibitem[Na]{nash} J.F.\ Nash: \emph{Real algebraic manifolds}. Ann.\ of Math.\ (2) 56 (1952), 405-421.

\bibitem[NDS]{nds} J. Nie, J. Demmel, B. Sturmfels: Minimizing polynomials via sum of squares over the gradient ideal. \textit{Math. Program.} 106 (2006), no. 3, Ser. A, 587--606.

\bibitem[Pa]{P2021} A.\ Parusi\'{n}ski: \emph{Algebro-geometric equisingularity of Zariski}, a chapter of Handbook of Geometry and Topology of Singularities II, eds. Jos\'e Luis Cisneros-Molina, D$\tilde{\mr{u}}$ng Tr\'ang L\^{e}, Jos\'e Seade, Springer-Verlag, 2021, pp.\ 177-222.

\bibitem[PP1]{paru} A.\ Parusi\'{n}ski, L.\ Paunescu: \emph{Arc-wise analytic stratification, Whitney fibering conjecture and Zariski equisingularity}, Adv. Math. 309 (2017), 254--305.

\bibitem[PP2]{PP} A.\ Parusi\'{n}ski, L.\ Paunescu: \emph{On Teissier's example of an equisingularity class that cannot be defined over the rationals}. arXiv:2501.10863

\bibitem[PR]{PR2020} A.\ Parusi\'{n}ski, G.\ Rond: \emph{Algebraic varieties are homeomorphic to varieties defined over number fields}. Comment. Math. Helv. 95 (2020), no. 2, 339-359.

\bibitem[Pa]{pawlucki} W. Paw{\l}ucki: Quasiregular boundary and Stokes's formula for a subanalytic leaf. \textit{Seminar on deformations} ({\L}\'od\'z/Warsaw, 1982/84), 235--252, {\em Lecture Notes in Math.}, {\bf1165}, \textit{Springer, Berlin}, (1985).

\bibitem[Q]{q} R. Quarez: Tight bounds for rational sums of squares over totally real fields. \textit{Rend. Circ. Mat. Palermo} (2) {\bf59} (2010), no. 3, 377--388.

\bibitem[R]{r} J.-J. Risler: Une caract\'erisation des id\'eaux des vari\'et\'es alg\'ebriques r\'eelles. {\em C. R. Acad. Sci. Paris S\'er.} A-B {\bf271} (1970), A1171--A1173.

\bibitem[Ro]{Ro} G.\ Rond: \emph{Local topological algebraicity with algebraic coefficients of analytic sets or functions}. Algebra Number Theory 12 (2018), no. 5, 1215-1231. 

\bibitem[SZ]{sz} M. Safey El Din, L. Zhi: Computing rational points in convex semialgebraic sets and sum of squares decompositions. \textit{SIAM J. Optim.} {\bf20} (2010), no. 6, 2876--2889.

\bibitem[Sch]{sche} C. Scheiderer: Sums of squares of polynomials with rational coefficients. {\em J. Eur. Math. Soc. (JEMS)} {\bf18} (2016), no. 7, 1495--1513.

\bibitem[Schn]{sch} B. Schnor: Involutions in the group of automorphisms of an algebraically closed field. {\em J. Algebra} {\bf152} (1992), no. 2, 520--524.

\bibitem[Sha]{sha} I.R. Shafarevich: Basic algebraic geometry. 1. Varieties in projective space. Third edition. Translated from the 2007 third Russian edition. Springer, Heidelberg, (2013).

\bibitem[Shi]{Sh} M.\ Shiota: \emph{Nash manifolds}. Lecture Notes in Mathematics, 1269. Springer-Verlag, Berlin, 1987.

\bibitem[St]{steinitz} E. Steinitz: Algebraische Theorie der K\"orper. {\em J. Reine Angew. Math.} {\bf137} (1910), 167--309.

\bibitem[Ste]{st} G. Stengle: A nullstellensatz and a positivstellensatz in semialgebraic geometry. {\em Math. Ann.} {\bf207} (1974), 87--97.

\bibitem[T]{Te90} B.\ Teissier: \emph{Un exemple de classe d'\'equisingularit\'e irrationnelle}. (French) [An example of an irrational equisingularity class] C. R. Acad. Sci. Paris S\'er. I Math. 311 (1990), no. 2, 111-113.

\bibitem[To1]{to1} A. Tognoli: Su una congettura di Nash. \textit{Ann. Scuola Norm. Sup. Pisa Cl. Sci.} (3) {\bf27} (1973), 167--185.

\bibitem[To2]{to2} A. Tognoli: Algebraic geometry and Nash functions. \textit{Institutiones Mathematicae}, III. Academic Press, Inc. [Harcourt Brace Jovanovich, Publishers], London-New York (1978).

\bibitem[Tr1]{tr1} D. Trotman: Stratification theory. \textit{Handbook of geometry and topology of singularities. I}, 243--273, Springer, Cham, (2020).

\bibitem[Tr2]{tr2} D. Trotman: Stratifications, equisingularity and triangulation. \textit{Introduction to Lipschitz geometry of singularities}, 87--110, {\em Lecture Notes in Math.}, {\bf2280}, Springer, Cham, (2020).

\bibitem[TV]{tv} D. Trotman, G. Valette: On the local geometry of definably stratified sets. \textit{Ordered algebraic structures and related topics}, 349--366, Contemp. Math., {\bf697}, \textit{Amer. Math. Soc., Providence, RI}, (2017).

\bibitem[Wa]{walker} R.J. Walker: Algebraic curves. Reprint of the 1950 edition. Springer-Verlag, New York-Heidelberg, (1978). 

\bibitem[Wh]{wh} H. Whitney: Elementary structure of real algebraic varieties. {\em Ann. of Math.} (2) {\bf66} (1957), 545--556.

\bibitem[Wi]{wi} A. Wiles: Modular elliptic curves and Fermat's last theorem. {\em Ann. of Math.} (2) {\bf141} (1995), no. 3, 443--551.

\bibitem[WSV]{wsv} Handbook of semidefinite programming. Theory, algorithms, and applications. Edited by Henry Wolkowicz, Romesh Saigal and Lieven Vandenberghe. International Series in Operations Research \& Management Science, 27. \textit{Kluwer Academic Publishers, Boston, MA}, (2000).

\bibitem[ZS1]{zs1} O. Zariski, P. Samuel: Commutative algebra, Volume I. With the cooperation of I. S. Cohen. {\em The University Series in Higher Mathematics}. D. Van Nostrand Company, Inc., Princeton, New Jersey, (1958).

\bibitem[ZS2]{zs2} O. Zariski, P. Samuel: Commutative algebra. Vol. II. {\em The University Series in Higher Mathematics}. D. Van Nostrand Co., Inc., Princeton, N. J.-Toronto-London-New York, (1960).
\end{thebibliography}

\end{document}